\documentclass[12pt, leqno, english]{book}
\usepackage{makeidx}
\usepackage[pagebackref]{hyperref}
\makeindex
\usepackage{graphicx,amssymb,stmaryrd,amsfonts,epsfig,amsthm,a4,amsmath,mathrsfs,url,
enumerate,enumitem,amsgen,amsmath,amstext,amsbsy,amsopn}
\usepackage{vmargin,amscd}
\usepackage[matrix,arrow,curve]{xy}
\newcommand{\dia}[1]{\begin{array}{c}{\xymatrix@C-3pt@M+2pt@R-4pt{#1 }}\end{array}}
\newtheorem{thm}{Theorem}[section]
\newtheorem{cor}[thm]{Corollary}
\newtheorem{lem}[thm]{Lemma}
\newtheorem{prop}[thm]{Proposition}
\newtheorem*{prop*}{Proposition}
\newtheorem{miles}[thm]{Milestone}
\theoremstyle{definition}
\newtheorem{defn}[thm]{Definition}
\newtheorem{dig}[thm]{Digression}
\newtheorem{exe}[thm]{Example}
\newtheorem{rem}[thm]{Remark}
\newtheorem{notat}[thm]{Notation}
\theoremstyle{remark}
\numberwithin{equation}{chapter}

\newcommand{\Conv}{\mathop{\scalebox{1.5}{\raisebox{-0.2ex}{$\ast$}}}}

\newcommand{\SL}{\textnormal{SL}}
\newcommand{\GL}{\textnormal{GL}}
\newcommand{\PSL}{\textnormal{PSL}}
\newcommand{\PGL}{\textnormal{PGL}}
\newcommand{\SO}{\textnormal{SO}}
\newcommand{\M}{\textnormal{M}}
\newcommand{\Isom}{\textnormal{Isom}}

\newcommand{\Z}{\mathbf{Z}}
\newcommand{\N}{\mathbf{N}}
\newcommand{\R}{\mathbf{R}}
\newcommand{\C}{\mathbf{C}}
\newcommand{\Q}{\mathbf{Q}}
\newcommand{\K}{\mathbf{K}}
\newcommand{\HNN}{\mathrm{HNN}}

\newcommand{\lp}{( \! (}
\newcommand{\rp}{) \! )}

\title{Metric geometry of locally compact groups}

\author{Yves Cornulier and Pierre de la Harpe}

\date{July 31,  2016}

\begin{document}

\maketitle

\chapter*{\centering \begin{normalsize}Abstract\end{normalsize}}
\label{fake_abstract}
\begin{quotation}
This book offers to study locally compact groups
from the point of view of appropriate metrics that can be defined on them,
in other words to study
``Infinite groups as geometric objects'',
as Gromov writes it in the title of a famous article.
The theme has often been restricted to finitely generated groups,
but it can favourably be played for locally compact groups.
\par

The development of the theory is illustrated by numerous examples, 
including 
matrix groups with entries in the the field of real or complex numbers, 
or other locally compact fields such as $p$-adic fields,
isometry groups of various metric spaces,
and, last but not least, discrete group themselves.
\par

Word metrics for compactly generated groups play a major role.
In the particular case of finitely generated groups,
they were introduced by Dehn around 1910 
in connection with the Word Problem.
\par

Some of the results exposed concern general locally compact groups,
such as criteria for the existence of compatible metrics
(Birkhoff-Kakutani, Kakutani-Kodaira, Struble).
Other results concern special classes of groups,
for example those mapping onto $\mathbf Z$
(the Bieri-Strebel splitting theorem, generalized to locally compact groups).
\par

Prior to their applications to groups, the basic notions of coarse and large-scale geometry
are developed in the general framework of metric spaces.
Coarse geometry is that part of geometry concerning properties of metric spaces
that can be formulated in terms of large distances only.
In particular coarse connectedness, coarse simple connectedness,
metric coarse equivalences, and quasi-isometries of metric spaces
are given special attention.
\par

The final chapters are devoted to
the more restricted class of compactly presented groups,
which generalize finitely presented groups to the locally compact setting. 
They can indeed be characterized 
as those compactly generated locally compact groups 
that are coarsely simply connected.
\end{quotation}

\vskip.2cm
\newpage

\noindent
\emph{2000 Mathematics Subject Classification.} 
\hfill\vskip.2cm
Primary:
20F65. 
Secondary:
20F05, 22D05, 51F99, 54E35, 57M07, 57T20.


\vskip.4cm

\noindent
\emph{Key words and phrases.}
\hfill\vskip.2cm
Locally compact groups, 
\hfill\par
left-invariant metrics, 
\hfill\par
$\sigma$-compactness, 
\hfill\par
second countability, 
\hfill\par
compact generation, 
\hfill\par
compact presentation, 
\hfill\par
metric coarse equivalence, 
\hfill\par
quasi-isometry, 
\hfill\par
coarse connectedness, 
\hfill\par
coarse simple connectedness, 
\hfill\par
growth, 
\hfill\par
amenability.

\vskip2cm

\noindent
Y.C. 
\hfill\vskip.2cm 
Laboratoire de  Math\'ematiques d'Orsay 
\hfill\par 
Universit\'e Paris-Sud 
\hfill\par 
CNRS 
\hfill\par 
Universit\'e de Paris-Saclay 
\hfill\par 
91405 Orsay 
\hfill\par  France 
\hfill\par 
yves.cornulier@math.u-psud.fr

\vskip.4cm

\noindent
P.H. 
\hfill\vskip.2cm 
Section de math\'ematiques 
\hfill\par 
Universit\'e de Gen\`eve 
\hfill\par 
C.P.~64 
\hfill\par 
CH--1211 Gen\`eve~4 
\hfill\par  
Suisse 
\hfill\par 
Pierre.delaHarpe@unige.ch

\tableofcontents

 \clearpage\phantom{a}\clearpage

\chapter{Introduction}
\label{chap_intro}

\section{Discrete groups as metric spaces}
\label{intro_discrete_metric}
Whenever a group $\Gamma$ appears in geometry,
which typically means that $\Gamma$ acts on a metric space of some sort 
(examples include universal covering spaces,
Cayley graphs and Rips complexes),
the geometry of the space reflects
some geometry of the group.
\par

This phenomenon goes back at least to Felix Klein and Henri Poincar\'e,
with tessellations of the half-plane 
related to subgroups of the modular groups, around 1880.
It has then been a well-established tradition
to study properties of groups which can be viewed, 
at least in retrospect, as geometric properties.
As a sample, we can mention: 
\begin{enumerate}[noitemsep]
\item[--]
``Dehn Gruppenbild'' (also known as Cayley graphs), 
used to picture finitely generated groups and their word metrics, 
in particular knot groups, around 1910.
Note that Dehn introduced word metrics for groups 
in his articles on decision problems (1910-1911).
\item[--]
Amenability of groups (von Neumann, Tarski, late 20's),
and its interpretation in terms of isoperimetric properties (F\o lner, mid 50's).
\item[--]
Properties ``at infinity'', 
or ends of groups (Freudenthal, early 30's),
and structure theorems for groups with two or infinitely many ends
(Stallings for finitely generated groups, late 60's, 
Abels' generalization for totally disconnected locally compact groups, 1974).
\item[--]
Lattices in Lie groups, and later in algebraic groups over local fields;
first a collection of examples, 
and from the 40's a subject of growing importance, 
with foundational work by Siegel, Mal'cev, Mostow, L.~Auslander, Borel \& Harish-Chandra, 
Weil, Garland, H.C.~Wang, Tamagawa, Kazhdan, Raghunathan, Margulis
(to quote only them);
leading to:
\item[--]
Rigidity of groups, and in particular of lattices in semisimple groups 
(Mostow, Margulis, 60's and 70's).
\item[--]
Growth of groups, introduced independently (!) 
by A.S.~Schwarz (also written \v Svarc) in 1955 and Milnor in 1968,
popularized by the work of Milnor and Wolf, 
and studied later by Grigorchuk, Gromov, and others,
including Guivarc'h, Jenkins and Losert for locally compact groups.
\item[--]
Structure of groups acting faithfully on trees
(Tits, Bass-Serre theory,
Dunwoody decompositions and accessibility of groups, 70's); tree lattices.
\item[--]
Properties related to random walks (Kesten, late 50's, Guivarc'h, 70's, Varopoulos).
\item[--]
And the tightly interwoven developments 
of combinatorial group theory and low dimensional topology,
from Dehn to Thurston, and so many others.
\end{enumerate}
From 1980 onwards, for all these reasons and under guidance of Gromov,
in particular of his articles \cite{Grom--81b, Grom--84, Grom--87, Grom--93},
the group community has been used to consider
a group (with appropriate conditions) as a metric space,
and to concentrate on large-scale properties of such metric spaces.

\vskip.2cm

Different classes of groups can be characterized 
by the existence of metrics with additional properties.
We often write \emph{discrete group} for \emph{group},
\index{Discrete! group}
in view of later sections about \emph{topological groups},
and especially \emph{locally compact groups}.
In the discrete setting, we distinguish four classes, 
each class properly containing the next one:
\par (all)
all discrete groups;
\par(ct)
countable groups;
\par(fg)
finitely generated groups;
\par(fp)
finitely presented groups.
\par\noindent
This will serve as a guideline below, in the more general setting
of locally compact groups.
\par

Every group $\Gamma$ has left-invariant metrics which induce the discrete topology,
for example that defined by $d(\gamma,\gamma') = 1$ 
whenever $\gamma,\gamma'$ are distinct. 
The three other classes can be characterized as follows.

\begin{prop}
\label{ctfgfpgroup}
Let $\Gamma$ be a group.
\begin{enumerate}[noitemsep,label=(\arabic*)]
\item[(ct)]\label{ctDEctfgfpgroup}
$\Gamma$ is countable if and only if it has a left-invariant metric
with finite balls.
Moreover, if $d_1, d_2$ are two such metrics,
the identity map $(\Gamma,d_1) \longrightarrow (\Gamma, d_2)$
is a metric coarse equivalence.
\end{enumerate}
Assume from now on that $\Gamma$ is countable.
\begin{enumerate}[noitemsep,label=(\arabic*)]
\item[(fg)]\label{fgDEctfgfpgroup}
$\Gamma$ is finitely generated if and only if,
for one (equivalently for every) metric $d$ as in (ct), 
the metric space $(\Gamma,d)$ is coarsely connected.
Moreover, a finitely generated group has 
a left-invariant large-scale geodesic metric with finite balls
(e.g.\ a word metric);
if $d_1, d_2$ are two such metrics,
the identity map $(\Gamma,d_1) \longrightarrow (\Gamma, d_2)$
is a quasi-isometry.
\item[(fp)]\label{fpDEctfgfpgroup}
$\Gamma$ is finitely presented if and only if,
for one (equivalently for every) metric $d$ as in (ct), the metric space $(\Gamma,d)$
is coarsely  simply connected.
\end{enumerate}
\end{prop}

The technical terms of the proposition can be defined as follows;
we come back to these notions in Sections \ref{coarselyLipschitzandlargescaleLipschitz},
\ref{coarselyconnectedetc}, and \ref{sectioncscpms}.
A metric space $(X,d)$ is
\begin{enumerate}[noitemsep]
\item[--]
\emph{coarsely connected} if there exists a constant $c > 0$ such that,
for every pair $(x,x')$ of points of $X$,
there exists a sequence $x_0=x, x_1, \hdots, x_n=x'$ of points in $X$
such that $d(x_{i-1},x_i) \le c$ for $i = 1, \hdots, n$,
\item[--]
\emph{large-scale geodesic} if there exist constants $a > 0$, $b \ge 0$
such that the previous condition holds with, moreover, $n \le a d(x,x') + b$,
\item[--] 
\emph{coarsely simply connected} if
every ``loop'' $x_0, x_1, \hdots, x_n=x_0$ of points in $X$
with an appropriate bound on the distances $d(x_{i-1}, x_i)$, 
can be ``deformed by small steps''  to a constant loop $x_0, x_0, \hdots, x_0$;
see  \ref{defcoarsely1con} for a precise definition.
\end{enumerate}
If $X$ and $Y$ are metric spaces,
a map $f : X \longrightarrow Y$ is
\begin{enumerate}[noitemsep]
\item[--]
a \textbf{metric coarse equivalence} if,
\begin{enumerate}[noitemsep]
\item[$\cdot$]
for every $c > 0$, there exists $C > 0$ such that,
\item[]
for $x,x' \in X$ with $d_X(x,x') \le c$,
we have $d_Y(f(x),f(x')) \le C$, 
\item[$\cdot$]
there exists $g : Y \longrightarrow X$
with the same property, satisfying
\item[]
$\sup_{x \in X} d_X(g(f(x)), x) < \infty$ and
$\sup_{y \in Y} d_Y(f(g(y)), y) < \infty$;
\end{enumerate}
\item[--]
a \textbf{quasi-isometry} if there exist $a > 0$, $b \ge 0$ such that
\begin{enumerate}[noitemsep]
\item[$\cdot$]
$d_Y(f(x), f(x')) \le ad_X(x,x') + b$ for all $x,x' \in X$, 
\item[$\cdot$]
there exists $g : Y \longrightarrow X$
with the same property, satisfying
\item[]
$\sup_{x \in X} d_X(g(f(x)), x) < \infty$ and
$\sup_{y \in Y} d_Y(f(g(y)), y) < \infty$.
\end{enumerate}
\end{enumerate}
Two metrics $d, d'$ on a set $X$ are coarsely equivalent [respectively quasi-isometric]
if the identity map $(X,d) \longrightarrow (X, d')$ 
is a metric coarse equivalence [resp. a quasi-isometry].
\par

The characterizations of Proposition \ref{ctfgfpgroup}
provide conceptual proofs of some basic and well-known facts.
Consider for example a countable group $\Gamma$, a subgroup of finite index $\Delta$,
a finite normal subgroup $N \lhd \Gamma$, 
and a left-invariant metric $d$ on $\Gamma$, with finite balls.
Coarse connectedness and coarse simple connectedness are properties
invariant by metric coarse equivalence.
A straightforward verification shows that the inclusion $\Delta \subset \Gamma$
is a metric coarse equivalence;
it follows that $\Delta$ is finitely generated (or finitely presented)
if and only if $\Gamma$ has the same property.
\par

It is desirable to have a similar argument for $\Gamma$ and $\Gamma / N$;
for this, it is better to rephrase the characterizations (ct), (fg), and (fp)
in terms of \emph{pseudo-metrics} rather than in terms of metrics.
``Pseudo'' means that the pseudo-metric evaluated on two distinct points can be $0$.
\par

It is straightforward to adapt to pseudo-metric spaces 
the technical terms defined above and Proposition \ref{ctfgfpgroup}.

\section{Discrete groups and locally compact groups}
\label{intro_discrete_LC}
It has long been known that the study of a group $\Gamma$
can be eased when it sits as a discrete subgroup of some kind 
in a locally compact group $G$.
\par

For instance, a cocompact discrete subgroup in a connected locally compact group
is finitely generated (Propositions \ref{powersSincpgroup} and \ref{stababcd}).
The following two standard examples, beyond the scope of the present book,
involve a lattice $\Gamma$ in a locally compact group $G$:
first, Kazhdan Property (T) is inherited from $G$ to $\Gamma$ \cite{BeHV--08};
second, if $\Gamma$ is moreover cocompact in $G$,
cohomological properties of $\Gamma$ can be deduced
from information on $G$ or on its homogeneous spaces
\cite{Brow--82, Serr--71}.
\par

Other examples of groups $\Gamma$ 
that are usefully seen as discrete subgroups of $G$ include
finitely generated torsion-free nilpotent groups,
which are discrete cocompact subgroups in simply connected nilpotent Lie groups
 \cite[Theorem 2.18]{Ragh--72},
 and polycyclic groups,
 in which there are subgroups of finite index
that are discrete cocompact subgroups in simply connected solvable Lie groups
 \cite[Theorem 4.28]{Ragh--72}.
For some classes of solvable groups,
the appropriate ambient group $G$ is not Lie, 
but a group involving a product of linear groups over non-discrete locally compact fields.
For example, for a prime~$p$,
the group $\Z [1/p]$ of rational numbers of the form $a/p^k$ (with $a \in \Z$ and $k \ge 0$), 
and the $p$-adic field $\Q_p$,
we have the diagonal embedding
\begin{equation*}
\Z [1/p] \, \lhook\joinrel\relbar\joinrel\rightarrow \, \R \times \Q_p , 
\end{equation*}
of which the image is a discrete subgroup with compact quotient.
We refer to Example \ref{abcdefg} for other examples, of the kind
\begin{equation*}
\Z [1/6] \rtimes_{1/6} \Z \, \lhook\joinrel\relbar\joinrel\rightarrow \, 
(\R \times \Q_2 \times \Q_3) \rtimes_{1/6} \Z , 
\end{equation*}
of which the image is again discrete with compact quotient.
\par

It thus appears that the natural setting is that of locally compact groups.
(This generalization from Lie groups to locally compact groups
is more familiar in harmonic analysis than in geometry.)
More precisely, for the geometric study of groups,
it is appropriate to consider 
\emph{$\sigma$-compact} and  \emph{compactly generated locally compact groups};
in the case of discrete groups, 
$\sigma$-compact groups are countable groups,
and compact generation reduces to finite generation.
Though it is not so well-known, 
there is a stronger notion of \emph{compact presentation} for locally compact groups,
reducing to finite presentation in the case of discrete groups.
One of our aims is to expose basic facts involving these properties.

We use \textbf{LC-group} as a shorthand for ``locally compact group''.

\section{Three conditions on LC-groups}
\label{intro_conditionsOn_LC}
Locally compact groups came to light in the first half of XXth century.
The notion of compactness slowly emerged from 19th century analysis,
and the term was coined by Fr\'echet in 1906; see \cite[Page 136]{Enge--89}.
Local compactness for spaces was introduced by Pavel Alexandrov in 1923
\cite[Page 155]{Enge--89}.
The first abstract definition of a topological group seems to be that of Schreier \cite{Schr--25};
early contributions to the general theory, from the period 1925--1940, include articles by
Leja, van Kampen, Markoff,  Freudenthal, Weyl, von Neumann
(these are quoted in \cite[chapitre 1]{Weil--40}),
as well as van Dantzig, Garrett Birkhoff, and Kakutani.
Influential books were published by Pontryagin \cite{Pont--39} and Weil \cite{Weil--40},
and later by Montgomery \& Zippin \cite{MoZi--55}.
\par

   (Lie groups could be mentioned, but with some care. Indeed, in the early theory
of XIXth century,  ``Lie groups'' are local objects,
not systematically distinguished from Lie algebras before Weyl,
even if some examples are ``global'', i.e., are topological groups in our sense.
See the discussion in \cite{Bore--01}, in particular his $\S$~I.3.)
\par

Among topological groups, LC-groups play a central role,
both in most of what follows and in other contexts,
such as ergodic theory and representation theory.
Recall that these groups have left-invariant regular Borel measures,
as shown by Haar (1933) in the second-countable case, 
and by Kakutani and Weil (late 30's) in the general case.
Conversely, a group with a ``Haar measure'' is locally compact;
see \emph{La r\'eciproque du th\'eor\`eme de Haar}, Appendice I in \cite{Weil--40}, 
and its sharpening in \cite{Mack--57}; see also Appendix B in \cite{GlTW--05}.
Gelfand and Raikov (1943) showed that 
LC-groups have ``sufficiently many'' 
irreducible continuous unitary representations
\cite[Corollary 13.6.6]{Dixm--69};
this does \emph{not} carry over to topological groups  
(examples of topological groups that are abelian, locally homeomorphic
to Banach spaces, and without any non-trivial continuous unitary representations
are given in \cite{Bana--83, Bana--91}).
\par

Let $G$ be an LC-group.
Denote by $G_0$ its identity component,
which is a normal closed subgroup;
$G_0$ is connected and
the quotient group $G/G_0$ is totally disconnected.
Our understanding of connected LC-groups,
or more generally of LC-groups $G$ with $G/G_0$ compact,
has significantly increased with the solution of Hilbert Fifth Problem
in the early 1950's (Gleason, Montgomery, Zippin, Yamabe, see \cite{MoZi--55}).
The seminal work of Willis 
(\cite{Will--94}, see also \cite{Will--01a, Will--01b})
on dynamics of automorphisms of totally disconnected LC-groups 
allowed further progress.
Special attention has been given on normal subgroups and topologically simple
totally disconnected LC-groups
\cite{Will--07, CaMo--11, CaRW--I, CaRW--II}.

\vskip.2cm

The goal of this book is to revisit
three finiteness conditions on LC-groups,
three natural generalizations of countability, finite generation, and finite presentation.
\par

The first two,
$\sigma$-compactness and compact generation,
are widely recognized as fundamental conditions in various contexts.
The third, compact presentation, 
was introduced and studied by the German school in the 60's 
(see Section \ref{intro_compactpres} below),
but mainly disappeared from the landscape until recently;
we hope to convince the reader of its interest.
For an LC-group $G$, here are these three conditions:
\begin{enumerate}[noitemsep]
\item[($\sigma$C)]
$G$ is \textbf{$\sigma$-compact} if it has a countable cover by compact subsets.
In analysis, this condition ensures that 
a Haar measure on $G$ is $\sigma$-finite,
so that the Fubini theorem is available.
%
\item[(CG)]
$G$ is \textbf{compactly generated} if it has a compact generating set $S$.
\item[(CP)]
$G$ is \textbf{compactly presented} if it has a presentation $\langle S \mid R \rangle$
with the generating set $S$ compact in $G$ 
and the relators in $R$ of bounded length.
\end{enumerate}
\par

Though it does not have the same relevance for the geometry of groups,
it is sometimes useful to consider one more condition:
$G$ is \emph{second-countable} if its topology has a countable basis;
the notion (for LC-spaces) goes back 
to Hausdorff's 1914 book \cite[Page 20]{Enge--89}.
\par

There is an abundance of groups that satisfy these conditions:
\begin{enumerate}[noitemsep]
\item[(CP)]
Compactly presented LC-groups include
connected-by-compact groups,
abelian and nilpotent compactly generated groups,
and reductive algebraic groups over local fields.
(Local fields are $p$-adic fields $\Q_p$ and their finite extensions,
and fields $\mathbf F_q \lp t \rp$ of formal Laurent series over finite fields.)
\item[(CG)]
Examples of compactly generated groups that are not compactly presented include
\begin{enumerate}[noitemsep,label=(\alph*)]
\item\label{aDECGpage11}
$\K^2 \rtimes \SL_2(\K)$, where $\K$ is a local field (Example \ref{excgnotcp}), 
\item\label{bDECGpage11}
$(\R \times \Q_2 \times \Q_3) \rtimes_{2/3} \Z$,
where the generator of $\Z$ acts by multiplication by $2/3$ on each of
$\R$, $\Q_2$, and $\Q_3$ (Example \ref{abcdefg}).
\end{enumerate}
\index{Affine group! $\K^2 \rtimes \SL_2(\K)$}
\item[($\sigma$C)]
$\GL_n(\K)$ and its closed subgroups
are second-countable $\sigma$-compact LC-groups,
for every non-discrete locally compact field $\K$ and every positive integer $n$.
\end{enumerate}
\par

The condition of $\sigma$-compactness rules out 
uncountable groups with the discrete topology,
and more generally LC-groups $G$ with an open subgroup $H$
such that the homogeneous space $G/H$ is uncountable;
see Remark \ref{neednotbelc}(3), 
Example \ref{non-sigma-compact-by-discrete},
Corollaries \ref{opensubgroupsinLCgroups} \& \ref{vanDantzigCor}(1), 
and Example \ref{ExTopsub}(1).
Among $\sigma$-compact groups, second countability rules out
``very large'' compact groups, such as uncountable direct products
of non-trivial finite groups.
\par

It is remarkable that each of the above conditions is equivalent to a metric condition,
as we now describe more precisely.

\section[Metric characterization]
{Metric characterization of topological properties of LC-groups}
\label{intro_charact_LCgroups}
A metric on a topological space is \emph{compatible} \index{Compatible! compatible metric}
if it defines the given topology.
It is appropriate to relax the condition of compatibility for metrics on topological groups,
for at least two reasons.
\par

On the one hand,
a $\sigma$-compact LC-group $G$ need not be metrizable
(as uncountable products of non-trivial finite groups show).
However, the Kakutani-Kodaira theorem (Theorem \ref{KK}) establishes that
there exists a compact normal subgroup $K$ such that $G/K$
has a left-invariant compatible metric $d_{G/K}$;
hence $d_G(g,g') := d_{G/K}(gK, g'K)$ defines a natural
\emph{pseudo-metric} $d_G$ on $G$,
with respect to which balls are compact.
On the other hand,
an LC-group $G$ generated by a compact subset $S$
has a \emph{word metric} $d_S$ defined by
\begin{equation*}
d_S(g,g') \, = \, \min \left\{ n \in \N \hskip.1cm \bigg\vert \hskip.1cm 
\aligned
& \text{there exist} \hskip.2cm
s_1, \hdots, s_n \in S \cup S^{-1} \hskip.2cm 
\\
& \text{such that} \hskip.2cm
g^{-1}g' = s_1 \cdots s_n 
\endaligned
\right\} .
\end{equation*}
Note that $d_S$ need not be continuous
as a function $G \times G \longrightarrow \R_+$.
For example, on the additive group of real numbers, we have 
$d_{\mathopen[0,1\mathclose]}(0,x) = \lceil \vert x \vert \rceil
:= \min \{ n \ge 0  \mid n \ge \vert x \vert \}$
for all $x \in \R$; hence
$1 = d_{\mathopen[0,1\mathclose]}(0,1) \ne 
\lim_{\varepsilon \to 0} d_{\mathopen[0,1\mathclose]}(0,1+\varepsilon) = 2$
(with $\varepsilon > 0$ in the limit).
\par

As a consequence, we consider
\textbf{non-necessarily continuous pseudo-metrics},
rather than compatible metrics.
In this context, it is convenient to introduce some terminology.
\par

A pseudo-metric $d$ on a topological space $X$
is \emph{proper} if its balls are relatively compact,
and \emph{locally bounded} if every point in $X$ has a neighbourhood
of finite diameter with respect to $d$.
A pseudo-metric on a topological group is \emph{adapted} 
if it is left-invariant, proper, and locally bounded;
it is \emph{geodesically adapted} if it is adapted
and large-scale geodesic.
Basic examples of geodesically adapted metrics are
the word metrics with respect to compact generating sets.
On an LC-group, a continuous adapted metric is compatible,
by Proposition \ref{contpropermet}.
\par

Every LC-group has a left-invariant pseudo-metric with respect to which
balls of \emph{small enough} radius are compact (Corollary \ref{CorollaryKK}).
Classes ($\sigma$C), (CG) and (CP) can be characterized as follows.

\begin{prop}[characterization of $\sigma$-compact groups]
\label{1.D.1}
An LC-group $G$ is $\sigma$-compact 
if and only if it has an adapted metric,
if and only if it has an adapted pseudo-metric,
if and only if it has an adapted continuous pseudo-metric
(Proposition \ref{existenceam}).
\par
If $G$ has these properties,
two adapted pseudo-metrics on $G$ are coarsely equivalent 
(Corollary \ref{2metricsce}).
\end{prop}

In particular, a $\sigma$-compact group 
can be seen as a pseudo-metric space,
well-defined up to metric coarse equivalence.
See also Corollary \ref{coarselypseudometric=sigmacompact}.
\par

The main ingredients of the proof are standard results:
the Birkhoff-Kakutani theorem, 
which characterizes topological groups that are metrizable,
the Struble theorem, 
which characterizes locally compact groups
of which the topology can be defined by a proper metric,
and the Kakutani-Kodaira theorem,
which establishes that $\sigma$-compact groups
are compact-by-metrizable
(Theorems \ref{BK}, \ref{Struble}, and \ref{KK}).

\begin{prop}[characterization of compactly generated groups]
\label{1.D.2}
Let $G$ be a  $\sigma$-compact LC-group
and $d$ an adapted pseudo-metric on $G$.
\par
$G$ is compactly generated if and only if 
the pseudo-metric space $(G,d)$ is coarsely connected.
Moreover, if this is so, there exists 
a geodesically adapted continuous pseudo-metric on $G$
(Proposition \ref{metric_char_cg}).
\par
If $G$ is compactly generated,
two geodesically adapted pseudo-metrics on $G$ are quasi-isometric
(Corollary \ref{uniqueness _uptoqi}).
\par
In particular, word metrics associated to compact generating sets of $G$
are bilipschitz equivalent to each other (Proposition  \ref{geodad+wordmetric}).
\end{prop}

Alternatively, an LC-group is compactly generated if and only if
it has a faithful geometric action (as defined in \ref{actions_mp_lb_cob})
on a non-empty geodesic pseudo-metric space (Corollary \ref{ahahah}).
In particular a compactly generated group can be seen as a pseudo-metric space,
well-defined up to quasi-isometry.

\vskip.2cm

To obtain characterizations with metrics, 
rather than pseudo-metrics,
second-countability is needed:
\vskip.2cm
\emph{
An LC-group $G$ is second-countable if and only if
it has a left-invariant proper compatible metric
(Struble Theorem \ref{Struble}).
}
\par
\emph{
An LC-group $G$ is second-countable and compactly generated if and only if
$G$ has a large-scale geodesic left-invariant proper compatible metric
(Proposition \ref{2countcggroup}).
}

\vskip.2cm

It is a crucial fact for our exposition that compact presentability can be characterized
in terms of adapted pseudo-metrics:

\begin{prop}[characterization of compactly presented groups]
\label{1.D.3}
Let $G$ be a compactly generated LC-group
and $d$ an adapted pseudo-metric on $G$.
\par
$G$ is compactly presented if and only if 
the pseudo-metric space $(G,d)$ is coarsely simply connected
(Proposition \ref{cp iff csc}).
\end{prop}

One of the motivations for introducing metric ideas as above 
appears in the following proposition (see Section \ref{sectionactions}),
that extends the discussion at the end of Section \ref{intro_discrete_metric}:

\begin{prop}
\label{1.D.4}
Let $G$ be a $\sigma$-compact LC-group, 
$H$ a closed subgroup of $G$ such that the quotient space $G/H$ is compact,
and $K$ a compact normal subgroup of $G$.
Then the inclusion map $i : H \lhook\joinrel\relbar\joinrel\rightarrow G$ 
and the projection $p : G \twoheadrightarrow G/K$
are metric coarse equivalences with respect to adapted pseudo-metrics on the groups.
Moreover, if $G$ is compactly generated and the pseudo-metrics are geodesically adapted,
the maps $i$ and $p$ are quasi-isometries.
\par

In particular, if $\Gamma$ is a discrete subgroup of $G$ such that $G/\Gamma$ is compact,
$G$ is compactly generated [respectively compactly presented]
if and only if $\Gamma$ is finitely generated [resp.\ finitely presented].
\end{prop}

For other properties that hold (or not) simultaneously for $G$ and $\Gamma$,
see Remarks \ref{PropertyPinvce} and \ref{PropertyQinvqi}.

\section{On compact presentations} 
\label{intro_compactpres}

Despite its modest fame for non-discrete groups, 
compact presentability has been used as a tool
to establish finite presentability of $S$-arithmetic groups.
Consider an algebraic group $\mathbf G$ defined over $\Q$.
By results of Borel and Harish-Chandra, 
the group $\mathbf G (\Z)$ is finitely presented \cite{Bore--62}.
Let $S$ be a finite set of primes, 
and let $\Z_S$ denote the ring of rational numbers
with denominators products of primes in $S$.
It is a natural question to ask whether $\mathbf G(\Z_S)$ is finitely presented;
for classical groups, partial answers were given in \cite{Behr--62}.
\par

The group $\mathbf G(\Z_S)$ is naturally a discrete subgroup in
$\mathcal G := \mathbf G(\R) \times \prod_{p \in S} \mathbf G(\Q_p)$.
In \cite{Knes--64}, Martin Kneser has introduced the notion of compact presentability,
and has shown that $\mathbf G(\Z_S)$ is finitely presented
if and only if $\mathbf G(\Q_p)$ is compactly presented for all $p \in S$;
an easy case, that for which $\mathcal G / \mathbf G(\Z_S)$ is compact,
follows from Corollary \ref{hereditaritycp} below.
Building on results of Bruhat and Tits,
Behr has then shown that, when $\mathbf G$ is reductive, 
$\mathbf G (\K)$ is compactly presented for every local field $\K$
\cite{Behr--67, Behr--69}.
To sum up, using compact presentability, Kneser and Behr have shown
that $\mathbf G (\Z_S)$ is finitely presented 
for every reductive group $\mathbf G$ defined over $\Q$
and every finite set $S$ of primes.
(Further finiteness conditions for such groups
are discussed in several articles by Borel and Serre; we quote \cite{BoSe--76}.)
\par

After these articles of Kneser and Behr from the 60's,
Abels discussed compact presentation in great detail
for solvable linear algebraic groups \cite{Abel--87},
showing in particular that several properties of $\mathbf G(\Z [1/p])$
are best understood together with those of $\mathbf G(\Q_p)$.
Otherwise, compact presentations seem to have disappeared from the literature.
In his influential article
on group cohomology and properties of arithmetic and $S$-arithmetic groups,
Serre does not cite Kneser, and he cites \cite{Behr--62, Behr--69} only very briefly
 \cite[in particular Page 127]{Serr--71}.

\section{Outline of the book}
\label{intro_outline}

Chapter \ref{chap_topaspects}
contains foundational facts on LC-spaces and groups,
the theorems of Birkhoff-Kakutani, Struble, and Kakutani-Kodaira, 
on metrizable groups, with proofs,
and generalities on compactly generated LC-groups.
The last section describes results on the structure of LC-groups;
they include a theorem from the 30's 
on compact open subgroups in totally disconnected LC-groups,
due to van Dantzig,
and (without proofs) results from the early 50's solving the Hilbert Fifth Problem, 
due among others to Gleason, Montgomery-Zippin, and Yamabe.
\par

Chapter \ref{chap_metriccoarse}
deals with two categories of pseudo-metric spaces that play a major role
in our setting:
the metric coarse category, 
in which isomorphisms are closeness classes of metric coarse equivalences
(the category well-adapted to $\sigma$-compact LC-groups),
and the large-scale category,
in which isomorphisms are closeness classes of quasi-isometries
(the category well-adapted to compactly generated LC-groups).
Section \ref{sectionmetriclattices} shows how pseudo-metric spaces
can be described in terms of their metric lattices,
i.e., of their subsets which are both uniformly discrete and cobounded. 
Section \ref{coarselyproperandgrowth} illustrates these notions
by a discussion of notions of growth and amenability
in appropriate spaces and groups.
Section \ref{sectioncoarsecat}, on what we call the coarse category,
alludes to a possible variation, involving bornologies rather than pseudo-metrics.

Chapter \ref{chap_gpspmspaces}
shows how the metric notions of Chapter \ref{chap_metriccoarse}
apply to LC-groups.
In particular, every $\sigma$-compact LC-group has an adapted metric
(Proposition \ref{1.D.1}),
and every compactly generated LC-group has a geodesically adapted pseudo-metric
(Proposition \ref{1.D.2}).
Moreover, $\sigma$-compact LC-groups are precisely the LC-groups
that can act on metric spaces in a ``geometric'' way, namely by actions
that are isometric, cobounded, locally bounded, and metrically proper;
and compactly generated LC-groups are precisely the LC-groups
that can act geometrically on coarsely connected pseudo-metric spaces
(Theorem \ref{ftggt}, 
sometimes called the fundamental theorem of geometric group theory).
Section \ref{local_ellipticity} illustrates these notions by discussing
locally elliptic groups, namely LC-groups in which every compact subset
is contained in a compact open subgroup
(equivalently LC-groups $G$ with an adapted metric $d$
such that the metric space $(G,d)$ has asymptotic dimension $0$).
\par

Chapter \ref{chap_excglcg} contains essentially examples
of compactly generated LC-groups,
including isometry groups of various spaces.
\par

Chapter \ref{chap_coarsely1conn} deals with the appropriate notion of simple connectedness
for pseudo-metric spaces, called coarse simple connectedness.
In Section 6.C, we introduce the ($2$-skeleton of the) Rips complex
of a pseudo-metric space.
\par

Chapter \ref{chap_boundedpresentation} introduces bounded presentation, 
i.e., presentations $\langle S \mid R \rangle$ with arbitrary generating set $S$
and relators in $R$ of bounded length.
It is a technical interlude before Chapter \ref{chap_cpgroups},
on compactly presented groups, i.e., on bounded presentations $\langle S \mid R \rangle$
of topological groups $G$ with $S$ compact in $G$.
\par

As explained in Chapter \ref{chap_cpgroups}, an LC-group is compactly presented
if and only if the pseudo-metric space $(G,d)$ is coarsely simply connected,
for $d$ an adapted pseudo-metric.
Important examples of compactly presented LC-groups include:
\begin{enumerate}[noitemsep,label=(\alph*)]
\addtocounter{enumi}{2}
\item
connected-by-compact groups,
\item
abelian and nilpotent compactly generated LC-groups,
\item
LC-groups of polynomial growth,
\item
Gromov-hyperbolic LC-groups.
\item
$(\R \times \Q_2 \times \Q_3) \rtimes_{1/6} \Z$
(compare with \ref{bDECGpage11} in Section \ref{intro_conditionsOn_LC}), 
\item
$\SL_n(\K)$, for every $n \ge 2$ and every local field $\K$.
\item
every reductive group over a non-discrete LC-field is compactly presented
(this last fact is not proven in this book).
\end{enumerate}
(Items \ref{aDECGpage11} and \ref{bDECGpage11} of the list appear above,
near the end of Section \ref{intro_conditionsOn_LC},
and refer to LC-groups which are \emph{not} compactly presented.)
A large part of Chapter \ref{chap_cpgroups}
is devoted to the Bieri-Strebel splitting theorem:
let $G$ be a compactly presented LC-group such that there exists
a continuous surjective homomorphism $\pi : G \twoheadrightarrow \Z$; 
then $G$ is isomorphic to an HNN-extension $\HNN (H, K, L, \varphi)$
of which the base group $H$ is a compactly generated subgroup of $\ker (\pi)$.
Among the prerequisites, there is a section exposing how to extend
the elementary theory of HNN-extensions and free products with amalgamation
from the usual setting of abstract groups to our setting of LC-groups 
(Section \ref{sectionamalgamHNN}).

\vskip.2cm

For another and much shorter presentation of the subject of this book,
see \cite{CoHa}.

\section{Acknowledgements}
\label{intro_acknow}

It is our pleasure to thank
Bachir Bekka, Pierre-Emmanuel Caprace, Rostislav Grigorchuk, 
Dominik Gruber, Adrien Le Boudec, Ashot Minasyan,
Vladimir Pestov, John Roe, and Romain Tessera,  
for useful and pleasant discussions and email exchanges. 

\vskip.2cm

The first-named author is grateful to 
the Swiss National Science Foundation for financial support,
as well as to the Section de math\'ematiques de l'Universit\'e de Gen\`eve
and the Chair EGG of the \'Ecole polytechnique f\'ed\'erale de Lausanne,
where part of this work was done.

\chapter{Basic properties}
\label{chap_topaspects}

\noindent
It is convenient to agree on the following \textbf{standing assumptions, notation, and abbreviations}: 
\begin{itemize}
\item[(A1)]
\label{A1}
Topological spaces and in particular topological groups appearing below 
are assumed to be \textbf{Hausdorff},
unless explicitly stated otherwise.
\index{Hausdorff! topological space, group}
\item[(A2)]
We denote by  
$\N$ the set of natural numbers $\{0, 1, 2, \hdots \}$, including $0$,
\index{$n$@$\N$, set of natural numbers, including zero}
\begin{itemize}
\item[]
$\Z$ the ring of rational integers,
\index{$z$@$\Z$, ring of rational integers}
\item[]
$\R$ the field of real numbers,
\index{$ra$@$\R$, field of real numbers}
\item[]
$\R_+$ the subset of non-negative real numbers,
\index{$rb$@$\R_+$, set of non-negative real numbers}
\item[]
$\R_+^\times$ the group of positive real numbers,
\index{$rc$@$\R_+^\times$, group of positive real numbers}
\item[]
$\overline \R_+$ the union $\R_+ \cup \{\infty\}$,
\index{$rd$@$\overline \R_+ = \R_+ \cup \{\infty\}$}
\item[]
$\C$ the field of complex numbers,
\index{$c$@$\C$, field of complex numbers}
\item[]
$\Z_p$ the ring of $p$-adic integers,
\index{$z$@$\Z_p$, ring of $p$-adic integers}
\item[]
$\Q_p = \Z_p[\frac{1}{p}]$ the field of $p$-adic numbers. 
\index{$q$@$\Q_p$, field of $p$-adic numbers}
\end{itemize}
If $R$ is a commutative ring with unit, 
$R^\times$ stands for its multiplicative group of units.
For example, $\R^\times = \R \smallsetminus \{0\}$
and $\Z_p^\times = \Z_p \smallsetminus p\Z_p$.
\item[(A3)]
For subsets $S,T$ of a group $G$, 
\begin{equation*}
ST \, = \,  \{g \in G \mid g=st \hskip.2cm \text{for some} \hskip.2cm s \in S, t \in T \} .
\end{equation*}
For $n \in \Z$, define $S^n$ by
$S^0 = \{1\}$, $S^1 = S$, $S^2 = SS$, $S^{n+1} = S^nS$ for $n \ge 2$,
$S^{-1} = \{g \in G \mid  g^{-1} \in S \}$,
and $S^{-n} = (S^{-1})^n$ for $n \ge 2$.
\item[(A4)]
Let $\mathcal P$ and $\mathcal R$ be properties of topological groups.
A group $G$ is 
\index{$p$@$\mathcal P  \text{-by-}  \mathcal R$ property of a group}
\begin{equation*}
\mathcal P  \text{-by-}  \mathcal R
\end{equation*}
if it has a closed normal subgroup $N$ with Property $\mathcal P$
such that the quotient group $G/N$ has Property $\mathcal R$.
If $\mathcal P$ is a property inherited by closed subgroups of finite index
and $\mathcal F$ is the property of being finite,
a $\mathcal P \text{-by-} \mathcal F$ group is
also called a virtually $\mathcal P$ group.
\index{Virtually $\mathcal P$ group|textbf}
\item[(A5)]
As already agreed above (end of \ref{intro_discrete_LC}),
we use the shorthand \textbf{LC-group}  \index{LC-group|textbf}
for ``locally compact topological group''. 
Similarly for \textbf{LC-space}.
\end{itemize}
Moreover, we wish to insist from the start on our use of the words
``metric'' and ``action'': 
\begin{itemize}
\item
On a topological space $X$, a \textbf{metric} $d$
need be neither continuous nor proper;
in particular, $d$ need not define the topology of $X$
(Remark \ref{remonmetrics}(1)).
\item
We will never use ``metric'' as a shorthand for ``Riemannian metric''.
\item
An LC-group with an adapted metric (as defined in \ref{metricadapted})
need not be metrizable
(Remark \ref{remoncontadapmetrics}).
\item
On a topological space or a metric space, an \textbf{action} of a topological group
need not be continuous
(Remark \ref{firstexamplesgeometricactions}).
\end{itemize}

\section[Topological spaces and pseudo-metric spaces]
{A reminder on topological spaces and pseudo-metric spaces}
\label{remindertspm}
Before dealing with groups, we recall some basic facts 
in the larger context of spaces.

\begin{defn}
\label{deftop}
A topological space is 
\textbf{locally compact} \index{Locally compact space|textbf}
if each of its points has a compact neighbourhood,
\textbf{$\sigma$-compact} 
\index{Sigma-compact! topological space|textbf}   
if it is a countable union of compact subspaces,
\textbf{first-countable} 
\index{First-countable! topological space|textbf}
if each of its points has a countable basis of open neighbourhoods,
\textbf{second-countable} \index{Second-countable! topological space|textbf}
if its topology has a countable basis of open sets,
and \textbf{separable} \index{Separable topological space|textbf}
if it contains a countable dense subset.  
\par
A subset of a topological space is \textbf{relatively compact} 
\index{Relatively compact subspace|textbf}
if its closure is compact.
\end{defn}

\begin{rem}
\label{neednotbelc}
(1)
In this book, $\sigma$-compact spaces  \emph{need not} be locally compact
(contrary to the convention in some books or articles,
such as  \cite{BTG1-4} and \cite{Dugu--66}).
\par

Similarly, our compactly generated groups \emph{need not} be locally compact.
See Definition \ref{cggroups} and Digression \ref{PestovEtc} below.

\vskip.2cm

(2)
An LC-space $X$ is $\sigma$-compact 
if and only if there exists a countable family $(K_n)_{n \ge 0}$
of compact subspaces whose \emph{interiors} cover $X$.
One may add the requirement that
$K_n \subset \operatorname{int}(K_{n+1})$ for all $n \ge 0$
\cite[Page I.68]{BTG1-4}.

\vskip.2cm

(3)
Though most LC-groups occurring in this book are $\sigma$-compact,
some groups of interest in geometry are not.
\par
One example is the discrete group $\R \otimes_{\Z} (\R / \pi \Q)$,
where $\otimes_{\Z}$ denotes a tensor product of $\Z$-modules;
this group is the receptacle of the Dehn-Sydler invariant,
in the theory of equidecomposition of polyhedra in Euclidean $3$-space
(see the expositions in  \cite{Bolt--78} and \cite{Cart--86}).
\par
Connected Lie groups endowed with the discrete topology
appear in the theory of foliations (among other places); 
see for example \cite{Haef--73} and \cite{Miln--83}.
\par
For other examples of LC-groups that are not $\sigma$-compact,
see Examples \ref{non-sigma-compact-by-discrete}
and \ref{ExTopsub}(1),  and Corollary \ref{vanDantzigCor}(1).
\index{Sigma-compact! LC-group, non-sigma-compact}

\vskip.2cm

(4)
A second-countable topological space is first-countable.
An uncountable discrete space is first-countable but is not second-countable.

\vskip.2cm

(5)
A second-countable space is always separable,
but the converse does not hold 
(even for compact groups, as Example \ref{sepnot2nd} shows).

\vskip.2cm

(6)
Though most LC-groups occurring in this book are second-countable,
some are not: 
see Examples \ref{sigmacandmetrizability}, \ref{sepnot2nd}, 
and Remark \ref{Gexplescgsigmacsc}.
\par

It is convenient to assume that LC-groups are second-countable in several occasions,
e.g.\ in ergodic theory (see \cite{Zimm--84})
or in the theory of unitary representations (reduction theory, direct integrals,
see \cite[Section 18.7]{Dixm--69}).

\vskip.2cm

(7)
For a discrete space, the following four properties are equivalent:
$X$ is $\sigma$-compact, $X$ is second-countable, 
$X$ is separable, $X$ is \textbf{countable}. 
\index{Countable space|textbf} 
\par
In this book, a countable set can be either finite or infinite.
\end{rem}

\begin{defn}
\label{properandlb}
If $X,Y$ are topological spaces, a map $f : X \longrightarrow Y$
(not necessarily continuous) is
\textbf{proper} \index{Proper! map|textbf}
if $f^{-1}(L)$ is relatively compact in $X$
for every compact subspace $L$ of $Y$.
\end{defn}

The adjective \textbf{proper} is used in several other situations:
for pseudo-metrics and for metric spaces (\ref{defproper2} below),
for actions (\ref{actions_mp_lb_cob}),
and for subsets 
(a subset $A$ of a set $X$ is proper if the complement $X \smallsetminus A$ is non-empty).

\begin{exe}
\label{properandlbex}
In the situation of Definition \ref{properandlb},
$f^{-1}(L)$ need not be closed,
because $f$ is not necessarily continuous.
\par

An example of a proper non-continuous map is that 
of the \textbf{floor function} \index{Floor function|textbf}
$\R \longrightarrow \Z$,
mapping $x \in \R$ to the largest integer $\lfloor x\rfloor \in \Z$ not greater than $x$. 
\end{exe}

\begin{defn}
\label{pm...metrizable}
A \textbf{pseudo-metric}
on a set $X$ is a map 
$d : X \times X \longrightarrow \R_+$
such that 
$d(x,x) = 0$, 
$d(x',x) = d(x,x')$, and $d(x,x'') \le d(x,x') + d(x',x'')$
for all $x,x',x'' \in X$.
An \textbf{\'ecart} \index{Ecart|textbf} on $X$ is a map on $X \times X$
with values in $\overline{\R}_+ = [0,\infty]$,
that formally satisfies the same identities and inequalities
(in this text, this will only appear in \S~\ref{sectionRips}).
\par

A  \textbf{pseudo-metric space} \index{Pseudo-metric space|textbf}
is a pair $(X,d)$
of a set $X$ and a pseudo-metric $d$ on $X$.
If convenient, we write $d_X$ for $d$.
Whenever possible, we write $X$ rather than $(X,d)$ or $(X,d_X)$.
For $x \in X$ and $A \subset X$, we define
\begin{equation*}
d_X(x,A) \, = \,  \inf \{ d_X(x,a) \mid a \in A \} .
\end{equation*}
Note that $d_X(x,\emptyset) = \infty$, 
in agreement with the natural convention $\inf \emptyset = \infty$.
The  \textbf{closed ball} \index{Closed ball|textbf} of radius $r > 0$ and centre $x \in X$ is
\begin{equation*}
B^x_X(r) \, = \, \{y \in X \mid d_X(x,y) \le r \} .
\end{equation*}
\textbf{Open balls} are defined similarly, with
the strict inequality $d_X(x,y) < r$.
The \textbf{diameter} 
of a subset $Y$ of a pseudo-metric space $(X,d_X)$
is $\operatorname{diam}(Y) = \sup_{y,y' \in Y} d_X(y,y')$.
\index{$b$@$B^{x}_X(r)$, closed ball of centre $x$ and radius $r$ in $X$}
\index{Open ball|textbf} 
\index{Diameter|textbf} 
\par

A  \textbf{metric} \index{Metric on a set|textbf} on a set $X$ is a pseudo-metric $d$ such that
$d(x,x')=0$ implies $x=x'$.
A  \textbf{metric space} \index{Metric space|textbf} is a pair $(X,d)$
of a set $X$ and a metric $d$ on $X$. Often, $X$ stands for $(X,d)$.

A metric space is canonically a topological space,
with open sets defined to be unions of open balls.
On a topological space, a metric is \textbf{compatible} 
\index{Compatible! compatible metric|textbf}
if it defines the original topology.
A topological space $X$ is \textbf{metrizable} \index{Metrizable! topological space|textbf}
if it has a compatible metric,
and \textbf{completely metrizable} \index{Completely metrizable topological space|textbf}
if it has a compatible metric for which it is a complete metric space.
A topological space is \textbf{Polish} \index{Polish! space|textbf}
if it is completely metrizable and separable.
\end{defn}

\begin{rem}
\label{abaslestoppseudometriques}
In compliance with the convention to only allow Hausdorff topological spaces
(see (A1)  Page \pageref{A1}), 
we do not consider any topology associated to a pseudo-metric space.
\end{rem}

\begin{defn}
\label{defproper2}
On a topological space $X$, a pseudo-metric $d$ is
\begin{enumerate}[noitemsep,label=(\arabic*)]
\item
\textbf{proper} \index{Proper! pseudo-metric|textbf}
if balls with respect to $d$ are relatively compact,
equivalently if the map $X \longrightarrow \R_+, \hskip.1cm x \longmapsto d(x_0,x)$
is proper for all $x_0 \in X$, 
\item
\textbf{locally bounded} 
\index{Locally bounded! pseudo-metric|textbf}
if every point in $X$ has a neighbourhood of finite diameter,
\item
\textbf{continuous} \index{Continuous pseudo-metric|textbf}
if the map $d : X \times X \longrightarrow \R_+$ is continuous.
\end{enumerate}
(Note that the map $X \times X \longrightarrow \R_+, \hskip.1cm (x,y) \longmapsto d(x,y)$
is proper if and only if $X$ is compact.)
\par

A metric space $(X,d)$ is \textbf{proper} \index{Proper! metric space|textbf}
if its subsets of finite diameter are relatively compact.
Observe that this holds if and only if, on the associated topological space, the metric $d$ is proper.
\end{defn}

\begin{rem}
\label{remonmetrics}
(1)
In this book, 
\begin{itemize}
\item
\textbf{metrics on topological spaces need be neither continuous nor proper;}
in particular, \textbf{metrics need not be compatible.}
\end{itemize}
On $\R$, the metric $d_{\mathopen[-1,1\mathclose]} : \R \times \R \longrightarrow \R_+$, 
which assigns to $(x,y)$ the smallest integer $\lceil y-x\rceil$ greater or equal to $\vert y-x \vert$,
is proper and not continuous.
The subscript $\mathopen[-1,1\mathclose]$ indicates that it is a \emph{word metric},
see Definition \ref{wordmetric}.
\par
Let $(X,d)$ be a metric space; assume that $X$ is not compact.
Then the metric $(x,y) \longmapsto \frac{ d(x,y) }{ 1 +d(x,y) }$
is compatible (in particular continuous) and non-proper.

\vskip.2cm

(2)
Let $X$ be a topological space, $d$ a metric on $X$,
and $X_d$ the underlying set of $X$ together with the topology defined by $d$.
Then $d$ is continuous on $X$ if and only if the identity map viewed as
a map $X \longrightarrow X_d$ is continuous.
This \emph{does not} imply that $d$ is compatible, 
in other words it does not imply that this map is a homeomorphism,
as shown by the case of the left-closed right-open interval $X = [0,1\mathclose[ \subset \R$
and the continuous metric $d$ defined on it by
$d(x,x') = \vert e^{2i\pi x} - e^{2i\pi x'} \vert$.
\par
The situation for morphisms of $\sigma$-compact LC-\emph{groups} is different.
See Corollaries \ref{Freudenthalcor} and \ref{Freudenthalcorbis}.

\vskip.2cm

(3)
Let $X$ be a topological space;
\begin{itemize}
\item
\textbf{for a proper metric $d$ on $X$, 
the metric space $(X,d)$ need not be proper.}
\end{itemize}
Consider for example the compact topological space $X = \R / \Z$, 
the non-compact topological space $Y = (\R / \Z) \times \R$,
a compatible metric $d_Y$ on $Y$ that is non-proper (for example a bounded metric),
a bijection $\varphi : X \longrightarrow Y$ (for example an isomorphism of abelian groups),
and the pullback metric $d_X = \varphi^*(d_Y)$ on $X$.
The metric $d_X$ on $X$ is proper (because $X$ is compact), 
but the metric space $(X,d_X)$ is not proper 
(because it is isometric to $(Y,d_Y)$).

\vskip.2cm

(4)
On a topological space, even on a locally compact one, 
a metric need not be locally bounded, as shown by the metric defined on $\R$
by $(x,y) \longmapsto \vert \gamma (x) - \gamma (y) \vert$,
where $\gamma$ is a non-continuous group automorphism of $\R$.
Indeed, there exists on $\R / \Z$ a \emph{proper} non-continuous metric
that is not locally bounded (Remark \ref{remonwordmetric}(5)).
\par
On a topological space $X$, a continuous pseudo-metric $d$ 
is necessarily locally bounded.
The word metric $d_{\mathopen[-1,1\mathclose]}$ on $\R$, as in (1) above,
is locally bounded and is not continuous.

\vskip.2cm

(5) 
Let $X$ a topological space and $d$ a pseudo-metric on $X$.
Suppose that $d$ is locally bounded.
Then every compact subset of $X$ has a finite diameter;
and so does every relatively compact subset of $X$.
\par
When $X$ is locally compact, the converse holds: 
if every compact subset of $X$ has finite diameter,
then $d$ is locally bounded.
For an LC-space $X$, Definition \ref{defproper2}(2) can therefore be reformulated as:
\begin{enumerate}
\item[]
A pseudo-metric $d$ on a locally compact space $X$ is \textbf{locally bounded} 
if every compact subset of $X$ is bounded with respect to $d$.
\end{enumerate}
\index{Locally bounded! pseudo-metric}

\vskip.2cm

(6)
A topological space on which
there exists a proper locally bounded pseudo-metric
is necessarily locally compact and $\sigma$-compact.
A fortiori, a proper metric space is locally compact and $\sigma$-compact.
\par
Conversely, there exists a compatible proper metric
on every $\sigma$-compact locally compact space \cite{Vaug--37}
(see also \cite[Theorem 1]{WiJa--87}).

\vskip.2cm

(7)
A metric space is always first-countable, but need not be second-countable.
\par
A first-countable compact space need not be separable,
hence need not be metrizable
\cite[Chapter V, Problem~J]{Kell--55}.   
\par
However a topological group is first-countable if and only if it is metrizable
(by the Birkhoff-Kakutani Theorem, \ref{BK} below).
\end{rem}

\begin{prop}
\label{contpropermet}
On a locally compact space, every proper continuous metric is compatible.
\end{prop}

\begin{proof}
Let $X$ be a locally compact space, $d$ a proper continuous metric on $X$,
and $X_d$ the space $X$ together with the topology defined by $d$.
The identity mapping provides a map $j : X \longrightarrow X_d$
which is continuous, because $d$ is continuous.
It remains to show that, for every closed subset $F$ of $X$,
its image in $X_d$ is also closed.
\par

Consider a point $x \in X$ and a sequence $(x_n)_{n \ge 0}$ of points in $F$ such that 
\hfill \par\noindent
$\lim_{n \to \infty} d(x_n,x) = 0$.
We have to show that $x \in F$.
Since $d$ is proper, there exists a compact subset of $X$
that contains $x_n$ for all $n \ge 0$.
Upon replacing $(x_n)_{n \ge 0}$ by a subsequence,
we may assume that $(x_n)_{n \ge 0}$ converges in $X$
towards some point $y \in F$.
Since $\lim_{n \to \infty} d(x_n,y) = 0$, we have $y=x$.
Hence $F$ is closed in $X_d$.
\end{proof}

About continuous and compatible metrics on \emph{groups},
see also Corollary \ref{Freudenthalcorbis}.

\begin{thm}[Second-countable LC-spaces]
\index{Theorem! Urysohn}
\label{Urysohn}
For a locally compact space $X$,
the following properties are equivalent:
\begin{enumerate}[noitemsep,label=(\roman*)]
\item
$X$ is second-countable;
\item
$X$ is metrizable and $\sigma$-compact; 
\item
$X$ is metrizable and separable; 
\item
$X$ is Polish;
\item
the topology of $X$ can be defined by a proper metric.
\end{enumerate}
\end{thm}

\begin{proof}[On the proof]
A second-countable LC-space is obviously $\sigma$-compact.
It is metrizable by a particular case of
\textbf{Urysohn's metrization theorem}: \index{Urysohn metrization theorem}
\emph{every second-countable regular space is metrizable}.
For this and other implications, 
see e.g.\  \cite[Page IX.21]{BTG5-10} or \cite[Theorem 5.3]{Kech--95}.
\end{proof}

The next proposition answers questions concerning
properties of closed subspaces, quotient spaces, and products.
We first recall a standard terminology.

\begin{defn}
\label{defeqrel}
Let $\mathscr{R}$ be an equivalence relation on a topological space $X$;
let $X / \mathscr{R}$ denote the quotient space.
\par
The relation $\mathscr{R}$ is \textbf{open} 
if  the canonical projection $X \longrightarrow X / \mathscr{R}$ is open, 
equivalently if the $\mathscr{R}$-closure of every open subset in $X$ is open.
\par
The relation $\mathscr{R}$ is \textbf{Hausdorff}
\index{Hausdorff! equivalence relation}
\index{Equivalence relation! open, Hausdorff}
if $X / \mathscr{R}$ is Hausdorff.
When $\mathscr{R}$ is open, then $\mathscr{R}$ is Hausdorff 
if and only if the graph of $\mathscr{R}$ is closed in $X \times X$
\cite[Page I.55]{BTG1-4}.
\end{defn}

\noindent \emph{Note.}
If $G$ is a topological group and $H$ a \emph{closed} subgroup,
the relation $\mathscr{R}$ defined on $G$ by 
$g_1\mathscr{R}g_2$ if $g_1^{-1}g_2 \in H$
is \emph{open} and Hausdorff.
In this case, $G/\mathscr{R}$ is rather written $G/H$;
the natural projection from $G$ onto $G/H$ need not be closed.

\vskip.2cm

For \ref{hereditarityTopSpace} (and for \ref{stababcd} below), 
we consider the following properties of topological spaces:
\begin{enumerate}[noitemsep,label=(\alph*)]
\item\label{aDEhereditarityTopSpace}
local compactness,
\item\label{bDEhereditarityTopSpace}
metrizability,
\item\label{cDEhereditarityTopSpace}
first countability,
\index{First-countable! topological space}
\item\label{dDEhereditarityTopSpace}
$\sigma$-compactness,
\item\label{eDEhereditarityTopSpace}
second countability,
\index{Second-countable! topological space}
\item\label{fDEhereditarityTopSpace}
separability.
\end{enumerate}

\begin{prop}[hereditary properties]
\label{hereditarityTopSpace}
\index{Hereditarity of various properties}
Let $X$ be a topological space, and $Y$ a subspace.
\begin{enumerate}[noitemsep,label=(\arabic*)]
\item\label{1DEhereditarityTopSpace}
Assume that $X$ is locally compact \ref{aDEhereditarityTopSpace}. 
Then $Y$ is locally compact
if and only if there exist in $X$ a closed subset $C$ and an open subset $U$  
such that $Y = C \cap U$.
\item\label{2DEhereditarityTopSpace}
If $X$ has one of Properties 
\ref{bDEhereditarityTopSpace}, \ref{cDEhereditarityTopSpace} or \ref{eDEhereditarityTopSpace},
then $Y$ has the same property.
\item\label{3DEhereditarityTopSpace}
If $X$ is locally compact \ref{aDEhereditarityTopSpace}
and $\sigma$-compact \ref{dDEhereditarityTopSpace}, 
closed subspaces of $X$ have the same properties,
but open subspaces need not have.
\item\label{4DEhereditarityTopSpace}
Assume that $X$ is separable \ref{fDEhereditarityTopSpace}. 
Open subspaces of $X$ are separable, but closed subsets need not be.
\end{enumerate}
Let $\mathscr{R}$ be an open Hausdorff equivalence relation on $X$.
\begin{enumerate}[noitemsep,label=(\arabic*)]
\addtocounter{enumi}{4}
\item\label{5DEhereditarityTopSpace}
If $X$ has one of Properties 
\ref{aDEhereditarityTopSpace},  \ref{cDEhereditarityTopSpace}, \ref{dDEhereditarityTopSpace},
\ref{eDEhereditarityTopSpace}, \ref{fDEhereditarityTopSpace},
then $X/\mathscr{R}$ has the same property.
This \emph{does not} carry over to \ref{bDEhereditarityTopSpace}.
\end{enumerate}
Let $I$ be a set, $(X_i)_{i \in I}$ a collection of topological spaces,
each of them containing at least two points,
and $X = \prod_{i \in I} X_i$ their product, with the product topology.
\begin{enumerate}[noitemsep,label=(\arabic*)]
\addtocounter{enumi}{5}
\item\label{6DEhereditarityTopSpace}
$X$ has one of the properties \ref{aDEhereditarityTopSpace}, \ref{dDEhereditarityTopSpace}
if and only if all the $X_i$ 's have this property
and all but possibly finitely many of them are compact.
\item\label{7DEhereditarityTopSpace}
$X$ has one of the properties
\ref{bDEhereditarityTopSpace}, \ref{cDEhereditarityTopSpace}, \ref{eDEhereditarityTopSpace}
if and only if all the $X_i$ 's have this property and $I$ is countable.
\item\label{8DEhereditarityTopSpace}
$X$ is separable \ref{fDEhereditarityTopSpace}
if and only if all the $X_i$ 's are separable and
the cardinality of $I$ is at most that of the continuum.
\end{enumerate}
\end{prop}

\begin{proof}[On proofs and references]
Most of the claims are straightforward.
For \ref{1DEhereditarityTopSpace},
see for example \cite[Corollary 3.3.10]{Enge--89}.
For \ref{4DEhereditarityTopSpace},
see \cite[Section VIII.7]{Dugu--66}.
For an example of a non-separable subspace of a separable space,
consider the topological space $E^1_u$ with underlying set $\R$
and the so-called upper limit topology, 
for which intervals of the form $]a,b]$ constitute a basis,
its square product, 
and the uncountable discrete subset of pairs 
$(x,-x)$ in $E^1_u \times E^1_u$ with $x$ irrational. 
\par

Concerning \ref{5DEhereditarityTopSpace},
see \cite[Theorem 3.3.15]{Enge--89} for \ref{aDEhereditarityTopSpace}
and \cite[Example 4.4.10]{Enge--89} for \ref{bDEhereditarityTopSpace}.
\par

Concerning \ref{6DEhereditarityTopSpace},  
see \cite[Theorem 3.3.13]{Enge--89} for \ref{aDEhereditarityTopSpace}.
Here is a simple argument for \ref{dDEhereditarityTopSpace}.
Assume that $I$ contains an infinite subset $J$ with $X_j$ non-compact for all $j \in J$.
Let $p_j : X \relbar\joinrel\twoheadrightarrow X_j$ denote the canonical projection.
Let $(C_n)_{n \ge 1}$ be a countable family of compact subsets of $X$.
For each $j \in J$ and $n \ge 1$, there exists a point $x_j \in X_j \smallsetminus p_j(C_n)$.
For each $i \in I \smallsetminus J$, choose $x_i \in X_i$. 
Then $(x_i)_{i \in I} \notin \bigcup_{n \ge 1} C_n$.
Hence $X$ is not $\sigma$-compact.
\par

Concerning \ref{7DEhereditarityTopSpace},
see \cite[Theorem 2.3.14]{Enge--89} for \ref{cDEhereditarityTopSpace} 
and \ref{eDEhereditarityTopSpace},
and \cite[Theorem 4.2.2]{Enge--89} for \ref{bDEhereditarityTopSpace}.
In the particular case of a product space $F^A$, 
with $F$ finite of cardinality at least $2$ and $A$ uncountable,
here is an easy argument to show that
$F^A$ is not second-countable.
If it were, every subspace would be separable.
But the space $F^{(A)}$ of finitely supported functions from $A$ to $F$
is not separable:  
every countable subset of $F^{(A)}$ 
is contained in the proper closed subspace $F^B$ of $F^A$, 
for some countable subset $B$ of $A$.
\par

For \ref{8DEhereditarityTopSpace},
see \cite[Corollary 2.3.16]{Enge--89},
as well as \cite{RoSt--64}.
\end{proof}

Note that, in the particular case of a product space $F^A$,
with $F$ a finite set  of cardinality at least $2$ an $A$ a set, 
the following four properties are equivalent:
(i) $F^A$ is first-countable,
(ii) $F^A$ is second-countable,
(iii) $F^A$ is metrizable, 
(iv) $A$ is countable.

\begin{exe}[metric graphs]
\label{metricrealizationgraph}
Let $X$ be a connected graph, 
i.e., a connected $1$-dimensional CW-complex.
\index{Graph|textbf}
There is a natural \textbf{combinatorial metric} 
\index{Combinatorial! metric on a connected graph|textbf}
$d_1$ on $X$,
for which every edge with distinct endpoints (resp.\ every loop)
is isometric to a real interval of length $1$ (resp.\ a circle of perimeter $1$),
and for which the distance between two points
is the minimum of the lengths of arcs connecting these two points;
see for example  \cite[Section I.1.9]{BrHa--99}.
The metric space $(X, d_1)$ is the \textbf{metric realization} of $X$.
\index{Realization! metric realization of a connected graph|textbf}
Note that standard books of algebraic topology 
consider other metrics on $X$, which are not proper in general; 
see for example Pages 111 and 119 in \cite{Span--66}.
\par

Recall that a graph $X$ is \textbf{locally finite} 
\index{Locally finite! graph|textbf}
if, for every vertex $x$ of $X$ (i.e., for every point $x$ in the $0$-skeleton of $X$),
the number $n_x$ of closed $1$-cells (i.e., of edges) of $X$ containing $x$ is finite.
The following four conditions are equivalent:
(i) the graph $X$ is locally finite,
(ii) the metric space $(X,d_1)$ is proper,
(iii) the topology defined on $X$ by $d_1$
coincides with the ``weak topology'' of the CW-complex,
(iv) the vertex set $X^0$ of $X$, viewed as a subspace of the metric space $(X,d_1)$, 
is locally finite in the sense of Definition \ref{defunifproper} below,
i.e., balls in $X^0$ with respect to $d_1$ are finite.
\par

A graph $X$ has \textbf{bounded valency}
\index{Bounded! bounded valency graph|textbf} 
if there is a positive constant $c$ such that, with the notation above, 
$n_x \le c$ for every vertex $x$ of $X$,
equivalently if the vertex set of $X$, as a subspace of $(X,d_1)$,
is uniformly locally finite in the sense of Definition \ref{defunifproper}.
\par

Given a constant $c > 0$,
there is a variation $(X,d_c)$ of this construction
in which edges [respectively loops] are isometric to a real interval of length $c$ 
[resp.\ to a circle of perimeter $c$]; this is used in \ref{the new space Xc etc}.
\par

There is a $2$-dimensional analogue of this example 
in Proposition \ref{donrealizationofsimpcx}.
\end{exe}

\begin{exe}[$\sigma$-compactness and metrizability]
\label{sigmacandmetrizability}
Metrizability and $\sigma$-compact\-ness are independent properties.
Indeed, let $M$ be a discrete group and $A$ a set.
The LC-group
$M \times (\Z/2\Z)^A$ is $\sigma$-compact if and only if $M$ is countable,
and metrizable if and only if $A$ is countable.
\index{Product of groups! $F^A$, with $F$ finite}
\par
It follows that $M \times (\Z/2\Z)^A$ is second-countable
(i.e., both $\sigma$-compact and metrizable)
if and only if both $M$ and $A$ are countable.
\par
Note however that a topological space that has a \emph{proper} compatible metric
is $\sigma$-compact, and therefore second-countable by Theorem \ref{Urysohn}.
\end{exe}

\begin{exe}[separable and not second-countable]
\label{sepnot2nd}
A second-countable LC-space is separable (Remark  \ref{neednotbelc}(5))
and $\sigma$-compact (Theorem \ref{Urysohn}),
but the converse implication does not hold. 
\index{Separable topological space}
\par

Indeed, consider a finite set $F$ of cardinality at least $2$
and an uncountable set $A$ of cardinality at most that of the continuum.
By Proposition \ref{hereditarityTopSpace}, the product space $F^A$
is compact, separable, not first-countable,
and a fortiori not second-countable.
Note that $F^A$ is naturally a compact group if $F$ is a finite group.
\index{Compact group! not first-countable}
\par

Recall however the following: 
a \emph{metrizable} topological space is second-countable 
if and only if it is separable (Theorem \ref{Urysohn}).

\vskip.2cm

We come back to second-countable LC-groups and $\sigma$-compact LC-groups
in Remark \ref{Gexplescgsigmacsc}.
\end{exe}

\begin{defn}
\label{defBaire}
A topological space $X$ is a \textbf{Baire space} 
\index{Baire space, Baire group, Baire theorem|textbf}
if every countable intersection of dense open subsets of $X$ is dense in $X$.
\end{defn}

In particular, a \textbf{Baire group} 
is a topological group that is a Baire space.

\begin{thm}[Baire] 
\index{Theorem! Baire}
\label{propBaire}
(1) Every LC-space is a Baire space.
\par
(2) Every completely metrizable topological space is a Baire space.
\end{thm}

\begin{proof}
See \cite[Page IX.55]{BTG5-10}.
\end{proof}

\begin{rem}
\label{cardinalBairegroup}
It follows from Baire's Theorem that a countable LC-space has an isolated point,
and therefore that a non-discrete LC-group is uncountable.
\end{rem}

\section[Metrizability and $\sigma$-compactness]
{Metrizability for topological groups and $\sigma$-compact LC-groups}
\label{Onlcgroups}

There are two natural questions concerning a compatible metric $d$
on a topological \emph{space} $X$:
is the metric space $(X,d)$ complete? is it proper?
If $d$ is a compatible metric on a topological \emph{group} $G$,
there is one more question: is $d$ left-invariant?
Some remarks are in order.

\begin{rem}[on metrizability for topological groups and LC-groups]
\label{onmetrizability} 
Let $G$ be a topological group.

\vskip.2cm

(1)
Theorem \ref{BK} below shows that,
if $G$ is \textbf{metrizable}
in the sense of Definition \ref{pm...metrizable},
\index{Metrizable! group|textbf}
i.e., if $G$ admits a compatible metric,
then $G$ admits a \emph{left-invariant} compatible metric.

\vskip.2cm

(2)
If $G$ is completely metrizable, even if $G$ is a Polish group,
then $G$ need not have any complete left-invariant compatible metric;
see Example \ref{Sym(X)andallthat}.
\index{Polish! group}

\vskip.2cm

(3)
In a topological group $G$, a sequence $(g_m)_{m \ge 1}$
is a \textbf{left Cauchy sequence} if, 
\index{Left Cauchy sequence in a topological group|textbf}
for every neighbourhood $U$ of $1$ in $G$,
there exists $\ell \ge 1$ such that $g_n^{-1}g_m \in U$ when $m,n \ge \ell$.
\par
When $G$ has a left-invariant compatible metric $d$,
a sequence $(g_m)_{m \ge 1}$ in $G$ is a left Cauchy sequence in this sense
if and only if it is a Cauchy sequence for the metric space $(G, d)$.
It follows that two left-invariant compatible metrics on $G$ 
have the same left Cauchy sequences.

\vskip.2cm

(4) 
On an LC-group $G$,
every left-invariant compatible metric $d$ is complete.
\index{Completeness of metrizable LC-groups}
\par
To check this, consider
a compact neighbourhood $K$ of $1$ in $G$
and a Cauchy sequence $(g_n)_{n \ge 0}$ in $G$ with respect to $d$.
Since $\lim_{\min \{m,n\} \to \infty} d(1, g_m^{-1}g_n) = 0$,
there exists an integer $n_0$ such that $g_m^{-1}g_n \in K$
whenever $m,n \ge n_0$,
in particular such that $g_n$ is in the compact subset $g_{n_0}K$
for all $n \ge n_0$.
It follows that $(g_n)_{n \ge 0}$  has a limit in $G$.

\vskip.2cm

(5)
An interesting class of Polish groups is that of cli-groups.
A \textbf{cli-group} \index{Cli-group|textbf}  \index{Polish! group} \index{Polish! cli-group}
is a Polish group 
which has a compatible metric that is both complete and left-invariant; see \cite{Beck--98},
and in particular his Proposition 3.C.2(d), on Polish groups that are cli.
For example, abelian Polish groups are cli, 
and the group $\operatorname{Sym} (\N)$  of Example \ref{Sym(X)andallthat}
is a non-cli Polish group.

\vskip.2cm

(6)
Let $G$ be an LC-group and $d$ a left-invariant compatible metric on $G$.
Balls with respect to $d$ with sufficiently small radius are relatively compact,
but large balls need not be (unless $d$ is proper).

\vskip.2cm

(7)
Let $G$ be a Baire group.
\index{Baire space, Baire group, Baire theorem}
If $G$ is $\sigma$-compact, in particular if $G$ is compactly generated,
then $G$ is locally compact.
\par

To check this, consider a sequence
$(K_n)_{n \ge 0}$ of compact subspaces of $G$
such that $G = \bigcup_{n \ge 0} K_n$.
Since $G$ is Baire,
there exists $n \ge 0$ such that $K_n$ has an interior point.
We conclude by observing that a group 
that contains one point with a compact neighbourhood is locally compact.

\end{rem}

The next three Theorems (\ref{BK}, \ref{Struble}, \ref{KK})
are basic criteria of metrizabiity for LC-groups.

\begin{thm}[Birkhoff-Kakutani] \index{Birkhoff-Kakutani theorem}
\index{Theorem! Birkhoff-Kakutani}
\label{BK}
For a topological group $G$, the following three properties are equivalent:
\index{Topological group}
\begin{enumerate}[noitemsep,label=(\roman*)]
\item\label{iDEBK}
$G$ is first-countable;
\index{First-countable! topological group}
\item\label{iiDEBK}
$G$ is metrizable;
\item\label{iiiDEBK}
there exists a left-invariant compatible metric on $G$.
\end{enumerate}
\end{thm}

The original articles are \cite{Birk--36} and \cite{Kaku--36}. 
Related references include
\cite[Theorem of 1.22, Page 34]{MoZi--55},
\cite[Theorem 8.5]{HeRo--63}, and \cite[Page IX.23]{BTG5-10}.

Our proofs of Theorem \ref{BK} and of its companion result, Theorem \ref{Struble},
follow Lemma \ref{lemBKetKK}.

\begin{rem}
\label{RemOnBK}
(1)
Recall that an LC-group need not be first-countable
(Example \ref{sigmacandmetrizability}).
\par
(2)
Any LC-group has a left-invariant continuous \emph{pseudo-}metric
with balls of radius $0$ compact; see Corollary \ref{CorollaryKK} below.
\end{rem}

\begin{thm}[Struble] \index{Struble theorem}
\index{Theorem! Struble}
\label{Struble}
For an LC-group $G$, the following three properties are equivalent:
\begin{enumerate}[noitemsep,label=(\roman*)]
\item\label{iDEStruble}
$G$ is second-countable;
\index{Second-countable! topological group}
\item\label{iiDEStruble}
$G$ is $\sigma$-compact and first-countable;
\index{Sigma-compact! LC-group}
\item\label{iiiDEStruble}
there exists a left-invariant proper compatible metric on $G$.
\end{enumerate}
\end{thm}

The original article \cite{Stru--74} is an adaptation of \cite{Birk--36, Kaku--36}.
On an LC-group that has the properties of Theorem \ref{Struble},
two left-invariant proper compatible metrics are ``coarsely equivalent'';
see Corollary \ref{2metricsce}\ref{2DE2metricsce}.

\begin{lem}
\label{lemBKetKK}
Let $G$ be a  topological group
and $(K_n)_{n \in \Z}$ a sequence of subsets of $G$
such that $\bigcup_{n \in \Z} K_n = G$,
with $K_n$ symmetric containing $1$ for all $n$.
We assume moreover that 
$K_n K_n K_n \subset K_{n+1}$ for all $n \in \Z$,
and that $K_m$ has non-empty interior for some $m$.
Define $d : G \times G \longrightarrow \R_+$ by
\begin{equation*}
d(g,h) \, = \, \inf \left\{t \in \R_+    \hskip.1cm \bigg\vert \hskip.1cm
\aligned
& \exists \hskip.2cm n_1, \hdots, n_k \in \mathbf Z
\hskip.2cm \textnormal{and} \hskip.2cm w_1 \in K_{n_1}, \hdots, w_k \in K_{n_k}
\\ 
& \textnormal{with} \hskip.2cm g^{-1}h = w_1 \cdots w_k
\hskip.2cm \textnormal{and} \hskip.2cm t =  2^{n_1} + \cdots + 2^{n_k}
\endaligned
\right\} 
\end{equation*}
and set $\vert g \vert = d(1,g)$.
Then:
\begin{enumerate}[noitemsep,label=(\arabic*)]
\item\label{1DElemBKetKK}
$d$ is a left-invariant pseudo-metric on $G$
for which every compact subset has a finite diameter;
\item\label{2DElemBKetKK}
if each $K_n$ is a neighbourhood of $1$, then $d$ is continuous;
\item\label{3DElemBKetKK}
for all $n \in \Z$, $\vert g \vert < 2^n$ implies $g \in K_n$;
\item\label{4DElemBKetKK}
$\{g \in G \mid d(1,g)=0 \} = \bigcap_{n \le 0} K_n$;
in particular, 
if $\bigcap_{n \le 0} K_n = \{1\}$, then $d$ is a metric.
\end{enumerate}
\end{lem}

\noindent
\emph{Note:} instead of ``such that $\bigcup_{n \in \Z} K_n = G$'',
one could  equivalently write ``such that $\bigcup_{n \in \Z} K_n$ generates $G$'';
the point is to ensure that $d(g,h)$ is finite for every pair $(g,h)$.

\begin{proof}
\ref{1DElemBKetKK}
It is obvious that $d$ is a left-invariant pseudo-metric on $G$.
Let $L \subset G$ be a compact set.
Since the interior $\operatorname{int}(K_m)$ of $K_m$ is non-empty, 
$K_{m+1}$ is a neighbourhood of $1$,
so that $L \subset \bigcup_{\ell \in L} \ell \operatorname{int}(K_{m+1})$,
and $L$ is covered by finitely many translates of $K_{m+1}$.
It follows that $\sup_{\ell, \ell' \in L} d(\ell, \ell') < \infty$.
\vskip.2cm

\ref{2DElemBKetKK}
Let $g,h \in G$, and $\varepsilon > 0$.
Let $n \ge 0$ be such that $2^{- n} \le \varepsilon / 2$.
Since $K_{- n}$ is a neighbourhood of $1$, 
for $x$ and $y$ in $G$ near enough $g$ and $h$ respectively,
we have $g \in xK_{- n}$, $x \in gK_{- n}$, $h \in yK_{- n}$, and $y \in hK_{- n}$.
Then 
\begin{equation*}
d(x,y) \, \le \,   d(1, x^{-1}g) + d(1, g^{-1}h) + d(1, h^{-1}y) \, \le \, d(g,h) + \varepsilon
\end{equation*}
and similarly $d(g,h) \le d(x,y) + \varepsilon$.
It follows that $d$ is continuous.
\vskip.2cm

\ref{3DElemBKetKK}
Let us spell out what has to be proven as follows:
\vskip.2cm
\noindent
\emph{Claim: 
let  $w \in G$, $n \in \Z$, $k \ge 0$, $n_1, \hdots, n_k \in  \Z$, 
and $v_1 \in K_{n_1}, \hdots, v_k \in K_{n_k}$, with}
\par \hskip1cm
$(*) \hskip.2cm w = v_1 \cdots v_k \hskip.2cm \text{and} \hskip.2cm
2^{n_1} + \cdots + 2^{n_k} < 2^n$;
\par
\emph{then $w \in K_n$.}
\par\noindent
(Observe that the hypothesis of the claim holds precisely when $\vert w \vert < 2^n$.)
\vskip.2cm

If $k=0$, then $w = 1$ and there is nothing to show.
If $k=1$, we have $n_1 < n$ by ($*$), so that $w \in K_{n-1} \subset K_n$.
If $k=2$, we have $n_1, n_2 < n$ by ($*$), so that $w \in K_{n-1}K_{n-1} \subset K_n$.
We continue by induction on $k$, assuming from now on that $k \ge 3$
and that the claim has been shown up to $k-1$ 
(this for all $n \in \Z$ altogether).
We isolate part of the argument as:
\vskip.2cm
\noindent
\emph{Subclaim: 
there exists $j \in \{1, \hdots, k\}$ such that 
$w = w_1 v_j w_2$, where $w_1 = v_1 \cdots v_{j-1}$ and $w_2 = v_{j+1} \cdots v_k$
satisfy}
\par \hskip1cm
$(**) \hskip.2cm  \vert w_1 \vert, \vert w_2 \vert < 2^{n-1}$.
\par\noindent
(Here, we think of $w$ as a word in $k$ letters.)
\vskip.2cm

If there is one $j \in \{1, \hdots, k\}$ with $n_j = n-1$, 
this $j$ is necessarily unique by ($*$),
so that we can write $w = w_1 v_j w_2$,
where each of $w_1, w_2$
is a product similar to that of ($*$) for a smaller value of $k$ and for $n-1$,
and the subclaim follows.
\par

Otherwise, we have  $n_i \le n-2$ for $i = 1, \hdots, k$.
Let $\ell$ be the largest of those $i$ for which $2^{n_1} + \cdots + 2^{n_i} < 2^{n-1}$. 
If $\ell = k$, it is enough to set $j=2$.
If $\ell < k$, set $j=\ell+1$, so that $w = w_1 v_j w_2$, with $w_1, w_2$ as in the subclaim.
On the one hand,  $\vert w_1 \vert < 2^{n-1}$ by the definition of $\ell$;
on the other hand, 
$2^{n_1} + \cdots + 2^{n_\ell} \ge 2^{n-1}$
and therefore $\vert w_2 \vert < 2^{n-1}$ by ($*$). 
This ends the proof of the subclaim.

We can resume the proof of the claim
(in the situation with $k \ge 3$).
Since $w_1, w_2$ are products similar to that of ($*$) 
for a smaller value of $k$ and for $n-1$,
we have $w_1, w_2 \in K_{n-1}$ by the induction hypothesis,
so that $w \in K_{n-1}  K_{n-1}  K_{n-1} \subset K_n$.

\vskip.2cm

\ref{4DElemBKetKK} 
This is a straightforward consequence of \ref{3DElemBKetKK}.
\end{proof}

\begin{proof}[Proof of the Birkhoff-Kakutani theorem \ref{BK}]
The implications 
\ref{iiiDEBK} $\Rightarrow$ \ref{iiDEBK} $\Rightarrow$ \ref{iDEBK} 
are straightforward.
\par
Let us assume that \ref{iDEBK} holds, 
i.e., that there exists in the topological group $G$
a countable basis $(V_n)_{n \ge 1}$ of neighbourhoods of $1$,
and let us show that \ref{iiiDEBK} holds, 
i.e.,  that there exists a left-invariant compatible metric on $G$.
For $n \ge 0$, set $K_n = G$.
For $n \ge 1$, choose inductively a symmetric neighbourhood $K_{-n}$ of $1$
such that $K_{-n} \subset V_n$ and $K_{-n} K_{-n} K_{-n} \subset K_{-n+1}$.
Let $d$ be defined as in Lemma \ref{lemBKetKK}.
Then $d$ is a metric on $G$, by \ref{4DElemBKetKK} of this lemma,
that is continuous, by \ref{2DElemBKetKK}, 
and indeed compatible with the topology of $G$, by \ref{3DElemBKetKK};
that is to say, $d$ is a left-invariant compatible metric on $G$,
as in \ref{iiiDEBK} of Theorem  \ref{BK}.
\end{proof}

\begin{proof}[Proof of the Struble theorem \ref{Struble}]
In view of Theorem \ref{Urysohn}, the implications
\ref{iiiDEStruble} $\Rightarrow$ \ref{iDEStruble} $\Rightarrow$ \ref{iiDEStruble}
are straightforward.
\par

Let us assume that \ref{iiDEStruble} holds.
Thus,  there exists a sequence $(L_n)_{n \ge 0}$
of symmetric compact subsets of $G$ containing $1$
such that $G = \bigcup_{n \ge 0} L_n$,
and also a countable basis $(V_n)_{n \ge 0}$ of neighbourhoods of $1$,
with $V_0$ relatively compact.
For $n \ge 0$, define inductively a symmetric compact subset $K_n$ of $G$ by
$K_0 = L_0 \cup \overline{V_0}$ and $K_{n+1} = L_{n+1} \cup K_nK_nK_n$.
For $n \ge 1$, choose inductively a sequence $K_{-n} \subset V_n$ 
of symmetric neighbourhoods of $1$ such that $K_{-n} \subset V_n$
and $K_{-n} K_{-n} K_{-n} \subset K_{-n+1}$,
as in the proof of \ref{iDEBK} $\Rightarrow$ \ref{iiiDEBK} of Theorem \ref{BK}.
Observe that $K_n K_n K_n \subset K_{n+1}$ holds for all $n \in \Z$.
The metric defined in Lemma \ref{lemBKetKK}
is now moreover proper, by (3) of this lemma.
Thus \ref{iiiDEStruble} of the theorem holds.
\end{proof}

\begin{thm}[Kakutani-Kodaira] \index{Kakutani-Kodaira theorem}
\index{Theorem! Kakutani-Kodaira}
\label{KK}
Let $G$ be a $\sigma$-compact LC-group.
For every sequence $(U_n)_{n \ge 0}$ of neighbourhoods of $1$ in $G$,
there exists a compact normal subgroup $K$ of $G$ contained in $\bigcap_{n \ge 0} U_n$
such that $G/K$ is metrizable. 
\index{Metrizable! group}
\par
In particular, every $\sigma$-compact LC-group is compact-by-metrizable.
\end{thm}

The original article is \cite{KaKo--44}.
The argument below, for $\sigma$-compact groups,
is that of \cite[Theorem 8.7]{HeRo--63},
stated there for compactly generated groups.
A convenient reference is \cite[Page 1166]{Comf--84}.

There are metrics with extra properties
on LC-groups with extra conditions.
We will come back to this in later chapters,
in particular in Propositions
\ref{existenceam} and  \ref{metric_char_cg}.

\begin{proof}
By hypothesis, there exists a sequence $(L_n)_{n \ge 0}$
of compact subsets of $G$ such that $1 \in L_n \subset L_{n+1}$ for each $n$
and $G = \bigcup_{n \ge 0} L_n$.
For $n \ge 0$, set $K_n = G$. For $n \ge 1$, define inductively
a symmetric compact neighbourhood $K_{-n}$ of $1$ as follows.
\par

Suppose that $n \ge 0$ and that $K_0, \hdots, K_{-n}$ have been defined.
Observe that the map 
\begin{equation*}
L_n \times G \longrightarrow G, \hskip.5cm
(\lambda, \kappa) \longmapsto \lambda \kappa \lambda^{-1}
\end{equation*}
is continuous, and its value is identically $1$ on $L_n \times \{1\}$.
By continuity, there exists for all $\lambda \in L_n$
an open neighbourhood $V_\lambda$ of $\lambda$
and a compact neighbourhood $W_\lambda$ of $1$
such that $\ell k \ell^{-1} \in K_{-n}$
for all $\ell \in V_\lambda$ and $k \in W_\lambda$.
Let $\lambda_1, \hdots, \lambda_j \in L_n$ 
be such that $L_n \subset \bigcup_{i=1}^j V_{\lambda_j}$.
Set $K_{-n-1} = \bigcap_{i=1}^j W_{\lambda_j}$.
\par

Then $\ell k \ell^{-1} \in K_{-n}$ 
for all $\ell \in L_n$ and $k \in K_{-n-1}$.
Upon replacing $K_{-n-1}$ by a smaller symmetric compact neighbourhood of $1$,
we can assume moreover that $(K_{-n-1})^3 \subset K_{-n}$ and $K_{-n-1} \subset U_n$.
\par

Define $K = \bigcap_{n \ge 1} K_{-n}$.
Then $K$ is clearly a closed subgroup of $G$ contained in $\bigcap_{n \ge 0} U_n$.
We claim that $K$ is normal in $G$.
Indeed, let $g \in G$. There exists $n_0 \ge 1$
such that $g \in L_n$ for all $n \ge n_0$.
Thus, for $n \ge n_0$, we have $g K_{-n-1} g^{-1} \subset K_{-n} \subset K_{-n_0}$,
and a fortiori $g K g^{-1} \subset K_{-n}$.
It follows that $g K g^{-1} \subset \bigcap_{n \ge n_0} K_{-n} = K$.
\par

By Lemma \ref{lemBKetKK}, 
there exists a left-invariant proper continuous pseudo-metric $d$ on $G$,
such that $d(1,g) = 0$ if and only if $g \in K$.
This induces a left-invariant proper compatible metric on $G/K$.
\end{proof}

\begin{rem}
(1)
Let $G$ be a non-metrizable $\sigma$-compact LC-group,
and $K$ a compact normal subgroup of $G$ such that $G/K$ is metrizable.
It follows from Proposition \ref{stababcd}\ref{2DEstababcd} below (case of \ref{bDEstababcd}) 
that $K$ is non-metrizable.
\par
In loose words, in a $\sigma$-compact LC-group $G$,
the existence of non-metrizable compact normal subgroups
is the only obstruction to the non-metrizability of $G$.

\vskip.2cm

(2)
The conclusion of the previous theorem can equally be phrased as
``such that  $G/K$ is second-countable''; see Theorem \ref{Urysohn}.
Thus, every $\sigma$-compact LC-group is compact-by-(second-countable).

\vskip.2cm

(3)
The inclusions
\begin{equation*}
\aligned
\{\text{$\sigma$-compact LC-groups}\} 
\, &\subset \, \{\text{compact-by-metrizable LC-groups}\} 
\\
\, &\subset \, \{\text{LC-groups}\}
\endaligned
\end{equation*}
are strict. Concerning the first inclusion,
note that uncountable discrete groups are metrizable and are not $\sigma$-compact.
Concerning the second,
Example \ref{non-sigma-compact-by-discrete} shows that
LC-groups need not be compact-by-metrizable,
indeed need not be ($\sigma$-compact)-by-metrizable. 
In particular, an LC-group does contain compactly generated open subgroups
(see Proposition \ref{powersSincpgroup}),
\index{Open subgroup! $\sigma$-compact}
\index{Subgroup! open|see {Open subgroup}}
but need not contain a \emph{normal} $\sigma$-compact open subgroup,
i.e.,  need not be ($\sigma$-compact)-by-discrete. 

\vskip.2cm

(4) 
Every LC-group contains a $\sigma$-compact open subgoup.
See Proposition \ref{powersSincpgroup}, that establishes a stronger property:
every LC-group contains compactly generated open subgroups.
\end{rem}

\begin{exe}[a topologically simple LC-group that is not $\sigma$-compact]
\label{non-sigma-compact-by-discrete}
\index{Sigma-compact! LC-group, non-sigma-compact}
For a set $X$, denote by $\operatorname{Sym}(X)$ 
the group of all permutations of $X$,
and by $\operatorname{Sym}_f(X)$ its subgroup of permutations with finite support.
Let $I$ be an infinite set and $J$ a finite set; assume that $\vert J \vert \ge 5$,
so that $\operatorname{Alt}(J)$ is a simple group. Consider the compact group
\index{Symmetric group $\operatorname{Sym}(X)$ of a set $X$}
\index{$ac$@$\simeq$ isomorphism of groups, of topological groups}
\begin{equation*}
K \, = \,
\big\{ g \in \operatorname{Sym}(I \times J) \mid
g(i,j) \in \{i\} \times J \hskip.2cm \text{for all} \hskip.2cm i \in I \big\}
\, \simeq \, (\operatorname{Sym}(J))^I ,
\end{equation*}
where $\simeq$ indicates an isomorphism of groups,
with the product topology $\mathcal T_K$ on $(\operatorname{Sym}(J))^I$.
Let $G$ be the subgroup of $\operatorname{Sym}(I \times J)$
generated by $K \cup \operatorname{Sym}_f(I \times J)$.
By Proposition \ref{PropTopsub} below,
there exists on $G$ a unique topology $\mathcal T_G$ such that:
\par
(1)
$G$ is a topological group, in which $K$ is open,
\par
(2)
the restriction of $\mathcal T_G$ to $K$ coincides with the topology $\mathcal T_K$.
\par\noindent
Observe that this topology on $G$ is locally compact.
Moreover, with this topology:
\par
(3)
$G$ is topologically simple.
\index{Simple group}
\par\noindent
This is a result of \cite[Section~5]{CaCo--14}, building on \cite{AkGW--08}.
Observe that (1) implies that $G$ is locally compact.
\par

Assume moreover that the set $I$ is not countable.
Then $G$ is not metrizable, because $K$ is not,
and not $\sigma$-compact, for any of the following two reasons.
A first reason is that it cannot be $\sigma$-compact,
by the Kakutani-Kodaira theorem \ref{KK}.
\index{Sigma-compact! LC-group}
To indicate a second reason, we anticipate on 
Section \ref{cg+sigmac}.
The natural action of $G$ on $I \times J$
is transitive, because $G$ contains $\operatorname{Sym}_f(I \times J)$,
and continuous, because the stabilizer in $G$ of every point $(i_0, j_0)$ in $I \times J$
is open (this stabilizer contains the compact open subgroup
$(\operatorname{Sym}(J))^{I \smallsetminus \{i_0\}}$ of $K$).
Hence the index of $K$ in $G$ is at least the cardinal of $I \times J$.
If $I$ is not countable, 
$G$ is not $\sigma$-compact by Corollary \ref{opensubgroupsinLCgroups}(1).
\end{exe}

\begin{exe}[(pseudo-)metrics on $\GL_n(\R)$ and related groups]
\label{pseudometriconGL}
Let $n$ be a positive integer; denote by $\M_n(\R)$ the algebra of $n$-by-$n$ real matrices.
A norm on the underlying real vector space 
defines a compatible metric $d$ on $\M_n(\R)$.
Since the general linear group 
$\GL_n(\R) = \{ g \in \M (\R) \mid g \hskip.2cm \text{invertible} \}$ is open in $\M_n(\R)$, 
\index{General linear group $\GL$! $\GL_n(\R)$, $\SL_n(\R)$}
the restriction of $d$ to $\GL_n(\R)$ is also a compatible metric;
but it is not left-invariant.
It follows from Theorem \ref{BK} that $\GL_n(\R)$ and its subgroups have left-invariant metrics;
but writing down explicitly some of these is not immediate.

\vskip.2cm

(1)
Denote by $G = \GL_n^+(\R)$ the identity $\{ g \in \GL_n(\R) \mid \det (g) > 0 \}$
\index{General linear group $\GL$! $\GL_n^+(\R)$}
of $\GL_n(\R)$, and by $\mathfrak g$ its Lie algebra.
A scalar product on $\mathfrak g$ 
defines uniquely a left-invariant Riemannian metric on $G$.
The associated metric on $G$ is a proper compatible metric 
(see Example \ref{abundanceofd}).
This carries over to every connected real Lie group $G$.

\vskip.2cm

(2)
Let $\textnormal{P}_n(\R)$ be the space of positive-definite symmetric matrices in $\M_n(\R)$.
Denote by ${}^t \hskip-.1cm g$ the transpose matrix of $g \in \M_n(\R)$.
There is a natural action
\begin{equation*}
\GL_n(\R) \times \textnormal{P}_n(\R) \, \longrightarrow \textnormal{P}_n(\R),
\hskip.5cm (g,p) \, \longmapsto \, g p {}^t \hskip-.1cm g
\end{equation*}
that is smooth and transitive; it induces a faithful action
of $\GL_n(\R) / \{ \pm \textnormal{id} \}$ on $\textnormal{P}_n(\R)$.
\par
The space $\textnormal{P}_n(\R)$ has a Riemannian metric
that makes it a symmetric space, and for which $\GL_n(\R) / \{ \pm \textnormal{id} \}$
acts by isometries; for this, we refer to \cite{Most--55}, or \cite[Chapter II.10]{BrHa--99}.
As on any group of isometries of a proper metric space, 
there exists on $\GL_n(\R) / \{ \pm \textnormal{id} \}$
left-invariant proper compatible metrics; 
see Proposition \ref{MainIsom(X)} below.

\vskip.2cm

(3)
For $g \in \GL_n(\R)$, 
denote by $p_g \in \textnormal{P}_n(\R)$ the positive square root of ${}^t \hskip-.1cm g g$, 
and let $u_g \in \textnormal{O} (n)$ be the orthogonal matrix 
defined by $g = p_g u_g$.
Let $\langle \cdot \mid \cdot \rangle$ denote the standard scalar product on $\R^n$.
Define the norm of a matrix $x \in \M_n(\R)$ by
\begin{equation*}
\Vert x \Vert = \sup \{ \sqrt{ \langle x \xi \mid x \xi \rangle} \mid \xi \in \R^n, \hskip.2cm
\langle \xi \mid \xi \rangle \le 1 \} .
\end{equation*}
For $g \in \GL_n(\R)$, set $\nu (g) = \log ( \Vert g \Vert \hskip.1cm \Vert g^{-1} \Vert )$;
using the polar decomposition $g = p_g u_g$, it can be checked that
$\nu (g) = 0$ if and only if $p_g$ is a scalar multiple of the identity matrix.
\par
It follows that $d : (g,h) \longmapsto \nu(g^{-1}h)$ is a left-invariant proper
continuous pseudo-metric on
$\SL^\pm_n(\R) = \{ g \in \GL_n(\R) \mid \det(g) = \pm 1 \}$.
The ball of radius $0$ around the origin, i.e., 
$\{ g \in \SL^\pm_n(\R) \mid d(1,g) = 0 \}$,
is the orthogonal group $\textnormal{O} (n)$.
\index{Orthogonal group! $\textnormal{O}(n)$, $\SO (n)$}

\vskip.2cm

(4)
Other left-invariant metrics on groups of this example
appear in Example \ref{abundanceofd}.
\end{exe}

\begin{exe}[metrics on automorphism groups of graphs]
\label{metricsonAut(X)}
Let $X$ be a connected locally finite graph.
\index{Graph}
\index{Locally finite! graph}
For the topology of pointwise convergence relative to the vertex set $X^0$ of $X$,
the automorphism group $\operatorname{Aut}(X)$ of $X$ is an LC-group.
\index{Pointwise topology} \index{Automorphism group! of a graph}
See for example \cite{Trof--85}, or Section \ref{isometrygroups};
see also Example \ref{extdLCgps}(5).
\par

We define a function 
$d_x : \operatorname{Aut}(X) \times \operatorname{Aut}(X) \longrightarrow \R_+$,
as in \cite{Trof--85}.
Choose $x \in X^0$. 
Recall that $B^x_{X^0}(n)$ denotes the closed ball of radius $n \in \N$ with respect to 
the combinatorial metric on $X^0$
(see Example \ref{metricrealizationgraph}). 
For $g,h \in \operatorname{Aut}(X)$, set 
\begin{equation*}
d_x(g,h) \, = \, \left\{
\aligned
0\phantom{^{-n}} \hskip.2cm &\text{if} \hskip.4cm g = h.
\\
2^{-n} \hskip.2cm &\text{if $n \ge 0$ is such that} \hskip.2cm
g(y) = h(y) \hskip.2cm \forall y \in B^x_{X^0}(n) 
\\
&\phantom{} \text{and}  \hskip.2cm 
g(z) \ne h(z) \hskip.2cm \text{for some} \hskip.2cm z \in B^x_{X^0}(n+1),
\\
n\phantom{^{-n}} \hskip.2cm &\text{if} \hskip.4cm 
d_X(x, g^{-1}h(x)) = n \ge 1 .
\endaligned
\right.
\end{equation*}
Then $d_x$ is a left-invariant proper compatible metric on $\operatorname{Aut}(X)$.
Since closed balls with respect to $d_x$ are both open and closed,
the group $\operatorname{Aut}(X)$ is totally disconnected.
\par
For two points $x,y \in X^0$, it is easy to check that the metrics $d_x$ and $d_y$
are equivalent.
\end{exe}

\begin{exe}
\label{profinitehausdim}
Let $G$ be a \textbf{profinite group}, i.e., a totally disconnecgted compact group.
(More on these in Example \ref{extdLCgps}(2).)
\index{Profinite group|textbf} \index{Compact group! profinite group}
Assume moreover that $G$ is second countable;
it follows that there exists
a sequence $G = G_0 \supset G_1 \supset \cdots \supset G_n  \supset \cdots$
of open normal subgroups of finite index such that
$\bigcap_{n \ge 0} G_n = \{1\}$.
For $g,h \in G$, set $n(g,h) = \max \{ n \ge 0 \mid g^{-1}h \in G_n \}$
if $g \ne h$ and $n(g,g) = \infty$.
The function
$$
d \, : \, G \times G \, \longrightarrow \, \R_+ , \hskip.2cm
(g,h) \longmapsto \frac{1}{n(g,h)}
$$
is a left- and right-invariant ultrametric on $G$, and the diameter of $(G, d)$ is $1$.
(This can be varied by defining $d(g,h)$ as an appropriate function of $n(g,h)$,
not necessarily the function $n \mapsto \frac{1}{n}$.)
The Hausdorff dimension of $(G, d)$ is investigated in \cite{Aber--94, BaSh--97}.
\par
Observe that $d$ is an adapted metric on $G$.
\end{exe}

\section{Compact generation }
\label{cg+sigmac}

\begin{defn}
\label{cggroups}
A \textbf{generating set}, or a \textbf{set of generators}, 
for a group $G$, is a subset $S \subset G$ such that
\begin{equation}
\label{Gunionpowersgenset}
G \, = \, \bigcup_{n \ge 0} {\widehat S}^n
\hskip.5cm \text{where} \hskip.2cm
\widehat S \, := \, S \cup S^{-1} \cup \{1\} ,
\end{equation}
or equivalently, such that for every $g \in G$
there exist $n \ge 0$ and $s_1, \hdots, s_n \in S \cup S^{-1}$ 
with $g = s_1 \cdots s_n$.
\par

To such a generating set,
we associate the free group $F_S$ over $S$ and the canonical projection
\index{Free group}
\begin{equation}
\label{pi_S}
\pi \, : \, F_S \longrightarrow G 
\end{equation}
which associates to $s \in S \subset F_S$ the element $s \in S \subset G$.
\par

A topological group is \textbf{compactly generated} 
\index{Compactly generated! topological group!|textbf}
if it has a compact generating set.
Recall from Remark \ref{neednotbelc} that, in this article,    
a compactly generated group need not be locally compact
(see Digression \ref{PestovEtc}).
\par

A  group is \textbf{finitely generated} \index{Finitely generated group|textbf}
if it has a finite generating set,
namely if it is compactly generated as a discrete group.
\end{defn}

\begin{rem}
\label{remsigmacetc}
(1)
The previous definition has been phrased 
with the assumption that $S$ is a subset of $G$.
In some situations, 
it is more convenient not to assume this,
so that a generating set for $G$ is rather a set $S$
given together with a surjective homomorphism
$F_S \longrightarrow G$;
the restriction of this surjection to $S$ need not be injective.
See Definition \ref{generation} and Remark \ref{remgenerationpresentation}.

\vskip.2cm

(2) 
For a topological group $G$, the following properties are equivalent:
\begin{enumerate}[noitemsep,label=(\roman*)]
\item
$G$ is locally compact and compactly generated;
\item
there exists a relatively compact and open subset of $G$ that is generating;
\item
there exists a relatively compact and open neighbourhood of $1$ that is generating.
\end{enumerate}
The proof is straightforward (see \cite[Theorem 5.13]{HeRo--63} if necessary).

\vskip.2cm

(3) Let $S$ be a generating set of a non-countable group
(for example of a non-discrete LC-group, see Remark \ref{cardinalBairegroup}).
Then the cardinality of $S$ is equal to that of $G$.
\index{Baire space, Baire group, Baire theorem}

\vskip.2cm

(4) Groups need not have minimal generating sets.
For instance, in the additive group $\Q$ of the rationals,
given a subset $S \subset \Q$ and an element $s \in \Q$, $s \notin S$,
it is easy to check that $S \cup \{s\}$ is a generating set
if and only if $S$ is a generating set.

\vskip.2cm

(5) 
Let $G$ be an LC-group. 
Consider a subset $S$ of $G$, the subgroup $H$ of $G$ generated by $S$,
and its closure $\overline H$.
If $S$ is compact, then there exists a compact subset of $G$ generating $\overline H$.
Indeed, if $T \subset \overline H$ is a compact subset  with non-empty interior in $\overline H$,
then $S \cup T$ is compact and generates $\overline H$.
\par
Similarly, let $N$ be a normal subgroup of $G$,
generated as a normal subgroup by a compact subset of $G$.
Then $\overline N$ is  a normal subgroup of $G$
compactly generated as a normal subgroup of $G$.
\end{rem}

In (1) below, ``filtering'' means that two compactly generated open subgroups of $G$
are contained in a single larger one.

\begin{prop}[compactly generated open subgroups]
\label{powersSincpgroup}
Let $G$ be an LC-group.
\begin{enumerate}[noitemsep,label=(\arabic*)]
\item\label{1DEpowersSincpgroup}
$G$ is the filtering union of its compactly generated open subgroups;
\index{Open subgroup! compactly generated}
in particular, $G$ contains compactly generated open subgroups.
\item\label{2DEpowersSincpgroup}
If $G$ is connected, every neighbourhood of $1$ in $G$ is a generating subset;
in particular, $G$ is compactly generated.
\end{enumerate}
Suppose moreover that $G$ is $\sigma$-compact.
\begin{enumerate}[noitemsep,label=(\arabic*)]
\addtocounter{enumi}{2}
\item\label{3DEpowersSincpgroup}
There exists a nested sequence
\begin{equation*}
G_1 \, \subset \, G_2 \, \subset \, \cdots \, \subset \,
G_n \, \subset \, G_{n+1} \, \subset \, \cdots
\end{equation*}
of compactly generated open subgroups of $G$ such that
$\bigcup_{n \ge 1} G_n = G$.
\end{enumerate}
Suppose moreover that $G$ has a compact generating set $S$.
\begin{enumerate}[noitemsep,label=(\arabic*)]
\addtocounter{enumi}{3}
\item\label{4DEpowersSincpgroup}
For $n$ large enough, ${\widehat S} ^n$ is a neighbourhood of $1$.
\item\label{5DEpowersSincpgroup}
For every compact subset $K$ of $G$, there exists $k \ge 0$
such that $K$ is contained in the interior of $\widehat S^k$.
\item\label{6DEpowersSincpgroup}
For every other compact generating set $T$  of $G$,
there exist $k,\ell \in \N$ such that $T \subset {\widehat S}^k$ and $S \subset {\widehat T}^\ell$.
\end{enumerate}
\end{prop}

\begin{proof}
\ref{1DEpowersSincpgroup}
Any element $g$ in $G$ is contained in some compact neighbourhood $V_g$ of $1$ in $G$,
and the union $\bigcup_{n \ge 0} (V_g \cup V_g^{-1})^n$ is an open subgroup of $G$.
Claims \ref{2DEpowersSincpgroup} and \ref{3DEpowersSincpgroup} are straightforward.
\par

\ref{4DEpowersSincpgroup}
By Baire's Theorem applied to the covering
$G = \bigcup_{n \ge 0} {\widehat S}^n$,
there exists $m \ge 0$ such that ${\widehat S}^m$ has a non-empty interior:
Hence the interior of ${\widehat S}^{n}$ is an open neighbourhood of $1$
for every $n \ge 2m$.
\par

\ref{5DEpowersSincpgroup}
For each $x \in K$, there exists $\ell(x) \ge 0$ such that
$x \in {\widehat S} ^{\ell(x)}$.
Let $m$ be as in \ref{4DEpowersSincpgroup};
then $x$ is in the interior of ${\widehat S}^{\ell(x) + 2m}$.
Hence $K \subset \bigcup_{k \ge 0} \operatorname{Int}({\widehat S}^k)$;
since $K$ is compact, the conclusion follows.
\par

Claim \ref{6DEpowersSincpgroup} follows from \ref{5DEpowersSincpgroup}.
\end{proof}

\begin{rem}
\label{remopensub}
(a)
Concerning Claim \ref{1DEpowersSincpgroup} of Proposition \ref{powersSincpgroup},
it is shown below that
an LC-group contains also open subgroups that are connected-by-compact
(Corollary \ref{vanDantzigCor}).
\par
Connected-by-compact LC-groups
are compactly generated (Proposition \ref{almostconnectedgroups}),
and compactly presented (Proposition \ref{connbycompcp}).

\vskip.2cm

(b) 
There is a converse to Claim \ref{2DEpowersSincpgroup} of Proposition \ref{powersSincpgroup}:
an LC-group in which every neighbourhood of $1$ is generating is connected
\cite[Page III.36]{BTG1-4}.
\end{rem}

\begin{defn}
\label{countableindex}
The \textbf{index} 
\index{Index of a subgroup, countable index}
 \index{Subgroup! of countable index}
of a subgroup $G_0$ of a group $G$ is the cardinal of the homogeneous set $G/G_0$.
In particular, in a topological group $G$, an open subgroup $G_0$ is of \textbf{countable index}
if the discrete homogeneous space $G/G_0$ is countable (possibly finite).
\par
A group $G$ is \textbf{finitely generated over a subgroup $G_0$} 
\index{Finite generation over a subgroup|textbf}
if there exists a finite subset $F$ of $G$ such that $F \cup G_0$ generates $G$.
\end{defn}

\begin{cor}
\label{opensubgroupsinLCgroups}
\begin{enumerate}[noitemsep,label=(\arabic*)]
\item\label{1DEopensubgroupsinLCgroups}
For an LC-group $G$, the following properties are equivalent:
\begin{enumerate}[noitemsep,label=(\roman*)]
\item\label{iDEopensubgroupsinLCgroups}
$G$ is $\sigma$-compact;
\item\label{iiDEopensubgroupsinLCgroups}
every open subgroup of $G$ is of countable index;
\item\label{iiiDEopensubgroupsinLCgroups}
there exists a compactly generated open subgroup of $G$ of countable index.
\end{enumerate}
\index{Sigma-compact! LC-group}
\item\label{2DEopensubgroupsinLCgroups}
For an LC-group $G$, the following properties are equivalent:
\begin{enumerate}[noitemsep,label=(\roman*)]
\addtocounter{enumi}{3}
\item\label{ivDEopensubgroupsinLCgroups}
$G$ is compactly generated;
\item\label{vDEopensubgroupsinLCgroups}
$G$ is finitely generated over every open subgroup;
\item\label{viDEopensubgroupsinLCgroups}
there exists a compactly generated open subgroup of $G$ 
over which $G$ is finitely generated.
\end{enumerate}
\index{Compactly generated! LC-group}
\end{enumerate}
\end{cor}

Note: compare \ref{1DEopensubgroupsinLCgroups} above
with (1) in Corollary \ref{vanDantzigCor}.

\begin{proof}
\ref{1DEopensubgroupsinLCgroups}
Implication \ref{iDEopensubgroupsinLCgroups} $\Rightarrow$ \ref{iiDEopensubgroupsinLCgroups} 
holds for every topological group.
Indeed, let $H$ be an open subgroup of a $\sigma$-compact group $G$;
the homogeneous space $G/H$ is countable, 
because it is both discrete and $\sigma$-compact.
Implication \ref{iiiDEopensubgroupsinLCgroups} $\Rightarrow$ \ref{iDEopensubgroupsinLCgroups} 
is equally straightforward.
For $G$ locally compact,
\ref{iiDEopensubgroupsinLCgroups} $\Rightarrow$\ref{iiiDEopensubgroupsinLCgroups} 
follows from Proposition \ref{powersSincpgroup}(1).
\par

The proof of \ref{2DEopensubgroupsinLCgroups} is analogous.
\end{proof}

\begin{defn}
\label{cocompact}
In a topological group $G$, a subgroup $H$ is \textbf{cocompact} 
\index{Cocompact! subgroup of a topological group|textbf}
if there exists a compact subset $K$ of $G$ such that $G = KH$.
\end{defn}

For \ref{stababcd}, we consider the following properties of topological \emph{spaces}:
\begin{enumerate}[noitemsep,label=(\alph*)]
\item\label{aDEstababcd}
local compactness,
\item\label{bDEstababcd}
metrizability,
\item\label{cDEstababcd}
first countability,
\index{First-countable! topological group}
\item\label{dDEstababcd}
$\sigma$-compactness,
\item\label{eDEstababcd}
second countability,
\index{Second-countable! topological group}
\item\label{fDEstababcd}
separability,
\end{enumerate}
as in \ref{hereditarityTopSpace},
as well as the property of 
\begin{enumerate}[noitemsep,label=(\alph*)]
\addtocounter{enumi}{6}
\item\label{gDEstababcd}
compact generation 
\end{enumerate}
for topological \emph{groups}.

\begin{prop}[on hereditarity]
\label{stababcd}
\index{Hereditarity of various properties}

Let $G$ be a topological group, $H$ a closed subgroup,
$G/H$ the corresponding homogeneous space,
and $\pi : G \longrightarrow G/H$ the canonical projection.
\begin{enumerate}[noitemsep,label=(\arabic*)]
\item\label{1DEstababcd}
If $G$ is has one of Properties \ref{aDEstababcd} to \ref{eDEstababcd}
then $H$ and $G/H$ have the same property.
When $G$ is locally compact, this extends to Property \ref{fDEstababcd}.
\item\label{2DEstababcd}
If $H$ and $G/H$ have both one of Properties  
\ref{aDEstababcd}, \ref{bDEstababcd}, \ref{cDEstababcd}, \ref{eDEstababcd}, \ref{fDEstababcd},
then $G$ has the same property.
When $G$ is locally compact, this extends to Property \ref{dDEstababcd}.
\item\label{3DEstababcd}
Suppose that $H$ is cocompact closed in $G$.
Then $G$ is compactly generated if and only if $H$ is so.
\index{Subgroup! cocompact closed}
\item\label{4DEstababcd}
Let $N$ be a closed normal subgroup of $G$.
If $G$ is compactly generated \ref{gDEstababcd}, 
so is $G/N$.
Suppose moreover that $G$ is locally compact;
if $N$ and $G/N$ are compactly generated,
$G$ is compactly generated.
\end{enumerate}
\end{prop} 


Concerning Claim \ref{4DEstababcd}, 
note that a closed subgroup (even a normal one) of a compactly generated LC-group
need not be compactly generated (Example \ref{exDehnetc}).
\par
We come back below on hereditarity of 
compact generation (Proposition \ref{sigmac+compactgofcocompact})
and compact presentation (Corollary \ref{hereditaritycp}).

\begin{lem}[lifting compact subspaces]
\label{KimagedeK}
Let $G$, $H$, and $\pi : G \longrightarrow G/H$,
be as in the previous proposition.
For the condition
\begin{itemize}
\item[(*)]
every compact subset of the coset space $G/H$ is the image by $\pi$
of a compact subset of the group $G$
\end{itemize}
to hold, 
it is sufficient that $G$ be locally compact, or that $G$ be completely metrizable.
\end{lem}

\begin{proof}[On the proof]
The lemma for $G$ locally compact is a particular case of a more general fact,
proved in \cite[Page I.80]{BTG1-4}:
let $X$ be an LC-space and $\mathcal R$ an open Hausdorff equivalence relation on $X$;
then every compact subspace of $X / \mathcal R$ is the image 
of a compact subspace of $X$.
\par
For $G$ a completely metrizable group,
the lemma is a particular case of Proposition 18 in \cite[Page IX.22]{BTG5-10}.
\par 
Note that the hypothesis on $G$ (locally compact or completely metrizable)
cannot be omitted; see \cite[Pages 64 and 324]{BeHV--08} for comments and references.
\end{proof}

\begin{proof}[Proof of Proposition \ref{stababcd}]
(1) 
It is an easy exercise to check that, if $G$ has one of Properties (a) to (e),
the same holds for $H$ and $G/H$;  
for the metrizability of $G/H$, see if necessary \cite[Theorem 1.23]{MoZi--55}.
We refer to \cite{CoIt--77}, both for the claim in case of (f)
and for the fact it \emph{does not} carry over to topological groups in general.

\vskip.2cm

(2) 
Claim (2) for (a) is originally due to Gleason \cite{Glea--49}.
We refer to \cite[Theorem of 2.2, Page 52]{MoZi--55},
or \cite[Theorem 2.6 of Chapter I]{Hoch--65}.
\par

For (c), we follow an argument of Graev,
as reported in \cite[Theorem 5.38.e]{HeRo--63}.
Let $(U_n)_{n \ge 0}$ and $(V_n)_{n \ge 0}$
be sequences of neighbourhoods of $1$ in $G$ such that
$(\pi(U_n))_{n \ge 0}$ is a basis of neighbourhoods of $\pi(1)$ in $G/H$
and $(V_n \cap H)_{n \ge 0}$ is a basis of neighbourhoods of $1$ in $H$;
we assume furthermore that $V_n$ is symmetric and $(V_{n+1})^2 \subset V_n$
for all $n \ge 0$.
For $n \ge 0$, set
\begin{equation*}
\aligned
\mathcal O_n \, &= \, V_{n+1} ( H \cap (G \smallsetminus V_n)) ,
\hskip.5cm
C_n \, = \, \overline{\mathcal O_n} ,
\\
P_n \, &= \, (U_nH) \cap (G \smallsetminus C_n) ,
\hskip.8cm
Q_n \, = \, P_0 \cap P_1 \cap \cdots \cap P_n .
\endaligned
\end{equation*}
Observe that $\mathcal O_n \cap V_{n+1} = \emptyset$
(as a consequence of $(V_{m+1})^2 \subset V_m$),
hence $C_n \cap V_{n+1} = \emptyset$.
Thus $P_n$ and $Q_n$
are open neighbourhoods of $1$ in $G$.
Let us check that $(Q_n)_{n \ge 0}$ is a basis of neighbourhoods of $1$ in $G$.
\par
Let $X$ be a neighbourhood of $1$ in  $G$;
we claim that $Q_n \subset X$ for $n$ large enough.
Choose a neighbourhood $Y$ of $1$ in $G$ such that $Y^2 \subset X$.
By hypothesis, there exist integers $m \ge 1$ and $k \ge m+1$ such that
$V_m \cap H \subset Y \cap H$
and $\pi(U_k) \subset \pi(Y \cap V_{m+1})$.
We have
\begin{equation*}
\aligned
Q_k \, \subset \, P_k \cap P_m \, &\subset \, 
 \big( U_k H \big) \cap 
 \Big( G \smallsetminus \overline{ V_{m+1}(H \cap (G \smallsetminus V_m)) } \Big)
\\
  &\subset \, \Big( (Y \cap V_{m+1}) H \Big) \cap
   \Big( G \smallsetminus  (Y \cap V_{m+1})(H \cap (G \smallsetminus V_m))  \Big)
\\
  &\subset \, \big( Y \cap V_{m+1} \big) \big( H \cap V_m\big)
  \, \subset \, Y^2 \, \subset \, X
\endaligned
\end{equation*}
and this proves the claim.
\par

Claim (2) for (b) is now a consequence of the Birkhoff-Kakutani Theorem \ref{BK}.
\par

In the special case where $G/H$ is paracompact 
and the bundle $G \longrightarrow G/H$ is locally trivial,
one can alternatively define a metric on $G$ 
in terms of metrics in $H$ and $G/H$,
and of a partition of unity by functions
with supports inside trivializing subsets of $G/H$.
However, in the general case of a locally compact group $G$,
the bundle $G \longrightarrow G/H$ need not be locally trivial \cite{Karu--58}.
\par

The case of (f) is easily checked:
if there exist a countable subset $\{a_m\}_{m \ge 1}$ in $G$
such that $\{\pi(a_m)\}_{m \ge 1}$ is dense in $G/H$
and a countable subset $\{b_n\}_{n \ge 1}$ dense in $H$,
then the set $\{a_mb_n\}_{m,n \ge 1}$ is countable and dense in $G$.
\par

In view of Lemma \ref{KimagedeK}, 
the case of (d) is similar.
%
%
\par

By Theorem \ref{Urysohn}, Case (e)
follows from Cases (b) and (f).

\vskip.2cm

(3)
Let $K$ be as in Definition \ref{cocompact}, so that $G = KH$.
If $H$ is generated by a compact set $T$,
then $G$ is generated by $T \cup K$.
Conversely, suppose that $G$ is generated by a compact set $S$.
There is no loss of generality in assuming $S$ symmetric and $K \ni 1$.
Set $T = KSK^{-1} \cap H$,
where $KSK^{-1}$ is the subset of $G$ of elements of the form
$g = k_1 s k_2^{-1}$, with $k_1, k_2 \in K$ and $s \in S$;
observe that $T$ is compact. 
It remains to check that $T$ generates $H$.
\par
Let $h \in H$.
There exist $s_1, \hdots, s_n \in S$ with $h = s_1 \cdots s_n$.
Set $k_0 = 1$;
since $HK = G$, we can choose inductively $k_1, \hdots, k_{n-1} \in K$
such that $k_{i-1}s_ik_i^{-1} \in H$;
set $k_n = 1$.
We have $h = \prod_{i=1}^n k_{i-1} s_i k_i^{-1}$,
and $k_{i-1}s_ik_i^{-1} \in T$ for $i=1, \hdots, n$.
\par
For $G$ locally compact, there is another proof of (3)
in Proposition \ref{sigmac+compactgofcocompact} below.

\vskip.2cm

(4) The first part of the claim is straightforward,
and the second part is a consequence of Lemma \ref{KimagedeK}.
\end{proof}

\begin{exe}
\label{exampleAut(G)}
Let $G$ be an LC-group.
Denote by $\operatorname{Aut}(G)$ the group of its continuous \textbf{automorphisms},
\index{Automorphism group! of a group|textbf}
with the topology generated by the subbasis
\begin{equation*}
N(K, V)  \, = \, \left\{ \alpha  \in \operatorname{Aut}(G) 
\hskip.2cm \Big\vert \hskip.2cm
\aligned
&\alpha(g) \in Vg \hskip.2cm \text{and} \hskip.2cm \alpha^{-1} (g) \in Vg
\\
&\text{for all} \hskip.2cm g \in K 
\endaligned
\right\} ,
\end{equation*}
where $K$ is a compact subset of $G$ and $V$ an open neighbourhood of $1$ in $G$.
Then $\operatorname{Aut}(G)$ is a topological group \cite[Section III.3, Page 40]{Hoch--65}.
This topology is often called the \textbf{Braconnier topology} on $\operatorname{Aut}(G)$, 
and the \emph{Birkhoff topology} in \cite[\S~IV.1]{Brac--48}.
\par

The topology defined above on $\operatorname{Aut}(G)$ coincides with
the compact-open topology in some cases,
for example if $G$ is connected-by-compact, or connected-by-discrete \cite{Wang--69}.
However, Wang shows that they do not coincide for the following example:
$G = F^{\N} \times F^{(\N)}$,
where $F$ is a finite group, $F \ne \{1\}$, and $G$ is the direct product
of the compact group $F^{\mathbf N}$ 
\index{Product of groups! $F^A$, with $F$ finite}
with the discrete countable group $F^{(\mathbf N)}$.
\index{Topology! compact-open}
\par

For an LC-group $G$, the topological group $\operatorname{Aut}(G)$
\emph{need not} be locally compact.
\index{LC-group! non-LC-group}
For example, if $G$ is the free abelian group $\mathbf Z ^{(\mathbf N)}$ 
\index{Free abelian group}
on an infinite countable set, or its Pontryagin dual $(\mathbf R / \mathbf Z)^{\mathbf N}$,
then $\operatorname{Aut}(G)$ is not locally compact
(see  \cite[Example 26.18]{HeRo--63}).
\par

There are large classes of LC-groups $G$ for which
$\operatorname{Aut}(G)$ is locally compact.
If $G$ is finitely generated, $\operatorname{Aut}(G)$ is discrete.
If $G$ is a Lie group such that $G/G_0$ is finitely generated,
$\operatorname{Aut}(G)$ is a Lie group
\cite[Ch.\ III, $\S$~10, n$^o$~2]{BGLA2-3}.
\par

A recent short guide on the Braconnier topology on $\operatorname{Aut}(G)$ can be found
in \cite[Appendix I]{CaMo--11}
\par

There are known sufficient conditions for the group 
$\operatorname{Int}(G)$ of inner automorphisms to be closed in $\operatorname{Aut}(G)$;
for example, this is obviously so when $G$ is compact 
(more on this in Example \ref{extdLCgps}(6) below).
In general, $\operatorname{Int}(G)$ need not be closed in $\operatorname{Aut}(G)$,
as shown by an example of the form $G = (\C \times \C) \rtimes \R$
given in \cite[Pages 127--128]{Helg--62}. 
Further properties of $\operatorname{Aut}(G)$
and sufficient conditions for 
$\operatorname{Int}(G)$ to be closed in $\operatorname{Aut}(G)$
are discussed in \cite{PeSu--78}.
\end{exe}

\begin{rem}[compact generation, $\sigma$-compactness, second countability]
\label{Gexplescgsigmacsc}
(1)
It follows from Equation (\ref{Gunionpowersgenset}) 
of Definition \ref{cggroups} that 
a compactly generated group is $\sigma$-compact.
The converse does \emph{not} hold; for example:
\begin{enumerate}[noitemsep,label=(\alph*)]
\item\label{aDEGexplescgsigmacsc}
Discrete groups such as the additive group $\Q$ of rational numbers
and its  subgroups $\Z\left[ \frac{1}{n} \right]$
are countable, and are not finitely generated
($n$ is an integer, $n \ge 2$).
\index{$z$@$\Z[1/n]$}
\item\label{bDEGexplescgsigmacsc}
For a prime $p$, the additive group $\Q_p$ of $p$-adic numbers
is $\sigma$-compact, second-countable, and is not compactly generated.
\index{$q$@$\Q_p$, field of $p$-adic numbers}
We come back to this in Examples \ref{localfieldscg}
(on $\Q_p$,  $\Q_p^\times$, and compact generation) 
and  \ref{examplesoflocallyellipticgroups}
(local ellipticity of $\Q_p$).
This carries over to the additive group $\Q_m$ of the $m$-adic numbers,
for every integer $m \ge 2$.
\end{enumerate}

\vskip.2cm

(2)
For an LC-group, compact generation and second countability
are independent properties. 
The four possibilities are illustrated by the following four examples:
$\Z$, $\Q$, $\R_{\text{dis}}$, that are discrete groups,
and $(\Z/2\Z)^{\R_{\text{dis}}}$,  that is a compact group.
Here, $\R_{\text{dis}}$ denotes
the additive group of real numbers with the discrete topology.

\vskip.2cm

(3)
An LC-group that contains a finitely generated dense subgroup is compactly generated.
\index{Compactly generated! LC-group}
\index{Subgroup! finitely generated dense}
\par

The converse does not hold.
Consider again a compact group $G = F^A$, with $F$ a non-trivial finite group and $A$ a set,
as in Example \ref{sepnot2nd}.
We leave it as an exercise for the reader to check
that any finitely generated subgroup of $F^A$ is finite,
more precisely that any subgroup of $F^A$ generated by $k$ elements
is isomorphic to a subgroup of $F^N$, with $N = \vert F \vert^k$.
\par

If $A$ is infinite countable, say $A = \N$, 
then $F^{\N}$ is a compact group that is infinite, metrizable, and second-countable,
but without finitely generated \emph{dense subgroups}.
\index{Compact group! without f.g.~dense subgroup}
(If the cardinality of $A$ is strictly larger than that of the continuum,
then $G$ is non-separable by Proposition 
\ref{hereditarityTopSpace}\ref{8DEhereditarityTopSpace},
i.e., $G$ does not contain any countable \emph{dense subset}.)

\vskip.2cm

(4)
Any connected Lie group $G$ contains finitely generated dense subgroups.
Indeed, let $X_1, \hdots, X_k$ be generators of the Lie algebra of $G$.
For each $j \in \{1, \hdots, k\}$,
the one-parameter subgroup $(\exp (tX_j))_{t \in \R} \simeq \R$
contains two elements, 
say $a_j = \exp (X_j)$ and  $b_j = \exp( \sqrt{2} X_j)$,
which generate a dense subgroup.
It follows that $\{a_1, b_1, \hdots, a_k, b_k\}$ generates 
a dense subgroup of $G$.
\par
A non-solvable connected Lie group $G$ contains finitely generated free dense subgroups;
see \cite{Kura--51} for $G$ semisimple, and \cite{BrGe--03} in general.
\index{Solvable group}
\index{Free subgroup}
For more on finitely generated dense subgroups of Lie groups, 
see \cite{Breu} and \cite{Glut--11}.

\vskip.2cm 

(5)
Let $G$ be a $\sigma$-compact group $G$.
Choose a compact normal subgroup $K$ of $G$ 
such that $G/K$ is metrizable (Theorem \ref{KK}).
Then $G$ is second-countable if and only if $K$ is so;
in loose words, second countability for $G$ can be seen on $K$.
(For $\sigma$-compactness and second countability, 
see also Example \ref{sepnot2nd}.)

\vskip.2cm

(6) An LC-group $G$ is \textbf{just non-compact} if it is not compact and
\index{Just non-compact LC-groups|textbf}
every closed normal subgroup $N \ne \{1\}$ of $G$ is cocompact.
Just non-compact LC-groups include just infinite discrete groups, $\R$, 
and wreath products $S \wr_X F$, with $S$ a simple LC-group 
and $F$ a finite group acting transitively on a set $X$.
Compactly generated LC-groups that are just non-compact are discussed in \cite{CaMo--11}.
\end{rem}

\begin{exe}
\label{algebraicgroupslocalfields}
Let $\K$ be a non-discrete locally compact field,
for example $\R$, $\Q_p$, or $\mathbf F_q \lp t \rp$;
\index{$fp$@$\mathbf F_p \lp t \rp$, $\mathbf F_q \lp t \rp$}
\index{$k$@$\K$, non-discrete locally compact field, often local field}
by ``field'' \index{Field = commutative field} 
we always mean ``commutative field''.
Then $\K$ is $\sigma$-compact 
and metrizable, and therefore second-countable (Theorem \ref{Urysohn});
see \cite[Chapter I, Proposition 2]{Weil--67}
or \cite[chap.\ 7, $\S$~1, no 10]{BInt7-8}.
\par

For $n \ge 1$, the general linear group $\GL_n(\K)$ 
\index{General linear group $\GL$! $\GL_n(\K)$ 
for $\K$ a non-discrete locally compact field}
and the special linear group $\SL_n(\K)$
\index{Special linear group $\SL$! $\SL_n(\K)$}
are locally compact, $\sigma$-compact, metrizable, second-countable, and separable.
They are moreover compactly generated (Example \ref{exSL_2}). 
\par 

Let $\mathbf G$ be an algebraic group defined over $\K$,
and let $G$ be the group of its $\K$-points, with its $\K$-topology.
Thus $G$ is a closed subgroup of $\GL_n(\K)$ for some $n$.
It follows from Proposition \ref{stababcd} that $G$ is
locally compact, $\sigma$-compact, second-countable, metrizable, and separable.
\par
The group $G$ need not be compactly generated
(see Example \ref{exDehnetc}(2)).
\end{exe}

\begin{exe}
\label{exDehnetc}
There are non-compactly generated closed subgroups
of compactly generated LC-groups.
\index{Hereditarity of various properties}

\vskip.2cm

(1)
Historically, the first example we know 
appears in an observation of Max Dehn concerning discrete groups:
a non-abelian free group of finite rank 
\index{Free group}
has subgroups that are not finitely generated
(\cite{Dehn--11} and \cite[Page 135]{DeSt--87}).
This is straightforward in terms of covering spaces of wedges of circles.
Indeed, if $X$ is a wedge of a finite number $k \ge 2$ of circles, 
every normal subgroup $\Gamma$ of $\pi_1(X)$
is the fundamental group of a Galois covering $Y \longrightarrow X$,
and $\Gamma$ is not finitely generated if and only if
$Y$ is neither simply connected nor compact,
if and only if $\Gamma$ is non-trivial and of infinite index in $\pi_1(X)$.

\vskip.2cm

(2)
Consider a prime $p$, the field $\Q_p$ of $p$-adic numbers, and the affine group 
\index{Affine group! $\K \rtimes {\K}^\times$}
\index{Semidirect product! $\Q_p^\times \ltimes \Q_p$}
\begin{equation*}
\operatorname{Aff}_1(\Q_p) \, = \, 
 \begin{pmatrix}
  \Q_p^\times & \Q_p \\
  0 & 1 
\end{pmatrix} \, \simeq \, 
\Q_p^\times  \ltimes \Q_p .
\end{equation*}
The usual $p$-adic topology on $\Q_p$ and $\Q_p^\times$
induces a topology on this affine group
which makes it a compactly generated LC-group
(see Example \ref{affinecg}).
The translation subgroup, identified with $\Q_p$,
is not compactly generated.
\index{p-adic topology on the group of $\Q_p$-points 
of an algebraic group}
\par

Similarly, in 
\begin{equation*}
\operatorname{BS}(1,p) \, = \,  
\begin{pmatrix} p^{\Z} & \Z[1/p] \\  0 & 1 \end{pmatrix} 
\, \simeq \,  \Z  \ltimes_p \Z[1/p] 
\, \simeq \,  \langle s,t  \mid t^{-1}s t = s^p \rangle ,
\end{equation*}
a solvable Baumslag-Solitar group generated by 
$s = \begin{pmatrix} 1 & 1 \\  0 & 1 \end{pmatrix}$ and 
$t = \begin{pmatrix} p & 0 \\  0 & 1 \end{pmatrix}$,
the subgroup $\begin{pmatrix} 1 & \Z[1/p] \\  0 & 1 \end{pmatrix} \simeq \Z[1/p]$
is not finitely generated.
\index{Baumslag-Solitar group}

\vskip.2cm

(3)
However, a closed subgroup of a compactly generated 
\emph{abelian or nilpotent} LC-group
is compactly generated (Proposition \ref{closubabcgiscg}).
\end{exe}

\begin{defn}
\label{deffibreproduct}
Let $G_1, G_2, H$ be LC-groups.
For $j = 1,2$, let $p_j : G_j \twoheadrightarrow H$ be a continuous surjective homomorphism. 
The corresponding \textbf{fibre product} is the group defined by
\index{Fibre product|textbf}
\begin{equation*}
G_1 \times_H G_2 \, = \, \{ (g_1, g_2) \in G_1 \times G_2 \mid
p_1(g_1) = p_2(g_2) \} .
\end{equation*}
As a closed subgroup of $G_1 \times G_2$, it is an LC-group.
\end{defn}

\begin{rem}
\label{remfibreproduct}
We keep the notation of the previous definition.

\vskip.2cm

(1) 
Assume furthermore that the homomorphisms $p_1$ and $p_2$ are open
(this is automatic in case $G_1$ and $G_2$ are $\sigma$-compact,
by Corollary \ref{Freudenthalcor} below).
For $\{j,k\} = \{1,2\}$, we have a short exact sequence
\begin{equation*}
N_j \, \lhook\joinrel\relbar\joinrel\rightarrow  \, G_1 \times_H G_2
\, \overset{\pi_k}{\relbar\joinrel\twoheadrightarrow} \, G_k 
\end{equation*}
where  $N_j = \ker (p_j)$,
and where $\pi_k$ denotes the canonical projection.
Indeed, we have $\ker (\pi_2) = \ker (p_1) \times \{1\}$, 
and similarly $\ker (\pi_1) = \{1\} \times \ker (p_2)$.
\par

It follows that $G_1 \times_H G_2$ is compactly generated
as soon as $G_1$ and $N_2$ are so,
or $G_2$ and $N_1$ are so
(Proposition \ref{stababcd}\ref{4DEstababcd}). 

\vskip.2cm

(2)
Suppose that $G_1$ and $G_2$ are compactly generated.
The fibre product $G_1 \times_H G_2$ need not be compactly generated,
as the following example shows
(the argument anticipates on a proposition of Chapter \ref{chap_cpgroups}).
\par

Let $F$ be a finitely presented group and
$p : F \twoheadrightarrow Q$ a surjective homomorphism. Set 
\begin{equation*}
E \, = \, F \times_Q F \, = \, 
\{ (x,y) \in F \times F \mid p(x) = p(y) \} .
\end{equation*}
Assume that $Q$ is not finitely presented;
we claim that $E$ is not finitely generated.
\par
Denote by $N$ the normal subgroup $\ker (p) \times \{1\}$ of $E$,
and let $\Delta$ the diagonal subgroup of elements of the form $(y,y)$, with $y \in F$.
Then $N \cap \Delta = \{1\}$, and $E = N \rtimes \Delta$
since every $(x,y) \in E$ can be written as the product $(xy^{-1}, 1)(y,y) \in N \Delta$.
By Proposition \ref{cpforgroupsandquotients}, 
the group $\ker (p)$ is not finitely generated as a normal subgroup of $F$.
Hence there exists a strictly increasing nested sequence
\begin{equation*}
K_1 \, \subsetneqq \, K_2 \, \subsetneqq \, \cdots \subsetneqq \, 
K_n \, \subsetneqq \, K_{n+1}  \, \subsetneqq \, \cdots \subset \, 
\bigcup_{n \ge 1} K_n = \ker (p)
\end{equation*}
of normal subgroups of $F$.
For each $n \ge 1$, set $E_n = \left( K_n \times \{1\} \right) \rtimes \Delta$.
Then we have
\begin{equation*}
E_1 \, \subsetneqq \, E_2 \, \subsetneqq \, \cdots \subsetneqq \, 
E_n \, \subsetneqq \, E_{n+1}  \, \subsetneqq \, \cdots \subset \, 
E = \bigcup_{n \ge 1} E_n ,
\end{equation*}
and this shows that $E$ is not finitely generated.
\end{rem}

For $G_1, G_2$ compactly generated LC-groups, short of having
compactly generated fibre products, we have the following:

\begin{prop}
\label{pullbackCgLCgroups}
Let $G_1, G_2, H$ be compactly generated LC-groups.
For $j = 1,2$, let $p_j : G_j \twoheadrightarrow H$ be a continuous surjective homomorphism. 
\par

There exist a compactly generated LC-group $E$
and two continuous surjective homomorphisms $\pi_1, \pi_2$ such that the diagram
\begin{equation*}
\dia{
E
\ar@{->>}[d]^{\pi_1}
\ar@{->>}[r]^{\pi_2}
& G_2 
\ar@{->>}[d]^{p_2}
\\
G_1
\ar@{->>}[r]^{p_1}
& H
}
\end{equation*}
commutes.
Moreover, for $\{i,j\} = \{1,2\}$, there exists an isomorphism of topological groups
from $\ker (\pi_i)$ to an open subgroup of $\ker (p_j)$.
\end{prop}

\begin{proof}
Let
\begin{equation*}
F \, := \, G_1 \times_H G_2 \, = \, 
\{ (g_1, g_2) \in G_1 \times G_2 \mid p_1(g_1)= p_2(g_2)) \} \, \subset \, G_1 \times G_2
\end{equation*}
denote the relevant fibre product.
For  $i \in \{1,2\}$, denote by $\pi_i$ the canonical projection
$G_1 \times G_2 \twoheadrightarrow G_i$ as well as its restriction to $F$
(and also its restriction to $E$ as soon as $E$ will be defined).
For $j=1,2$, there exists by Lemma \ref{KimagedeK} a compact subset $S_j$ of $F$
such that $\pi_j(S_j)$ generates $G_j$.
Let $U$ be a compact subset of $F$ containing $S_1 \cup S_2$; 
assume moreover that $U$ has a non-empty interior.
Then we can define $E$ as the open subgroup of $F$ generated by $U$.
\par
The last claim follows from the equalities
$\ker (\pi_2) = ( \ker (p_1) \times \{1\} ) \cap D$
and $\ker (\pi_1) = ( \{1\} \times \ker (p_2) ) \cap D$.
See Remark \ref{remfibreproduct}(1).
\end{proof}

\begin{dig}[compact generation beyond local compactness]
\label{PestovEtc}
A topological group $G$ generated by a compact subset $K$ 
need not be locally compact, as shown by Examples (1) and (3) below.
\index{LC-group! non-LC-group}

\vskip.2cm

(1) 
Let $\mathcal H$ be a Hilbert  space, with the weak topology;
\index{Hilbert space}
it is a topological group under addition.
The unit ball $\mathcal H_1$ is a compact generating set of  $\mathcal H$.
If $\mathcal H$ is infinite-dimensional,
this group is compactly generated and is not locally compact
(see (2) below).

This carries over to the dual of any infinite-dimensional Banach space 
with the w$^*$-topology, by a theorem of Alaoglu.
\par
Note that, for the norm topology, an infinite-dimensional Banach space 
is never compactly generated.
This follows from (2) just below and Remark \ref{onmetrizability}(7).

\vskip.2cm

(2)
Recall the following theorem. 
Let $\K$ be a non-discrete complete valued field
and $E \ne \{0\}$ a Hausdorff topological vector space over $\K$;
if $E$ is locally compact, then $\K$ is a non-discrete locally compact field 
and $E$ is finite-dimensional \cite[chapitre 1, $\S$~2, n$^o$ 4]{BEVT}.
Recall also that, with the weak topology, 
an infinite-dimensional Hilbert space is non-metrizable;
its unit ball is metrizable if and only if the space is separable
\cite[Problems 18, 19,  and 21]{Halm--74}.

\vskip.2cm

(3)
Free topological groups\footnote{In
this book, with the exception of the present digression, 
we \emph{do not} consider any topology on a free group $F_S$,
even when $G = F_S /N$ is a topological group 
rather than just a group.}
\index{Free group}
and free abelian topological groups
\index{Free abelian group}
have been defined and investigated in the early 1940's, by A.A. Markov. 
For references, let us indicate Section II.8 in \cite{HeRo--63},
Abels' article \cite{Abel--72'}, 
an introduction to free topological groups \cite[Section 25]{Todo--97},
and an extended review \cite{Sipa--03}.
\par
Let $X$ be a topological space,
identified with a subset of the free group $F_X$ on $X$.
Then $F_X$ is given the finest group topology which induces on $X$
a topology coarser than the given one.
Then, every continuous mapping from $X$ to a topological group $G$
extends uniquely to a continuous homomorphism from $F_X$ to $G$;
moreover, if $X$ is Hausdorff and completely regular,
the inclusion of $X$ in $F_X$ is a homeomorphism.
The topological group $F_X$ is locally compact 
if and only if the topological space $X$ is discrete \cite[Page 5793]{Sipa--03}.
In particular, if $X$ is compact and infinite, 
$F_X$ is a topological group which is compactly generated and not locally compact.

\vskip.2cm

(4)
Let $\operatorname{Cr}_2(\C)$ be the \emph{plane Cremona group over $\C$,} 
\index{Cremona group of the plane}
i.e., the group of birational transformations of the complex projective plane.
There exists a Hausdorff topology on $\operatorname{Cr}_2(\C)$ 
making it a topological group,
and inducing the usual topology on its standard subgroups 
$\PGL_3(\C)$ and $\PGL_2(\C) \times \PGL_2(\C)$.
The group $\operatorname{Cr}_2(\C)$ is compactly generated 
and is \emph{not} locally compact;
see \cite{BlFu--13}, in particular their Lemmas 5.15 and 5.17.
\end{dig}

\begin{rem}
\label{problem}
(1)
Any countable group 
is a subgroup of a $2$-generator group.
This is the \emph{HNN embedding theorem} of \cite{HiNN--49};
\index{HNN! embedding theorem}
\index{Theorem! HNN embedding}
see \cite[Theorem 11.71]{Rotm--95},
or the very short exposition in \cite{Galv--93},
which includes a nine-lines-long proof.
The theorem is also an immediate consequence of the following result,
where $\operatorname{Sym}(X)$ stands for 
the symmetric group of all permutations of an infinite countable set $X$:
every countable subgroup of $\operatorname{Sym}(X)$ 
is contained in a $2$-generator subgroup of $\operatorname{Sym}(X)$;
\index{Symmetric group $\operatorname{Sym}(X)$ of a set $X$}
see Example \ref{Symuncountablecof}.

\vskip.2cm

(2)
It is natural to ask whether (1) extends to the locally compact setting,
more precisely to ask whether 
every $\sigma$-compact LC-group embeds
as a \emph{closed} subgroup of a compactly generated LC-group.
The answer is negative, as shown in \cite{CaCo--14}.

\vskip.2cm

(3)
Pestov has shown that every $\sigma$-compact topological group $H$
is isomorphic to a closed subgroup of a compactly generated topological group $G$;
in this statement, neither $H$ nor $G$ need be locally compact
(\cite{Pest--86}, see also \cite[Section 24]{Todo--97}).
\par
Also: every separable topological group embeds as a subgroup
in a topological group containing a dense subgroup generated by $2$ elements \cite{MoPe--98}.
\end{rem}

\begin{rem}
\label{ConditionsCn}
For discrete groups, finite generation is the first of a sequence $(F_n)_{n \ge 1}$
of increasingly stronger properties, introduced by C.T.C. Wall in 1965.
For $n \ge 1$, a group $\Gamma$ is of \textbf{type $F_n$} 
\index{Type! $F_n$ (for groups)}
if there exists a connected CW-complex 
with finite $n$-skeleton, 
fundamental group $\Gamma$,
and trivial higher homotopy groups.
A group is of type $F_1$ if and only if it is finitely generated,
and of type $F_2$ if and only if it is finitely presented.
For each $n \ge 1$, a group of type $F_n$ need not be of type $F_{n+1}$,
as shown by examples of Stallings for $n = 2$ and Bieri for $n \ge 3$ (among others).
\par
For LC-groups, Abels and Tiemeyer \cite{AbTi--97}
have defined a sequence of types $(C_n)_{n \ge 1}$, with the following properties:
\index{Type! $C_n$ (for LC-groups)}
\begin{enumerate}[noitemsep]
\item[(1)]
an LC-group is of type $C_1$ if and only if it is compactly generated,
\item[(2)]
an LC-group is of type $C_2$ if and only if it is compactly presented
(Chapter \ref{chap_cpgroups}), 
\item[(dis)]
for all $n \ge 1$,
a discrete group is of type $C_n$ if and only if it is of type $F_n$.
\end{enumerate}
Moreover, these properties are invariant ``up to compactness'':
more precisely, consider an LC-group $G$,
a cocompact closed subgroup $H$, a compact normal subgroup $N$, 
and an integer $n \ge 1$;
then $G$ is of type $C_n$ if and only if $H$ is of type $C_n$,
if and only if $G/N$ is of type $C_n$.
\par
In \cite{Tiem--97}, it is shown that reductive groups
over local fields of characteristic $0$ are of type $C_n$, for all $n \ge 1$.
This is used to give a new proof of a result of Borel and Serre \cite{BoSe--76}:
$S$-arithmetic subgroups of reductive groups are of type $F_n$ for all $n \ge 1$.
\index{Reductive group}
\end{rem}

\section{Miscellanea}
\label{2misc}

In this section, we collect a few very standard facts, for convenient reference.

\vskip.2cm

Let $X$ be a topological \emph{space}. 
If $X$ is Hausdorff, then every point of $X$ is closed.
It is well-known that the converse does not hold; see e.g.\
Chap.\ I, $\S$~8, Exercise 7 in \cite[Page I.101]{BTG1-4}.
However, for a possibly non-Hausdorff topological \emph{group}, we have:

\begin{prop}
\label{GHausdorff}
Let $G$ be a (possibly non-Hausdorff) topological group. 
Then $G$ is Hausdorff if and only $\{1\}$ is closed in $G$.
\end{prop}

\begin{proof}
Recall that a topological space $X$ is Hausdorff if and only if 
 \index{Hausdorff! topological space, group}
the diagonal $\{ (x,y) \in X \times X \mid x=y \}$
is closed in the product space $X \times X$.
\par
To show the non-trivial implication,
suppose $\{1\}$ is closed in $G$.
The diagonal is closed in $G \times G$, because it is the inverse image of $\{1\}$
by the continuous map $(g,h) \longmapsto gh^{-1}$ from $G \times G$ to $G$.
Hence $G$ is Hausdorff.
\end{proof}

\begin{prop}
\label{GtoHcovering}
Let $G, H$ be topological groups and $p : G \twoheadrightarrow H$ a continuous surjective homomorphism.
Then $p$ is a covering if and only if $p$ is a local homeomorphism at $1$.
\end{prop}

\begin{proof}
Suppose that $p$ is a local homeomorphism at $1$: there exist
a neighbourhood $U$ of $1$ in $G$ and a neighbourhood $V$ of $1$ in $H$
so that the restriction $p_U$ of $p$ to $U$ is a homeomorphism of $U$ onto $V$.
Then $p^{-1}(V) = \bigcup_{x \in \ker (p)} xU \subset G$.
\par
For $x,y \in \ker (p)$ with $x \ne y$, we have $xU \cap yU = \emptyset$.
Indeed, otherwise, there would exist $u_1, u_2 \in U$ with $u_1 = x^{-1}yu_2$,
hence with $p_U(u_1) = p_U(u_2)$, and this is impossible since $p_U$ is injective.
Therefore, we have a disjoint union
$p^{-1}(V) = \bigsqcup_{x \in \ker (p)} xU$,
so that $p^{-1}(V)$ is naturally homeomorphic to $\ker (p) \times V$,
that is to the product of the discrete space $\ker (p)$ with $V$.
\par
For every $h \in H$, we have by left translation a neighbourhood $hV$ of $h$ in $H$
such that $p^{-1}(hV)$ is homeomorphic to $\ker (p) \times hV$.
This means that $p$ is a covering.
\par
The converse implication is trivial.
\end{proof}

\begin{prop}
\label{discretenormalcentral}
(1)
In a connected topological group, every discrete normal subgroup is central.
\par
(2)
The fundamental group of a connected Lie group is abelian.
\index{Subgroup! discrete normal}
\end{prop}

\emph{Note:} the fundamental group of a connected Lie group
is moreover finitely generated (Corollary \ref{fundgpLietypefini}).

\begin{proof}
(1)
Let $G$ be a connected topological group and $N$ a discrete normal subgroup.
For $n \in N$, the image of the continuous map
$G \longrightarrow G, \hskip.1cm g \longmapsto gng^{-1}n^{-1}$
is both connected and contained in $N$.
Hence this image is $\{1\}$, and $n$ is central in $G$.
\par
(2)
For a connected Lie group $L$ with universal covering group $\widetilde L$, 
the fundamental group $\pi_1(L)$ can be identified with the kernel
of the covering projection $\widetilde L \twoheadrightarrow L$.
\index{Universal cover}
Hence $\pi_1(L)$ is abelian, by (1) applied to the connected group $\widetilde L$.
\end{proof}

\begin{rem}[some properties of LC-groups]
\label{remLCGetc}
The following properties of LC-groups and their homomorphisms are of frequent use.

\vskip.2cm

(1) Locally compact groups are paracompact.
Indeed, let $G$ be an LC-group,
$V$ a compact symmetric neighbourhood of $1$,
and $U = \bigcup_{n \ge 0} V^n$ as in the proof of Proposition \ref{powersSincpgroup}.
Then $U$ is a subgroup of $G$ that is $\sigma$-compact, and therefore paracompact.
Since $U$ is an open subgroup, $G$ is homeomorphic to $U \times G/U$,
where $G/U$ is discrete.
It follows that $G$ is paracompact 
\cite[Theorem 5, Page I.70]{BTG1-4}.
\par

More generally, if $G$ is an LC-group and $H$ a closed subgroup,
the homogeneous space $G/H$ is paracompact \cite[Page III.35]{BTG1-4}.
\par

Recall that a Hausdorff topological space $X$ is \textbf{paracompact} 
\index{Paracompact space|textbf}
if every open covering of $X$ has a locally finite refinement.
For such a space $X$ and any open covering $\mathcal U$ of $X$,
it is a basic fact that there exists a continuous partition of unity,
subordinate to $\mathcal U$.

\vskip.2cm

(2) Since paracompact spaces are normal \cite[Section 8.2]{Dugu--66},
it follows from~(1) that a locally compact \emph{group} is normal.
Note that a locally compact \emph{space} is completely regular,
but need not be normal \cite[Section 11.6]{Dugu--66}.
\par

Recall that a Hausdorff topological space $X$ is
\textbf{normal} \index{Normal topological space|textbf}
(or T$_4$) 
\index{T$_x$ for a topological space! T$_4$|textbf}
if, for every pair $(Y,Z)$ of disjoint closed subsets of $X$,
there exists a continuous function $f : X \longrightarrow \mathopen[0,1\mathclose]$
such that $f(y) = 0$ for all $y \in Y$ and $f(z) = 1$ for all $z \in Z$.
The space $X$ is  \textbf{completely regular} 
\index{Completely regular topological space|textbf}
(or T$_{3\frac{1}{2}}$)  
\index{T$_x$ for a topological space! T$_{3\frac{1}{2}}$|textbf}
if the analogous condition holds with $Y$ a singleton subspace $\{y\}$.
The space $X$ is \textbf{regular} \index{Regular topological space|textbf}
(or T$_{3}$) 
\index{T$_x$ for a topological space! T$_{3}$|textbf}
if, given any point $y \in X$ and closed subset $Z \subset X$,
there exist neighbourhoods $U$ of $y$ and $V$ of $Z$ such that $U \cap V = \emptyset$.

\vskip.2cm

(3)
A continuous bijection between locally compact \emph{spaces} 
need not be a homeomorphism,
as shown by the bijection $x \longmapsto e^{2i\pi x}$ 
from the half-open half-closed interval $\mathopen[0,1\mathclose[$ 
to the unit circle in $\C$
(compare with Remark \ref{remonmetrics}(2)).
However, a continuous bijective homomorphism between locally compact \emph{groups}
is an isomorphism of topological groups
(Corollary \ref{Freudenthalcor} below).
\par
More generally, we have the following proposition,
which can be viewed as an open mapping theorem:
\index{Open mapping theorem}
\index{Theorem! open mapping theorem}
\end{rem}

\begin{prop}[Freudenthal]
\label{Freudenthal}
Let $G$ be an LC-group, $X$ a locally compact space, 
$G \times X \longrightarrow X$ a transitive continuous action,
$a \in X$ a base point, and $H = \{g \in G \mid g(a) = a \}$
the corresponding isotropy subgroup.
\par
If $G$ is $\sigma$-compact, the orbit map $i_a : G \longrightarrow X$, $g \longrightarrow g(a)$,
is open, and the quotient map $G/H \twoheadrightarrow X$ is therefore a homeomorphism.
\end{prop}

\noindent \emph{Note.} 
Our attribution to Freudenthal rests on \cite[Item 26]{Freu--36}.
Hewitt and Ross refer to a article by A.A.\ Markov, from 1935
\cite[Pages 37 and 51]{HeRo--63}.
The proposition can also be found in \cite[Theorem 2.5 of Chapter I]{Hoch--65}.

\begin{proof}
Let $g \in G$ and $V$ an open neighbourhood of $g$.
We have to show the existence of 
an open neighbourhood of $i_a(g)$ in $X$ contained in $i_a(V)$.
\par
Let $U$ be a symmetric compact neighbourhood of $1$ in $G$ 
such that $gU^2 \subset V$.
Since $G$ is $\sigma$-compact, there exists a sequence $(g_n)_{n \ge 0}$ in $G$
such that $G = \bigcup_{n \ge 0} g_nU$;
thus $X = \bigcup_{n \ge 0} i_a(g_nU)$.
Since $X$ is a Baire space,
there exists $n \ge 0$ such that the interior $\operatorname{int}(i_a(g_nU))$
is non-empty; 
as $x \longmapsto g_nx$ is a homeomorphism of $X$,
we have $\operatorname{int}(i_a(U)) \ne \emptyset$.
Hence there exists $g' \in U$ such that $i_a(g') \in \operatorname{int}(i_a(U))$,
and therefore $a \in \operatorname{int}(i_a(U^2))$.
It follows that
 $i_a(g) \in \operatorname{int}(i_a(gU^2)) \subset i_a(V)$,
as was to be shown.
\end{proof}

\begin{cor}
\label{Freudenthalcor}
Let $G_1, G_2$ be two LC-groups and 
$\varphi : G_1 \longrightarrow G_2$ a continuous homomorphism.
Assume that $G_1$ is $\sigma$-compact.
\par

(1) If $\varphi$ is surjective, then $\varphi$ is open.
In particular, if $\varphi$ is a bijective continuous homomorphism,
then $\varphi$ is an isomorphism of topological groups.
\par

(2) The homomorphism $\varphi$ is proper if and only if
$\ker(\varphi)$ is compact in $G_1$ and $f(G_1)$ closed in $G_2$.
\end{cor}

\begin{proof} 
(1)
This follows from Proposition \ref{Freudenthal}, applied to the action
$G_1 \times G_2 \longrightarrow G_2$ of $G_1$ defined by
$(g_1, g_2) \longmapsto \varphi(g_1)g_2$.
\par

(2) If $\varphi$ is proper, then $\ker(\varphi)$ is compact and $\varphi(G_1)$ is closed,
as are all fibers and all images of all proper continuous maps from one
topological space to another. 
\par

For the non-trivial implication,
assume that $\ker(\varphi)$ is compact,
so that the canonical projection 
$\pi : G_1 \longrightarrow G_1 / \ker(\varphi)$ is proper,
and that $\varphi(G_1)$ is closed, 
so that the inclusion 
$j : \varphi(G_1) \longrightarrow G_2$ is proper.
By (1), 
the natural map $\psi : G_1 / \ker(\varphi) \longrightarrow \varphi(G_1)$
is an isomorphism, in particular is proper.
It follows that the composition $\varphi = j \circ \psi \circ \pi$ is proper.
\end{proof}

In Corollary \ref{Freudenthalcor}(1),
the hypothesis ``$G_1$ is $\sigma$-compact'' cannot be deleted,
as the example 
$\R_{\text{dis}} \overset{\operatorname{id}}{\longrightarrow} \R$ shows,
where $\R_{\text{dis}}$ stands for the group $\R$ with the discrete topology.

\vskip.2cm

As a particular case of Corollary \ref{Freudenthalcor}, 
with $G_1 = G_2$ as abstract groups:

\begin{cor}
\label{Freudenthalcorbis}
On a $\sigma$-compact LC-group, a continuous left-invariant metric is compatible.
\end{cor}

\section[Structure of LC-groups]
{LC-groups: structure and approximation by Lie groups}
\label{structure}

\begin{defn}
\label{deftotdiscetc}
Let $X$ be a topological space.
The \textbf{connected component of a point} 
\index{Connected component|textbf}
$x \in X$ is the union of the connected subspaces of $X$ containing $x$;
it is a closed subspace of $X$
\cite[Page I.84]{BTG1-4}.
A topological space is \textbf{totally disconnected}
\index{Totally disconnected! topological space|textbf}
if every point in it is its own connected component.
A subspace of a topological space is \textbf{clopen}
\index{Clopen subspace|textbf}
if it is both closed and open.
A topological space is \textbf{zero-dimensional} 
\index{Zero-dimensional topological space|textbf}
if its clopen subspaces constitute a basis of its topology.
\end{defn}

Any zero-dimensional space is clearly totally disconnected.
There is a partial converse. In a totally disconnected LC-space,
every neighbourhood of a point contains
a clopen neighbourhood of this point  \cite[Theorem 3.5]{HeRo--63}.
Thus:

\begin{prop}
\label{0dimtd}
A locally compact space is totally disconnected 
if and only if it is zero-dimensional.
\end{prop}

Let $X$ be a topological space.
Let $\mathcal R$ be the relation defined on $X$ by
$x \mathcal R y$ if the connected components of $x$ and $y$ coincide;
then $\mathcal R$ is an equivalence relation on $X$. 
The quotient space
\footnote{As 
agreed in (A1), Page \pageref{A1}, 
we note here that
$H_0(X)$ need not be Hausdorff.
An example is provided by the subspace
\begin{equation*}
X \, = \,  \{(0,0)\} \cup \{(0,1)\} \cup \Big(\bigcup_{n \ge 1} (\{1/n\} \times \mathopen[0,1\mathclose]) \Big)
\end{equation*}
of the Euclidean plane.
}
$H_0(X) := X / \mathcal R$ is totally disconnected,
see again \cite[Page I.84]{BTG1-4};
note that $H_0(X)$ is not necessarily Hausdorff.
If $X$ is a non-empty topological space, 
$X$ is connected if and only if
$H_0(X)$ is the one-point space,
and $X$ is totally disconnected if and only if
the canonical projection $X \longrightarrow H_0(X)$ is a homeomorphism.

\vskip.2cm

Let $G$ be a topological group.
The \textbf{identity component} 
\index{Identity component of a topological group|textbf}
of $G$, denoted by $G_0$,  
\index{$g$@$G_0$, identity component of $G$}
is the connected component of the identity in $G$;
it is a closed subgroup of $G$ that is normal, indeed characteristic.          
The quotient group
$H_0(G) = G/G_0$ is totally disconnected
(it is a particular case of the result on $X/\mathcal R$ quoted above).
In particular, $G$ is totally disconnected 
if and only if $G_0 = \{1\}$.
\index{Totally disconnected! topological group|textbf}
A topological group $G$ is \textbf{connected-by-compact} 
\index{Connected-by-compact topological group|textbf}
if $G/G_0$ is compact.
This terminology fits with our general terminology, 
see (A4) in Chapter \ref{chap_topaspects};
many authors use ``almost connected'' for ``connected-by-compact''.
\par

If $G$ is an LC-group, its identity component $G_0$
is also the intersection of the open subgroups of $G$
\cite[Page III.35]{BTG1-4}.
\index{Open subgroup}
This does \emph{not} hold in general;
a trivial example is provided by the group $\Q$,
with the topology inherited from the usual topology on $\R$;
moreover, there are known examples of totally disconnected \emph{Polish} groups
in which \emph{any} neighbourhood of $1$ generates the whole group;
see the discussion in \cite{Sole--05}. 
\index{Polish! group}
\par

In an LC-group $G$, the identity component need not be open;
for example, when $G$ is totally disconnected, $G_0 = \{1\}$ is open
if and only if $G$ is discrete.
However, in a locally path-connected LC-group,
for example in a Lie group, connected components are open.
\par

One of the main open problems about totally disconnected LC-groups
is the \emph{Hilbert-Smith conjecture}, 
\index{Hilbert-Smith conjecture}
which can be formulated as a question:
Does there exist an LC-group, which is not a Lie group,
that has a faithful continuous action on a connected topological manifold?
Equivalently: does there exist for some prime $p$ a faithful continuous action of $\Z_p$ 
on a connected topological manifold?
For the equivalence, see \cite{Lee--97}.
The answer is known to be negative for manifolds of
dimension at most three \cite{Pard--13}.

\begin{exe}[non-discrete locally compact fields]
\label{panoramalocalfield}
Let $\K$ be non-discrete locally compact field.
Then $\K$ is either Archimedean, and then isomorphic to one of $\R$ or $\C$,
or non-Archimedean, in which case it is defined to be a \textbf{local field}. 
\index{Local field|textbf}
\index{$k$@$\K$, non-discrete locally compact field, often local field}
\par

[Some authors, including \cite[Page 20]{Weil--67} and \cite{Marg--91}, 
define a \emph{local field}
to be any commutative non-discrete locally compact field.
Others are keen on not formulating a precise definition, 
like Serre in his ``Corps locaux'' (outside the introduction).]
\par

A non-discrete locally compact field is connected if and only if it is Archimedean,
i.e., if it is isomorphic to  $\R$ or $\C$.
The multiplicative group $\C^\times$ is connected, 
and $\R^\times$ has two connected components.
\par

Here are some basic facts on local fields and their classification.
More information on the subject can be found, among many other places, 
in \cite[Chapter II]{Neuk--99} or/and \cite{Cass--86}.
Absolute values on local fields are briefly discussed in 
Remark \ref{onabsolutevalues}

\vskip.2cm

From now on, let $\K$ be a local field.
Then $\K$ has a unique maximal compact subring
\begin{equation*}
\mathfrak o_{\K} \, \, = \, 
\big\{ x \in \K    \hskip.1cm \big\vert \hskip.1cm     
\{x^n \mid n \ge 1 \}  \hskip.2cm \text{is relatively compact}  \big\} 
\end{equation*}
and $\mathfrak o_{\K}$ has a unique maximal ideal
\begin{equation*}
\mathfrak p_{\K} \, = \, \{ x \in \K \mid \lim_{n \to \infty} x^n = 0 \} .
\end{equation*}
Both $\mathfrak o_{\K}$ and $\mathfrak p_{\K}$
are compact and open in $\K$.
Moreover, the ideal $\mathfrak p_{\K}$ is principal in $\mathfrak o_{\K}$:
there exists a ``prime element'' $\pi$
such that $\mathfrak p_{\K}$ is the ideal $(\pi)$ generated by $\pi$.
The nested sequence of ideals 
$(\pi) \supset \cdots \supset (\pi^n) \supset (\pi^{n+1}) \supset \cdots$
constitute a basis of clopen neighbourhoods of $0$ in $\K$,
so that $\K$ is totally disconnected.
The multiplicative group $\K^\times$ 
is also totally disconnected.
\par

If $\K$ is of characteristic zero,
then $\K$ is a finite extension of a field $\Q_p$;
for each prime $p$ and each degree $n$, 
the number $N_{p,n}$ of extensions of $\Q_p$ of degree $n$
satisfies $1 \le N_{p,n} < \infty$
\cite[Section 3.1.3 and Theorem 3.1.6]{Robe--00}.
\par

In the particular case $\K = \Q_p$, 
we have $\mathfrak o_{\K} = \Z_p$ and a choice for $\pi$ is $p$.
Every $x \in \Q_p$ can be written uniquely as $[x]+\sum_{i=1}^nx_ip^{-i}$,
with $[x] \in \Z_p$, $n \ge 0$, and $x_i \in \{0, 1, \hdots, p-1\}$ for $i=1, \hdots, n$,
with $x_n \ne 0$ when $n \ge 1$.
Every $x \in \Q_p^\times$ can be written uniquely as $p^nu$,
with $n \in \Z$ and $u \in \Z_p^\times$;
it follows that the multiplicative group $\Q_p^\times$ is the disjoint union
$\bigsqcup_{n \in \Z} p^n \Z_p^\times$;
hence $\Q_p^\times$  is homeomorphic to the direct product 
$p^\Z  \times  \Z_p^\times$.
The multiplicative group
$\Z_p^\times$ is homeomorphic to the direct product of 
the finite group $\{ x \in \Z_p \mid x^{p-1} = 1\}$
of $(p-1)$th roots of $1$, which is cyclic of order $p-1$,
and of the compact group $1 + p\Z_p$,
itself homeomorphic to $\Z_p$ when $p$ is odd
and to $\Z_2 \times (\Z / 2\Z)$ when $p=2$
(see for example \cite[chap. 2]{Serr--70}).
\index{$q$@$\Q_p$, field of $p$-adic numbers|textbf}
\index{$z$@$\Z_p$, ring of $p$-adic integers|textbf}
\par

If $\K$ is of finite characteristic $p$, then
$\K$ is a field of formal Laurent series $\mathbf F_q \lp t \rp$ over a finite field 
\index{$fp$@$\mathbf F_p \lp t \rp$, $\mathbf F_q \lp t \rp$|textbf}
$\mathbf F_q$, where $q = p^d$ for some $d \ge 1$.
Then $\mathfrak o_{\K}$ is the ring of formal power series $\mathbf F_q [[ t]]$,
\index{$fq$@$\mathbf F_q [t]$, $\mathbf F_q [[t]]$}
and a choice for $\pi$ is $t$.
The multiplicative group $\mathbf F_q \lp t \rp^\times$
is isomorphic to a product of three groups:
the infinite cyclic group generated by $t$,
the finite group $\{ x \in \mathfrak o_{\K} \mid x^{q-1} = 1 \} \simeq \Z / (q-1)\Z$
of $(q-1)^{\text{th}}$ roots of unity,
and the multiplicative group $1 + \mathfrak p_{\K}$, which is compact.

\vskip.2cm

Similarly, let $R$ be a finite commutative ring.
Then $R  \lp t \rp$ is a totally disconnected LC-ring
in which $R [[t]]$ is an open subring.
\end{exe}

\begin{exe}[totally disconnected LC-groups]
\label{extdLCgps}
Discrete groups are clearly totally disconnected. Here are some other examples.
\index{Totally disconnected! topological group}

\vskip.2cm

(1)
Let $(F_\alpha)_{\alpha \in A}$ be an infinite family
of finite groups, viewed as compact groups.
The product $\prod_{\alpha \in A} F_\alpha$
is a totally disconnected compact group, i.e.,
a \textbf{profinite group}
(as defined in Example \ref{profinitehausdim}).
\index{Profinite group}  \index{Compact group! profinite group}

\vskip.2cm

(2)
Let $I$ be a directed set, with its order relation denoted by $\alpha \le \beta$.
Let $(F_\alpha, \pi_{\alpha, \beta})$ be an inverse system of
finite groups (viewed as compact groups) and surjective homomorphisms.
The inverse limit $\varprojlim F_\alpha$ is a profinite group,
as a closed subgroup of $\prod_{\alpha \in A} F_\alpha$.
\par
Conversely, let $K$ be a profinite group.
Every neighbourhood $U$ of $1$ in $K$
contains a normal compact open subgroup $N$ of $K$ 
(see \cite[Section 2.5]{MoZi--55}, and compare with Remark \ref{vanDantzigRem}).
The finite groups $K/N$ and the natural surjective homomorphisms constitue an inverse system,
and the group $K$ is the inverse limit of this inverse system.
\par
For example, the groups $\Z_p = \varprojlim_n \Z / p^n\Z$
and $\GL(d, \Z_p) = \varprojlim_n \GL(d, \Z / p^n \Z)$
are profinite groups.

\vskip.2cm

(3)
Let $\K$ be a local field,
$\mathbf G$ an algebraic group defined over $\K$,
and $G$ the group of its $\K$-points.
The group $G$ can be seen as a subgroup of $\GL_n(\K)$, 
and the latter is an open subspace of the algebra $\M_n(\K)$
of $n$-by-$n$ matrices over $\K$; it follows that $G$
has a natural topology $\mathcal T$ inherited from the topology of $\K$,
and $G$ with $\mathcal T$ is a totally disconnected LC-group.
When $\K$ is $\Q_p$ or a finite extension of $\Q_p$,
this topology is sometimes called the \textbf{$p$-adic topology} on $G$.
\index{p-adic topology on the group of $\Q_p$-points 
of an algebraic group}

\vskip.2cm

(4)
For any non-discrete LC-field $\K$, 
there is a notion of \emph{Lie group over $\K$.}
These groups are totally disconnected when $\K$ is so.
If $\K = \Q_p$, they include all closed subgroups of $\GL_n(\Q_p)$;
see \cite[Chap.~3, $\S$~8, n$^o$~2]{BGLA2-3}.
\index{Lie group! over an LC-field $\K$}

\vskip.2cm  

(5)
Let $X$ be a connected locally finite graph.
\index{Graph}
\index{Locally finite! graph}
For the topology of pointwise convergence,
the automorphism group $\operatorname{Aut}(X)$ 
is a totally disconnected LC-group (Example \ref{metricsonAut(X)}).
\index{Automorphism group! of a graph}
It is a discrete group if and only if 
there is a finite subset $V$ of the vertex set of $X$
such that $\{g \in \operatorname{Aut}(X) \mid gv=v
\hskip.2cm \text{for all} \hskip.2cm v \in V \}$ is trivial.
The following examples have non-discrete automorphism groups.
\par

Let $T_k$ be a regular tree of some valency $k \ge 3$.
\index{Tree}
Let $\operatorname{Aut}^+(T_k)$ be the index two subgroup of $\operatorname{Aut}(T_k)$
generated by the vertex stabilizers.
Then $\operatorname{Aut}^+(T_k)$ is a non-discrete LC-group;
as an abstract group, it is simple \cite{Tits--70}.
\index{Simple group}
When $k = p+1$ for some prime $p$, the group $\operatorname{Aut}(T_k)$
has subgroups isomorphic to the projective linear groups
$\PGL_2(\Q_p)$ and $\PGL_2( \mathbf F_p \lp t \rp )$,
see \cite{Serr--77}; these are closed subgroups, by Proposition \ref{GinIsom(X)}(2).
\index{Projective linear group PGL, and PSL! $\PGL_2(\Q_p)$, $\PGL_2(\mathbf F_p \lp t \rp )$}
\par

Consider two integers $m,n \in \Z \smallsetminus \{0\}$, 
the Baumslag-Solitar group $\operatorname{BS}(m,n)
= \langle s,t  \mid t^{-1}s^mt = s^n \rangle$,
\index{Baumslag-Solitar group|textbf}
its Bass-Serre tree $T_{m+n}$ (a regular tree of valency $m+n$),
and the closure $G_{m,n}$ of the natural image of $\operatorname{BS}(m,n)$
in the automorphism group of $T_{m+n}$.
The groups $G_{m,n}$ are non-discrete totally disconnected LC-group;
moreover, for $m,n,m',n' \in \Z \smallsetminus \{0\}$, the following conditions are equivalent:
\par
(i) $\operatorname{BS}(m',n')$ is isomorphic to $\operatorname{BS}(m,n)$,
\par
(ii) $G_{m',n'}$ is isomorphic to $G_{m,n}$,
\par
(iii) $(m',n')$ is one of $(m,n)$, $(n,m)$, $(-m,-n)$, $(-n,-m)$.
\par\noindent
See \cite{ElWi}.
\par

Cayley graphs of appropriate Coxeter groups
\index{Coxeter group}
provide other examples of graphs $X$ with non-discrete
$\operatorname{Aut}(X)$.
See  \cite[Theorem 1.3]{HaPa--98}.

\vskip.2cm

(6)
Let $K$ be a compact group.
With the natural topology (see Example \ref{exampleAut(G)}), 
the outer automorphism group
$\operatorname{Aut}(K) / \operatorname{Int} (K)$
\index{Automorphism group! of a group}
is totally disconnected
\cite[Theorem 1]{Iwas--49}; see also \cite[th\'eor\`eme~6]{Serr--50}.
\par

Let $G$ be a totally disconnected LC-group. 
The automorphism group $\operatorname{Aut}(G)$
is totally disconnected \cite{Brac--48}.
\par

[Recall that  $\operatorname{Aut}(K)$ and $\operatorname{Aut}(G)$
need not be locally compact,   
as already noted in Example \ref{exampleAut(G)}.]

\vskip.2cm

(7)
With the Krull topology, the Galois group $\operatorname{Gal}(L/K)$
of a Galois extension $L$ of a field $K$ is a totally disconnected compact group.
\end{exe}

Total disconnectedness is inherited by closed subgroups,
and more generally by subspaces of topological spaces.
Claim (1) of the next proposition
shows that it is also inherited by quotient spaces of LC-groups;
if it far from being so for quotient spaces of topological spaces,
since every metrizable compact space is a continuous image of the Cantor space.

\begin{prop}
\label{pi_0finite}
Let $G$ be an LC-group, $H$ a closed subgroup, 
$G/H$ the quotient space, and $\pi : G \longrightarrow G/H$
the canonical projection.
\begin{enumerate}[noitemsep,label=(\arabic*)]
\item\label{1DEpi_0finite}
The connected components of $G/H$ are the closures of the images by $\pi$
of the connected components of $G$.
In particular: if $G$ is totally disconnected, so is $G/H$.
\end{enumerate}
Assume moreover that $H$ is a normal subgroup,
so that $G/H$ is a topological group.
\begin{enumerate}[noitemsep,label=(\arabic*)]
\addtocounter{enumi}{1}
\item\label{2DEpi_0finite}
If $G$ is locally compact and $L = G/H$ a Lie group, then $\pi(G_0) = L_0$.
\end{enumerate}
\end{prop}

\begin{proof}
See \cite[Page III.36]{BTG1-4} for (1)
and  \cite[Lemma 2.4]{CCMT--15} for (2).
\end{proof}

Claim \ref{1DEpi_0finite} does not hold without the hypothesis of local compactness.
Indeed, let $G$ be the group of sequences of rational numbers
having a limit in $\R$, with the $\ell^{\infty}$-topology,
and let $H$ be the subgroup of sequences converging to $0$.
Then $G$ is totally disconnected, 
$H$ is closed in $G$, and $G/H$ can be identified to the
\emph{connected} group $\R$, with its standard topology
\cite[Page III.70 (exercice 17)]{BTG1-4}.

\begin{thm}[van Dantzig] \index{van Dantzig theorem}
\index{Theorem! van Dantzig}
\label{vanDantzig}
Let $G$ be a totally disconnected LC-group and $U$ a neighbourhood of $1$ in $G$.
Then $U$ contains a compact open subgroup of $G$.
\index{Open subgroup! compact}
\par                               
In other words:
compact open subgroups constitute a fundamental system
of neighbourhoods of $1$ in $G$.
\end{thm}

\noindent \emph{Note.}
This theorem is part of \cite[Theorem 7.7]{HeRo--63},
where it is quoted from the thesis of van Dantzig (1931); see also \cite{Dant--36}.
Related references include
\cite[Section 2.3]{MoZi--55}, \cite[Page III.36]{BTG1-4}, 
Willis \cite{Will--94} and Tao \cite{Tao--14}.

\begin{proof}
By Proposition \ref{0dimtd}, there exists 
a compact open neighbourhood $V$ of $1$ contained in $U$.
We define below a compact neighbourhood $S$ of $1$ in $G$
generating an open subgroup $K$ inside $V$.
This will complete the proof, because an open subgroup is also closed,
and a closed subgroup of the compact set $V$ is compact.
\par
For each $x \in V$, there exist
a neighbourhood $S_x$ of $1$ in $G$ and a neighbourhood $V_x$ of $x$ in $V$
such that $S_x V_x \subset V$;
we assume furthermore that $S_x$ is symmetric.
Since $V$ is compact, there exists
a finite subset $F \subset V$ such that $V = \bigcup_{f \in F} V_f$.
Set $S = \bigcap_{f \in F} S_f$ ; it is a symmetric neighbourhood of $1$.
We have $S V_f \subset V$ for all $f \in F$, and therefore $SV \subset V$.
\par
Let $K = \bigcup_{n \ge 0} S^n$ be the subgroup of $G$ generated by $S$;
since $S$ is a neighbourhood of $1$, 
the subgroup $K$ is open.
We have $S^n V \subset V$ for all $n \ge 0$, by induction on $n$,
and therefore $K \subset V$. 
\end{proof}

Claim (1) of the next corollary appears as Lemma 1.4 in \cite{Glea--51}.

\begin{cor}
\label{vanDantzigCor}
Let $G$ be an LC-group.
\par
(1)
There exists an open subgroup $H$ of $G$, containing $G_0$,
such that the quotient group $H/G_0$ is compact.
\par
In particular, every LC-group contains
an open subgroup that is connected-by-compact,
and therefore compactly generated, a fortiori $\sigma$-compact.
\index{Sigma-compact! LC-group}
\index{Connected-by-compact topological group}
\par

(2)
Assume moreover that $G_0$ is compact 
(e.g., that $G$ is totally disconnected).
Any compact subgroup of $G$ is contained in a compact open subgroup.
\end{cor}

\begin{proof}
(1)
Let $K$ be a compact open subgroup 
of the totally disconnected group $G/G_0$.
The inverse image $H$ of $K$ in $G$ is an open subgroup of $G$.
Observe that $G_0$, which is normal in $G$, is a posteriori normal in $H$.
\par

The quotient  $H /G_0$ is compact, isomorphic to $K$.
Since $G_0 \subset H_0$,
the quotient group $H/H_0$ is compact.
Since $H$ is connected-by-compact, $H$ is compactly generated
(see Proposition \ref{almostconnectedgroups}),
and in particular $\sigma$-compact.
\par

(2) Let now $K$ be a compact open subgroup of $G$.
Let $L$ be an arbitrary compact subgroup of $G$.
There exist $\ell_1, \hdots, \ell_n \in L$ such that $L \subset \bigcup_{i=1}^n \ell_i K$.
Set $N = \bigcap_{\ell \in L} \ell K \ell^{-1}$.
Observe that $N = \bigcap_{1 \le i \le n} \ell_i K \ell_i^{-1}$, 
so that $N$ is a compact open subgroup of $G$.
Since $\ell N \ell^{-1} = N$ for all $\ell \in L$,
the product $NL$ is a subgroup of $G$,
indeed a compact open subgroup containing $L$.
\end{proof}

\begin{rem}
\label{vanDantzigRem}
A totally disconnected LC-group need not contain
a \emph{normal} compact open subgroup,
i.e.,  \emph{need not be}  ``compact-by-discrete'',
as the two following examples show.
\index{Subgroup! compact normal} 
\index{Normal compact open subgroup}
\par

For every $d \ge 2$ and local field $\K$, 
the totally disconnected LC-group $\PSL_d(\K)$ is simple.
It has compact open subgroups
$\PSL_d(\pi^n \mathfrak o_\K)$ for every $n \ge 0$
(with the notation of Example \ref{panoramalocalfield}).
\par

Consider 
a prime $p$, 
the discrete group $\Z$,
the totally disconnected group $\Q_p$,
and the action of $\Z$ on $\Q_p$ given by
$(n,x) \longmapsto p^nx$.
The corresponding semidirect product $G =  \Q_p \rtimes_p \Z$,
\index{Semidirect product! $\Q_p \rtimes_p \Z$}
with the direct product topology, is a totally disconnected group.
Observe that all orbits of $\Z$ in $\Q_p \smallsetminus \{0\}$ are unbounded.
The group $G$ has compact open subgroups, for example $\Z_p$,
but it is easy to check that $G$ does not have any non-trivial compact normal subgroup.
(Up to details, this is the example of \cite[Section 2.5]{MoZi--55}.)

\par

In contrast, it is well-known that a compactly generated totally disconnected nilpotent LC-group
has a basis of neighbourhoods of $1$ consisting of compact open normal subgroups;
see \cite{Will--97}.
\end{rem}

\begin{prop}
\label{propCayleyAbels}
For an LC-group, the following two properties are equivalent:
\begin{enumerate}[noitemsep,label=(\roman*)]
\item\label{iDEpropCayleyAbels}
$G$ is compactly generated and has a compact open subgroup;
\item\label{iiDEpropCayleyAbels}
there exist a connected graph $X$ of bounded valency
and a continuous action of $G$ on $X$ which is proper and vertex-transitive.
\end{enumerate}
\end{prop}

\noindent \emph{Note.} 
A compactly generated totally disconnected LC-group
has Property \ref{iDEpropCayleyAbels}, 
by the van Dantzig theorem.

\begin{defn}
\label{defCayleyAbels}
A $G$-graph $X$ with 
Property \ref{iiDEpropCayleyAbels} of Proposition \ref{propCayleyAbels} 
is a \textbf{Cayley-Abels graph} for $G$.
\index{Cayley-Abels graph|textbf}
\end{defn}

\begin{proof}[Proof of Proposition \ref{propCayleyAbels}]
To show that \ref{iDEpropCayleyAbels} implies \ref{iiDEpropCayleyAbels}, 
let $S_0$ be a compact generating set of $G$
and $K$ a compact open subgroup of $G$.
Set $S = K(S_0 \cup S_0^{-1})K \smallsetminus K$; 
it is symmetric, $K$-biinvariant, open (because $K$ is open),
and $S \cup K$ generates $G$.
\par

Define a relation $\sim$ on $G$ by $g \sim g'$ if $g^{-1}g' \in S$;
this relation is symmetric ($g \sim g'$ implies $g' \sim g$),
antireflexive ($g \nsim g$ for all $g \in G$),
and $K$-biinvariant ($g \sim g'$ implies $k_1gk_2 \sim k_1 g' k_2$ for all $k_1, k_2 \in K$).
This relation induces a left-$K$-invariant relation on $G/K$,
again denoted by $\sim$, again symmetric and antireflexive.
\par

Define a graph $X$ as follows: 
its vertex set is $G/K$,
and there is an edge between $gK$ and $g'K$ if $gK \sim g'K$.
Then $X$ is connected, because $S \cup K$ generates $G$.
The natural action of $G$ on $X$ preserves the graph structure.
\par

The set of neighbours of the vertex $1K$ can be identified
with the set $\{g \in G \mid gK \sim 1K \} / K$, which is $S/K$;
it is finite, because the right action of $K$ on $S$ has open orbits, and $S$ is compact.
By homogeneity, every vertex in $X$ has the same number $\vert S/K \vert$ of neighbours.
\par

To show that \ref{iiDEpropCayleyAbels} implies \ref{iDEpropCayleyAbels}, 
we refer to Corollary \ref{ahahah} below.
As $X$ can be viewed as a (coarsely) geodesic metric space,
and as the action of $G$ on $X$ is geometric,
the group $G$ is compactly generated.
Moreover, for every vertex $x_0 \in X$, the isotropy subgroup
$\{ g \in G \mid g(x_0) = x_0 \}$ is a compact open subgroup.
\end{proof}

\begin{rem}
\label{remCayleyAbelsg}
With the terminology of Definition \ref{actions_mp_lb_cob},
the natural actions of $G$ on its Cayley-Abels graphs are geometric.
In particular, any two Cayley-Abels graphs for a group $G$ as in Proposition \ref{propCayleyAbels} 
are quasi-isometric \cite[Theorem 2.7]{KrMo--08}.
\par
For finitely generated groups, Cayley graphs in the usual sense (see e.g.\ \cite{Cann--02})
are examples of Cayley-Abels graphs.
``Cayley graphs'' for infinite groups appeared first in articles by Dehn, 
see \cite{Dehn--10};
the terminology ``Dehn Gruppenbild'' is also used.
\index{Dehn Gruppenbild|see {Cayley graph}}

\par
Cayley-Abels graphs appear, sometimes with different names, 
in \cite[Beispiel 5.2]{Abel--74},
\cite[Corollary 1]{Moll--03},
\cite[Section 2]{KrMo--08},
and shortly after Proposition 2.1 in \cite{CCMT--15}.
\par

Compactly generated LC-groups containing compact open subgroups
appear again in Theorem \ref{cpcovers}.
\end{rem}

Before we state the next structure theorem, 
it is convenient to recall the standard terminology 
related to real and complex Lie groups; see e.g.\ \cite{Hoch--65}.

\begin{defn}
\label{defanalyticLie}
A real \textbf{analytic group} 
\index{Analytic group|textbf}
is a group, 
and a \emph{connected} real analytic manifold,  
such that the group operations are analytic.
A real \textbf{Lie group} \index{Lie group|textbf}
is a topological group $G$
of which the identity component is both open and a real analytic group.
There are similar definitions for complex analytic groups and complex Lie groups.
We often write ``Lie group'' for ``real Lie group''.
\end{defn}

\begin{rem}
\label{alggroupsalmostconnected}
There is no restriction on the number of connected components of a Lie group.
In particular, every discrete group is a Lie group.
\par

If $G$ is the group of real points of an algebraic group 
defined over the reals, then $G$ is connected-by-finite.
More generally, real algebraic varieties have finitely many connected components;
this is due to Whitney \cite{Whit--57}.
\end{rem}

\begin{thm}[Gleason-Yamabe] \index{Gleason-Yamabe theorem}
\index{Theorem! Gleason-Yamabe}
\label{Yamabe}
Let $G$ be a connected-by-compact LC-group
and let $U$ be a neighbourhood of $1$ in $G$.
\par
There exists a compact normal subgroup $N$ of $G$ contained in $U$
such that $L = G/N$ is a Lie group with finitely many connected components.
\index{Subgroup! compact normal} 
\par

In particular, a connected-by-compact LC-group is
compact-by-(analytic-by-finite).
\end{thm}

\begin{proof}[On the proof] 
This is \cite[Theorem 4.6]{MoZi--55}. 
Observe that the quotient $L/L_0$ is finite,
because it is discrete ($L$ is a Lie group)
and compact (by Proposition \ref{pi_0finite}).
\par

The original articles, \cite{Glea--51} and \cite{Yam--53b},
were building upon a large amount of deep work by various authors, 
including
Chevalley, Kuranishi, Iwasawa, Montgomery, and Zippin.
See also \cite{Bore--50}, \cite{Tao--14},  and \cite{DrGo--15}. 
\end{proof}

\begin{rem}
\label{RemOnYamabe}
(1)
There exists also a \emph{maximal} compact normal subgroup of $G$.
See Example \ref{exW(G)2}(1).

\vskip.2cm

(2)
One way to express the theorem is that
every connected-by-compact LC-group
can be \emph{approximated by Lie groups}, the $G/N$ 's. 
\par
Here is another way. A topological group $G$ is a \textbf{generalized Lie group}
\index{Generalized Lie group|textbf}
(in the sense of Gleason \cite{Glea--49, Glea--51})
if, for every neighbourhood $U$ of $1$ in $G$,
there exist an open subgroup $H$ of $G$
and a compact normal subgroup $N$ of $H$, with $N \subset U$, 
such that $H/N$ is a Lie group;
a generalized Lie group is locally compact.
Corollary \ref{vanDantzigCor} and Theorem \ref{Yamabe}
show that, conversely, every LC-group is a generalized Lie group.

\vskip.2cm

(3)
The Gleason-Yamabe Theorem on the structure of LC-groups
goes together with the final solution of the Hilbert Fifth Problem,
\index{Hilbert Fifth Problem}
published in 1952 \cite{Glea--52, MoZi--52}
(see Theorem \ref{MontgomeryZipin}).
There is a classical exposition in
\cite[in particular Sections 2.15 and 4.10]{MoZi--55};
see also \cite{Glus--57}, \cite{Tao--14}, and \cite{DrGo--15}.

\vskip.2cm

(4)
The hypothesis ``connected-by-compact'' cannot be omitted
in Theorem \ref{Yamabe},
as the following example shows.
\par
Let $G = \SO (3)^\Z \rtimes_{\operatorname{shift}} \Z$, 
\index{Orthogonal group! $\textnormal{O}(n)$, $\SO (n)$}
\index{Semidirect product! $\SO (3)^\Z \rtimes_{\operatorname{shift}} \Z$}
where the semidirect product refers to the action of $\Z$ by shifts;
the topology is that for which $\SO(3)^\Z$ has
the compact product topology, 
\index{Product of groups}
and is open in $G$.
Then $G_0 =  \SO(3)^\Z$, and 
$G$ has exactly two compact normal subgroups which are $G_0$ and $\{1\}$.
\par

In particular, 
for a neighbourhood $U$ of $1$ in $G$ such that $U \not\supset G_0$,
there cannot exist a compact normal subgroup $N$ of $G$ 
such that $N \subset U$ and $G/N$ is a Lie group.
\end{rem}

\begin{thm} \index{Cartan-Chevalley-Iwasawa-Malcev-Most\-ow theorem}
\index{Theorem! Cartan, Chevalley, Iwasawa, Malcev, Mostow}
\label{deHochschild+}
Let $G$ be a connected-by-compact LC-group.
\begin{enumerate}[noitemsep,label=(\arabic*)]
\item\label{1DEdeHochschild+}
The group $G$ contains maximal compact subgroups,
such that every compact subgroup of $G$ is contained in one of them.
\index{Maximal! compact subgroup}
\index{Subgroup! maximal compact}
\index{Compact group! maximal compact subgroup}
\item\label{2DEdeHochschild+}
Maximal compact subgroups of $G$ are conjugate with each other.
\item\label{3DEdeHochschild+}
If $K$ is a maximal compact subgroup of $G$,
the quotient space $G/K$ is homeomorphic to $\R^n$
for some $n \ge 0$;
moreover, $G_0K = G$, i.e., the natural action of $G_0$ on $G/K$ is transitive. 
\item\label{4DEdeHochschild+}
For $K$ and $n$ as in \ref{3DEdeHochschild+}, there exists on 
$G/K \overset{\operatorname{homeo}}{\approx} \R^n$
\index{$ad$@$\approx$ homeomorphism}
a $G$-invariant structure of analytic manifold.
\end{enumerate}
\end{thm}
\index{Theorem! maximal compact subgroups}

\begin{proof}[On the proof]
For the particular case of a connected LC-group, the theorem is essentially that of
Section 4.13 in \cite{MoZi--55}.
For the theorem as stated, a convenient strategy is to split 
the argument in two steps, as follows.
\par

Assume first that $G$ is a Lie group;
we refer to \cite[Theorem 3.1 of Chapter XV and Observation Page 186]{Hoch--65}.
Credits for this are shared by several authors:
``The names to be attached to this result are:
\'E.\ Cartan, C.\ Chevalley, K.\ Iwasawa, A.\ Malcev, G.D.\ Mostow''
(quoted from \cite[Preface]{Hoch--65}).
\par

Assume now that $G$ is a connected-by-compact LC-group.
Let $N$ be as in Theorem \ref{Yamabe} (for an arbitrary $U$).
Inverse images of maximal compact subgroups of $G/N$
are maximal compact subgroups of $G$,
and are conjugate with each other.
Let $K$ be one of them.
Then $G/K \approx (G/N)/(K/N)$,
and the right-hand side is homeomorphic to $\R^n$ for some $n \ge 0$,
by the Lie group case.
\par

Let us check the equality $G_0K=G$
(even for $G$ a Lie group, it is not explicitly in \cite{Hoch--65}).
Since $K$ is compact, $G_0K$ is closed.
The homogeneous space $G/G_0K$
is both totally disconnected (as a quotient of $G/G_0$)
and connected (as a quotient of $G/K$),
hence is just one point.
\par

Set
\begin{equation*}
J \, = \, \bigcap_{g \in G} gKg^{-1} \, = \, 
\{ g \in G \mid g \hskip.2cm
\text{acts as the identity on} \hskip.2cm
G/K \approx \R^n \} .
\end{equation*}
It is a compact normal subgroup of $G$ contained in $K$.
Since the natural action of $G/J$ on $G/K$ is faithful,
$G/J$ is a Lie group by the next theorem.
Thus $\R^n \approx G/K \approx (G/J)/(K/J)$ has a natural analytic structure
which is invariant by $G/J$, and a fortiori by $G$
\cite[Theorems 2.1 and 3.1 of Chapter VIII]{Hoch--65}.
\end{proof}

\begin{rem}
\label{remsursuiteYamabe}
(1)
Let $G$ be a connected-by-compact LC-group,
and $K$ a maximal compact subgroup, as in Theorem \ref{deHochschild+}.
The topology of $K$ can be as rich as that of $G/K$ is poor.
\par
The topology of connected compact Lie groups has been intensively studied
in the period 1925--1960 by many authors
including Weyl (one of his first results in this field is that $\pi_1(G)$ is finite
for a semisimple connected compact Lie group),
\'E.~Cartan, A.~Borel, and Bott (periodicity theorems, 
including the isomorphisms
$\pi_{j+2}(\lim_{n \to \infty} \operatorname{U}(n)) \simeq 
\pi_{j}(\lim_{n \to \infty} \operatorname{U}(n))$).
See for example \cite{Bore--01}, \cite{Same--52}, \cite{Bore--55}, and \cite{Bott--59}.
\index{Compact group! topology}

\vskip.2cm

(2)
A compact group $K$ is homeomorphic to the direct product
$K_0 \times K/K_0$ \cite[Corollary 10.38, Page 559]{HoMo--06}.

\vskip.2cm

(3)
Without the hypothesis ``connected-by-compact'',
each conclusion of Theorem \ref{deHochschild+} may fail,
as the following examples show.
\par

(3$_{\textnormal{a}}$)
An infinite locally finite group does not contain maximal finite subgroups.
Recall that a group $\Gamma$ is \textbf{locally finite} 
\index{Locally finite! group|textbf}
if every finite subset of $\Gamma$ is contained in a finite subgroup of $\Gamma$.
\par

(3$_{\textnormal{b}}$)
Let $A,B$ be two non-trivial finite groups.
In the free product $A \Conv B$,
\index{Free product}
maximal finite subgroups are either conjugate of $A$ or conjugate of $B$,
and $A,B$ are not conjugate with each other
\cite[Section 4.1, Problem 11]{MaKS--66}.
\par
For an integer $n \ge 2$ and a prime $p$,
the locally compact group $\SL_n(\Q_p)$ 
\index{Special linear group $\SL$! $\SL_n(\Q_p)$}
contains $n$ distinct conjugacy classes of maximal compact subgroups
(\cite[Proposition 3.14]{PlRa--94}, see also \cite{Bruh--64}).
Note that, in $\GL_n(\Q_p)$, 
maximal compact subgroups are conjugate with each other. 
\par

(3$_{\textnormal{c}}$)
Let $\Gamma$ be an infinite discrete group and $K$ a finite subgroup.
The quotient space $\Gamma/K$ is not homeomorphic
to a Euclidean space.
\end{rem}

\begin{thm}[Montgomery-Zippin] \index{Montgomery-Zippin theorem}
\index{Theorem! Montgomery-Zippin}
\label{MontgomeryZipin}
Let $G$ be a $\sigma$-compact LC-group that acts
transitively, faithfully, and continuously on a connected topological manifold.
\par
Then $G$ is isomorphic to a Lie group.
\end{thm}

\begin{proof}
This follows simply from \cite[Corollary of Section 6.3]{MoZi--55}.
\par
A convenient reference is
\cite[Proposition 1.6.5, Pages 116--126]{Tao--14}.
Tao's proof uses first Proposition \ref{Freudenthal}, to show that
the natural map from $G/H$ to the manifold
is a homeomorphism
(where $H$ is some isotropy subgroup), 
and then the van Dantzig $\&$ Gleason-Yamabe Theorems, 
\ref{vanDantzig} $\&$ \ref{Yamabe}.
\end{proof}

In particular, if a $\sigma$-compact LC-group $G$
is locally homeomorphic to a Euclidean space,
Theorem \ref{MontgomeryZipin} applied to the action of $G_0$ on itself
implies that $G$ is isomorphic to a Lie group:
this is the solution of Hilbert Fifth Problem
\index{Hilbert Fifth Problem}.

\begin{cor}
\label{maximalcompact}
Let $G$ be an LC-group.
There exist a compact subgroup $K$ of $G$,
a non-negative integer $n$, 
and a discrete subspace $D$ of $G$,
such that the homogeneous space $G/K$ is homeomorphic to $D \times \R^n$.
The identity component $G_0$ acts transitively on each connected component of $G/K$.
\par

On the space $D \times \R^n$, there exists a $G$-invariant structure
of analytic manifold (in general non-connected).
\end{cor}

\begin{proof}
By Corollary \ref{vanDantzigCor} of van Dantzig Theorem,
there exists a connected-by-compact open subgroup $H$ of $G$.
Let $K$ be a maximal compact subgroup of $H$.
By Theorem \ref{deHochschild+},
we know that $H/K \approx \R^n$ for some $n \ge 0$, 
and that $H_0$ acts transitively on $H/K$,
so that $H_0/(K \cap H_0) = H/K$.
\par

Let $D$ be a discrete subspace of $G$ 
such that $G = \bigsqcup_{d \in D} dH$ (disjoint union),
i.e., such that $G$ is homeomorphic to $D \times H$;
without loss of generality, we can assume that $1 \in D$.
Then 
\begin{equation*}
G/K \, = \,  \bigsqcup_{d \in D} d(H/K) 
\, = \,  \bigsqcup_{d \in D} d(H_0/(K \cap H_0))
\approx D \times \R^n .
\end{equation*}
Since the identity component $G_0$ contains $H_0$ and is connected,
$G_0 / (K \cap G_0)$ is connected.
Therefore $H/K = H_0 / (K \cap H_0) = G_0 / (K \cap G_0) \approx \{1\} \times \R^n$
is $G_0$-invariant; and $G_0$ acts transitively on $\{1\} \times \R^n$.
Since $G_0$ is normal in $G$, the group $G_0$ acts transitively
on each component $\{d\} \times \R^n$.
\par

The last claim follows from Theorem \ref{deHochschild+}\ref{4DEdeHochschild+}.
\end{proof}

Compact normal subgroups appear in Theorem \ref{Yamabe}.
We continue on this theme as follows.

\begin{defn}
\label{defW(G)}
The \textbf{polycompact radical} of a topological group $G$ 
\index{Radical! polycompact radical|textbf}
is the union $W(G)$ of all compact normal subgroups of $G$.
Actually, it is a subgroup, indeed a topologically characteristic subgroup.
\par
A topological group $G$ \textbf{has a compact radical}
\index{Radical! compact radical|textbf}
if $W(G)$ is compact;
in this case $W(G)$ is called the compact radical of $G$,
and $W(G/W(G)) = \{1\}$.
\index{Normal compact subgroup}
\end{defn}

\begin{rem}
\label{remW(G)}
\par
For $G$ and $W(G)$ as above, 
note that $W(G)$ is compact if and only if it is relatively compact.
Indeed if the latter holds, then $\overline{W(G)}$ is a compact normal subgroup 
and hence is contained in $W(G)$, whence $W(G)=\overline{W(G)}$.
\end{rem}

\begin{exe}
\label{exW(G)1}
We show an example where the polycompact radical is not compact,
and another one where the polycompact radical is not closed.

\vskip.2cm 

(1)
If $G = \Q_p$ for some prime $p$,
or if $G$ is an infinite locally finite discrete abelian group,
then $W(G) = G$; in particular, $W(G)$ is closed in $G$ and is not compact.

\vskip.2cm

(2)
Consider an integer $n \ge 3$, the finite cyclic group $A = \Z / n\Z$,
the group of units $B = (\Z / n\Z)^\times$ of $A$ viewed as a ring,
and the semidirect product
\index{Semidirect product! $A^{(\N)} \rtimes B^N$}
\begin{equation*}
G \, = \, \Big( \bigoplus_{k \ge 1} A_k \Big) \rtimes \Big( \prod_{k \ge 1} B_k \Big) ,
\end{equation*}
where every $A_k$ is a copy of $A$ and every $B_k$ a copy of $B$.
The action of the product on the direct sum is the natural one,
for which $B_k$ acts on the corresponding $A_k$ by multiplication;
the topology is that for which
$\bigoplus_{k \ge 1} A_k$ is discrete 
and $\prod_{k \ge 1} B_k$ has the compact product topology.
It can be checked that  
\begin{equation*}
W(G) \, = \, \Big( \bigoplus_{k \ge 1} A_k \Big) \rtimes \Big( \bigoplus_{k \ge 1} B_k \Big) .
\end{equation*}
In particular, $W(G)$ is a dense, proper subgroup of $G$ and is not closed. 
This group is a particular case of Example 1 in Section 6 of \cite{WuYu--72},
related to an example in the proof of Proposition 3 in \cite{Tits--64}.

\vskip.2cm

It is known that, in a compactly generated LC-group, the polycompact radical is closed
\cite[Theorem 1]{Corn--15}.
\end{exe}

\begin{exe}
\label{exW(G)2}
We describe compact radicals in five classes of LC-groups.

\vskip.2cm

(1)
Let $G$ be a connected-by-compact LC-group,
as in Theorems \ref{Yamabe} and \ref{deHochschild+}.
Then $G$ has a compact radical
$W(G) = \bigcap_{g \in G} g K g^{-1}$,
where $K$ is any maximal compact subgroup of $G$.
The quotient $G/W(G)$ is a Lie group
with finitely many connected components.
\par
The quotient $G/W(G)$ need not be connected;
indeed, if $G = \PGL_2(\R)$, then $W(G) = \{1\}$
and $G/W(G)=G$ has two connected components.
\index{Projective linear group PGL, and PSL! $\PGL_2(\R)$}

\vskip.2cm

(2)
Let $A,B$ be two LC-groups, $C$ a proper open subgroup of both $A$ and $B$,
and $G = A \Conv_C B$ the locally compact amalgamated product,
as in Proposition \ref{topamalgamation} below.
\index{Amalgamated product}
Suppose moreover that $W(C)$ is compact.
Then $G$ has a compact radical: $W(G)$ is the largest compact subgroup of $C$
which is normal in both $A$ and $B$.
Compare with \cite{Corn--09}.
\vskip.2cm

(3)
Let $H$ be an LC-group, 
$K, L$ two isomorphic open subgroups,
$\varphi : K \overset{\simeq}{\longrightarrow} L$ a topological isomorphism,
and $G = \HNN (H, K, L, \varphi)$ the corresponding HNN-extension,
as in Proposition \ref{topHNN} below.
\index{HNN! extension}
Suppose moreover that $W(K)$ is compact.
Then $G$ has a compact radical:
$W(G)$ is the largest compact subgroup of $K$
which is normal in $H$ and invariant by $\varphi$.

\vskip.2cm

(4)
Let $G$ be a Gromov-hyperbolic LC-group (see Remark \ref{GromovHyp}).
\index{Gromov! hyperbolic group}
Then $G$ has a compact radical;
if moreover $G$ is not 2-ended, 
then $W(G)$ is the kernel of the $G$-action on its boundary.
See \cite[Lemma 5.1]{CCMT--15}.

\vskip.2cm

(5)
An LC-group $G$ acting geometrically on a CAT($0$)-space $X$
has a compact radical. 
(See Section \ref{sectionactions} for ``geometrically'', 
and \cite{BrHa--99} for CAT($0$)-spaces.)
\par

Indeed, let $B$ be a bounded subset of $X$ such that $X = \bigcup_{g \in G} gB$.
Set $\Omega = \{ g \in G \mid gB \cap B \ne \emptyset \}$;
by properness, $\Omega$ is relatively compact in $G$.
We claim that every compact normal subgroup $K$ of $G$ is contained in $\Omega$;
modulo this claim, the union $W(G)$ of all compact normal subgroups of $G$ is in $\Omega$.
Hence $W(G) \subset \Omega$ is relatively compact, and therefore compact
by Remark \ref{remW(G)}.
\par

Let us finally prove the claim. Let $K$ be a normal compact subgroup of $G$.
Let $X^K$ denote the subspace of $X$ of the points fixed by $K$.
As $X$ is a CAT($0$)-space $X^K$ is non-empty.
As $K$ is normal, $X^K$ is $G$-invariant; it follows that $X^K$ contains a point $x \in B$.
Since $kx=x$ for all $k \in K$, we have $K \subset \Omega$.
\par

More generally, the argument of (5) holds for every metric space $X$ 
with the following property:
every group acting on $X$ by isometries with a bounded orbit has a fixed point.
The property holds for CAT($0$)-spaces by Proposition 2.7 of Chapter II.2 in \cite{BrHa--99}.
\end{exe}

\chapter[Metric coarse and large-scale categories]
{The metric coarse category and the large-scale category of pseudo-metric spaces}
\label{chap_metriccoarse}

\section[Coarsely Lipschitz and large-scale Lipschitz maps]
{Coarsely Lipschitz maps and large-scale Lipschitz maps}
\label{coarselyLipschitzandlargescaleLipschitz}

Recall that, in this book, $\R_+ = \mathopen[0,+\infty\mathclose[$ 
and $\overline{\R}_+ = \mathopen[0,+\infty\mathclose]$.
It will be convenient to call an
\textbf{upper control} \index{Upper control|textbf}
a non-decreasing function 
$\Phi_+ : \R_+ \longrightarrow \R_+$,
and a \textbf{lower control} \index{Lower control|textbf} 
a non-decreasing function 
$\Phi_- : \R_+ \longrightarrow \overline \R_+$
such that $\lim_{t \to \infty} \Phi_-(t) = \infty$.

\begin{defn}[controls]
\label{defuclc}
Let $X,Y$ be two pseudo-metric spaces (as in \ref{pm...metrizable})
and $f : X \longrightarrow Y$ a map.
An \textbf{upper control for $f$} is an upper control $\Phi_+$ such that
\begin{equation*}
d_Y(f(x),f(x')) \, \le \, \Phi_+(d_X(x,x'))
\hskip.5cm \text{for all} \hskip.2cm x,x' \in X .
\end{equation*}
A \textbf{lower control for $f$} is a lower control $\Phi_-$ such that
\begin{equation*}
d_Y(f(x),f(x')) \, \ge \, \Phi_-(d_X(x,x')) 
\hskip.5cm \text{for all} \hskip.2cm x,x' \in X .
\end{equation*}
\end{defn}

It is suitable to allow infinite values for lower controls,
for example to allow $\Phi_-(t) = \infty$ for $f : X \longrightarrow Y$
with $X$ of finite diameter, say $D$, and $t > D$.
See the proof of Proposition \ref{reformulation_c_et_ce}.
\par

If there exist controls $\Phi_+, \Phi_-$ for $f$, as above, 
note that there exist  \emph{continuous}  controls 
$\widetilde \Phi_-, \widetilde \Phi_+$ for $f$ such that 
\begin{equation*}
\aligned
\widetilde \Phi_-(t) \, &\le \,  \Phi_-(t)
\hskip.2cm \text{for all} \hskip.2cm t \in \R_+ ,
\hskip.5cm \text{and} \hskip.5cm
\widetilde \Phi_-(0) = 0 ,
\\
\Phi_+(t) \, &\le \, \widetilde \Phi_+(t)
\hskip.2cm \text{for all} \hskip.2cm t \in \R_+ .
\endaligned
\end{equation*} 
It suffices to set $\widetilde \Phi_-(t) = \int_{\max \{t-1,0\} }^t \min\{ u, \Phi_- (u) \} du$
and $\widetilde \Phi_+(t) = \int_t^{t+1} \Phi_+ (u) du$.

\begin{defn}
\label{defcobounded}
A subspace $Z$ of a pseudo-metric space $Y$ is \textbf{cobounded} if
\begin{equation*}
\sup\{ d(y,Z) \mid y \in Y \} \, < \,  \infty .
\end{equation*}
\end{defn}
\index{Cobounded! cobounded subspace of a pseudo-metric space|textbf}

Note that the empty subspace is cobounded if $Y = \emptyset$,
and is not if $Y \ne \emptyset$.

\begin{defn}
\label{defcoarse}
Let $X,Y$ be two pseudo-metric spaces.
A map $f : X \longrightarrow Y$ is
\begin{enumerate}[noitemsep,label=(\alph*)]
\item\label{aDEdefcoarse}
\textbf{coarsely Lipschitz} \index{Coarsely! Lipschitz map|textbf}
if there exists an upper control for $f$;
\item\label{bDEdefcoarse}
\textbf{coarsely expansive} \index{Coarsely! expansive map|textbf}
if there exists a lower control for $f$;
\item\label{cDEdefcoarse}
a \textbf{coarse embedding} \index{Coarse! embedding|textbf}
if it is coarsely Lipschitz and coarsely expansive;
\item\label{dDEdefcoarse}
\textbf{essentially surjective} \index{Essentially surjective map|textbf}
if $f(X)$ is cobounded in $Y$;
\item\label{eDEdefcoarse}
a \textbf{metric coarse equivalence} \index{Metric coarse! equivalence|textbf}
if it is coarsely Lipschitz, coarsely expansive, and essentially surjective.
\end{enumerate}
Two pseudo-metric spaces are \textbf{coarsely equivalent}
\index{Coarsely equivalent spaces|textbf}
if there exists a metric coarse equivalence from one to the other
(and thus conversely, see Proposition \ref{epimonoiso}\ref{3DEepimonoiso}).
\par
Two maps $f,f'$ from $X$ to $Y$ are \textbf{close},
\index{Close maps from one pseudo-metric space to another|textbf}
and we write $f \sim f'$, if
\index{$aa$@$\sim$ various equivalence relations}
\begin{equation*}
\sup_{x \in X} d_Y(f(x), f'(x)) \, < \,  \infty .
\end{equation*}
Observe that, for maps, 
closeness is an equivalence relation.
\end{defn}

\begin{rem}
\label{terminologynotestablished3A}
The terminology is not well-established.
\par

If $f,f' : X \longrightarrow Y$ are two maps between metric spaces,
other terms used for ``close'' are
 ``parallel''  \cite[1.A', Page 23]{Grom--93},
``at bounded distance'', ``equivalent'', and ``bornotopic''.
\par

A metric coarse equivalence here is called a ``quasi-isometry'' 
in \cite[Def.\ 2.1.1]{Shal--04},
in contradiction with the meaning of ``quasi-isometry'' used in this book.
Coarsely Lipschitz maps here are ``uniformly bornologous'' maps in \cite{Roe--93, Roe--03}.
Coarsely expansive maps here are ``effectively proper maps'' in \cite{BlWe--92}
and  ``uniformly expansive maps'' in \cite{BeDr--08}.
Coarse embeddings have been introduced under the name of ``placements'' or ``placings''
in \cite[Section 4.1]{Grom--88};
they are the ``effectively proper Lipschitz maps'' in \cite{BlWe--92},
and the ``uniform embeddings'' in  \cite[7.E, Page 211]{Grom--93}.
In \cite[2.D, Page 6]{Grom--93}, two metrics $d_1, d_2$ on a space $X$ are
``uniformly equivalent on the large scale'' if the identity $(X,d_1) \longrightarrow (X,d_2)$
is a metric coarse equivalence.
\end{rem}

Definitions \ref{defcoarse}\ref{aDEdefcoarse} and \ref{defcoarse}\ref{bDEdefcoarse} 
can be reformulated without explicit references to controls:

\begin{prop}[coarsely Lipschitz and coarsely expansive maps]
\label{reformulation_c_et_ce}
Let $X,Y$ be two pseudo-metric spaces and $f : X \longrightarrow Y$ a map.
\par
The following properties are equivalent:
\begin{enumerate}[noitemsep,label=(\arabic*)]
\item\label{1DEreformulation_c_et_ce}
the map $f$ is coarsely Lipschitz;
\item\label{2DEreformulation_c_et_ce}
for all $R \ge 0$, there exists $S \ge 0$ such that,
if $x,x' \in X$ satisfy $d_X(x,x') \le R$, then $d_Y(f(x), f(x')) \le S$;
\item\label{3DEreformulation_c_et_ce}
for every pair $(x_n)_{n \ge 0}$, $(x'_n)_{n \ge 0}$ 
of sequences of points in $X$
with \newline
$\sup_{n \ge 0} d_X(x_n, x'_n) < \infty$,
we have $\sup_{n \ge 0} d_Y(f(x_n), f(x'_n)) < \infty$.
\end{enumerate}
Similarly, the following properties are equivalent:
\begin{enumerate}[noitemsep,label=(\arabic*)]
\addtocounter{enumi}{3}
\item\label{4DEreformulation_c_et_ce}
The map $f$ is coarsely expansive;
\item\label{5DEreformulation_c_et_ce}
for all $s \ge 0$, there exists $r \ge 0$ such that,
if $x,x' \in X$ satisfy $d_X(x,x') \ge r$, then $d_Y(f(x), f(x')) \ge s$;
\item\label{6DEreformulation_c_et_ce}
for every pair $(x_n)_{n \ge 0}$, $(x'_n)_{n \ge 0}$ 
of sequences of points in $X$
with \newline 
$\lim_{n \to \infty} d_X(x_n, x'_n) = \infty$,
we have $\lim_{n \to \infty} d_Y(f(x_n), f(x'_n)) = \infty$.
\end{enumerate}
\end{prop}

\begin{proof}
Implications \ref{1DEreformulation_c_et_ce}  $\Rightarrow$
\ref{2DEreformulation_c_et_ce}  $\Rightarrow$ \ref{3DEreformulation_c_et_ce}
are straightforward.

Suppose that $f$ satisfies Condition \ref{3DEreformulation_c_et_ce}, 
and let us check that $f$ is coarsely Lipschitz.
Define 
\begin{equation*}
\Phi_+(c) = \sup \{ d_Y(f(x),f(x')) \mid
x,x' \in X \hskip.2cm \text{with} \hskip.2cm d_X(x,x') \le c \}
\hskip.5cm \text{for all} \hskip.2cm c \in \R_+.
\end{equation*}
We have to check that $\Phi_+$ is an upper control for $f$.
Since $\Phi_+$ is obviously non-decreasing,
it is enough to check that $\Phi_+$ takes finite values only.
Suppose ab absurdo that one had $\Phi_+(c) = \infty$ for some $c$;
there would exist two sequences $(x_n)_{n \ge 0}$, $(x'_n)_{n \ge 0}$ in $X$
such that $\sup_{n \ge 0} d_X(x_n,x'_n) \le c$,
and $\lim_{n \to \infty} d_Y(f(x_n),f(x'_n)) = \infty$,
and this would contradict the condition above.

\vskip.2cm

Similarly, implications \ref{4DEreformulation_c_et_ce}  $\Rightarrow$
\ref{5DEreformulation_c_et_ce}  $\Rightarrow$ \ref{6DEreformulation_c_et_ce}
are straightforward.

Suppose that $f$ satisfies Condition \ref{6DEreformulation_c_et_ce}, 
and let us check that $f$ is coarsely expansive.
We proceed as above, defining at the appropriate point
\begin{equation*}
\Phi_-(c) = \inf \{ d_Y(f(x),f(x')) \mid
x,x' \in X \hskip.2cm \text{with} \hskip.2cm d_X(x,x') \ge c \}
\hskip.5cm \text{for all} \hskip.2cm c \in \R_+.
\end{equation*}
[Note that, in case $X$ has finite diameter, say $D$,
then $\Phi_-(t) = \infty$ for $t > D$.]
\end{proof}

The control $\Phi_-$ 
of the previous proof is the \textbf{compression function} 
\index{Compression function|textbf}
 of $f$,
and $\Phi_+$ is the \textbf{dilation function} 
\index{Dilation function|textbf} 
of $f$.

\begin{prop}[coarse properties of maps and closeness]
\label{coarsestability}
Let $X,Y, Z$ be three pseudo-metric spaces 
$f,f' : X \longrightarrow Y$  two close maps,
and $g, g' : Y \longrightarrow Z$ two close maps.
\begin{enumerate}[noitemsep,label=(\arabic*)]
\item\label{1DEcoarsestability}
The map $f'$ is coarsely Lipschitz
[respectively coarsely expansive, a coarse embedding, essentially surjective, 
a metric coarse equivalence]
if and only if $f$ has the same property.
\item\label{2DEcoarsestability}
The composite maps $g \circ f$ and $g' \circ f'$ are close.
\item\label{3DEcoarsestability}
If $f$ and $g$ are coarsely Lipschitz
[respectively coarsely expansive, coarsely Lipschitz and essentially surjective],
then the composition $g \circ f$ has the same property.
\end{enumerate}
\end{prop}

The proof is straightforward. $\square$

\begin{defn}[metric coarse category]
\label{defcoarsecat}
Let $X,Y$ be two pseudo-metric spaces.
A \textbf{coarse morphism} \index{Coarse! morphism|textbf} 
from $X$ to $Y$ is a closeness class of coarsely Lipschitz maps from $X$ to $Y$.
By abuse of notation, we often denote a coarsely Lipschitz map and its class
by the same letter.
\par
The \textbf{metric coarse category} \index{Metric coarse! category|textbf}
is the category whose objects are pseudo-metric spaces
and whose morphisms are coarse morphisms.
\end{defn}

\begin{defn}
\label{defqi}
Let $X,Y$ be two pseudo-metric spaces and $f : X \longrightarrow Y$ a map.
Then $f$ is
\begin{enumerate}[noitemsep,label=(\alph*)]
\item\label{aDEdefqi}
\textbf{large-scale Lipschitz} \index{Large-scale! Lipschitz map|textbf}
if it has an affine upper control, 
in other words if there exist constants $c_+ > 0$, $c_+' \ge 0$
such that
$d_Y(f(x),f(x')) \le c_+ d_X(x,x') + c_+'$ for all $x,x' \in X$;
\item\label{bDEdefqi}
\textbf{large-scale expansive} \index{Large-scale! expansive map|textbf}
if it has an affine lower control,
in other words if there exist constants $c_- > 0$, $c_-' \ge 0$
such that
$d_Y(f(x),f(x')) \ge c_- d_X(x,x') - c_-'$ for all $x,x' \in X$;
\item\label{cDEdefqi}
\textbf{large-scale bilipschitz}, \index{Large-scale! bilipschitz map|textbf}
or a \textbf{quasi-isometric embedding}, \index{Quasi-isometric embedding|textbf}
if it is large-scale Lipschitz and
large-scale expansive;
\addtocounter{enumi}{1}
\item\label{eDEdefqi}
\textbf{a quasi-isometry} \index{Quasi-isometry|textbf}
if it is large-scale bilipschitz and essentially surjective.
\end{enumerate}
An item (d), missing above, is identical with that of Definition \ref{defcoarse}.
For a reformulation of the definition of quasi-isometry,
see Remark \ref{quasiisopartiellementdef}.
\par

Two pseudo-metric spaces are \textbf{quasi-isometric}
\index{Quasi-isometry! quasi-isometric spaces}
if there exists a quasi-isometry from one to the other
(and thus conversely, see Proposition \ref{epimonoisobis}\ref{3DEepimonoisobis}).
\end{defn}

\begin{rem}
Let us again record some straightforward observations.
To be large-scale Lipschitz is invariant under closeness of maps.
The composition of two large-scale Lipschitz maps is large-scale Lipschitz.
Note that every large-scale Lipschitz map between pseudo-metric spaces is coarsely Lipschitz;
similarly a large-scale expansive map is coarsely expansive.
\par

The words ``coarse equivalence''  and ``quasi-isometric'' appear in \cite{Grom--81},
not quite as they do here (this article is an early discussion of Gromov
on the notion of hyperbolicity).
The terminology ``large-scale Lipschitz'' appears in
\cite[1.A', Page 22]{Grom--93}.
\par

Quasi-isometry can be extended from maps to \emph{relations}, e.g., to ``multi-valued maps''. 
See \cite{Cann--02}, where quasi-isometries are called ``quasi-Lipschitz equivalences''.
\end{rem}

We have an analogue of Proposition \ref{coarsestability}:

\begin{prop}[large-scale properties of maps and closeness]
\label{quasiisostability}
Let $X,Y, Z$ be three pseudo-metric spaces 
$f,f' : X \longrightarrow Y$  two close maps,
and $g, g' : Y \longrightarrow Z$ two close maps.
\begin{enumerate}[noitemsep,label=(\arabic*)]
\item\label{1DEquasiisostability}
The map $f'$ is large-scale Lipschitz
[respectively large-scale expansive, large-scale bilipschitz, 
a quasi-isometry]
if and only if $f$ has the same property.
\item\label{2DEquasiisostability}
If $f$ and $g$ are  large-scale Lipschitz
[respectively large-scale expansive, large-scale Lipschitz and essentially surjective],
then the composition $g \circ f$ has the same property.
\end{enumerate}
\end{prop}

\begin{defn}
[large-scale category]
\label{deflargescalecat}
Let $X,Y$ be two pseudo-metric spaces.
A \textbf{large-scale Lipschitz morphism} \index{Large-scale! Lipschitz morphism|textbf}
from $X$ to $Y$
is a closeness class (in the sense of Definition \ref{defcoarse}) 
of large-scale Lipschitz maps from $X$ to $Y$.

The \textbf{large-scale category} \index{Large-scale! category|textbf}
is the subcategory of the metric coarse category
whose objects are pseudo-metric spaces
and whose morphisms are large-scale Lipschitz morphisms.

In particular, it is a \emph{wide subcategory} of the metric coarse category
(same objects, ``less'' morphisms).
\end{defn}

For later reference,
we recall the following more classical definition.
Compare \ref{a'DELipschitz}, \ref{c'DELipschitz}, \ref{e'DELipschitz},
below with \ref{aDEdefqi}, \ref{cDEdefqi}, \ref{eDEdefqi}
of Definition \ref{defqi}.

\begin{defn}
\label{Lipschitz}
Let $X,Y$ be two pseudo-metric spaces, and $f : X \longrightarrow Y$ a map.
Then $f$ is
\begin{enumerate}[noitemsep,label=(\alph*$'$)]
\item\label{a'DELipschitz}
\textbf{Lipschitz} \index{Lipschitz! map|textbf}
if the condition of \ref{defqi}\ref{aDEdefqi}
holds with $c_+'  = 0$,
\addtocounter{enumi}{1}
\item\label{c'DELipschitz}
\textbf{bilipschitz} \index{Bilipschitz! bilipschitz map|textbf}
if there exist constants $c_+, c_- > 0$ such that
\hfill\newline
$c_-d_X(x,x') \le d_Y(f(x),f(x')) \le c_+ d_X(x,x')$
for all $x,x' \in X$,
\addtocounter{enumi}{1}
\item\label{e'DELipschitz}
a \textbf{bilipschitz equivalence} \index{Bilipschitz! bilipschitz equivalence|textbf}
if it is bilipschitz and onto.
\end{enumerate}
\par
The \textbf{Lipschitz metric category} \index{Lipschitz! metric category|textbf}
is the category 
whose objects are metric spaces
and whose morphisms are Lipschitz maps.
\par
The \textbf{Lipschitz pseudo-metric category}
 \index{Lipschitz! pseudo-metric category|textbf}
is the category 
whose objects are pseudo-metric spaces
and whose morphisms are equivalence classes of Lipschitz maps,
where two maps $f,f'$ from $X$ to $Y$ are equivalent
if $d_Y(f(x), f'(x)) = 0$ for all $x \in X$.
\par
Observe that the Lipschitz metric category is a full subcategory 
of the Lipschitz pseudo-metric category,
and also a subcategory of that of metric spaces and continuous maps.
\end{defn}

There is a variation: the \textbf{pointed Lipschitz metric category},
\index{Pointed Lipschitz metric category|textbf}
of pointed non-empty metric spaces and base point preserving Lipschitz maps.

\begin{rem}
\label{remarks_categories}
(1) 
A large-scale Lipschitz map between metric spaces, 
and a fortiori a coarsely Lipschitz map,    
\emph{need not} be continuous.
\par
\index{Floor function}
For example, if $\Z$ and $\R$ are given their natural metrics, 
defined by $d(x,y) = \vert y-x \vert$, the floor function 
$\R \longrightarrow \Z, \hskip.1cm x \longmapsto \lfloor x\rfloor$
of Example \ref{properandlbex} is a quasi-isometry.

\vskip.2cm

(2)
Let $X,Y$ be pseudo-metric spaces.
A map from $f : X \longrightarrow Y$ is \textbf{metrically proper}
if $f^{-1}(B)$ is bounded in $X$ for every bounded subset $B$ of $Y$.
\index{Metrically proper! map|textbf}
\index{Proper! metrically proper map|textbf}
\par

A coarsely expansive map is metrically proper.
But a metrically proper map need not be coarsely expansive,
as shown by the map $x \longmapsto \sqrt{x}$
from $\R_+$ to itself, with the usual metric.
\par

Coarse expansiveness can be viewed as
uniform notion of metric properness.

\vskip.2cm

(3)
Let $X,Y$ be pseudo-metric spaces. Assume that $X$ is metric,
and that there exists $c > 0$
such that $d_X(x,x') \ge c$ for every pair $(x,x')$ of distinct points of $X$.
(For example, assume that $(X,d)$ is a group with a word metric.)
A map $f : X \longrightarrow Y$ is large-scale Lipschitz if and only if it is Lipschitz.

\vskip.2cm

(4) 
Given a non-principal ultrafilter $\omega$ on $\N$, 
there is a functor called \emph{asymptotic cones}
from the large-scale category to the pointed Lipschitz metric category.
A basic reference for asymptotic cones is \cite[Chapter 2]{Grom--93};
later ones include \cite{Drut--01} and \cite{Corn--11}. 
\vskip.2cm

(5)
Notions as in Definition \ref{defqi} 
have been used in the early 1960's by Efremovi\v c, Tihomirova,
and others in the Russian school.
They play an important role in articles by Mostow and Margulis.
\par

A map satisfying the inequalities of \ref{defqi}\ref{aDEdefqi} and \ref{defqi}\ref{bDEdefqi}
with $c_+' = c_-' = 0$ is usually called a {\it bilipschitz embedding}; 
in \cite[$\S$~9, Page 66]{Most--73}, it was called a pseudo-isometry.
In \cite{Marg--70}, Margulis observes that two finite generating sets
for a group $\Gamma$ provide two metrics $d_1, d_2$ on $\Gamma$ 
that are bilipschitz equivalent.
\par

The rigidity theorems of Mostow were decisive in establishing
\index{Theorem! Mostow rigidity} \index{Mostow rigidity}
the importance of notions like quasi-isometry or pseudo-isometry
(see Example \ref{MostowRigidity}).
This was of course due to Mostow himself,
but also to Thurston, Gromov, and others, 
who revisited the rigidity theorem on several occasions;
see e.g.\ Chapter 5 in \cite{Thur--80}, \cite{Grom--84},  Lecture IV in \cite{BaGS--85},
as well as \cite{KlLe--97}.

\vskip.2cm

(6)
The notion of coarse embedding has been used in establishing
cases of the Baum-Connes conjecture.
More precisely, let $\Gamma$ be a countable group,
viewed as a metric space
for some adapted metric (see Section \ref{adaptedpseudometric});
if $\Gamma$ admits a coarse embedding into a Hilbert space,
then $\Gamma$ satisfies the Baum-Connes conjecture.
See Theorem 6.1 in \cite{SkTY--02},
and related articles by
Higson, Kasparov, Skandalis, Tu, and Yu \cite{Yu--00}.
Later, coarse embeddings
into uniformly convex Banach spaces appear in the same context \cite{KaYu--06}.

\vskip.2cm

(7)
It is crucial and now well-recognized 
(but it has not always been so)
that one should carefully distinguish
between coarsely Lipschitz maps and large-scale Lipschitz maps, 
and similarly distinguish between metric coarse equivalences and quasi-isometries.

\vskip.2cm

(8)
It is sometimes useful to compare pseudo-metrics in a finer way
than in Definitions \ref{defcoarse}\ref{eDEdefcoarse}, 
\ref{defqi}\ref{eDEdefqi},
and even \ref{Lipschitz}\ref{e'DELipschitz}.
The following is used in \cite{AbMa--04}:
two pseudo-metrics $d_1, d_2$ on a set $X$ are \textbf{coarsely equal}
if there exists a constant $c \ge 0$ such that
$\vert d_2 (x,x') - d_1(x,x') \vert \le c$ for all $x,x' \in X$,
i.e., if the identity from $(X,d_1)$ to $(X,d_2)$ 
is a a quasi-isometry with multiplicative constants $c_+ = c_- = 1$.
\end{rem}

\begin{exe}
\label{excoarse}
(1)
If $Y$ is a pseudo-metric space,
the empty map $\emptyset \longrightarrow Y$
is a coarse embedding.
If $Y$ is non-empty, it is not essentially surjective:
$d_Y(y,\emptyset) = \infty$ for all $y \in Y$.

\vskip.2cm

(2) 
Let $X,Y$ be two pseudo-metric spaces and $f : X \longrightarrow Y$ a map.
If $X$ has finite diameter, then $f$ is large-scale expansive;
the function defined by 
$\Phi_-(s) = \max \{s - \operatorname{diam}(X), 0 \}$ for all $s \in \R_+$
is a lower control for $f$. 
If $f(X)$ has finite diameter, then $f$ is large-scale Lipschitz;
the function defined by $\Phi_+(s) = \operatorname{diam}(f(X))$ for all $s \in \R_+$
is an upper control for $f$.
If $Y$ has finite diameter and $X$ is non-empty, $f$ is essentially surjective.
\par
In particular, every non-empty pseudo-metric space of finite diameter 
is quasi-isometric to the one-point space.

\vskip.2cm

(3)
Let $a \in \R_+^\times$. The map 
$\R_+ \longrightarrow \R_+, \hskip.1cm x \longrightarrow x^a$
is coarsely Lipschitz if and only if it is large-scale Lipschitz,
if and only if $a \le 1$.

\vskip.2cm

(4)
Let $(X,d_X)$ be a pseudo-metric space.
Let $X_{\operatorname{Haus}}$ be the 
\textbf{largest Hausdorff quotient} of $X$,
\index{Hausdorff! largest Hausdorff quotient}
more precisely the quotient of $X$ by the relation ``$x \sim y$ if $d_X(x,y) = 0$''.
\index{$aa$@$\sim$ various equivalence relations}
It is naturally a metric space, 
for a metric that we denote by $d_{\operatorname{Haus}}$,
and the quotient map 
$(X,d_X) \longrightarrow (X_{\operatorname{Haus}}, d_{\operatorname{Haus}})$
is an isomorphism in the Lipschitz pseudo-metric category.

\vskip.2cm

(5)
Let $(X,d)$ be a pseudo-metric space.
Define a metric $d_1$ on $X$ by 
\begin{equation*}
d_1(x,x') \, = \, \max \{1, d(x,x')\} \hskip.5cm \text{for all} \hskip.2cm
x,x' \in X \hskip.2cm \text{with} \hskip.2cm x \ne x' .
\end{equation*}
The identity map defines a quasi-isometry
$(X,d) \longrightarrow (X,d_1)$.
\par

In particular,  every metric space is quasi-isometric to a discrete metric space. 
This shows that quasi-isometries 
do not respect at all the local structure of pseudo-metric spaces.
See also Remark \ref{metriclatticesareqi}.
\par

Define another pseudo-metric $d_{\ln}$ on $X$ by
\begin{equation*}
d_{\ln}(x,x') \, = \, \ln ( 1 + d(x,x'))  \hskip.5cm \text{for all} \hskip.2cm x,x' \in X .
\end{equation*}
Viewed as a map $f : (X,d) \longrightarrow (X,d_{\ln})$,
the identity is now a metric coarse equivalence;
the functions $\Phi_+, \Phi_-$ defined on $\R_+$
by $\Phi_+(r) =  \Phi_-(r) = \ln(1 + r)$
are respectively an upper control and a lower control for $f$
(alternatively: $\Phi_+(r) = r$).
Note that $f$ \emph{is not} a quasi-isometry,
unless $X$ has finite diameter.
The example of (6) below is of the same flavour.
\par

Similarly, for every metric space $(X,d)$, 
the identity map $(X,d) \longrightarrow (X, \sqrt{d})$
is a metric coarse equivalence.
It is a quasi-isometry if and only if
the diameter of $(X,d)$ is finite.

\vskip.2cm

(6)
Let $\mathbf H^2$ denote the upper half-plane model
$\{ z \in \C \mid \operatorname{Im}(z) > 0 \}$
for the hyperbolic plane.
Consider the map $f : \R \longrightarrow \mathbf H^2, \hskip.1cm x \longmapsto x+i$,
which is a parametrisation of the horocycle $h$ based at $\infty$ containing $i$.
Then \cite[II.8, Page 80]{Iver--92}:
\begin{equation*}
d_{\mathbf H^2}(f(x),f(x+\ell)) \, = \,
\arg\cosh (1 + \ell^2 /2)
\, \underset{\ell \to\infty}{\sim} \, 2 \ln \ell .
\end{equation*}
The map $f$ can be seen as a metric coarse equivalence
from $\R$ with the usual metric
to the horocycle $h$ with the metric induced from $\mathbf H^2$,
and $f$ is not a quasi-isometry.
\par

This carries over to the hyperbolic space of any dimension $n \ge 2$,
and provides maps $f : \R^{n-1} \longrightarrow \mathbf H^n$
whose images are horospheres. 
\index{Hyperbolic space $\mathbf H^n$}
\par

In Example \ref{growthR2H2}, we will state (and not completely prove)
that Euclidean spaces $\R^m$, $m \ge 1$, 
and hyperbolic spaces $\mathbf H^n$, $n \ge 2$,
are pairwise non-quasi-isometric.

\vskip.2cm

(7)
A pseudo-metric space $X$ is \textbf{hyperdiscrete} if
\index{Hyperdiscrete pseudo-metric space|textbf}
\begin{equation*}
\{(x,x') \in X^2 \mid x \ne x' \hskip.2cm \text{and} \hskip.2cm d_X(x,x') \le c \}
\end{equation*}
is a finite set for all $c > 0$.
An infinite example is given by the set of squares in $\N$,
with the usual metric defined by $d(m^2,n^2) = \vert m^2 - n^2 \vert$.
Then every map 
from a hyperdiscrete pseudo-metric space
to any pseudo-metric space is coarsely Lipschitz.

\vskip.2cm

(8)
Let $G$ be a connected real Lie group.
The metrics associated to left-invariant Riemannian metrics on $G$
are all bilipschitz equivalent. 
Compare with Example \ref{abundanceofd}.
\par

It is known that $G$ is quasi-isometric to
a connected closed subgroup of real upper triangular matrices
\cite[Lemma 6.7]{Corn--08}.
\index{Triangular group}

\vskip.2cm

(9)
Let $n \ge 1$. Recall that an invertible matrix $g \in \GL_n(\R)$ is
\textbf{distal} \index{Distal! matrix $g \in \GL_n(\R)$|textbf}
if all its eigenvalues $\lambda \in \C$ are of modulus one,  $\vert \lambda \vert = 1$,
and \textbf{semisimple} \index{Semisimple matrix $g \in \GL_n(\R)$|textbf}
if it is conjugate in $\GL_n(\C)$ to a diagonal matrix.
\par
Consider the Lie group $\GL_n(\R)$ 
with some left-invariant Riemannian metric,
a matrix $g \in \GL_n(\R)$,  and the map
$f : \Z \longrightarrow \GL_n(\R)$ given by $k \longmapsto g^k$.
\index{General linear group $\GL$! $\GL_n(\R)$, $\SL_n(\R)$}
Then $f$ is 
\par (9$_1$)
a Lipschitz map in all cases,
\par (9$_2$)
a bilipschitz map if and only if 
$g$ is non-distal,
\par (9$_3$)
coarsely expansive if and only if
$g$ is either non-distal or non-semisimple,
\par (9$_4$)
bounded if and only if $g$ is distal and semisimple.
\par\noindent
Let $\K$ be a local field 
(see Example \ref{panoramalocalfield}), 
$n \ge 1$ an integer, and $g \in \GL_n(\K)$~;
for ``$g$ distal'', see now Definition \ref{defdistal}. 
\index{General linear group $\GL$! $\GL_n(\K)$ 
for $\K$ a non-discrete locally compact field}
Let $f : \Z \longrightarrow \GL_n(\K)$ be similarly defined
by $f(k) = g^k$. The following four properties are equivalent:
\par
(9$_{\text{local}}$)
$f$ is bilipschitz, $g$ is non-distal, $f$ is coarsely expansive, $f$ is unbounded.
\par

We refer to \cite[Chapter 3, on distortion]{Grom--93}
for other examples of inclusions of sub-(semi-)groups in groups,
looked at from the point of view of inclusions of metric spaces inside each other.

\vskip.2cm

(10)
Let $F_2$ denote the free group on a set $S$ of two elements, 
with its natural word metric $d_S$
(Definition \ref{wordmetric}). 
\index{Free group}
As shown below in Corollary \ref{exgrowthXtoY},
there does not exist any coarse embedding of $F_2$ (of exponential growth)
into the Euclidean space $\R^n$ (of polynomial growth), for every $n \ge 1$.
\index{Euclidean space $\R^n$}
\index{Coarse! embedding}
\par

There does not exist a large-scale bilipschitz map of $F_2$,
indeed of a free \emph{semi-}group of rank $2$,
in a separable infinite-dimensional Hilbert space $\mathcal H$ \cite{Bour--86}.
But there exist coarse embeddings of $(F_2, d_S)$ into $\mathcal H$;
\index{Hilbert space}
more precisely, there exists an isometric embedding of $(F_2, \sqrt{d_S})$ into $\mathcal H$,
see for example Item C.2.2(iii) in \cite{BeHV--08}.

\vskip.2cm

(11)
There is in \cite{Tess} a characterization of functions 
that can be lower controls of Lipschitz embeddings 
of the non-abelian free group $F_2$ into Hilbert spaces.

\vskip.2cm
\noindent \textbf{Three basic examples from a later chapter}
\vskip.2cm

(12)
Given a $\sigma$-compact LC-group $G$ and a cocompact closed subgroup $H$
(for example a cocompact lattice),
the inclusion $j : H \lhook\joinrel\relbar\joinrel\rightarrow G$ 
can be seen as a metric coarse equivalence.
If, moreover, $G$ is compactly generated, 
$j$ can be seen as a quasi-isometry (Proposition \ref{sigmac+compactgofcocompact}). 

\vskip.2cm

(13)
Given a $\sigma$-compact LC-group $G$
and a compact normal subgroup $K$,
the canonical projection $\pi : G \twoheadrightarrow G/K$
can be seen as a metric coarse equivalence.
If, moreover, $G$ is compactly generated,
$\pi$ can be seen as a quasi-isometry (Proposition \ref{sigmac+compactgofquotients}).

\vskip.2cm

(14)
Let $G$ be a $\sigma$-compact LC-group and $H$ a closed subgroup.
Let $d_G$ be an adapted metric on $G$ and $d_H$ an adapted metric on $H$.
The inclusion $j : (H,d_H) \lhook\joinrel\relbar\joinrel\rightarrow (G,d_G)$ 
is a coarse embedding (Corollary \ref{2metricsce}).
\par

In general, $j$ is not a quasi-isometric embedding,
even if $G,H$ are both compactly generated and $d_G,d_H$
are word metrics defined by compact generating sets
(Remark \ref{Heisenberg}).
\end{exe}

\begin{lem}[``inverses'' of controls]
\label{PhiEtPsi}
(1)
Let $\Phi_+ : \R_+ \longrightarrow \R_+$ be an upper control. 
The function $\Psi_- : \R_+ \longrightarrow \overline \R_+$ defined by
\begin{equation*}
\Psi_-(s) \, = \, \inf \{ r \in \R_+ \mid \Phi_+(r) \ge s \}
\hskip.5cm \text{for} \hskip.2cm s \in \R_+
\end{equation*}
is a lower control such that $\Psi_-(\Phi_+(t)) \le t$ for all $t \in \R_+$.
\par

(2)
Let $\Phi_- : \R_+ \longrightarrow \overline \R_+$ be a lower control. 
The function $\Psi_+ : \R_+ \longrightarrow  \R_+$ defined by
\begin{equation*}
\Psi_+(s) \, = \, \sup \{ r \in \R_+ \mid \Phi_-(r) \le s \}
\hskip.5cm \text{for} \hskip.2cm s \in \R_+
\end{equation*}
is an upper control such that $t \le \Psi_+(\Phi_-(t))$ for all $t \in \R_+$.
\end{lem}

The proof is an elementary exercise.
We use the natural conventions:
$\inf \emptyset = \infty$
and $\sup \emptyset = 0$.
\par
Consider for example the usual metric on $\R_+$
and the inclusion $f$ of the interval $\mathopen[0,1\mathclose]$ in $\R_+$.
Define $\Phi_+ : \R_+ \longrightarrow \R_+$ 
and $\Phi_- : \R_+ \longrightarrow \overline \R_+$
by $\Phi_+(r) = \Phi_-(r) = r$ if $r \le 1$, 
and $\Phi_+(r) = 1$, $\Phi_-(r) = \infty$ if $r > 1$.
Then $\Phi_+$ is an upper control for $f$ and $\Phi_-$ a lower control for $f$;
moreover, with the definition of Lemma \ref{PhiEtPsi},
we have $\Psi_+ = \Phi_+$ and $\Psi_- = \Phi_-$.

\begin{prop}[on coarse morphisms]
\label{epimonoiso}
Let $X,Y$ be two pseudo-metric spaces,
$f : X \longrightarrow Y$ a coarsely Lipschitz map,
and $\underline f$ the corresponding morphism
(in the sense of Definition \ref{defcoarsecat}).
Then, in the metric coarse category:
\begin{enumerate}[noitemsep,label=(\arabic*)]
\item\label{1DEepimonoiso}
If $X$ is non-empty, 
$\underline f$ is an epimorphism if and only if $f$ is essentially surjective;
\item\label{2DEepimonoiso}
$\underline f$ is a monomorphism if and only if $f$ is coarsely expansive;
\item\label{3DEepimonoiso}
$\underline f$ is an isomorphism if and only if
$f$ is a metric coarse equivalence;
if moreover $X$ is non-empty, this holds if and only if
$\underline f$ is both an epimorphism and a monomorphism.
\end{enumerate}
\end{prop}      

\noindent \emph{Note.}
If $X$ is empty and $Y$ non-empty bounded, 
the morphism defined by the map $X \longrightarrow Y$
is an epimorphism that is not essentially surjective.

\begin{proof}
\ref{1DEepimonoiso}
Suppose first that $f$ is essentially surjective;
set $c = \sup_{y \in Y} d_Y(y,f(X))$.
Consider a pseudo-metric space $Z$ 
and two coarsely Lipschitz maps $h_1,h_2$ from $Y$ to $Z$.
To show that $\underline f$ is an \textbf{epimorphism},
we assume that $h_1 f \sim h_2 f$
and we have to show that $h_1 \sim h_2$.
\index{Epimorphism}
\par
Since $h_1 f \sim h_2 f$,
there exists $c' > 0$ such that
$d_Z(h_1(f(x)), h_2(f(x)) ) \le c'$ for all $x \in X$.
Since $h_i$ is a coarsely Lipschitz map,
there exists $c_i > 0$ such that
$d_Z(h_i(y), h_i(y')) \le c_i$ for all $y,y' \in Y$ with $d_Y(y,y') \le c$
($i=1,2$).
Let $y \in Y$. Choose $x \in X$ such that $d_Y(y,f(x)) \le c$.
Then
\begin{equation*}
\aligned
d_Z(h_1(y),h_2(y)) \, &\le \,
d_Z(h_1(y), h_1f(x)) + d_Z(h_1f(x), h_2f(x)) + d_Z(h_2f(x),h_2(y))
\\
\, &\le \, c_1 + c' + c_2  \hskip.1cm  .
\endaligned
\end{equation*}
Hence $h_1 \sim h_2$.
\par

For the converse implication, suppose that $f$ is not essentially surjective.
Define two functions $h_1, h_2 : Y \longrightarrow \R_+$
by $h_1(y) = 0$ and $h_2(y) = d_Y(y,f(X))$ for all $y \in Y$;
note that $h_2$ takes finite values because $X \ne \emptyset$.
Observe that $h_1$ and $h_2$ are coarsely Lipschitz functions with $h_1f = h_2f = 0$,
so that in particular $h_1 f \sim h_2 f$.
But $h_1(Y) = \{0\}$ is bounded in $\R_+$ and $h_2(Y)$ is not,
so that $h_1 \nsim h_2$.
Hence $\underline f$ is not an epimorphism.

\vskip.2cm

\ref{2DEepimonoiso}
Suppose first that $f$ is coarsely expansive;
let $\Phi_-$ be a lower control for $f$,
and define an upper control $\Psi_+$ as in Lemma \ref{PhiEtPsi}.
Consider a pseudo-metric space $W$ and two coarsely Lipschitz maps $g_1, g_2$
from $W$ to $X$.
To show that $\underline f$ is a \textbf{monomorphism},
we assume that $f g_1 \sim f g_2$
and we have to show that $g_1 \sim g_2$.
\index{Monomorphism}
\par

Since $f g_1 \sim f g_2$,
there exists a constant $c'$ such that $d_Y(f g_1 (w), f g_2 (w)) \le c'$
for all $w \in W$.
Since $\Phi_-$ is a lower control for $f$, we have
$\Phi_-(d_X(g_1(w), g_2(w))) \le d_Y(f g_1 (w), f g_2 (w))$
for all $w \in W$.
Hence 
\begin{equation*}
\aligned
d_X(g_1(w), g_2(w)) \, &\le \, 
\Psi_+( \Phi_-(d_X(g_1(w), g_2(w))) ) 
\\
\, &\le \,  \Psi_+(d_Y(fg_1(w), fg_2(w)))  \,  \le \,  
\Psi_+(c')
\hskip.5cm \text{for all} \hskip.2cm w \in W .
\endaligned
\end{equation*}
It follows that $g_1 \sim g_2$.
\par

For the converse implication, suppose that $f$ is not coarsely expansive.
There exist a constant $c > 0$
and two sequences $(x_n)_{n \ge 0}, (x'_n)_{n \ge 0}$
of points in $X$ such that
$\lim_{n \to \infty}d_X(x_n, x'_n) = \infty$ and
$d_Y(f(x_n), f(x'_n)) \le c$ for every $n \ge 0$.
Let $W$ denote the space of squares in $\N$
as in Example \ref{excoarse}(7);
define two maps $g_1, g_2 : W \longrightarrow X$
by $g_1(n^2) = x_n$ and $g_2(n^2) = x'_n$ for all $n^2 \in W$.
Then $g_1, g_2$ are coarsely Lipschitz maps, they are not close,
and $f g_1,  f g_2$ are close.
Hence $\underline f$ is not a monomorphism.

\vskip.2cm

\ref{3DEepimonoiso} 
When $X \ne \emptyset$, 
the equivalence between ``metric coarse equivalence''
and ``epimorphism \& monomorphism''
follows from the definitions,
and \ref{1DEepimonoiso} \& \ref{2DEepimonoiso}.
\par
Suppose now that $f$ is a metric coarse equivalence;
let us show that $\underline f$ is an isomorphism.
If $X = \emptyset = Y$, there is nothing to show.
Otherwise, by Example \ref{excoarse}(1),
both $X$ and $Y$ are non-empty.
Let $\Phi_-$ and $\Phi_+$ be a lower control and an upper control for $f$,
and let $c > 0$ be such that $d_Y(y, f(X)) \le c$ for all $y \in Y$.
Let $\Psi_+$ be the upper control defined as in Lemma \ref{PhiEtPsi},
and let $\Psi_-$ be the lower control defined by
$\Psi_- (s) = \inf \{ r \in \R_+ \mid \Phi_+(r) + 2c \ge s \}$
for all $s \in \R_+$. 
For each $y \in Y$, choose $x_y \in X$ such that
$d_Y(y, f(x_y)) \le c$; set $g(y) = x_y$;
thus $g$ is a map $Y \longrightarrow X$ such that
$\sup_{y \in Y} d_Y(y, fg(y)) \le c$.
We claim that $g$ is coarsely Lipschitz and coarsely expansive.
\par
Let $y,y' \in Y$. On the one hand, we have 
\begin{equation*}
\Phi_-(d_X(g(y), g(y'))) \, \le \, d_Y(fg(y), fg(y')) ,
\end{equation*}
hence
\begin{equation*}
\aligned
d_X(g(y), g(y')) \, &\le \,
\Psi_+ \Phi_-(d_X(g(y), g(y')))
\\
\, &\le \, 
\Psi_+(d_Y(fg(y), fg(y'))) \, \le \,
\Psi_+(d_Y(y,y') + 2c) ,
\endaligned
\end{equation*}
so that $s \mapsto \Psi_+(s + 2c)$ is an upper control for $g$,
and $g$ is a coarsely Lipschitz map.
On the other hand, denoting by $\widetilde \Phi_+$ the upper control $\Phi_+ + 2c$, 
we have
\begin{equation*}
\aligned
d_Y(y,y') \, &\le \,
d_Y(y, fg(y)) + d_Y(fg(y), fg(y')) + d_Y(fg(y'),y') 
\\
\, &\le \, 
\Phi_+(d_X(g(y), g(y'))) + 2c \, = \,
\widetilde \Phi_+ (d_X(g(y), g(y')))
\endaligned
\end{equation*}
hence
\begin{equation*}
\Psi_-(d_Y(y,y')) \, \le \,
\Psi_- \widetilde \Phi_+ (d_X(g(y), g(y'))) \, \le \,
d_X(g(y), g(y')) ,
\end{equation*}
so that $g$ is coarsely expansive.
\par

We leave it to the reader to check that $g$ is essentially surjective,
and therefore that $g$ is a metric coarse equivalence.
From the definition of $g$, it is clear that 
$fg \sim \operatorname{id}_Y$.
It follows that $fgf \sim f$;
since $\underline f$ is a monomorphism (by \ref{2DEepimonoiso}),
this implies that $gf \sim \operatorname{id}_X$.
We have shown that $\underline f$ is an isomorphism, 
with inverse isomorphism $\underline g$.

For the converse implication, we have to check that
a coarse isomorphism is a metric coarse equivalence.
This follows from \ref{1DEepimonoiso} and \ref{2DEepimonoiso} when $X \ne \emptyset$.
If $X = \emptyset$, then $Y = \emptyset$ too.
\end{proof}

\begin{defn}
\label{defcoarsebis}
Let $X, Y$ be two pseudo-metric spaces and $f : X \longrightarrow Y$ a map.
In addition to Definition \ref{defcoarse}, define $f$ to be
\begin{enumerate}[noitemsep,label=(\alph*)]
\item\label{aDEdefcoarsebis}
\textbf{coarsely right-invertible} 
if there exists a coarsely Lipschitz map $g : Y \longrightarrow X$ 
such that $f \circ g \sim \operatorname{id}_Y$;
\item\label{bDEdefcoarsebis}
\textbf{coarsely left-invertible} 
if there exists a coarsely Lipschitz map $h : Y \longrightarrow X$ 
such that $h \circ f \sim \operatorname{id}_X$;
\item\label{cDEdefcoarsebis}
\textbf{coarsely invertible} if it is both coarsely left-invertible and coarsely right-invertible.
\end{enumerate}
\end{defn}
\index{Coarsely! left-, right-invertible|textbf}
\index{Left, right-invertible coarsely Lipschitz map|textbf}

\begin{rem}
\label{remcoarsecat}
Let $X, Y$ be two pseudo-metric spaces, 
$f : X \longrightarrow Y$ a coarsely Lipschitz map,
and $\underline f$ its closeness class. 
Then, in the metric coarse category:
\begin{enumerate}[noitemsep,label=(\alph*)]
\item\label{aDEremcoarsecat}
if $f$ is coarsely right-invertible, then $\underline f$ is an epimorphism;
\item\label{bDEremcoarsecat}
if $f$ is coarsely left-invertible, then $\underline f$ is a monomorphism.
\end{enumerate}
Indeed, in any category, right-invertible morphisms are epimorphisms
and left-invertible morphisms are monomorphisms.
The converse implications of \ref{aDEremcoarsecat} and \ref{bDEremcoarsecat} above
\emph{do not} hold: see Example \ref{du26avril2015} below.
\par
Suppose $f$ is coarsely invertible; let $g$ and $h$ be as in Definition \ref{defcoarsebis}.
Then $h \sim g$, and $h$ can indeed be replaced by $g$; in particular:
\begin{enumerate}[noitemsep,label=(\alph*)]
\addtocounter{enumi}{2}
\item\label{cDEremcoarsecat}
$f$ is coarsely invertible if and only if $\underline f$ is an isomorphism,
if and only if $f$ is a metric coarse equivalence.
\end{enumerate}
\end{rem}

\begin{defn}
\label{coarseretractdef}
Let $Y$ be a  pseudo-metric space.
\par

A subspace $Z$ of $Y$ is a \textbf{coarse retract} of $Y$ 
\index{Coarse! retract|textbf}
\index{Retract! coarse retract|textbf}
if the inclusion $\iota_{Y \supset Z}$ of $Z$ in $Y$ is left-invertible.
\par

A \textbf{coarse retraction from $Y$ to $Z$}
is a coarsely Lipschitz map 
$r : Y \longrightarrow Z$ such that $r \circ \iota_{Y \supset Z} \sim \operatorname{id}_Z$.
\end{defn}

\begin{prop}
\label{prop1decoarseretract}
Let $X, Y$ be two pseudo-metric spaces and
$f : X \longrightarrow Y$ a coarsely Lipschitz map.
Denote by $f_{\text{im}} : X \longrightarrow f(X)$ the map induced by $f$.
\begin{enumerate}[noitemsep,label=(\arabic*)]
\item\label{1DEprop1decoarseretract}
$f$ is coarsely expansive 
if and only if the induced map $f_{\text{im}}$ is a metric coarse equivalence.
\item\label{2DEprop1decoarseretract}
$f$ is coarsely left-invertible 
if and only if it is coarsely expansive and $f(X)$ is a coarse retract.
\end{enumerate}
\end{prop}

\begin{proof}
Claim \ref{1DEprop1decoarseretract} is immediate from the definitions. 
\par

For \ref{2DEprop1decoarseretract}, suppose first that $f$ is coarsely left-invertible,
and left $h : Y \longrightarrow X$ be a coarsely Lipschitz map such that 
$h \circ f = h \circ \iota_{Y \supset f(X)} \circ f_{\text{im}} \sim \operatorname{id}_X$.
Then $f$ is coarsely expansive 
(Remark \ref{remcoarsecat} and Proposition \ref{epimonoiso}\ref{2DEepimonoiso}),
so that $f_{\text{im}}$ is a metric coarse equivalence by \ref{1DEprop1decoarseretract}.
Moreover
\begin{equation*}
 \operatorname{id}_{f(X)} \circ f_{\text{im}}  \, = \, 
 f_{\text{im}} \circ \operatorname{id}_X \, \sim \, 
f_{\text{im}} \circ h \circ \iota_{Y \supset f(X)} \circ f_{\text{im}} 
 \end{equation*}
and it follows that $\operatorname{id}_{f(X)} \sim f_{\text{im}} \circ h \circ \iota_{Y \supset f(X)}$
 (by composition on the right with an inverse of $f_{\text{im}}$);
hence $f_{\text{im}} \circ h$ is a coarse retraction from $Y$ to $f(X)$.
\par

Conversely, if $f$ is coarsely expansive and there exists a coarse retraction $r$
from $Y$ to $f(X)$, then there exists a coarsely Lipschitz map $j : f(X) \longrightarrow X$
such that the classes of $j$ and $f_{\text{im}}$ are inverse to each other,
and $j \circ r : Y \longrightarrow X$ is a coarse left-inverse of $f$.
\end{proof}

\begin{defn}
\label{defcoarselyretractable}
For $X,Y$ two pseudo-metric spaces,
$Y$ is \textbf{coarsely retractable on $X$}
if there exists a coarsely right-invertible coarsely Lipschitz map from $Y$ to $X$,
or equivalently if there exists a coarsely left-invertible coarsely Lipschitz map from $X$ to $Y$.
\end{defn}

Here is the analogue of Proposition \ref{epimonoiso}
for large-scale morphisms.
The proof of the proposition, which is similar, 
as well as the analogues of Items \ref{defcoarsebis} to \ref{defcoarselyretractable},
are left to the reader.

\begin{prop}[on large-scale morphisms]
\label{epimonoisobis}
Let $X,Y$ be two pseudo-metric spaces,
$f : X \longrightarrow Y$ a large-scale Lipschitz map,
and $\underline f$ the corresponding morphism
(in the sense of Definition \ref{deflargescalecat}). 
Then, in the large-scale category:
\begin{enumerate}[noitemsep,label=(\arabic*)]
\item\label{1DEepimonoisobis}
If $X$ is non-empty, 
$\underline f$ is an epimorphism if and only if $f$ is essentially surjective;
\item\label{2DEepimonoisobis}
$\underline f$ is a monomorphism if and only if $f$ is large-scale expansive;
\item\label{3DEepimonoisobis}
$\underline f$ is an isomorphism if and only if
$f$ is a quasi-isometry;
if moreover $X$ is non-empty, this holds if and only if
$\underline f$ is both an epimorphism and a monomorphism.
\end{enumerate}
\end{prop}

\section
{Coarse properties and large-scale properties}
\label{coarselyconnectedetc}

In this section, we begin to describe properties
of pseudo-metric spaces
which will later make good sense for LC-groups.

\subsection{Coarsely connected, coarsely geodesic, 
and large-scale geodesic pseudo-metric spaces}
\label{subsection_3Ca}

There are several notions of connectedness, or rather path-connectedness, suited
to the categories introduced in the previous section:

\begin{defn}
\label{defcoarselyconn}
Let $(X,d)$ be a pseudo-metric space and $c > 0$ a constant.
\par

For $x,x' \in X$ and $n \ge 0$, a \textbf{$c$-path of $n$ steps from $x$ to $x'$ in $X$}
\index{c-path of $n$ steps in a pseudo-metric space|textbf}
is a sequence $x=x_0, x_1, \hdots, x_n=x'$ of points in $X$
such that $d(x_{i-1},x_i) \le c$ for $i = 1, \hdots, n$.
\begin{enumerate}[noitemsep,label=(\alph*)]
\item\label{aDEdefcoarselyconn}
The space $X$ is \textbf{$c$-coarsely connected} 
if, for every pair $(x,x')$ of points in $X$,
there exists a $c$-path from $x$ to $x'$ in $X$.
The space $X$ is  \textbf{coarsely connected}
\index{Coarsely! connected pseudo-metric space|textbf}
\index{Connected|see {Coarsely connected}}
if it is $c$-coarsely connected for some $c > 0$.

\item\label{bDEdefcoarselyconn}
The space $X$ is \textbf{$c$-coarsely geodesic}
if there exist an upper control  $\Phi$ such that,
for every pair $(x,x')$ of points in $X$,
there exists a $c$-path of at most $\Phi(d(x,x'))$ steps from $x$ to $x'$ in $X$.
The space $X$ is  \textbf{coarsely geodesic}
\index{Coarsely! geodesic pseudo-metric space|textbf}
if it is $c$-coarsely geodesic for some $c > 0$.

\item\label{cDEdefcoarselyconn}
The space $X$ is \textbf{$c$-large-scale geodesic}
\index{Large-scale! geodesic pseudo-metric space|textbf}
if it has the property of (b) above with $\Phi$ affine,
in other words if there exist constants $a > 0, b \ge 0, c > 0$  such that,
for every pair $(x,x')$ of points in $X$,
there exists a $c$-path of a most $a d(x,x') + b$ steps from $x$ to $x'$ in $X$.
The space $X$ is \textbf{large-scale geodesic}
if it is $c$-large-scale geodesic for some $c > 0$.

\item\label{dDEdefcoarselyconn}
The space $X$ is \textbf{$c$-geodesic} if,
\index{Geodesic pseudo-metric space|textbf}
for every pair $(x,x')$ of points in $X$,
there exists a $c$-path $x_0 = x, x_1, \hdots, x_n = x'$ 
from $x$ to $x'$ in $X$ such that $\sum_{i=1}^n d(x_{i-1},x_i) = d(x,x')$.

\item\label{eDEdefcoarselyconn}
The space $X$ is  \textbf{geodesic}
if, for every pair of points $x,x' \in X$ with $d(x,x') > 0$,
there exists an isometry $f$ from the interval $[0,d(x,x')]$ to $X$
such that $f(0) = x$ and $f(d(x,x')) = x'$.
[Alternatively, we could delete the condition ``$d(x,x') > 0$''
and introduce a non-Hausdorff ``real interval'' $[0^-,0^+]$ 
of zero length, with two distinct points.]
\end{enumerate}
Instead of writing that the \emph{space} $X$ is coarsely connected 
(or coarsely geodesic, etc.), we sometimes write that the \emph{metric} $d$
is coarsely connected (or coarsely geodesic, etc.).
\end{defn}

\begin{rem}
\label{remoncCCaso}
``Coarsely connected'' here means the same as ``long-range connected''
in \cite[0.2.A$_2$]{Grom--93}.
\par

Obviously, in Definition \ref{defcoarselyconn}, each of the notions defined in 
\ref{bDEdefcoarselyconn}, \ref{cDEdefcoarselyconn}, 
\ref{dDEdefcoarselyconn}, \ref{eDEdefcoarselyconn}, 
implies the previous one. 
\par

There are characterizations of Properties 
\ref{aDEdefcoarselyconn}, \ref{bDEdefcoarselyconn},
and \ref{cDEdefcoarselyconn},
of Definition \ref{defcoarselyconn} in Proposition \ref{invariance_cg_lsg}.
\par

Let $(X,d)$ and $c > 0$ be as in the previous definition, and $C \ge c$.
If $X$ is $c$-coarsely connected, it is obvious that $X$ is $C$-coarsely connected.
Analogous implications hold for $c$-coarsely geodesic, $c$-large-scale geodesic,
and $c$-geodesic spaces.
It follows that, in each of (a), (b), and (c) above, the final 
``for some $c > 0$'' can be replaced by ``for any $c$ large enough''.
\end{rem}

\begin{exe}
\label{excoarselyconnNotcoarselygeod}
A coarsely connected pseudo-metric space need not be
coarsely geodesic, as the following example shows.
\par

For integers $m,n \ge 0$, denote by $C_{m,n}$ 
the metric realization of the graph with
vertex set $(x^{(m,n)}_j)_{j \in \Z / (m+n)\Z}$ 
and edge set $( \{x^{(m,n)}_{j-1}, x^{(m,n)}_{j} \} )_{j \in \Z / (m+n)\Z}$,
with the combinatorial metric (Example \ref{metricrealizationgraph}).
Let $I_{m,n}$ be the subspace of $C_{m,n}$ corresponding to the subgraph
spanned by the vertices $x^{(m,n)}_j$ with $0 \le j \le n$, with the induced metric.
Define a graph $B$ whose set of vertices is the disjoint union 
of a singleton $v_0$ and the disjoint union of all $C_{m,n}$ for all $m,n \ge 3$ with $m \le n/2$. 
Its set of edges is given by the edges of all $C_{m,n}$, as well as the edges $\{v_0,x_0^{(m,n)}\}$. 
We define $X$ as the metric subspace of $B$ given as 
the union of $\{v_0\}$ and the $I_{m,n}$ for all $m,n \ge 3$ with $m \le n/2$.
\par

On the one hand,
being a connected subgraph, $X$ is 1-connected, 
and hence $c$-connected for all $c\ge 1$. 
On the other hand, for any $c>0$, $X$ is not $c$-coarsely geodesic. 
Indeed, fix such a $c$ and an integer $k\ge 2$, and choose $m>c$. 
Then the distance in $X$ of $x_0^{(m,km)}$ and $x_{km}^{(m,km)}$ is $m$; 
the minimal number of steps of length $\le c$ to go 
from $x_0^{(m,km)}$ and $x_{km}^{(m,km)}$ is $km/c>k$. 
Letting $k$ tend to infinity, we deduce that $X$ is not $c$-coarsely geodesic.
\end{exe}

\begin{prop}
\label{bbb}
Let $X, Y$ be pseudo-metric spaces.

\begin{enumerate}[noitemsep,label=(\arabic*)]
\item\label{1DEbbb}
Suppose that there exists an essentially surjective coarsely Lipschitz map 
$h$ from $Y$ to $X$.
If $Y$ is coarsely connected, then so is $X$.
\item\label{2DEbbb}
Suppose that $Y$ is coarsely retractable on $X$.
If $Y$ is coarsely geodesic then so is $X$.
\end{enumerate}
\end{prop}

\begin{proof}
\ref{1DEbbb}
Let $\Phi_h$ be an upper control for $h$, 
and $K$ a constant such that $d_X(x, h(Y)) \le K$ for all $x \in X$;
let $c$ be a constant such that any pair of points in $Y$ can be connected by a $c$-path.
Let $x, x' \in X$. We can find $n \ge 0$ and $y, y_0, y_1, \hdots, y_n, y'$ in $Y$ such that
$d_X(x, h(y_0)) \le K$, $d_Y(y_{i-1},y_i) \le c$ for all $i$ with $1 \le i \le n$,
and $d_X(h(y_n), x') \le K$. 
If $C = \max\{K, \Phi_h(c)\}$,
then $x, h(y_0), h(y_1), \hdots, h(y_n), x'$ is a $C$-path.
It follows that $X$ is $C$-coarsely connected.
\par

\ref{2DEbbb}
By hypothesis, there exist two coarsely Lipschitz maps
$h : Y \longrightarrow X$ and $f : X \longrightarrow Y$ 
such that $h \circ f \sim \operatorname{id}_X$.
Let $\Phi_h, \Phi_f$ be upper controls for $h$, $f$ respectively, 
$K$, $c$ two constants, 
and $\Psi$ a function, 
such that $d_X(x, h(f(x))) \le K$ for all $x \in X$,
and such that any pair of points $y,y' \in Y$
can be joined by a $c$-path of at most $\Psi(d_Y(y,y'))$ steps.
Let $x,x' \in X$. 
Set now $y = f(x)$ and $y' = f(x')$; observe that $d_Y(y,y') \le \Phi_f(d_X(x,x'))$.
We can find $n \le \Psi(\Phi_f(d_X(x,x')))$
and $y_0 = y, y_1, \hdots, y_n = y'$ in $Y$
such that $d_Y(y_{i-1}, y_i) \le c$ for all $i$ with $1 \le i \le n$.
If $C = \max \{ K, \Phi_h(c) \}$ is as in the proof of \ref{1DEbbb},
then $x, h(y_0), h(y_1), \hdots, h(y_n), x'$ is a $C$-path
of at most $\Psi(\Phi_f(d_X(x,x'))) + 2$ steps.
It follows that $X$ is coarsely geodesic.
\end{proof}

\begin{defn}
\label{the new space Xc etc}
Let $c > 0$ be a constant, 
and $(X,d_X)$ a $c$-coarsely connected pseudo-metric space.
We define now a \textbf{graph $X_c$ associated to $X$}
and \emph{two} metrics $d_c, d'_c$ on it.
\index{Graph associated to a space}
\par

Let $(X_{\operatorname{Haus}}, d_{\operatorname{Haus}})$ 
be the largest Hausdorff quotient of $X$ (Example \ref{excoarse}(4)).
 \index{Hausdorff! largest Hausdorff quotient}
Let $X_c$ denote the connected graph with vertex set $X_{\operatorname{Haus}}$,
in which edges connect  pairs 
$(x, y) \in X_{\operatorname{Haus}} \times X_{\operatorname{Haus}}$ 
with $0 < d_{\operatorname{Haus}}(x,y) \le c$.
Let $d_c$ denote the combinatorial metric on $X_c$,
with edges of length $c$ (Example \ref{metricrealizationgraph}).
\par

Because of (4) in Lemma \ref{the new space Xc aso} below, 
we define a second metric $d'_c$ on $X_c$ by
\begin{equation*}
d'_c(u,v) =  
\left\{
\aligned
&d_c(u,v) \hskip.2cm \text{if} \hskip.2cm
u,v \hskip.2cm \text{are in a common edge of} \hskip.2cm X_c,
\\
& \inf \{ d_c(u,x) + d_X(x,y) + d_c(y,v) 
\mid x \in \operatorname{extr}(u), y \in \operatorname{extr}(v) \} 
\hskip.2cm \text{otherwise}.
\endaligned
\right.
\end{equation*}
We have denoted by $\operatorname{extr}(u)$ the set of vertices in $X_c$
(= the set of points in $X_{\operatorname{Haus}}$) 
that are incident to an edge containing $u$.
Thus, if $u$ is in the interior of an edge connecting two distinct vertices $x$ and $x'$,
then $\operatorname{extr}(u) = \{x,x'\}$.
\end{defn}

\begin{lem}
\label{the new space Xc aso}
Let $c > 0$ be a constant, 
$(X,d_X)$ a $c$-coarsely connected pseudo-metric space,
and $(X_c, d_c)$ as in Definition \ref{the new space Xc etc}.
\par

The natural mapping $\varphi : (X,d_X) \longrightarrow (X_c,d_c), 
x \longmapsto [x]$ 
has the following properties:
\begin{enumerate}[noitemsep,label=(\arabic*)]
\item\label{1DEthe new space Xc aso}
The target space $(X_c,d_c)$ is geodesic (and therefore in particular connected).
\item\label{2DEthe new space Xc aso}
We have $d_c([x],[y]) \ge  d_X(x,y)$ for $x,y \in X$;
in particular $\varphi$ is large-scale expansive (a fortiori coarsely expansive).
Observe that, for $x,y \in X$ with $0 < d_c(x,y) \le c$,
we have  $d_c([x],[y]) = c \ge  d_X(x,y)$.
\item\label{3DEthe new space Xc aso}
The mapping $\varphi$ is essentially surjective:
$\sup_{w \in X_c} d_c(w, \varphi (X) ) \le c/2$.
\item\label{4DEthe new space Xc aso}
The mapping $\varphi$ need not be coarsely Lipschitz.
\item\label{5DEthe new space Xc aso}
If $(X,d_X)$ is coarsely geodesic, $\varphi$ is coarsely Lipschitz,
and thus a metric coarse equivalence.
\item\label{6DEthe new space Xc aso}
If $(X,d_X)$ is large-scale geodesic, $\varphi$ is large-scale Lipschitz,
and thus a quasi-isometry.
\end{enumerate}
The natural mapping $\psi : (X,d_X) \longrightarrow (X_c, d'_c), 
x \longmapsto [x]$ 
has the following properties:
\begin{enumerate}[noitemsep,label=(\arabic*)]
\addtocounter{enumi}{6}
\item\label{7DEthe new space Xc aso}
The target space $(X_c,d'_c)$ is connected.
\item\label{8DEthe new space Xc aso}
The mapping $\psi$ is an isometry and is essentially surjective;
in particular, it is a metric coarse equivalence.
\end{enumerate}
\end{lem}

\begin{proof}
Proofs are straightforward, and left to the reader.
Note that Properties \ref{5DEthe new space Xc aso} and \ref{6DEthe new space Xc aso} 
for $\varphi$ hold also for $\psi$,
but we will not use this below.
\end{proof}

\begin{prop}[notions invariant by coarse metric equivalence and by quasi-isometry]
\label{invariance_cg_lsg}
Let $(X,d_X)$ be a pseudo-metric space.
\index{Property! of a pseudo-metric space invariant 
by metric coarse equivalence or by quasi-isometries}
\begin{enumerate}[noitemsep,label=(\arabic*)]
\item\label{1DEinvariance_cg_lsg}
Coarse connectedness is a property invariant by metric coarse equivalence.
\item\label{2DEinvariance_cg_lsg}
$X$ is coarsely connected 
if and only if $X$ is metric coarse equivalent
to some connected metric space.
\item\label{3DEinvariance_cg_lsg}
Coarse geodesicity is a property invariant by metric coarse equivalence.
\item\label{4DEinvariance_cg_lsg}
$X$ is coarsely geodesic 
if and only if $X$ is metric coarse equivalent
to some geodesic metric space.
\item\label{5DEinvariance_cg_lsg}
Large-scale geodesicity is a property invariant by quasi-isometry.
\item\label{6DEinvariance_cg_lsg}
$X$ is large-scale geodesic 
if and only if  $X$ is quasi-isometric to some geodesic metric space.
\end{enumerate}
\end{prop}

\noindent \emph{Note.}
There are also characterizations of the properties in 
\ref{1DEinvariance_cg_lsg}, \ref{3DEinvariance_cg_lsg}, and \ref{5DEinvariance_cg_lsg},  
in terms of $X$ and its Rips $2$-complexes: see Proposition \ref{inclusionXinRips}.
\par

Below, further properties are shown to be invariant by metric coarse equivalence:
\begin{enumerate}[noitemsep,label=(\arabic*)]
\addtocounter{enumi}{6}
\item\label{7DEinvariance_cg_lsg}
being coarsely ultrametric (Proposition \ref{propcoarselyultrametric}),
\item\label{8DEinvariance_cg_lsg}
asymptotic dimension (Proposition \ref{re_asymptoticdimension_n}),
\item\label{9DEinvariance_cg_lsg}
coarse properness (Corollary \ref{cpstablebymce}),
\item\label{10DEinvariance_cg_lsg}
uniformly coarse properness (Corollary \ref{ucpstablebymce}),
\item\label{11DEinvariance_cg_lsg}
amenability (Proposition \ref{propamucp}), 
\item\label{12DEinvariance_cg_lsg}
coarse simple connectedness (Proposition \ref{coarse1conninvbycoarseeq}). 
\end{enumerate}
Also, for uniformly coarsely proper pseudo-metric spaces:
\begin{enumerate}[noitemsep,label=(\arabic*)]
\addtocounter{enumi}{12}
\item\label{13DEinvariance_cg_lsg}
growth functions, degrees of polynomial growth and exponential growth
are invariant by quasi-isometry
(Definitions \ref{defgrowthX} \& \ref{defpolgrowthX}, and Proposition \ref{propgrowthXtoY}).
\end{enumerate}
Compare with the situation for \emph{groups}
(Remarks \ref{PropertyPinvce} and \ref{PropertyQinvqi}).

\begin{proof}[Proof of Proposition \ref{invariance_cg_lsg}]
Claims 
\ref{1DEinvariance_cg_lsg}, \ref{3DEinvariance_cg_lsg}, and \ref{5DEinvariance_cg_lsg},
are obvious. 
\par

For Claim \ref{2DEinvariance_cg_lsg}, 
suppose $X$ is coarsely equivalent to a connected metric space $Y$. 
For every $c > 0$, the $c$-coarsely connected components of a metric space are open;
in particular, since $Y$ is connected, $Y$ is coarsely connected.
By \ref{1DEinvariance_cg_lsg},
the space $X$ is also coarsely connected.
\par

Conversely, if $X$ is coarsely connected, then $X$ is coarsely equivalent
to a connected metric space 
by \ref{7DEthe new space Xc aso} and \ref{8DEthe new space Xc aso} 
of Lemma \ref{the new space Xc aso}. 
\par

Similarly, Claims \ref{4DEinvariance_cg_lsg} and \ref{6DEinvariance_cg_lsg} 
follow from
\ref{1DEthe new space Xc aso}, \ref{5DEthe new space Xc aso} 
and \ref{6DEthe new space Xc aso} of Lemma \ref{the new space Xc aso}. 
\end{proof}

\begin{rem}
\label{remcoarselygeod}
(1)
Let $(X,d)$ be a coarsely geodesic pseudo-metric space.
Let $\Phi$ and $c$ be as in Definition \ref{defcoarselyconn}\ref{bDEdefcoarselyconn}.
Define $\nu : X \times X \longrightarrow \R_+$ by
\begin{equation*}
\nu(x,x') \, = \, \min \left\{ 
n \ge 0 \hskip.1cm \Bigg\vert \hskip.1cm
\aligned
& \exists \hskip.1cm x = x_0, x_1, \hdots, x_n = x' \hskip.2cm \text{such that}
\\
& d(x_{i-1},x_i) \le c \hskip.2cm \text{for} \hskip.2cm i=1, \hdots, n 
\endaligned
\right\} .
\end{equation*}
It follows from the definitions that $\nu$ is a pseudo-metric on $X$ and that
\begin{equation*}
\frac{1}{c} d(x,x') \, \le \, \nu(x,x') \, \le \, c\Phi(d(x,x'))
\hskip.5cm \text{for all} \hskip.2cm x,x' \in X .
\end{equation*}
Hence the map
$(X,d) \overset{\operatorname{id}}{\longrightarrow} (X,\nu)$
is a metric coarse equivalence.
(A pseudo-metric similar to $\nu$ will appear in Definition \ref{defDeltaNu}.)

\vskip.2cm

(2)
Let $(X,d)$ be a large-scale geodesic pseudo-metric space.
Similarly,  the map
$(X,d) \overset{\operatorname{id}}{\longrightarrow} (X,\nu)$
is a quasi-isometry.

\vskip.2cm

(3)
A connected metric space is $c$-coarsely connected for every $c > 0$.
Indeed, in every metric space, the $c$-coarsely connected components are open in $X$.
\end{rem}

\begin{prop}[on maps defined on large-scale geodesic spaces]
\label{ceXYqi}
Let $X,Y$ be two pseudo-metric spaces and $f : X \longrightarrow Y$ a map.
\par

(1) 
If $X$ is large-scale geodesic and $f$ coarsely Lipschitz,
then $f$ is large-scale Lipschitz.
\par

(2)
If $X$ and $Y$ are large-scale geodesic and if $f$ is a metric coarse equivalence,
then $f$ is a quasi-isometry.
\end{prop}

\noindent
\emph{Note.} 
(a) A slight variation of Claim (2) appears as
the ``trivial'' lemma of \cite[0.2.D, Page 7]{Grom--93}.
%
\par

(b)
Proposition \ref{ceGG'qi} below is the particular case of Proposition \ref{ceXYqi}
for compactly generated LC-groups.
\par

(c)
Even if $X$ and $Y$ are large-scale geodesic and $f$ coarsely expansive,
$f$ need not be large-scale expansive
(Example \ref{coarselyexpansivenotlargescaleexpansive}(2) below).

\begin{proof}
It is enough to prove the first claim,
because the second follows from the first
applied both to $f$ and to a map $g$ 
such that $f$ and $g$ represent morphisms inverse to each other.

As $X$ is large-scale geodesic, 
there exist constants $a,b,c$ as in (\ref{defcoarselyconn}\ref{cDEdefcoarselyconn}).
As $f$ is coarsely Lipschitz, there exist $c' > 0$ such that
$d_Y(f(x), f(x')) \le c'$ as soon as $d_Y(x,x') \le c$.
Let $x,x' \in X$. 
Let $n$ and $x=x_0, \hdots, x_n=x'$ be as in 
(\ref{defcoarselyconn}\ref{cDEdefcoarselyconn}).
Then
\begin{equation*}
d_Y(f(x), f(x')) \, \le \,
\sum_{i=1}^n d_Y(f(x_{i-1}), f(x_i)) \, \le \,
nc' \, \le \, (c'a) d_X(x,x') + (c'b) .
\end{equation*}
It follows that $f$ is large-scale Lipschitz.
\end{proof}

\begin{exe}
\label{coarselyexpansivenotlargescaleexpansive}
(1)
Let $(X,d)$ be a coarsely geodesic metric space.
The space $(X, \sqrt d)$ is coarsely geodesic,
and is large-scale geodesic if and only if the diameter of $(X,d)$ is finite.
The same holds for the space $(X,d_{\ln})$ of Example \ref{excoarse}(5).

\vskip.2cm

(2)
Given a ring $R$, let $H(R)$ denote the \textbf{Heisenberg group}
\index{Heisenberg group|textbf}
$H(R) =
\begin{pmatrix}
1 & R & R 
\\
0 &1 & R
\\
0 & 0 & 1 
\end{pmatrix}
$
with coefficients in $R$.
Anticipating on \S~\ref{geodesicallyadaptedpseudometricspaces}
and Definition \ref{geodadapted}, consider the groups
$\Z$ and $H(\Z)$ as metric spaces for the word metrics $d_U, d_T$
respectively defined by the generating sets
$U := \{1\} \subset \Z$ and $T := \{s,t,u\} \subset H(\Z)$, where
\begin{equation*}
s = 
\begin{pmatrix}
1 & 1 & 0 
\\
0 &1 & 0
\\
0 & 0 & 1 
\end{pmatrix}
, \hskip.2cm
t = 
\begin{pmatrix}
1 & 0 & 0 
\\
0 &1 & 1
\\
0 & 0 & 1 
\end{pmatrix}
, \hskip.2cm
u = s^{-1}t^{-1}st =
\begin{pmatrix}
1 & 0 & 1 
\\
0 &1 & 0
\\
0 & 0 & 1 
\end{pmatrix} .
\end{equation*} 
The metric spaces $(\Z, d_U)$ and $(H(\Z), d_T)$ are both large-scale geodesic.
The inclusion 
$(\Z, d_U) \lhook\joinrel\relbar\joinrel\rightarrow (H(\Z), d_T), n \longmapsto u^n$
is coarsely expansive, and is not large-scale expansive.
We come back to this example in Remark \ref{Heisenberg}.
\end{exe}

\begin{exe}
\label{du26avril2015}
Let $X$ be the subset
$\{x \in \R_+ \mid x = n^2 \hskip.2cm \text{for some} \hskip.2cm n \in \N \}$ of $\R_+$,
with the metric induced by the standard metric $d(x,y) = \vert x-y \vert$ of $\R_+$.

\vskip.2cm

(1)
Let $j : X  \lhook\joinrel\relbar\joinrel\rightarrow \R_+$ be the natural inclusion.
In the metric coarse category,
the closeness class of $j$ is a monomorphism, because $j$ is an isometric embedding;
but it is not 
left-invertible,
because $\R_+$ is coarsely geodesic and $X$ is not.
See Definitions \ref{defcoarse}, \ref{defcoarsecat}, Remark \ref{remcoarsecat}, 
and Proposition \ref{invariance_cg_lsg}.

\vskip.2cm

(2)
Let $h : X  \lhook\joinrel\relbar\joinrel\rightarrow \R_+$ be defined by $h(n^2) = n$.
In the metric coarse category,
the closeness class of $h$ is an epimorphism, because $h$ is essentially surjective;
but it is not 
right-invertible, for the same reason as in (1).
\end{exe}

\subsection{Coarsely ultrametric pseudo-metric spaces}
\label{subsection_3Bb}

\begin{defn}
\label{defultrametricpseudom}
A pseudo-metric $d$ on a set $X$ is \textbf{ultrametric} if
\index{Ultrametric! pseudo-metric|textbf}
\begin{equation*}
d(x,x'') \, \le \,  \max \{ d(x,x'), d(x',x'') \}
\hskip.5cm \text{for all} \hskip.2cm x, x', x'' \in X .
\end{equation*} 
\end{defn}

\begin{rem}
\label{remultrametricpseudom}
(1)
Let $(X,d)$ be a pseudo-metric space.
Its largest Hausdorff quotient 
$(X_{\operatorname{Haus}}, d_{\operatorname{Haus}})$, as in Example \ref{excoarse}(4),
is ultrametric if and only if the pseudo-metric $d$ is ultrametric.

\vskip.2cm

(2)
Let $X$ be a set and $d$ a pseudo-metric on $X$.
Define   
\begin{equation*}
\lfloor d \rfloor  \, : \, X \times X \longrightarrow \N ,
\hskip.2cm (x,x') \longmapsto  \lfloor d (x,x') \rfloor 
\end{equation*}   
where $\lfloor \cdot \rfloor$ is the floor function,
as in Example \ref{properandlbex}.
\index{Floor function}
Observe that $\lfloor d \rfloor$ need not be a pseudo-metric.
\par
However, if $d$ is an ultrametric pseudo-metric, then so is $\lfloor d \rfloor$.

\vskip.2cm

(3)
Let $G$ be a topological group and
$d$ a left-invariant ultrametric pseudo-metric on $G$.
Suppose that there exists a neighbourhood $S$ of $1$ in $G$
of diameter $\operatorname{diam}_d(S) = \sup_{s \in S} d(1,s) < 1$.
Then $\lfloor d \rfloor$ is locally constant; 
in particular, $\lfloor d \rfloor$ is continuous.
\par
Indeed, for every $g,g' \in G$ and $s,s' \in S$, we have
$\lfloor d \rfloor (1,s) = \lfloor d \rfloor (1,s') = 0$, hence
\begin{equation*}
\lfloor d \rfloor (gs, g's') \le \max \{ \lfloor d \rfloor(gs,g), \lfloor d \rfloor(g,g'), \lfloor d \rfloor(g', g's') \}
= \lfloor d \rfloor (g,g')
\end{equation*}
and similarly $\lfloor d \rfloor (g,g') \le \lfloor d \rfloor (gs,g's')$,
so that $\lfloor d \rfloor (g,g') = \lfloor d \rfloor (gs,g's')$.

\vskip.2cm

(4) 
Let $G$ be a topological group and
$d$ a left-invariant continuous pseudo-metric on $G$.
Consider a radius $r \ge 0$,
the closed ball $B^1_G(r) = \{g \in G \mid d(1,g) \le r \}$,
a subset $S$ of $G$ contained in $B^1_G(r)$,
and the subgroup $H$ of $G$ generated by $S$.
\par

If $d$ is ultrametric,
then $B^1_G(r)$ is a closed subgroup of $G$, and $H \subset B^1_G(r)$.
(Compare with Remark \ref{remoncontadapmetrics}(4) below.)
\end{rem}

\begin{defn}
\label{defcoarselyultrametric}
A pseudo-metric space is \textbf{coarsely ultrametric} if,
for every $r \ge 0$, the equivalence relation generated by the relation
``being at distance at most $r$'' 
has equivalence classes of uniformly bounded diameter.
\index{Coarsely! ultrametric pseudo-metric space|textbf}
\index{Ultrametric! coarsely|textbf}
\end{defn}

\begin{rem}
\label{remcoarselyultrametric}
Let $(X,d)$ be a pseudo-metric space.
For $x,x' \in X$, define $d^{\text{u}}(x,x')$ to be the infimum of $r \ge 0$
such that there exists a finite sequence $x_0 = x, x_1, \hdots, x_n = x'$ 
with $d(x_{i-1},x_i) \le r$ for $i = 1, \hdots, n$.
Then $d^{\text{u}}$ is an ultrametric pseudo-metric on $X$,
and $d^{\text{u}} = d$ if and only if the pseudo-metric $d$ itself is ultrametric.
Let $u$ denote the identity of $X$ viewed as a map $(X,d) \longrightarrow (X,d^{\text{u}})$.
Observe that $u$ is surjective and coarsely Lipschitz.
\par

Definition \ref{defcoarselyultrametric} can be reformulated as follows:
$(X,d)$ is coarsely ultrametric if, for every $r \ge 0$,
there exists $R \ge 0$ such that, for all $x,x' \in X$,
the inequality $d^{\text{u}}(x,x') \le r$ implies $d(x,x') \le R$.
By Proposition \ref{reformulation_c_et_ce}\ref{2DEreformulation_c_et_ce},
this means that the following three properties are equivalent:
\begin{enumerate}[noitemsep,label=(\roman*)]
\item\label{iDEremcoarselyultrametric}
$(X,d)$ is coarsely ultrametric,
\item\label{iiDEremcoarselyultrametric}
$u$ is coarsely expansive,
\item\label{iiiDEremcoarselyultrametric}
$u$ is a coarse metric equivalence.
\end{enumerate}
\par

As a digression, we note that
Definition \ref{defcoarselyconn}\ref{aDEdefcoarselyconn}
can also be reformulated in terms of $d^u$: for a constant $c > 0$,
a pseudo-metric space $(X,d)$ is $c$-coarsely connected
if and only if the diameter of $(X, d^{\text{u}})$ is at most $c$.
\end{rem}

\begin{prop}
\label{propcoarselyultrametric}
Let $(X,d)$ be a pseudo-metric space.
The following properties are equivalent:
\begin{enumerate}[noitemsep,label=(\roman*)]
\item\label{iDEpropcoarselyultrametric}
$X$ is coarsely ultrametric;
\item\label{iiDEpropcoarselyultrametric}
$X$ is coarsely equivalent to an ultrametric pseudo-metric space;
\item\label{iiiDEpropcoarselyultrametric}
there exists a metric $d'$ on $X$ such that $d'$ is ultrametric
and the identity map $(X,d) \longrightarrow (X,d')$
is a coarse equivalence.
\end{enumerate}
\end{prop}

\begin{proof}
To show that \ref{iDEpropcoarselyultrametric} 
implies \ref{iiiDEpropcoarselyultrametric},
we consider the metric $d_1$ of Example \ref{excoarse}(5),
and then apply to $(X, d_1)$ the argument of Remark \ref{remcoarselyultrametric},
which provides the desired ultrametric structure on $X$.
\par
It is clear that \ref{iiiDEpropcoarselyultrametric} 
implies \ref{iiDEpropcoarselyultrametric}.
\par
Finally, let us show that \ref{iiDEpropcoarselyultrametric}
implies \ref{iDEpropcoarselyultrametric}.
Consider two pseudo-metric spaces $(X,d_X)$ and $(Y,d_Y)$.
Assume that $(Y,d_Y)$ is ultrametric, 
and that there exists a metric coarse equivalence $f : X \longrightarrow Y$,
with some upper control $\Phi_+$ and lower control $\Phi_-$.
Define ultrametric pseudo-metrics $d^{\text{u}}_X$ and $d^{\text{u}}_Y$ 
as in Remark \ref{remcoarselyultrametric};
note that $d^{\text{u}}_Y = d_Y$.
\par

Let $r \ge 0$. 
There exists $R \ge 0$ such that, for $t \in \R_+$, the inequality $\Phi_-(t) \le \Phi_+(r)$
implies $t \le R$.
Let $x,x' \in X$ be such that $d^{\text{u}}_X(x,x') \le r$.
Let $x_0 = x, x_1, \hdots, x_n = x'$ be a sequence in $X$
with $d_X(x_{i-1},x_i) \le r$ for $i = 1, \hdots, n$.
Then 
\begin{equation*}
\Phi_-\left(d_X(x,x')\right)
\, \le \, 
d_Y(f(x),f(x')) 
\, \le \, 
\max_{i=1, \hdots, n} d_Y(f(x_{i-1}),f(x_i)) 
\, \le \, 
\Phi_+(r) ,
\end{equation*}
so that $d_X(x,x') \le R$.
Hence $(X,d_X)$ is coarsely ultrametric.
\par

Finally, a pseudo-metric space $(X,d_X)$ is coarsely equivalent to 
an ultrametric pseudo-metric space $(Y, d_Y)$
if and only if $(X,d_X)$ is coarsely equivalent to
the ultrametric largest Hausdorff quotient of $(Y, d_Y)$,
as observed in Remark \ref{remultrametricpseudom}(1).
\end{proof}

\begin{rem}
\label{hyperdNotqiUltram}
Every hyperdiscrete pseudo-metric space, 
in the sense of Example \ref{excoarse}(7), 
is coarsely ultrametric. 
However, it is not necessarily quasi-isometric to an ultrametric space. 
\index{Hyperdiscrete pseudo-metric space}
\par

For instance, if $X$ is the set of square integers 
(which appears in Examples \ref{excoarse}(7) and \ref{du26avril2015}), 
then there is no large-scale bilipschitz map from $X$ to any ultrametric space. 
Indeed, assume by contradiction that $f : X \longrightarrow Y$ is such a map. 
On the one hand,
for some $c > 0$, we have $d(f(x),f(x')) \ge cd(x,x') - c$ for all $x,x'\in X$. 
In particular, $d(f(0),f(n^2)) \ge cn^2-c$ for all $n \in \N$. 
Since $Y$ is ultrametric, this means that 
there exists $i \in \{0,\dots,n-1\}$ such that $d(f(i^2),f((i+1)^2)) \ge cn^2-c$. 
On the other hand, since $f$ is large-scale Lipschitz, 
there exists $c' > 0$ such that $d(f(x),f(x')) \le c'(d(x,x')+1)$ for all $x,x'\in X$. 
In particular, $d(f(n^2),f((n+1)^2)) \le c'(2n+2)$ for all $n \in \mathbf{N}$. 
Combining both inequalities yields a contradiction for $n$ large enough.
\end{rem}

\subsection{Asymptotic dimension}
\label{subsection_3Bc}

The following definition appears in \cite[1.E, Page 28]{Grom--93}.

\begin{defn}
\label{defasymptoticdimension}
Let $X$ be a pseudo-metric space, $n$ a non-negative integer, and $r, R > 0$.
A \textbf{$(n,r,R)$-covering of $X$} is the data of 
index sets $I_0,\cdots, I_n$, 
and of subsets $(X_{c,i})_{0 \le c \le n,i \in I_c}$ of $X$ 
such that
\begin{itemize}
\item 
$d(X_{c,i},X_{c,j})\ge r$ for all $c\in\{0,\dots,n\}$ and distinct $i,j\in I_c$ 
(in particular, for a given $c$, the $X_{c,i}$ are pairwise disjoint when $i$ ranges over $I_c$);
\item 
$\mathrm{diam}(X_{c,i})\le R$ for all $c,i$;
\item 
$X = \bigcup_{c=0}^n\bigsqcup_{i\in I_c}X_{c,i}$.
\end{itemize}
A pseudo-metric space $(X,d)$ has {\bf asymptotic dimension} $\le n$, 
written $\mathrm{asdim}(X) \le n$, if for every $r>0$ 
there exists $R>0$ such that $X$ admits a $(n,r,R)$-covering.
The asymptotic dimension of a non-empty pseudo-metric space $X$
is the smallest $n \ge 0$ for which this holds;
if it holds for no $n$, it is defined to be $\infty$. The asymptotic dimension
of the empty space is defined to be $-\infty$.
\end{defn}
\index{Asymptotic dimension! $\mathrm{asdim}(X)$ of a pseudo-metric space $X$|textbf}

\begin{prop}
\label{re_asymptoticdimension_n}
Let $X, Y$ be pseudo-metric spaces. 
Assume that there exists a coarse embedding $f : X \to Y$.
Then $\mathrm{asdim}(X) \le \mathrm{asdim}(Y)$. 
\par
In particular, if $X$ and $Y$ are coarsely equivalent, 
then $\mathrm{asdim}(X) = \mathrm{asdim}(Y)$.
\end{prop}

\begin{proof}
Set $n = \mathrm{asdim}(Y)$; we can suppose $n < \infty$.
Fix $r>0$. 
Since $f$ is coarsely Lipschitz, there exists $r' > 0$ such that,
for $x,x' \in X$,
$$
d(f(x), f(x')) \ge r' 
\hskip.5cm \text{implies}  \hskip.5cm 
d(x,x') \ge r .
$$
Since $\mathrm{asdim}(Y) \le n$, there exist $R' > 0$
and a $(n, r', R')$-covering $(Y_{c,i})_{c,i}$ of $Y$,
with index sets $I_0, \hdots, I_n$.
Since $f$ is coarsely expansive, there exists $R > 0$ such that,
for $x,x' \in X$, 
$$
d(f(x), f(x')) \le R' 
\hskip.5cm \text{implies}  \hskip.5cm 
d(x,x') \le R .
$$
For each $c \in \{0, \hdots, n\}$ and $i \in I_c$, set $X_{c,i}=f^{-1}(Y_{c,i})$.
Then $(X_{c,i})_{c,i}$ is clearly a covering of $X$.
Moreover, for all $c \in \{0 \hdots, n\}$, 
the diameter of $X_{c,i}$ is at most $R$ for all $i \in I_c$
and $d(X_{c,i},X_{c,j}) \ge r$ for all distinct $i,j \in I_c$. 
Hence $(X_{c,i})_{c,i}$ is a $(n,r,R)$-covering of $X$.
\end{proof}

\begin{exe}
\label{exmasdim}
A non-empty ultrametric space $(X,d)$ has asymptotic dimension $0$.
Indeed, for every $r > 0$, the closed balls of radius $r$
constitute a partition $X = \bigsqcup_{i \in I} B_i$,
and $d(B_i, B_j) > r$ for every $i,j \in I$ with $i \ne j$.
\par

More generally, a non-empty coarsely ultrametric pseudo-metric space
has asymptotic dimension $0$.
\end{exe}

\begin{exe}
\label{ex_asymptoticdimension_n}
Recall that the \textbf{torsion-free rank} 
\index{Torsion-free rank|textbf}
of an abelian group $A$ is the dimension of the $\Q$-vector space $A \otimes_{\Z} \Q$.
The asymptotic dimension of any countable discrete abelian group is equal to its torsion-free rank
\cite[Proposition 9.4]{BCRZ--14}.
\par

For every $n \ge 2$, the following spaces have asymptotic dimension $n$:
the Euclidean space $\mathbf E^n$,
the free abelian group $\Z^n$ (with a word metric), 
\index{Free abelian group}
the hyperbolic space $\mathbf H^n$, and fundamental groups of closed
hyperbolic $n$-manifolds (covered by $\mathbf H^n$).
See \cite{BeDr--11}, in particular Proposition 6 (for $\R^n$) 
and Corollary 22 (for $\mathbf{H}^n$);
see also \cite{BeDr--08} and \cite{NoYu--12}.
\index{Euclidean space $\R^n$}
\index{Hyperbolic space $\mathbf H^n$}
\end{exe}

\section[Metric lattices in pseudo-metric spaces]
{Metric lattices in pseudo-metric spaces}
\label{sectionmetriclattices}

A metric space $(X,d)$ is \textbf{discrete}, 
\index{Discrete! space}
i.e., is such that the underlying topological space is discrete,
if and only if there exists for each $x \in X$ a constant $c_x > 0$ such that
$d(x,x') \ge c_x$ for all $x' \in X$ with $x' \ne x$. 
We define explicitly the corresponding \emph{uniform} notion:

\begin{defn}
\label{defcbsupspace}
Ler $c > 0$ be a constant.
A  pseudo-metric space $D$ is
\textbf{$c$-uniformly discrete} if
\begin{equation*}
\inf\{ d(x,x') \mid  x,x' \in D \hskip.1cm \text{and} \hskip.1cm x \ne x' \} \, \ge \,  c ;
\end{equation*}
it is \textbf{uniformly discrete} 
\index{Uniformly discrete metric space|textbf}
if it is so for some positive constant.
Note that a uniformly discrete pseudo-metric space is a metric space.
\par

A \textbf{$c$-metric lattice in $X$}
is a subspace $L$ that is $c$-uniformly discrete and cobounded
(as defined in \ref{defcobounded});
a \textbf{metric lattice} 
\index{Metric lattice|textbf}    \index{Lattice! metric lattice|textbf}
is a $c$-metric lattice for some $c > 0$.
\end{defn}

Metric lattices are called \emph{separated nets} in \cite{Grom--93}.
\index{Net, see ``metric lattice''}
They already appear in \cite{Efre--53}.
They also occur in the mathematical physics literature, 
often in connection with quasi-crystals, where they are called \emph{Delone sets}.

\begin{rem}
\label{metriclatticesareqi}
The inclusion of a metric lattice (indeed of any cobounded subspace)
in its ambient space is a quasi-isometry.
\par
It follows that, in a pseudo-metric space, any two metric lattices are quasi-isometric.
\end{rem}

\begin{prop}[existence of metric lattices]
\label{reformqi}
Let $X$ be a non-empty pseudo-metric space and $x_0 \in X$.
For every constant $c > 0$,
there exists a $c$-metric lattice $L$ in $X$ containing $x_0$,
such that $\sup \{ d(x,L) \mid x \in X \} \le 2c$.
\par

In particular, every pseudo-metric space has a metric lattice.
\par

More generally, if $M$ is a $c$-uniformly discrete subspace of $X$,
there exists a $c$-metric lattice in $X$ that contains $M$.
\end{prop}

\begin{proof}
Subsets $L$ of $X$ such that $M \subset L$ and
\begin{equation}
\label{existencemetriclattice}
\inf\{ d_X(\ell, \ell') \mid \ell, \ell' \in L \hskip.1cm \text{and} \hskip.1cm \ell \ne \ell' \} \, \ge \, c
\end{equation}
are ordered by inclusion, and Zorn Lemma applies.
If $L$ is maximal, then $L$ is a $c$-metric lattice in $X$ containing $M$,
and  $\sup \{ d(x,L) \mid x \in X \} \le 2c$.
This proves the last statement.
The first statement follows with $M = \{x_0\}$.
\end{proof}

\begin{exe}
\label{exofmetriclattices}
(1)
Uniform lattices in $\sigma$-compact LC-groups (Definition \ref{deflattice})
are metric lattices.
\index{Lattice! in an LC-group}
\par

Note however that, for a $\sigma$-compact LC-group $G$ with a pseudo-metric $d$,
the pseudo-metric space $(G,d)$ contains always metric lattices by the previous proposition,
but $G$ need not contain any uniform lattice,
indeed any lattice whatsoever (Section \ref{sectionlattices}).

\vskip.2cm

(2)
The reader can check that every metric lattice in $\R$ is bilipschitz equivalent with $\Z$.

\vskip.2cm

(3)
Let $L$ be the metric lattice in $\R^2$ 
obtained by placing a point at the centre of each tile of a Penrose tiling of $\R^2$.
Yaar Solomon \cite{Solo--11} has shown that $L$ and $\Z^2$ are bilipschitz equivalent,
thus answering a question raised in \cite{BuKl--02};
see also \cite{AlCG--13}.
\index{Bilipschitz! bilipschitz equivalence}

\vskip.2cm

(4)
The question of the existence of metric lattices in Euclidean spaces
that are not bilipschitz with each other
was asked independently by 
Furstenberg in a context of ergodic theory (see the introduction of \cite{BuKl--02})
and by Gromov in a context of geometry \cite[Page 23]{Grom--93}.
\index{Euclidean space $\R^n$}
In \cite[Item 3.24$_+$]{Grom--99}, Gromov asks more precisely
for which spaces $X$ it is true that two metric lattices in $X$
are always bilipschitz.
\par

Metric lattices in $\R^2$ not bilipschitz with $\Z^2$
were found independently in \cite{BuKl--98} and \cite{McMu--98};
for later examples, see \cite{CoNa} and references there.
\par

Let $A,B$ be two non-trivial finite groups,
and let $\Gamma = A \wr \Z$,
where $\wr$ indicates a wreath product.
If $\vert B \vert$ is not a product a primes dividing $\vert A \vert$, 
Dymarz has shown that $\Gamma$ and the direct product $\Gamma \times B$
are not bilipschitz equivalent \cite{Dyma--10}.
This provides one more example of a metric space 
(here $\Gamma  \times B$, with a word metric)
with two metric lattices 
(here $\Gamma \times B$ itself and $\Gamma$)
which are not bilipschitz.
There are in \cite{DyPT--15} further examples of pairs of groups
which are quasi-isometric and not bilipschitz equivalent.
\par

\index{Nilpotent! group}
It is not known if there exists a finitely generated nilpotent group
that contains two finite index subgroups that are not bilipschitz.
See the discussion in \cite{BuKl--02}.
\par
In contrast, two finitely generated non-amenable groups are bilipschitz
if and only if they are quasi-isometric
\cite[Theorems 7.1 and 4.1]{Whyt--99}.
\index{Bilipschitz! bilipschitz equivalence}

\vskip.2cm

(5)
Despite what the examples above may suggest, 
metric lattices \emph{need not} be locally finite
(in the sense of Section \ref{coarselyproperandgrowth}).
For example, consider the real line $\R$ with the word metric $d_{\mathopen[-1,1\mathclose]}$
(see Section \ref{geodesicallyadaptedpseudometricspaces}),
and a constant $c > 0$.
If $c \le 1$, then $\R$ itself is a $c$-metric lattice in $\R$.
\end{exe}

\begin{exe}[on Mostow rigidity] 
\index{Theorem! Mostow rigidity} \index{Mostow rigidity}
\label{MostowRigidity}
Let $M$ be a connected compact Riemannian manifold.
Let $\widetilde M$ denote its universal cover,
with the Riemannian structure making the covering map a local isometry.
\index{Universal cover}
The fundamental group $\pi_1(M)$ acts on $\widetilde M$ by isometries.
Since $\pi_1(M)$ is finitely generated, we can view it as a metric space,
for the word metric defined by some finite generating set 
(see Section \ref{geodesicallyadaptedpseudometricspaces} below).
Choose $x_0 \in \widetilde M$;
the orbit map provides a quasi-isometry $\pi_1(M) \longrightarrow\widetilde M$
(Theorem \ref{ftggt}).
Its image is a metric lattice in $\widetilde M$, say $K$.
\par

Consider similarly another connected compact Riemannian manifold $N$,
its Riemannian universal covering $\widetilde N$, a point $y_0 \in \widetilde N$, 
and the metric lattice $L = \pi_1(N)y_0$ in $\widetilde N$.
Assume that there exists a homotopy equivalence $f : M \longrightarrow N$.
Denote by $f_* : \pi_1(M) \longrightarrow \pi_1(N)$ the resulting group isomorphism.
Let $\widetilde f : \widetilde M \longrightarrow \widetilde N$ be a lift of $f$;
then $\widetilde f$ is a $f_*$-equivariant quasi-isometry.
\par 

Let now $M$ and $N$ be closed hyperbolic manifolds,
of the same dimension, $n \ge 3$.
Their universal coverings $\widetilde M$ and $\widetilde N$ are both isometric
to the $n$-dimensional hyperbolic space $\mathbf H^n$.
\index{Hyperbolic space $\mathbf H^n$}
Standard topological arguments, 
only relying on the fact that $\widetilde M$ and $\widetilde N$ are contractible, 
show that $M$ and $N$ are homotopy equivalent
if \emph{and only if} $\pi_1(M)$ and $\pi_1(N)$ are isomorphic
(see \cite[Proposition 1B.9]{Hatc--02}, 
or \cite{Hure--36} for an historical reference).
Moreover:
\begin{equation*}
\aligned
&\textbf{Mostow Rigidity Theorem. }
\text{The manifolds $M$ and $N$ are isometric}
\\
&\text{if and only if the groups $\pi_1(M)$ and $\pi_1(N)$ are isomorphic.}
\endaligned
\end{equation*}
Indeed, if $\pi_1(M)$ and $\pi_1(N)$ are isomorphic,
there exists a homotopy equivalence $f : M \longrightarrow N$,
and therefore an equivariant lift 
$\widetilde f : \mathbf H^n \longrightarrow \mathbf H^n$ as above.
It can be shown that $\widetilde f$ extends to the appropriate compactification 
$\mathbf H^n \cup \partial \mathbf H^n$ of $\mathbf H^n$.
The strategy of the proof is to show,
using extra properties of the resulting map from $\partial \mathbf H^n$ to itself,
that $\widetilde f$ is close to an isometry
(in the sense of Definition \ref{defcoarse}).
See \cite{Most--68}, as well as \cite[Sections 5.9 and 6.3]{Thur--80}
and \cite{Grom--81a, BaGS--85}.
\par

A similar strategy applies to Mostow rigidity for
locally Riemannian symmetric spaces of rank at least $2$ \cite{Most--73}.
\par

Mostow Rigidity Theorems have established
the importance of the notion of quasi-isometry,
as already noted in Remark \ref{remarks_categories}(5).
\end{exe}

The following proposition will be used in \S~\ref{coarselyproperandgrowth}.
Part (3) is from \cite[1.A', Page 22]{Grom--93}.

\begin{prop}
\label{dernierepopde3.A}
Let $X,Y$ be pseudo-metric spaces and $f : X \longrightarrow Y$ a map.
\par

(1) Assume that $f$ is coarsely expansive.
There exists $c > 0$ such that, for every $c$-metric lattice $L$ in $X$,
we have $d(f(\ell), f(\ell')) \ge 1$ for all $\ell, \ell' \in L$, $\ell \ne \ell'$.
\par

(2) Assume that $f$ is large-scale Lipschitz.
For every metric lattice $L$ in $X$, the restriction $f \vert_L : L \longrightarrow Y$
is a Lipschitz map.
\par

(3) Assume that $f$ is large-scale bilipschitz.
There exists $k > 0$ such that, for every $k$-metric lattice $L$ in $X$,
the mapping $f \vert_L : L \longrightarrow f(L)$ is bilipschitz.
\par
In particular, if $f$ is a quasi-isometry,
there exists $k > 0$ such that $f(L)$ is a metric lattice in $Y$
for every $k$-metric lattice $L$ in $X$.
\end{prop}

\begin{proof}
(1) 
Let $\Phi_-$ be a lower control for $f$.
It is enough to choose $c$ such that $\Phi_-(c) \ge 1$.
\par

(2) 
Let $L$ be a metric lattice in $X$.
Let $c_+, c > 0$ and $c'_+ \ge 0$ be such that
\begin{equation*}
\aligned
d(f(x), f(x')) \, &\le \,  c_+ d(x,y) + c'_+ \hskip.5cm \text{for all} \hskip.2cm x,x' \in X ,
\\
d(\ell, \ell') \, &\ge \,  c \hskip.5cm \text{for all} \hskip.2cm \ell, \ell' \in L, \hskip.2cm \ell \ne \ell' .
\endaligned
\end{equation*}
Then $d(f(\ell), f(\ell')) \le \left( c_+ + \frac{c'_+}{c} \right) d(\ell, \ell')$
for all $\ell, \ell' \in L$.
\par

(3) Let $c_- > 0, c'_- \ge 0$ be constants such that 
$d(f(x), f(x')) \ge c_- d(x,x') - c'_-$ for all $x,x' \in X$.
Set $k = 2c'_- / c_-$.
Let $L$ be a $k$-metric lattice in $X$.
Then 
\begin{equation*}
d(f(\ell), f(\ell')) \ge c_- d(\ell, \ell') - c'_- \ge (c_- / 2) d(\ell, \ell')
\hskip.5cm \text{for all} \hskip.2cm \ell, \ell' \in L .
\end{equation*}
This and Claim (2) show that $f$ is bilipschitz.
\end{proof}

\begin{rem}
\label{quasiisopartiellementdef}
Let $X,Y$ be pseudo-metric spaces.
Definitions \ref{defcoarse} and \ref{defqi} are formulated in terms of
everywhere defined maps from $X$ to $Y$, but this can be relaxed.
\index{Quasi-isometry}
\par

For example, let $f : X \longrightarrow Y$ be a quasi-isometry.
Let $c_+, c_- > 0$, and $c'_+, c'_-, c' \ge 0$ be such that
$$
\aligned
c_- d_X(x,x') - c'_- \, \le \, d_Y(f(x),f(x')) \, &\le \, c_+ d_X(x,x') + c'_+
\hskip.5cm \text{for all} \hskip.2cm x,x' \in X ,
\\
d_Y(y, f(X)) \, &\le \, c' \hskip.5cm \text{for all} \hskip.2cm y  \in Y .
\endaligned
$$
Let $k > 0$ be such that $c_- k - c'_-  > 0$, 
and let $L$ be a $k$-metric lattice in $X$.
Then the restriction of $f$ to $L$ is a bilipschitz map, and $f(L)$ is a metric lattice in $Y$.
\par

Conversely, suppose that there exists a metric lattice $L$ in $X$
and a bilipschitz map $f_L : L \longrightarrow Y$ with cobounded image.
Set $k = \sup \{ d(x,L) \mid x \in X \}$. For every $x \in X$, choose $\ell \in L$
with $d(x,\ell) \le k$, and set $f(x) = f_L(\ell)$. 
The map $f : X \longrightarrow Y$ defined this way is a quasi-isometry in the sense
of Definition \ref{defqi}.
\par

Consequently, 
\begin{center}
\emph{$X$ and $Y$ are quasi-isometric if and only if 
\\
there exist a metric lattice $L$ in $X$
\\
and a bilipschitz map $f : L \longrightarrow Y$ with cobounded image.}
\end{center}
\par\noindent
(This is how quasi-isometries are first defined in \cite{GrPa--91}.)

Instead of $L$ being a metric lattice in $X$ and $f(L)$ being cobounded in $Y$,
a weaker requirement could be
that $L$ is ``Hausdorff equivalent'' with $X$
and $f(L)$ Hausdorff equivalent with $Y$;
for the definition, we refer to \cite[No 0.2.A]{Grom--93}.
This would be the characterization of quasi-isometries given in
\cite[No 0.2.C]{Grom--93}.
\end{rem}

\section[Growth and amenability]
{Uniformly coarsely proper spaces, growth, 
\\
and amenability}
\label{coarselyproperandgrowth}

\subsection{Growth for uniformly locally finite pseudo-metric 
\\
spaces}
\label{subsection_3Da}

For a pseudo-metric space $X$, a point $x \in X$, and a radius $r \in \R_+$,
recall that $B^{x}_X(r)$ denotes the \textbf{closed ball} $\{y \in X \mid d(x,y) \le r\}$.

\begin{defn}
\label{defunifproper}
A pseudo-metric space $(D,d)$ is \textbf{locally finite} 
\index{Locally finite! pseudo-metric space|textbf}
if all its balls are finite,
and \textbf{uniformly locally finite} 
\index{Uniformly locally finite pseudo-metric space|textbf}
if $\sup_{x \in D} \vert B^{x}_D(r) \vert < \infty$ for all $r \ge 0$.
\end{defn}

A discrete metric space is locally finite if and only if it is proper
(see Definition \ref{defproper2}).
\index{Discrete! space}
Note that the formulation of Definition \ref{defunifproper}
avoids ``discrete pseudo-metric spaces'',
in compliance with (A1) of Page~\pageref{A1}
and Remark \ref{abaslestoppseudometriques}.
\par
The vertex set of a connected graph, for the combinatorial metric (Example \ref{metricrealizationgraph}),
is uniformly locally finite if and only if the graph is of bounded valency;
if $d$ is a valency bound, balls of radius $n$ have at most $d(d-1)^{n-1}$ vertices,
for all $n \ge 1$.

\begin{defn}
\label{defgrowth}
Let $D$ be a non-empty locally finite pseudo-metric space and $x \in D$.
The \textbf{growth function $\beta^x_D$ of $D$ around $x$} is defined by
\index{Growth function! of a
non-empty locally finite pseudo-metric space|textbf}
\begin{equation*}
\beta^x_D(r) \, = \, \vert B^x_D(r) \vert \, \in \, \N \cup \{\infty\} 
\hskip.5cm \text{for all} \hskip.2cm r \in \R_+ .
\end{equation*}
\end{defn}

\begin{defn}
\label{defgrowthclass}
Let $\beta, \beta'$ be two non-decreasing functions from $\R_+$ to $\overline \R_+$.
We define 
\begin{equation*}
\aligned
\beta \,  &\preceq \, \beta'  \hskip.2cm \text{if there exist} \hskip.2cm
\lambda, \mu > 0, c \ge 0
\\
& \hskip1cm
\text{such that} \hskip.2cm
\beta(r) \le \lambda \beta' (\mu r + c) + c 
\hskip.2cm \text{for all} \hskip.2cm  r \ge 0 ,
\\
\beta \, &\simeq \, \beta' , \hskip.2cm \text{i.e., $\beta$ 
and $\beta'$ are \textbf{equivalent}, if} \hskip.2cm
\beta   \preceq  \beta' \hskip.2cm \text{and} \hskip.2cm \beta'  \preceq  \beta ,
\\
\beta &\precnsim \beta' \hskip.2cm \text{if} \hskip.2cm 
\beta \preceq \beta' \hskip.2cm \text{and} \hskip.2cm \beta' \npreceq \beta.
\endaligned
\end{equation*}
\index{$ab$@$\simeq$ equivalence of growth functions}
\index{Equivalence of growth functions|textbf}
The \textbf{class} of $\beta$ is its equivalence class modulo the relation $\simeq$.
\par

Note that the relation $\preceq$ induces a partial order,
denoted by $\preceq$ again, on the set of classes modulo $\simeq$.
\par

We often write abusively $\beta$  for the class of an actual function $\beta$.
Consequently, if $\beta$ and $\beta'$ are two classes,
$\beta \simeq \beta'$ stands for $\beta = \beta'$.
\end{defn}

\begin{exe}
\label{excompgrowthclass}
Let $a,b,c,d \in \R$ with $a,b > 0$ and $c,d > 1$.
We have $r^a \preceq r^b$ if and only if $a \le b$,
and $c^r \simeq d^r$ for all $c, d$.
Moreover, $r^a \precnsim e^{\sqrt r} \precnsim e^{r/\ln r}  \precnsim e^r \precnsim e^{e^r}$.
\end{exe}

\begin{defn}
\label{firstbetaD}
Let $D$ be a non-empty uniformly locally finite pseudo-metric space.
Let $x,y \in D$. 
Observe that $B_D^y(r) \subset B_D^x(r + d(x,y))$
and $B_D^x(r) \subset B_D^y(r + d(x,y))$, 
so that $\beta_D^x \simeq \beta_D^y$.
\par
The \textbf{growth type} of $D$ is the class, denoted by $\beta_D$,
of the function $\beta_D^x$.
The growth type of the empty space is defined to be the class of the zero function.
\end{defn}

In \ref{defgrowthX}, the definition will be extended to a larger class of spaces.

\begin{prop}
\label{zzzz}
Let $D, E$ be non-empty uniformly locally finite pseudo-metric spaces, 
$x \in D$, and $y \in E$.

\vskip.2cm

(1)
Let  $L$ be a $c$-metric lattice in $D$ containing $x$, for some $c > 0$.
Then $\beta^x_L \simeq \beta^x_D$ and $\beta_L = \beta_D$.

\vskip.2cm

(2)
Assume that there exist $c > 0$ and an injective map 
$f : D \lhook\joinrel\relbar\joinrel\rightarrow E$ such that 
$d(f(x'), f(x''))  \le cd(x',x'')$ for all $x',x'' \in D$. 
Then $\beta^x_D \preceq \beta^y_E$ and $\beta_D \preceq \beta_E$.
\par
If moreover $f(D)$ is cobounded in $E$, then $\beta^x_D \simeq \beta^y_E$
and $\beta_D = \beta_E$.

\vskip.2cm

(3)
If $D$ and $E$ are quasi-isometric, then $\beta^x_D \simeq \beta^y_E$
and $\beta_D = \beta_E$.
\end{prop}

\begin{proof}
(1)
On the one hand, since
$L \subset D$, we have $\beta^x_L \preceq \beta^x_D$.
On the other hand, we have
\begin{equation*}
B^x_D (r) \, \subset \, \bigcup_{ x' \in B^x_L (r+c)} B^{x'}_D(c) 
\hskip.2cm \text{for all} \hskip.2cm r \ge 0 .
\end{equation*}
Taking cardinals, we have $\beta^x_D(r) \le \lambda \beta^x_L(r+c)$,
where $\lambda = \sup_{x' \in D} \vert B^{x'}_D(c) \vert$; hence $\beta^x_D \preceq \beta^x_L$.
\par

(2)
Set $c' = d(f(x), y)$. 
The restriction of $f$ to the ball $B^x_D(r)$
is a bijection into the ball $B^{f(x)}_E (cr)$,
and the latter is contained in $B^y_E(cr+c')$.
Hence $\beta^x_D (r) \le  \beta^y_E (cr + c')$ for all $r \ge 0$.
\par

Suppose moreover that $s := \sup_{y' \in E} d(y', f(D)) < \infty$.
For every $y' \in E$, there exists $x' \in X$ with $d(y', f(x')) \le s$; hence
\begin{equation*}
B^y_E(r) \, \subset \, B^{f(x)}_E(r + c') \, \subset \,
\bigcup_{\substack{x' \in D \\ d(f(x)),f(x')) \le r + c' + s}} B^{f(x')}_E(s) 
\, \subset \, \bigcup_{\substack{x' \in D \\ d(x,x') \le c(r + c' + s)}} B^{f(x')}_E(s) .
\end{equation*}
Taking cardinals, we have
$\beta^y_E(r) \le \mu \beta^x_D(cr + cc' + cs)$, where
$\mu = \sup_{y' \in E} \vert B^{y'}_E (s) \vert$;
hence $\beta^y_E \preceq \beta^x_D$.

(3)
If $D$ and $E$ are quasi-isometric, there exists
a metric lattice $L$ in $D$ containing $x$
such that there exists an injective maps  
$L \lhook\joinrel\relbar\joinrel\rightarrow E$.
Hence $\beta^x_D \simeq \beta^x_L$ by (1)
and $\beta^x_L \simeq \beta^y_E$ by (2).
\end{proof}

\begin{exe} 
\label{botanicgarden}
Proposition \ref{zzzz} \emph{does not} carry over to locally finite spaces.
More precisely, 
if $D$ is a locally finite metric space and $E$ a metric lattice in $X$,
growth functions of $D$ need not be similar to those of $E$,
as the following example shows.
\par

Consider the metric tree $T_0$ with vertex set $(x_k)_{k \in \N}$
and edge set $(\{x_k, x_{k+1}\})_{k \in \N}$.
\index{Tree}
Let $s = (s_k)_{k \in \N}$ be a sequence of non-negative integers;
let $T_s$ be the tree obtained from $T_0$ by adding
extra vertices $y_{k,\ell}$ and extra edges $\{x_k, y_{k,\ell} \}$,
for $k \ge 0$ and $1 \le \ell \le s_k$.
\par

The vertex set $D_0$ of $T_0$ is a $1$-metric lattice in the vertex set $D_s$ of $T_s$.
On the one hand, $D_0$ is uniformly locally finite, with linear growth.
On the other hand, the growth $\beta_{D_s}$ of $D_s$ 
is highly sensitive to the choice of the sequence $s$,
and can be made arbitrarily large;
in particular the space $D_s$ is uniformly locally finite if and only if $s$ is bounded.
\par

Anticipating on Definition \ref{defgrowthX}, 
the growth type of $T_s$ is linear although the cardinal of its balls of radius $r$ 
can grow arbitrarily fast. 
This shows the importance of restricting to uniformly locally finite spaces 
in Definition \ref{firstbetaD}.
\end{exe}

\begin{prop}
\label{toutefonctionestcroissance}
Let $D$ be a pseudo-metric space which is uniformly locally finite and large-scale geodesic.
Then $\beta_D(r) \preceq \exp (r)$.
\end{prop}

\begin{proof}
By Lemma \ref{the new space Xc aso}\ref{6DEthe new space Xc aso},
the space $D$ is quasi-isometric to a
connected graph $E$ which is uniformly locally finite;
since this graph is of bounded degree, we have $\beta_E(r) \preceq \exp (r)$.
By Proposition \ref{zzzz}(3), we have also $\beta_D(r) \preceq \exp (r)$.
\end{proof}

\begin{exe}
\label{toutefonctionestcroissanceEXPLE}
The two next examples show that Proposition \ref{toutefonctionestcroissance}
need not hold when $D$ is not large-scale geodesic.

\vskip.2cm

(1)
Let $f : \R_+ \longrightarrow \R_+$ be a concave strictly increasing function 
with $f(0) = 0$ and $\lim_{+\infty}f = +\infty$. 
Let $d$ be the standard metric on $\Z$. 
Then $D = (\Z, f \circ d)$ is a uniformly locally finite metric space: 
indeed, the translation are isometries 
and thus all balls of given radius have the same cardinal. 
The slower $f$ grows to $+\infty$, the faster $\beta_D$ is. 
For instance, if $f(x) = \ln(1+\ln(1+x))$, then $\beta^0_D(r) = 2\lfloor e^{e^r-1}\rfloor-1$.

\vskip.2cm

(2)
Let $s = (s_k)_{k \ge 1}$ be a sequence of positive integers.
Let $T_s$ be the tree with vertex set
$D_s =  \Z \times \N$,
in which, for every $k \ge 1$, 
every vertex $(j,k)$ of the $k^{\text{th}}$ row
is connected by an edge to the $s_k$ vertices
\begin{equation*}
(s_kj, k-1),\hskip.1cm (s_kj+1, k-1),\hskip.1cm \hdots,\hskip.1cm (s_k(j+1)-1, k-1)
\end{equation*}
of the previous row.
We endow the vertex set $D_s$ with the combinatorial metric of the tree $T_s$.
\par

Consider the row $E_s = \{ (j,0) \mid j \in \Z \}$ of $D_s$,
with the induced metric. We have
\begin{equation*}
\beta_{E_s}^{(j,0)}(2k) \, = \,
\left\vert B_{E_s}^{(j,0)}(2k) \right\vert \, = \, s_1 s_2 \cdots s_k
\hskip.5cm \text{for all} \hskip.2cm j \in \Z
\hskip.2cm \text{and} \hskip.2cm k \ge 1 .
\end{equation*}
In particular, $E_s$ is uniformly locally finite.
For any non-decreasing function $\alpha : \R_+ \longrightarrow \R_+$,
there exists a sequence $s$ such that $\alpha \preceq \beta_{E_s}$.
\end{exe}

\subsection{Growth for uniformly coarsely proper pseudo-metric spaces
and $\sigma$-compact LC-groups}
\label{subsection_3Db}

\begin{defn}
\label{defcoarselyproper}
A pseudo-metric space $(X, d_X)$ is {\bf coarsely proper} 
if there exists $R_0 \ge 0$ such that 
every bounded subset in $X$ can be covered by finitely many balls of radius $R_0$.
\index{Coarsely! proper pseudo-metric space|textbf}
\end{defn}

\begin{prop} 
\label{predefcoarselyproper}
For a pseudo-metric space $(X,d_X)$, the following conditions are equivalent:
\begin{enumerate}[noitemsep,label=(\roman*)]
\item\label{iDEpredefcoarselyproper}
$X$ is coarsely proper;
\item\label{iiDEpredefcoarselyproper}
$X$ is coarsely equivalent to
a locally finite discrete metric space;
\item\label{iiiDEpredefcoarselyproper}
$X$ is quasi-isometric to 
a locally finite discrete metric space;
\item\label{ivDEpredefcoarselyproper}
the space $X$ contains a locally finite metric lattice;
\index{Metric lattice}    \index{Lattice! metric lattice}
\item\label{vDEpredefcoarselyproper}
for all $c > 0$, the space $X$ contains a locally finite $c$-metric lattice;
\item\label{viDEpredefcoarselyproper}
there exists $c_0 > 0$ such that, for every $c \ge c_0$, 
every $c$-metric lattice  in $X$ is locally finite.
\end{enumerate}
If moreover $X$ is large-scale geodesic, these conditions are equivalent to:
\begin{enumerate}[noitemsep,label=(\roman*)]
\addtocounter{enumi}{6}
\item\label{viiDEpredefcoarselyproper}
$X$ is quasi-isometric to a locally finite connected graph.
\end{enumerate}
\end{prop}

\begin{proof}
Implications 
\ref{vDEpredefcoarselyproper} $\Rightarrow$ 
\ref{ivDEpredefcoarselyproper} $\Rightarrow$ 
\ref{iiiDEpredefcoarselyproper} $\Rightarrow$ 
\ref{iiDEpredefcoarselyproper},
\ref{iiiDEpredefcoarselyproper} $\Rightarrow$ 
\ref{iDEpredefcoarselyproper},
and
\ref{viiDEpredefcoarselyproper} $\Rightarrow$ 
\ref{iiiDEpredefcoarselyproper}
are straightforward.
It suffices to show
\ref{iDEpredefcoarselyproper} $\Rightarrow$ 
\ref{iiiDEpredefcoarselyproper},
\ref{iiDEpredefcoarselyproper} $\Rightarrow$
\ref{ivDEpredefcoarselyproper} $\Leftrightarrow$
\ref{viDEpredefcoarselyproper},
\ref{ivDEpredefcoarselyproper} $\Rightarrow$
\ref{vDEpredefcoarselyproper}
and, with the additional condition, 
\ref{iiiDEpredefcoarselyproper} $\Rightarrow$
\ref{viiDEpredefcoarselyproper}.

\vskip.2cm

\ref{iDEpredefcoarselyproper} $\Rightarrow$ 
\ref{iiiDEpredefcoarselyproper}
Assume that $X$ is coarsely proper; 
let $R_0$ be as in Definition \ref{defcoarselyproper}.
Let $L$ be a $3R_0$-metric lattice in $X$ (Proposition \ref{reformqi});
it is enough to show that $L$ is locally finite.
Let $F$ be a bounded subset of $L$;
then $F$ is included in the union of a finite number of balls of radius $R_0$ in $X$;
since two distinct points in $L$ are at distance at least $3R_0$ apart,
each of these balls contains at most one point of $F$;
hence $F$ is finite.

\vskip.2cm

\ref{iiDEpredefcoarselyproper} $\Rightarrow$
\ref{ivDEpredefcoarselyproper}.
Assume there exists a locally finite discrete metric space $(D,d_D)$
and a metric coarse equivalence $f : D \longrightarrow X$.
By Proposition \ref{dernierepopde3.A}(1),
there exists a metric lattice $E \subset D$
such that $d_X(f(e), f(e')) \ge 1$ for all $e,e' \in E$, $e \ne e'$.
Then $f(E)$ is a metric lattice in $X$;
we denote by $d_E$ the restriction of $d_D$ to $E$.
\par

If $B$ is a bounded subset of $f(E)$, 
then $f^{-1}(B)$ is bounded because $f$ is coarsely expansive. 
Since the space $(E, d_E)$ is proper, $f^{-1}(B) \cap E$ is finite.
By injectivity of $f|_E$, we deduce that $B$ is finite. 
Hence $f(E)$ is locally finite, and (iv) holds.

\vskip.2cm

\ref{ivDEpredefcoarselyproper} $\Rightarrow$
\ref{viDEpredefcoarselyproper}.
Assume there exists a locally finite metric lattice $D$ in $X$.
It is $R$-cobounded for some $R > 0$.
Let $c_0 = 2R$.
Consider a constant $c > c_0$ and a $c$-metric lattice $E$ in $X$;
we have to show that $E$ is locally finite.
\par

For every $e \in E$, we can choose an element in $D$, call it $f(e)$,
such that $d_X(e, f(e)) \le R$.
The resulting map $f : E \longrightarrow D$ is injective;
indeed, if $e',e'' \in E$ are such that $f(e') = f(e'')$, 
then $d_X(e',e'') \le 2R < c$,
hence $e' = e''$ since $E$ is a $c$-metric lattice.
Moreover, for a radius $r \ge 0$, the image by $f$ of the ball $B_E^e(r)$
is inside the ball $B_D^{f(e)}(r + 2R)$.
Since the latter is finite by hypothesis on $D$, so is $B_E^e(R)$.
Hence $E$ is locally finite.

\vskip.2cm

\ref{ivDEpredefcoarselyproper} $\Rightarrow$ \ref{vDEpredefcoarselyproper}.
Recall that, for every $c > 0$, every metric lattice contains a $c$-metric lattice,
by Proposition \ref{reformqi}. The implication follows.

\vskip.2cm

\ref{viDEpredefcoarselyproper} $\Rightarrow$ \ref{ivDEpredefcoarselyproper}.
This follows from the same proposition,
which establishes the existence of $c$-metric lattices in $X$.

\vskip.2cm

\ref{iiiDEpredefcoarselyproper} $\Rightarrow$
\ref{viiDEpredefcoarselyproper}.
Assuming that $(X,d)$ is $c$-large-scale geodesic, for some $c > 0$.
We can suppose that $X$ is a locally finite discrete metric space. 
Consider the graph for which vertices are elements of $X$ 
and two distinct elements are linked by an edge if they are at distance $\le c$. 
This graph is connected and, if $d'$ is the resulting metric, 
the identity map $(X,d)\to (X,d')$ is a quasi-isometry. 
Since $X$ is locally finite, this graph obviously is locally finite.
\end{proof}

\begin{cor}
\label{cpstablebymce}
For pseudo-metric spaces, coarse properness
is a property invariant by metric coarse equivalence.
\end{cor}

\begin{proof}
See \ref{iiDEpredefcoarselyproper} in the previous proposition.
\end{proof}

\begin{prop}
\label{propcoarselyproper}
(1) 
A proper metric space is coarsely proper.
\par
(2)
A pseudo-metric space that can be coarsely embedded 
into a coarsely proper metric space is itself coarsely proper.
\end{prop}

\begin{proof}
(1)
Let $X$ be a pseudo-metric space.
By Proposition \ref{reformqi},
there exists in $X$ a metric lattice $L$
such that $\inf_{\ell,\ell' \in L, \ell \ne \ell'} d(\ell, \ell') > 2$,
and the inclusion $L \subset X$ is a quasi-isometry.
Suppose moreover that $X$ is a proper metric space
(recall from \ref{pm...metrizable} that
``proper pseudo-metric space'' is not a notion for this book);
it suffices to check that balls in $L$ are finite.
\par

If $L$ had an infinite ball 
$B_L^{\ell_0}(R)  = \{ \ell \in L \mid d(\ell_0, \ell) \le R \}$,
the space $X$ would contain an infinite family
$\left( \{ x \in X \mid d(\ell,x) \le 1 \} \right)_{\ell \in B}$
of pairwise disjoint non-empty balls of radius $1$, 
all contained in the relatively compact ball
$B_X^{\ell_0}(R+1) = \{ x \in X \mid d(\ell_0, x) \le R+1 \}$,
and this is impossible.
\par

Claim (2) is obvious for isometric embeddings, 
and combining with Corollary \ref{cpstablebymce} yields the general case.
\end{proof}

\begin{exe}
\label{homogeneousinfinitetree+Hilbert}
(1) For other examples of pseudo-metric spaces that are coarsely proper,
see Example \ref{exemplesucpspaces} and Proposition \ref{growth=growth}.
We indicate here some examples that do not have the property.

\vskip.2cm

(2)
Let $S$ be the tree obtained by attaching to a vertex $s_0$ 
an infinite sequence $(R_n)_{n \ge 1}$ of infinite rays;
$S$ is a metric space for the combinatorial metric (``$S$'' stands for ``star'').
Then $S$ is not coarsely proper.
\index{Tree}
\par

Indeed, let $c > 0$ and $L$ a $c$-metric lattice in $S$ containing $s_0$.
For each $n \ge 1$, let $s_n$ be the point at distance $2c$ of $s_0$ on $R_n$.
There exists a point $\ell_n \in L$ at distance at most $c$ of $s_n$;
observe that $\ell_n \in R_n$.
The ball $B^{s_0}_L(3c)$ is infinite, 
because it contains $\ell_n$ for all $n \ge 1$.
Hence $L$ is not locally finite, and $S$ is not coarsely proper.

\vskip.2cm

(3)
Let $X$ be a tree with at least one vertex of infinite valency 
and without vertices of valency $1$.
Then $X$ is not coarsely proper, because it contains a subtree isomorphic to $S$.
\par
In particular, the homogeneous tree of infinite valency is not coarsely proper.

\vskip.2cm

(4)
An infinite-dimensional Hilbert space $\mathcal H$, 
together with the metric given by the norm,
is not coarsely proper.
\index{Hilbert space}
\par
Indeed, there exists a bilipschitz embedding of $S \longrightarrow \mathcal H$,
with image contained in the union of the positive half-axes 
with respect to some orthonormal basis of the Hilbert space.
\end{exe}

Definition \ref{defunifcoarselyproper} and
Proposition \ref{predefuniformlycoarselyproper} 
constitue a variation on 
Definition \ref{defcoarselyproper} and
Proposition \ref{predefcoarselyproper}.

\begin{defn}
\label{defunifcoarselyproper}
A pseudo-metric space $(X,d_X)$ is {\bf uniformly coarsely proper} 
if there exists $R_0 \ge 0$ such that, for every $R \ge 0$, 
there exists an integer $N$ such that every ball of radius $R$ in $X$ 
can be covered by $N$ balls of radius $R_0$.
\index{Uniformly coarsely proper pseudo-metric space|textbf}
\end{defn} 

Note that a uniformly coarsely proper pseudo-metric space is coarsely proper.
Some authors write ``$X$ has \textbf{bounded geometry}''
\index{Bounded! bounded geometry|textbf}
rather than ``$X$ is uniformly coarsely proper''.
Others use ``bounded geometry'' 
for ``uniformly discrete and uniformly locally finite''  \cite{Whyt--99}, 
or ``uniformly locally finite'' \cite[Page 8]{NoYu--12}.

\begin{prop}
\label{predefuniformlycoarselyproper}
For a pseudo-metric space $(X,d_X)$, the following conditions are equivalent:
\begin{enumerate}[noitemsep,label=(\roman*)]
\item\label{iDEpredefuniformlycoarselyproper}
$X$ is uniformly coarsely proper;
\item\label{iiDEpredefuniformlycoarselyproper}
$X$ is coarsely equivalent to
a uniformly locally finite discrete metric space;
\item\label{iiiDEpredefuniformlycoarselyproper}
$X$ is quasi-isometric to 
a uniformly locally finite discrete metric space;
\item\label{ivDEpredefuniformlycoarselyproper}
$X$ contains a uniformly locally finite metric lattice;
\item\label{vDEpredefuniformlycoarselyproper}
for all $c > 0$, the space $X$ contains a uniformly locally finite $c$-metric lattice;
\item\label{viDEpredefuniformlycoarselyproper}
there exists $c_0 > 0$ such that, for every $c \ge c_0$, 
every $c$-metric lattice  in $X$ is uniformly locally finite.
\end{enumerate}
If moreover $X$ is large-scale geodesic, these conditions are equivalent to:
\begin{enumerate}[noitemsep,label=(\roman*)]
\addtocounter{enumi}{6}
\item\label{viiDEpredefuniformlycoarselyproper}
$X$ is quasi-isometric to a connected graph of bounded valency.
\end{enumerate}
\end{prop}

\begin{proof}
We leave it to the reader to adapt the proof of Proposition \ref{predefcoarselyproper}.
\end{proof}

\begin{cor}
\label{ucpstablebymce}
For pseudo-metric spaces, uniformly coarse properness
is a property invariant by metric coarse equivalence.
\end{cor}

\begin{proof}
See \ref{iiDEpredefuniformlycoarselyproper} in the previous proposition.
\end{proof}

\begin{exe}
\label{exemplesucpspaces}
Let $M$ be a connected compact manifold, 
furnished with a Riemannian metric.
Let $\widetilde M$ denote the universal cover of $M$,
furnished with the covering Riemannian metric and the corresponding metric $d$.
\index{Universal cover}
Then $(\widetilde M,d)$ is uniformly coarsely proper,
because every orbit in $\widetilde M$ of the fundamental group of $\pi_1(M)$
is a uniformly locally finite metric lattice in $\widetilde M$
(see Example \ref{MostowRigidity}).
\par

In particular, Euclidean spaces $\R^m$ ($m \ge 1$) and
hyperbolic spaces $\mathbf{H}^n$ ($n \ge 2$) are uniformly coarsely proper.
\index{Euclidean space $\R^n$}
\end{exe}

\begin{prop}
\label{2latticessamegrowth}
Let $X$ be a non-empty uniformly coarsely proper pseudo-metric space.
Let $L_0,L_1$ be  metric lattices in $X$ and $x_0 \in L_0$, $x_1 \in L_1$. Then
\begin{equation*}
\beta^{x_0}_{L_0} \,  \simeq \, \beta^{x_1}_{L_1} .
\end{equation*}
\end{prop}

\begin{proof}
Denote by $i_j : L_j \lhook\joinrel\relbar\joinrel\rightarrow X$ the inclusion ($j=0,1$).
Choose a quasi-isometry $g_1 : X \longrightarrow L_1$ inverse of $i_1$,
and let $f : L_0 \longrightarrow L_1$ be the composition $g_1 i_0$.
By Proposition \ref{dernierepopde3.A}, 
there is a metric lattice $M_0$ in $L_0$ 
such that $x_0 \in M_0$, 
and $M_1 := f(M_0)$ is a metric lattice in $L_1$.
Since $L_0, L_1$ are uniformly locally finite (Proposition \ref{predefuniformlycoarselyproper}),
the claim  follows, by Proposition \ref{zzzz}.
\end{proof}

\begin{defn}
\label{defgrowthX}
Let $X$ be a uniformly coarsely proper pseudo-metric space.
If $X$ is non-emtpy,
the \textbf{growth type} \index{Growth function! of a pseudo-metric space|textbf}
of $X$ is the class, denoted by $\beta_X$,
of the function $\beta^x_L$, 
for $L$ a uniformly locally finite metric lattice in $X$ and $x \in L$;
the growth type of the empty space is the class of the zero function.
\end{defn}

\begin{defn}
\label{defpolgrowthX}
For $X$ as in the previous definition,
the \textbf{upper degree of polynomial growth}
\index{Upper degree of polynomial growth of a non-empty uniformly coarsely proper
pseudo-metric space|textbf}
\begin{equation*}
\operatorname{poldeg}(X) \, = \,  \limsup_{r \to \infty} \frac{ \ln \beta_X (r) }{ \ln r }
\, \in \, \mathopen[0,   \infty \mathclose]
\end{equation*} 
is well-defined.
For uniformly coarsely proper pseudo-metric spaces,
the upper degree of polynomial growth is a quasi-isometry invariant.
\index{Property! of a pseudo-metric space invariant 
by metric coarse equivalence or by quasi-isometries}
\par

A uniformly coarsely proper pseudo-metric space has
\textbf{polynomial growth} 
\index{Polynomial growth|textbf}
\index{Growth! polynomial growth|textbf}
if $\operatorname{poldeg}(X) < \infty$;
it has \textbf{exponential growth} 
\index{Exponential growth|textbf}
\index{Growth! exponential growth|textbf}
if $\beta_X(r) \simeq e^r$.
\end{defn}

\begin{rem}
\label{remongrowthfunction}
Let $X$ be a uniformly coarsely proper pseudo-metric space.
If $X$ is large-scale geodesic, the growth type of $X$ is at most exponential, 
$\beta_X(r) \preceq e^r$. 
In other cases, it can grow much faster, 
as shown in Example \ref{toutefonctionestcroissanceEXPLE}.
\end{rem}

\begin{prop}
\label{propgrowthXtoY}
Let $X,Y$ be two pseudo-metric spaces.
Suppose that $Y$ is uniformly coarsely proper.
\begin{enumerate}[noitemsep,label=(\arabic*)]
\item\label{1DEpropgrowthXtoY}
If there exists a coarse embedding $f$ of $X$ in $Y$,
then $X$ is uniformly coarsely proper.
In particular, the growth type of $X$ is well-defined.
\item\label{2DEpropgrowthXtoY}
If $f$ is moreover large-scale Lipschitz,
then $\beta_X  \preceq  \beta_Y$.
\item\label{3DEpropgrowthXtoY}
If $f$ is moreover a quasi-isometry,
then  $\beta_X  \simeq  \beta_Y$.
\end{enumerate}
\end{prop}

\begin{proof}
By hypothesis, there exists a subspace $Y_0$ of $Y$
and a surjective metric coarse equivalence 
$f : X  \relbar\joinrel\twoheadrightarrow Y_0$.
Hence, there exists a metric coarse equivalence $g : Y_0 \longrightarrow X$  
such that the morphisms defined by $f$ and $g$ are inverse to each other
in the metric coarse category (Proposition \ref{epimonoiso}).
Let $c > 0$ be such that $d_X(g(y),g(y')) \ge 1$ for all $y,y' \in Y_0$ with $d_Y(y,y') \ge c$.
\par

Let $M_0$ be a $c$-metric lattice in $Y_0$, 
and $M$ a $c$-metric lattice in $Y$ that contains $M_0$
(Proposition \ref{reformqi}).
Since $M$ is uniformly locally finite by hypothesis on $Y$, so is $M_0$.
It is then straightforward to check that $g(M_0)$ is a lattice in $X$
that is uniformly locally finite; this proves \ref{1DEpropgrowthXtoY}.
\par

Claims \ref{2DEpropgrowthXtoY} and \ref{3DEpropgrowthXtoY} 
follow from Proposition \ref{zzzz}.
\end{proof}

\begin{cor}
\label{exgrowthXtoY}
Let $X,Y$ be two pseudo-metric spaces.
Suppose that $X$ is large-scale geodesic
and that $Y$ is uniformly coarsely proper.
\par

If there exists a coarse embedding of $X$ in $Y$,
then $\beta_X \preceq \beta_Y$.
\end{cor}
\index{Coarse! embedding}

\begin{proof}
Note that $X$ is also uniformly coarsely proper,
by Proposition \ref{propgrowthXtoY}(1).
\par

Let $f : X \longrightarrow Y$ be a coarse embedding;
set $Y_0 = f(X)$.
Let $c > 0$ be such that
\begin{equation*}
d_Y(f(x), f(x')) \, \ge \, 1 
\hskip.5cm \text{for all} \hskip.2cm x,x' \in X
\hskip.2cm \text{with} \hskip.2cm d_X(x, x') \ge c .
\end{equation*}
By Proposition \ref{ceXYqi}, $f$ is large-scale Lipschitz, 
so that there exist $c_+ > 0$, $c'_+ \ge 0$ such that
\begin{equation*}
d_Y(f(x), f(x')) \, \le \, c_+ d_X(x, x') + c'_+ 
\hskip.5cm \text{for all} \hskip.2cm x,x' \in X .
\end{equation*}
\par

If $X$ is empty, $\beta_X = 0$, and there is nothing to show;
we assume from now on that there exists a point $x_0 \in X$.
Let $L$ be a $c$-metric lattice in $X$ containing $x_0$.
Then $M_0 := f(L)$ is a $1$-metric lattice in $Y_0$,
and there exists a $1$-metric lattice $M$ in $Y$ containing $M_0$.
Observe that the restriction $f \vert_L : L \longrightarrow M$ is injective, 
and that
\begin{equation*}
f(B^{x_0}_L (r)) \, \subset \, B^{f(x_0)}_M (c_+ r + c'_+)
\hskip.5cm \text{for all} \hskip.2cm r \ge 0 .
\end{equation*}
Since $f \vert_L$ is injective, this implies
$\beta_L^{x_0}(r) \le \beta_M^{f(x_0)}(c_+ r + c'_+)$ for all $r \ge 0$,
i.e., $\beta_L^{x_0} \preceq \beta_M^{f(x_0)}$.
It follows that $\beta_X \preceq \beta_Y$.
\end{proof}

An illustration of Corollary \ref{exgrowthXtoY}
has already been given in Example \ref{excoarse}(10):
there does not exist any coarse embedding of a non-abelian free group of finite rank
into a Euclidean space.
\index{Free group}
Compare with Example \ref{Hilbertnotqifggroup}.
\par
The hypothesis on $X$ being large-scale geodesic cannot be omitted.
Indeed, consider on $\R$ the usual metric $d$, and the metric
$d_{\ln}$ defined by $d_{\ln}(x,x') = \ln( 1 + \vert x-x' \vert)$,
as in Example \ref{excoarse}(5).
Then $(R,d_{\ln})$ is not large-scale geodesic,
is of exponential growth, and the identity map
$(R,d_{\ln}) \longrightarrow (R,d)$ is a metric coarse equivalence
with a space of linear growth.

\begin{exe}
\label{growthR2H2}
(1)
For every $m \ge 1$,  the Euclidean space $\R^m$ contains $\Z^m$ as a metric lattice,
and is of polynomial growth:
\begin{equation*}
\beta_{\R^m}(r) \, \simeq \, \beta_{\Z^m} (r) \, \simeq \, r^m .
\end{equation*}
For $m \ne m'$, Proposition \ref{propgrowthXtoY} implies that
$\R^m$ and $\R^{m'}$ are not quasi-isometric.

\vskip.2cm

(2)
For very $k \ge 2$, the free group $F_k$ on $k$ generators
is of exponential growth, as it follows from a straightforward computation.
By Proposition \ref{propgrowthXtoY},  every finitely generated group
containing a non-abelian free subgroup is of exponential growth.
\index{Free subgroup}
In particular, consider an integer $n \ge 2$,
the hyperboic space $\mathbf H^n$,
a closed Riemannian $n$-manifold that is of constant curvature $-1$,
i.e., that is covered by $\mathbf H^n$,
and the fundamental group $\Gamma = \pi_1(M)$;
since $\Gamma$ has non-abelian free subgroups, 
$\Gamma$ is of exponential growth. 
It follows that $\mathbf H^n$ is of exponential growth:
\begin{equation*}
\beta_{\mathbf H^n}(r)  \, \simeq \, \beta_{\Gamma}(r) \, \simeq \, e^r  .
\end{equation*}
See \cite{Miln--68}.
\par

For $m \ge 1$ and $n \ge 2$, again by Proposition \ref{propgrowthXtoY}, 
the spaces $\R^m$ and $\mathbf H^n$ are not quasi-isometric.
Moreover, there does not exist any quasi-isometric embedding 
of $\R^2$ into $\mathbf H^n$;
this can be viewed as a particular case of facts
 cited below in Remark \ref{PropertyQinvqi}(1).
\par

Dually, there does not exist any quasi-isometric embedding 
of $\mathbf H^2$ into $\R^m$.
Though it is not explicitly stated in \cite{GhHa--90}, 
this follows easily from Theorem 6 of Chapter 5 there.

\vskip.2cm

(3)
For $n, n' \ge 2$, $n \ne n'$, it can be shown that $\mathbf H^n$ and $\mathbf H^{n'}$
are not quasi-isometric.
One proof involves \emph{Gromov boundaries},
for which we refer to \cite{GhHa--90}.
Another proof consists in computing the asymptotic dimension of $\mathbf H^n$,
which is $n$;
see Definition \ref{defasymptoticdimension} 
and Example \ref{ex_asymptoticdimension_n}.
\index{Asymptotic dimension! $\mathrm{asdim}(X)$ of a pseudo-metric space $X$}

\vskip.2cm

(4)
The spaces $\R^m$ and $\mathbf H^n$ are geodesic,
and therefore large-scale geodesic (see Definition \ref{defcoarselyconn}).
Since two distinct of them are not quasi-isometric, as just seen,
Proposition \ref{ceXYqi}(2) implies that they are
not coarsely equivalent.

\vskip.2cm

(5)
A regular tree of valency $d \ge 3$ is of exponential growth.
A regular tree of valency $2$ is of polynomial growth, indeed of linear growth.
\index{Tree}
\end{exe}

\begin{exe}
\label{Hilbertnotqifggroup}
Let $X$ be either a regular tree of infinite valency 
or an infinite-dimensional Hilbert space,
and $G$ a $\sigma$-compact LC-group.
There does not exist any coarse embedding of $X$ in $G$;
this is an immediate consequence of 
Proposition \ref{propgrowthXtoY} and \ref{growth=growth},
and Example \ref{homogeneousinfinitetree+Hilbert}.
\index{Tree} \index{Hilbert space}
In particular, $X$ is not quasi-isometric to any compactly generated LC-group.
\end{exe}

\begin{rem}
\label{refpourcroissance}
Early articles on growth of Riemannian manifolds, finitely generated groups, 
and compactly generated groups, include
\cite{Efre--53}, \cite{Svar--55}, \cite{Miln--68}, \cite{Guiv--70, Guiv--71, Guiv--73}, 
\cite{Jenk--73}, \cite{Grom--81b}, and \cite{Lose--87, Lose--01}.
We would like to quote \cite{Nekr--98} on growth of metric spaces,
and the book \cite{Mann--12} on growth of
finitely generated groups.
\par

In the particular case of finitely generated groups,
the subject of growth has special features.
One is that the growth function of such a group $\Gamma$
is \textbf{submultiplicative}: 
$\beta_\Gamma(m+n) \le \beta_\Gamma(m)\beta_\Gamma(n)$
for all $m,n \in \N$.
\index{Submultiplicative function}
Another one is the existence of
finitely generated groups of \textbf{intermediate growth}, 
that is with growth function $\beta$
such that $r^n \precnsim \beta(r) \precnsim e^r$ for all $n \ge 1$. 
Indeed, it can be shown that
finitely generated groups achieve uncountably many
growth functions distinct from each other \cite{Grig--84};
in particular, there are uncountably many quasi-isometry classes
of finitely generated groups.
\index{Intermediate growth|textbf}
\index{Growth! intermediate growth|textbf}
\end{rem}

On metrizable LC-groups, the Haar measure provides a way
-- and a very standard way -- to define growth functions.
To compare it with the way exposed above, 
we anticipate on Chapter \ref{chap_gpspmspaces}.
\index{Growth function! of an LC-group|textbf}
\par

More precisely, let $G$ be a $\sigma$-compact LC-group.
Choose a left-invariant Haar measure $\mu$ on $G$
and an adapted pseudo-metric $d$ on $G$
(see \S~\ref{adaptedpseudometric}).
Assume that, moreover, $d$ is measurable on $G$;
this is always possible (e.g.~by choosing $d$ continuous).
Note that, when $G$ is compactly generated, 
word metrics are measurable whenever they are defined 
either by compact generating sets (which are upper semi-continuous)
or by relatively compact open generating sets (which are lower semi-continuous).
\par

When $d$ is measurable, balls are measurable.

\begin{defn}
\label{volumegrowthG}
Let $G$ be an LC-group, $\mu$ a left-invariant Haar measure on $G$, 
and $d$ a measurable adapted pseudo-metric on $G$. 
\par
For $r \ge 0$,
the \textbf{volume} of the ball $B_G^1(r)$ of radius $r$ in $G$ is defined by
\begin{equation*}
\operatorname{Vol}(B^1_G(r)) \, = \,
\int_{ \{g \in G \mid d(1,g) \le r\} } d\mu .
\end{equation*}
The function $v_{G,d,\mu}$ is defined by
\begin{equation*}
v_{G,d,\mu}(r) \, = \, \operatorname{Vol}(B^g_G(r)) .
\end{equation*}
\end{defn}

Concerning this definition, three remarks are in order.
\begin{enumerate}[noitemsep,label=(\alph*)]
\item\label{aDEvolumegrowthG}
Since $d$ is adapted, balls of large enough radius have non-empty interiors, 
and therefore positive $\mu$-volumes: $v_{G,d,\mu}(r) > 0$ for $r$ large enough.
\item\label{bDEvolumegrowthG}
Let $d,d'$ be two quasi-isometric adapted pseudo-metrics on $G$.
Then $v_{G,d,\mu}(\cdot)$ and $v_{G,d',\mu}(\cdot)$ are equivalent 
in the sense of Definition \ref{defgrowthclass}, because there exist constants $c>0, c_+ \ge 0$
such that 
$\{g \in G \mid d(1,g) \le r \} \subset \{g \in G \mid d'(1,g) \le cr+c_+ \}$
for all $r \ge 0$, and similarly for $d,d'$ exchanged.
\item\label{cDEvolumegrowthG}
The equivalence class of the function $v_{G,d,\mu}(\cdot)$ is independent of 
the choice of the Haar measure $\mu$, since the latter is unique up to a multiplicative constant.
\end{enumerate}
In the next proposition, equivalences of functions are in the sense
of Definition \ref{defgrowthclass}.

\begin{prop}
\label{growth=growth}
Let $G$ be a $\sigma$-compact LC-group,
$\mu$, $d$, and $v_{G,d,\mu}$  as above. 
\begin{enumerate}[noitemsep,label=(\arabic*)]
\item\label{1degrowth=growth}
The pseudo-metric space $(G,d)$ is uniformly coarsely proper.
In particular, its growth function $\beta_G$ is well-defined
(see \ref{defgrowthX}).
\item\label{2degrowth=growth}
The functions $\beta_G$ and $v_{G,d,\mu}$ are equivalent.
\item\label{3degrowth=growth}
Let $d'$ be another measurable adapted pseudo-metric on $G$.
If $d$ and $d'$ are quasi-isometric, 
the functions $v_{G,d,\mu}$ and $v_{G,d',\mu}$ are equivalent.
\end{enumerate}
\end{prop}

\begin{proof} 
Let $s > 0$ be large enough so that the ball $C := B_G^1(s)$
to be a neighbourhood of $1$ in $G$.
Let $c > 2s$.
By Proposition \ref{reformqi} and its proof,
there exists a $c$-metric lattice $L$ in $(G,d)$ containing $1$,
such that
\begin{equation*}
\inf \{ d(\ell, \ell') \mid \ell, \ell' \in L, \hskip.1cm \ell \ne \ell' \} \ge c > 2s
\hskip.5cm \text{and} \hskip.5cm
\sup \{ d(g,L) \mid g \in G \} < \infty .
\end{equation*}
The inclusion $(L,d) \subset (G,d)$ is a metric coarse equivalence,
indeed a quasi-isometry.
For \ref{1degrowth=growth}, it remains to check that $(L,d)$ is uniformly locally finite.
\par

Let $r \ge 0$ and $\ell \in L$.
Since the balls $\ell' C = B_G^{\ell'}(s)$, with $\ell' \in L$ and $d(\ell, \ell') \le r$,
are pairwise disjoint, and all inside $B_G^\ell(r + s)$, we have
\begin{equation}
\label{betaAndv}
\beta_L^\ell (r) v_{G, d, \mu} (s) 
\, \le \, 
v_{G, d, \mu} (r + s) ,
\end{equation}
and therefore
\begin{equation*}
\sup \{ \beta_L^\ell (r) \mid \ell \in L \}
\, \le \, 
v_{G, d, \mu} (r + s) / v_{G, d, \mu} (s)
\end{equation*} 
for all $r \ge 0$.
Claim \ref{1degrowth=growth} follows.
\par

On the one hand,  we have $\beta_L^\ell  \preceq  v_{G, d, \mu}$ by (\ref{betaAndv}).
On the other hand, 
with $B$ a ball in $G$ of centre $1$ and radius $R \ge \sup \{ d(g,L) \mid g \in G \}$,
the balls $\left(\ell B \right)_{\ell \in L}$ cover $G$.
Hence we have
\begin{equation*}
B_G^1(r) \, \subset \, \bigcup_{\ell \in L \cap B_G^1(r+R)} \ell B ,
\end{equation*}
and therefore
\begin{equation*}
v_{G, d, \mu}(r) \, \le \, \beta_L^1(r+R) v_{G, d, \mu}(R)
\end{equation*}
for all $r \ge 0$. Hence $v_{G, d, \mu} \preceq \beta_L^1$, 
and \ref{2degrowth=growth} follows.
\par

Claim \ref{3degrowth=growth}
is an immediate consequence of Remark \ref{bDEvolumegrowthG}
after Definition \ref{volumegrowthG}.
\end{proof}

\begin{rem}
\label{growthofGindepofdandmu}
Let $G$ be a compactly generated LC-group.
Let $d, \mu$ and $v_{G, d, \mu}$ be as in Definition \ref{volumegrowthG}.
Suppose moreover that $d$ is geodesically adapted, e.g.\ that $d$ is a word metric
with respect to a compact generating set of $G$
(we anticipate here on \S~\ref{geodesicallyadaptedpseudometricspaces}).
\par

Then any other measurable geodesically adapted pseudo-metric on $G$
is quasi-isometric to $d$, so that
the class of $v_{G,d,\mu}$,
in the sense of Definition \ref{defgrowthclass},
is independent of the choices of $d$ and $\mu$.
\end{rem}

\subsection{Amenability}
\label{subsection_3Dc}

\begin{defn}
\label{defamulf}
A uniformly locally finite pseudo-metric space $(D,d)$
is \textbf{ame\-nable} if,
\index{Amenable! uniformly locally finite pseudo-metric space|textbf}
for every $r \ge 0$ and $\varepsilon > 0$,
there exists a non-empty finite subset $F \subset D$ such that
\begin{equation*}
\frac{ \vert B_D^F(r) \vert }{ \vert F \vert } \, \le \, 1 + \varepsilon ,
\end{equation*}
where $B_D^F(r) = \{ x' \in D \mid d(F,x') \le r \}$.
\index{$bb$@$B^{F}_D(r)$, subset of points in $D$ at distance at most $r$ from $F$}
\end{defn}

In \ref{defamucp}, the definition will be extended to a larger class of spaces.

\begin{rem}
\label{remamulf}
(1)
If the condition of Definition \ref{defamulf} fails for one pair $(r, \varepsilon)$,
then it also fails for $(2r, 2\varepsilon + \varepsilon^2)$,
because $B_D^F(2r) \supset B_D^{B_D^F(r)}(r)$.
\par
It follows that, given $r_0 > 0$, a space $(D,d)$ as in Definition \ref{defamulf}
is amenable if and only if, 
for every $r \ge r_0$ and $\varepsilon > 0$,
there exists a non-empty finite subset $F \subset D$ such that
the inequality of the definition holds.

\vskip.2cm

(2)
From the negation of the condition of Definition \ref{defamulf},
we obtain:
\begin{enumerate}[noitemsep]
\item[]
A uniformly locally finite non-empty pseudo-metric space $(D,d)$
\item[]
is \textbf{non-amenable} 
if and only if it satisfies the following condition:
\item[]
there exists a constant $K > 0$ such that
$\vert \mathcal B_D^F(K) \vert \ge 2 \vert F \vert$
\item[]
for every non-empty finite subset $F$ of $D$.
\end{enumerate}
In particular: 

\vskip.2cm

(3)
Let $(D,d)$ be a uniformly locally finite non-empty pseudo-metric space;
assume that there exists a map $f : D \longrightarrow D$ such that
$\sup_{x \in D} d(x, f(x)) < \infty$ and
$\vert f^{-1}(x) \vert \ge 2$ for all $x \in D$;
then $(D,d)$ is non-amenable.

This is described in \cite[Lemma 6.17]{Grom--99} as
``the best means for showing that a group is non-amenable''.
\end{rem}

\begin{prop}
\label{propamulf}
Let $D, E$ be two uniformly locally finite pseudo-metric spaces.
Assume that $D$ and $E$ are coarsely equivalent.
\par
Then $D$ is amenable if and only if $E$ is amenable.
\end{prop}

\begin{proof}
If one of $D, E$ is empty, so is the other, 
and neither $D$ nor $E$ is amenable;
we assume now that $D, E$ are non-empty.
We assume moreover that $E$ is amenable, and we have to show that $D$ is amenable.
\par

Let $f : D \longrightarrow E$, $g : E \longrightarrow D$ be two metric coarse equivalences
and $c > 0$ a constant
such that
\begin{equation*}
d_D(g(f(x)), x) \le c \hskip.2cm \forall x \in D
\hskip.5cm \text{and} \hskip.5cm
d_E(f(g(y)), y) \le c \hskip.2cm \forall y \in E .
\end{equation*}
Let $\Phi$ be an upper control such that
\begin{equation*}
d_E (f(x), f(x')) \, \le \, \Phi( d_D(x, x')) \hskip.2cm \forall x,x' \in D .
\end{equation*}
There exist two constants $k, \ell > 0$ such that
\begin{equation*}
\vert f^{-1} (y) \vert \le k \hskip.2cm \forall y \in E
\hskip.5cm \text{and} \hskip.5cm
\vert g^{-1} (x) \vert \le \ell \hskip.2cm \forall x \in D .
\end{equation*}
Consider $r \ge c$ and $\varepsilon > 0$;
concerning the choice of $r$, see Remark \ref{remamulf}(1).
Since $E$ is amenable, 
there exists a non-empty finite subset $F' \subset E$ such that
\begin{equation*}
\left\vert B_E^{F'} ( \Phi(r) + c ) \right\vert \, \le \,
\left( 1 + \frac{\varepsilon}{k\ell } \right) \vert F' \vert .
\end{equation*}
\par

Define $F  = \{x \in D \mid \ d_E(f(x), F') \le c \}$.
We have $g(F') \subset F$, and consequently
\begin{equation*}
\vert F \vert \ge \vert g(F') \vert \ge \frac{1}{\ell} \vert F' \vert .
\end{equation*}
Let $x \in D$ be such that $x \notin F$ and $d_D(F, x) \le r$.
There exist 
\par
$x' \in F$ with $d_D(x', x) \le r$
and 
\par
$y \in F'$ with $d_E(f(x'), y) \le c$,
\par\noindent
so that $d_E(f(x), F') \le d_E(f(x), f(x')) + c \le \Phi(r) + c$,
and therefore
\begin{equation*}
f(x) \, \in \, B_E^{F'}(\Phi(r) + c) \smallsetminus B_E^{F'}(c)
\, \subset \,  B_E^{F'}(\Phi(r) + c) \smallsetminus F' .
\end{equation*}
Hence
\begin{equation*}
f \left( \{ x \in D \mid x \notin F \hskip.2cm \text{and} \hskip.2cm d_D(F,x) \le r \} \right)
\, \subset \,  B_E^{F'}(\Phi(r) + c) \smallsetminus F'
\end{equation*}
and
\begin{equation*}
\left\vert \{ x \in D \mid x \notin F \hskip.2cm \text{and} \hskip.2cm d_D(F,x) \le r \} \right\vert
\, \le \, \frac{\varepsilon}{\ell} \vert F' \vert 
\, \le \, \varepsilon \vert F \vert .
\end{equation*}
Since $B_D^F(r) = F \cup  \{ x \in D \mid x \notin F 
\hskip.2cm \text{and} \hskip.2cm d_D(F,x) \le r \}$,
it follows that
\begin{equation*}
\left\vert B_D^F(r) \right\vert \, \le \, (1 + \varepsilon) \vert F \vert
\end{equation*}
and this ends the proof.
\end{proof}
Proposition \ref{propamulf} justifies the following definition
(which appears for example in \cite{BlWe--92})
and makes the next proposition straightforward.

\begin{defn}
\label{defamucp}
A uniformly coarsely proper pseudo-metric space $(X,d)$
is \textbf{amenable} if
\index{Amenable! uniformly coarsely proper pseudo-metric space|textbf}
it is coarsely equivalent to
an amenable uniformly locally finite pseudo-metric space,
equivalently
if every uniformly locally finite pseudo-metric space
which is coarsely equivalent to $(X,d)$ is amenable
in the sense of Definition \ref{defamulf}.
\end{defn}

Note that, if $X$ is bounded and non-empty, 
then it is amenable, while the empty space is not amenable.

\begin{prop}
\label{propamucp}
For uniformly coarsely proper pseudo-metric spaces,
amenability is invariant by metric coarse equivalence.
\end{prop}

\begin{defn}
\label{boudaryInGraphes}
Let $X$ be a graph, with vertex set $X^0$.
Let $Y^0$ be a subset of $X^0$.
\par
The \textbf{subgraph of $X$ induced by $Y^0$}
is the subgraph with vertex set $Y^0$ and with edge set
those edges of $X$ with both their ends in $Y^0$.
\index{Induced subgraph}
\par
The \textbf{boundary}
\index{Boundary in a graph}
$\partial Y^0$ is defined by
\begin{equation*}
\partial Y^0 \, = \, \left\{ y \in Y^0 \hskip.1cm \Bigg\vert \hskip.1cm
\aligned
& \hskip1.5cm \text{there exists an edge in $X$}
\\
&\text{connecting $y$ to a vertex in $X$ outside $Y^0$}
\endaligned
\right\} .
\end{equation*}
\end{defn}

\begin{lem}
\label{cIsoperimetriqueArbres}
Let $T$ be a tree in which every vertex has degree at least $3$.
For every non-empty finite subset $U^0$ of the vertex set of $T$, we have 
\begin{equation*}
\vert \partial U^0 \vert \ge \frac{1}{2} \vert U^0 \vert .
\end{equation*}
\end{lem}

\begin{proof}
\emph{First case: the subgraph $U$ of $T$ induced by $U^0$ is a tree.}
Since the lemma is obvious when $\vert U^0 \vert = 1$, we can assume that
$n := \vert U^0 \vert \ge 2$.
For $i = 1,2$, denote by $n_i(U)$ the number of vertices in $U^0$
with degree $i$ inside $U$. 
\par

We claim that
\begin{equation*}
2 n_1 (U) + n_2(U) \, \ge \, n .
\end{equation*}
To prove the claim, we proceed by induction on $n$.
If $n=2$, we have $n_1(U) = 2$, $n_2(U) = 0$, and the claim is true.
Assume from now on that $n \ge 3$, and that the claim holds
for a subtree $V$ of $U$ containing all vertices of $U$
except one vertex, say $u$,  of degree $1$ in $U$;
denote by $v$ the vertex in $V$ which is adjacent to $u$.
We distinguish three cases:
\begin{enumerate}[noitemsep,label=(\arabic*)]
\item
If $v$ has degree $1$ in $V$, then $n_1(U) = n_1(V)$ and $n_2(U) = n_2(V)+1$.
\item
If $v$ has degree $2$ in $V$, then $n_1(U) = n_1(V)+1$ and $n_2(U) = n_2(V)-1$.
\item
If $v$ has degree $\ge 3$ in $V$, then $n_1(U) = n_1(V)+1$ and $n_2(U) = n_2(V)$.
\end{enumerate}
In each case:
\begin{equation*}
2 n_1(U) + n_2(U) \, \ge \,  2 n_1(V)+n_2(V)+1 
\, \overset{\text{hyp rec}}{\ge} \, 
(n-1) + 1 = n .
\end{equation*}
The claim follows.
\par
We have now
\begin{equation*}
\vert \partial U^0 \vert \, \ge \,
n_1(U) + n_2(U) \, \ge \, \frac{1}{2}\left( 2n_1(U) + n_2(U) \right) \, \ge \, \frac{1}{2} n
\, = \, \frac{1}{2} \vert U^0 \vert.
\end{equation*}

\vskip.2cm

\emph{General case: the subgraph $U$ of $T$ induced by $U^0$ is a forest.}
Let $U_1, \hdots, U_k$ be the trees of which the disjoint union is the forest $U$.
Then $U^0$ is the disjoint union of $U_1^0, \hdots, U_k^0$,
and $\partial U^0$ is the disjoint union of $\partial U_1^0, \hdots, \partial U_k^0$.
Hence
\begin{equation*}
\vert \partial U^0 \vert \, = \,  \sum_{j=1}^k \vert \partial U_j^0 \vert
\overset{\text{first case}}{\ge} 
\sum_{j=1}^k \frac{1}{2} \vert U_j^0 \vert \, = \,  \frac{1}{2} \vert \partial U^0 \vert ,
\end{equation*}
as was to be shown.
\end{proof}

\begin{exe}
\label{examenmetricspaces}
(1)
Let $D$ be a \emph{non-empty} locally finite pseudo-metric space, $x$ a base point in $X$,
and $\beta_D^x$ the corresponding growth function.
Assume that $D$ has \textbf{subexponential lower growth}, i.e., that
\index{Subexponential lower growth|textbf}
\index{Growth! subexponential lower growth|textbf}
\begin{equation*}
\liminf_{s \to \infty} \beta_D^x(s)^{1/s} \, = \, 1 .
\end{equation*}
Then $D$ is amenable.
\par

Indeed, a computation yields    
\begin{equation*}
\liminf_{s \to \infty} \frac{\beta_D^x(s+r)}{\beta_D^x(s)} \, = \, 1
\hskip.5cm \text{for all} \hskip.2cm r \ge 0 .
\end{equation*}
Hence, for every $r \ge 0$ and $\varepsilon > 0$, 
there exists $s \ge 0$ such that, for $E = B_D^x(s)$, we have
\begin{equation*}
\frac{ \vert B_D^E(r) \vert }{ \vert E \vert } \, \le \, 
\frac{ \beta_D^x(s+r) }{ \beta_D^x(s) } \, \le \,
1 + \varepsilon .
\end{equation*}

\vskip.2cm

(2) 
For every integer $m \ge 1$, the free abelian group $\Z^m$ is of polynomial growth,
and therefore amenable, by (1).
\index{Free abelian group}
The Euclidean space $\R^m$ contains $\Z^m$ as a metric lattice,
indeed as a uniformly locally finite metric lattice,
so that $\R^m$ is amenable.
\index{Euclidean space $\R^n$}

\vskip.2cm

(3) 
A tree in which every vertex has degree at least $3$ is non-amenable.
This follows from Lemma \ref{cIsoperimetriqueArbres}. 
\par
Since a non-abelian finitely generated free group has a Cayley graph
which is a regular tree of degree $4$ or more,
non-abelian free groups are non-amenable.
\index{Free group}
We repeat that ``non-amenable'' is meant here in the sense of Definition \ref{defamulf},
and involves a left-invariant proper metric on the free group;
but this coincides with 
a standard definition of amenability for discrete groups,
discussed in Section \ref{sectionamenableLCgroups}.
\index{Amenable! locally compact group}

\vskip.2cm

(4)
A nonempty subspace of an amenable metric space need not be amenable.
Consider for example a tree $D$ that is regular of degree $3$,
and the space $D_+$ obtained from the union of $D$ and a half-line $[0, \infty[$
by gluing some vertex of $D$ to $0$.
Then $D_+$ is amenable (compare with (2) for $m=1$),
but its subspace $D$ is not (see (3)).
\par
This contrasts with the fact that any subgroup of an amenable discrete group is amenable.

\vskip.2cm

(5)
For every integer $n \ge 2$, the hyperbolic space $\mathbf H^n$
is non-amenable.
This follows from (4) and the following standard fact:
there exist discrete groups of isometries $\Gamma$ of $\mathbf H^n$
such that, given $x \in \mathbf H^n$, 
the orbit map $\Gamma \longrightarrow \mathbf H^n,
\gamma \longmapsto \gamma(x)$
is a quasi-isometry.
Examples of such groups include fundamental groups
of closed hyperbolic manifolds of dimension $n$. 
\index{Hyperbolic space $\mathbf H^n$}
\end{exe}

\section{The coarse category}
\label{sectioncoarsecat}

We have defined in \S~\ref{coarselyLipschitzandlargescaleLipschitz} the metric coarse category, 
whose objects are pseudo-metric spaces. 
In this section, that is not used later, we introduce as a variation
a larger category, where no reference is made to any pseudo-metric.
It shows in particular that coarse maps between LC-groups
can be defined without reference to any pseudo-metric.

\begin{defn}
\label{uniformbornology}
A \textbf{uniform bornology} \index{Uniform bornology|textbf}
on a set $X$ is a family $\mathcal B$ of subsets of $X \times X$ such that
\begin{enumerate}[noitemsep]
\item[--]
the diagonal $\textnormal{diag}(X) = \{(x_1,x_2) \in X \times X \mid x_1 = x_2 \}$ 
is in $\mathcal B$,
\item[--]
if $C \subset B \subset X \times X$ and $B \in \mathcal B$, then $C \in \mathcal B$,
\item[--]
if $B,C \in \mathcal B$, then $B \cup C \in \mathcal B$,
\item[--]
if $B,C \in \mathcal B$, then $B \circ C \in \mathcal B$,
\item[--]
if $B \in \mathcal B$, then $B^{-1}  \in \mathcal B$,
\end{enumerate}
where 
\begin{enumerate}[noitemsep]
\item[]
$ B \circ C \, = \, \{ (x_1,x_3) \in \mathcal B \mid \exists \hskip.1cm x_2 \in X 
\hskip.2cm \text{with} \hskip.2cm (x_1, x_2) \in B
\hskip.2cm \text{and} \hskip.2cm (x_2, x_3) \in C \} $,
\item[]
$ B^{-1} \, = \, \{ (x_1, x_2) \in \mathcal B \mid (x_2, x_1) \in B \} $.
\end{enumerate}
The elements of $\mathcal B$ are called \textbf{entourages} 
\index{Entourages of a uniform bornology|textbf}
of the uniform bornology $\mathcal B$.
\par
For $j = 1,2$, let $X_j$ be a set and $\mathcal B_j$ a uniform bornology on $X_j$.
A \textbf{coarse map} 
\index{Coarse! map from one space with a uniform bornology to another|textbf}
from $(X_1, \mathcal B_1)$ to $(X_2, \mathcal B_2)$
is a map $f : X_1 \longrightarrow X_2$ 
such that $(f \times f)(B) \in \mathcal B_2$ for all $B \in \mathcal B_1$.
\end{defn}

A uniform bornology is a ``coarse structure'' in \cite{Roe--03}.

\begin{defn}
\label{coarsecat}
The \textbf{coarse category} \index{Coarse! category|textbf}
is the category whose objects are pairs $(X, \mathcal B)$, with $X$ a set 
and $\mathcal B$ a uniform bornology on $X$,
and the morphisms are defined below.
\par
For two objects $(X_1, \mathcal B_1)$ and $(X_2, \mathcal B_2)$ in this category,
two maps $f,f' : X_1 \longrightarrow X_2$ are \textbf{equivalent}, $f \sim f'$, 
\index{$aa$@$\sim$ various equivalence relations}
if the image of $(f,f') : X_1 \longrightarrow X_2 \times X_2$ is in $\mathcal B_2$.
Note that, if $f \sim f'$, then $f$ is a coarse map if and only if $f'$ is a coarse map.
\par
A \textbf{coarse equivalence} \index{Coarse! equivalence|textbf}
is a coarse map 
$f : (X_1, \mathcal B_1) \longrightarrow (X_2, \mathcal B_2)$
for which there exists a coarse map
$g : (X_2, \mathcal B_2) \longrightarrow (X_1, \mathcal B_1)$
such that $gf \sim \operatorname{id}_{X_1}$ and $gf \sim \operatorname{id}_{X_2}$.
\par
A \textbf{morphism} from $(X_1, \mathcal B_1)$ to $(X_2, \mathcal B_2)$ 
is an equivalence class of coarse maps.
It can be checked that, in the coarse category, isomorphisms are precisely
equivalence classes of coarse equivalences. 
\end{defn}

There is in \cite{Roe--93, Roe--03} a category 
of which the objects are metric spaces,
and the morphisms are coarse maps 
with an additional condition of properness
that \emph{is not} required above.

\begin{exe}
\label{examplebornology}
(1)
Let $(X,d)$ be a pseudo-metric space.
Then 
\begin{equation*}
\mathcal B_d := \left\{ B \subset X \times X \hskip.1cm \Big\vert \hskip.1cm 
\aligned
&\text{the function} \hskip.2cm (x_1, x_2) \mapsto d_X(x_1,x_2)
\\
&\text{from $B$ to $\R_+$ is bounded}
\endaligned
\right\}
\end{equation*}
is a uniform bornology on $X$.
\par
Let $(X,d)$ and $(X',d')$ be two pseudo-metric spaces
and $f : X \longrightarrow X'$ a map.
If $f$ is coarsely Lipschitz in the sense of Definition \ref{defcoarse}, 
then $f$ is coarse in the sense of Definition \ref{uniformbornology}.
It follows that there is a functor from the metric coarse category to the coarse category,
such that the first one can be viewed as a full subcategory of the second one.

\vskip.2cm

(2) 
Let $G$ be an LC-group. 
Let $q : G \times G \longrightarrow G$ 
denote the map defined by $q(g_1,g_2) = g_1^{-1}g_2$.
Then 
\begin{equation*}
\mathcal B_G := \left\{ B \subset G \times G \mid 
\overline{q(B)} \hskip.2cm \text{is compact in $G$} \right\}
\end{equation*}
is a uniform bornology on $G$.
\end{exe}

\begin{defn}
\label{defcoarselymetric}
A pair consisting of a set $Y$ and a uniform bornology $\mathcal C$ on $Y$
is \textbf{coarsely pseudo-metric} 
if $(Y, \mathcal C)$ is coarsely equivalent to a pseudo-metric space,
more precisely to a pair $(X, \mathcal B_d)$ as in Example \ref{examplebornology},
for some pseudo-metric space $(X,d)$.
\end{defn}

We have the simple characterization:

\begin{prop} 
\label{justforthefun}
Let $(X,\mathcal{B})$ be a set endowed with a uniform bornology. 
The following are equivalent:
\begin{enumerate}[noitemsep,label=(\roman*)]
\item\label{iDEjustforthefun}
$(X,\mathcal{B})$ is coarsely pseudo-metric,
\item\label{iiDEjustforthefun}
$(X,\mathcal{B})$ is coarsely equivalent to a pseudo-metric space,
\item\label{iiiDEjustforthefun}
$\mathcal{B} = \mathcal{B}_d$ for some pseudo-metric $d$ on $X$,
\item\label{ivDEjustforthefun}
$X \times X$ is a countable union of entourages.
\end{enumerate}
\end{prop}

\begin{proof} 
Trivially \ref{iiiDEjustforthefun} implies \ref{iDEjustforthefun} 
and \ref{iDEjustforthefun} implies \ref{iiDEjustforthefun}. 
\par

Assume that \ref{iiDEjustforthefun} holds.
There exist a pseudo-metric space $(Y,d)$ and coarse maps
$f : (X, \mathcal B) \longrightarrow (Y, \mathcal B_d)$,
$g : (Y, \mathcal B_d) \longrightarrow (X, \mathcal B)$
such that $gf \sim \operatorname{id}_X$ and $gf \sim \operatorname{id}_Y$.
Set $A = \{ (g(f(x)),x) \in X \times X \mid x \in X \}$; then $A \in \mathcal B$.
For $n \ge 0$, define $B_{d,n} = \{ (y_1,y_2) \in Y \times Y \mid d(y_1, y_2) \le n \} \in \mathcal B_d$.
Then $X \times X = \bigcup_{n \ge 0} A^{-1} \circ (g \times g)(B_{d,n}) \circ A$.
Hence \ref{ivDEjustforthefun} holds.
\par

Suppose \ref{ivDEjustforthefun}, 
$X \times X = \bigcup_{n \ge 0}  B_n$ with $B_n \in \mathcal{B}$,
and let us show \ref{iiiDEjustforthefun}. 
Upon replacing $B_n$ by $B_n \cup B_n^{-1} \cup \textnormal{diag}(X)$,
we can suppose $B_n$ flip-invariant and containing the diagonal. 
Define by induction $C_0 = \textnormal{diag}(X)$ 
and $C_n = B_n \cup \Big( \bigcup_{p+q=n} B_p \circ B_q \Big)$. 
Then a simple verification shows that the function $d(x,y)=\min\{n \mid (x,y)\in C_n\}$ 
is a pseudo-metric on $X$, and the identity $\mathcal{B} = \mathcal{B}_d$.
\end{proof}

\begin{cor}
\label{coarselypseudometric=sigmacompact}
Let $G$ be an LC-group.
Then $G$ is $\sigma$-compact
if and only if $(G, \mathcal B_G)$ is coarsely pseudo-metric
(notation of Example \ref{examplebornology}).
\end{cor}

\begin{proof}
This is a straightforward consequence of the equivalence of (i) and (iv)
in Proposition \ref{justforthefun}.
\end{proof}

It follows that, for $\sigma$-compact LC-groups,
it is equivalent to use the language of the coarse category
(Definition \ref{coarsecat})
or that of the metric coarse category
(Definition \ref{defcoarsecat}).

\begin{defn}
\label{defEconnected}
Let $X$ be a set and $E$ a subset of $X \times X$ 
containing the diagonal $\textnormal{diag}(X)$.
Recall that the \textbf{equivalence relation generated by $E$} 
\index{Equivalence relation! on $X$ generated by $E \subset X \times X$|textbf}
is the relation $R_E \subset X \times X$ defined by
$(x,y) \in R_E $ if there exists a finite sequence $(x_0 = y, x_1, \hdots, x_k = y)$
of points in $X$ such that 
$(x_{i-1},x_i) \in E \cup E^{-1}$ for $i = 1, \hdots, k$.
The set $X$ is \textbf{$E$-connected} if $R_E$ is the trivial equivalence relation.
\index{Connected, of a set $X$ with respect to $E \subset X \times X$
containing the diagonal|textbf}
The \textbf{trivial equivalence relation} \index{Trivial equivalence relation|textbf}
is the relation for which any two points are equivalent.
\end{defn}

\begin{defn}
\label{defboundedlygeneratedXB}
A pair consisting of a set $X$ and a uniform bornology $\mathcal B$ on $X$
is \textbf{boundedly connected} 
\index{Boundedly! boundedly connected bornology|textbf}
if there exists $B \in \mathcal B$
such that the equivalence relation generated by $B$ is trivial.
\end{defn}

\begin{prop}
\label{boundedlygenerated=compactlygenerated}
Let $G$ be an LC-group.
Then $G$ is compactly generated
if and only if $(G, \mathcal B_G)$ is boundedly connected
(notation of Example \ref{examplebornology}).
\end{prop}

\begin{proof}[Proof:] it is left as an exercise.
\end{proof}

\vskip.5cm

For yet another category well-suited for the coarse study of locally compact groups,
see Subsection 3.A of \cite{CoHa}.

\chapter{Groups as pseudo-metric spaces}
\label{chap_gpspmspaces}

\section[Adapted (pseudo-)metrics on $\sigma$-compact groups]
{Adapted (pseudo-)metrics and $\sigma$-compactness}
\label{adaptedpseudometric}

\begin{defn}
\label{metricadapted}
On a topological group $G$, 
a pseudo-metric $d$ is \textbf{adapted} 
\index{Adapted pseudo-metric|textbf}
if it is left-invariant, proper, and locally bounded;
it means that
$d$ is left-invariant, 
all balls $\{g\in G \mid d(1,g) \le R\}$ are relatively compact, 
and are neighbourhoods of $1$ for $R$ large enough
(Definition \ref{defproper2}). 
\end{defn}

Note that a topological group that has an adapted metric is locally compact.
For examples of adapted metrics, see \ref{wordmetric} and \ref{abundanceofd}. 
\index{Metrizable! group}

\begin{prop}[metric characterization of $\sigma$-compactness]
\label{existenceam}
Let $G$ be an LC-group.
The following four properties are equivalent:
\begin{enumerate}[noitemsep,label=(\roman*)]
\item\label{iDEexistenceam}
$G$ is $\sigma$-compact;
\index{Sigma-compact! LC-group}
\item\label{iiDEexistenceam}
there exists an adapted continuous pseudo-metric on $G$;
\item\label{iiiDEexistenceam}
there exists an adapted pseudo-metric on $G$;
\item\label{ivDEexistenceam}
there exists an adapted metric on $G$.
\end{enumerate}
\end{prop}

\begin{proof}
Let $G$ be a $\sigma$-compact LC-group.
There exists a compact normal subgroup $K$ of $G$
such that $G/K$ is metrizable (Kakutani-Kodaira Theorem \ref{KK}). 
There exists on $G/K$ a metric, say $\underline d$,
that is left-invariant, proper, and compatible
(Struble Theorem \ref{Struble}),
equivalently that is adapted and continuous
(Proposition \ref{contpropermet}).
Then $d : G \times G \longrightarrow \R_+$, 
$(g,h) \longmapsto \underline d(gK,hK)$,
is a pseudo-metric on $G$ which is adapted and continuous.
This shows Implication \ref{iDEexistenceam} $\Rightarrow$ \ref{iiDEexistenceam}.
\par
Implication \ref{iiDEexistenceam} $\Rightarrow$  \ref{iiiDEexistenceam} is trivial.
For \ref{iiiDEexistenceam} $\Rightarrow$ \ref{ivDEexistenceam}, 
let $d$ is an adapted pseudo-metric on $G$;
define $d_+$ by $d_+(g,h) = d(g,h) + 1$ if $g \ne h$ and $d_+(g,g) = 0$;
then $d_+$ is an adapted metric on $G$.
\par
We have already observed that \ref{ivDEexistenceam} implies \ref{iDEexistenceam}; 
see Remark  \ref{remonmetrics}(6).
\end{proof}

\begin{rem}
\label{remoncontadapmetrics}
(1)
Recall Struble Theorem \ref{Struble}, according to which
the following two properties are equivalent for an LC-group $G$:
\begin{enumerate}[noitemsep,label=(\roman*)]
\addtocounter{enumi}{4}
\item\label{vDEremoncontadapmetrics}
$G$ is second-countable;
\item\label{viDEremoncontadapmetrics}
there exists a left-invariant proper compatible metric on $G$
(i.e., there exists an adapted continuous metric on $G$, 
by Proposition \ref{contpropermet}).
\end{enumerate}

\vskip.2cm

(2)
To illustrate Properties \ref{iiDEexistenceam}, \ref{ivDEexistenceam}, 
and \ref{viDEremoncontadapmetrics} above, 
let us consider
a non-metrizable compact group $K$,
for example a direct product of 
uncountably many copies of a non-trivial compact group.
\index{Product of groups}
Then:
\begin{enumerate}[noitemsep,label=(\alph*)]
\item\label{aDEremoncontadapmetrics}
there exist adapted continuous pseudo-metrics on $K$,
for example that defined by $d(g,h) = 0$ for all $g,h \in K$;
\item\label{bDEremoncontadapmetrics}
there does not exist any adapted continuous metric on $K$,
because $K$ is not second-countable;
\item\label{cDEremoncontadapmetrics}
there exist adapted metrics on $K$,
for example that defined by $d(g,h) = 1$ for all $g,h \in K$, $g \ne h$.
\end{enumerate}
\index{Compact group! metric, pseudo-metric}

\vskip.2cm

(3)
Beware: as \ref{cDEremoncontadapmetrics} above shows:
\begin{itemize}
\item
\textbf{an LC-group with an adapted metric need not be metrizable.}
\end{itemize}
(This also holds for geodesically adapted metrics,
as defined below in \ref{geodadapted}.)
Moreover, we insist on the following facts: 
\begin{enumerate}[noitemsep,label=(\alph*)]
\addtocounter{enumi}{3}
\item\label{dDEremoncontadapmetrics}
an adapted (pseudo-)metric need not be continuous
(Remark \ref{remonmetrics}(1));
\item\label{eDEremoncontadapmetrics}
a topological group that has a proper locally bounded pseudo-metric
is locally compact and $\sigma$-compact (Remark \ref{remonmetrics}(6)); 
\item\label{fDEremoncontadapmetrics}
on a locally compact group, a metric is adapted and continuous
if and only if it is left-invariant, proper, and compatible
(Proposition \ref{contpropermet});
\item\label{gDEremoncontadapmetrics}
a left-invariant proper metric
need not be locally bounded (Remark \ref{remonwordmetric}(5)).
\end{enumerate}

\vskip.2cm

(4)
Let $G$ be an LC-group and $d$ an adapted pseudo-metric on $G$
(as for example in Proposition 4.A.2(iii)).
Set $K = \{ g \in G \mid d(1,g) = 0 \}$. Then:
\begin{enumerate}[noitemsep,label=(\alph*)]
\addtocounter{enumi}{7}
\item\label{hDEremoncontadapmetrics}
$K$ is a relatively compact subgroup of $G$;
\addtocounter{enumi}{1}
\item\label{jDEremoncontadapmetrics}
the diameter of $(G,d)$ is finite if and only if the group $G$ is compact.
\end{enumerate}
In particular, if $d$ is moreover continuous, 
$K$ is a compact subgroup of $G$.
\end{rem}

\begin{cor}
\label{CorollaryKK}
On every LC-group, there exists a continuous left-invariant pseudo-metric
with respect to which balls of radius at most $1$ are compact.
\end{cor}

\begin{proof}
Let $G$ be an LC-group, 
and $U$ a compactly generated open subgroup of $G$ (Proposition \ref{powersSincpgroup}).
\index{Open subgroup! compactly generated}
Let $d$ be an adapted continuous pseudo-metric on $U$ (Proposition \ref{existenceam});
upon replacing $d$ by $2d/(1+d)$, 
we can assume that $d(u_1,u_2) < 2$ for all $u_1, u_2 \in U$.
Define a pseudo-metric $d_G$ on $G$ by
$d_G(g_1, g_2) = d(1, g_1^{-1}g_2)$ if $g_1^{-1}g_2 \in U$ 
and $d_G(g_1,g_2) = 2$ otherwise.
The pseudo-metric $d_G$ is continuous, and its balls of radius $<2$ are compact.
\end{proof}

\begin{prop}[continuous homomorphisms of $\sigma$-compact LC-groups 
as coarse morphisms]
\label{morphisms_grps_spaces}
For $j=1,2$,
let $G_j$ be a $\sigma$-compact LC-group
and $d_j$ an adapted pseudo-metric on it. 
Let $f : G_1 \longrightarrow G_2$ be a continuous homomorphism.
Then:
\begin{enumerate}[noitemsep,label=(\arabic*)]
\item\label{1DEmorphisms_grps_spaces}
the map $f : (G_1, d_1) \longrightarrow (G_2, d_2)$ is coarsely Lipschitz;
\item\label{2DEmorphisms_grps_spaces}
$f$ is proper if and only if the map $f : (G_1, d_1) \longrightarrow (G_2, d_2)$ 
is coarsely expansive.
\end{enumerate}
\end{prop}

\begin{proof}
\ref{1DEmorphisms_grps_spaces} 
Let $R_1 > 0$.
The ball $B_1 := \{g \in G_1 \mid d_1(1,g) \le R_1 \}$
is relatively compact because $d_1$ is proper.
Its image $f(B_1)$ is relatively compact because $f$ is continuous.
There exists $R_2 > 0$ such that
$B_2 := \{g \in G_2 \mid d_2(1,g) \le R_2 \}$
contains $f(B_1)$ because $d_2$ is locally bounded 
(Remark  \ref{remonmetrics}(5)). 
By the first part of Proposition  \ref{reformulation_c_et_ce},
and because $d_1,d_2$ are left-invariant,
it follows that $f$ is coarsely Lipschitz.
\par

\ref{2DEmorphisms_grps_spaces} 
Suppose that $f$ is coarsely expansive.
Let $\Phi_-$ be a continuous lower control for $f$
(see Definition \ref{defuclc}).
Consider a compact subset $L$ of $G_2$.
There exists $R_2 > 0$ such that $L \subset B_2$,
where $B_2$ is defined as in the proof of \ref{1DEmorphisms_grps_spaces}.
Set $R_1 = \inf \{ R \ge 0 \mid \Phi_-(R) \ge R_2 \}$.
Since $\Phi_-(d_1(1,g)) \le d_2(1,f(g))$ for all $g \in G_1$,
we have $f^{-1}(L) \subset f^{-1}(B_2) \subset B_1$,
so that $f^{-1}(L)$ is compact. Hence $f$ is proper.
\par

We leave the checking of the converse implication to the reader.
\end{proof}

\begin{cor}[uniqueness of adapted pseudo-metrics up to metric coarse equivalence]
\label{2metricsce}
Let $G$ be a $\sigma$-compact LC-group,
$H$ a closed subgroup,
$d_G, d'_G$ two adapted pseudo-metrics on $G$,
and $d_H$ an adapted pseudo-metric on $H$.
\begin{enumerate}[noitemsep,label=(\arabic*)]
\item\label{1DE2metricsce}
The inclusion $(H,d_H)  \lhook\joinrel\relbar\joinrel\rightarrow (G,d_G)$
is a coarse embedding.
\item\label{2DE2metricsce}
The identity, viewed as a map $(G,d_G) \longrightarrow (G,d'_G)$,
is a metric coarse equivalence.
\end{enumerate}
\index{Coarse! embedding}
\end{cor}

\begin{proof}
The inclusion homomorphism $H  \lhook\joinrel\relbar\joinrel\rightarrow G$ 
is continuous and proper.
Hence Claim \ref{1DE2metricsce}
is a particular case of Proposition \ref{morphisms_grps_spaces}.
Claim \ref{2DE2metricsce} follows.
\end{proof}

In combination with Proposition \ref{re_asymptoticdimension_n} on asymptotic dimension, 
we deduce that $\mathrm{asdim}(G, d_G) = \mathrm{asdim}(G, d'_G)$,
and in particular we denote by $\mathrm{asdim}(G)$
this \textbf{asymptotic dimension} of the $\sigma$-compact LC-group $G$.
\index{Asymptotic dimension! $\mathrm{asdim}(G)$ 
of a $\sigma$-compact LC-group $G$|textbf}
Again, the same proposition yields:

\begin{cor}
\label{asymptoticdimsubggroup}
Let $G$ be a $\sigma$-compact LC-group and $H$ a closed subgroup. 
Then $\mathrm{asdim}(H) \le \mathrm{asdim}(G)$.
\end{cor}

Propositions \ref{existenceam} \& \ref{morphisms_grps_spaces} and
Corollary \ref{2metricsce} mark a first important step in our exposition.
They show:

\begin{miles}
\label{miles_sigmacompact}
\index{Milestone}
\index{Sigma-compact! LC-group, milestone}
\begin{enumerate}[noitemsep,label=(\arabic*)]
\item\label{1DEmiles_sigmacompact}
On every $\sigma$-compact LC-group,
\index{Adapted pseudo-metric}
there exists an adapted pseudo-metric
(indeed there exists an adapted metric)
that makes it an object in the metric coarse category,
well-defined up to  metric coarse equivalence.
\item\label{2DEmiles_sigmacompact}
Any continuous homomorphism between such groups
can be viewed as a coarsely Lipschitz map.
\end{enumerate}
\index{Metric coarse! equivalence}
\index{Coarsely! Lipschitz map}
\end{miles}

\begin{rem}
\label{PropertyPinvce}
Let $\mathcal P$ be a property of pseudo-metric spaces 
that is invariant by metric coarse equivalence.
\index{Property! of a pseudo-metric space invariant 
by metric coarse equivalence or by quasi-isometries}
By \ref{miles_sigmacompact},
it makes sense to write that
a $\sigma$-compact LC-group \textbf{has Property $\mathcal P$}
\index{Property! of a $\sigma$-compact LC-group invariant by metric coarse equivalence}
when the pseudo-metric space $(G,d)$ has Property $\mathcal P$
for $d$ an adapted pseudo-metric on~$G$.
Examples of such properties include:
\index{Geometric property}

\vskip.2cm

(1)
Coarse connectedness (Proposition \ref{invariance_cg_lsg}), related to
compact generation for groups (Proposition \ref{metric_char_cg}).

\vskip.2cm

(2)
Coarse simple connectedness (Proposition \ref{coarse1conninvbycoarseeq}), 
related to
compact presentation for groups (Corollary \ref{qinv}).

\vskip.2cm

(3)
Asymptotic dimension
(Proposition \ref{re_asymptoticdimension_n}),
and in particular asymptotic dimension $0$, related to
local ellipticity for groups (Proposition \ref{equivalentlocalelliptic}),

\vskip.2cm

(4)
Being coarsely embeddable in some kind of spaces,
such as Hilbert spaces, or uniformly convex Banach spaces.
\index{Hilbert space}

\vskip.2cm

(5)
G. Yu's \textbf{Property A}, \index{Property! A of Yu}
for \emph{uniformly locally finite metric spaces}.
This definition has been useful in relation with the Baum-Connes conjecture.
We refer to \cite{Yu--00} for its definition
and to \cite[Proposition 4.2]{Tu--01} for its invariance by coarse equivalence.
There is a relation to (4): a uniformly locally finite metric space that has Property A
can be coarsely embedded into a Hilbert space \cite[Theorem 2.2]{Yu--00};
the converse does not hold \cite{ArGS--12}.

\vskip.2cm

(6)
For $\sigma$-compact LC-groups,
the property of being both amenable and unimodular is invariant 
by metric coarse equivalence.
See Section \ref{sectionamenableLCgroups}.

\vskip.2cm

(7)
This section has shown how to associate 
to an appropriate LC-group $G$ a well-defined object in the metric coarse category.
This has been extended by C.~Rosendal to other metrizable groups, 
not necessarily locally compact;
we refer to \cite{Rose--a, Rose--b}.
\end{rem}

\begin{exe}
(1)
Let us define adapted metrics on the countable abelian groups $\Q / \Z$ and $\Q$.
\par
Every $q \in \Q$ can be written in a unique way as $q = \frac{a}{n!}$
with $a \in \Z$, $n \in \N$, $n \ge 1$, and $n$ minimal;
for example: $7 = \frac{7}{1!}$, $2/5 = \frac{48}{5!}$, and $0 = \frac{0}{1!}$.
Set $\delta (q) = n-1$. Observe that 
$\{q \in \Q \mid \delta(q) = 0 \} = \Z$.
\par
For $x,x' \in \Q / \Z$, choose representatives $q,q' \in \Q$ of $x,x'$ 
and set $d_{\Q / \Z}(x,x') = \delta(q'-q)$.
We leave it to the reader to check that $d_{\Q / \Z}$ 
is an adapted metric on $\Q / \Z$,
and that $\{ x \in \Q / \Z \mid d_{\Q / \Z}(0,x) \le n-1 \} = \big(\frac{1}{n!}\Z\big) /\Z$ for all $n \ge 1$.
\par
For $q,q' \in \Q$, set $d_\Q(q,q') = d_{\Q / \Z}(\pi(q), \pi(q')) + \vert q'-q \vert$,
where $\pi$ denotes the canonical projection of $\Q$ onto $\Q / \Z$.
Then $d_\Q$ is an adapted metric on $\Q$.

\vskip.2cm

(2)
For any prime $p$, a similar definition provides an adapted metric $d$
on the Pr\"ufer group $\Z [1/p]/\Z$.
For $q \in \Z[1/p]$  written as $q = \frac{a}{p^k}$
with $a \in \Z$, $k \in \N$, and $k$ minimal,
set $\delta(q) = k$. Observe that
$\{q \in \Z[1/p] \mid \delta(q) = 0 \} = \Z$.
For $x,x' \in \Z[1/p] / \Z$, choose representatives $q,q' \in \Z[1/p]$
and set $d(x,x') = \delta(q'-q)$.

\vskip.2cm

(3)
Let  $\Gamma$ be a countable group.
By the HNN embedding theorem, already quoted in Remark \ref{problem},
there exist a group $\Delta$ generated by a finite subset $S$
(which can be chosen to consist of two elements)
and an injective homomorphism $j : \Gamma \longrightarrow \Delta$.
If $d_S$ is the word metric on $\Delta$ with respect to $S$
(see Definition \ref{geodadapted} below),
the distance $d$ defined on $\Gamma$ by
$d(g,h) = d_S(j(g), j(h))$ is adapted.

\vskip.2cm

(4)
Let $\K$ be a local field and $\vert \cdot \vert_{\K}$
an absolute value; 
see Example \ref{panoramalocalfield} and Remark \ref{onabsolutevalues}.
Then $d : \K \times \K \longrightarrow \R, \hskip.2cm (x,x') \longmapsto \vert x'-x \vert$
is an adapted metric on the additive group of $\K$.
\end{exe}

\section[Pseudo-metrics on compactly generated groups]
{Geodesically adapted (pseudo-)metrics and compact generation}
\label{geodesicallyadaptedpseudometricspaces}

\begin{defn}
\label{geodadapted}
On a topological group $G$, 
a pseudo-metric $d$ is \textbf{geodesically adapted} 
\index{Geodesically adapted pseudo-metric|textbf}
if it is adapted and large-scale geodesic 
(see Definitions \ref{metricadapted} and \ref{defcoarselyconn}). 
\end{defn}

\begin{defn}
\label{wordmetric}
Let $G$ be a group
and $S$ a generating set of $G$.
The \textbf{word metric} 
\index{Word metric|textbf}
defined by $S$ on $G$ is defined by
\begin{equation*}
d_S(g,h) = \min \left\{ 
n \ge 0 \hskip.1cm \Bigg\vert \hskip.1cm
\aligned
& \exists \hskip.2cm s_1, \hdots, s_n \in S \cup S^{-1} 
\\
&\text{such that} \hskip.2cm
g^{-1}h = s_1 \cdots s_n 
\endaligned
\right\} .
\end{equation*}
The corresponding \textbf{word length}, or \textbf{$S$-length},
\index{Word length|textbf}
is defined by
\begin{equation*}
\ell_S \, : \, G \longrightarrow \N , \hskip.5cm
g \longmapsto d_S(e,g) .
\end{equation*}
\par

Occasionally, if $\pi : F_S \longrightarrow G$ is the canonical projection
associated to the generating set $S$, we denote also by $\ell_S$
the word length $F_S \longrightarrow \N$.
Note that $\ell_S(g) =  \min \{ \ell_S(w) \mid w \in \pi^{-1}(g) \}$
for all $g \in G$.
\end{defn}

\begin{rem}
\label{remonwordmetric}
(1)
Let $G$ be a group and $S$ a generating set of $G$.
The word metric $d_S$ is left-invariant and $1$-geodesic;
in particular, it is large-scale geodesic.
\par
Suppose moreover that $G$ is locally compact and $S$ compact.
Then $d_S$ is geodesically adapted on $G$.

\vskip.2cm

(2)
Poincar\'e has defined a word length 
in a finitely generated Fuchsian group \cite[Page 11]{Poin--82},
using ``exposant d'une substitution'' $g$ for our $\ell_S(g)$.
But this does not re-appear later in the article.
\par
Word metrics were later used more systematically by Max Dehn, 
see \cite{Dehn--11} and \cite[Pages 130 and 143]{DeSt--87}.
Let $\Gamma$ be the fundamental group of a closed surface
and $S$ a finite generating set of $\Gamma$.
Using the word length $\ell_S$ on $\Gamma$, 
shown to be quasi-isometric to a metric coming from hyperbolic geometry,
Dehn has established that the ``conjugacy problem'' is solvable in $\Gamma$.

\vskip.2cm

(3)
An adapted metric need not be geodesically adapted.
For example, let $G$ be a compactly generated LC-group
and $d$ a geodesically adapted metric on $G$ (Proposition \ref{geodad+wordmetric}).
Then $\sqrt d$ is an adapted metric on $G$,
and $\sqrt d$ is geodesically adapted if and only if the diameter of $(G,d)$ is finite,
if and only if the group $G$ is compact.
The same holds for $\ln (1+d)$.

\vskip.2cm

(4)
In the circle group $\R / \Z$,
there exists a generating set $T$ such that the word metric $d_T$
makes $\R / \Z$ a metric space of infinite diameter.
Indeed, there exists a (non-continuous) group isomorphism 
$u : \mathbf{R}/\mathbf{Z} \longrightarrow \mathbf{R}\times\mathbf{Q}/\mathbf{Z}$. 
Then the word length in $\mathbf{R}/\mathbf{Z}$ 
with respect to the generating subset $u^{-1}([-1,1] \times \mathbf{Q}/\mathbf{Z})$ 
is unbounded, proper and not locally bounded.

\vskip.2cm

(5)
On a topological group, a word metric with respect to a compact generating subset
need not be locally bounded.
For example, let $F$ be a dense subgroup of $\SO(3)$
\index{Orthogonal group! $\textnormal{O}(n)$, $\SO (n)$}
that is free of rank $2$, 
endowed with the topology induced by that of $\SO(3)$;
note that this topological group is not locally compact.
Let $S = \{s_1, s_2\}$ be a free basis of $F$.
The word metric $d_S$ on $F$ is not locally bounded.
\par
\index{Free group}
\index{Free subgroup}

Indeed, let $U$ be a neighbourhood of $1$ in $F$.
Consider a sequence $\{\gamma_n\}_{n \ge 1}$ of pairwise distinct points in $F$
converging to $1$ in $\SO(3)$.
For every $k \ge 1$, the $d_S$-diameter of $\{\gamma_n\}_{n \ge k}$ is infinite.
Since this set is contained in $U$ for $k$ large enough,
the $d_S$-diameter of $U$ is infinite.
Hence $d_S$ is not locally bounded.

\vskip.2cm

(6)
Consider a positive integer $k$, the free group $F_k$, and a group word $w \in F_k$.
For a group $G$, the corresponding \textbf{verbal map} is defined as
\index{Verbal map|textbf}
$$
\Phi_w : G^k \longrightarrow G, \hskip.1cm 
(g_1, \hdots, g_k) \longmapsto w(g_1, \hdots, g_k) .
$$
Beware that $\Phi_w$ is generally not a homomorphism.
Denote by $G_w$ the image of $\Phi_w$.
Let $\langle G_w \rangle$ denote the subgroup of $G$ generated by $G_w$,
and $d_w$ the $G_w$-word metric on $\langle G_w \rangle$.
For an integer $n \ge 0$, the word $w$ has
\textbf{width at most $n$ in $G$} if
\index{Width of a word in a group|textbf}
$d_w(1,g) \le n$ for all $g \in \langle G_w \rangle$;
thus, the width of $w$ in $G$ 
is the diameter of the metric space $(\langle G_w \rangle, d_w)$.
The word $w$ is \textbf{silly} if either $w = 1 \in F_k$
or $G_w \cup G_w^{-1} = G$ for every group $G$.
For example, $x \in F_1$ and $yxyzy \in F_3$ are silly, 
and $x^{-1}y^{-1}xy \in F_2$ is non-silly.
\par
The question of finite width of words in groups has received a lot of attention.
Rhemthulla has shown that a non-silly word has infinite width
in every non-abelian free group  \cite{Rhem--68}; see also \cite{Sega--09}.

\vskip.2cm

(7)
If we specify the verbal map of the previous paragraph to the special case of 
$w = x^{-1}y^{-1}xy$, 
the distance $d_w(1,g)$ is the \textbf{commutator length}
\index{Commutator! length|textbf}
of $g \in G_w = [G,G]$ (notation of Definition \ref{defcommnilp}).
Let $c_n$ denote the supremum of the commutator lengths of elements in $\SL_n(\Z)$.
Then $c_n$ is finite for all $n \ge 3$,  $c_2 = \infty$, and $c_3 \le 43$ \cite{Newm--87}.
Commutator lengths are nicely discussed in \cite{Bava--91}. 
\par
In 1951, Ore conjectured that, in a non-abelian finite simple group $G$,
every element is a commutator, i.e., the commutator width of $G$ is $1$.
The conjecture has recently been established in \cite{LOSP--10}.

\vskip.2cm

(8) 
Let $G$ be an LC-group and $H$ a closed subgroup.
Assume that $G$ has an adapted metric $d$;
the restriction of $d$ to $H$ is an adapted metric on $H$.
\par
In particular, let $H$ be a countable group.
By the HNN embedding theorem (see Remark \ref{problem}),
\index{HNN! embedding theorem}
\index{Theorem! HNN embedding}
there exists a finitely generated group $G$ (indeed a $2$-generated group)
containing $H$ as a subgroup.
Let $S \subset G$ be a finite generating set;
the corresponding word metric $d_S$ is geodesically adapted on $G$.
The restriction $d$ of $d_S$ to $H$ is an adapted metric,
nevertheless it need not be geodesically adapted;
for example, by Proposition \ref{geodad+wordmetric} below,
$d$ is not geodesically adapted when $H$ is not finitely generated.

\vskip.2cm

(9) 
Let $G$ be a group, $S$ a generating set, 
and $w : S \cup S^{-1} \longrightarrow \R_+^\times$ a \textbf{weight function}, 
\index{Weight on a generating set}
such that $w(s^{-1}) = w(s)$ for all $s \in S \cup S^{-1}$.
As a variation on Definition \ref{wordmetric}, 
one has a ``weighted word metric'' on $G$ by
\begin{equation*}
d_{S,w}(g,h) = \inf \left\{ 
\sum_{i=1}^n w(s_i) \hskip.1cm \Bigg\vert \hskip.1cm
\aligned
& \exists \hskip.2cm s_1, \hdots, s_n \in S \cup S^{-1} 
\\
&\text{such that} \hskip.2cm
g^{-1}h = s_1 \cdots s_n 
\endaligned
\right\} .
\end{equation*}
This has been useful, for example to estimate word growth
of finitely generated groups \cite{Bart--98}.
\end{rem}

\begin{prop}
\label{geodad+wordmetric}
Let $G$ be a topological group.
\begin{enumerate}[noitemsep,label=(\arabic*)]
\item\label{1DEgeodad+wordmetric}
If there exists an adapted pseudo-metric $d$ on $G$
such that the pseudo-metric space $(G,d)$ is coarsely connected
(in particular if $d$ is geodesically adapted),
then $G$ is locally compact and compactly generated.
\item\label{2DEgeodad+wordmetric}
Assume that $G$ is locally compact and has a compact generating set $S$. 
The word metric $d_S$ is geodesically adapted.
\item\label{3DEgeodad+wordmetric}
Let $G$ be a compactly generated LC-group, 
$S, S'$ two compact generating sets,
and $d,d'$ the corresponding word metrics.
Then the identity $(G,d)  \overset{\operatorname{id}}{\longrightarrow} (G,d')$
is a bilipschitz equivalence.
\end{enumerate}
\end{prop}
\index{Bilipschitz! bilipschitz equivalence}

\begin{proof}
\ref{1DEgeodad+wordmetric}
Suppose that $d$ is a coarsely connected adapted pseudo-metric on $G$.
Let $c$ be a constant as in Definition \ref{defcoarselyconn}\ref{aDEdefcoarselyconn}.  
Then the closure of $\{ g \in G \mid d(1,g) \le c \}$
is a compact generating set of $G$. 
\par

\ref{2DEgeodad+wordmetric}
Let $n \ge 0$ be such that ${\widehat S}^n$ is a neighbourhood of $1$
(Proposition \ref{powersSincpgroup}).
Its $d_S$-diameter is bounded by $2n$;
hence $d_S$ is locally bounded.
As $d_S$ is clearly left-invariant, proper, and large-scale geodesic,
this establishes Claim \ref{2DEgeodad+wordmetric}.
\par

\ref{3DEgeodad+wordmetric}
By Proposition \ref{powersSincpgroup}(6), 
\begin{equation*}
c \, := \, \max \Big\{ \sup_{s \in S' \cup {S'}^{-1}} d(1,s) , \hskip.1cm
\sup_{s \in S \cup {S}^{-1}} d'(1,s) \Big\}
\end{equation*}
is a finite constant. We have $d'(g,h) \le c d(g,h)$ and $d(g,h) \le c d'(g,h)$
for all $g,h \in G$.
\end{proof}

\begin{defn}
\label{defDeltaNu}
Let $X$ be a set 
and $E \subset X \times X$ a symmetric subset containing the diagonal;
we suppose that $X$ is $E$-connected (Definition \ref{defEconnected}).
\index{Connected, of a set $X$ with respect to $E \subset X \times X$
containing the diagonal}
Define
$\nu_E : X \times X \longrightarrow \N$ by
\begin{equation*}
\nu_E(x,y) \, = \, \inf \left\{ n \ge 0 \hskip.1cm \Big\vert \hskip.1cm
\aligned
&\exists \hskip.1cm x=x_0, x_1, \hdots, x_n = y \hskip.2cm \text{in} \hskip.2cm X
\\
&\text{with} \hskip.2cm
(x_{i-1},x_i) \in E   \hskip.2cm \text{for} \hskip.2cm i = 1, \hdots, n 
\endaligned
\right\} 
\end{equation*}
(compare with Remark \ref{remcoarselygeod}).
Suppose that $X$ is moreover given with a pseudo-metric $d$.
Define $\delta_{E,d} : X \times X \longrightarrow \R_+$ by
\begin{equation*}
\delta_{E,d}(x,y) \, = \, \inf \left\{ \ell \ge 0 \hskip.1cm \Bigg\vert \hskip.1cm
\aligned
&\exists \hskip.1cm x=x_0, x_1, \hdots, x_n = y \hskip.2cm \text{in} \hskip.2cm X
\\
&\text{with} \hskip.2cm
(x_{i-1},x_i) \in E   \hskip.2cm \text{for} \hskip.2cm i = 1, \hdots, n 
\\
&\text{and} \hskip.2cm
\ell = \sum_{i=1}^n d(x_{i-1},x_i) 
\endaligned
\right\} .
\end{equation*}
Note that 
\begin{enumerate}[noitemsep,label=(\alph*)]
\item\label{aDEdefDeltaNu}
$\nu_E$ and $\delta_{E,d}$ take finite values (because $X$ is $E$-connected),
\item\label{bDEdefDeltaNu}
$\nu_E$ is a metric on $X$ which is $1$-geodesic (Definition \ref{defcoarselyconn}),
\item\label{cDEdefDeltaNu}
$\delta_{E,d}$ is a pseudo-metric on $X$,
\item\label{dDEdefDeltaNu}
$d(x,y) \le \delta_{E,d}(x,y)$ for all $x,y \in X$.
\end{enumerate}
\end{defn}

For example, if $X=G$ is a group generated 
by a symmetric set $S$ containing $1$, 
and if $E := \{ (g,h) \in G \times G \mid g^{-1}h \in S \}$,
then $\nu_E$ is the word metric $d_S$ 
of Definition \ref{wordmetric}.

\begin{defn}
\label{defcCcontrolled}
Let $(X,d)$ be a pseudo-metric space and $E \subset X \times X$
a symmetric subset containing the diagonal;
we assume that $X$ is $E$-connected.
Let $\nu_E, \delta_{E,d}$ be as in Definition \ref{defDeltaNu}, 
and $c,C$ be positive constants.
The triple $(X,d,E)$ is $\textbf{$(c,C)$-controlled}$ if
\index{Controlled triple $(X,d,E)$, 
where $(X,d)$ is a pseudo-metric space 
and $E$ in $X \times X$ contains the diagonal|textbf}
\begin{equation*}
B_X^x(c) \, \subset \, \{ y \in X \mid (x,y) \in E \}  \, \subset \, B_X^x(C)
\end{equation*}
for all $x \in X$
(recall that $B_X^x(c) = \{y \in X \mid d(x,y) \le c \}$,
as in $\S$~\ref{coarselyproperandgrowth}).
The right-hand side inclusion implies:
\begin{enumerate}[noitemsep,label=(\alph*)]
\addtocounter{enumi}{4}
\item\label{eDEdefDeltaNu}
$d(x,y) \le C \nu_E(x,y)$ for all $x,y \in X$.
\end{enumerate}
Say that a sequence $(x_0, x_1, \hdots, x_n)$ of points in $X$
with $(x_{i-1},x_i) \in E$ for $i = 1, \hdots, n$ is an \textbf{$E$-path of length $n$}.
Such an $E$-path is \textbf{$c$-good} if either $n \le 1$,
or $n \ge 2$ and $d(x_{i-1}, x_i) + d(x_i, x_{i+1}) > c$ for all $i = 1, \hdots, n-1$.
Then, because of the left-hand side inclusion above, 
\begin{enumerate}[noitemsep,label=(\alph*)]
\addtocounter{enumi}{5}
\item\label{fDEdefDeltaNu}
in the definitions of $\nu_E(x,y)$ and $\delta_{E,d}(x,y)$,
we can take the infima on sets of $c$-good $E$-paths only.
\end{enumerate}
\end{defn} 

The following lemma is essentially Lemma 4 of \cite{MoSW--02}.

\begin{lem}
\label{lemmeMoSW}
Let $(X,d)$ be a pseudo-metric space and $E \subset X \times X$
a symmetric subset containing the diagonal.
We assume that $X$ is $E$-connected
and that $(X,d,E)$ is $(c,C)$-controlled for some $c,C > 0$. Then
\begin{enumerate}[noitemsep,label=(\arabic*)]
\item\label{1DElemmeMoSW}
$\delta_{E,d} \, \le C \nu_E$.
\item\label{2DElemmeMoSW}
If $x,y \in X$ satisfy $\delta_{E,d}(x,y) < c$ or $d(x,y) < c$, 
then $\delta_{E,d}(x,y) = d(x,y)$.
\item\label{3DElemmeMoSW}
For all $x,y \in X$, we have
$\delta_{E,d}(x,y) \ge \frac{c}{2} \left( \nu_E (x,y) - 1 \right)$.
\item\label{4DElemmeMoSW}
$\nu_E$ and $\delta_{E,d}$ are quasi-isometric pseudo-metrics on $X$.
\item\label{5DElemmeMoSW}
If $X$ is a topological space and $d$ continuous,
then $\delta_{E,d}$ is continuous.
\end{enumerate}
\end{lem}

\begin{proof}
 \ref{1DElemmeMoSW} 
This inequality follows from the definitions of $\delta_{E,d}$ and $\nu_E$,
and from the controlled hypothesis.

\vskip.2cm

\ref{2DElemmeMoSW}
Let $x,y \in X$. For the proof of \ref{2DElemmeMoSW}, we consider several cases.
\par

Suppose first that $\delta_{E,d}(x,y) < c$. 
Choose $c' > 0$ with $\delta_{E,d}(x,y) < c' < c$. 
There exists a $c$-good $E$-path $(x_0=x, x_1, \hdots, x_n=y)$
with $\sum_{i=1}^n d(x_{i-1},x_i) < c'$.
Since the path is $c$-good, we cannot have $n \ge 2$.
Hence $n=1$ and $d(x,y) < c'$.
As this hold for all $c' \in ]\delta_{E,d}(x,y), c[$,
we have $d(x,y) \le \delta_{E,d}(x,y)$.
\par

Suppose now that $d(x,y) < c$.
Then $(x,y)$ is a $E$-path of length $1$, 
and $\delta_{E,d}(x,y) \le d(x,y)$.
By the previous case, we have also $d(x,y) \le \delta_{E,d}(x,y)$.
Hence $\delta_{E,d}(x,y) = d(x,y)$.
\par

Finally, if $\delta_{E,d}(x,y) < c$ again,
then $d(x,y) < c$ by the first case, 
and $\delta_{E,d}(x,y) = d(x,y)$ by the second case.

\vskip.2cm

\ref{3DElemmeMoSW}
Let $x,y \in E$, $x \ne y$. Choose $\ell > 0$ such that $\ell > \delta_{E,d}(x,y)$.
There exists a $c$-good $E$-path $(x_0=x, x_1, \hdots, x_n=y)$
with $\sum_{i=1}^n d(x_{i-1},x_i) < \ell$.
This path being $c$-good, we have
$\sum_{i=1}^n d(x_{i-1},x_i) >  \lfloor n/2 \rfloor c$,
hence $(n-1)/2 < \ell/c$.
Since $\nu_E(x,y) \le n$, we have also
$\nu_E(x,y) - 1 < 2\ell/c$.
As this hold for every $\ell > \delta_{E,d}(x,y)$,
we finally have $\nu_E(x,y) - 1 \le (2/c) \delta_{E,d}(x,y)$.

\vskip.2cm

\ref{4DElemmeMoSW} follows from \ref{1DElemmeMoSW} and \ref{3DElemmeMoSW},
and \ref{5DElemmeMoSW} follows from \ref{2DElemmeMoSW}.
\end{proof}

\begin{prop}[metric characterization of compact generation]
\label{metric_char_cg}
Let $G$ be a $\sigma$-compact LC-group.
Choose an adapted pseudo-metric $d$ on $G$
(see Proposition \ref{existenceam}).
The following properties are equivalent:
\begin{enumerate}[noitemsep,label=(\roman*)]
\item\label{iDEmetric_char_cg}
$G$ is compactly generated;
\item\label{iiDEmetric_char_cg}
the pseudo-metric space $(G,d)$ is coarsely connected;
\item\label{iiiDEmetric_char_cg}
the pseudo-metric space $(G,d)$ is coarsely geodesic;
\item\label{ivDEmetric_char_cg}
there exists a geodesically adapted pseudo-metric on $G$, 
say $d_{\textnormal{ga}}$
(note that, in this case,  the pseudo-metric space $(G,d_{\textnormal{ga}})$ 
is large-scale geodesic);
\item\label{vDEmetric_char_cg}
there exists a geodesically adapted metric on $G$;
\item\label{viDEmetric_char_cg}
there exists a geodesically adapted continuous pseudo-metric on $G$.
\end{enumerate}
In particular, among $\sigma$-compact LC-groups,
the property of compact generation is invariant by metric coarse equivalence.
\index{Property! of a $\sigma$-compact LC-group invariant by metric coarse equivalence}
\end{prop}

\noindent
\emph{Note.}
It is an old observation that compact generation
``can be considered as a very weak type of connectedness condition.''
The quotation is from the survey article \cite[Page 685]{Palm--78}.

\begin{proof}
Implications \ref{iiiDEmetric_char_cg} $\Rightarrow$ \ref{iiDEmetric_char_cg},
\ref{vDEmetric_char_cg} $\Rightarrow$ \ref{ivDEmetric_char_cg},
and \ref{viDEmetric_char_cg} $\Rightarrow$ \ref{ivDEmetric_char_cg}
follow from the definitions.
Implications \ref{iiDEmetric_char_cg} $\Rightarrow$ \ref{iDEmetric_char_cg} and
\ref{ivDEmetric_char_cg} $\Rightarrow$ \ref{iDEmetric_char_cg} $\Rightarrow$ \ref{vDEmetric_char_cg}
are contained in Proposition \ref{geodad+wordmetric}.
It suffices to check implications
\ref{iDEmetric_char_cg} $\Rightarrow$ \ref{iiiDEmetric_char_cg}
and
\ref{iDEmetric_char_cg} $\Rightarrow$ \ref{viDEmetric_char_cg}.
\par

For \ref{iDEmetric_char_cg} $\Rightarrow$ \ref{iiiDEmetric_char_cg}, 
choose a compact generating set $S$ of $G$,
and let $d_S$ be the corresponding word metric.
Since the space $(G,d_S)$ is large-scale geodesic
and the map $(G,d) \overset{\operatorname{id}}{\longrightarrow} (G,d_S)$
is a metric coarse equivalence 
(Corollary \ref{2metricsce}),
the space $(G,d)$ is coarsely geodesic 
(Proposition \ref{invariance_cg_lsg}).
\par

Let us show that \ref{iDEmetric_char_cg} implies \ref{viDEmetric_char_cg}.
By Corollary \ref{CorollaryKK},
there exists a continuous pseudo-metric $d$ on $G$
such that closed balls of radius at most $1$ are compact.
If $G$ is compactly generated, we can choose
a symmetric compact generating set $S$ of $G$ containing
the unit ball $B_G^1(1) = \{g \in G \mid d(1,g) \le 1 \}$.
Since $d$ is continuous and $S$ compact,
there exists a constant $C > 0$ such that $S \subset B_G^1(C)$.
Set $E = \{(g,h) \in G \times G \mid g^{-1}h \in S \}$.
Let $\nu_E$ and $\delta_{E,d}$ be as in Definition \ref{defDeltaNu}.
Observe that $\nu_E$ is the word metric $d_S$;
in particular, $\nu_E$ is geodesically adapted.
Lemma \ref{lemmeMoSW} shows that
the identity map 
$(G, \nu_E) \overset{\operatorname{id}}{\longrightarrow} (G, \delta_{E,\nu})$
is a quasi-isometry,
so that the pseudo-metric $\delta_{E,d}$ is also geodesically adapted,
and that $\delta_{E,d}$ is continuous, because $d$ is continuous.
Hence $\delta_{E,d}$ fulfills the conditions of \ref{viDEmetric_char_cg}.
\par

The last claim follows by Proposition \ref{invariance_cg_lsg}(1).
\end{proof}

Here is an analogue of the previous proposition for second-countable groups:

\begin{prop}
\label{2countcggroup}
Let $G$ be a second-countable LC-group.
The following properties are equivalent
\begin{enumerate}[noitemsep,label=(\roman*)]
\item\label{iDE2countcggroup}
$G$ is compactly generated;
\item\label{iiDE2countcggroup}
there exists a large-scale geodesic left-invariant proper compatible metric on~$G$.
\end{enumerate}
\end{prop}

\begin{proof}
By Theorem \ref{Struble}, 
there exists a left-invariant proper compatible metric $d$ on $G$.
\par
Assume that \ref{iDE2countcggroup} holds. 
Choose a symmetric compact generating set $S$ containing
$\{g \in G \mid d(1,g) \le 1 \}$,
and proceed as in the proof of \ref{iDEmetric_char_cg} $\Rightarrow$ \ref{viDEmetric_char_cg}
of Proposition \ref{metric_char_cg}.
Note that $\delta_{E,d}$ is now a metric (rather than a pseudo-metric),
by Lemma \ref{lemmeMoSW}\ref{2DElemmeMoSW}.
Since $\delta_{E,d}$ is proper 
(see \ref{dDEdefDeltaNu} in \ref{defDeltaNu}) 
and continuous, 
it is also compatible (Proposition \ref{contpropermet}).
This shows that \ref{iiDE2countcggroup} holds.

The converse implication is trivial.
\end{proof}

\begin{prop}[continuous homomorphisms of compactly generated LC-groups 
as large-scale morphisms]
\label{ceGG'qi}
For $j=1,2$,
let $G_j$ be a compactly generated LC-group
and $d_j$ a geodesically adapted pseudo-metric on it.
\begin{enumerate}[noitemsep,label=(\arabic*)]
\item
Let $f : G_1 \longrightarrow G_2$ be a continuous homomorphism.
Then $f : (G_1,d_1) \longrightarrow (G_2,d_2)$ is large-scale Lipschitz.
\item
Any metric coarse equivalence $(G_1, d_1) \longrightarrow (G_2, d_2)$ 
is a quasi-isometry.
\end{enumerate}
\end{prop}

\begin{proof}
In view of Proposition \ref{morphisms_grps_spaces}, 
this is a particular case of  Proposition \ref{ceXYqi}.
\end{proof}

\begin{cor}[Uniqueness of geodesically adapted pseudo-metrics up to quasi-isometry]
\label{uniqueness _uptoqi}
Let $G$ be a compactly generated LC-group
and let $d_1, d_2$ be two geodesically adapted pseudo-metrics on $G$.
\par

Then the identity, viewed as a map $(G,d_1) \longrightarrow (G,d_2)$,
is a quasi-isometry.
\end{cor}

\begin{rem}
\label{Heisenberg}
(1) In the situation of Proposition \ref{ceGG'qi}, 
the analogue of Part (2) in Proposition \ref{morphisms_grps_spaces} 
does not hold: a proper continuous homomorphism need not be large-scale expansive.
\par
Indeed, notation being as in Example \ref{coarselyexpansivenotlargescaleexpansive},
consider the discrete Heisenberg group $H(\Z)$
\index{Heisenberg group}
generated by $S := \{s,t\}$.
Its center, say $Z$, is the infinite cyclic group generated by $u$.
On the one hand, since $H(\Z) / Z \simeq \Z^2$, it is easy to check
that $d_S(1, s^k) = d_S(1, t^k) = \vert k \vert$ for all $k \in \Z$.
On the other hand,
a straightforward computation shows that $[s^k, t^k] = u^{k^2}$
for all $k \in \Z$, so that $d_S(1, u^{k^2}) \le 4 \vert k \vert$ for all $k \in \Z$.
Let $f$ denote the inclusion of $Z$ in $H(\Z)$;
observe that $f$ is proper (being injective).
Since 
\begin{equation*}
d_{\{u\}}(0, k^2) \, = \, k^2
\hskip.2cm \text{and} \hskip.2cm
d_S(f(0), f(k^2)) = d_S(1, u^{k^2}) \, \le \, 4 \vert k \vert
\hskip.2cm \text{for all} \hskip.2cm k \in \Z ,
\end{equation*}
$f : (Z, d_{\{u\}}) \longrightarrow (H(\Z), d_S)$ is not large-scale expansive.
\par
It can be shown that there exist constants $c_1, c_2 > 0$ such that
\begin{equation*}
c_1 \sqrt{\vert n \vert} \, \le \,  d_S(1, f(n)) \, \le \,  c_2 \sqrt{\vert n \vert}
\hskip.2cm \text{for all} \hskip.2cm n \in \Z .
\end{equation*}
In other words, the centre in $H(\Z)$ 
is \emph{distorted}, with \emph{quadratic relative growth}.
\par

Another example of distorted subgroup is that of the infinite cyclic subgroup
$\langle s \rangle$ of the Baumslag-Solitar group
\index{Baumslag-Solitar group}
$\langle s,t  \mid t^{-1}st = s^2 \rangle$,
with \emph{exponential relative growth}.
For distorsion, see \cite[Chapter 3]{Grom--93}.

\vskip.2cm

(2) In the situation of Corollary \ref{uniqueness _uptoqi},
the hypothesis of geodesic adaptedness cannot be omitted, 
as shown by the two adapted metrics defined on $\Z$
by $d_1(m,n) = \vert m-n \vert$ and $d_2(m,n) = \sqrt{ \vert m-n \vert}$.
(Compare with Remark \ref{remonwordmetric}(3).)
\end{rem}

Propositions \ref{metric_char_cg} \& \ref{ceGG'qi} and
Corollary \ref{uniqueness _uptoqi} 
mark a second important step in our exposition (compare with \ref{miles_sigmacompact}).
They show:

\begin{miles}
\label{miles_compactgen}
\index{Milestone}
\index{Compactly generated! LC-group, milestone}
\begin{enumerate}[noitemsep,label=(\arabic*)]
\item
On every compactly generated LC-group,
there exists a geo\-de\-sically adapted pseudo-metric 
\index{Geodesically adapted pseudo-metric}
(indeed there exists a geodesically adapted metric)
that makes it an object in the large-scale category, 
well-defined up to quasi-isometry.
\index{Quasi-isometry}
\item
Any continuous homomorphism between such groups 
can be viewed as a large-scale Lipschitz map.
\end{enumerate}
 \index{Large-scale! Lipschitz map}
\end{miles}

\begin{rem}
\label{PropertyQinvqi}
Let $\mathcal Q$ be a property of pseudo-metric spaces 
that is invariant by quasi-isometries.
\index{Property! of a pseudo-metric space invariant 
by metric coarse equivalence or by quasi-isometries}
By \ref{miles_compactgen},
it makes sense to write that
a compactly generated LC-group \textbf{has Property $\mathcal Q$}
\index{Property! of a compactly generated LC-group invariant by quasi-isometries}
when the pseudo-metric space $(G,d)$ has Property $\mathcal Q$
for $d$ a geodesically adapted pseudo-metric on $G$
(for example a word metric).
Let us indicate examples of such properties, 
that \emph{would not} fit in Remark \ref{PropertyPinvce}.
\index{Geometric property}

\vskip.2cm

(1)
For a pseudo-metric space $X$ and an integer $n \ge 1$, 
define \textbf{Property Flat($n$)} 
\begin{equation}
\text{there exists a quasi-isometric embedding of $\R^n$ into $X$,}
\tag{Flat($n$)}\label{Flat(n)}
\end{equation}
where $\R^n$ is viewed together with its Euclidean metric.
This property is obviously invariant by quasi-isometries,
but it is not invariant by metric coarse equivalence
(as we leave it to the reader to check).
\index{Property! Flat(n) of a pseudo-metric space|textbf}
\index{Quasi-isometric embedding}
\par

 For example, let $X$ be a Riemannian symmetric space of the non-compact type.
Let $r$ be the \emph{rank} of $X$, i.e., the largest integer for which $X$ has
a subspace isometric to $\R^r$.
As a consequence of results in \cite{EsFa--08},
the space $X$ has Property Flat$(n)$ if and only if $n \le r$.
In particular, for all $n \ge 2$, the hyperbolic space $\mathbf{H}^n$ 
and the space $\SL_n(\R)/\operatorname{SO}(n)$
have Property Flat$(n-1)$, and not Flat$(n)$.

\vskip.2cm

(2)
For uniformly coarsely proper pseudo-metric spaces,
the growth function is invariant by quasi-isometries
(Proposition \ref{propgrowthXtoY}).
Growth functions provide important quasi-isometry invariants
of compactly generated LC-groups;
see the references quoted in 
Remark \ref{refpourcroissance}-

\vskip.2cm

(3)
Remark \ref{GromovHyp} shows a property that is invariant by quasi-isometries 
\emph{for a subclass of pseudo-metric spaces only},
but nevertheless that makes sense for compactly generated LC-groups.

\vskip.2cm

(4)
As in Remark \ref{PropertyPinvce}(7), 
we refer to \cite{Rose--a, Rose--b} for a way to define a quasi-isometric type
of some non-LC metrizable groups.
\end{rem}

\begin{rem}[Gromov-hyperbolicity]
\label{GromovHyp}
The notions of Gromov-hyperbolic spaces and groups have had a deep influence
on geometric group theory \cite{Grom--87}. Let us recall the definition:
a metric space $(X,d)$ is \textbf{Gromov-hyperbolic} if
\begin{equation}
\aligned
& \text{there exists} \hskip.2cm \delta \ge 0 \hskip.2cm \text{such that} \hskip.2cm
\\
& (y \vert z)_x \, \ge \, \min \{ (y \vert w)_x, (w \vert z)_x \} - \delta
\hskip.5cm \text{for all} \hskip.2cm x,y,z,w \in X ,
\endaligned
\tag{GromHyp}\label{GromHyp}
\end{equation}
where the \textbf{Gromov product} \index{Gromov! product|textbf}
is defined by
\begin{equation*}
(y \vert z)_x \, = \,  \frac{1}{2} \big( d(x,y) + d(x,z) - d(y,z) \big) .
\end{equation*}
An LC-group $G$ is \textbf{Gromov-hyperbolic}
\index{Gromov! hyperbolic group|textbf}
if it is compactly generated and if, with a word metric,
it is a Gromov-hyperbolic metric space.
Equivalently an LC-group is hyperbolic if it
has a continuous proper cocompact isometric action
on some proper geodesic hyperbolic metric space
\cite[Corollary 2.6]{CCMT--15}.
\par

Gromov-hyperbolicity \emph{for geodesic metric spaces}
is invariant by quasi-isometries \cite[Chapter 5, \S~2]{GhHa--90}.
In particular, 
Gromov-hyperbolicity is well-suited for compactly generated LC-groups 
(and their Cayley graph), 
and is invariant by quasi-isometries of these.
But Gromov-hyperbolicity for arbitrary proper metric spaces 
is not invariant by quasi-isometries, as shown by the following example.
\par

Let $f$ be a function $\N \longrightarrow \R$;
denote its graph by  
\begin{equation*}
X_f \, := \, \{ (x,y) \in \R^2 \mid x \in \N , \hskip.1cm y = f(x) \} \, \subset \, \R^2 .
\end{equation*}
Consider $\N$ together with the usual metric, $d_\N(m,n) = \vert m-n \vert$,
and $X_f$ together with the $\ell^1$-metric, 
$d_X((x_1,y_1),(x_2,y_2)) = \vert x_1 - x_2\vert + \vert y_1 - y_2 \vert$.
Observe that $(\N,d_\N)$ is Gromov-hyperbolic; 
indeed, it satisfies Condition \eqref{GromHyp} with $\delta = 0$.
View $f$ as a bijection from $(\N,d_\N)$ onto $(X_f,d_X)$.
\par

Now let $(a_n)_{n \ge 1}$ be the sequence in $\mathbf{N}$ defined by 
$a_1 = 1$ and $a_{n+1} = a_n+2n$ for $n\ge 1$. 
Assume that $f$ is bilipschitz, and such that
$f(a_n)=0$ and $f(a_n+n)=n$ for all $n\ge 1$; 
a simple verification shows that there exists bilipschitz functions satisfying these conditions. 
Note that the bijection $f : (\N,d_\N) \longrightarrow (X_f,d_X)$ is a quasi-isometry.
\par

Let us check that $(X_f,d_X)$ is not Gromov-hyperbolic.
Set
\begin{equation*}
x=(a_n,0), \hskip.1cm y=(a_n+n,n), \hskip.1cm z=(a_{n+1},0), \hskip.1cm w=(a_{n+1}+n+1,n+1)
\hskip.1cm \in X_f .
\end{equation*}
A direct computation yields
\begin{equation*}
(y|z)_x=n, \hskip.1cm (y|w)_x=2n, \hskip.1cm (w|z)_x=2n. 
\end{equation*}
So $(y|z)_x-\min\{(y|w)_x,(w|z)_x\} = -n$,
and this implies that $X_f$ is not hyperbolic.
\end{rem}

\section{Actions of groups on pseudo-metric spaces}
\label{sectionactions}

\begin{defn}
\label{actions_mp_lb_cob}
Consider a topological group $G$,
a non-empty pseudo-metric space $(X,d_X)$,
and a (not necessarily continuous)
action $\alpha : G \times X \longrightarrow X, (g,x) \longmapsto gx$.
\index{Action|textbf}
For $x \in X$ and $R \ge 0$, denote by 
\par
$i_x : G \longrightarrow X, \hskip.1cm g \longmapsto gx$
the orbit map;
\par
$B^x_X(R)$ the ball $\{y \in X \mid d_X(x,y) \le R \}$ in $(X,d_X)$;
\par
$S_{x,R}$ the subset $\{ g \in G \mid d_X(x,g(x)) \le R \} 
= i_x^{-1}(B^x_X(R))$ of $G$.
\par\noindent
The action $\alpha$ is 
\begin{itemize}
\item[]
\textbf{faithful} \index{Faithful action|textbf}
if, for every $g \ne 1$ in $G$, there exists $x \in X$ with $gx \ne x$
(some authors use ``effective'' for ``faithful'');
\item[]
\textbf{metrically proper} \index{Metrically proper! action|textbf}
if $S_{x,R}$ is relatively compact for all $x \in X$ and $R \ge 0$;
\item[]
\textbf{locally bounded} 
\index{Locally bounded! action|textbf}
if, for every element $g \in G$ and bounded subset $B \subset X$,
there exists a neighbourhood $V$ of $g$ in $G$ such that $VB$ is bounded in $X$;
\item[]
\textbf{cobounded} \index{Cobounded! cobounded action|textbf}
if there exists a subset $F$ of $X$ of finite diameter
such that $\bigcup_{g \in G} gF = X$;
\item[]
\textbf{isometric} \index{Isometric action|textbf}
if $d_X(gx,gx') = d_X(x,x')$ for all $g \in G$ and $x,x' \in X$;
\item[]
\textbf{geometric} \index{Geometric action|textbf}
if it is isometric, cobounded, locally bounded, and metrically proper.
\end{itemize}
Suppose that $X$ is moreover an LC-space. The action $\alpha$ is
\begin{itemize}
\item[]
\textbf{proper} \index{Proper! action|textbf}
if the map $G \times X \longrightarrow X \times X, (g,x) \longmapsto (gx,x)$ is proper 
(Definition \ref{properandlb}),
equivalently if $\{g \in G \mid gL \cap L \ne \emptyset\}$
is relatively compact  for every compact subset $L$ of $X$;
\item[]
\textbf{cocompact} \index{Cocompact! action|textbf}
if there exists a compact subset $F$ of $X$
such that $\bigcup_{g \in G} gF = X$.
\end{itemize}
\end{defn}

\begin{rem}
\label{existactionimpliesLC}
Let $G$ be a topological group.
The mere existence of an action of some sort of $G$ on some space $X$
can have strong consequences on $G$. 
\par

Here is a first example:
if $G$ has a proper continuous action on a non-empty LC-space, 
then $G$ is an LC-group and the quotient space $G \backslash X$ is an LC-space
\cite[Page III.33]{BTG1-4}.
For other examples, see  \ref{ftggt}  -- \ref{remonftggt}.
\end{rem}

\begin{rem}
\label{firstexamplesgeometricactions}
(1)
In this book,
\begin{itemize}
\item
\textbf{actions of topological groups on topological spaces and metric spaces need not be continuous.}
\end{itemize}
This includes geometric actions.
Consider for example the regular action 
\begin{equation*}
\lambda \, : \, 
\R \times (\R, d_{\mathopen[-1,1\mathclose]}) \longrightarrow (\R, d_{\mathopen[-1,1\mathclose]}),
\hskip.2cm (t,x) \longmapsto \lambda_t(x) := t+x
\end{equation*}
of the additive group $\R$ 
on the discrete metric space $(\R, d_{\mathopen[-1,1\mathclose]})$,
where $d_{\mathopen[-1,1\mathclose]}$ denotes 
the word metric with respect to  $\mathopen[-1,1\mathclose]$.
This action is clearly not continuous, 
however this action is geometric.
\par
Similarly, actions of groups on Rips complexes 
(see Sections \ref{sectionRips}, \ref{sectionRipscomplex}, and \ref{DefExForCpGroups})
and Cayley graphs 
\index{Cayley graph}
are geometric and need not be continuous.

\vskip.2cm

(2)
Let $G$ be an LC-group and $X$ a non-empty pseudo-metric space.
An \emph{isometric} action of $G$ on $X$ is locally bounded if and only if $Kx$ is bounded
for every compact subset $K$ of $G$ and every $x \in X$.

\vskip.2cm

(3)
It is often sufficient to check local boundedness on generating sets.
More precisely, 
consider an LC-group $G$ with a compact generating subset $S$, 
a non-empty pseudo-metric space $X$, 
an isometric action $\alpha$ of $G$ on $X$, and a point $x_0 \in X$,
as in Definition \ref{actions_mp_lb_cob}.
Then $\alpha$ is locally bounded if and only if $Sx_0$ is bounded in $X$,
as we check now.
\par

First, suppose that $\alpha$ is locally bounded.
For all $s \in S$, there exists a neighbourhood $V_s$ of $s$ in $G$ 
such that $V_s x_0$ is bounded in $X$.
Since $S$ is compact, there exist $s_1, \hdots, s_n \in S$
such that $S \subset \bigcup_{i=1}^n V_{s_i}$.
Then $S x_0$, included in $ \bigcup_{i=1}^n V_{s_i} x_0$, is bounded.
\par

Suppose now that $S x_0$ is bounded.
Let $n \ge 1$ be such that $V := \widehat{S}^n$ is a neighbourhood of $1$ in $G$
(Proposition \ref{powersSincpgroup}).
Observe that $V x_0$ is bounded, indeed that $V x$ is bounded for all $x \in X$.
Given $(g,x) \in G \times X$, there exists $k \ge 1$
such that $g$ is in the neighbourhood $V^k$ of $1$ in $G$,
and $V^k x$ is bounded in $X$.
Hence $\alpha$ is locally bounded.

\vskip.2cm

(4)
Let $X$ be a topological space and $d_X$ a locally bounded pseudo-metric on $X$.
A continuous action of a topological group on $X$ 
is locally bounded with respect to $d_X$.
\par
But an arbitrary action need not be locally bounded.
For example, let $\R$ be given its usual metric and let
$\gamma : \R \longrightarrow \R$ 
be a non-continuous group endomorphism; then the isometric action 
$\R \times \R \longrightarrow \R, \hskip.2cm (z,z') \longmapsto \gamma(z) + z'$
is not locally bounded.

\vskip.2cm

(5) 
Consider again $G$, $X$, $\alpha$ and $x$ as in Definition \ref{actions_mp_lb_cob};
assume now that $G$ is locally compact and
$\alpha$ is isometric locally bounded.
Let $R \ge 0$; we claim that there exists $R' \ge R$ 
such that $d_X(x,g(x)) \le R'$ for all $g \in \overline{S_{x,R}}$.
\par

Indeed, let $K$ be a compact neighbourhood of $1$ in $G$.
There exists $r > 0$ such that $Kx$ lies inside the ball of centre $x$ and radius $r$.
Let $g \in \overline{S_{x,R}}$.
There exist $h \in S_{x,R}$ and $k \in K$ such that $g = hk$.
Then 
\begin{equation*}
d_X(x,g(x)) \le d_X(x,h(x)) + d_X(h(x),hk(x)) \le R + d_X(x,k(x)) \le R + r .
\end{equation*}
and the claim holds with $R' = R + r$.

\vskip.2cm

(6)
Let $G \times X \longrightarrow X$ be a continuous action
of an LC-group $G$ on a non-empty \emph{proper} metric space $X$.
\par 
The action is metrically proper if and only if it is proper.
\par
The action is cobounded if and only if it is cocompact.
\par\noindent
In particular, a continuous isometric action on $X$
that is proper and cocompact is geometric.
\end{rem}

\begin{exe}
\label{examplesgeometricactions}
(1) 
Let $G$ be a topological group and $X$ the one-point metric space.
The unique action of $G$ on $X$ is geometric if and only if $G$ is compact.

\vskip.2cm

(2)
Let $\mathbf E^1$ denote an affine Euclidean line,
with a standard metric.
\index{Affine line}
The group of isometries of $\mathbf E^1$ is a semidirect product
$\Isom(\mathbf E^1) \simeq \R \rtimes C_2$, where $C_2$ 
is the group of order $2$ generated by the central symmetry at $0$.
Up to conjugation inside the affine group $\R \rtimes \R^\times$,  
\index{Affine group! $\K \rtimes {\K}^\times$|textbf}
the group $\Isom(\mathbf E^1)$ has four cocompact closed subgroups:
\begin{enumerate}[noitemsep,label=(\alph*)]
\item
the group $\Z$ of all iterates of one non-trivial translation;
\item
the infinite dihedral group $\Z \rtimes C_2 = \Isom(\Z)$;
\index{Dihedral group|textbf}
\item
the group $\R$ of all translations;
\item
$\Isom(\mathbf E^1)$ itself, containing all translations and all half-turns.
\end{enumerate}
Each of these groups has a natural geometric action on $\mathbf E^1$.
More generally, let $G$ be a group acting 
continuously and geometrically on $\mathbf E^1$,
and let $\rho : G \longrightarrow \Isom(\mathbf E^1)$ be the corresponding homomorphism;
then $\rho(G)$ is conjugated inside $\R \rtimes \R^\times$ to one of (a) to (d) above;
moreover, $\rho(G)$ is isomorphic to 
the quotient of $G$ by its unique maximal compact normal subgroup.
\index{Maximal! compact normal subgroup}
\par

Let $G$ be an LC-group with two ends.
\index{Ends}
Then $G$ has a maximal compact normal subgroup $W$, 
and $G/W$ is isomorphic to one of the four groups (a) to (d) above.
See \cite[Corollary 4.D.2]{Corn}, building on \cite{Abel--74}.

\vskip.2cm

(3)
We anticipate here on $\S$~\ref{isometrygroups}:
for a proper metric space $X$, the isometry group $\Isom (X)$ is an LC-group
and its natural action on $X$ is continuous.
\par
Let $X$ be a proper metric space such that the action of $\Isom (X)$ on $X$ is cobounded.
Then the action of every cocompact closed subgroup of $\Isom (X)$ on $X$ is geometric.
\par

Here is a kind of converse.
Let $G$ be an LC-group 
acting continuously and geometrically on a proper metric space $X$,
and let $\rho : G \longrightarrow \Isom (X)$ be the corresponding homomorphism
(see Proposition \ref{GinIsom(X)}).
Then the kernel of $\rho$ is compact and
the image of $\rho$ is a cocompact closed subgroup of $\Isom(X)$.

\vskip.2cm

(4)
Let $G$ be a compactly generated LC-group containing a compact open subgroup,
and $X$ a Cayley-Abels graph for $G$ (Definition \ref{defCayleyAbels}).
\index{Cayley-Abels graph}
The action of $G$ on $X$ is geometric.
\end{exe}

\vskip.2cm

The first part of the following easy result is sometimes called (often for $G$ discrete only)
the \textbf{Schwarz-Milnor lemma (or theorem)}\footnote{The
mathematical physicist Albert Schwarz.
His name is also written \v Svarc in MathSciNet.
}, 
\index{Schwarz-Milnor lemma|see {Fundamental theorem}}
also the \textbf{fundamental theorem (or observation) of geometric group theory}
(see e.g.\ \cite{FaMo--02}).
The original articles are \cite{Svar--55} and \cite[see Lemma 2]{Miln--68}.

\begin{thm}
\index{Fundamental theorem (or observation) of geometric group theory|textbf}
\index{Theorem! fundamental theorem of geometric group theory}
\label{ftggt}
Let $G$ be an LC-group.
\vskip.2cm

(1)
Let $(X,d_X)$ be a non-empty 
pseudo-metric space and $x$ a point of $X$.
Suppose there exists a geometric action
\begin{equation*}
\alpha \,  : \, G \times X \longrightarrow X .
\end{equation*} 
Define $d_G : G \times G \longrightarrow \R_+$
by $d_G(g,g') = d_X(gx, g'x)$.
\par

Then $d_G$ is an adapted pseudo-metric on $G$,
and the orbit map 
\begin{equation*}
(G,d_G) \longrightarrow (X,d_X), \hskip.2cm g \longmapsto gx
\end{equation*}
is a quasi-isometry.
In particular, $G$ is $\sigma$-compact.
\par
If $(X,d_X)$ is moreover coarsely connected,
then $G$ is moreover compactly generated.
More precisely, for $R$ large enough,
$\overline{S_{x,R}}$ is a compact generating set of $G$
(notation $S_{x,R}$ as in Definition \ref{actions_mp_lb_cob}).

\vskip.2cm

(2)
Assume that $G$ is $\sigma$-compact, 
and choose an adapted pseudo-metric $d$ on $G$.
\par
Then the action by left-multiplication of the group $G$ on $(G,d)$ is geometric.
\end{thm}

\begin{proof}
(1)
The map $d_G$ is a pseudo-metric that is
left-invariant (because $\alpha$ is isometric),
proper (because $\alpha$ is metrically proper), 
and locally bounded (because $\alpha$ is so).
In other words, $d_G$ is an adapted pseudo-metric, 
and the first part of Claim (1) follows by Proposition \ref{existenceam}.
If $(X,d_X)$ is coarsely geodesic, then $G$ is compactly generated by
Proposition \ref{metric_char_cg}.
\par

Claim (2) is straightforward.
\end{proof}

\begin{cor}
\label{ahahah}
On a topological group $G$, the following conditions are equivalent:
\begin{enumerate}[noitemsep,label=(\roman*)]
\item\label{iDEahahah}
$G$ is locally compact and compactly generated;
\item\label{iiDEahahah}
there exists a geometric action of $G$ on
a non-empty coarsely geodesic pseudo-metric space;
\item\label{iiiDEahahah}
there exists a geometric action of $G$ on
a non-empty geodesic metric space;
\item\label{ivDEahahah}
there exists a geometric faithful action of $G$ on
a non-empty geodesic metric space.
\end{enumerate}
\end{cor}

\begin{proof}
\ref{iDEahahah} $\Rightarrow$ \ref{iiDEahahah}
Let $d$ be an adapted pseudo-metric on $G$.
Then $d$ is coarsely geodesic  
(Proposition \ref{metric_char_cg}),
so that \ref{iiDEahahah} holds by Theorem \ref{ftggt}.
\par

\ref{iiDEahahah} $\Rightarrow$ \ref{iiiDEahahah}
Assume that $G$ has an action 
on a non-empty coarsely geodesic pseudo-metric space $(X,d)$;
let $c > 0$ be such that pairs of points in $X$ can be joined by $c$-paths
(Definition \ref{defcoarselyconn}).
Let $(X_c,d_c)$ be the metric graph of Definition \ref{the new space Xc etc}.
\index{Graph associated to a space}
Recall from Lemma \ref{the new space Xc aso}
that the natural map $(X,d) \longrightarrow (\mathcal X_c, d_c)$
is a metric coarse equivalence,
and observe that the action of $G$ on $(X,d)$ has a natural extension
as an action on $(X_c,d_c)$.
The latter action satisfies Condition \ref{iiiDEahahah}.
\par

\ref{iiiDEahahah} $\Rightarrow$ \ref{ivDEahahah}
Let $\mu$ be a left-invariant Haar measure on $G$
and $B$ the unit ball in the Hilbert space $L^2(G,\mu)$.
\index{Hilbert space}
The metric $d_B$ defined on $B$ by $d_B(b,b') = \Vert b-b' \Vert$
is geodesic, and the natural action of $G$ on $L^2(G, \mu)$
induces a faithful continuous isometric action of $G$ on $(B, d_B)$.
Let $(X,d_X)$ be a $G$-space as in \ref{iiiDEahahah}.
Set $Y = X \times B$ and let $d_Y$ be a product metric,
say that defined by 
\begin{equation*}
d_Y((x,b),(x',b'))^2 \, = \,  d_X(x,x')^2 + d_B(b,b')^2 .
\end{equation*}
The diagonal action of $G$ on $(Y, d_Y)$ 
satisfies Condition \ref{ivDEahahah}.
\par

\ref{ivDEahahah} $\Rightarrow$ \ref{iDEahahah}
Consider an action of $G$ on a space $(Y,d_Y)$ 
that satisfies Condition (iv).
Choose a point $y_0 \in Y$ and define a pseudo-metric $d$  on $G$ by
$d(g,h) = d_Y(g(y_0),h(y_0))$.
Then $d$ is adapted and $(G,d)$ is coarsely connected,
so that $G$ is locally compact and compactly generated.
\end{proof}

\begin{rem}
\label{remonftggt}
(1)
Let $G$ be an LC-group. 
Assume there exists a non-empty pseudo-metric space $X$
on which $G$ has a geometric action.
In Theorem \ref{ftggt}, we have already written that $G$
is necessarily $\sigma$-compact.
Moreover, since the orbit map $(G,d_G) \longrightarrow (X,d_X)$ is a quasi-isometry,
$X$ is coarsely proper.
\par

Thus, for example, 
$X$ can be neither a homogenous tree of infinite valency,
nor an infinite-dimensional Hilbert space with the metric of the norm
\index{Tree} \index{Hilbert space}
(Example \ref{homogeneousinfinitetree+Hilbert}).

\vskip.2cm

(2)
It is known that a compactly generated LC-group always admits a
\emph{continuous} geometric action on a non-empty locally compact geodesic metric space
\cite[Proposition 2.1]{CCMT--15}.

\vskip.2cm

(3)
Theorem \ref{ftggt} and Corollary \ref{ahahah}
have analogues for compact presentation, 
see Theorem \ref{analogueof4.D.4} and Corollary \ref{analoguede4.D.5}.
\end{rem}

Claims (1') and (2') of the following standard proposition
provide many examples of geometric actions.

\begin{prop}[$G$-invariant metrics]
\label{alaKoszul}

(1)
Let $G$ be a $\sigma$-compact LC-group acting continuously and properly
on a locally compact metrizable space $X$ such that $G \backslash X$ is paracompact. 
Then $X$ admits a $G$-invariant compatible metric.
\index{Invariant! compatible metric on a locally compact proper $G$-space}
\par
(1')
In particular, if $K$ is a compact subgroup of $G$,
there exists a $G$-invariant compatible metric $d$ on the homogenous space $G/K$.
The natural action of $G$ on $(G/K, d)$ is geometric.

\vskip.2cm

(2)
Let $G$ be a Lie group acting  
differen\-tiably and properly
on a paracompact smooth manifold $X$.
Then $X$ admits a $G$-invariant Riemannian structure.
\index{Invariant! Riemannian metric on a locally compact proper $G$-space}
\par
(2')
In particular, if $K$ is a compact subgroup of $G$,
there exists a $G$-invariant Riemannian structure
on the homogenous manifold $G/K$, say $\mathcal R$.
If $d_{\mathcal R}$ denotes the associated metric,
the natural action of $G$ on $(G/K, d_{\mathcal R})$ is geometric.
\end{prop}

\begin{proof}[On the proof]
Since the action is proper,
averaging arguments can be used for the existence claims.
Details can be found, for example, in Pages 7 and 8 of \cite{Kosz--65}.
For (1), Koszul's argument and another argument 
are revisited in \cite[Section 6]{AbMN--11}.
For (2), see also \cite[Theorem 4.3.1]{Pala--61}.
\par

For the paracompactness hypothesis in (2), note that
$G \backslash X$ is paracompact as soon as
at least one of the following properties holds:
(a) the connected components of $X$ are open,
(b) $X$ is $\sigma$-compact.
Indeed, an LC-space is paracompact if and only if it is
a disjoint union of $\sigma$-compact spaces
\cite[Chap.\ XI, Th.\ 7.3]{Dugu--66}.
\end{proof}

\begin{exe}[abundance of metrics]
\label{abundanceofd}
On a given $\sigma$-compact LC-group 
[respectively on a given compactly generated LC-group], 
there can be an abundance of adapted metrics
[respectively of geodesically adapted metrics].
\index{Adapted pseudo-metric}
\index{Geodesically adapted pseudo-metric}
\par
Suppose here that $G$ is a connected Lie group.
Let $\mathfrak{g}$ denote its Lie algebra.

\vskip.2cm

(1)
A scalar product on $\mathfrak{g}$
defines a left-invariant Riemannian metric on $G$,
and in particular a proper compatible geodesically adapted metric on $G$
(as in Example \ref{pseudometriconGL}).

\vskip.2cm

(2)
Let $A$ be a bounded subset of $\mathfrak{g}$
such that every $g \in G$ can be represented as a product
$g = \exp (t_1a_1) \cdots \exp(t_ka_k)$ for some $a_1, \hdots, a_k \in A$
and $t_1,\hdots, t_k \in \R$.
Define the ``norm'' $\vert g \vert_A$ 
as the infimum of the numbers $\sum_{i=1}^k \vert t_i \vert$
over all such representations of~$g$,
and the adapted metric $d_A$ on $G$ by $d_A(g,h) = \vert g^{-1}h \vert_A$.
The set of metrics on $G$ of the form $d_A$ is known to coincide
with a set definable in terms of ``Carnot-Carath\'eodory-Finsler metrics'' on $G$
\cite{Nosk}.

\vskip.2cm

(3)
A compact generating set of $G$ defines a word metric on $G$,
that is geodesically adapted.
Let us evoke three kinds of such generating sets, still for $G$ connected.
\begin{enumerate}[noitemsep,label=(\alph*)]
\item
Every compact neighbourhood of $1$ in $G$ is generating, 
as already observed in Proposition \ref{powersSincpgroup}
\item
Let $P$ be a maximal subgroup of $G$;
suppose that $P$ is closed and connected,
let $T$ be a compact neighbourhood of $1$ in $P$,
and let $g \in G$, $g \notin P$; 
then $T \cup \{g\}$ is a compact generating set of $G$.
Example: $P$ is a maximal parabolic subgroup in a semisimple group $G$.
\item
Let $K$ be a maximal proper subgroup of $G$,
suppose that $K$ is compact,
and let $g \in G$, $g \notin K$; 
then $K \cup \{g\}$ is a compact generating set of $G$.
Example: $K =  \SO(n)$ in $G = \SL_n(\R)$, see \cite{Brau--65}.
\index{Orthogonal group! $\textnormal{O}(n)$, $\SO (n)$} 
\index{General linear group $\GL$! $\GL_n(\R)$, $\SL_n(\R)$}
(A variation on this last example appears in Example \ref{exSL_2}.)
\end{enumerate}

\vskip.2cm

(4)
Every continuous isometric action of $G$ on a connected Riemannian manifold $M$
with compact isotropy groups gives rise to geodesically adapted pseudo-metric on $G$,
as in Theorem \ref{ftggt}.
\end{exe}

\begin{exe}
\label{abundanceofdAM}
We continue in the spirit of the previous example,
but now for the group $G$ of $\K$-points
of a reductive algebraic group defined over some local field $\K$.
Here are three kinds of metrics on $G$:
\index{Reductive group}
\begin{enumerate}[noitemsep,label=(\alph*)]
\item
A word metric $d_{\text{word}}$ with respect to some compact generating set $S$ of $G$.
\item
A pseudo-metric $d$ defined by $d_{\text{geom}}(g,h) = d_X(gx_0, hx_o)$,
associated to the action of $G$ on the corresponding Bruhat-Tits building $X$,
where $d_X$ is a $G$-invariant metric on $X$
for which every apartment of $X$ is isometric to a Euclidean space,
and $x_0$ is some base point in $X$;
compare with Example \ref{pseudometriconGL}(2).
\item
A left-invariant pseudo-metric $d_{\text{op}}$ defined by
\begin{equation*}
d_{\text{op}}(1,g) \, = \, \sup \{ \vert \log \Vert g \Vert \vert ,  \vert \log \Vert g^{-1} \Vert \vert  \} ,
\end{equation*}
where $G$ is identified with a subgroup of $\GL_n(\K)$ 
by some faithful linear representation defined over $\K$,
and $\Vert \cdot \Vert$ denotes the operator norm of operators in $\K^n$;
compare with Example \ref{pseudometriconGL}(3).
\end{enumerate}
It is a part of the main result in \cite{Abel--04}
that the three metrics $d_{\text{word}}$, $d_{\text{geom}}$ and $d_{\text{op}}$
are quasi-isometric on $G$.
\end{exe}

\begin{prop}[$\sigma$-compactness and compact generation of cocompact closed subgroups]
\label{sigmac+compactgofcocompact}
\index{Hereditarity of various properties}
Let $G$ be an LC-group and $H$ a cocompact closed subgroup.
\index{Subgroup! cocompact closed}
\begin{enumerate}[noitemsep,label=(\arabic*)]
\item\label{1DEsigmac+compactgofcocompact}
Assume that $G$ is $\sigma$-compact,
or equivalently that $H$ is so (Proposition \ref{stababcd}).
Let $d$ be an adapted pseudo-metric on $G$ (Proposition \ref{existenceam}).
\par
Then, with respect to $d$, the inclusion of $H$ in $G$ is a metric coarse equivalence.
\item\label{2DEsigmac+compactgofcocompact}
$G$ is compactly generated if and only if $H$ is so.
\item\label{3DEsigmac+compactgofcocompact}
Assume that the conditions of (2) hold.
Let $d$ be a geodesically adapted pseudo-metric on $G$
(Proposition \ref{geodad+wordmetric}). 
\par
Then the inclusion of $H$ in $G$ is a quasi isometry.
\end{enumerate}
\end{prop}

\noindent
\emph{Note.} Claim \ref{2DEsigmac+compactgofcocompact} 
is contained in Proposition \ref{stababcd},
and also in \cite[chap.\ VII, $\S$~3, n$^o$ 2, lemme 3]{BInt7-8}.
Below, we provide another proof.
\par
In the literature, there are many other occurrences of various parts of the proposition;
for \ref{2DEsigmac+compactgofcocompact} and \ref{3DEsigmac+compactgofcocompact},
see for example \cite[Lemma 5]{MoSW--02}.

\begin{proof}
\ref{1DEsigmac+compactgofcocompact}
Since the inclusion of $H$ in $G$ is a continuous homomorphism,
the inclusion $j : (H,d) \lhook\joinrel\relbar\joinrel\rightarrow (G,d)$
is a coarse embedding, by Proposition \ref{morphisms_grps_spaces}.
Since $G/H$ is compact, it is essentially surjective.
\par

\ref{2DEsigmac+compactgofcocompact}
Suppose first that $G$ is compactly generated,
and let $d$ be a geodesically adapted metric on it.
The action by left multiplications of $H$ on $G$ is geometric,
and $H$ is compactly generated by Theorem \ref{ftggt}(1).
Suppose now that $H$ has a compact generating set $T$.
Let $U$ be a compact subset of $G$ such that $G = UH$
(Lemma \ref{KimagedeK}). 
Then $T \cup U$ is a compact generating set of $G$.
\par
There is another proof in \cite{MaSw--59}.

\ref{3DEsigmac+compactgofcocompact}
Since $d$ is geodesically adapted,
the metric coarse equivalence $H \longrightarrow G$
is a quasi-isometry by Proposition \ref{ceXYqi}.
\end{proof}

Similarly:  

\begin{prop}[$\sigma$-compactness and compact generation of 
quotients by compact normal subgroups]
\label{sigmac+compactgofquotients}
Let $G$ be an LC-group,
$K$ a compact normal subgroup,
$Q = G/K$ the quotient group,
and $\pi : G \longrightarrow Q$ the canonical projection.
Assume that $G$ is $\sigma$-compact,
or equivalently that $Q$ is so. 
\begin{enumerate}[noitemsep,label=(\arabic*)]
\item\label{1DEsigmac+compactgofquotients}
Let $d_G, d_Q$ be adapted pseudo-metrics on $G,Q$ respectively.
\par
Then $\pi : (G,d_G) \longrightarrow (Q,d_Q)$ is a metric coarse equivalence.
\item\label{2DEsigmac+compactgofquotients}
$G$ is compactly generated if and only if $Q$ is so.
\item\label{3DEsigmac+compactgofquotients}
Assume that the conditions of (2) hold.
\par
Let $d_G, d_Q$ be geodesically adapted pseudo-metrics on $G,Q$ respectively.
\par
Then $\pi : (G,d_G) \longrightarrow (Q,d_Q)$ is a quasi isometry.
\end{enumerate}
\end{prop}
\index{Subgroup! compact normal} 

\begin{rem}
\label{remcopci}
A homomorphism between LC-groups is a \textbf{copci} 
if it is \textbf{continuous proper with cocompact image}.
Consider two compactly generated LC-groups $G$ and $H$. 
A copci $\varphi : G \longrightarrow H$ factors as the composition
of a quotient map with compact kernel $G \longrightarrow G / \ker (\varphi)$
and an injective homomorphism $G / \ker (\varphi) \longrightarrow H$ 
that is a topological isomorphism onto a cocompact closed subgroup.
If follows from Propositions \ref{sigmac+compactgofcocompact} 
and \ref{sigmac+compactgofquotients} that a copci is a quasi-isometry.
\par

Two LC-groups $G$ and $H$ are \textbf{commable} 
\index{Commable LC-groups|textbf}
if there exist an integer $k$ 
and a sequence of copci homomorphisms
\begin{equation*}
G \, = \, G_0 \hskip.2cm  \frac{\phantom{aaa}}{\phantom{bbb}} \hskip.2cm 
G_1 \hskip.2cm  \frac{\phantom{aaa}}{\phantom{bbb}} \hskip.2cm 
G_2 \hskip.2cm  \frac{\phantom{aaa}}{\phantom{bbb}} \hskip.2cm 
\hdots  \hskip.2cm  \frac{\phantom{aaa}}{\phantom{bbb}} \hskip.2cm 
G_k = H
\end{equation*}
where each sign $\frac{\phantom{aaa}}{\phantom{bbb}}$
stands for either $\longrightarrow$ or $\longleftarrow$.
\par

If two compactly generated LC-groups are commable, they are quasi-isometric.
The converse does not hold: if $\Gamma_1, \Gamma_2$ are two cocompact lattices
in $\SL_2(\C)$ that are not commensurable,
then the free products $\Gamma_1 \Conv \Z$ and $\Gamma_2 \Conv \Z$
\index{Free product}
are quasi-isometric and are not commable 
(this is due to Mathieu Carette and Romain Tessera,
see \cite[Section 5.B]{Corn}).
\par

For copci homomorphisms and commable groups, see \cite{Corn--15}.
\end{rem}

\section{Local ellipticity}
\label{local_ellipticity}

\begin{defn}
\label{deflocallyelliptic}
An LC-group $G$ is \textbf{locally elliptic}
\index{Locally elliptic LC-group|textbf}
if every compact subset of $G$ is contained in a compact open subgroup.
\index{Open subgroup! compact}
\end{defn}

\begin{rem}
``Locally elliptic'' here was called ``topologically locally finite'' 
in \cite{Plat--66} and \cite{Capr--09},
and ``k-k group'' in \cite{Abel--74}.
\par

A  discrete group is locally elliptic if and only if it is locally finite
\index{Locally finite! group}
(as defined in Remark \ref{remsursuiteYamabe}).
\par

An LC-group that is both locally elliptic and compactly generated is compact.
\end{rem}

\begin{prop}
\label{equivalences_locell}
For an LC-group $G$, the following four properties are equivalent:
\begin{enumerate}[noitemsep,label=(\roman*)]
\item\label{iDEequivalences_locell}
every finite subset of $G$ is contained in a compact subgroup;
\item\label{iiDEequivalences_locell}
every compact subset of $G$ is contained in a compact subgroup;
\item\label{iiiDEequivalences_locell}
$G$ is locally elliptic;
\item\label{ivDEequivalences_locell}
$G_0$ is compact and $G/G_0$ is locally elliptic.
\end{enumerate}
\end{prop}

\begin{proof}
For the equivalence of \ref{iDEequivalences_locell} and \ref{iiDEequivalences_locell}, 
we refer to \cite[Lemma 2.3]{Capr--09}.
\par

Assume that $G$ has Property \ref{iiDEequivalences_locell}.
Let $L$ be a compact subset of $G$.
Choose a compact neighbourhood $B$ of $1$ in $G$.
By hypothesis, $L \cup B$ is contained in some compact subgroup $K$ of $G$.
Since $B \subset K$, the subgroup $K$ is also open.
Thus $G$ has Property \ref{iiiDEequivalences_locell}.
\par

[In case $G$ is totally disconnected, the equivalence between 
\ref{iiDEequivalences_locell} and \ref{iiiDEequivalences_locell}
follows also from Corollary \ref{vanDantzigCor}(2).]
\par

Implication \ref{ivDEequivalences_locell} $\Rightarrow$ \ref{iiiDEequivalences_locell} 
is straightforward. 
Conversely, if \ref{iiiDEequivalences_locell} holds, 
then clearly $G/G_0$ and $G_0$ are elliptic; 
since $G_0$ is also compactly generated (Proposition 2.C.3(2)), it has to be compact.
\end{proof}

\begin{prop}
\label{equivalentlocalelliptic}
Let $G$ be a $\sigma$-compact LC-group,
and let $d$ be an adapted pseudo-metric on $G$.
The following properties are equivalent:
\begin{enumerate}[noitemsep,label=(\roman*)]
\item\label{iDEequivalentlocalelliptic}
$G$ is locally elliptic,
\item\label{iiDEequivalentlocalelliptic}
$(G,d)$ has asymptotic dimension $0$.
\end{enumerate}
\index{Asymptotic dimension! $\mathrm{asdim}(G)$ 
of a $\sigma$-compact LC-group $G$}
\end{prop}

\begin{proof}
\ref{iDEequivalentlocalelliptic} $\Longrightarrow$ \ref{iiDEequivalentlocalelliptic}
Choose $r > 0$.
Since $d$ is proper, the ball $B^1_G(r)$ with respect to $d$ is compact.
Assume that $G$ is locally elliptic,
so that $B^1_G(r)$ is contained in a compact open subgroup $K$ of $G$.
On the one hand, 
for the partition of $G$ into left-cosets of the form $gK$, 
we have $\operatorname{diam}(gK) = \operatorname{diam}(K)$ for every $g \in G$.
On the other hand, for $g,g' \in G$, the inequality $d(gK, g'K) \le r$ implies $gK = g'K$;
equivalently, if $gK \ne g'K$, then $d(gK, g'K) > r$.
Hence $\operatorname{asdim}(G,d) = 0$.
\par

\ref{iiDEequivalentlocalelliptic} $\Longrightarrow$ \ref{iDEequivalentlocalelliptic}
Assume that $(G,d)$ has asymptotic dimension $0$.
Choose a compact subset $S$ of $G$;
we have to show that $S$ is contained in a compact open subgroup of $G$.
Upon enlarging it, we can assume that 
$S$ is a compact symmetric neighbourhood of $1$ in $G$.
Let $r > 0$ be such that $S \subset B^1_G(r)$.
By hypothesis, there exists a partition $G = \bigsqcup_{i \in I} X_i$
such that $\sup_{i \in I}\operatorname{diam}(X_i) < \infty$ 
and $d(X_i, X_j) \ge 2r$ for $i,j \in I$, $i \ne j$.

Let $g \in X_i, g' \in X_j$ and $s \in S$ such that $gs = g'$.
Then $d(g,g') = d(1,g^{-1}g') = d(1,s) \le r$, and therefore $i = j$.
It follows that, for all $i \in I$, we have $X_i S \subset X_i$,
and by iteration $X_i S^n \subset X_i$ for all $n \ge 0$.
Since $S$ is a symmetric neighbourhood of $1$,
the set $K := \bigcup_{n \ge 0}S^n$
is an open subgroup of $G$, containing $S$,
and $X_i K \subset K_i$ for all $i \in I$.
Let $i_0$ be the index in $I$ such that $1 \in X_{i_0}$,
and let $R$ denote the diameter of $X_{i_0}$;
then $K \subset X_{i_0} \subset B^1_G(R)$,
so that $K$ is a bounded subgroup.
Since $K$ is closed (being an open subgroup), $K$ is compact.
Hence $G$ is locally elliptic.
\end{proof}

\begin{cor}
Among $\sigma$-compact LC-groups,
local ellipticity is invariant by metric coarse equivalence.
\index{Property! of a $\sigma$-compact LC-group invariant by metric coarse equivalence}
\end{cor}

\begin{proof}
See Example \ref{exmasdim} and Proposition \ref{equivalentlocalelliptic}.
\end{proof}

\begin{prop}
\label{stabilitelocalelliptic}
Let $G$ be an LC-group, $L$ a closed subgroup, 
and $N$ a normal closed subgroup;
set $Q = G/N$.
\begin{enumerate}[noitemsep,label=(\arabic*)]
\item\label{1DEstabilitelocalelliptic}
If $G$ is locally elliptic, so are $L$ and $Q$.
\par
\item\label{2DEstabilitelocalelliptic}
If $N$ and $Q$ are locally elliptic, so is $G$.
\end{enumerate}
\end{prop}

\begin{proof}
\ref{1DEstabilitelocalelliptic}
From the definition, the claim for $L$ is clear.
Because of Lemma \ref{KimagedeK}, so is the claim for $Q$.
\par

\ref{2DEstabilitelocalelliptic}
Let $C$ be a compact subset of $G$.
It suffices to show that $C$ is contained in a locally elliptic open subgroup of $G$.
\par

Let $U$ be a compact neighbourhood of $1$ in $G$.
Let $H$ denote the subgroup of $G$ generated by $C \cup U$;
observe that $H$ is open and $\sigma$-compact.
Let $\pi : G\relbar\joinrel\twoheadrightarrow Q$ denote the canonical projection. 
Since $Q$ is locally elliptic, 
there exists a compact open subgroup $R$ of $Q$ containing $\pi(C \cup U)$.
Observe that $N \cap H$ is a normal and closed subgroup of $H$, 
and that $H/(N \cap H)$ is isomorphic to an open subgroup of $Q$
contained in $R$, hence that $H/(N \cap H)$ is compact.
Since $H$ is compactly generated,
the inclusion of $N \cap H$ in $H$ is a metric coarse equivalence 
(Proposition \ref{sigmac+compactgofcocompact}).
Since $N \cap H$ is locally elliptic by \ref{1DEstabilitelocalelliptic},
$H$ is locally elliptic by 
Propositions \ref{re_asymptoticdimension_n} and \ref{equivalentlocalelliptic}.
\end{proof}

\begin{exe}
\label{examplesoflocallyellipticgroups}
(1)
Finite groups and compact groups are clearly locally elliptic.

\vskip.2cm

(2)
Every locally finite discrete groups is locally elliptic.
Similarly, an LC-group that is an increasing union of compact open subgroups
is locally elliptic.
\par
In particular, for a prime $p$ and a positive integer $n$, 
the additive group $\Q_p^n$ is locally elliptic.
See Example \ref{localfieldscg}.
\index{$q$@$\Q_p$, field of $p$-adic numbers}

\vskip.2cm

(3)
An LC-group  containing $\Z$ as a discrete subgroup is not locally elliptic.
In particular, for a prime $p$ and an integer $n \ge 2$, 
the groups $\Q_p^\times$ and $\SL_n(\Q_p)$ are not locally elliptic.
\index{Special linear group $\SL$! $\SL_n(\Q_p)$}

\vskip.2cm

(4)
Let $A$ be a countable abelian group.
Choose a compatible metric $d$ on $A$.
Then  $\operatorname{asdim}(A,d) = 0$ 
if and only if  $A$ is locally elliptic,
if and only if $A$ is a torsion group
(Proposition \ref{equivalentlocalelliptic}).

\vskip.2cm  

(5)
Let $\K$ be a non-discrete locally compact field and $n \ge 2$ an integer. 
Let $\operatorname{Unip}_n(\K)$ denote the subgroup of $\GL_n(\K)$ of matrices 
$(u_{i,j})_{1 \le i,j \le n}$ with $u_{i,i} = 1$ and $u_{i,j} = 0$ whenever $i > j$;
it is an LC-group homeomorphic to $\K^{n(n-1)/2}$. 
\par

If $\K$ is $\R$ or $\C$, then $\operatorname{Unip}_n(\K)$ contains discrete infinite cyclic subgroups,
and therefore is not locally elliptic by (3).
\par

If $\K$ is a local field, then $\operatorname{Unip}_n(\K)$ is locally elliptic.
Indeed, consider an absolute value $\vert \cdot \vert$ on $\K$
(see Remark \ref{onabsolutevalues}) and,
for every $\alpha > 0$, the subset
${\operatorname U}_\alpha$ of matrices $(u_{i,j})_{1 \le i,j \le n}$ 
in $\operatorname{Unip}_n(\K)$
such that $\vert u_{i,j} \vert \le \alpha^{j-i}$ for all $i,j$ with $j > i$.
Then ${\operatorname U}_{\alpha}$ 
is a compact open subgroup of $\operatorname{Unip}_n(\K)$, and
$\operatorname{Unip}_n(\K) = 
\bigcup_{\alpha \in \R_+^\times} {\operatorname U}_{\alpha}$.
The local ellipticity of $\operatorname{Unip}_n(\K)$ follows.

\vskip.2cm

(6)
As a consequence of Proposition \ref{equivalences_locell}, 
a connected LC-group is locally elliptic if and only if it is compact.

\vskip.2cm

(7)
An LC-group $G$ possesses a unique maximal locally elliptic closed normal subgroup,
called the \textbf{locally elliptic radical} of $G$.
\index{Radical! Locally elliptic radical|textbf}
\index{Locally elliptic radical|textbf}
This was introduced by \cite{Plat--66}, who called it the locally finite radical;
see also \cite{Capr--09}; 
we avoid here (as in \cite{CCMT--15} 
and several subsequent articles) this confusing terminology.
\par
The quotient of $G$ by its locally elliptic radical
is an LC-group without non-trivial locally elliptic closed normal subgroup.

\vskip.2cm

(8)
Note that a locally elliptic LC-group is amenable
\index{Amenable! locally compact group}
(by F\o lner's Criterion \cite{Foln--55}) 
and unimodular
(as a union of open subgroups which are compact, and therefore unimodular).
\index{Unimodular group}
\end{exe}

\begin{prop}
\label{localellipticityintermsofaction}
For a $\sigma$-compact LC-group $G$ with adapted metric $d_G$, 
the following conditions are equivalent:
\begin{enumerate}[noitemsep,label=(\roman*)]
\item\label{iDElocalellipticityintermsofaction}
the group $G$ is locally elliptic;
\item\label{iiDElocalellipticityintermsofaction}
there exists a non-empty ultrametric space $(X,d)$
on which $G$ acts continuously, properly, and isometrically.
\item\label{iiiDElocalellipticityintermsofaction}
the pseudo-metric space $(G,d_G)$ is coarsely equivalent to an ultrametric space;
\item\label{ivDElocalellipticityintermsofaction}
the pseudo-metric space $(G,d_G)$ is coarsely ultrametric.
\end{enumerate}
\end{prop}
\index{Ultrametric! pseudo-metric}
\index{Coarsely! ultrametric pseudo-metric space}

\begin{proof}
\ref{iDElocalellipticityintermsofaction} $\Rightarrow$ \ref{iiDElocalellipticityintermsofaction}.
Let $K_0 \subset \cdots \subset K_n \subset K_{n+1} \subset \cdots$
be a nested sequence of compact subspaces of $G$
such that $\operatorname{int}(K_0) \ne \emptyset$,
and $K_n \subset \operatorname{int}(K_{n+1})$ for all $n \ge 0$,
and $G = \bigcup_{n \ge 0} K_n$; see Remark \ref{neednotbelc}(2).
For $n \ge 0$, let $G_n$ denote the subgroup of $G$ generated by $K_n$;
observe that $G_n$ is an open subgroup of $G$.
(Note that, here, $G_0$ \emph{is not} meant to be the connected component of $G$.)
\par
Assume now that $G$ is locally elliptic,
so that $G_n$ is a compact open subgroup of $G$ for all $n \ge 0$.
Define $d' : G \times G \longrightarrow \R_+$ by
$d'(g,h) =  \inf \{n \ge 0 \mid  g^{-1}h \in G_n \}$;
then $d'$ is a left-invariant pseudo-ultrametric on $G$.
Set $X = G/G_0 = G / \sim$, where ``$g \sim h$'' for $g,h \in G$ means  ``$d'(g,h) = 0$'', 
\index{$aa$@$\sim$ various equivalence relations}
and let $d$ be the natural ultrametric induced by $d'$ on $X$.
Then the natural action of $G$ on $(X,d)$ is continuous, proper, and isometric.

\vskip.2cm

\ref{iiDElocalellipticityintermsofaction} $\Rightarrow$ \ref{iDElocalellipticityintermsofaction}.
Assume now that $G$ acts on an ultrametric space $(X,d)$,
as in \ref{iiDElocalellipticityintermsofaction};
let $x_0 \in X$.
For $n \ge 1$, set
\begin{equation*}
H_n \, = \, \{ g \in G \mid d(x_0,g(x_0)) < n \} .
\end{equation*}
Then $H_n$ is open because the action is continuous,
$H_n$ is a subgroup because $d$ is an ultrametric,
and $H_n$ is compact because the action is proper.
Hence $G = \bigcup_{n \ge 1} H_n$ is locally elliptic.

\vskip.2cm

Implication \ref{iiDElocalellipticityintermsofaction} $\Rightarrow$
\ref{iiiDElocalellipticityintermsofaction}
follows from Theorem \ref{ftggt},
and  \ref{iiiDElocalellipticityintermsofaction} $\Leftrightarrow$
\ref{ivDElocalellipticityintermsofaction} holds by Proposition \ref{propcoarselyultrametric}.

\vskip.2cm

\ref{iiiDElocalellipticityintermsofaction} $\Rightarrow$ \ref{iiDElocalellipticityintermsofaction}. 
Let $d_G$ be an adapted metric on $G$ and $(X,d_X)$ an ultrametric space
such that there exists a coarse equivalence $f : G \longrightarrow X$.
Let $G_{\textnormal{disc}}$ denote the set $G$ 
furnished with the discrete ultrametric,
defined by $d_{\textnormal{disc}}(g,h) = 1$ whenever $g \ne h$.
Upon replacing $(X,d_X)$ by $(X \times G_{\textnormal{disc}},d_X^+)$, 
with $d_X^+((x,g), (y,h)) = \max \{ d_X(x,y), d_{\textnormal{disc}}(g,h) \}$,
we can assume that $f$ is injective.
Let $\Phi_-$ and $\Phi_+$ be lower and upper controls for $f$; observe that
\begin{equation*}
\aligned
\Phi_-(d_G (g,h)) \, &\le \, d_X (f(g), f(h)) \, \le \, \sup_{k \in G} d_X (f(kg), f(kh))
\\
\, &\le \, \sup_{k \in G} \Phi_+(d_G (kg, kh)) \,  = \, \Phi_+(d_G (g,g'))
\hskip.5cm \forall \hskip.2cm g,h \in G.
\endaligned
\end{equation*}
Define 
\begin{equation*}
d'_G (g,h) \, = \,   \sup_{k \in G} d_X (f(kg), f(kh))
\hskip.5cm \forall \hskip.2cm g,h \in G ;
\end{equation*}
then $d'_G$ is obviously a left-invariant pseudo-metric,
and satisfies the ultrametric inequality; 
moreover it is an ultrametric on $G$ because $f$ is injective.
\par
The identity map $(G,d_G) \longrightarrow (G, d'_G)$ is a coarse equivalence.
It follows that the natural action of $G$ on $(G,d'_G)$ has the desired properties.
\end{proof}   

Our next goal is Theorem \ref{localellipticityforalgebraicgroups}.
Before this, we recall some background and definitions.
\par

\begin{rem}[absolute values on local fields]
\label{onabsolutevalues}
An \textbf{absolute value}, or \textbf{valuation}, on a field $K$,
\index{Absolute value|textbf}
\index{Valuation|see {Absolute value}}
is a function$\vert \cdot \vert_K : K \longrightarrow \R$ such that
there exists a constant $C > 0$ such that
\begin{enumerate}[noitemsep,label=(\arabic*)]
\item\label{1DEonabsolutevalues}
$\vert x \vert_K \ge 0$ for all $x \in K$, 
and $\vert x \vert_K = 0$ if and only if $x = 0$, 
\item\label{2DEonabsolutevalues}
$\vert xy \vert_K = \vert x \vert_K \vert y \vert_K$
for all $x,y \in K$,
\item\label{3DEonabsolutevalues}
if $\vert x \vert_K \le 1$, then $\vert 1+x \vert_K \le C$.
\end{enumerate}
We often use 
``absolute value'' (as in \cite{Neuk--99}), where Cassels uses ``valuation'' \cite{Cass--86}.
If $K$ is given with an absolute value, 
we consider on $K$ the topology for which the sets
$\{ x \in K \mid \vert x \vert_K < \varepsilon \}$, with $\varepsilon > 0$, 
constitute a basis of neighbourhoods of $0$; it is a field topology.
\par

Two absolute values $\vert \cdot \vert_1$, $\vert \cdot \vert_2$ on a same field $K$
are \textbf{equivalent} if there exists $\gamma \in \R_+$, $0 < \gamma < 1$,
such that  $\vert x \vert_2 = (\vert x\vert_1)^\gamma$ for all $x \in K$. 
Two absolute values on a field $K$ induce the same topology on $K$
if and only if they are equivalent \cite[Chapter 2, Lemma 3.2]{Cass--86}.
\par 

Let $K$ be a field, $\vert \cdot \vert_K$ an absolute value on $K$
with respect to which $K$ is complete,
and $L$ a finite extension of $K$.
There exists a unique absolute value $\vert \cdot \vert_L : L \longrightarrow \R_+$
extending $\vert \cdot \vert_K$,
and $L$ is complete with respect to $\vert \cdot \vert_L$;
see for example \cite[Chapter~7]{Cass--86}.
In this situation, we write $\vert \cdot \vert$ 
for both $\vert \cdot \vert_K$ and $\vert \cdot \vert_L$

\vskip.2cm

A local field $\K$ (see Example \ref{panoramalocalfield} for the definition)
has a \textbf{canonical absolute value}, defined as follows.
\index{Local field}
For a prime $p$, the canonical absolute value on $\Q_p$ is given by
$\vert p^m a/b \vert_p = p^{-m}$
for $a,b \in \Z$ coprime to $p$, $b \ne 0$, $m \in \Z$; and $\vert 0 \vert_p = 0$.
\index{$q$@$\Q_p$, field of $p$-adic numbers}
If $\mathbf L$ is a local field of characteristic zero,
then $\mathbf L$ is a finite extension of $\Q_p$ for an appropriate prime $p$;
by the result recalled above,
there is a unique absolute value on $\mathbf L$
extending the absolute value $\vert \cdot \vert_p$ on $\Q_p$.
\par
If $\mathbf L$ is of finite characteristic, say $\operatorname{char}(\mathbf L) = p$,
then $\mathbf L$ is isomorphic to the field $\mathbf F_q \lp t \rp$
of formal Laurent series with coefficients in the finite field $\mathbf F_q$,
where $q$ is a power of $p$.
\index{$fp$@$\mathbf F_p \lp t \rp$, $\mathbf F_q \lp t \rp$}
The standard absolute value of
$\sum_{n=N}^\infty f_n t^n \in \mathbf F_q \lp t \rp$ is $q^{-N}$,
where $N \in \Z$, $f_n \in \mathbf F_q$ for all $n \ge N$, and  $f_N \ne 0$. 
\end{rem}

\begin{defn}
\label{spectralradius}
The \textbf{spectral radius} of a matrix 
\index{$ma$@$\M_n(K)$, square matrices of size $n$ over $K$}
\index{Spectral radius of a matrix|textbf}
$A \in \M_n(\mathbf L)$,
for some integer $n \ge 1$ and non-discrete LC-field $\mathbf L$, 
is the maximum $\rho(A)$ of the absolute values $\vert \lambda \vert$,
where $\lambda$ describes the eigenvalues of $A$
(in appropriate finite extensions of $\mathbf L$).
\end{defn}

\begin{rem}
Consider a matrix $A \in \M_n(\mathbf L)$,
as in the previous definition.
The eigenvalues $\lambda$ of $A$
are elements in appropriate finite extensions of $\mathbf L$.
Below, we write $\vert \lambda \vert$ 
with respect to the canonical absolute value on $\mathbf L$.
It follows from Remark \ref{onabsolutevalues} that
each of the (in)equalities 
$\vert \lambda \vert = 0$,
$\vert \lambda \vert < 1$,
$\vert \lambda \vert = 1$,
$\vert \lambda \vert > 1$,
holds for every absolute value on $\mathbf L$
if and only if it holds for the canonical absolute value on $\mathbf L$.
\par
The remark carries over to (in)equalities
$\rho(A) = 0$, $\rho(A) < 1$,  $\rho(A) = 1$,  $\rho(A) > 1$, for the spectral radius of $A$.
\end{rem}

The following lemma is standard
\cite[Page VII.3]{BAlg4-7}.

\begin{lem}
\label{nilpotent=spectralradius0}
For an integer $n \ge 1$ and  a field $K$, 
a matrix $A \in \M_n(K)$ is \textbf{nilpotent},  \index{Nilpotent! matrix|textbf} 
i.e., $A^k = 0$ for $k$ large enough,
if and only if $0$ is the only eigenvalue of $A$ (in any extension of $K$).
\par
Thus, in the particular case of a local field $\mathbf L$, 
a matrix $A \in \M_n (\mathbf L)$ is nilpotent if and only if
its spectral radius is $0$.
\end{lem}

\begin{defn}
\label{defdistal}
Let $V$ be a finite-dimensional vector space
over a non-discrete LC-field $\mathbf L$.
A linear automorphism $g \in \GL(V)$ is \textbf{distal}
\index{Distal! linear automorphism|textbf}
if every eigenvalue $\lambda$ of $g$ has absolute value $\vert \lambda \vert = 1$,
equivalently if $g$ has spectral radius $1$.
\par

A subgroup $G$ of $\GL(V)$ is distal if $g$ is distal for all $g \in G$.
\index{Distal! linear group|textbf}
(Compare with Example \ref{excoarse}(9).)
\end{defn}


Let $G$ be a subgroup of $\GL(V)$ with $V \simeq \R^n$.
Then $G$ is distal if and only if every orbit closure in $V$
is a minimal $G$-space \cite{Abel--78}.

\begin{defn}
\label{defblocktrigonalizable}
Consider an integer $n \ge 1$ and
a non-discrete LC-field $\mathbf L$.
A subgroup $G$ of $\GL_n(\mathbf L)$ is
\textbf{block-trigonalizable  over $\mathbf L$ with relatively compact diagonal} 
\index{Block-trigonalizable matrix group 
over $\mathbf L$ with relatively compact diagonal|textbf}
if there exist  integers $k \ge 1$ and $n_1, \hdots, n_k \ge 1$ with $n_1 + \cdots + n_k = n$,
and $\gamma \in \GL_n(\mathbf L)$, 
and a corresponding block-decomposition of matrices in $\GL_n(\mathbf L)$, such that
\begin{equation*}
\gamma G \gamma^{-1} \, \subset \, T := 
\begin{pmatrix}
C_1 & \M_{n_1,n_2}(\mathbf L) & \cdots & \M_{n_1,n_k}(\mathbf L) 
\\
0 & C_2 & \cdots & \M_{n_2,n_k}(\mathbf L)
\\
\vdots & \vdots   & \ddots & \vdots
\\
0 & 0 & \cdots & C_k
\end{pmatrix}
\end{equation*}
with $C_1,\hdots, C_k$ compact subgroups 
of $\GL_{n_1}(\mathbf L), \hdots, \GL_{n_k}(\mathbf L),$ respectively.
Here, for $n_1, n_2 \ge 1$, we denote by
$\M_{n_1, n_2}(\mathbf L)$ the set of $n_1$-by-$n_2$ matrices over $\mathbf L$.
\index{$mb$@$\M_{n_1, n_2}(K)$, matrices over $K$ with $n_1$ rows and $n_2$ columns}
\end{defn}

\begin{thm}
\label{localellipticityforalgebraicgroups}
Let $\mathbf L$ be a non-discrete LC-field, $n$ a positive integer,
and $G$ a closed subgroup of $\GL_n(\mathbf L)$.
\index{General linear group $\GL$! $\GL_n(\K)$ 
for $\K$ a non-discrete locally compact field}
The following properties of $G$ are equivalent:
\begin{enumerate}[noitemsep,label=(\roman*)]
\item\label{iDElocalellipticityforalgebraicgroups}
$G$ is distal;
\item\label{iiDElocalellipticityforalgebraicgroups}
$G$ is block-trigonalizable over $\mathbf L$ with relatively compact diagonal.
\end{enumerate}
\index{Locally elliptic LC-group}
Properties \ref{iDElocalellipticityforalgebraicgroups}
and \ref{iiDElocalellipticityforalgebraicgroups}
are implied by the following Property \ref{iiiDElocalellipticityforalgebraicgroups}.
If moreover $\mathbf L$ is a local field, 
Properties \ref{iDElocalellipticityforalgebraicgroups},
\ref{iiDElocalellipticityforalgebraicgroups},
and \ref{iiiDElocalellipticityforalgebraicgroups} are equivalent.
\begin{enumerate}[noitemsep,label=(\roman*)]
\addtocounter{enumi}{2}
\item\label{iiiDElocalellipticityforalgebraicgroups}
$G$ is locally elliptic.
\end{enumerate}
\end{thm}

\begin{proof}
\ref{iiDElocalellipticityforalgebraicgroups} $\Rightarrow$ \ref{iDElocalellipticityforalgebraicgroups} 
This boils down to assuming that $G$ is compact, which clearly implies that it is distal.
\par

\ref{iiiDElocalellipticityforalgebraicgroups} $\Rightarrow$ 
\ref{iDElocalellipticityforalgebraicgroups}
Suppose that Property \ref{iDElocalellipticityforalgebraicgroups} does not hold;
hence there exist $g \in G$ and an eigenvalue $\lambda$ of $g$
(in an appropriate extension of $\mathbf L$) such that $\vert \lambda \vert > 1$.
Since the image of the homomorphism 
$\Z \longrightarrow G, \hskip.2cm j \longmapsto g^j$
is an infinite discrete subgroup of $G$, 
Property \ref{iiiDElocalellipticityforalgebraicgroups} does not hold.
\par
 
\ref{iiDElocalellipticityforalgebraicgroups} $\Rightarrow$ \ref{iiiDElocalellipticityforalgebraicgroups} 
Assume that Property \ref{iiDElocalellipticityforalgebraicgroups} holds,
and that $L$ is a local field. 
If $k = 1$, then $G$ is closed in $C_1$.
In particular, $G$ is compact, and therefore locally elliptic.
Assume now that $k \ge 2$.
\par

With the same notation as just before the proposition,
consider the closed subgroup
\begin{equation*}
U \, = \,
\begin{pmatrix}
1_{n_1} &  \M_{n_1,n_2}(\mathbf L) & \cdots & \M_{n_1,n_k}(\mathbf L)
\\
0 & 1_{n_2} & \cdots & \M_{n_2,n_k}(\mathbf L)
\\
\vdots & \vdots   & \ddots & \vdots
\\
0 & 0 & \cdots & 1_{n_k}
\end{pmatrix} 
\end{equation*}
of $\GL_n(\mathbf L)$
and the natural short exact sequence
\begin{equation*}
U \, \lhook\joinrel\relbar\joinrel\rightarrow \, T \, 
\relbar\joinrel\twoheadrightarrow \, \prod_{i=1}^k C_i \hskip.1cm .
\end{equation*}
Since $\mathbf L$ is non-Archimedean, the unipotent group $U$ is locally elliptic
(Example \ref{examplesoflocallyellipticgroups}(5)).
The quotient group $\prod_{i=1}^k C_i$ is compact, and therefore locally elliptic.
By Proposition \ref{stabilitelocalelliptic}, $T$ is locally elliptic.
Being isomorphic to a closed subgroup of $T$,
the group $G$ is locally elliptic as well.
\par

For \ref{iDElocalellipticityforalgebraicgroups} $\Rightarrow$
\ref{iiDElocalellipticityforalgebraicgroups}, 
we could quote \cite{CoGu--74}, 
but this paper has a gap in case of a field of finite characteristic.
We  rather use a result of Levitzki, recalled below for the reader's convenience.
\par

It is enough to show that, if $G$ is both distal and irreducible, 
with $n \ge 1$, then $G$ is compact.
\par

Let us show first that $G$ is bounded.
Let $W$ be the set of limits of convergent sequences of the form $r_i^{-1}g_i$ 
where $g_i \in G$, $r_i \in \mathbf L$, and $\lim_{i \to \infty} \vert r_i \vert_{\mathbf L} = \infty$. 
Then $W$ is a subsemigroup: 
indeed (with self-explained notation) 
if $r_i^{-1}g_i$ tends to $w$ and ${r'}_i^{-1}g'_i$ tends to $w'$, 
then $(r_ir'_i)^{-1}g_i g'_i$ tends to $ww'$,
by continuity of the multiplication 
$\M_{n}(\mathbf L) \times \M_{n}(\mathbf L) \longrightarrow \M_{n}(\mathbf L)$.
Also, it is stable by both left and right multiplication by every $g \in G$: 
indeed $r_i^{-1}gg_i$ tends to $gw$ and $r_i^{-1}g_i g$ tends to $wg$. 
All this holds for an arbitrary group $G$.
Now, because $G$ is distal, 
the spectral radius of $r_i^{-1}g_i$ tends to $0$, so $W$ consists of nilpotent elements.
\par

Set $V = \bigcap_{w \in W} \ker(w)$.
By Levitzki's theorem, $V \ne \{0\}$.
Since $Wg = W$ for all $g \in G$, 
it follows that $V$ is $G$-stable.
By irreducibility, it follows that $V = \mathbf L^n$, whence $W=\{0\}$. 
Since $\mathbf L$ is locally compact, this implies that $G$ is bounded.

Now we know that $G$ is bounded. 
Let $H$ be its closure in $\M_d(\mathbf L)$, so $H$ is compact. 
Then $H\subset \GL_d(\mathbf L)$: 
indeed otherwise some sequence $(g_i)$ in $G$ tends to $h \in H$ not invertible, 
which implies that the spectral radius of $g_i^{-1}$ tends to infinity and contradicts that $G$ is bounded. Thus $G$ is compact.
\end{proof}

We have used the following purely algebraic fact:
\index{Levitizki theorem}
\index{Theorem! Levitzki}

\begin{thm}[Levitzki]
\label{Levitzki}
Let $K$ be any field, $d \ge 1$ an integer, 
and $S \subset \M_d(K)$ a multiplicative subsemigroup of nilpotent matrices.
\par
Then there exists a vector $v \in K^d$, $v \ne 0$, 
such that $sv = 0$ for all $s \in S$.
\end{thm}

\begin{proof}[On the proof]
The original article is \cite{Levi--31}.
\par
There is a convenient proof in \cite[Theorem 2.1.7]{RaRo--00}.
There, the theorem is formulated for $K$ algebraically closed.
But the general case follows;
indeed, let $\overline K$ denote the algebraic closure of $K$.
The existence of $v \in \overline K ^d$ is equivalent to
the solvability of a system of linear equations;
since this system has coefficients in $K$,
it has a non-zero solution in $K^d$ if and only if it has one in $\overline K^d$.
\end{proof}

Note that, by a straightforward iteration, Levitzki's theorem implies
that a nilpotent multiplicative subsemigroup of $\M_d(K)$ is trigonalizable.

\section{Capped subgroups}
\label{sectionomegaboundedness}

In this section, we point out that compact generation
can be used to define in an LC-group
subsets that are of a restricted size, in the following sense:

\begin{defn}
\label{defomegabd}
A subset $K$ of an LC-group $G$ is \textbf{capped}\footnote{In
French: ``plafonn\'e''.}
\index{Capped subset of an LC-group|textbf}
if it is contained in a compactly generated closed subgroup,
and \textbf{uncapped} otherwise.
Proposition \ref{propomegabd} shows that 
an equivalent definition is obtained if ``closed'' is replaced by ``open''.
\par

In particular, when $G$ is a discrete group, a subset of $G$
is capped if it is contained in a finitely generated subgroup.
\end{defn}

\begin{rem}
\label{remcapped}
Let $G$ be an LC-group.
\begin{enumerate}[noitemsep,label=(\arabic*)]
\item\label{1DEremcapped}
$G$ is capped in itself if and only if $G$ is compactly generated. 
If this holds, then every subset of $G$ is capped.
\item\label{2DEremcapped}
If $G$ is locally elliptic, 
its capped subsets are exactly its relatively compact subsets.
\item\label{3DEremcapped}
In $G$, capped subsets have a $\sigma$-compact closure. 
In particular, in a discrete group, capped subsets are countable.
\item\label{4DEremcapped}
If $d$ is a left-invariant locally bounded {\it ultrametric} pseudo-metric, then capped subsets of $G$ are bounded.
\end{enumerate}
\end{rem}

The following proposition yields a converse to Remark 
\ref{remcapped}\ref{4DEremcapped}

\begin{prop}
\label{ultrametricsigmacgroup}
Let $G$ be a $\sigma$-compact LC-group.
There exists on $G$ a left-invariant continuous ultrametric pseudo-metric $d$
such that the bounded subsets of $(G,d)$ are exactly the capped subsets of $G$.
\end{prop}
\index{Ultrametric! pseudo-metric}

\begin{proof}
Let $(S_n)_{n \ge 0}$ be a sequence of compact neighbourhoods of $1$ in $G$
such that $S_n \subset S_{n+1}$ for all $n \ge 0$ and $G = \bigcup_{n \ge 0} S_n$.
For each $n \ge 0$, let $H_n$ be the subgroup of $G$ generated by $S_n$;
note that $H_n$ is  a compactly generated open subgroup of $G$.
\par

Define a function $\ell : G \longrightarrow \N$ by
$\ell (g) = \min \{ n \in \N \mid g \in H_n \}$ for $g \in G$;
observe that $\ell(g^{-1}g') \le n$ for all $n \ge 1$ and $g,g' \in S_n$.
Then $d : (g,g') \longmapsto \ell (g^{-1} g')$ is 
a left-invariant ultrametric pseudo-metric on $G$.
Moreover, by the argument of Remark \ref{remultrametricpseudom}(3), 
$d$ is continuous, indeed locally constant, on $G$.
Observe that   
\begin{equation}
\label{eqfromcapped}
H_n \, = \, \{ g \in G \mid d(1,g) \le n \}
\hskip.5cm \text{for every integer} \hskip.2cm n \ge 0 .
\end{equation}
\par

Let $K$ be a subset of $G$.
If the $d$-diameter of $K$ is finite, there exists $n \ge 0$
such that $\sup_{k \in K} d(1,k) \le n$, and therefore $K \subset H_n$, hence $K$ is capped.
Conversely, if $K$ is capped, there exists a compact subset $S$ of $G$
generating a closed subgroup $H$ of $G$ containing $K$;
then there exists $n \ge 0$ such that $S \subset S_n$, hence $K \subset H \subset H_n$,
and the diameter $\sup_{k,k' \in K} d(k,k')$ of $K$ is at most $n$.
\end{proof}

\begin{prop}
\label{propomegabd}
Consider an LC-group $G$, a subset $K$ of $G$, and the following conditions:
\begin{enumerate}[noitemsep,label=(\roman*)]
\item\label{iDEpropomegabd}
$K$ is capped;
\item\label{iiDEpropomegabd}
$K$ is contained in a compactly generated \emph{open} subgroup;
\item\label{iiiDEpropomegabd}
for every left-invariant locally bounded ultrametric pseudo-metric $d$ on $G$,
the $d$-diameter of $K$ is finite;
\item\label{ivDEpropomegabd}
as in \ref{iiiDEpropomegabd}, with ``continuous'' instead of ``locally bounded''.
\end{enumerate}
Then we have the implications   
\ref{iiDEpropomegabd} $\Longleftrightarrow$ 
\ref{iDEpropomegabd} \hskip.3cm $\Longrightarrow$ \hskip.3cm
\ref{iiiDEpropomegabd} $\Longleftrightarrow$ 
\ref{ivDEpropomegabd}.
\par

If $G$ is moreover $\sigma$-compact, all four conditions are equivalent.
\end{prop}

\begin{proof}
Let $T$ be a compact neighbourhood of $1$ in $G$.
Every closed subgroup $H$ of $G$ generated by a compact set $S$
is contained in the open subgroup of $G$ generated by the compact set $S \cup T$.
The equivalence of \ref{iDEpropomegabd} and \ref{iiDEpropomegabd} follows.
\par

Consider a pseudo-metric $d$ on an arbitrary topological group $L$.
If $d$ is locally bounded, every compact subset of $L$ has a finite $d$-diameter
(Remark \ref{remonmetrics}(5)).
If $d$ is moreover left-invariant and ultrametric,
every compactly generated subgroup of $L$ has a finite $d$-diameter.
Hence \ref{iDEpropomegabd} implies \ref{iiiDEpropomegabd}.
\par

Since continuous pseudo-metrics are locally bounded (Remark \ref{remonmetrics}(4)), 
\ref{iiiDEpropomegabd} implies \ref{ivDEpropomegabd}.
\par

Let $d$ be a left-invariant \emph{locally bounded} ultrametric pseudo-metric on $G$,
as in \ref{iiiDEpropomegabd}.
Let $S$ be a neighbourhood of $1$ in $G$ such that the diameter
$\sup_{s \in S} d(1,g)$ is bounded, say is strictly smaller than some constant $D$.
Then $d' := \lfloor D^{-1}d \rfloor$ 
is a left-invariant \emph{continuous} ultrametric pseudo-metric on $G$
(Remark \ref{remultrametricpseudom}(3)). 
Suppose that \ref{ivDEpropomegabd} holds;
then the $d'$-diameter of $K$ is finite, and it follows that the $d$-diameter of $K$ is finite.
Hence   \ref{iiiDEpropomegabd} holds.
\par

Assume now that $G$ is $\sigma$-compact. 
It suffices to show that the negation of \ref{iDEpropomegabd}
implies the negation of \ref{ivDEpropomegabd},
and this follows from Proposition \ref{ultrametricsigmacgroup}.
\end{proof}

The property for a group to have every countable subset capped
is related to the following property; the relation is the object of
Proposition \ref{countablesubsetcappedanduncountablecof}.

\begin{defn}
\label{defuncountablecof}
A topological group has \textbf{uncountable cofinality} if it cannot be written
as the union of an infinite countable  strictly increasing sequence of open subgroups,
\index{Uncountable cofinality|textbf}
and of \textbf{countable cofinality} otherwise.
\end{defn}

\begin{prop}
\label{propuncountablecof}
For a topological group $G$, the following two conditions are equivalent:
\begin{enumerate}[noitemsep,label=(\roman*)]
\item\label{iDEpropuncountablecof}
$G$ has uncountable cofinality;
\item\label{iiDEpropuncountablecof}
every continuous isometric action of $G$ on an ultrametric space
has bounded orbits.
\end{enumerate}
\end{prop}

\begin{proof}
\ref{iiDEpropuncountablecof} $\Rightarrow$ \ref{iDEpropuncountablecof}
We show the contraposition, and we assume that
$G$ has countable cofinality,
i.e., that there exists an infinite countable strictly increasing sequence of open subgroups 
$H_0 \subsetneqq \cdots \subsetneqq H_n \subsetneqq H_{n+1} \subsetneqq \cdots$
such that $\bigcup_{n \ge 0} H_n = G$.
Define a function $\ell : G \longrightarrow \N$ by
$\ell (g) = \min \{ n \in \N \mid g \in H_n \}$ for $g \in G$.
Then $d : (g,g') \longmapsto \ell (g^{-1} g')$ 
is a left-invariant ultrametric pseudo-metric on $G$.
(Compare with the proof of Proposition \ref{ultrametricsigmacgroup}.)
Observe that $d(g,g') = d(gh, g'h')$ for all $g,g' \in G$ and $h,h' \in H_0$.
\par

On the quotient space $G / H_0$,
we define an ultrametric $\underline d$
by 
\begin{equation*}
\underline d (g H_0, g' H_0) \, = \, d(g,g') .
\end{equation*}
By the observation above,
this does not depend on the choices of $g$ and $g'$ in their classes modulo $H_0$.
Because of the hypothesis on the sequence $(H_n)_{n \ge 0}$,
the diameter of the ultrametric space $(G/H_0, \underline d)$ is infinite.
The natural action of $G$ on $G/H_0$ is continuous, isometric, transitive,
and its unique $G$-orbit is unbounded with respect to $\underline d$.
\par

\ref{iDEpropuncountablecof} $\Rightarrow$ \ref{iiDEpropuncountablecof}
Consider a continuous action of an LC-group $G$ 
on an ultrametric space $(X, d)$, and a point $x_0 \in X$.
For every integer $n \ge 0$, define
$H_n = \{g \in G \mid d(gx_0, x_0) < n+1 \}$;
it is an open subgroup of $G$.
Moreover $G = \bigcup_{n \ge 0} H_n$.
If $G$ has uncountable cofinality,
then $G = H_n$ for some $n \ge 0$,
and it follows that every $G$-orbit in $X$ is bounded.
\end{proof}

\begin{rem}
\label{firstexamplesuncountablecof}
(1) 
A $\sigma$-compact LC-group has uncountable cofinality
if and only if it is compactly generated.
Thus, uncountable cofinality is really of interest only for non-$\sigma$-compact groups.

\vskip.2cm

(2)
Uncountable discrete groups with uncountable cofinality can be found in \cite{KoTi--74}:
if $F$ is a non-trivial finite perfect group and $I$ an infinite set,
the uncountable product $F^{I}$, endowed with the discrete toplogy,
has uncountable cofinality. 

\vskip.2cm

(3)
If a topological group has uncountable cofinality,
the same holds for all its quotient groups.

\vskip.2cm

(4)
A locally compact abelian group $A$ has uncountable cofinality
if and only if it is compactly generated,
as a consequence of (3) and of Proposition \ref{Kulikov}\ref{2DEKulikov} below.
\index{LCA-group}

\vskip.2cm

(5)
Let $G$ be a group and $\mathcal T, \mathcal T'$ two group topologies on $G$
such that the identity $(G, \mathcal T') \longrightarrow (G, \mathcal T)$ is continuous.
If $(G, \mathcal T')$ has uncountable cofinality, so has $(G, \mathcal T)$.
\par

In particular, if $G$ has uncountable cofinality as a discrete group,
then $G$ has uncountable cofinality for every group topology.
\end{rem}

In (4) above, we have used:

\begin{prop}
\label{Kulikov}
Let $E$ be an LC-group; assume that $E$ is not $\sigma$-compact. 
Then:
\begin{enumerate}[noitemsep,label=(\arabic*)]
\item\label{1DEKulikov}
$E$ has a $\sigma$-compact, open subgroup that is not compactly generated.
\item\label{2DEKulikov}
 If $E$ is abelian, then $E$ has a $\sigma$-compact quotient group 
 that is not compactly generated.
\end{enumerate}
\end{prop}

\begin{proof}
\ref{1DEKulikov}
is proved by an immediate induction: 
we define a strictly increasing sequence of open subgroups $(E_n)_{n\ge 0}$, 
where $E_0$ is generated by an arbitrary compact subset of non-empty interior, 
and $E_{n+1}$ is generated by $E_n$ and an arbitrary element of $E\smallsetminus E_n$ 
(this is possible because, since $E_n$ is compactly generated, it is not equal to $E$). 
Then the open subgroup $H = \bigcup_{n \ge 0} E_n$ is not compactly generated 
and is $\sigma$-compact.
\par

Let us now prove \ref{2DEKulikov}. 
First, we can assume, 
modding out if necessary by some compactly generated open subgroup, 
that $E$ is discrete and uncountable. 
For every abelian group $E$, there exist a divisible abelian group $D$
and an injective morphism $\varphi : E \lhook\joinrel\relbar\joinrel\rightarrow D$
\cite[Theorem 10.30]{Rotm--95}.
\par

The divisible abelian group $D$ decomposes 
as a direct sum $\bigoplus_{i\in I} D_i$ of directly indecomposable (divisible) subgroups, 
and all $D_i$ are countable: see for example
\cite[Theorems 10.1, 10.36, and 10.28]{Rotm--95}.
\par

Now let $C$ be a countable, infinitely generated subgroup $H$ of $E$ 
(as provided by \ref{1DEKulikov}). 
There exists a countable subset $J$ of $I$ such that 
$\varphi(C) \subset \bigoplus_{j \in J}D_j$. 
Let $\pi$ be the projection of $D$ onto the countable group $M = \bigoplus_{j \in J}D_j$, 
and $\psi = \pi \circ \varphi$. 
Then $\psi:E \longrightarrow M$ is injective on $C$. 
So $\psi(E)$ is a countable, infinitely generated quotient of $E$.
\end{proof}

\begin{prop}
\label{countablesubsetcappedanduncountablecof}
Let $G$ be an LC-group in which every countable subset is capped.
Then $G$ has uncountable cofinality.
\end{prop}

\begin{proof}
Suppose by contradiction that $G$ is the union 
of a strictly increasing infinite countable sequence of open subgroups 
$H_0 \subsetneqq \cdots \subsetneqq H_n \subsetneqq H_{n+1} \subsetneqq \cdots$.
Choose for each $n \ge 1$ an element $k_n \in H_n$, $k_n \notin H_{n-1}$.
Then the set $K = \{k_n\}_{n=1}^\infty$ is contained in a subgroup of $G$ generated
by some finite set $S$; since $S \subset H_m$ for some $m \ge 1$,
we have $K \subset H_m$, and this is preposterous.
\end{proof}

Proposition \ref{propuncountablecof} characterizes uncountable cofinality
in terms of isometric actions on ultrametric spaces.
Let us now consider isometric actions on arbitrary metric spaces.

\begin{defn}
\label{defTSB}
A topological group $G$ is \textbf{strongly bounded}
\index{Strongly bounded topological group|textbf}
if every continuous isometric action of $G$ on any metric space has bounded orbits.
\end{defn}

\begin{rem}
\label{remTSB}
(1) 
The idea of strong boundedness appears in \cite{Berg--06}.
\vskip.2cm

(2)
A strongly bounded group has uncountable cofinality,
by Proposition \ref{propuncountablecof}.

\vskip.2cm

(3)
A $\sigma$-compact LC-group $G$ is strongly bounded if and only if it is compact.
\par

Indeed, choose a compact normal subgroup $K$ of $G$ 
(Theorem \ref{KK}, Kakutani-Kodaira)
and a left-invariant proper compatible metric $d$ on $G/K$
(Theorem \ref{Struble}, Struble).
The natural action of $G$ on $(G/K, d)$ is continuous, isometric, and transitive.
If $G$ is strongly bounded, then $(G/K, d)$ has finite diameter,
hence is compact; it follows that $G$ itself is compact.
The converse implication is straightforward.

\vskip.2cm

(4)
Let $G$ be a group and $\mathcal T, \mathcal T'$ two group topologies on $G$
such that the identity $(G, \mathcal T') \longrightarrow (G, \mathcal T)$ is continuous.
If $(G, \mathcal T')$ is  strongly bounded, so is $(G, \mathcal T)$.
\par

In particular, if $G$ is strongly bounded as a discrete group,
then $G$ is strongly bounded for every group topology.
\end{rem}

\begin{defn}
\label{defucapped}
Let $u = (u_n)_{n \ge 0}$ be a sequence of positive real numbers.
An LC-group $G$ is \textbf{uniformly $u$-capped} if,
for every sequence $(g_n)_{n \ge 0}$ of elements of $G$,
there exists a compact subset $S$ of $G$
such that $g_n \in \langle S \rangle$ and $\ell_S (g_n) \le u_n$ for all $n \ge 0$;
here $\langle S \rangle$ denotes the subgroup of $G$ generated by $S$,
and $\ell_S$  the word length with respect to $S$, see Definition \ref{wordmetric}.
\par
If $G$ is uniformly $u$-capped for some sequence $u$, 
the LC-group $G$ is {\bf strongly distorted}.
\index{Strongly distorted LC-group|textbf}
\end{defn}

\begin{prop}
\label{propucapped}
Let $G$ be an LC-group. Suppose that $G$ is 
strongly distorted.
\par

Then $G$ is strongly bounded.
\end{prop}

\begin{proof}
Assume by contradiction that $G$ acts continuously and isometrically 
on a metric space $X$ with an unbounded orbit. 
Fix a sequence $(v_n)_{n \ge 0}$ of positive real numbers
with $v_n \gg u_n$ (e.g.\ $v_n = 2^{n+u_n}$). 
Then we can find $x \in X$ and a sequence $(g_n)_{n \ge 0}$ in $G$ such that $d(x,g_nx)\ge v_n$. 
If $S$ is as in Definition \ref{defucapped} and $r = \sup_{g \in S} d(x,gx)$, 
then $d(x,g_nx) \le ru_n$ for all $n \ge 0$. 
Thus $v_n \le ru_n$ for all $n \ge 0$, a contradiction.
\end{proof}

\begin{prop}
\label{propCayleybounded}
For a topological group $G$, the following three conditions are equivalent:
\begin{enumerate}[noitemsep,label=(\roman*)]
\item\label{iDEpropCayleybounded}
G is strongly bounded;
\item\label{iiDEpropCayleybounded}
every left-invariant continuous pseudo-metric on G is bounded;
\item\label{iiiDEpropCayleybounded}
every left-invariant locally bounded pseudo-metric on G is bounded.
\end{enumerate}
If moreover $G$ is metrizable, 
then these condition are also equivalent to: 
\begin{enumerate}[noitemsep,label=(\roman*)]
\addtocounter{enumi}{3}
\item\label{ivDEpropCayleybounded}
every left-invariant locally bounded metric on $G$ is bounded.
\end{enumerate}
\end{prop}

\begin{proof}
\ref{iDEpropCayleybounded} $\Rightarrow$ \ref{iiDEpropCayleybounded}
Let $d$ be a left-invariant continuous pseudo-metric on $G$.
Denote by $K$ the closed subgroup $\{g \in G \mid d(1,g) = 0 \}$,
see Remark \ref{remoncontadapmetrics}(4), 
and by $\underline d$ the $G$-invariant continuous metric on $G/K$ induced by $d$.
The action of $G$ on $(G/K, \underline d)$ is continuous, isometric, and transitive.
If \ref{iDEpropCayleybounded} holds, then $(G/K, \underline d)$ is bounded. 
It follows that the pseudo-metric $d$ is bounded on $G$.
\par
 
\ref{iiDEpropCayleybounded} $\Rightarrow$ \ref{iDEpropCayleybounded}
Let $\alpha$ be a continuous action of $G$ on some metric space $(X,d_X)$.
Choose a point $x_0 \in X$ and define a continous pseudo-metric $d_G$ on $G$
by $d_G(g,h) = d_X(gx_0, hx_0)$, for all $g,h \in G$.
If  \ref{iiDEpropCayleybounded} holds, then $d_G$ is bounded,
and the $G$-orbit $Gx_0$ is bounded in $X$ with respect to $d_X$.
\par

Since
\ref{iiiDEpropCayleybounded} $\Rightarrow$ \ref{iiDEpropCayleybounded}
is trivial, it remains to prove the converse. We proceed by contraposition.
\par

Assume that \ref{iiiDEpropCayleybounded} does not hold,
and let $d'$ be an unbounded, locally bounded, left-invariant pseudo-metric on $G$. 
Multiplying by a suitable positive scalar, 
we can suppose that the 1-ball for $d'$ is a neighbourhood of 1 in $G$. 
For $n \in \N$, define $K_n$ to be the $3^n$-ball of $(G,d')$ around $1$;
for $n \in \Z$ with $n \le -1$, define inductively $K_n$ 
to be a symmetric neighbourhood of $1$ such that $K_n^3 \subset K_{n+1}$.
Let $d$ be the resulting pseudo-metric defined in Lemma \ref{lemBKetKK},
whose assumptions are fulfilled. 
Since all $K_n$ are neighbourhoods of 1, Conclusion \ref{2DElemBKetKK} of this lemma
ensures that $d$ is continuous. 
Furthermore, \ref{3DElemBKetKK} of this lemma yields that $d$ is unbounded: 
more precisely $d(1,g) > 2^n$ for all $g \notin K_n$ and all $n \ge 0$,
and \ref{iiDEpropCayleybounded} does not hold.
\par

Suppose now that $G$ is metrizable.
There exists a left-invariant bounded compatible metric $\delta$ on $G$;
see Theorem \ref{BK} and Remark  \ref{remonmetrics}(1).
Assume that \ref{ivDEpropCayleybounded} holds. 
Let $d$ be a left-invariant locally bounded pseudo-metric on $G$.
The left-invariant locally bounded metric $d+\delta$ is bounded by hypothesis,
hence $d$ itself is bounded, and \ref{iiiDEpropCayleybounded} holds.
The converse implication
\ref{iiiDEpropCayleybounded} $\Rightarrow$ \ref{ivDEpropCayleybounded}
is trivial.
\end{proof}

\begin{exe}
\label{Symuncountablecof}
In the full symmetric group of an infinite set $X$, 
every countable subset is capped,
indeed every countable subset is contained in a $2$-generated subgroup 
\cite[Theorem 3.3]{Galv--95}.
It follows that $\operatorname{Sym}(X)$ has uncountable cofinality as a discrete group,
as originally shown by Macpherson and Neumann \cite{MaNe--90}.
\par

Indeed, arguments from the proof of Theorem 3.1 in \cite{Galv--95},
show that $\operatorname{Sym}(X)$ is
uniformly $u$-capped, with $u_n = 12 n + 16$.
Hence, as a discrete group, 
$\operatorname{Sym}(X)$ is strongly bounded.
\index{Symmetric group $\operatorname{Sym}(X)$ of a set $X$}
\par

If $\mathcal H$ is a separable infinite-dimensional complex Hilbert space,
its unitary group $\operatorname{U}(\mathcal H)$ with the strong topology
is strongly bounded \cite{RiRo--07}.
\index{Unitary group of a Hilbert space}
\par

For strongly bounded groups, see 
\cite{Berg--06}, \cite{CaFr--06}, \cite{Corn--06}, and \cite{Rose--09}.
\end{exe}

The property of uncountable cofinality appears in 
a characterization of  the following property, due to Serre.

\begin{defn}
\label{defFA}
A topological group $G$ has \textbf{Property (FA)}
if every continuous action of $G$ on a tree has bounded orbits.
\index{Property! (FA)|textbf}
\end{defn}

It is known that, for a group action on a tree, 
the existence of a bounded orbit implies the existence of a fixed point in the $1$-skeleton. This explains the acronym FA, which stands for 
``Fixed point property for actions on Trees''
(in French, tree is Arbre). 
This property can be characterized as follows, 
as in  \cite[Section 6.1]{Serr--77}.
The setting of Serre is that of groups, 
but his arguments carry over to  topological groups.
For amalgamated products,
see Section \ref{sectionamalgamHNN} below.

\begin{thm}[Serre]
\label{caractFA}
A topological group $G$ has Property (FA) if and only if it satisfies
the following three conditions:
\begin{enumerate}[noitemsep,label=(\arabic*)]
\item\label{1DEcaractFA}
$G$ does not have any continuous surjective homomorphism onto $\Z$;
\item\label{2DEcaractFA}
$G$ does not decompose as an amalgamated product
over any proper open subgroup;
\item\label{3DEcaractFA}
$G$ has uncountable cofinality.
\end{enumerate}
\end{thm}

\begin{exe}
\label{exFA}
Groups with Property (FA) include
finitely generated torsion groups,
triangle groups
$\langle s,t \mid s^a = t^b = (st)^c = 1 \rangle$
where $a,b,c$ are integers, at least $2$,
and $\SL_3(\Z)$ and its subgroups of finite index
\cite[Sections I.6.3 and I.6.6]{Serr--77}.
An LC-group with Kazhdan's Property (T) has Property (FA);
see, e.g., \cite[Theorems 2.3.6 and 2.12.4]{BeHV--08}.
\index{Special linear group $\SL$! $\SL_n(\Z)$, $\GL_n(\Z)$} 
\end{exe}

\begin{rem}
Let $G$ be a $\sigma$-compact LC-group.
As already observed in Remark \ref{firstexamplesuncountablecof}(1), 
$G$ has uncountable cofinality if and only if
$G$ is compactly generated.
Hence, by Theorem \ref{caractFA},
a $\sigma$-compact LC-group that is not compactly generated
does not have Property (FA).
In particular, a countable infinite locally finite group does not have Property (FA).
\end{rem}

\section{Amenable and geometrically amenable LC-groups}
\label{sectionamenableLCgroups}

\begin{defn}
\label{defamenable}
An LC-group $G$ with left-invariant Haar measure $\mu$ is 
\textbf{amenable} if,
\index{Amenable! locally compact group|textbf}
for every  compact subset $Q$ of $G$ and $\varepsilon > 0$,
there exists a compact subset $F$ of $G$ with $\mu(F) > 0$
such that
\begin{equation}
\label{eqdefmoyen}
\frac{\mu( Q F )}{\mu(F)} \, \le \, 1 + \varepsilon .
\end{equation}
\end{defn}

\begin{rem}
(1)
An equivalent definition of amenability is obtained 
by replacing in Definition \ref{defamenable}
compact subsets $Q$ by subsets $\{1, q\}$, for $q \in Q$; 
see e.g.\ \cite{EmGr--67}.
Part of the theory of amenable LC-groups consists in showing the equivalence
of a large number of definitions; 
for some of these, see e.g.\ \cite[Appendix G]{BeHV--08}.
\par

For the definition above, amenability can be viewed 
as an isoperimetric condition.
Indeed, when $Q \ni 1$,  
the complement of $F$ in $QF$ is a ``$Q$-boundary'' of $F$,
and Condition (\ref{eqdefmoyen}) means
that the measure $\mu(QF \smallsetminus F)$ of this boundary
is smaller than $\varepsilon$ times the measure $\mu(F)$ of $F$.
This kind of definition, going back for discrete groups to F\o lner  \cite{Foln--55},
is quite different (though equivalent) 
from earlier definitions due to von Neumann and Tarski, in the late 20's.

\vskip.2cm

(2)
For $\sigma$-compact LC-groups, 
amenability \emph{is not} invariant by metric coarse equivalence.
For example, let $B$ denote the upper triangular subgroup of $\GL_2(\R)$.
Then $B$ is solvable, and therefore amenable, but $\GL_2(\R)$ is non-amenable;
yet the inclusion of $B$ in $\GL_2(\R)$ is a quasi-isometry 
(with respect to geodesically adapted metrics on the groups) 
because $B$ is cocompact is $\GL_2(\R)$.
\index{Triangular group}

\vskip.2cm

(3)
Nevertheless, amenability is a coarsely invariant property
among \emph{unimodular} locally compact groups.
We limit here the discussion to this invariance property
(Proposition \ref{propamenablerightemaneble} 
and Corollary \ref{invariancerightam}),
and refer to standard sources such as \cite{Gree--69} 
for general facts on amenability for LC-groups.
\end{rem}

\begin{defn}
\label{defgeoamenable}
An LC-group $G$ with left-invariant Haar measure $\mu$ is 
\textbf{geometrically amenable} (or \textbf{right-amenable}) if, 
\index{Right-amenable locally compact group|textbf}
\index{Geometrically amenable locally compact group|textbf}
\index{Amenable|see {Geometrically amenable locally compact group}}
for every  compact subset $Q$ of $G$ and $\varepsilon > 0$,
there exists a compact subset $F$ of $G$ with $\mu(F) > 0$
such that
\begin{equation*}
\frac{\mu( FQ )}{\mu(F)} \, \le \, 1 + \varepsilon .
\end{equation*}
\end{defn}

The terminology ``right-amenable'' emphasizes that $Q$ multiplies $F$ 
\emph{on the right} (compare with \ref{defamenable}),
even though the measure $\mu$ is \emph{left-invariant}.

\begin{lem}
\label{lemamenablerightemanable}
Let $G$ be an LC-group.
\begin{enumerate}[noitemsep,label=(\arabic*)]
\item\label{1DElemamenablerightemanable} 
If $G$ is geometrically amenable, then $G$ is unimodular.
\item\label{2DElemamenablerightemanable} 
Assume that $G$ is unimodular.
Then $G$ is amenable if and only if $G$ is geometrically amenable.
\par

In particular, a discrete group is amenable if and only if it is geometrically amenable.
\end{enumerate}
\end{lem}

\begin{proof}
Let $\mu$ denote a left-invariant Haar measure on $G$,
and $\Delta$ the modular function of $G$.
\par

\ref{1DElemamenablerightemanable} 
We prove the contraposition. 
Suppose that $G$ is not unimodular.
There exists $q \in G$ such that $\Delta(q) > 1$;
set $Q = \{q\}$.
For every compact subset $F$ of $G$ with $\mu(F) > 0$, we have
$\mu(FQ)/\mu(F) = \mu(Fq)/\mu(F) = \Delta(q) > 1$.
It follows that $G$ is not geometrically amenable.
\par

\ref{2DElemamenablerightemanable} 
Since $G$ is unimodular, $\mu(P^{-1}) = \mu(P)$ for every compact subset $P$ of $G$.
Suppose that $G$ is amenable. 
Consider $\varepsilon > 0$ and a compact subset $Q$ of $G$.
Since $Q^{-1}$ is compact,
there exists by hypothesis a compact subset of $G$, 
of the form $F^{-1}$ for some $F \subset G$,
such that $\mu(Q^{-1} F^{-1}) / \mu(F^{-1}) \le 1 + \varepsilon$,
i.e.,  $\mu(FQ) / \mu(F) \le 1 + \varepsilon$.
Hence $G$ is geometrically amenable.
The proof of the converse implication is similar.
\end{proof}

\begin{prop}
\label{propamenablerightemaneble}
For an LC-group $G$ to be geometrically amenable, 
it is necessary and sufficient that $G$ is both amenable and unimodular.
\index{Unimodular group}
\end{prop}

\begin{proof}
The proof of the sufficiency follows from \ref{2DElemamenablerightemanable}
in Lemma \ref{lemamenablerightemanable}.
Then the proof of necessity follows from \ref{1DElemamenablerightemanable} 
in the same lemma.
\end{proof}

\begin{rem}
In Definitions \ref{defamenable} and \ref{defgeoamenable}, 
``F'' stands for ``F\o lner''  \cite{Foln--55}. 
In the proof of the next proposition, finite sets are denoted by $E$,
rather than by $F$ as in several other places of this book.
\end{rem} 

\begin{lem}
\label{discretemetriclattice}
Let $G$ be a $\sigma$-compact LC-group
and $d$ a measurable adapted pseudo-metric on $G$.
Let $s > 0$ be such that the ball $B_G^1(s)$ of centre $1$ and radius $s$
is a neighbourhood of $1$ in $G$.
Let $L$ be a $c$-metric lattice in $(G,d)$, for some $c > 2s$.
\par

Then, for any compact subset $K$ of $G$, the intersection $K \cap L$ is finite.
\end{lem}

\noindent
\emph{Note.} The condition $c > 2s$ cannot be deleted;
see Example \ref{exofmetriclattices}(4).
For measurable adapted pseudo-metrics, see just before Example \ref{volumegrowthG}.

\begin{proof}
Let $\mu$ be a left-invariant Haar measure on $G$.
Recall that, with respect to $d$, all balls are relatively compact,
and therefore of finite $\mu$-measure.
\par

Let $R > 0$ be such that $\left( xB_G^1(R) \right)_{x \in L}$
is a covering of $G$.
Let $k_1, \hdots, k_n \in K$ be such that $\left( k_j B_G^1(R) \right)_{j=1, \hdots, n}$
is a covering of $K$.
For $x \in K \cap L$, the balls $xB_G^1(s)$ are disjoint and contained in
$\bigcup_{j=1}^n k_j B_G^1(R+s)$.
Hence
\begin{equation*}
\vert K \cap L \vert \mu(B_G^1(s)) \, \le \, n  \mu (B_G^1(R+s)) .
\end{equation*}
Since $\mu(B_G^1(s)) > 0$ by the choice of $s$, the conclusion follows.
\end{proof}

\begin{prop}
\label{amenable=amenable}
Let $G$ be a $\sigma$-compact LC-group
and $d$ a measurable adapted pseudo-metric on $G$.
\par

Then $G$ is geometrically amenable in the sense of Definition \ref{defgeoamenable}
if and only if $(G,d)$ is amenable in the sense of Definition \ref{defamucp}.
\end{prop}

\begin{proof}
\emph{Step one: on packing and covering.}
Consider for now an LC-group $G$, with a left-invariant Haar measure $\mu$.
For  subsets $X,Y$ of $G$, let
\begin{enumerate}[noitemsep]
\item[$\circ$]
$m_X^-(Y) \in \N \cup \{\infty\}$ be the maximum number of 
pairwise disjoint left-translates of $X$ contained in $Y$,
\item[$\circ$]
$m_X^+(Y) \in \N \cup \{\infty\}$ be the minimum number of 
left-translates of $X$ covering $Y$.
\end{enumerate}
We agree that $m_X^-(Y) = \infty$ if the number of 
pairwise disjoint left-translates of $X$ contained in $Y$ is not bounded,
and $m_X^+(Y) = \infty$ if $Y$ cannot be covered by
a finite number of left-translates of $X$.
If $X$ and $Y$ are measurable, observe that
\begin{equation}
\label{eqavecmuetB}
m_X^-(Y) \mu(X) \, \le \, \mu(Y) \, \le \, m_X^+(Y) \mu(X) .
\end{equation}

\vskip.2cm

\emph{Step two: choice of $C$, $L$, and $B$.}
Assume now (as in the proposition) that $G$ is $\sigma$-compact,
with a measurable adapted pseudo-metric $d$.
Choose
\begin{enumerate}[noitemsep,label=(\roman*)]
\item\label{aDEStepone}
a ball $C$ in $G$ of centre $1$ and radius $s$ large enough
for $C$ to be a neighbourhood of $1$,
\item\label{bDEStepone}
a $c$-metric lattice $L$ in $G$ for some $c > 2s$,
\item\label{cDEStepone}
a ball $B$ of centre $1$ in $G$ such that
$(xB)_{x \in L}$ is a covering of $G$.
\end{enumerate}
(Observe that $C$, $L$, and $B$ are as in the proof
of Proposition \ref{growth=growth}.)
\par

Let $E$ be a subset of $L$. Then
\begin{equation}
\label{eqavecencadrervertEvert}
m_B^+(EB) \, \le \, \vert E \vert \, \le \, m_C^-(EB) .
\end{equation}
Set $D =\left( E(B^3) \smallsetminus E \right) \cap L$. We claim that
\begin{equation*}
E(B^2) \smallsetminus EB \, \subset \, DB .
\end{equation*}
Indeed, let $g \in G$, written as $g = xb$ with $x \in L$ and $b \in B$.
Assume that $g \in E(B^2)$, i.e., that $g = eb_1b_2$ 
with $e \in E$ and $b_1, b_2 \in B$;
then $x = eb_1b_2b^{-1} \in E(B^3)$.
Assume moreover that $g \notin EB$;
then $x \notin E$.
It follows that $x \in D$, and therefore that $g \in DB$.
This proves the claim.
We have therefore
\begin{equation}
\label{eqavecvertDvert}
m_B^+\left( E(B^2) \smallsetminus EB \right) 
\, \le \, m_B^+(DB) 
\, \le \, \vert D \vert 
\end{equation}
by (\ref{eqavecencadrervertEvert}).

\vskip.2cm

\emph{Step three: if $(G,d)$ is amenable, then $G$ is geometrically amenable.}
Recall that $(G,d)$ is uniformly coarsely proper,
by Proposition \ref{growth=growth},
so that amenability in the sense of
Definition \ref{defamucp} makes sense for $(G,d)$.
\par

Let $Q$ be a compact subset of $G$ and $\varepsilon > 0$;
there is no loss of generality if we assume that $1 \in Q$.
Upon increasing the radius of $B$, we can assume that $Q \subset B$.
To prove the implication, we will show that 
there exists a compact subset $F$ of $G$ 
such that $\mu(FQ \smallsetminus F) \le \varepsilon \mu(F)$.
\par

Since $L$ is amenable by hypothesis on $(G,d)$,
there exists a finite subset $E$ of $L$ such that
\begin{equation}
\label{eqLmoyennable}
\vert \{ x \in L \mid d(x,e) \le 3 \hskip.2cm \text{and} \hskip.2cm x \notin E \} \vert
\, \le \, 
\varepsilon \mu(C) \mu(B)^{-1} \vert E \vert 
\end{equation}
(see Definition \ref{defamulf}).
Observe that, for every $k \ge 1$, we have
$E(B^k) \subset \{ g \in G \mid d(g,E) \le k \}$;
in particular, with $D$ as in Step two, we have
\begin{equation}
\label{eqDcontenu}
D \, := \, 
\left( E(B^3) \smallsetminus E \right) \cap L 
\, \subset \, 
\{ x \in L \mid d(x,E) \le 3 \hskip.2cm \text{and} \hskip.2cm x \notin E \} .
\end{equation}
Hence we have
\begin{equation*}
\begin{array}{cccc}
\mu (EBQ \smallsetminus EB ) \, &\le \,
   &\mu \left( E(B^2) \smallsetminus EB \right)
   \hskip.3cm &\text{because} \hskip.2cm Q \subset B
\\
   \, &\le \, 
   &m_B^+ \left( E(B^2) \smallsetminus EB \right) \mu (B)
   \hskip.3cm &\text{see Equation (\ref{eqavecmuetB})} 
\\
   \, &\le \, 
   &\vert D \vert \mu(B)
   \hskip.3cm &\text{see Equation (\ref{eqavecvertDvert})}
\\
  \, &\le \,
   &\varepsilon \mu(C) \vert E \vert 
   \hskip.3cm &\text{see Equations (\ref{eqLmoyennable})
   and (\ref{eqDcontenu})}
\\
   \, &= \, 
   &\varepsilon \mu(EC)
   \hskip.3cm &\text{by (ii) of Step two}
\\
  \, &\le \,
  &\varepsilon \mu(EB)    
  \hskip.3cm &\text{because} \hskip.2cm C \subset B .
\end{array}
\end{equation*}  
To finish Step three, it suffices to set $F = EB$.

\vskip.2cm

\emph{Step four: more on packing and covering.}
Let $A$ be a ball of centre $1$ in $G$ 
such that $(xA)_{x \in L}$ is a covering of $G$,
indeed such that $B \subset A$.
Observe that $C \subset A$.
Let $P$ be any subset of $G$.
\par

Set $E = PA \cap L$. Note that $P \subset EA$; 
indeed, let $p \in P$ be written as $p = xa$, with $x \in L$ and $a \in A$;
then $x = pa^{-1} \in PA \cap L$, i.e., $x \in E$.
We have therefore
\begin{equation}
\label{duStepfour}
\vert E \vert \, \ge \, m_A^+(EA) 
\, \ge \, m_A^+(P)
\, \ge \, \mu (P) \mu(A)^{-1}
\end{equation}
by  the left-hand side inequality of (\ref{eqavecencadrervertEvert})
and by (\ref{eqavecmuetB}).
Since $EA \cap L \subset P(A^2) \cap L$ and $E \cap L = E = PA \cap L$,
we have 
\begin{equation*}
(EA \smallsetminus E) \cap L \, \subset \, \left( P(A^2) \smallsetminus PA \right) \cap L .
\end{equation*}
Since
$\left( P(A^2) \smallsetminus PA \right) B \subset P(A^3) \smallsetminus P$,
we have also
\begin{equation*}
((EA \smallsetminus E) \cap L)B \, \subset \,  
\big( P(A^2) \smallsetminus PA \big) B \, \subset \,  
P(A^3) \smallsetminus P ,
\end{equation*}
hence
\begin{equation}
\label{m-C-moinscommemajorant}
\vert ( EA \smallsetminus E ) \cap L \vert  \, \le \, 
m_C^-\left( \big( (EA \smallsetminus E) \cap L \big) B \right) \, \le \, 
m_C^-\left( P(A^3) \smallsetminus P \right).
\end{equation}
by the right-hand side inequality of (\ref{eqavecencadrervertEvert}).

\vskip.2cm

\emph{Step five: if $G$ is geometrically amenable, then $(G,d)$ is amenable.}
To prove the implication, we have to show that $L$ is amenable.
Denote by $r_0$ the radius of the ball $B$.
By Remark \ref{remamulf}, it suffices to show that 
for $r \ge r_0$ and $\varepsilon > 0$,
there exists a non-empty finite subset $E$ in $L$ such that
$\vert \{ x \in L \mid d(x,E) \le r \} \vert \le (1 + \varepsilon) \vert E \vert$.
\par

Let $A$ denote the closed ball of centre $1$ and radius $r$ in $G$.
For any finite subset $E$ of $L$, we have
$(EA \smallsetminus E) \cap L =
\{ x \in L \mid d(x,E) \le r \hskip.2cm \text{and} \hskip.2cm x \notin E \}$.
\par

Since $G$ is geometrically amenable, 
there exists a compact subset $P$ of $G$ of non-zero measure
such that
\begin{equation}
\label{eqfromhypGrightam}
\mu \left( P(A^3) \smallsetminus P \right) \, \le \,
\varepsilon \mu(P) \mu (C) \mu(A)^{-1} .
\end{equation}
Define $E = PA \cap L$, as in Step four; 
it is a finite subset of $L$ by Lemma \ref{discretemetriclattice}.
We have
\begin{equation*}
\label{derniereffort}
\begin{array}{cccc}
\vert (EA \smallsetminus E) \cap L \vert \, &\le \,
   &m_C^-\left( P(A^3) \smallsetminus P \right)
   \hskip1cm &\text{see Equation \ref{m-C-moinscommemajorant})}
\\
   \, &\le \, 
   &\mu \left( P(A^3) \smallsetminus P \right)  \mu(C)^{-1}
   \hskip1cm &\text{see Equation (\ref{eqavecmuetB})}
\\
   \, &\le \, 
    &\varepsilon \mu(P)  \mu(A)^{-1}
    \hskip1cm &\text{see  Equation (\ref{eqfromhypGrightam})}
\\
   \, &\le \,
   &\varepsilon \vert E \vert
   \hskip1cm &\text{see Equation (\ref{duStepfour}).}
\end{array}
\end{equation*}
It follows that $L$ is amenable.
\end{proof}

\begin{cor}
\label{invariancerightam}
For $\sigma$-compact LC-groups,
right-amenability (i.e., amenabiliy and unimodularity)
is invariant by metric coarse equivalence.
\par

In particular, for countable groups,
amenability is invariant by metric coarse equivalence.
\end{cor}

\begin{rem}
\label{remTesseraEtc}
There is in \cite{Tess--08} a discussion of right-amenability,
under the name of  ``geometric amenability''.
In this article, Tessera relates right-amenability
to other properties of metric measure spaces 
(i.e., metric spaces endowed with measures),
involving Sobolev inequalities
and spectral gaps for appropriate linear operators.
\par

Corollary \ref{invariancerightam} appears as Proposition 11 in \cite{Tess--08}.
\end{rem}

\chapter{Examples of compactly generated LC-groups}
\label{chap_excglcg}

\section[Connected, abelian, nilpotent, Lie,  and algebraic]
{Connected groups, abelian groups, nilpotent groups, Lie groups, and algebraic groups}
\label{Liegroupsandalgebraicgroups}
As observed in Proposition \ref{powersSincpgroup}, 
a connected LC-group is compactly generated.
More generally:

\begin{prop}
\label{almostconnectedgroups}
Let $G$ be an LC-group and $G_0$ its identity component.

\vskip.2cm

(1)
The group $G$ is compactly generated
if and only if the quotient group 
$G/G_0$ is compactly generated.
In particular:
\begin{enumerate}[noitemsep]
\item[--]
an LC-group that is connected-by-compact (for example connected)
is compactly generated;
\index{Connected-by-compact topological group}
\item[--]
a real or complex Lie group with finitely many connected components
is compactly generated;
\item[--]
a fortiori, the group of real points of an algebraic $\R$-group
is compactly generated (see Remark \ref{alggroupsalmostconnected}).
\end{enumerate}

\vskip.2cm

(2)
The group $G$ is $\sigma$-compact if and only if $G/G_0$ is $\sigma$-compact.

\vskip.2cm

(3)
The group $G$ is second-countable if and only if both $G_0$ and $G/G_0$ 
are second-countable.
\end{prop}

Concerning (3), note that a connected LC-group need not be second-countable.
For example, an uncountable product of a non-trivial connected compact group
is not second-countable.
\index{Product of groups}

\begin{proof}
(1)
Let $G$ be a locally compact group.
Since $G_0$ is locally compact and connected,
it is generated by any compact neighbourhood of the identity;
let $S$ be one of these.
If $G/G_0$ is compactly generated,
there exists by Lemma \ref{KimagedeK} a compact subset $T$ of $G$
of which the image in $G/G_0$ is generating.
It follows that the compact set $S \cup T$ generates $G$.
Conversely, if $G$ is  compactly generated, so is $G/G_0$.
\par

The proof of (2) is similar. 
Claim (3) is a particular case of part of Proposition \ref{stababcd}.
\end{proof}

\begin{exe}[non-discrete locally compact fields]
\label{localfieldscg}
Let $\K$ be a non-discrete locally compact field.
\par

As already observed (Example \ref{panoramalocalfield}),
the additive groups $\R$, $\C$, and the multiplicative group $\C^\times$ are connected, 
and $\R^\times$ has two connected components. 
These four groups are compactly generated.
We suppose from now on that $\K$ is a local field,
\index{Local field}
and we use the notation of \ref{panoramalocalfield}.
\par

The additive group $\K$ is the union 
$\bigsqcup_{n = 0}^{\infty} \pi^{-n} \mathfrak o_{\K}$
of a strictly increasing sequence of compact open subgroups.
Hence $\K$ is locally elliptic (see \S~\ref{local_ellipticity})
and is not compactly generated.
\index{Compactly generated! LC-group, not compactly generated}

On the contrary, $\K^\times$ is always compactly generated.
Indeed, the group $\mathfrak o_{\K}^\times$ 
of invertible elements in $\mathfrak o_{\K}$,
which is the complement of 
$\mathbf p_{\K} = \pi \mathfrak o_{\K}$ in $\mathfrak o_{\K}$,
is compact,
and $\K^\times$ is isomorphic to the direct product $\mathfrak o_{\K}^\times \times \pi^\Z$
(where $\pi^\Z$ stands for the infinite cyclic group generated by $\pi$).
\end{exe}

\begin{exe}[locally compact abelian groups]
\label{LCA}
Let $A$ be a locally compact abelian group, or for short an \textbf{LCA-group}.
\index{LCA-group|textbf}
Then, for some unique $\ell \in \N$, the group $A$ is isomorphic to a product $\R^\ell \times H$,
where $H$ is a locally compact abelian group containing a compact open subgroup
\cite{Kamp--35}, \cite[chap.\ II, $\S$ 2, no 2]{BTS1-2}.
\par

Denote by $\widehat A$ the Pontryagin dual of $A$,
which is the group of homomorphisms $A \longrightarrow \R / \Z$,
with the compact-open topology.
\index{Pontryagin dual of an LCA-group}
Then $\widehat A$ is again a locally compact abelian group 
\cite[chap.\ II, $\S$ 2, no 1]{BTS1-2}. Moreover:
\index{Topology! compact-open}
\begin{enumerate}[noitemsep,label=(\arabic*)]
\item\label{1DELCA}
$A$ is compact if and only if 
$\widehat A$ is discrete.
\item\label{2DELCA}
$A$ is compactly generated if and only if 
it is of the form $\mathbf R^\ell \times \mathbf Z^m \times K$,
where $\ell, m$ are non-negative integers
and $K$ a compact subgroup (indeed \emph{the} maximal compact subgroup),
\index{Maximal! compact subgroup}
if and only if $\widehat A$ is of the form $\R^\ell \times (\R / \Z)^m \times D$,
where $D$ is a finitely generated abelian group,
if and only if $\widehat A$ is a Lie group.
\par

In particular, $A$ and $\widehat A$ are both compactly generated
if and only if $A$ is of the form
$\R^\ell \times \Z^m \times (\R / \Z)^n \times F$,
where $F$ is a finite abelian group.
\item\label{3DELCA}
$A$ is $\sigma$-compact if and only if 
$\widehat A$ is metrizable,
if and only if $\widehat A$ has a second-countable open subgroup.
\item\label{4DELCA}
$A$ is second-countable if and only if $\widehat A$ is second-countable.
In particular, a compact LCA-group is second-countable
if and only if the discrete group $\widehat A$ is countable.
\item\label{5DELCA}
$A$ is connected 
if and only if the locally elliptic radical of $\widehat A$ is $\{0\}$,
if and only if $\widehat A$ is of the form $\R^\ell \times D$
with $\ell \ge 0$ and $D$ discrete torsion-free.
\item\label{6DELCA}
$A$ is totally disconnected if and only if
\index{Totally disconnected! topological group}
$\widehat A$ is locally elliptic.
\item\label{7DELCA}
$A$ is locally connected if and only if 
it is of the form $\R^\ell \times D \times \widehat E$ where $D$ is discrete and 
$E$ discrete locally free abelian
(i.e., such that all its finitely generated subgroups are free abelian).
\index{Locally connected group}
\end{enumerate}
See \cite{BTS1-2},
Chap.\ II, $\S$~1, no 9, and $\S$~2 
(including Exercises 1 and 8, or alternatively \cite{Dixm--57}).
\end{exe}

Before Proposition \ref{closubabcgiscg}, 
we recall the following definition, a lemma, and a remark.

\begin{defn}
\label{defcommnilp}    
Let $G$ be a group.
\par

The \textbf{commutator} of two elements $x,y \in G$ is 
$x^{-1}y^{-1}xy$, often denoted by $[x,y]$.
\index{Commutator! of two elements|textbf}
If $H$, $K$ are subgroups of $G$, we denote by $[H,K]$ the subgroup generated by commutators $[h,k]$ when $(h,k)$ ranges over $H \times K$. 
In particular, the \textbf{group of commutators} of $G$ is the subgroup
generated by all commutators, denoted by $[G,G]$, or $C^2G$.
\index{Commutator! group of commutators|textbf}
The \textbf{lower central series} of a group $G$ 
is the sequence $(C^iG)_{i \ge 1}$ of subgroups inductively defined by
$C^1G = G$ and $C^{i+1}G = [C^iG, G]$.
\par

The group $G$ is {\bf nilpotent}
\index{Nilpotent! group|textbf}
if $C^{k+1}G = \{1\}$ for some $k\in\mathbf{N}$. 
If so, the \textbf{nilpotency class} of $G$ 
is the smallest non-negative integer $k$ such that  $C^{k+1}G = \{1\}$,
and $G$ is a \textbf{nilpotent group of class $k$}.
\end{defn}

Note that a group is nilpotent of class $0$ if and only if $G = \{1\}$,
and nilpotent of class $1$ if and only if it is abelian with more than one element.

\begin{lem}
\label{GcgnilpkThenCkGcg}
Let $G$ be an LC-group; 
assume that the underlying group is nilpotent of class $k$.
\par

If $G$ is compactly generated, so is $\overline{C^k G}$.
\end{lem}

\begin{proof}
Let $G$ be a group and $S$ a subset of $G$.
Define inductively a sequence $S_i$ of subsets of $G$
by $S_1 = S$ and $S_{i+1} = \{ z \in G \mid z = [x,y] 
\hskip.2cm \text{for some} \hskip.2cm x \in S \hskip.2cm \text{and} \hskip.2cm y \in S_i \}$.
It is an easy fact to check that, if $S$ generates $G$,
then the image of $S_i$ generates $C^i G / C^{i+1}G$.
\par

For all $i \ge 1$, the mapping
\begin{equation*}
\left\{ \aligned
\left( C^i G / C^{i+1}G \right) \times \left( G / C^2G \right) \hskip.2cm
& \longrightarrow \hskip.2cm     C^{i+1}G / C^{i+2}G
\\
(x C^{i+1}G , yC^2G ) \hskip1cm
& \longmapsto \hskip.5cm [x,y]     C^{i+2} G
\endaligned \right.
\end{equation*}
is well-defined, $\Z$-bilinear, and its image generates $C^{i+1}G / C^{i+2}G$.
This is left as an exercise to the reader 
(alternatively, see \cite[Section 5.2]{Robi--96}).
If $G$ is nilpotent of class $k$, 
it follows that the map
\begin{equation*}
\left\{ \aligned
\left(  G / C^2 G \right) \times \cdots \times \left(  G / C^2 G \right) \hskip.2cm
& \longrightarrow \hskip2cm     C^k G
\\
(x_1 C^2G, \hdots, x_k C^2G)  \hskip.8cm
& \longmapsto \hskip.2cm  [x_1, [x_2, \hdots [x_{k-1}, x_k] \cdots ]]
\endaligned \right.
\end{equation*}
is well-defined, $\Z$-multilinear, and its image generates $C^k G$.
Moreover, if $S$ is a generating subset of $G$,
and $S_{\text{ab}}$ denotes its image in $G / C^2G$, 
then the image of $S_{\text{ab}} \times \cdots \times S_{\text{ab}}$ 
(with $k$ terms) generates $C^k G$.
\par

It follows that, if $G$ is compactly generated, then $C^k G$ is compactly generated.
Hence $\overline{C^k G}$ is compactly generated, by Remark \ref{remsigmacetc}(5).
\end{proof}

\begin{rem}
If $G$ is a nilpotent LC-group, the groups $C^i G$ need not be closed in $G$.
Hence, in the proof of Proposition \ref{closubabcgiscg}, 
$\overline{C^k G}$ cannot be replaced by $C^k G$.
\par
For example, let $\widetilde G = H \times (\R / \Z)$,
where $H = H(\R)$ denotes the real Heisenberg group,
see Example \ref{coarselyexpansivenotlargescaleexpansive}.
Let $c \in Z(H)$ be a non-trivial element of the centre of $H$,
and $C$ the discrete subgroup of $\widetilde G$
spanned by an element $(c, \theta)$,
where $\theta \in \R / \Z$ is the class of an irrational number $\tilde \theta \in \R$.
Set $G = \widetilde G / C$. We leave it to the reader to check that
$G$ is a connected nilpotent real Lie group, $C^3 G = \{1\}$, 
and the commutator subgroup $C^2G$ is not closed in $G$.
\par
(It can be checked that $G$ does not have any faithful continuous
finite-dimensional linear representation. Indeed, in a connected Lie group
having such a representation, the commutator subgroup is closed
\cite[Theorem XVIII.4.5]{Hoch--65}.)
\end{rem}

\begin{prop}
\label{closubabcgiscg}
(1)
In a compactly generated LCA-group,
every closed subgroup is compactly generated.
\par

(2)
More generally,
in a compactly generated nilpotent LC-group,
every closed subgroup is compactly generated.
\index{LCA-group}
\end{prop}

\noindent
\emph{Note.} In a solvable \emph{connected Lie group},
every closed subgroup is compactly generated.
See \cite[Proposition 3.8]{Ragh--72}.
\index{Solvable group}

\begin{proof}
We first check (1) for two particular cases.

\emph{Case a.} 
Suppose that $A$ is either $\R^\ell$ for some $\ell \ge 0$
or a finitely generated abelian group. Then it is standard
that every closed subgroup of $A$ is compactly generated.
\par

\emph{Case b.}
Suppose now that $A$ has an open subgroup isomorphic to $\R^\ell$ for some $\ell \ge 0$.
Without loss of generality, we can assume that $A = \R^\ell \times D$
for some finitely generated abelian group $D$.
Denote by $\pi_D : A \longrightarrow D$ the canonical projection.
If $B$ is a closed subgroup of $A$, we have an extension
\begin{equation*}
B \cap \R^\ell \,  \lhook\joinrel\relbar\joinrel\rightarrow \,
B \, \overset{\pi_D}{\relbar\joinrel\twoheadrightarrow} \, 
\pi_D(B) .
\end{equation*}
By Case a, both $B \cap \R^\ell$ and  $\pi_D(B)$
are compactly generated; hence so is $B$.
\par

\emph{General case.}
By Example \ref{LCA}\ref{2DELCA},
we can assume that $A = \R^\ell \times H$
for some $\ell \ge 0$ and some compactly generated LCA group $H$
having a compact open subgroup, say $K$.
We can identify $A/K$ to $\R^\ell \times D$, with $D$ as in Case b.
The subgroup $B+K$ of $A$ is closed, because $K$ is compact,
so that $(B+K)/K$ is compactly generated by Case b.
Hence $B+K$ is compactly generated.
Since $B$ is cocompact in $B+K$, the group $B$ is compactly generated
by Proposition \ref{sigmac+compactgofcocompact}.

\vskip.2cm

(2)
For a compactly generated nilpotent LC-group $G$,
we prove (2) by induction on the smallest integer $k$ such that $C^k G$ is central in $G$;
note that $\overline{C^k G}$ is also central in $G$,
and compactly generated by Lemma \ref{GcgnilpkThenCkGcg}.
If $k = 1$, the group $G$ is abelian, and (2) holds by (1).
Assume now that $k \ge 2$, and that (2) holds up to $k-1$.
Let $H$ be a closed subgroup of $G$.
Then $H \cap \overline{C^k G}$ is abelian,
hence compactly generated by (1), 
and $H / (H \cap \overline{C^k G})$ is compactly generated,
by the induction hypothesis.
Hence $H$ is (compactly generated)-by-(compactly generated),
so that $H$ is compactly generated,
by Proposition \ref{stababcd}\ref{4DEstababcd}.
\end{proof}

\begin{rem}
\label{NoetherianGroups}
Let us build up on Claim (2) of the previous proposition,
in the context of discrete groups.
A group $G$ is \textbf{Noetherian}
\index{Noetherian group|textbf}
if all its subgroups are finitely generated,
equivalently if, for every increasing sequence 
$H_1 \subset  \cdots \subset H_n \subset H_{n+1} \subset \cdots \subset G$,
we have $H_{n+1} = H_n$ for all $n$ large enough.
Here are two standard results: polycyclic-by-finite groups 
(in particular finitely generated nilpotent groups) are Noetherian,
and a soluble group is Noetherian if and only if it is polycyclic.
\par

Noetherian groups have been studied by Baer (\cite{Baer--56}, and further articles).
The problem to decide whether there exist
Noetherian groups that are not polycyclic-by-finite has been open for a long time,
before being solved positively by Ol'shanskii in the late 70's:
\par

There are $2$-generated infinite groups 
for which there exists a constant $N$
such that every proper subgroup is finite of cardinal at most $N$.
There are infinite torsion-free simple $2$-generated groups 
in which every proper subgroup is cyclic.
See \cite[\S\S~27-28]{Ol'sh--91}.
\end{rem}

\begin{exe}[countable discrete abelian groups]
\label{coutdiscag}
Two countable discrete abelian groups are coarsely equivalent
if and only if they satisfy the following two conditions \cite{BaHZ--10}:
\begin{enumerate}[noitemsep,label=(\arabic*)]
\item
their torsion-free ranks coincide;
\item
they are either both finitely generated or both infinitely generated.
\end{enumerate}
(Recall from Example \ref{ex_asymptoticdimension_n}
that their torsion-free ranks coincide 
if and only if their asymptotic dimensions coincide.)
As a consequence:
\begin{enumerate}[noitemsep,label=(\alph*)]
\item
Two infinite countable locally finite abelian groups are coarsely equivalent.
\item
The groups $\Q$ and $\Z[1/n]$ are coarsely equivalent, for every $n \ge 2$.
\item
Two finitely generated abelian groups are quasi-isometric
if and only if their torsion-free ranks coincide.
\item
Every countable discrete abelian group is coarsely equivalent 
to exactly one of the following groups: 
$\mathbf{Z}^n$ for some $n\ge 0$, 
$\mathbf{Q}/\mathbf{Z}$, 
$\mathbf{Q}^n$ for some $n\ge 1$,
and the direct sum $\Q^{(\N)}$ of a countable infinite number of copies of $\Q$.
\item
Moreover, every $\sigma$-compact LCA-group $A$ 
is coarsely equivalent to some countable discrete abelian group. 
Indeed, $A$ is isomorphic to $\mathbf{R}^k \times H$, 
where $H$ has a compact open subgroup $W$ (see Example \ref{LCA}). 
Then $A$ is coarsely equivalent to the discrete group $\mathbf{Z}^k\times H/W$.
\end{enumerate}
\end{exe}

\begin{exe}[affine groups]
\label{affinecg}
If $\K$ is a non-discrete locally compact field, the affine group 
\index{Affine group! $\K \rtimes {\K}^\times$}
$
G \, = \, 
\begin{pmatrix}
\K^\times & \K \\
0 & 1 
\end{pmatrix}
$
is compactly generated.
\end{exe}

\begin{proof}
We freely use the notation introduced in Example \ref{localfieldscg}.
Set
\begin{equation*}
S_1 \, = \, \left\{
\begin{pmatrix} a & 0 \\ 0 & 1 \end{pmatrix} \in G
\hskip.1cm \Big\vert \hskip.1cm
a \in \mathfrak o_{\K}^\times \cup \{\pi^{-1}\} 
\right\}
\hskip.3cm \text{and} \hskip.3cm
S_2 \, = \, \left\{
\begin{pmatrix} 1 & b \\ 0 & 1 \end{pmatrix} \in G
\hskip.1cm \Big\vert \hskip.1cm
b \in  \mathfrak o_{\K}
\right\} .
\end{equation*}
We claim that the compact subset $S := S_1S_2$ generates $G$.
Any $g \in G$ is a product $g_1g_2$ with
$g_1 = \begin{pmatrix} x & 0 \\ 0 & 1 \end{pmatrix}$
and $g_2 = \begin{pmatrix} 1 & y \\ 0 & 1 \end{pmatrix}$
for some $x \in \K^\times$ and $y \in \K$.
On the one hand, $g_1$ is in the subgroup of $G$ generated by $S_1$.
On the other hand, for $n$ large enough,
\begin{equation*}
\begin{pmatrix} \pi^{-1} & 0 \\ 0 & 1 \end{pmatrix}^n
g_2
\begin{pmatrix} \pi & 0 \\ 0 & 1 \end{pmatrix}^{n}
\, = \, 
\begin{pmatrix} 1 & \pi^{-n}y \\ 0 & 1 \end{pmatrix}
\, \in \, S_2 .
\end{equation*}
This proves the claim.
\end{proof}

\begin{exe}
\label{exSL_2}
The locally compact group $\SL_2(\Q_p)$ is compactly generated,
by the compact subset 
$S := \SL_2(\Z_p) \cup \left\{ 
\begin{pmatrix}
p & 0 \\
0 & p^{-1} 
\end{pmatrix}
\right\}$.
\par

More generally, $\SL_n(\Q_p)$ is compactly generated for every $n \ge 2$.
\index{Special linear group $\SL$! $\SL_n(\Q_p)$}
\end{exe}

\begin{proof}
Here is a first proof for the case $n=2$.
Denote by $\vert x \vert_p$ the $p$-adic absolute value of $x \in \Q_p$.
We claim that any
$g = \begin{pmatrix}
a & b \\
c & d 
\end{pmatrix}
\in  \SL_2(\Q_p)$
is in the subgroup generated by $S$.
Indeed, upon multiplying $g$ on the left or/and on the right by
$\begin{pmatrix}
0 & 1 \\
-1 & 0 
\end{pmatrix}$,
we can assume that $\vert a \vert_p = 
\max \{\vert a \vert_p, \vert b \vert_p, \vert c \vert_p, \vert d \vert_p \}$,
so that $x := - ba^{-1}$ and $y := - ca^{-1}$ are in $\Z_p$.
The computation
\begin{equation*}
\begin{pmatrix}
1 & 0 \\
y & 1 
\end{pmatrix}
\begin{pmatrix}
a & b \\
c & d 
\end{pmatrix}
\begin{pmatrix}
1 & x \\
0 & 1 
\end{pmatrix}
\, = \,
\begin{pmatrix}
1 & 0 \\
y & 1 
\end{pmatrix}
\begin{pmatrix}
a & 0 \\
c & a^{-1} 
\end{pmatrix}
\, = \,
\begin{pmatrix}
a & 0 \\
0 & a^{-1} 
\end{pmatrix}
\end{equation*}
shows that we can assume that $g$ is diagonal.
Since $a = zp^n$ for some $z \in \Z_p^\times$ and $n \in \Z$,
this shows the claim.
\par

Let us allude to a second proof.
There exists a standard action of $\SL_2(\Q_p)$ 
on the regular tree of valency $p+1$  \cite{Serr--77}.
\index{Tree}
This action is geometric, and Theorem \ref{ftggt}
implies that $\SL_2(\Q_p)$ is compactly generated.
(This tree is the particular case 
of the Bruhat-Tits building attached to $G = \SL_n (\Q_p)$.)
\par

Consider now the case $n \ge 3$.
Let $e_1, \hdots, e_n$ denote the canonical basis of $\Q_p^n$.
For $i,j \in \{1, \hdots, n\}$ with $i<j$, 
let $\varepsilon_{i,j} : \SL_2(\Q_p) \longrightarrow \SL_n(\Q_p)$
denote the canonical embedding:
$\varepsilon_{i,j}(g)$ acts on the linear span of $\{e_i, e_j\}$
as $\SL_2(\Q_p)$ acts on that of $\{e_1,e_2\}$,
and $\varepsilon_{i,j}(g)e_k = e_k$ if $k \notin \{i,j\}$.
Since elementary matrices generate $\SL_n(\Q_p)$,
\index{Elementary matrices}
the union $\bigcup_{i,j \text{ with } 1 \le i < j \le n} \varepsilon_{i,j}(S)$
is a compact generating set of $\SL_n(\Q_p)$.
\par

The result and the proof carry over with minor changes
when $\Q_p$ is replaced by an arbitrary local field.
\end{proof}

\begin{thm}
\index{Theorem! Borel-Tits}
\label{reductivecg}
Let $\mathbf G$ be a reductive group defined over a local field $\K$.
\par
Then the group of $\K$-points of $\mathbf G$ is compactly generated.
\end{thm}

\begin{proof}[References for the terminology and the proof]
An algebraic group $\mathbf G$ is \textbf{reductive} 
\index{Reductive group|textbf}
if its unipotent radical is trivial,
in other words if $\mathbf G$ does not contain 
any non-trivial connected unipotent normal algebraic subgroup
(for further explanations, we refer to \cite[Section 11]{Bore--91}).
This notion does not depend on the field of definition.
\par

For local fields of characteristic $0$, 
Theorem \ref{reductivecg} is a particular case of \cite[Th\'eor\`eme 13.4]{BoTi--65}.
For arbitrary characteristic, see \cite{Behr--69}.
See also \cite[Corollary I.2.3.5]{Marg--91} for the compact generation 
of the group of $\K$-points of a \emph{semisimple} group defined over $\K$. 
\par

More is true: see Theorem \ref{TheoremBehr}.
\end{proof}

Theorem \ref{reductivecg} could be stated when $\K$ is any
non-discrete locally compact field of characteristic $0$.
Indeed, in the remaining case, when $\K$ is $\R$ or $\C$, 
for \emph{any} algebraic group $\mathbf G$ defined over $\K$,
the group $\mathbf G (\K)$ is a Lie group with finitely many connected components,
and it is compactly generated regardless 
of the assumption that $\mathbf G$ is reductive.

\begin{thm}
\label{cgforalgebraicgroupschar0}
\index{Compactly generated! LC-group}
Let $\K$ be a local field of characteristic $0$.
Let $\mathbf G$ be an algebraic group, with connected component $\mathbf G_0$;
we assume that $\mathbf G$ is defined over $\K$.
Let $G$ denote the locally compact group $\mathbf G (\K)$ 
of the $\K$-points of $\mathbf G$.
The following properties are equivalent:
\begin{enumerate}[noitemsep,label=(\roman*)]
\item\label{iDEcgforalgebraicgroupschar0}
$G$ is not compactly generated;
\item\label{iiDEcgforalgebraicgroupschar0}
$G$ has a non-compact locally elliptic quotient;
\item\label{iiiDEcgforalgebraicgroupschar0}
$\mathbf G$ has a connected normal $\K$-subgroup $\mathbf N$
such that $\mathbf G_0 / \mathbf N$ is 
a reductive group anisotropic over $\K$,
and $\mathbf N$ admits a surjective $\K$-homomorphism
onto the additive $\K$-group of dimension $1$.
\end{enumerate}
\end{thm}

\begin{proof}[On the proof]
Implication \ref{iiDEcgforalgebraicgroupschar0} $\Rightarrow$ \ref{iDEcgforalgebraicgroupschar0} 
is straightforward 
(see Proposition \ref{equivalences_locell}).
The deepest implication is  \ref{iDEcgforalgebraicgroupschar0} $\Rightarrow$ 
\ref{iiiDEcgforalgebraicgroupschar0};
it follows from \cite[Theorem 13.4]{BoTi--65}.
Implication \ref{iiiDEcgforalgebraicgroupschar0} $\Rightarrow$
\ref{iiDEcgforalgebraicgroupschar0}
is also essentially in \cite{BoTi--65};
more precisely, suppose that \ref{iiiDEcgforalgebraicgroupschar0} holds;
let $\mathbf V$ be the intersection of the kernels
of the $\K$-morphisms from $\mathbf N$ 
to the additive $\K$-group of dimension $1$;
if $N = \mathbf N (\K)$ and $V = \mathbf V (\K)$, 
then $G/N$ is compact and $N/V$ is a non-compact locally elliptic group,
so that \ref{iiDEcgforalgebraicgroupschar0} holds.
\end{proof}

It is unknown to us whether the equivalence between 
\ref{iDEcgforalgebraicgroupschar0}  and \ref{iiDEcgforalgebraicgroupschar0} holds 
when $\K$ is a local field of positive characteristic.

\section{Isometry groups}
\label{isometrygroups}

We want to show that the isometry group of a proper metric space
has a natural topology that makes it a second-countable LC-group,
and that the action of the isometry group on this space is continuous and proper.
\par

Consider a metric space $X = (X,d)$. Its \textbf{isometry group},
\index{Isometry group $\Isom (X)$|textbf}
consisting of all invertible isometries of $X$ onto itself,
is denoted by $\Isom (X)$, or  $\Isom (X,d)$ whenever necessary.
There are three standard and equivalent ways to define
the appropriate topology on $\Isom (X)$. 
The \textbf{compact-open topology} on $\Isom (X)$
\index{Compact-open topology|textbf}
\index{Topology! compact-open|textbf}
\index{Topology! pointwise|textbf}
\index{Topology! uniform convergence on compact sets,
ucc-topology|textbf}
is that generated by the subbasis 
\begin{equation*}
\aligned
\mathcal O_{K,V} \, = \,  &\{ f \in \Isom (X) \mid f(K) \subset V \} 
\\
&\text{for} \hskip.2cm K \subset X \hskip.2cm \text{compact} \hskip.2cm
\text{and} \hskip.2cm V \subset X \hskip.2cm \text{open.} 
\endaligned
\end{equation*}
The \textbf{pointwise topology} on $\Isom (X)$
\index{Pointwise topology|textbf}
is that generated by the subbasis 
\begin{equation*}
\aligned
\mathcal O_{x,V} \, = \,  &\{ f \in \Isom (X) \mid f(x) \in V \} 
\\
&\text{for} \hskip.2cm x \in X  \hskip.2cm
\text{and} \hskip.2cm V \subset X \hskip.2cm \text{open.} 
\endaligned
\end{equation*}
The \textbf{topology of uniform convergence on compact subsets}, 
or \textbf{ucc-topology}, on $\Isom (X)$
\index{Ucc-topology|textbf}
is that generated by the subbasis
\begin{equation*}
\aligned
\mathcal O_{f_0,K,\varepsilon} \, = \, &\Big\{f \in \Isom(X) \mid
\sup_{x \in K}  d_Y(f_0(x),f(x)) < \varepsilon  \Big\}
\\
&\text{for} \hskip.2cm f_0 \in \Isom (X), \hskip.2cm K \subset X \hskip.2cm 
\text{compact, and} \hskip.2cm  \varepsilon > 0. 
\endaligned
\end{equation*}
Elements of the three subbasis just defined will abusively be called
``basic'' open sets or ``basic'' neighbourhoods.

\begin{lem}
\label{co=ucc for Isom(X)}
Let $X$ be a metric space.
The compact-open topology and the ucc-topology coincide on $\Isom (X)$.
\end{lem}

\begin{proof}
Let $f_0 \in \Isom (X)$. 
\par

Let $K \subset X$ be a compact subset and $V \subset X$
an open subset of $X$ such that $\mathcal O_{K,V}$ 
is a basic compact-open neighbourhood of $f_0$.
There exists $\varepsilon > 0$ such that 
$\{ x \in X \mid d(f_0(K), x) < \varepsilon \} \subset V$.
Hence $\mathcal O_{f_0, K, \varepsilon}$ is an ucc-neighbourhood of $f_0$
contained in $\mathcal O_{K,V}$.
\par

Conversely, let $K \subset V$ be a compact subset of $X$ and $\varepsilon > 0$,
so that $\mathcal O_{f_0, K, \varepsilon}$ is a basic ucc-neighbourhood of $f_0$.
There exists a finite subset $\{x_1, \hdots, x_n\}$ of $K$ such that
the interiors of the balls $B^{x_i}_X(\varepsilon / 3)$ cover $K$.
For $i = 1, \hdots, n$, set 
\begin{equation*}
K_i \, = \, K \cap  \Big(\hskip.1cm  \overline{B^{ x_i}_X ( \varepsilon / 3)} \hskip.1cm \Big)
\hskip.2cm \text{and} \hskip.2cm
V_i \, = \, \operatorname{int}\left( B^{ f_0(x_i)}_X ( \varepsilon / 3) \right).
\end{equation*}
Observe that 
\begin{equation*}
\mathcal O ' \, = \, \bigcap_{i=1}^n \mathcal O_{\{ x_i \}, V_i}
\end{equation*}
is a compact-open neighbourhood of $f_0$.
Let $f \in \mathcal O'$. 
For $x \in K$, if $i$ is such that $x \in K_i$, we have
\begin{equation*}
\aligned
d (f_0(x),f(x)) \, &\le \, 
d (f_0(x), f_0(x_i)) + d (f_0(x_i), f(x_i)) + d(f(x_i), f(x))
\\
\, &= \, d(f_0(x_i), f(x_i)) + 2 d(x_i, x) 
\, < \, 
\frac{\varepsilon}{3} + \frac{2\varepsilon}{3}  = \varepsilon,
\endaligned
\end{equation*}
and therefore $f \in \mathcal O_{f_0,K,\varepsilon}$. 
Hence $\mathcal O'$ is a compact-open neighbourhood of $f_0$
contained in $\mathcal O_{f_0,K,\varepsilon}$.
\end{proof}

\begin{lem}
\label{ucc=pw for Isom(X)}
Let $X$ be a metric space.
The ucc-topology and the pointwise topology coincide on $\Isom (X)$.
\end{lem}

\begin{proof}
Consider $f_0 \in \Isom (X)$
and a basic ucc-neighbourhood $\mathcal O_{f_0,K,\varepsilon}$ 
of $f_0$ in $\Isom (X)$.
There exists a finite subset $\{x_1, \hdots, x_n\}$ of $K$
such that the interiors of the balls 
$B^{x_i}_X (\varepsilon/3)$ cover $K$.
The set
\begin{equation*}
U = \{ f \in \Isom (X) \mid 
d (f_0(x_i), f(x_i)) < \varepsilon/3 \hskip.2cm \text{for} \hskip.2cm i = 1, \hdots, n \}
\end{equation*}
is a neighbourhood of $f_0$ in $\Isom (X)$ for the pointwise topology.
Let $f \in U$; for $x \in K$, if $i$ is such that $x \in B^{x_i}_X (\varepsilon/3)$,
we have
\begin{equation*}
d (f_0(x),f(x)) \, \le \, 
d (f_0(x), f_0(x_i)) + d (f_0(x_i), f(x_i)) + d (f(x_i), f(x)) 
\, < \, \varepsilon,
\end{equation*}
as in the previous proof,
and therefore $f \in \mathcal O_{f_0,K,\varepsilon}$. 
Hence $f_0 \in U \subset \mathcal O_{f_0,K,\varepsilon}$.
\par

Conversely, any basic neighbourhood of $f_0$ for the pointwise topology
contains obviously an ucc-neighbourhood of $f_0$.
\end{proof}

From now on, we consider the group $\Isom (X)$ furnished
with the compact-open topology, equivalently the pointwise topology, 
equivalently the ucc-topology.

\begin{lem}
\label{Isom(X)istopgroup}
Let $X$ be a metric space.
With the topology defined above, $\Isom (X)$ is a topological group.
\end{lem}

\begin{proof}
Let $f,g \in \Isom (X)$. 
Let $K$ be a compact subset and $W$ an open subset  of $X$
such that $gf(K) \subset W$, so that $\mathcal O_{K,W}$
is a compact-open neighbourhood of $gf$ in $\Isom (X)$. 
There exists a relatively compact-open neighbourhood $V$ of $f(K)$
such that $f(K) \subset V \subset \overline{V} \subset g^{-1}(W)$.
Then $\mathcal O_{K,V} \times \mathcal O_{\overline{V},W}$
is an open neighbourhood of $(f,g)$ in $\Isom (X) \times \Isom (X)$
with image by the multiplication map in $\mathcal O_{K,W}$.
Hence the multiplication is continuous.
\par

Let $f_0 \in \Isom(X)$.
Consider a compact subset $L$ of $X$, a number $\varepsilon > 0$,
and the resulting ucc-neighbourhood $\mathcal O_{f_0^{-1}, L, \varepsilon}$ of $f_0^{-1}$.
Set $K = f_0^{-1}(L)$.
For $f \in \mathcal O_{f_0, K, \varepsilon}$, we have
\begin{equation*}
\aligned
\sup_{y \in L} d(f_0^{-1}(y), f^{-1}(y)) \, &= \, 
\sup_{x \in K} d(x,  f^{-1} f_0 (x)) 
\\
\, &= \,
\sup_{x \in K} d(f(x), f_0(x) ) \, < \, \varepsilon \hskip.1cm ,
\endaligned
\end{equation*}
and therefore $f^{-1} \in \mathcal O_{f_0^{-1}, L, \varepsilon}$.
Hence the map $f \mapsto f^{-1}$ is continuous on $\Isom (X)$.
\par

In case we know the topology of $\Isom (X)$ is locally compact
(see Lemma \ref{ActionIsom(X)onX}),
another argument rests on the following general fact:
in a group $G$ with a locally compact topology $\mathcal T$,
if the multiplication is continuous (even separately continuous), 
then the inverse is also continuous,
and $(G, \mathcal T)$ is a topological group \cite{Elli--57}.
\end{proof}

\begin{lem}
\label{ActionIsom(X)onX}
Let $X$ be a \emph{proper} metric space.
\begin{enumerate}[noitemsep,label=(\arabic*)]
\item\label{1DEActionIsom(X)onX}
The topology defined above on $\Isom (X)$ is second-countable and locally compact.
\item\label{2DEActionIsom(X)onX}
The natural action of $\Isom (X)$ on $X$
is continuous and proper.
\end{enumerate}
\index{Proper! action}
\end{lem}

\begin{proof}
\ref{1DEActionIsom(X)onX}
Let $(U_i)_{i \in I}$ be a countable basis of open sets in $X$,
all of them being relatively compact.
Then $(\mathcal O_{\overline{U_i}, U_j})_{i \in I, j \in J}$
is a countable subbasis of open sets for the compact-open topology on $\Isom (X)$.
If necessary, details of the proof can be read in \cite[Chapter IV, Lemma 2.1]{Helg--62}.
\par

Consider an isometry $f  \in \Isom (X)$, a non-empty compact subset $K$ of $X$,
and a relatively compact open subset $V$ of $X$ containing $f(K)$.
Then $\mathcal O_{K,V}$ is an open neighbourhood of $f$.
It follows from the Arezel\`a-Ascoli theorem that $\mathcal O_{K,V}$
is relatively compact \cite[Page X.17]{BTG5-10}.
Hence $\Isom (X)$ is locally compact.
\par

\ref{2DEActionIsom(X)onX}
Let $F : \Isom (X) \times X \longrightarrow X, \hskip.2cm (f,x) \longmapsto f(x)$
denote the evaluation map.
Let $f \in \Isom (X)$ and $x \in X$.
Let $V$ be an open neighbourhood of $f(x)$ in $X$.
Since $X$ is proper, there exists
a relatively compact and open neighbourhood $U$ of $x$ such that $f(\overline U) \subset V$.
Then $\mathcal O_{\overline{U},V} \times U$ 
is an open neighbourhood of $(f,x)$ in $\Isom (X) \times X$, 
and its image by $F$ is inside $V$.
Hence the natural action of $\Isom (X)$ on $X$ is continuous.
\par

The natural map 
\begin{equation*}
\Phi \, : \, 
\Isom (X) \times X \longrightarrow X \times X, \hskip.2cm 
(f,x) \longmapsto (f(x),x) = (F(f,x),x)
\end{equation*}
is closed, as one can check thinking of the ucc-topology on $\Isom (X)$.
To show that the action $F$ is proper, i.e., that the map $\Phi$ is proper,
it remains to see that, for compact subsets $K,L$ of $X$,
the inverse image $\Phi^{-1}(K \times L)$ is compact in $\Isom (X) \times X$;
for this, we think of the compact-open topology on $\Isom (X)$.
\par

Let $D$ denote the diameter of $L$. 
Set $V = \{ v \in X \mid d(v, K) < D+1 \}$.
For $(f,x) \in \Phi^{-1}(K \times L)$, we have $f(x) \in K$, so that 
\begin{equation*}
f(K) \subset  \{ y \in X \mid d(y,K) \le D \}  \subset V .
\end{equation*}
Hence $\Phi^{-1}(K \times L)$ is contained in the set $\mathcal O_{K,V} \times L$,
which is relatively compact by the Arzel\`a-Ascoli theorem.
\end{proof}

For the following proposition, recall that a proper metric-space
is second-countable and separable
(Theorem \ref{Urysohn}).
l
\begin{prop}
\label{MainIsom(X)}
Let $(X, d_X)$ be a proper metric space.
\begin{enumerate}[noitemsep,label=(\arabic*)]
\item\label{1DEMainIsom(X)}
On the group $\Isom (X)$ of isometries of $X$, 
the following three topologies coincide:
the compact-open topology, the pointwise topology, 
and the topology of uniform convergence on compact sets.
\index{Topological group}
\index{Topology! on isometry groups}
\item\label{2DEMainIsom(X)}
For this topology, $\Isom (X)$ is a second-countable LC-group;
moreover, the natural action of $\Isom (X)$ on $X$ is continuous and proper.
\item\label{3DEMainIsom(X)}
Let $x_0$ be a base point in $X$. 
The function $d$ defined on $\Isom (X) \times \Isom (X)$ by
\begin{equation*}
d(f,g) \, = \, \sup_{x \in X} d_X(f(x), g(x)) e^{-d_X(x_0, x)}
\end{equation*}
is a left-invariant proper compatible metric on $\Isom (X)$.
\item\label{4DEMainIsom(X)}
If $X$ is compact, the topological group $\Isom(X)$ is compact,
and the function $d'$ defined by
\begin{equation*}
d'(f,g) \, = \, \sup_{x \in X} d_X(f(x), g(x)) 
\end{equation*}
is a left- and right-invariant proper compatible metric on $\Isom (X)$.
\end{enumerate}
\end{prop}

\begin{proof}
Claims \ref{1DEMainIsom(X)} and \ref{2DEMainIsom(X)} sum up the conclusions 
of Lemmas \ref{co=ucc for Isom(X)} to \ref{ActionIsom(X)onX}.
\par

Since $\Isom (X)$ is locally compact and second-countable, 
it is also $\sigma$-compact and metrizable (Theorem \ref{Urysohn}).
It follows from Theorem  \ref{Struble} that there exists on $\Isom (X)$ 
a metric which is left-invariant, proper, and compatible.
\par
In \ref{4DEMainIsom(X)},
compactness is a consequence of the  Arzel\`a-Ascoli theorem,
and the formula for $d'$ is rather standard.
In \ref{3DEMainIsom(X)},
the formula for $d$ is that of Busemann.
We leave it to the reader to check the details; 
see \cite[Section 4]{Buse--55}, or the more recent exposition \cite[Section 4.4]{Papa--05}.
\par
In specific cases, other metrics are also of interest: see for example \ref{metricsonAut(X)}.
\end{proof}

\begin{prop}
\label{GinIsom(X)}
Let $G$ be a topological group,
$X$ a proper metric space,
$\Isom (X)$ its isometry group as in Proposition \ref{MainIsom(X)}, 
$\alpha : G \times X \longrightarrow X$ a continuous isometric action,
and $\rho : G \longrightarrow \Isom(X)$ the corresponding homomorphism.
\begin{enumerate}[noitemsep,label=(\arabic*)]
\item\label{1DEGinIsom(X)}
$\alpha$ is continuous if and only if $\rho$ is continuous.
\end{enumerate}
Assume moreover that the group $G$ is locally compact 
and the action $\alpha$ is proper.
\begin{enumerate}[noitemsep,label=(\arabic*)]
\addtocounter{enumi}{1}
\item\label{2DEGinIsom(X)}
The image $\rho(G)$ is a closed subgroup of $\Isom (X)$.
In particular, when the action $\alpha$ is faithful,
the homomorphism $\rho$ identifies $G$ with a closed subgroup of $\Isom (X)$.
\end{enumerate}
Assume moreover that the group $G$ is $\sigma$-compact. 
\begin{enumerate}[noitemsep,label=(\arabic*)]
\addtocounter{enumi}{2}
\item\label{3DEGinIsom(X)}
The homomorphism $G \longrightarrow \rho(G)$ is open.
\end{enumerate}
\end{prop}

\begin{proof}
\ref{1DEGinIsom(X)}
If $\alpha$ is continuous, the continuity of $\rho$ follows from the definition
of the pointwise topology on $\Isom(X)$.
If $\rho$ is continuous, it is obvious that $\alpha$ is separately continuous;
since isometries are equicontinuous, separate continuity implies continuity.
\par

\ref{2DEGinIsom(X)}
Consider the maps
\begin{equation*}
\begin{array}{cccc}
\rho \times \operatorname{id}_X \, &: \, 
&G \times X \longrightarrow \Isom (X) \times X , \hskip.2cm &(g,x) \longmapsto (\rho(g),x)
\\
\Phi \, &: \, 
&\Isom (X) \times X \longrightarrow X \times X , \hskip.2cm &(h,x) \longmapsto (h(x),x)
\\
\alpha \times \operatorname{id}_X \, &: \, 
&G \times X \longrightarrow X \times X , \hskip.2cm &(g,x) \longmapsto (\rho(g)(x),x)
\end{array}
\end{equation*}
and observe that
\begin{equation*}
\alpha \times \operatorname{id}_X \,  = \, 
\Phi \circ (\rho \times \operatorname{id}_X) .
\end{equation*}
We know that $\Phi$ is proper (Lemma 5.B.4).
By \cite[Page I.73]{BTG1-4},
it follows that $\rho \times \operatorname{id}_X$ is proper 
if and only if $\alpha \times \operatorname{id}_X$ is proper.
In particular, if the action $\alpha$ is proper, 
then $\rho \times \operatorname{id}_X$ is proper, 
hence $\rho$ is a closed map, and $\rho(G)$ is closed in $\Isom (X)$.
\par

(Note that the condition of properness of the action cannot be omitted,
as shown by actions $\Z \times \R/\Z \longrightarrow \R/\Z$
by irrational rotations.)
\par

\ref{3DEGinIsom(X)}
This follows by Corollary \ref{Freudenthalcor}.
\end{proof}

\begin{rem}
\label{Isonotclosed}
Let $X$ be a metric space.
\par

(1)
An isometry $f$ from $X$ into $X$ need not be surjective.
This is the case of the one-sided shift $f$ on
the set $\N$ of non-negative integers,
defined by $f(n) = n+1$ for all $n \in \N$,
where $\N$ has the $\{0,1\}$-valued metric, defined by $d(m,n) = 1$ whenever $m \ne n$.
The shift is a non-surjective isometry.
(The same fact holds for the usual metric, defined by $(m,n) \longmapsto \vert m-n \vert$.)

\vskip.2cm

(2)
The group $\Isom(X)$ need not be closed in the space $\Isom(X,X)$
of isometric maps from $X$ into $X$, say for the pointwise topology.
Indeed, consider again $\mathbf{N}$ with the $\{0,1\}$-valued metric
and the same one-sided shift $f$.
For every integer $k \ge 0$, define $f_k \in \Isom(\N, \N)$ by
\begin{equation*}
f_k(n) \, = \, \left\{
\aligned
n+1 \hskip.5cm &\text{if} \hskip.2cm 0 \le n \le k,
\\
0     \hskip.8cm &\text{if} \hskip.2cm n = k+1, 
\\
n     \hskip.8cm &\text{if} \hskip.2cm n \ge k+2 .
\endaligned
\right.
\end{equation*}
The sequence $(f_k)_{k \ge 0}$ converges to $f$ 
for the pointwise topology, each $f_k$ is invertible, and $f$ is not.
\end{rem}

\begin{exe}[the symmetric group of an infinite countable set]
\label{Sym(X)andallthat}
Consider the set $\N$ of natural integers, with the discrete topology.
Let $X$ be the infinite product $\N^{\N}$, with the product topology,
i.e., with the pointwise topology.
This is a metrizable space, and a compatible metric $d_X$ 
can be defined by $d_X(f,g) = 2^{-n}$,
where $n$ is the smallest non-negative integer such that $f(n) \ne g(n)$, for $f \ne g$.
The metric space $(X,d_X)$ is clearly separable and complete;
thus $X$ is a Polish space.
\par

For the $\{0,1\}$-valued metric $d$ on $\N$,
as in the previous remark,
the group $\Isom(\N) = \operatorname{Sym}(\N)$ 
\index{Symmetric group $\operatorname{Sym}(X)$ of a set $X$}
of all permutations of $\N$ is a subspace of $X$.
For the subspace topology, it is also a Polish space.
Indeed,
the metric $(f,g) \longmapsto d_X(f,g) + d_X(f^{-1},g^{-1})$
is compatible and complete
(note that it \emph{is not} a left-invariant metric).
\par

The pointwise topology makes $\operatorname{Sym}(\N)$ a topological group
(Lemma \ref{Isom(X)istopgroup}); thus it is a Polish group.
(Incidentally, it is the unique topology making $\operatorname{Sym}(\N)$ a Polish group
\cite{Gaug--67, Kall--79},
indeed the unique topology making it a non-discrete Hausdorff topological group 
\cite{KeRo--07}.)
\par

Let $\mathcal F$ denote the family of finite subsets of $\N$.
For $F \in \mathcal F$, set
\begin{equation*}
\operatorname{Sym}(\N \smallsetminus F) \, = \, 
\{ g \in \operatorname{Sym}(\N) \mid
gx = x \hskip.2cm \text{for all} \hskip.2cm x \in F \} ;
\end{equation*}
this is an open subgroup of $\operatorname{Sym}(\N)$.
Observe that the group $\operatorname{Sym}(\N \smallsetminus F)$ 
is not compact because, 
if $(y_j)_{1 \le j < \infty}$ is an enumeration of $\N \smallsetminus F$,
it is the infinite disjoint union of the open subsets
$\{ g \in \operatorname{Sym}(\N \smallsetminus F) \mid gy_1 = y_j \}$, for $1 \le j < \infty$.
The family $\left( \operatorname{Sym}(\N \smallsetminus F) \right)_{F \in \mathcal F}$
is a basis of closed neighbourhoods of $1$ in $\operatorname{Sym}(\N)$.
As none of them is compact, 
the group $\operatorname{Sym}(\N)$ is not locally compact.
Note also that the evaluation map 
$ \operatorname{Sym}(\N) \times \N \longrightarrow \N$
is not proper,
because isotropy groups of points 
are not compact in $\operatorname{Sym}(\N)$.
\par

Consider again the metric $d_X$, 
now restricted to the group $\operatorname{Sym}(\N)$;
it is a left-invariant compatible metric.
The sequence of Remark \ref{Isonotclosed} is a left Cauchy sequence
(in the sense of Remark \ref{onmetrizability}(3))
which does not converge in $\operatorname{Sym}(\N)$.
\index{Left Cauchy sequence in a topological group}
In particular, $(f_k)_{k \ge 0}$ is a non-converging Cauchy sequence 
for every left-invariant compatible metric on $\mathrm{Sym}(\N)$. 
It follows that there is no complete left-invariant compatible metric on
the Polish group $\operatorname{Sym}(\N)$; thus, $\operatorname{Sym}(\N)$
is not a cli-group.
\index{Cli-group} \index{Polish! group} \index{Polish! cli-group}
\par

The example is classical: it essentially appears in \cite{Dieu--44}.
\end{exe}

\begin{rem}
\label{Sym(X)andallthatbis}
(1) Here is one more example showing that
the hypothesis of properness on $X$ cannot be omitted in Proposition \ref{MainIsom(X)}.
\par

Let $X$ be the subset of the Euclidean space $\R^2$
consisting of the first axis and the point $(0,1)$. 
Let $d_E$ denote the Euclidean metric on $X$.
Define another metric $d$ on $X$ by $d(x,y) = \inf \{1, d_E(x,y)\}$.
Then $\Isom (X,d)$ is locally compact, but its action on $X$ is not proper,
because the isotropy subgroup of $(0,1)$ contains the translations of the first axis
and is therefore non-compact
(example from \cite{MaSt--03}).

\vskip.2cm

(2)
The following result of van Dantzig and van der Waerden 
is of historical interest:
if $(X,d)$ is a locally compact metric space that is \emph{connected}
(but not necessarily proper),
the group of isometries $\Isom(X,d)$ is locally compact and acts properly on $X$.
See \cite{DaWa--28} and  \cite[Chapter I, Theorem 4.7]{KoNo--63}.
This carries over to the so-called
\emph{pseudo-connected}  separable spaces   \cite[Proposition 5.3]{GaKe--03}.
\end{rem}

\begin{prop}
\label{IsomXcg}
Let $X$ be a non-empty proper metric space.
Let $\Isom(X)$ be its group of isometries, as in Proposition \ref{MainIsom(X)}. 
Assume moreover that the action of $\Isom(X)$ on $X$ is cocompact.
\par

Then this action is geometric.
In particular,  $\Isom(X)$ is compactly generated if and only if 
$X$ is coarsely connected.
\end{prop}

\begin{proof}
For every proper metric space, the action of $\Isom(X)$ on $X$
is continuous, metrically proper 
(Proposition \ref{MainIsom(X)} and Remark \ref{firstexamplesgeometricactions}(6)),
and locally bounded (Remark \ref{firstexamplesgeometricactions}(4)).
Hence, when the action of $\Isom(X)$ on $X$ is cocompact as assumed here,
it is geometric.
Therefore, the proposition follows 
from Theorem \ref{ftggt} and Proposition \ref{metric_char_cg}.
\end{proof}

\begin{exe}[automorphism groups of connected graphs and polyhedral complexes]
\label{Aut(X)=Isom(X,d)}
Let $X$ be a connected graph, 
$X^0$ its vertex set,
and $d_1$ the combinatorial metric on $X$ 
(Example \ref{metricrealizationgraph}).
\index{Graph} 
\index{Combinatorial! metric on a connected graph}
Assume for simplicity that $X$ has neither loops nor multiple edges;
an automorphism of $X$ is then a permutation $\alpha$ of $X^0$ such that,
for $x,y \in X^0$, there is an edge connecting $\alpha(x)$ and $\alpha(y)$
if and only if there is one connecting $x$ and $y$.
\par

The automorphism group $\operatorname{Aut}(X)$ of the connected graph $X$ coincides 
\index{Automorphism group! of a graph|textbf}
with $\Isom(X^0, d_1)$.
\par

For a subset $Y$ of $X^0$, set 
$V_Y = \{ g \in \operatorname{Aut}(X) \mid gy = y \hskip.2cm \forall y \in Y \}$.
Let $\mathcal F$ be the set of finite subsets of $X$,
and set $\mathcal N = (V_Y)_{Y \in \mathcal F}$;
let $\mathcal B$ be the set of bounded subsets of $X$,
and set $\mathcal N' = (V_Y)_{Y \in \mathcal B}$.
Then $\mathcal N$ [respectively $\mathcal N'$] 
is a basis of neighbourhoods of the identity 
for a group topology $\mathcal T$ [respectively $\mathcal T'$]
on $\Isom(X^0, d_1)$, i.e., on $\operatorname{Aut}(X)$.
\par

Suppose now that the graph $X$ is locally finite.
Then the topologies $\mathcal T$ and $\mathcal T'$ coincide with each other,
and also with the topology of Proposition \ref{MainIsom(X)}, 
for which $\operatorname{Aut}(X)$ is a second-countable LC-group
(this is in \cite[2.5]{Tits--70}, for $X$ a tree).
\par

Suppose moreover that the action of $\operatorname{Aut}(X)$ on $X$ is cobounded.
Then this action is geometric (see Example \ref{examplesgeometricactions}(3)), 
and it follows that 
the group $\operatorname{Aut}(X)$ is compactly generated (Corollary \ref{ahahah}).
In particular, if $T_k$ denotes a regular tree of some valency $k \ge 0$,
then $\operatorname{Aut}(T_k)$ is compactly generated.

\vskip.2cm

More generally, the group of automorphisms of a connected locally finite polyhedral complex
is an LC-group for its natural topology (say for the compact-open topology),
in which the stabilizer of each cell is a compact open subgroup.
\index{Automorphism group! of a polyhedral complex}
A discussion on these groups has appeared in \cite{FaHT--11}.
\end{exe}

\begin{exe}[isometry groups of countable groups]
\label{finitudeIsom(G,d)}
Let $G$ be a countable discrete group, 
given together with a proper left-invariant metric $d$.
Then $\Isom(G, d)$ is compactly generated if and only if $G$ is finitely generated;
this follows from Theorem \ref{ftggt}, 
because it is easy to check that the natural action of $\Isom (G, d)$ on $G$ is geometric.
\par

We have $\Isom(G, d) = GK$, 
where $G$ is naturally embedded as the group of left translations, 
and $K$ is the stabilizer of $1$ in $\Isom(G, d)$. 
The group $\Isom(G, d)$ is discrete if and only if the profinite group $K$ is finite.
\index{Profinite group}
\par

Consider for example the free abelian group $\Z^n$ on a $\Z$-basis 
$B = \{e_1, \hdots, e_n\}$ and the word metric $d_B$.
\index{Free abelian group}
Then $\Isom(\Z^n, d_B)$ is finitely generated;
indeed, it is the semi-direct product $\Z^n \rtimes F$,
where $F$ is the finite symmetry group of the polytope in $\R^n$
with set of vertices $\{\pm e_1, \hdots, \pm e_n \}$.
On the contrary, let $F$ be a free group on a finite basis $S$
of at least two elements, with the word metric $d_S$.
\index{Free group}
Then $\Isom(F, d_S)$ is a non-discrete compactly generated group;
indeed,  $\Isom(F, d_S)$ is the group written $\operatorname{Aut}(T_k)$
in Example \ref{extdLCgps}(5), for $k$ twice the rank of $F$.
\end{exe}

\begin{exe}[isometry groups of manifolds]
Let $X$ be a 
connected Riemannian manifold.
Its group of isometries $\Isom(X)$ is a Lie group,
possibly with infinitely many connected components.
Moreover, if $n$ denotes the dimension of $X$, then $\dim (\Isom (X)) \le \frac{1}{2}n(n+1)$.
See \cite{MySt--39}, as well as 
\cite[Chapter VI, Section 3]{KoNo--63}.
\par

For a generic Riemannian structure 
on a given connected closed smooth manifold $M$ of dimension at least $2$,
the corresponding isometry group is discrete, indeed trivial.
More precisely, on the space $\mathcal M$ of all smooth Riemannian structures on $M$,
there is a natural topology (that of uniform convergence of the Riemannian metric tensor and of all its derivatives)
such that the subspace of those $g \in \mathcal M$
for which the isometry group $\Isom (M,g)$ is $\{\operatorname{id}_M\}$
is open dense  \cite{Ebin--70}.
\par

Examples of Riemannian manifolds $(M, g)$ with $\Isom(M, g)$ non-discrete include
homogeneous spaces $G/K$, with $G$ a connected Lie group, $K$ a compact subgroup,
and $g$ appropriate,
as well as various bundles with homogeneous fibers.
When $(\widetilde M, \widetilde g)$ is the universal cover of a closed Riemannian manifold $(M, g)$,
the subgroup of $\Isom(\widetilde M, \widetilde g)$ of covering transformations is discrete,
and naturally isomorphic to the fundamental group of $M$.
Special attention has been given to the 
closed aspherical Riemannian manifolds $M$
such that $\Isom(\widetilde M, \widetilde g)$ is non-discrete \cite{FaWe--08}.
\end{exe}

Let us finally quote from \cite{MaSo--09} the following: 

\begin{thm}[Malicki and Solecki]
\label{MalickiSolecki}
Let $G$ be a second-countable LC-group.
There exists a second-countable proper metric space $(X, d)$ such that
$G$ is topologically isomorphic to $\Isom(X, d)$.
\end{thm}

\section{Lattices in LC-groups}
\label{sectionlattices}

\begin{defn}
\label{deflattice}
In an LC-group $G$, a \textbf{lattice} 
\index{Lattice! in an LC-group|textbf}   
is a discrete subgroup $\Gamma$ such that
$G/\Gamma$ has a $G$-invariant probability measure on Borel subsets,
and a \textbf{uniform lattice} \index{Uniform lattice|textbf}
is a cocompact discrete subgroup.
\par
More generally, a closed subgroup $H$ of $G$ 
has \textbf{finite covolume} if there exists a $G$-invariant probability measure 
defined on Borel subsets of the quotient space $G/H$.
\index{Subgroup! finite covolume}
\index{Subgroup! lattice}
\index{Subgroup! uniform lattice}
\end{defn}

\begin{rem}
\label{remdeflattice}
Metric lattices in pseudo-metric spaces (Definition \ref{defcbsupspace})
and lattices in LC-groups are two notions that should not be confused.
\index{Metric lattice}    \index{Lattice! metric lattice}
\end{rem}

\begin{prop}
\label{propdeflattice}
Let $G$ be a topological group and $\Gamma$ a cocompact discrete subgroup.
\begin{enumerate}[noitemsep,label=(\arabic*)]
\item\label{1DEpropdeflattice}
$G$ is locally compact and $\Gamma$ is a lattice.
\item\label{2DEpropdeflattice}
$G$ is compactly generated if and only if $\Gamma$ is finitely generated.
\item\label{3DEpropdeflattice}
Suppose that the conditions of \ref{2DEpropdeflattice} hold. 
Let $d_G$ be a geodesically adapted metric on $G$ 
and $d_\Gamma$ a word metric on $\Gamma$, with respect to some finite generating set.
Then the inclusion 
$(\Gamma,d_\Gamma) \lhook\joinrel\relbar\joinrel\rightarrow (G,d_G)$ 
is a quasi-isometry.
\end{enumerate}
\end{prop}

\begin{proof}
\ref{1DEpropdeflattice}
That $G$ is a locally compact group is a particular case of 
Claim \ref{2DEstababcd} for \ref{cDEstababcd} in Proposition \ref{stababcd}.
That $\Gamma$ is a lattice is a standard fact
in the theory of measures on LC-groups;
see e.g.\ \cite[Chapter 1]{Ragh--72}.
Recall that, for every discrete subgroup $\Gamma$ of $G$,
there exists a \emph{semi-invariant} Borel measure $\mu$ on $G/\Gamma$,
i.e., a measure for which there exists 
a continuous homomorphism $\chi : G \longrightarrow \R_+^\times$
such that $\mu(gB) = \chi(g) \mu(B)$ for all $g \in G$ and Borel subset $B$ in $G/\Gamma$.
Moreover, if $\mu(G/\Gamma) < \infty$, then $\chi(g) = 1$ for all $g \in G$.
\par

\ref{2DEpropdeflattice} \& \ref{3DEpropdeflattice}
This is contained in Proposition \ref{sigmac+compactgofcocompact}. 
\end{proof}

\begin{rem}
\label{remfglattice}
In complement to \ref{2DEpropdeflattice} of the previous proposition,
let us mention that, in connected Lie groups, \emph{all lattices} are finitely generated,
indeed finitely presented.
For uniform lattices, this is rather straightforward:
it follows from the above proposition for finite generation
and Corollary \ref{hereditaritycp} for finite presentation.
For non-uniform lattices, there is no short proof available;
we refer to the discussion in \cite[Item V.20]{Harp--00}.
\par

The general situation is more intricate. 
For instance, a non-cocompact lattice in a compactly generated LC-group 
need not be finitely generated. 
Consider for example an integer $n \ge 2$
and the group $\SL_n(\mathbf{F}_q [t])$,
isomorphic to the non-uniform lattice $\SL_n(\mathbf{F}_q[t^{-1}])$ 
in the compactly generated (indeed compactly presented)
LC-group $\SL_n(\mathbf{F}_q \lp t \rp)$.
If $n=2$, then $\SL_2(\mathbf{F}_q[t])$ is not finitely generated;
see Exercice 3 of \S~II.1.6 in \cite[Page 212]{Serr--77}.
\par

If $n=3$, then $\SL_n(\mathbf{F}_q[t])$ is finitely generated and not finitely presented
\cite{Behr--79}.
If $n \ge 4$, then $\SL_n(\mathbf{F}_q[t])$ is finitely presented
\cite[Page 184]{ReSo--76}.

\end{rem}

\begin{prop}
\label{prop2deflattice}
Let $G$ be an LC-group which admits a lattice.
\begin{enumerate}[noitemsep,label=(\arabic*)]
\item\label{1DEprop2deflattice}
$G$ is unimodular.
\item\label{2DEprop2deflattice}
Assume moreover that $G$ is a connected Lie group.
Then the group of inner automorphisms of $G$
is closed in the group of all automorphisms of $G$ (see \ref{exampleAut(G)}).
\end{enumerate}
\end{prop}

\begin{proof}[References for a proof]
For \ref{1DEprop2deflattice}, 
see for example \cite[Remark 1.9]{Ragh--72}.
For \ref{2DEprop2deflattice}, see \cite[Theorem 2]{GaGo--66}.
\end{proof}

\begin{rem}
\label{rem2deflattice}
Conditions \ref{1DEprop2deflattice} and \ref{2DEprop2deflattice}
of the previous proposition are not sufficient
for $G$ to have a lattice.
\par

Indeed,
there is in the Appendix of \cite{GaGo--66} an example of
a four-dimensional solvable connected Lie group
of type $H = \R^3 \rtimes \R$ that has the following three properties:
$H$ is unimodular, 
the group of inner automorphisms of $H$ is closed 
in $\operatorname{Aut}(H)$,
and $H$ does not have any lattice.
\end{rem}

\begin{exe}
\label{firstexampleslattices}
(1)
Here are some examples among the easiest to check:
$\Z^n$ is a uniform lattice in $\R^n$ for every $n \ge 1$;
in particular, the additive group of the Gaussian integers $\Z[i]$ 
is a uniform lattice in $\C$.
The infinite cyclic group $2^{\Z}$ is a uniform lattice in $\C^\times$.
\index{$w$@$w^{\Z}$, infinite cyclic group generated by $w$}
\par

Let $p$ be a prime.
The additive group $\Q_p$ has no lattice\footnote{The 
meaning of ``lattice'' here should not be confused
with different notions, used in different contexts.
For example, 
in a finite-dimensional vector space $V$ over $\Q_p$, 
the $\Z_p$-span of a basis of $V$ is also called a lattice.};
\index{$q$@$\Q_p$, field of $p$-adic numbers}
indeed, being locally elliptic and torsion-free, 
$\Q_p$ has no discrete subgroup other than $\{0\}$.
The multiplicative group $\Q_p^\times$ has cocompact lattices,
for example $p^{\Z}$.
Let $q$ be a prime power.
The ring of polynomials $\mathbf F_q [t^{-1} ]$ is a lattice
in the local field $\mathbf F_q \lp t \rp)$ of Laurent series with coefficients in $\mathbf F_q$.
\index{$fq$@$\mathbf F_q [t]$, $\mathbf F_q [[t]]$}

\vskip.2cm

(2)
A compactly generated LCA-group $A$
has a cocompact lattice that is a finitely generated free abelian group.
More precisely,
if $A = \R^\ell \times \Z^m \times K$ as in Example \ref{LCA}\ref{2DELCA},
then $\Z^{\ell + m}$ is a cocompact lattice in $A$.

\vskip.2cm

(3) 
In a nilpotent LC-group, finite covolume closed subgroups are necessarily cocompact.
In particular, lattices are necessarily uniform \cite[Lemma 3.3]{BeQu--14}.
\par

Lattices in solvable \emph{Lie} groups with countably many connected components
are necessarily uniform.
This is a result of Mostow, from 1957
\cite[Theorem 3.1]{Ragh--72}.
\index{Solvable group}

\vskip.2cm

(4)
A simply connected nilpotent real Lie group $G$ has uniform lattices if and only if
its Lie algebra has a basis 
with respect to which the constants of structure are in $\Q$.
This is a result of Malcev, from 1951
\cite[Theorem 2.12]{Ragh--72}.
\index{Nilpotent! group}
\par

For example, the discrete Heisenberg group
\index{Heisenberg group}
$\begin{pmatrix}
1 & \Z & \Z 
\\
0 &1 &\Z
\\
0 & 0 & 1 
\end{pmatrix}$
is a uniform lattice in the real Heisenberg group
$\begin{pmatrix}
1 & \R & \R 
\\
0 &1 &\R
\\
0 & 0 & 1 
\end{pmatrix}$.

\vskip.2cm

(5)
An LC-group $G$ is \textbf{approximated by discrete subgroups}
\index{Approximation by discrete subgoups|textbf}
if, in the space of closed subgroups of $G$, 
the set of discrete subgroups is dense for the Chabauty topology,
i.e., if, for every integer $k \ge 1$ and for every $k$-uple 
$(U_1, \hdots, U_k)$ of non-empty open subsets of $G$,
there exists a discrete subgroup $\Gamma$ of $G$
such that $\Gamma \cap U_i \ne \emptyset$ for all $i \in \{1, \hdots, k\}$.
In the particular case of a second countable LC-group $G$, 
this condition is equivalent to:
there exists a sequence $(\Gamma_n)_{n \ge 1}$ of discrete subgroups of $G$
such that, for every non-empty open subset $U$ of $G$,
there exists $n_0 \ge 1$ such that $U \cap \Gamma_n \ne \emptyset$ for all $n \ge n_0$.
\par

For a connected Lie group $G$, Kuranishi \cite{Kura--51} has shown:
(i) if $G$ is approximated by discrete subgroups, then $G$ is nilpotent;
(ii) if $G$ is nilpotent and simply connected, 
then $G$ is approximated by discrete subgroups if and only if $G$ has uniform lattices.
Moreover, there is a necessary and sufficient condition 
for a nilpotent connected Lie group
(not necessarily simply connected) to be approximated by discrete subgroups,
for which we refer to the original article by Kuranishi.

\vskip.2cm

(6) 
For every integer $n \ge 2$, the group $\SL_n(\Z)$ 
\index{Special linear group $\SL$! $\SL_n(\Z)$, $\GL_n(\Z)$} 
\index{General linear group $\GL$! $\GL_n(\R)$, $\SL_n(\R)$}
is a non-uniform lattice in $\SL_n(\R)$.
The result is due to Minkowski \cite{Mink--91};
see \cite[th\'eor\`eme 1.4 and lemme 1.11]{Bore--69}.

\vskip.2cm  

(7)
There are groups with non-uniform lattices and without uniform lattices.
For example, consider an integer $n \ge 2$
and the natural semi-direct product $G = \R^n \rtimes \SL_n(\R)$.
On the one hand, $\Z^n \rtimes \SL_n(\Z)$
is a non-uniform lattice in $G$.
On the other hand, $G$ does not have any uniform lattice, as we check now.
\index{Semidirect product! $\R^n \rtimes \SL_n(\R)$}
\par

Let $\Gamma$ be a lattice in $G$.
Let $\pi : G \longrightarrow \SL_n(\R)$
denotes the canonical projection on the quotient of $G$ by its solvable radical $\R^n$.
Then $\pi(\Gamma)$ is a lattice in $\SL_n(\R)$,
by a theorem of H.C.\ Wang, see \cite[Number 3.4]{Wang--63},
or \cite[Lemma 6.4]{BeQu--14}.
If $\Gamma$ was uniform, 
then $\Gamma \cap \R^n$ would be a lattice in $\R^n$
by Theorem 1 of \cite{Ausl--63};
upon conjugating it in $G$, 
one could therefore assume that $\Gamma \cap \R^n = \Z^n$;
since $\pi(\Gamma)$ would have to normalize $\Z^n$,
it would be a finite index subgroup of $\GL_n(\Z)$,
and this is impossible if $\Gamma$ is uniform.
\par

Other LC-groups with non-uniform lattices and without uniform lattices
are alluded to in Example \ref{exlattcesinlocal}(1).

\vskip.2cm

(8) 
In every dimension $n \ge 3$, there are uncountably many
pairwise non-isomorphic simply connected real Lie groups,
but only countably many of them have lattices \cite{Wink--97}.

\vskip.2cm

(9)
There exist solvable LC-groups with non-uniform lattices.
The following example is due to Bader, Caprace, Gelander, and Mozes;
see \cite[Example 3.5]{BeQu--14} and \cite[Example 4.3]{Gela--14}.

For a prime $p$, denote by $\mathbf F_p$ the additive group of the field of order $p$
and by $\mathbf F_p^\times$ the multiplicative group of this field,
which is cyclic of order $p-1$.
Let $S$ be an infinite set of primes such that $\sum_{p \in S} p^{-1} < \infty$.
Consider the countable abelian group $A = \bigoplus_{p \in S}  \mathbf F_p$, 
the compact group $K = \prod_{p \in S}  \mathbf F_p^\times$,
which is naturally a group of automorphisms of $A$,
and the metabelian LC-group $G = A \rtimes K$,
in which $A$ is discrete and $K$ compact open.

For $p \in S$, choose a generator $s_p$ 
of the multiplicative group $\mathbf F_p^\times$.
Let $H_p$ denote the cyclic subgroup $\{ (1-s_p^j,s_p^j) \}_{j=0, 1, \hdots,p-2}$ 
of $\mathbf F_p \rtimes \mathbf F_p^\times$,  generated by $(1-s_p, s_p)$.
Let $a_p \in A$ be the element with $p^{\text{th}}$ coordinate $1-s_p$ and other coordinates $0$,
and let $k_p \in K$ be the element with $p^{\text{th}}$ coordinate $s_p$ and other coordinates $1$.
Set $h_p = (a_p, k_p) \in A \rtimes K = G$;
let $H$ be the subgroup of $G$ generated by $\{h_p\}_{p \in S}$.
Then $H$ is a non-uniform lattice in $G$, as we are going to show.

Note first that $H$ is a discrete subgroup of $G$, 
because the intersection of $H$ with the compact open subgroup $K$ is trivial.

To show that $H$ is not cocompact in $G$, consider the set $\mathcal I$
of all finite subsets of $S$.
For $I \in \mathcal I$, 
consider the subgroup $A_I = \bigoplus_{p \in I} \mathbf F_p$ of $A$
and the subgroup $G_I = A_I \rtimes K$ of $G$.
Then $A_I$ is finite, of order $\prod_{p \in I} p$,
and $G_I$ is compact open in $G$, because it is a finite union of $K$-cosets in $G$.
We have $G_I \subset G_J$ whenever $I \subset J$,
and the family $(G_I)_{I \in \mathcal I}$ is a covering of $G$.
It follows that any compact subset of $G$ is contained in some $G_I$
(in particular, with the terminology of Definition \ref{deflocallyelliptic}, $G$ is locally elliptic).
Let $\pi$ denote the canonical projection of $G$ onto the homogeneous space $G/H$.
If $G/H$ were compact, there would exist a compact subset $L$ of $G$
such that $\pi(L) = G/H$ (Lemma  \ref{KimagedeK}),
hence some $I \in \mathcal I$ such that $\pi(G_I) = G/H$;
since this is not true, $G/H$ is not compact.
\index{Locally elliptic LC-group}

Denote by $\mu$ the Haar measure on $G$ normalized by $\mu(K) = 1$.
We claim that $\mu$ induces a measure $\mu_{G/H}$ on $G/H$ 
which is $G$-invariant and finite.
For $I \in \mathcal I$, 
the restriction $\mu_I$ of $\mu$ to $G_I$ is the Haar measure on $G_I$
for which $K$ is of measure $1$, 
and $\mu_I(G_i) = \vert A_I \vert = \prod_{p \in I} p$.
Consider the subgrooup 
$H_I = H \cap G_I = \langle h_p \mid p \in I \rangle$ of $G_I$.
Since $H_I$ is finite of order $\prod_{p \in I} (p-1)$,
the measure $\mu_I$ induces on $G_I/H_I$ a finite $G_I$-invariant measure
of total mass $\prod_{p \in I} \frac{p}{p-1} = \prod_{p \in I} (1 + \frac{1}{p-1})$.
Since $G_I$ is open in $G$,
the space $G_I / H_I$ can be identifed with an open subspace of $G/H$.
For $I,J \in \mathcal I$ with $I \subset J$, we have
$G_I / H_I \subset G_J / H_J$,
and  $G/H = \bigcup_{I \in \mathcal I} G_I / H_I$.
We use only now the condition $\sum_{p \in S} p^{-1} < \infty$
to conclude that $\mu_{G/H}(G/H) = \prod_{p \in S} (1 + \frac{1}{p-1})$ is finite.
\end{exe}

\begin{rem}[lattices in discrete groups]
\label{remlatticediscretegps}
(1)
Inside a discrete group, a subgroup is a lattice if and only if it is of finite index.
These lattices are uniform.

\vskip.2cm

(2)
Let $\Delta$ be a finitely generated group, say generated by $n$ elements, 
and $\Gamma$ a finite index subgroup, say of index $j$.
Then $\Gamma$ is finitely generated;
this is a particular case of the non-trivial implication of Proposition 
\ref{sigmac+compactgofcocompact}\ref{2DEsigmac+compactgofcocompact}.
This can also be established by the so-called Reidemeister-Schreier method,
which provides for $\Gamma$ a set of $n(j-1)+1$ generators.
See \cite[Corollary 2.7.1 and Theorem 2.10]{MaKS--66},
and also \cite{MaSw--59}.
\end{rem}

\begin{exe}[lattices in semisimple and reductive Lie groups]
\label{exlattcesinssL+s}
(1)
Let $G$ be a semisimple real Lie group with a finite number of connected components.
Then $G$ contains lattices \cite{BoHa--62}.
More precisely, $G$ contains cocompact lattices (\cite{Bore--63}, 
announced in \cite{BoHa--61}).
It was previously known that,  if $G$ is non-compact, 
then it also admits non-cocompact lattices (usually obtained by taking integral points);
see \cite[Theorem 14.1]{Ragh--72}.
\par

If moreover $G_0$ has a finite centre, 
discrete groups quasi-isometric to lattices in $G$ 
and the classification of such lattices up to quasi-isometry
have received a lot of attention, as discussed in \cite{Farb--97}.
Relying on these results and on work of Kleiner-Leeb \cite{KlLe--09}, 
it is checked in \cite[Section 3]{Corn}
that a compactly generated LC-group is quasi-isometric to $G$
if and only if it admits a copci homomorphism (as defined in Remark \ref{remcopci})
into $\operatorname{Aut}(\mathfrak g)$, where $\mathfrak g$ is the Lie algebra of $G$.
\par

The archetypical example of a lattice in a simple Lie group
is $\SL_n(\Z)$ in $\SL_n(\R)$, already cited in \ref{firstexampleslattices}(6).

\vskip.2cm

(2)
The simple real Lie group $\PSL_2(\R)$ contains as uniform lattices
fundamental groups of closed surfaces of genus at least $2$
(and this in many ways, by the uniformization theory for Riemann surfaces).
Let $\Gamma$ be such a lattice,
and $\widetilde \Gamma$ its inverse image in $\SL_2(\R)$.
The reductive real Lie group $\GL_2^+(\R)$, 
where the ``$+$'' indicates matrices of positive determinants, 
or equivalently the identity component of $\GL_2(\R)$, 
has uniform lattices of the form $\widetilde \Gamma \times 2^{\Z}$,
because the map
\begin{equation*}
\left\{
\aligned
\SL_2(\R) \times \R_+^\times 
& \longrightarrow \hskip1cm \GL^+_2(\R)
\\
\left(
\begin{pmatrix}
a & b 
\\
c & d 
\end{pmatrix}
, t
\right)
& \longmapsto
\begin{pmatrix}
a & b 
\\
c & d 
\end{pmatrix}
\begin{pmatrix}
\sqrt t & 0
\\
0 & \sqrt t
\end{pmatrix}
\endaligned
\right.
\end{equation*}
is an isomorphism of Lie groups.
\index{Projective linear group PGL, and PSL! $\PSL_2(\R)$, $\PSL_2(\C)$}

\vskip.2cm

(3)
The group $\PSL_2(\C)$ can be seen as a simple complex Lie group,
\index{Special linear group $\SL$! $\SL_2(\C)$}
and also as a simple real Lie group by ``restriction of scalars''.
It contains as uniform lattices fundamental groups
of compact hyperbolic $3$-manifolds. 
The map
\begin{equation*}
\iota \, : \, 
\left\{
\aligned
\SL_2(\C) \times \C^\times 
& \relbar\joinrel\twoheadrightarrow \hskip1cm \GL_2(\C)
\\
\left(
\begin{pmatrix}
a & b 
\\
c & d 
\end{pmatrix}
, t
\right)
& \longmapsto
\begin{pmatrix}
a & b 
\\
c & d 
\end{pmatrix}
\begin{pmatrix}
t & 0
\\
0 & t
\end{pmatrix}
\endaligned
\right.
\end{equation*}
is a surjective homomorphism of Lie groups with kernel of order $2$.
Let $\Gamma \subset \SL_2(\C)$ be a uniform lattice,
say the inverse image of a uniform lattice in $\PSL_2(\C)$.
Then $\iota\left( \Gamma \times 2^{\Z} \right)$
is a uniform lattice in $\GL_2(\C)$.
\end{exe}

\begin{exe}[lattices in algebraic groups over other fields]
\label{exlattcesinlocal}
(1)
Let $\mathbf G$ be an algebraic group defined over a 
local field $\K$ of characteristic $0$;
if $G = \mathbf G (\K)$ has a lattice, 
then $\mathbf G$ is reductive and the lattice is uniform;
see \cite{Tama--65}, or \cite[Section II.1.5]{Serr--77}.
The following converse is a result of Borel and Harder:
every reductive group over a local field of characteristic zero
contains a uniform lattice \cite{BoHa--78}.
\index{Local field}
\index{Reductive group}
\par

In finite characteristic, some reductive groups have uniform lattices and others do not.
We refer to the discussion in 
\cite[Remark IX.1.6(viii) Page 295, and Pages 316--317]{Marg--91}.

\vskip.2cm

(2)
For the diagonal embedding, $\Z[1/p]$ is a cocompact lattice
in $\R \times \Q_p$, as we check now.
\index{$z$@$\Z[1/n]$}
\par

On the one hand, we have 
$\Z[1/p] \cap \Z_p = \Z \subset \Q_p$
and
$\Z \cap \mathopen]-1,1\mathclose[ = \{0\} \subset \R$.
Hence $\mathopen]-1,1\mathclose[ \times \Z_p$,
which is an open neighbourhood of $(0,0)$ in $\R \times \Q_p$,
intersects $\Z[1/p]$ in the singleton $\{(0,0)\}$.
This shows that $\Z[1/p]$ is a discrete subgroup of $\R \times \Q_p$.
\par

On the other hand, consider $(x,y) \in \R \times \Q_p$.
Write $y = [y] + \sum_{i=1}^n y_i p^{-i}$, 
with $[y] \in \Z_p$, $n \ge 0$,  and $y_i \in \{0, 1, \hdots, p-1\}$
for $i = 1, 2, \hdots, n$.
Set $z = k + \sum_{i=1}^n y_i p^{-i} \in \Z[1/p]$, 
with $k \in \Z$ such that $\vert x - z \vert \le 1/2$;
observe that $y - z \in \Z_p$.
Hence $\R \times \Q_p = \bigcup_{z \in \Z[1/p]} z + ([-1/2,1/2] \times \Z_p)$.
This shows that $\Z[1/p]$ is cocompact in $\R \times \Q_p$.
\par

The quotient $S_p := (\R \times \Q_p) / \Z[1/p]$ is an abelian compact group
known as the \textbf{$p$-adic solenoid}. 
\index{p-adic solenoid|textbf}
\index{Solenoid|textbf}
\index{Compact group! solenoid}
It can also be seen as an inverse limit $\varprojlim \R / p^n \Z$.
There is a short exact sequence
$\Z_p \lhook\joinrel\relbar\joinrel\rightarrow S_p \relbar\joinrel\twoheadrightarrow \R / \Z$,
and $S_p$ is a covering of a circle with $\Z_p$ fibers.
The compact group $S_p$ contains both a dense subgroup isomorphic to $\R$
(it follows that $S_p$ is connected)
and a dense subgroup isomorphic to $\Q_p$.
Note that $S_p$ is an LCA-group that is connected and not locally connected.
 \index{Locally connected group}
On $S_p$, see for example \cite[Appendix to Chapter 1]{Robe--00}. 
\par

Similarly, $\Q$ is a cocompact lattice in $\R \times \mathbf{A}_{\Q}$,
where $\mathbf{A}_{\Q}$ denotes the adeles of $\Q$;
see \cite[Chapter IV,  $\S$~2]{Weil--67}, and Example \ref{ExTopsub}(4)

\vskip.2cm

(3)
We describe other examples of lattices
in products of algebraic groups defined over different fields;
compare with the simple one given in \S~\ref{intro_discrete_LC}.
Consider an integer $m \ge 2$, and the prime divisors $p_1, \hdots, p_k$ of $m$.
The direct product $\mathbf A := \R \times \prod_{j=1}^k \Q_{p_j}$ 
is an abelian locally compact group, indeed a locally compact ring.
Let $\alpha_m$ be the automorphism of $\mathbf A$ 
that is multiplication by $m$ in each factor of the product.
Let $\mathbf A \rtimes_m \Z$ be the semidirect product, 
with respect to the action for which
the generator $1$ of $\Z$ acts on $\mathbf A$ by $\alpha_m$.
\index{Semidirect product! $\mathbf A \rtimes_m \Z$}
\index{Unimodular group}
\par

The solvable Baumslag-Solitar group 
$\operatorname{BS}(1,m) = \Z [1/m] \rtimes_m \Z$
is naturally a cocompact lattice in $\mathbf A \rtimes_m \Z$.
We come back to this example in \ref{abcdefg},
and allude to other Baumslag-Solitar groups in \ref{GJCV}.
\index{Baumslag-Solitar group} \index{Solvable group}
\index{Semidirect product! $\Z [1/m] \rtimes_m \Z$}

\vskip.2cm

(4)
The group
$\SL_n(\Z [1/p])$ is naturally a non-cocompact lattice in $\SL_n(\R) \times \SL_n(\Q_p)$.
\index{Special linear group $\SL$! $\SL_n(\Q_p)$}  
\index{General linear group $\GL$! $\GL_n(\R)$, $\SL_n(\R)$}
\index{Special linear group $\SL$! $\SL_n(\Z [1/p])$}
\end{exe}

\begin{exe}[tree lattices]
\label{exlatticestree}
Let $X$ be a locally finite tree;
\index{Graph} \index{Tree} \index{Locally finite! graph}
let $G = \operatorname{Aut}(X)$ denote its automorphism group,
\index{Automorphism group! of a graph}
with its natural separable locally compact topology,
as in Example  \ref{extdLCgps}(5) above.
Then $G$ contains uniform lattices if and only if
$G$ is unimodular and the graph $G \backslash X_{\textnormal{bary}}$ is finite.
 \index{Unimodular group}
(Here, $X_{\textnormal{bary}}$ denotes the barycentric subdivision of $X$;
the point is that $G$ naturally acts on $X_{\textnormal{bary}}$
with quotient a graph, whereas $G \backslash X$ need not be a graph
in the strict sense.) 
Suppose that $G$ contains uniform  lattices;
then any such lattice is virtually free 
\index{Free group}
\index{Free subgroup}
(i.e., has finitely generated free subgroups of finite index),
and two such lattices $\Gamma_1, \Gamma_2$ are commensurable
(i.e., there exists $g \in G$ such that $\Gamma_1 \cap g \Gamma_2 g^{-1}$
is of finite index in both $\Gamma_1$ and $\Gamma_2$).
See \cite{BaKu--90}.
\par

Among trees for which $G$ has uniform lattices,
there exist trees for which $G$ has non-uniform lattices
and trees for which $G$ has none \cite{BaLu--01}.
\end{exe}

\begin{rem}
\label{GJCV}
Every Baumslag-Solitar group $\operatorname{BS}(m,n)$, solvable or not, is a cocompact lattice
in a semi-direct product $(\overline B \times \R) \rtimes \Z$,
where $\overline B$ is an appropriate closed subgroup 
of the group of automorphisms of the regular tree of valency $m+n$.
This is a particular case of Proposition 5.4 in \cite{CoVa--15},
building up on \cite{GaJa--03}.
\end{rem}

\begin{rem}
\label{BCGM12}
There are compactly generated simple LC-group without any lattice \cite{BCGM--12}.
\index{Simple group} 
\end{rem}

\chapter[Coarse simple connectedness]
{Simple connectedness in the metric coarse category}
\label{chap_coarsely1conn}  

\section{Coarsely simply connected pseudo-metric 
\\
spaces}
\label{sectioncscpms}
Our first definition provides a coarse notion of simple connectedness,
as Definition \ref{defcoarselyconn} provides coarse notions of connectedness.
For a definition with less numerical constants, see \cite[Section 11]{HiPR--97}.

\begin{defn}
\label{c-homotopic}
Let $(X,d)$ be a pseudo-metric space.
Let $c > 0$ be a constant.
Recall that $c$-paths in $X$ have been defined (\ref{defcoarselyconn}),
and that a $c$-path is a $c'$-path for every $c' \ge c$.
\par

Two $c$-paths $\xi = (x_0, x_1, \hdots, x_m)$  and
$\eta = (y_0, y_1, \hdots, y_n)$ in $X$ are \textbf{$c$-elementarily homotopic}
\index{c-elementarily homotopic $c$-paths|textbf}
\index{Elementarily! homotopic $c$-paths|textbf}
if they have the same origin, $x_0 = y_0$, the same end, $x_m = y_n$,
and if one path can be obtained from the other by inserting one new point;
the last condition means more precisely that 
\begin{equation*}
\aligned
&\text{either}  \hskip.2cm n=m+1 \hskip.2cm \text{and} \hskip.2cm
(y_0, \hdots, y_n) = (x_0, \hdots, x_i, y_{i+1}, x_{i+1}, \hdots, x_m),
\\
&\text{or} \hskip.2cm n+1=m \hskip.2cm \text{and} \hskip.2cm
(x_0, \hdots, x_m) = (y_0, \hdots, y_{i}, x_{i+1}, y_{i+1}, \hdots, y_n),
\endaligned
\end{equation*}
for some $i$.
\par

Two $c$-paths $\xi, \eta$ are \textbf{$c$-homotopic}
if there exists a sequence of $c$-paths $(\xi_0 = \xi, \xi_1, \hdots , \xi_\ell = \eta)$
such that $\xi_{j-1}$ and $\xi_j$ 
are $c$-elementarily homotopic for $j = 1, \hdots, \ell$.
Note that, for $c' \ge c$, two $c$-paths are $c'$-homotopic if they are so as $c'$-paths.
\index{c-homotopic $c$-paths|textbf}
\index{Homotopic $c$-paths|textbf}
\par

Let $x_0$ be a point in $X$.
A \textbf{$c$-loop in $X$ at $x_0$} 
is a $c$-path that starts and ends at $x_0$.
\index{c-loop|textbf}
\end{defn}

Examples of $c'$-homotopic paths are provided by

\begin{prop}[$c$-near $c$-paths are $2c$-homotopic]
\label{nearbypathsare$c$-homotopic}
Let $(X,d)$ be a  pseudo-metric space, 
$c > 0$ a constant,
and $\xi = (x_0, \hdots, x_n)$, $\eta = (y_0, \hdots, y_n)$
two $c$-paths with the same number of steps
from $x_0=y_0$ to $x_n=y_n$ in $X$.
\par

If $d(x_i,y_i) \le c$ for all $i \in \{1, \hdots, n-1\}$,
then $\xi$ and $\eta$ are $2c$-homotopic.
\end{prop}

\begin{proof}
Inspection shows that each of the following 
is a $2c$-path of $n$ or $n+1$ steps in $X$
\begin{equation*}
\begin{array}{ccccc}
\xi &= 
&\xi_0 &= &(y_0, x_1, x_2, x_3, \hdots, x_n) 
\\
&&\xi_1 &= &(y_0, y_1, x_1, x_2, x_3, \hdots, x_n) 
\\
&&\xi_2 &= &(y_0, y_1, x_2, x_3, \hdots, x_n) 
\\
&&& \vdots &
\\
&&\xi_{2n-4} &= &(y_0, y_1, \hdots, y_{n-2}, x_{n-1}, x_n) 
\\
&&\xi_{2n-3} &= &(y_0, y_1, \hdots, y_{n-2}, y_{n-1}, x_{n-1}, x_n) 
\\
\eta &=
&\xi_{2n-2} &= &(y_0, y_1, \hdots, y_{n-2}, y_{n-1}, x_n) 
\end{array}
\end{equation*}
and that $\xi_{j-1}$ is $2c$-elementarily homotopic to $\xi_j$
for $j = 1, \hdots, 2n-2$.
\end{proof}

\begin{defn}
\label{defSC} 
Consider a pseudo-metric space $X$ with a point $x_0 \in X$.
For constants $c'' \ge c' > 0$, define the property
\begin{equation}
\text{any $c'$-loop in $X$ at $x_0$ is $c''$-homotopic to the trivial loop $(x_0)$.}
\tag{SC($c',c''$)}\label{SC($c',c''$)}
\end{equation}
\end{defn}

\begin{rem}
\label{SCindex0}
(1)
For constants $c,c',c''$ with $c'' \ge c' \ge c > 0$
and a $c$-coarsely connected space $X$,
observe that Property \textup{(SC($c',c''$))} holds for one choice of $x_0$ 
if and only if it holds for any other choice of a base point in $X$.
In the following definitions, the base point will not be mentioned at all.
\vskip.2cm

(2)
Consider moreover constants $C',C''$ with $c'' \ge C'' \ge C' \ge c'$.
Then Property \textup{(SC($C',C''$))} implies Property \textup{(SC($c',c''$))}.
\end{rem}

\begin{defn}
\label{defcoarsely1con}
For $c > 0$,
a non-empty pseudo-metric space $X$ is
\textbf{$c$-coarsely simply connected}
\index{Coarsely! simply connected pseudo-metric space|textbf}
\index{Simply connected|see {Coarsely simply connected}}
if $X$ is $c$-coarsely connected and, for all $c' \ge c$, there exists $c'' \ge c'$
such that $X$ has Property
\eqref{SC($c',c''$)}.
\par

The space $X$ is \textbf{coarsely simply connected}
if it is $c$-coarsely simply connected for some $c > 0$
(equivalently for all $c$ large enough).
\end{defn}

\begin{rem}
\label{remoncCSC}
(1)
Let $c,C$ be two constants with $C \ge c > 0$ 
and $X$ a $c$-coarsely connected non-empty pseudo-metric space.
Then $X$ is $c$-coarsely simply connected if and only if
$X$ is $C$-coarsely simply connected.

\vskip.2cm

(2)
Coarse connectedness is defined for pseudo-metric spaces,
and coarse simple connectedness for pseudo-metric spaces with base points.
Thus, an empty pseudo-metric space is coarsely connected,
but simple coarse connectedness does  not make sense for it.
Similarly, an empty topological space is  connected,
but simple connectedness does not make sense for it.
\end{rem}

The following straightforward proposition is the analogue here
of Proposition \ref{invariance_cg_lsg}(1)
for coarse connectedness.

\begin{prop}
\label{coarse1conninvbycoarseeq}
For pseudo-metric spaces, the property of being coarsely simply connected
is invariant by metric coarse equivalence.
\index{Property! of a pseudo-metric space invariant 
by metric coarse equivalence or by quasi-isometries}
\end{prop}

\begin{exe}
\label{excoarse1conntrivial}
Let $(X,d)$ be a pseudo-metric space, with base point $x_0$.
Let $\xi = (x_0, \hdots, x_n)$ be a $c$-loop, with $x_n = x_0$,
for some $c > 0$.
Set 
\begin{equation*}
c' \, = \,  \max\{ c,  \hskip.1cm d(x_i,x_j), 0 \le i,j \le n \} .
\end{equation*}
Then $\xi$ and the trivial loop $(x_0)$
are $c'$-homotopic $c$-loops.
\par

It follows that a non-empty pseudo-metric space of finite diameter
is coarsely simply connected.

\end{exe}

\begin{exe}
\label{bunchofcircles}
(1)
Let $C_R = (Re^{2i\pi t})_{0 \le t \le 1}$ be the circle of radius $R > 0$ in the complex plane,
with the metric induced by the Euclidean metric of the plane;
thus the diameter of $C_R$ is $2R$.

\vskip.2cm

\emph{Claim. Consider a constant $c > 0$ and the $c$-loop 
\begin{equation*}
\xi \,  =  \, (R, Re^{2i \pi (1/m)}, Re^{2i \pi (2/m)}, \hdots, Re^{2i \pi ((m-1)/m) }, R) 
\end{equation*}
in $C_R$ at the point $R$,
with $m$ large enough so that $\vert Re^{2i\pi (1/m)} - R \vert < c$.
If $\xi$ is $c$-homotopic to the constant loop $(R)$,
then $c \ge \sqrt{3} R$.}

\vskip.2cm

To prove the claim,
we define as follows a rudimentary kind of rotation number $\rho(\cdot)$
for ``discrete loops'' in the circle $C_R$.
Define three disjoint half-open circular arcs 
\begin{equation*}
A_\alpha \, = \, 
\left[ R e^{2i\pi\frac{\alpha}{3}} , R e^{2i\pi\frac{\alpha+1}{3}} \right[ \subset C_R ,
\hskip.2cm \alpha \in \Z / 3\Z ,
\end{equation*}
each of length $\frac{2}{3} \pi R$.
For a ``discrete loop'' $\eta = (\eta_j)_{j \in \Z / n\Z}$, with $\eta_j \in C_R$ for $j \in \Z / n\Z$,
and $\alpha, \beta \in \Z / 3\Z$, let
\begin{equation*}
T_{\alpha,\beta}(\eta) \, = \, 
\# \{ j \in \Z / n\Z \mid \eta_j \in A_\alpha, \eta_{j+1} \in A_{\beta} \}
\end{equation*}
denote the number of steps of $\eta$ from $A_\alpha$ to $A_\beta$,
and define 
\begin{equation*}
\rho(\eta) \, = \, \sum_{ \alpha \in \Z / 3\Z } 
\left( T_{\alpha, \alpha+1} (\eta) - T_{\alpha+1, \alpha} (\eta) \right).
\end{equation*}
For example, if $\xi$ is as above with $m \ge 3$,
then $\rho (\xi) = 3$; if  $\xi_0 = (R, R, \hdots, R)$ is a constant loop with $m$ steps,  
then $\rho (\xi_0)  = 0$.
\par

Let $c$ be a constant, $0 < c < \sqrt{3}R$.
Let $\eta, \eta'$ be two $c$-loops in $C_R$ at $R$.
Suppose  that $\eta$ and $\eta'$ are $c$-elementarily homotopic, say
\begin{equation*}
\aligned
\eta \, &= \,  
(y_0, \hdots, y_i, y_{i+1}, \hdots, y_{n-1}, y_0) ,
\\
\eta'  \, &= \,  
(y_0, \hdots, y_i, y', y_{i+1}, \hdots, y_{n-1}, y_0) .
\endaligned
\end{equation*} 
Let $\alpha, \beta, \gamma \in \Z / 3\Z$ be defined by
$y_i \in A_\alpha$, $y' \in A_\beta$, $y_{i+1} \in A_\gamma$.
Since $\vert y_{i+1} - y' \vert$, $\vert y' -y_i \vert$, $\vert y_{i+1}-y_i \vert$ 
are all strictly smaller than $\sqrt{3}R$, 
elementary geometry shows that
the indices $\alpha, \beta, \gamma$ are not all distinct;
it follows that  $\rho(\eta) = \rho(\eta')$.
Hence, more generally, if $\eta$ and $\eta'$ are $c$-homotopic,
then $\rho(\eta) = \rho(\eta')$.
\par

Since $\rho(\xi) = 3 \ne 0 = \rho(\xi_0)$, the loop $\xi$ cannot be $c$-homotopic to $\xi_0$,
and the claim is proved. 

\vskip.2cm

(2)
Consider now $C_R$ as a metric space for itself (\emph{not} as a subspace of the plane).
Denote by $x_R$ the base point $R$.
Let $X := \bigvee_{n=1}^{\infty} C_n$ be the space obtained from the disjoint union
$\bigsqcup_{n=1}^{\infty} C_n$ by identifying the base points;
thus $X$ is an infinite wedge of larger and larger circles.
For $m \ne n$, the distance between two points $x \in C_m \subset X$
and $x' \in C_n \subset X$ is the sum $d(x,x_R) + d(x_R, x')$.
It follows from (1) that $X$ is not coarsely simply connected.
Of course, any finite wedge, for example $X_N := \bigvee_{i=1}^N C_n$, 
is coarsely simply connected (Example \ref{excoarse1conntrivial}).

\vskip.2cm

(3)
The notation being as in (2), 
let $C'_R$ be the complement in $C_R$ of the point $-R$.
The wedge $Y = \bigvee_{n=1}^{\infty} C'_n$ is contractible;
note that it is not coarsely simply connected, for example because
the inclusion $Y \subset X$ is a quasi-isometry.
This shows that
\begin{itemize}
\item
\textbf{a simply connected metric space need not be coarsely simply connected.}
\end{itemize}
Conversely, a connected coarsely simply connected metric space
need not be simply connected (examples $X_N$ of (2)).
\end{exe}

The next proposition shows that, for geodesic spaces,
simple connectedness does imply coarse simple connectedness.
For $c$-geodesic spaces, see 
Proposition \ref{propositionpourc-geodesicspace} below.

\begin{prop}
\label{scimplycsc}
Let $X$ be a geodesic non-empty metric space.
If $X$ is simply connected, then $X$ is coarsely simply connected.
\end{prop}

\begin{proof}
Choose a constant $c' > 0$.
Let $x_0 \in X$ and $\xi = (x_0, x_1, \hdots, x_n=x_0)$
be a $c'$-loop in $X$ at $x_0$.
Set $L = \sum_{i=1}^n d(x_{i-1},x_i)$;
since $X$ is geodesic, 
there exist a speed one continuous loop 
$\varphi : [0,L] \longrightarrow X$ 
and a sequence of real numbers $(s_i)_{0 \le i \le N}$ such that
\begin{equation*}
\aligned
& 0 = s_0 \le s_1 \le \cdots \le s_N = L ,
\\
&  \varphi(s_i) \ = \,  x_i   
\hskip.2cm \text{for} \hskip.2cm   i = 0, \hdots, n ,  
\hskip.2cm \text{in particular} \hskip.2cm   \varphi(0) = \varphi(L) = x_0 ,
\\
& \vert s_i - s_{i-1} \vert \, = \, d(x_i, x_{i-1}) \, \le \, c'
\hskip.2cm \text{for} \hskip.2cm i = 1, \hdots, n , .
\endaligned
\end{equation*}
Since $X$ is simply connected, there exists a continuous homotopy
$H : \mathopen[0,L\mathclose] \times \mathopen[0,1\mathclose] \longrightarrow X$ such that
\begin{equation*}
\begin{array}{ccc}
H(s,0) = \varphi(s) \hskip.2cm \text{and} \hskip.2cm H(s,1) = x_0
& \text{for all} & s \in [0,L] ,
\\
H(0,t) = H(L,t) = x_0 & \text{for all} & t \in \mathopen[0,1\mathclose].
\end{array}
\end{equation*}
By uniform continuity, there exists an integer $N \ge 1$ such that
\begin{equation*}
\vert H(s,t) - H(s',t') \vert \, \le \, c'
\hskip.2cm \text{whenever} \hskip.2cm
\vert s-s' \vert  \le \frac{L}{N} 
\hskip.2cm \text{and} \hskip.2cm
\vert t-t' \vert \le \frac{1}{N} .
\end{equation*}
There exist a sequence of real numbers $(r_h)_{0 \le h \le M}$,
with $0 = r_0  \le r_1 \le \cdots \le r_M = L$,  such that
$\vert r_h - r_{h-1} \vert \le L/N$ for all $h = 1, \hdots, M$,
and such that $(s_i)_{0 \le i \le N}$ is a subsequence of $(r_h)_{0 \le h \le M}$.
\par

For $j \in \{0, \hdots, N\}$, define $\xi_j$ to be the sequence
$\left( H(r_h, j/N) \right)_{h = 0, 1, \hdots, M}$.
In particular, $\xi_0$ is obtained by adding new points
in between those of $\xi$ along the image of~$\varphi$~; 
it follows that $\xi$ and $\xi_0$ are $c'$-homotopic.
By Proposition \ref{nearbypathsare$c$-homotopic}, 
the $c'$-loops $\xi_0$ and $\xi_N$ are $2c'$-homotopic.
Since  $\xi_N$ is $c'$-homotopic to the trivial loop $(x_0)$,
this ends the proof.
\end{proof}

\begin{prop}
\label{coarseretractprop}
Let $(Y,d)$ be a pseudo-metric space
and $Z \subset Y$ a coarse retract of $Y$,
in the sense of Definition \ref{coarseretractdef}.
We assume that $Z$ is non-empty.
\par

If $Y$ is coarsely simply connected, then so is $Z$.
\end{prop}

\noindent
\emph{Reminder:} if $Y$ is coarsely connected, so is $Z$, by Proposition \ref{bbb}.

\begin{proof}
Let $r : Y \longrightarrow Z$ be a coarse retraction.
Upon redefining $r$ at one point, 
we can assume that there exists a point $z_0 \in Z$ such that $r(z_0) = z_0$.
Let $\Phi$ be an upper control (Definition \ref{defuclc}) such that
$d(r(y), r(y')) \le \Phi(d(y,y'))$ for all $y,y' \in Y$.
Without loss of generality, 
we assume that $\Phi(t) > 0$ for all $t \ge 0$.
Let $K$ be a constant such that $d(z, r(z)) \le K$ for all $z \in Z$.
\par

Choose a constant $c$ such that
 $Y$ is $c$-coarsely connected, and a constant $c' \ge c$.
 Consider a $\Phi(c')$-loop $\eta$ in $Z$ at $z_0$.
 By hypothesis on $Y$, there exist $k'' \ge \Phi(c')$ and a sequence
 $\xi_0 = \eta, \xi_1, \hdots, \xi_\ell = (z_0)$ of $k''$-loops in $Y$ at $z_0$
 such that $\xi_{j-1}$ and $\xi_j$ are $k''$-elementarily homotopic
 for $j = 1, \hdots, \ell$.
 Then $r(\xi_0) = \eta, r(\xi_1), \hdots, r(\xi_\ell) = (z_0)$
 is a sequence of $\Phi(k'')$-loops in $Z$ at $z_0$
 such that $r(\xi_{j-1})$ and $r(\xi_j)$ are $\Phi(k'')$-elementarily homotopic
 for $j=1, \hdots, \ell$.
 This shows that $Z$ has Property \textup{(SC($\Phi(c'), \Phi(k'')$))}.
 Hence $Z$ is coarsely simply connected.
\end{proof}

\section{On simplicial complexes}
\label{sectionsc}
We begin with some generalities on simplicial complexes and loops in them.
\vskip.2cm

For a \textbf{simplicial complex} $X$, \index{Simplicial complex}
we denote by $X^0 \subset X^1 \subset X^2 \subset \cdots$
the nested sequence of its skeletons
and by $X_{\textnormal{top}}$ its \textbf{geometric realization}, 
\index{Realization! geometric realization of a simplicial complex|textbf}
which is a Hausdorff topological space obtained inductively
by attaching cells of dimension $1, 2, \hdots$ on $X^0$.
A \textbf{graph} \index{Graph}
is a simplicial complex of dimension~$1$;
it has simplices of dimensions $0$ and $1$ only
(as in Example \ref{metricrealizationgraph}).
We use ``vertex'' for ``$0$-cell'' and ``edge'' for ``$1$-cell''.

\begin{defn}
Let $X$ be a simplicial complex.
A \textbf{combinatorial path} 
\index{Combinatorial! path in a simplicial complex|textbf}
in $X$ from a vertex $x$ to a vertex $y$
is a sequence of oriented edges 
\begin{equation*}
\xi \, = \, 
\big( (x_0,x_1), (x_1,x_2), \hdots, (x_{n-1},x_m) \big)
\end{equation*}
with $x_0 = x$ and $x_m=y$;
such a path is denoted by the sequence of vertices
\begin{equation*}
\xi \, = \,  (x_0, x_1, \hdots, x_m) .
\end{equation*}
The \textbf{inverse path} \index{Inverse path|textbf}
is the path  $\xi^{-1} = (x_m, x_{m-1}, \hdots, x_0)$.
The \textbf{product} of two combinatorial paths  
\index{Product of combinatorial paths|textbf}
$(x_0, x_1, \hdots, x_m)$ and $(y_0, y_1, \hdots, y_n)$
is defined when $x_m=y_0$, and is then
$(x_0, \hdots, x_m, y_1, \hdots, y_n$).
A \textbf{combinatorial loop} \index{Combinatorial! loop|textbf}
in $X$ based at some vertex $x_0 \in X$
is a combinatorial path from $x_0$ to~$x_0$.
\par

On sets of combinatorial paths with fixed extremities, 
we define now two equivalence relations.
\par

Two combinatorial paths  in $X$ from $x$ to $y$ are
\textbf{elementarily graph homotopic} 
\index{Elementarily! graph homotopic combinatorial paths|textbf}
if they are of the form 
\begin{equation*}
(x_0,x_1,\dots,x_n) \hskip.2cm \text{and} \hskip.2cm 
(x_0,\dots,x_i,u,x_i,x_{i+1},\dots,x_n) ,
\hskip.2cm \text{with} \hskip.2cm x_0 = x, \hskip.1cm x_n = y ,
\end{equation*} 
where $(x_i,u)$ is an edge of $X$ (so that $(u,x_i)$ is the opposite edge).
Two combinatorial paths $\xi, \xi'$ in $X$ from $x$ to $y$ are \textbf{graph homotopic}
\index{Graph homotopic combinatorial paths|textbf}
if there exists a sequence $\xi_0 = \xi, \xi_1, \hdots, \xi_\ell = \xi'$
of combinatorial paths from $x$ to $y$
such that $\xi_{j-1}$ and $\xi_j$ are elementarily graph homotopic for $j = 1, \hdots, \ell$.
Observe that, for this relation, combinatorial paths 
can be viewed indifferently in $X$ or in $X^1$.

\par

Two combinatorial paths  in $X$ from $x$ to $y$ are
\textbf{triangle homotopic} 
\index{Triangle homotopic combinatorial paths|textbf}
if they are of the form 
\begin{equation*}
(x_0,x_1,\dots,x_n) \hskip.2cm \text{and} \hskip.2cm 
(x_0,\dots,x_i, u,x_{i+1},\dots,x_n) ,
\hskip.2cm \text{with} \hskip.2cm x_0 = x, \hskip.1cm x_n = y ,
\end{equation*} 
where $\{x_i,u,x_{i+1}\}$ is a 2-simplex in $X$.
Two combinatorial paths $\xi, \xi'$ in $X$ from $x$ to $y$ are 
\textbf{combinatorially homotopic}
\index{Combinatorially homotopic combinatorial paths|textbf}
if there exists a sequence $\xi_0 = \xi, \xi_1, \hdots, \xi_\ell = \xi'$
of combinatorial paths from $x$ to $y$
such that $\xi_{j-1}$ and $\xi_j$ are 
either elementarily graph homotopic or triangle homotopic, for $j = 1, \hdots, \ell$.
Observe that, for this relation, combinatorial paths 
can be viewed indifferently in $X$ or in $X^2$.
\par

Combinatorial homotopy between combinatorial paths
is an equivalence relation that is compatible with inverses and products;
in other words, combinatorial homotopy classes of combinatorial paths
are the elements of a groupoid.
More precisely, 
let $\xi, \xi', \eta, \eta'$ be combinatorial paths such that 
the product $\xi \eta$ is defined, $\xi$ is combinatorially homotopic to $\xi'$,
and $\eta$ is combinatorially homotopic to $\eta'$;
then $\xi^{-1}$ is combinatorially homotopic to ${\xi'}^{-1}$,
and $\xi' \eta'$ is defined and combinatorially homotopic to $\xi \eta$.
\end{defn}

\begin{defn}
Let $X$ be a simplicial complex.
Let $\xi = (x_0, x_1, \hdots, x_n)$ be a combinatorial path in $X$.
The \textbf{topological realization}
\index{Realization! topological realization of a combinatorial path|textbf}
of $\xi$ is a continuous path
$\xi_{\operatorname{top}} : I \longrightarrow X_{\textnormal{top}}$
with origin $x_0$ and end $x_n$, well-defined up to reparameterization,
where $I = [t_0, t_n] = \bigcup_{j=1}^n [t_{j-1},t_j]$ is an interval of the real line 
made up of $n$ subintervals with disjoint interiors,
and where $\xi_{\operatorname{top}}$ maps
successively $[t_0, t_1]$ onto the edge of $\xi$ from $x_0$ to $x_1$, ....,
and $[t_{n-1},t_n]$ onto the edge of $\xi$ from $x_{n-1}$ to $x_n$.
\end{defn}

\begin{lem}
\label{lem:path_X1}
Let $X$ be a connected simplicial complex, 
given with a base point $x_0 \in X^0$.
\par

(1)
A loop in $X_{\textnormal{top}}$ based at $x_0$ is homotopic to 
the topological realization of a combinatorial loop based at $x_0$.
\par

(2)
Let $\xi, \xi'$ be combinatorial paths, 
and $\xi_{\operatorname{top}}, \xi'_{\operatorname{top}}$ 
their topological realizations.
Then $\xi_{\operatorname{top}}$ and $\xi'_{\operatorname{top}}$ 
are homotopic in the topological sense
if and only if $\xi$ and $\xi'$ are combinatorially homotopic.
\end{lem}

\begin{proof} 
The reader can check it as an exercise, or refer to \cite[Chap.3, Sec.6,
Lemma~12 and Theorem~16]{Span--66}.
\end{proof}

\begin{lem}
\label{lem:graphomotopic_simplices}
Let $X$ be a connected simplicial complex, and $x_0 \in X^0$.
Let $\xi = (x_0,x_1,\dots, x_n)$ be a combinatorial loop, with $x_n = x_0$.
Suppose that $\xi_{\operatorname{top}}$ is homotopically trivial
as a loop in $X_{\textnormal{top}}$ based at $x_0$.

Then $\xi$ is combinatorially homotopic to a combinatorial loop 
$\prod_{j=1}^N u_j r_j u_j^{-1}$, 
where $u_j$ is a combinatorial path from $x_0$ to some vertex $z_j \in X^0$, 
and $r_j$ is a combinatorial loop based at $z_j$, 
such that all vertices of $r_j$ belong to a common $2$-simplex in $X$, 
and such that the length of $r_j$ is $3$, 
i.e., $r_j$ is of the form $(z_j ,z'_j ,z''_j ,z_j)$ 
for some $2$-simplex $\{ z_j ,z'_j ,z''_j \}$.
\end{lem}

\begin{proof} 
Consider two triangle homotopic combinatorial loops of the form
\begin{equation*}
\aligned
\eta \, &= \, (x_0, y_1, \hdots, y_{i-1}, y_i , y_{i+1}, \hdots, y_{k-1}, x_0)
\\
\eta' \, &= \, (x_0, y_1, \hdots, y_{i-1}, y_{i+1}, \hdots, y_{k-1}, x_0) ,
\endaligned
\end{equation*}
where $\{y_{i-1}, y_i, y_{i+1}\}$ is a $2$-simplex of $X$. Set
\begin{equation*}
\aligned
u \, &= \, (x_0, y_1, \hdots, y_{i-1}, y_{i+1})
\\
r \, &= \, (y_{i+1}, y_{i-1}, y_i, y_{i+1}) .
\endaligned
\end{equation*}
Then $\eta$ is elementarily graph homotopic to $uru^{-1}\eta'$;
similarly, $\eta'$ is elementarily graph homotopic to $ur^{-1}u^{-1}\eta$.
\par

By hypothesis and Lemma \ref{lem:path_X1}, there exists a sequence
$\xi_0 = \xi, \xi_1, \hdots, \xi_\ell = (x_0)$ of combinatorial loops such that,
for $j \in \{1, \hdots, \ell\}$, the loops $\xi_{j-1}$ and $\xi_j$ are
either elementarily graph homotopic (say for $M$ of the $j$ 's),
or triangle homotopic (say for $N$ of the $j$ 's, with $M+N = \ell$).
Applying $N$ times the argument written above for $\eta$ and $\eta'$,
we obtain the conclusion.
\end{proof}

\begin{prop}
\label{donrealizationofsimpcx}
Let $X$ be a connected simplicial complex.
\par

On the geometric realization $Z$ of the $2$-skeleton $X^2$ of $X$,
there exists a unique \textbf{combinatorial metric} $d_2$ 
\index{Combinatorial! metric on a simplicial complex|textbf}
that makes each edge of $Z$ an interval of length $1$,
each $2$-cell of $Z$ 
a Euclidean equilateral triangle of side-length $1$,
and such that $(Z, d_2)$ 
is a complete geodesic space.
\par

If $d_1$ denotes the combinatorial metric
on the geometric realization $Y$ of $X^1$ (Example  \ref{metricrealizationgraph}), 
the inclusion $(Y,d_1) \subset (Z,d_2)$ is a quasi-isometry.
\end{prop}

\noindent \emph{Note.} If $X$ is a non-connected simplicial complex,
the same construction yields a \textbf{combinatorial \'ecart}
\index{Combinatorial! \'ecart|textbf}
on the geometric realization of $X^2$ (see Definition \ref{pm...metrizable}).

\begin{proof}
We refer to \cite[Theorem 7.19 and Proposition 7.31]{BrHa--99}.
\end{proof}

\section{The Rips $2$-complex of a pseudo-metric space}
\label{sectionRips}

\begin{defn}
\label{defRipst}
Let $(X,d)$ be a non-empty pseudo-metric space; let $c >0$ be a constant.
The \emph{Rips simplicial $2$-complex} $\operatorname{Rips}^2_c(X,d)$ 
is the simplicial complex of dimension $2$
with $X$ as set of vertices,
pairs $(x,y)$ of distinct points of $X$ with $d(x,y) \le c$ as set of oriented edges,
and triples $(x,y,z)$ of distinct points of $X$ with mutual distances bounded by $c$,
up to cyclic permutations, as set of oriented $2$-simplices.
\par

The \textbf{Rips $2$-complex} 
\index{Rips $2$-complex! $\operatorname{Rips}^2_c(X,d)$|textbf}
is the geometric realization of this complex,
denoted again by $\operatorname{Rips}^2_c(X,d)$,
endowed with the combinatorial \'ecart $d_2$
of Proposition \ref{donrealizationofsimpcx}.
Note that any connected component of $\left( \operatorname{Rips}^2_c(X,d), d_2 \right)$
is a complete geodesic space.
\par

Whenever convenient, we write $\operatorname{Rips}^2_c(X)$
for $\operatorname{Rips}^2_c(X,d)$.
\end{defn}

Rips complexes are usually defined as complexes of unbounded dimension;
but our restriction to dimension $2$ will suffice below,
because we are not interested here in coarse analogues of connectedness conditions
$\pi_j(-) = \{0\}$ for $j \ge 2$.
Complexes $\operatorname{Rips}_c(-)$ have been used by Rips
and Gromov in the context of hyperbolic groups \cite[Section~1.7]{Grom--87},
but the idea goes back to Vietoris \cite{Viet--27, Haus--95}.
\par

If $(X,d_c)$ is the metric space associated to a connected graph
and its combinatorial metric, as in Example \ref{metricrealizationgraph},
there is a variant of the Rips $2$-complex, $\operatorname{P}_k(X)$,
obtained from the graph $X$ by adding $\ell$-gons 
at every combinatorial loop of length $\ell \le k$.
The spaces $\operatorname{Rips}^2_c(X)$ and $\operatorname{P}_k(X)$
are quasi-isometric.
The two spaces are useful for the same kind of purposes,
but $\operatorname{P}_k(X)$
is sometimes more convenient: see \cite{SaTe}.

\begin{prop}[on the inclusion of $X$ in its Rips complex]
\label{inclusionXinRips}
Let $X$ be a non-empty pseudo-metric space, with a base point $x_0$,
and $\operatorname{Rips}^2_c(X)$ its Rips $2$-complex.
For a constant $c>0$:
\begin{enumerate}[noitemsep,label=(\arabic*)]
\item\label{1DEinclusionXinRips}
$X$ is $c$-coarsely connected
if and only if $\operatorname{Rips}^2_c(X)$ is connected;
\item\label{2DEinclusionXinRips}
$X$ is $c$-coarsely geodesic  
if and only if  the natural inclusion 
$X \lhook\joinrel\relbar\joinrel\rightarrow \operatorname{Rips}^2_c(X)$
is a metric coarse equivalence;
\item\label{3DEinclusionXinRips}
$X$ is $c$-large-scale geodesic 
if and only if  the natural inclusion 
$X \lhook\joinrel\relbar\joinrel\rightarrow \operatorname{Rips}^2_c(X)$
is a quasi-isometry.
\end{enumerate}
For two constants $c'' \ge c' > 0$, there is a canonical inclusion 
$j : \operatorname{Rips}^2_{c'}(X) 
\lhook\joinrel\relbar\joinrel\rightarrow \operatorname{Rips}^2_{c''}(X)$,
which is the identity on the $0$-skeletons. We have:
\begin{enumerate}[noitemsep,label=(\arabic*)]
\addtocounter{enumi}{3}
\item\label{4DEinclusionXinRips}
if $\operatorname{Rips}^2_{c'}(X)$ is connected, so is $\operatorname{Rips}^2_{c''}(X)$;
\item\label{5DEinclusionXinRips}
the space $X$ has Property \eqref{SC($c',c''$)} of Definition \ref {defSC}
if and only if the homomorphism of fundamental groups
$j_* : \pi_1(\operatorname{Rips}^2_{c'}(X)) 
\longrightarrow \pi_1(\operatorname{Rips}^2_{c''}(X))$ 
is trivial. 
\end{enumerate}
\end{prop}

\begin{proof}
Claims \ref{1DEinclusionXinRips} to \ref{4DEinclusionXinRips} 
are straightforward consequences of the definitions.
Let us check Claim~\ref{5DEinclusionXinRips}.
\par

First, assume that $X$ has Property \eqref{SC($c',c''$)}.
Consider $\gamma \in \pi_1(\operatorname{Rips}^2_{c'}(X))$.
By Lemma \ref{lem:path_X1}(1), $\gamma$ can be represented by a
combinatorial loop based at $x_0$, 
and the latter defines a $c'$-loop in $X$ at $x_0$, say $\xi$.
Since $X$ has Property \eqref{SC($c',c''$)},
there is a $c''$-homotopy from this $c'$-loop $\xi$ to the trivial loop $(x_0)$.
This $c''$-homotopy provides a combinatorial homotopy
from $\xi$ viewed as a combinatorial loop in $\operatorname{Rips}^2_{c''}(X)$
to the trivial loop. Hence $j_*(\gamma) = 1 \in \pi_1(\operatorname{Rips}^2_{c''}(X))$.
\par

Assume, conversely,  that $j_*$ has trivial image. 
Let $\xi$ be a $c'$-loop in $X$ at $x_0$.
This can be viewed as a combinatorial loop in $\operatorname{Rips}^2_{c'}(X)$.
Since the image of $j_*$ is trivial, 
there is a homotopy from  $\xi_{\operatorname{top}}$ 
viewed as a loop in  $\operatorname{Rips}^2_{c''}(X)$ to the constant loop. 
By Lemma \ref{lem:path_X1}(2), there exists a combinatorial homotopy
from $\xi$ viewed as a combinatorial loop in $\operatorname{Rips}^2_{c''}(X)$
to the constant loop.
The latter combinatorial homotopy provides a $c''$-homotopy
from the $c'$-loop $\xi$ to the constant loop.
Hence $X$ has Property \eqref{SC($c',c''$)}.
\end{proof}

\begin{lem}
\label{lemmapourc-geodesicspace}
Let $c, c', c''$ be constants with $c'' \ge c' \ge c > 0$,
and $X$ a $c$-geodesic metric space, with base point $x_0$.
\par

If $X$ has Property \textup{(SC($c',c''$))}, 
then $X$ has Property   \textup{(SC($c'',c''$))}. 
\end{lem}

\begin{proof}
Let $\xi = (x_0, x_1, \hdots, x_m = x_0)$ be a $c''$-loop in $X$ at $x_0$.
Since $X$ is $c'$-geodesic, 
there is a way of inserting new points in between those of $\xi$ 
to obtain a $c'$-loop $\eta = (x_0, y_1 \hdots, y_n = x_0)$
that is $c''$-homotopic to $\xi$.
If $X$ has Property \textup{(SC($c',c''$))}, this loop $\eta$
is $c''$-homotopic to the constant loop $(x_0)$. 
Hence $\xi$ is $c''$-homotopic to $(x_0)$. 
The conclusion follows.
\end{proof}

\begin{prop}
\label{propositionpourc-geodesicspace}
Let $X$ be a non-empty pseudo-metric space, with $x_0 \in X$.
Assume that $X$ is $c$-geodesic, for some $c > 0$.
The following three properties are equivalent:
\begin{enumerate}[noitemsep,label=(\roman*)]
\item\label{iDEpropositionpourc-geodesicspace}
the space $X$ is $c$-coarsely simply connected;
\item\label{iiDEpropositionpourc-geodesicspace}
there exists $k \ge c$ such that $\operatorname{Rips}^2_{k}(X)$ is simply connected;
\item\label{iiiDEpropositionpourc-geodesicspace}
there exists $k \ge c$ such that $\operatorname{Rips}^2_{K}(X)$ is simply connected
for all $K \ge k$.
\end{enumerate}
\end{prop}

\begin{proof}
\ref{iDEpropositionpourc-geodesicspace} $\Rightarrow$ 
\ref{iiiDEpropositionpourc-geodesicspace}
Suppose that \ref{iDEpropositionpourc-geodesicspace} holds.
There exists $k \ge c$ such that
$X$ has Property \textup{(SC($c,k$))}.
By Lemma \ref{lemmapourc-geodesicspace},
$X$ has also Property \textup{(SC($K,K$))}
for all $K \ge k$.
Property \ref{iiiDEpropositionpourc-geodesicspace} 
follows from Proposition \ref{inclusionXinRips}(5).
\par

\ref{iiDEpropositionpourc-geodesicspace} $\Rightarrow$ 
\ref{iDEpropositionpourc-geodesicspace}
Suppose that \ref{iiDEpropositionpourc-geodesicspace} holds.
Since $\operatorname{Rips}^2_{k}(X)$ is a complete geodesic space,
it is coarsely simply connected, by Proposition \ref{scimplycsc}.
Since $X$ is $c$-geodesic, the inclusion of $X$ in $\operatorname{Rips}^2_{k}(X)$
is a quasi-isometry, by Proposition \ref{inclusionXinRips}(3).
Hence $X$ has Property \ref{iDEpropositionpourc-geodesicspace}, 
by Proposition \ref{coarse1conninvbycoarseeq}.
\end{proof}

\begin{exe}
\label{exdeNavecdeshighways}
In Proposition  \ref{propositionpourc-geodesicspace}, the condition on $X$
to be $c$-geodesic cannot be relaxed to being large-scale geodesic, 
as we show here with an example
of a metric space that is bilipschitz to $\N$ (the half-line graph), 
and such that, for every $c > 0$, the Rips $2$-complex
$\operatorname{Rips}^2_c(X)$ is not simply connected.
\par

The idea is to start from the graph $\N$, 
and to force some pairs of points at distance $3n$ to have distance $n$. 
This can be constructed as a connected (but not geodesic) subgraph of a larger graph, 
as follows.
\par

Let $Y$ be the graph with vertex set
\begin{equation*}
\operatorname{Vert}(X) \ = \, \left\{ u_n \right\}_{n \in \N} \sqcup
\Big( \bigsqcup_{n \in \N} \{ v_{n,1}, v_{n,2}, \hdots, v_{n, n-1} \} \Big)
\end{equation*}
and edge set consisting of
\begin{equation*}
\aligned
&\text{one edge} \hskip.2cm 
\{u_{n}, u_{n+1}\} 
\hskip.2cm \text{for all} \hskip.2cm n \ge 0,
\\
&\text{$n$ edges} \hskip.2cm 
\{u_{10^n}, v_{n,1}\}, \{v_{n,1}, v_{n,2}\}, \hdots, \{v_{n,n-1},u_{10^n+3n} \}
\hskip.2cm \text{for all} \hskip.2cm n \ge 1.
\endaligned
\end{equation*}
Let $d$ denote the combinatorial metric on $Y$ (Example \ref{metricrealizationgraph}).
The map
\begin{equation*}
u \, : \, \N \longrightarrow Y , \hskip.2cm n \longmapsto u_n
\end{equation*}
is bilipschitz, with
$\frac{1}{3} \vert m-n \vert \le d(u_m, u_n) \le \vert m-n \vert$ for all $m,n \in \N$.
\par

Set $X = u(\N)$, with the restriction $d$ of the metric of $(Y,d)$.
The metric space $(X,d)$ is large-scale geodesic, because it is quasi-isometric to $\N$.
Observe that, for all $n \ge 1$,
\begin{equation*}
\aligned
&\cdot \hskip.2cm 
d(u_{10^n}, u_{10^n+3n}) \, = \,  n 
\hskip.1cm ,
\\
&\cdot \hskip.2cm 
\text{the length of the path} \hskip.2cm
(u_{10^n}, u_{10^n+1}, u_{10^n+2}, \hdots, u_{10^n+3n-1}, u_{10^n+3n}) \hskip.2cm
\text{is} \hskip.2cm 3n 
\hskip.1cm ;
\endaligned
\end{equation*}
hence $(X,d)$ is not geodesic.
\par

Let $c > 0$ be a constant. 
If $c < 1$, the Rips $2$-complex $\operatorname{Rips}^2_c(X)$ is not connected.
Assume $c \ge 1$, so that $\operatorname{Rips}^2_c(X)$ is connected.
Let $n$ be the integer such that $n \le c < n+1$.
In $\operatorname{Rips}^2_c(X)$, there is an edge, $\{u_{10^n}, u_{10^n + 3n} \}$,
that is not contained in any $2$-simplex.
It follows that a loop in $\operatorname{Rips}^2_c(X)$ going around
\begin{equation*}
u_{10^n}, \hskip.1cm
u_{10^n+1}, \hskip.1cm
u_{10^n+2}, \hskip.1cm
\hdots,
u_{10^n+3n}, \hskip.1cm
u_{10^n}
\end{equation*}
is not homotopic to a constant loop, so that
$\operatorname{Rips}^2_c(X)$ is not simply connected.
\end{exe}

\begin{prop}
\label{invariance_scg_lsg}
Let $X$ be a non-empty pseudo-metric space, with $x_0 \in X$.
Assume that $X$ is coarsely geodesic. The following properties are equivalent:
\begin{enumerate}[noitemsep,label=(\roman*)]
\item\label{iDEinvariance_scg_lsg}
$X$ is coarsely simply connected;
\item\label{iiDEinvariance_scg_lsg}
$X$ is coarsely equivalent to a simply connected geodesic metric space.
\end{enumerate}
\end{prop}

\begin{proof}
\ref{iDEinvariance_scg_lsg} $\Rightarrow$ \ref{iiDEinvariance_scg_lsg}. 
Assume that $X$ has Property \ref{iDEinvariance_scg_lsg}.
Since $X$ is coarsely geodesic, there exists a geodesic metric space $Z$
coarsely equivalent to $X$  (Proposition \ref{invariance_cg_lsg}(4)).
Since $X$ is coarsely simply connected, so is $Z$
(Proposition \ref{coarse1conninvbycoarseeq}).
For $c$ large enough, 
$\operatorname{Rips}^2_c(Z)$ is simply connected
(Proposition  \ref{propositionpourc-geodesicspace})
and the inclusion $Z  \subset \operatorname{Rips}^2_c(Z)$ is a metric coarse equivalence
(Proposition  \ref{inclusionXinRips}).
It follows that $X$ and the simply connected geodesic metric space
$\operatorname{Rips}^2_c(Z)$ are coarsely equivalent.
\par

\ref{iiDEinvariance_scg_lsg} $\Rightarrow$ \ref{iDEinvariance_scg_lsg}. 
Assume conversely that
$X$ is coarsely equivalent to a simply connected geodesic metric space $Y$.
Then $Y$ is coarsely simply connected (Proposition \ref{scimplycsc}),
and so is $X$ (Proposition \ref{coarse1conninvbycoarseeq}). 
\end{proof}

\chapter{Bounded presentations}
\label{chap_boundedpresentation}

\section{Presentations with relators of bounded length}
\label{sectionboundedlength}

We revisit  Definition \ref{cggroups} as follows:

\begin{defn}[generation]
\label{generation}
A \textbf{generation} \index{Generation of a group|textbf}
of a group $G$ is a pair made up of a set $S$
and a surjective homomorphism $\pi : F_S \twoheadrightarrow G$ 
of the free group on $S$ onto $G$. \index{Free group}
The \textbf{relations} \index{Relations in a group with generating set|textbf}
of such a generation are the elements of $\textnormal{Ker} (\pi)$.
\end{defn}

\begin{defn}[presentation]
\label{presentation}
A \textbf{presentation} \index{Presentation of a group|textbf}
of a group $G$ is a triple made up
of a generation $(S, \pi)$ as above and a subset $R$ of $F_S$
generating $\textnormal{Ker} (\pi)$ \emph{as a normal subgroup}.
The notation
\begin{equation*}
G \, = \, \langle S \mid R \rangle
\end{equation*}
summarizes such data. The set $S$ is often called an \textbf{alphabet}.
\par

Such a subset $R \subset F_S$ is called a \textbf{relating subset};
its elements are the \textbf{relators} 
\index{Relators of a presentation|textbf}
of the presentation.
Observe that the relations
are the elements of the form $\prod_{i=1}^k w_i r_i w_i^{-1}$, 
with $k \ge 0$, $r_1, \hdots, r_k \in R \cup R^{-1}$,
and $w_1, \hdots, w_k \in F_S$.
\par

A group is \textbf{finitely presented} 
\index{Finitely presented group|textbf}
if it is given as $G = \langle S \mid R \rangle$ with both $S$ and $R$ finite,
and \textbf{finitely presentable} if it can be given in such a way.
It is improper, but nevertheless standard, to use ``finitely presented group''
for ``finitely presentable group''.
\par

Consider two alphabets $S_1, S_2$, 
the free group $F_1$ on $S_1$ 
and the free group $F$ on the disjoint union $S_1 \sqcup S_2$,
and identify naturally $F_1$ to a subgroup of $F$.
Consider also subsets $R_1 \subset F_1$ and $R_2 \subset F$.
We can define two groups $G_1, G$ by the presentations
$G_1 = \langle S_1 \mid R_1 \rangle$, $G = \langle S_1, S_2 \mid R_1, R_2 \rangle$.
We write shortly (and abusively)
\begin{equation*}
G \, = \, \langle G_1, S_2 \mid R_2 \rangle
\hskip.5cm \text{for} \hskip.5cm
G \, = \, \langle S_1, S_2 \mid R_1 , R_2 \rangle .
\end{equation*}
In some important cases, such as amalgamated products and HNN-extensions
as in Section  \ref{sectionamalgamHNN},
it can be established that the natural homomorphism $G_1 \longrightarrow G$ is injective.
\end{defn}

\begin{rem}
\label{remgenerationpresentation}
(1)
In Definitions \ref{generation} and \ref{presentation},
as in the footnote to \ref{PestovEtc}(3),
we emphasize that, even if $G$ is a topological group and $S$ a subset of $G$,
the free group $F_S$ is \emph{not} considered with any topology.

\vskip.2cm

(2)
If the restriction of $\pi$ to $S$ is injective,
it is standard to identify $S$ with its image in $G$.

\vskip.2cm

(3) Let $\pi : F_S \longrightarrow G = \langle S \mid R \rangle$
be a presentation of a group $G$.
Let $\underline R$ be the natural image of $R$ in $F_{\pi(S)}$;
in other words,  the elements of $\underline R$ 
are the words in the letters of $\pi(S)$ and their inverses
obtained from the relators in $R$ 
first by substitution of $\pi(s)$ for $s$,  for all $s \in S \cup S^{-1}$,
and then by deletion of those $\pi(s)$ equal to $1$ in $G$.
Let $\pi(S)'$ stand for $\pi(S) \smallsetminus \{1\}$ if $1 \in \pi(S)$,
and for $\pi(S)$ otherwise;
we have a generation $F_{\pi(S)'} \longrightarrow G$
and a presentation $\langle \pi(S)' \mid \underline R \rangle$ of $G$.
\end{rem}

\begin{exe}
\label{expresentation1}
Let $G$ be a group acting by homeomorphisms on a non-empty topological space $X$;
let $U$ be an open subset of $X$ such that $GU = X$. 
Set
\begin{equation*}
\aligned
S \, &= \, \{ s \in G \mid U \cap sU \ne \emptyset \} ,
\\
R \, &= \, \{ stu^{-1} \in F_S \mid 
s,t,u \in S \hskip.1cm \text{with} \hskip.1cm
st = u \hskip.1cm \text{in} \hskip.1cm G \hskip.1cm \text{and} \hskip.1cm
U \cap sU \cap uU \ne \emptyset \} .
\endaligned
\end{equation*}
\begin{enumerate}[noitemsep,label=(\arabic*)]
\item\label{1DEexpresentation1}
If $X$ is connected, $S$ generates $G$.
\item\label{2DEexpresentation1}
If $X$ is simply connected and $X,U$ are path-connected,
$\langle S \mid R \rangle$ is a presentation of $G$.
\end{enumerate}
\end{exe}

\noindent \emph{Note:} if the action is proper and $U$ is relatively compact, 
then $S$ and $R$ are finite.

\begin{proof}
\ref{1DEexpresentation1}
Let $H$ be the subgroup of $G$ generated by $S$.
Set $Y = HU$ and $Z = (G \smallsetminus H)U$;
they are open subsets of $X$.
Then $Y \cap Z = \emptyset$;
indeed, otherwise there would exist $h \in H$ and $g \in G \smallsetminus H$
such that $h^{-1}gU \cap U \ne \emptyset$,
hence such that $g \in hS \subset H$,
in contradiction with $g \in G \smallsetminus H$.
Since $Y$ is non-empty (because $U \subset Y$)
and $X$ connected, $Y = X$.
\par

Consider now $g \in G$. There exists $h \in H$ with $gU \cap hU \ne \emptyset$.
Thus $h^{-1}g \in S$, and $g \in hS \subset H$. Hence $H = G$.
\par

\ref{2DEexpresentation1}
The proof, that we do not reproduce here, has three steps.
First one introduces an $S$-labelled ``Cayley complex'' $K$,
with vertex set $G$, labelled edges of the form $(g_1, g_2, s)$ where $g_2 = g_1s$, $s \in S$,
and $2$-cells attached to edge-loops labelled by words in $R$.
Then one shows that $K$ is simply connected.
From this one concludes that $G = \langle S \mid R \rangle$.
\par

The result is from \cite{Macb--64}.
Actual proofs can also be found in 
\cite[Appendix of Section I.3, Pages 45--46]{Serr--77}
and \cite[Theorem I.8.10, Page 135]{BrHa--99}.
Serre indicates related references.
\end{proof}

\begin{defn}[bounded presentation]
A \textbf{bounded presentation} for a group $G$ 
\index{Bounded! bounded presentation|textbf}
is a presentation $G = \langle S \mid R \rangle$
with $R$ a set of relators of bounded length.
A group is \textbf{boundedly presented over $S$} 
\index{Boundedly! boundedly presented group|textbf}
if it is given by such a presentation.
\end{defn}

\begin{exe}
\label{examplesboundpres}
(1)
Let $G$ be any group. Let $S$ be a set given with a bijection
$G \longrightarrow S, g \longmapsto s_g$.
In the free group $F_S$ on $S$, define the subset
\begin{equation*}
R \, = \,  \{s_g s_h s_k^{-1} \in F_S \mid
g,h \in G \hskip.2cm \text{and} \hskip.2cm k = gh \} .
\end{equation*}
Then $G = \langle S \mid R \rangle$ is a bounded presentation,
with all relators of length $3$.

\vskip.2cm

(2)
If $G$ has a finite generating set $S$, 
then $G$ is boundedly presented over $S$
if and only if $G$ is finitely presentable.

\vskip.2cm

(3 -- Coxeter groups and Artin-Tits groups).
Let $S$ be an abstract set (possibly infinite) with Coxeter data, 
i.e.,  with a symmetric matrix $M = (m_{s,t})_{s,t \in S}$
having coefficients in $\N_{\ge 1} \cup \{\infty\}$,
with $m_{s,s} = 1$ for all $s \in S$ and
$m_{s,t} \ge 2$ whenever $s \ne t$.
The corresponding \textbf{Coxeter group}
\index{Coxeter group|textbf}
is defined by the presentation
\begin{equation*}
W_M \, = \, 
\langle 
S \mid
(st)^{m_{s,t}}  = 1 \hskip.2cm \forall \hskip.1cm s,t \in S
\hskip.2cm \text{with} \hskip.2cm  m_{s,t} < \infty
\rangle ,
\end{equation*}
and the \textbf{Artin-Tits group} by \index{Artin-Tits group|textbf}
\begin{equation*}
A_M \, = \, 
\langle S \mid  
s t s \cdots = t s t \cdots 
\hskip.2cm \forall \hskip.1cm s,t \in S
\hskip.2cm \text{with} \hskip.2cm s \ne t
\hskip.2cm \text{and} \hskip.2cm  m_{s,t} < \infty
\rangle
\end{equation*}
(with a product of $m_{s,t}$ factors on each side of the last equality sign).
These presentations are bounded
if and only if 
$\sup \{ m_{s,t} \mid s,t \in S \hskip.2cm \text{with} \hskip.2cm m_{s,t} < \infty \}$
is finite. 
\par

For general $M$ as above, Artin-Tits groups are rather mysterious. 
For example, it is conjectured but unknown that $A_M$ is always torsion-free.
This and other open problems on Artin-Tits groups are discussed in \cite{GoPa--12}.

\vskip.2cm

(4 -- Steinberg groups).
Let $R$ be an associative ring (not necessarily commutative or unital) 
and $n$ an integer, $n \ge 3$. 
The \textbf{Steinberg group}  $\operatorname{St}_n(R)$
\index{Steinberg group|textbf}
is the group defined by the following bounded presentation:
there are generators $e_{i,j}(r)$ indexed by 
pairs $(i,j) \in \{1, \hdots, n\}^2$ with $i \ne j$ and $r \in R$, 
and relators
\begin{equation*}
\aligned
e_{i,j}(r)e_{i,j}(s)=e_{i,j}(r+s), & \quad i,j\textnormal{ distinct,} 
\\
[e_{i,j}(r),e_{j,k}(s)]=e_{i,k}(rs), & \quad i,j,k\textnormal{ distinct,} 
\\
[e_{i,j}(r),e_{k,\ell}(s)]=1, & \quad i \ne j, \ell \hskip.2cm \text{and} \hskip.2cm  k \ne j, \ell ,
\endaligned
\end{equation*}
all of length $\le 5$ (where $[a,b]$ stands for $aba^{-1}b^{-1}$).
\par

When $n=2$, the group $\operatorname{St}_2(R)$
is usually defined with additional generators and relators;
we refer to \cite[Pages 40 and 82]{Miln--71}.

\vskip.2cm

(5) For $R$ and $n$ as in (4), there is a natural homomorphism
$\operatorname{St}_n(R) \longrightarrow \SL_n(R)$.
In some cases, we know a presentation of $\SL_n(R)$
closely related to this homomorphism and the presentation of (4). 
For example, when $R = \Z$ and $n \ge 3$,
we have the \emph{Steinberg presentation}
\index{Special linear group $\SL$! $\SL_n(\Z)$, $\GL_n(\Z)$}
\footnotesize
\begin{equation*}
\SL_n(\Z) \, = \, \left\langle e_{i,j}  
\aligned
\hskip.2cm & \text{with} \hskip.2cm i,j =1, \hdots, n 
\\
\hskip.2cm & \text{and} \hskip.2cm i \ne j
\endaligned
\hskip.2cm \Bigg\vert \hskip.2cm
\aligned
& [e_{i,j}, e_{j,k}] = e_{i,k} 
\hskip.2cm \text{if} \hskip.2cm i,j,k \hskip.2cm \text{are distinct} 
\\
& [e_{i,j}, e_{k,\ell}] = 1
\hskip.2cm \text{if} \hskip.2cm i \ne j, \ell 
\hskip.2cm  \text{and} \hskip.2cm k \ne j, \ell
\\
& (e_{1,2}e_{2,1}^{-1}e_{1,2})^4 = 1 
\endaligned
\right\rangle
\end{equation*}
\normalsize
with relators of length $\le 12$
\cite[$\S$~10]{Miln--71}.
Here, $e_{i,j}$ is the \textbf{elementary matrix} in $\SL_n(\Z)$,
\index{Elementary matrices|textbf}
with $1$ 's on the diagonal and at the intersection of the $i$th row and the $j$th column,
and $0$ 's elsewhere.
For $\SL_2(\Z)$, we refer to \cite[Theorem 10.5]{Miln--71}.
\end{exe}

\begin{defn}[defining subset]
Let $G$ be a group and $S$ a generating subset. 
We say that $S$ is a defining generating subset, 
or shortly a \textbf{defining subset}, \index{Defining subset in a group|textbf}
if $G$ is presented with $S$ as generating set, 
and $\{ stu^{-1} \in F_S \mid s,t,u \in S \hskip.2cm \text{and} \hskip.2cm st=u \}$
as set of relators.
\end{defn}

Clearly, if $S$ is a defining generating subset of $G$, 
then $G$ is boundedly presented over~$S$. 
The following lemma, which is \cite[Hilfssatz 1]{Behr--67},
is a weak converse. 
Recall that ${\widehat S}^n := (S \cup S^{-1} \cup \{1\})^n$
has been defined in \ref{cggroups}.

\begin{lem}
\label{lem:boun_pres_def_gen}
Let $G$ be a group given with a bounded presentation
\begin{equation*}
G \, = \, \left\langle S \mid
R = \left(r_\iota\right)_{\iota \in I} \right\rangle
\end{equation*}
with relators $r_\iota$ of length bounded by some integer $n$.
Let $m$ be the integer part of $\frac{n+2}{3}$.
Then ${\widehat S}^m$ is a defining generating subset of $G$.
\end{lem}

\begin{proof}
Without loss of generality, we assume that the generating set
is symmetric and contains $1$, i.e., that $\widehat S = S$;
thus 
\begin{equation*}
\{1\} = S^0 \, \subset \, S^1 = S \,  \subset \,  S^2 \, \subset \,  S^3 \, \subset \,  \cdots .
\end{equation*}
Since $1 \in S$, we can assume that each relator $r \in R$
has length exactly $n$.
\par

For every $m \ge 1$, let $S_m$ be a set given 
with a bijection with the subset $S^m$ of $G$.
For a group element $w \in S^m$, 
we write $\underline w$ the corresponding ``letter'' in $S_m$.
Moreover, we identify $S_m$ with the appropriate subset of $S_{m+1}$.
\par

For $m \ge 2$, define $Q_m$ to be the set of words in the letters of $S_m$
consisting of the relators $\underline{w} \hskip.1cm \underline{s} = \underline{ws}$,
with $\underline{w} \in S_{m-1}$, $\underline{s} \in S_1$, 
and therefore $\underline{ws} \in S_m$;
such an equation stands improperly for a word
$\underline{w} \hskip.1cm \underline{s} \hskip.1cm (\underline{ws})^{-1}$ 
of length at most $3$.
We have
\begin{equation*}
\aligned
G \, = \, \left\langle S \mid  R \right\rangle 
\, &= \, \left\langle S_2 \mid  R \cup Q_2 \right\rangle 
\, = \, \left\langle S_3 \mid  R \cup Q_2 \cup Q_3 \right\rangle 
\\
\, = \, \cdots
\, &= \, \left\langle S_m \mid  R \cup Q_2 \cup \cdots \cup Q_m \right\rangle .
\endaligned
\end{equation*}
Assume now that $m \ge \frac{n+2}{3}$.
Each $r_\iota \in R$ can be first split 
in at most three subwords of lengths at most $m$,
and then rewritten as a word of length at most three
in the letters of $S_m$.
Hence $G = \left\langle S_m \mid   Q_2 \cup \cdots \cup Q_m \right\rangle$.
If $R'$ is the set of relators of the form 
$\underline{u} \hskip.1cm \underline{v} \hskip.1cm \underline{w}^{-1}$,
with $\underline{u},  \underline{v}, \underline{w} \in S_m$
and $uv = w \in S^m \subset G$,
then  $Q_2 \cup \cdots \cup Q_m \subset R'$ and $G = \left\langle S^m \mid  R' \right\rangle$.
Hence $S^m$ is a defining subset of $G$.
\end{proof}

\begin{lem}
\label{boundedpresentationforS1andS2}
Let $G$ be a group, $S_1,S_2$ generating subsets of $G$,
and $m \ge 1$ an integer.
\begin{enumerate}[noitemsep,label=(\arabic*)]
\item\label{1DEboundedpresentationforS1andS2}
$G$ is boundedly presented over ${\widehat S}_1^m$
if and only if
it is boundedly presented over $S_1$. 
\item\label{2DEboundedpresentationforS1andS2}
If $S_1 \subset S_2 \subset {\widehat S}_1^m$
and if $G$ is boundedly presented over $S_1$, 
then it is boundedly presented over $S_2$.
\item\label{3DEboundedpresentationforS1andS2} 
If ${\widehat S}_1^m \subset {\widehat S}_2^n \subset {\widehat S}_1^{m'}$ 
for some $n,m' \ge 1$, 
then $G$ is boundedly presented over $S_1$ if and only if it is
boundedly presented over $S_2$.
\item\label{4DEboundedpresentationforS1andS2}
Assume that $G$ is an LC-group; 
let $S_1, S_2$ be two compact generating sets.
If $G$ is boundedly presented over $S_1$,
then $G$ is boundedly presented over $S_2$.
\end{enumerate}
\end{lem}

\begin{proof}
\ref{1DEboundedpresentationforS1andS2}
This claim is straightforward;
compare with the proof of Lemma \ref{lem:boun_pres_def_gen}.
\par

\ref{2DEboundedpresentationforS1andS2}
Let $G = \langle S_1 \mid R_1 \rangle$ be a bounded presentation,
with relators in $R_1$ of $S_1$-lengths at most $b$, say.
Since $S_2 \subset {\widehat S}_1^m$,
there exists for each $s \in S_2 \smallsetminus S_1$ 
a $S_1$-word $w_s$ of length at most $m$ such that $s = w_s$;
let $R_2$ denote the set of $S_2$-words of the form $s^{-1}w_s$.
Then $\langle S_2 \mid R_1 \sqcup R_2 \rangle$
is a presentation of $G$ in which each relator is of length
at most $\max \{b,m+1\}$.
\par

\ref{3DEboundedpresentationforS1andS2} 
Suppose that $G$ is boundedly presented over $S_1$. 
Then $G$ is boundedly presented 
\begin{enumerate}[noitemsep]
\item[]
over ${\widehat S}_1^m$, by \ref{1DEboundedpresentationforS1andS2},
\item[]
and over ${\widehat S}_ 2^n$, by \ref{2DEboundedpresentationforS1andS2} applied to
${\widehat S}_1^m \subset {\widehat S}_2^n \subset {\widehat S}_1^{mm'}$, 
\item[]
and also over $S_2$, by \ref{1DEboundedpresentationforS1andS2} again.
\end{enumerate}
Since ${\widehat S}_2^{nm} \subset {\widehat S}_1^{m'm} \subset {\widehat S}_2^{m'n}$,
the converse holds by the same argument.
\par

\ref{4DEboundedpresentationforS1andS2}
We may assume that $m$ is such that ${\widehat S}_1^m$ is a compact neighbourhood of~$1$
(Proposition \ref{powersSincpgroup}).
The group $G$ is boundedly presented over ${\widehat S}_1^m$, 
by \ref{1DEboundedpresentationforS1andS2}.
Moreover, for appropriate values of $n$ and $m$,
we have ${\widehat S}_1^m \subset {\widehat S}_2^n \subset {\widehat S}_1^{m'}$.
Thus \ref{4DEboundedpresentationforS1andS2} 
follows from~\ref{3DEboundedpresentationforS1andS2}.
\end{proof}

\begin{exe}
\label{rem:Z_not_bp_by_big_subset}
(1)
Any ``boundedly generated'' group is boundedly presented.
More precisely, let $G$ be a group, $S$ a symmetric generating set,
and $d_S$ the corresponding word metric on $G$.
Assume that the metric space $(G,d_S)$ has finite diameter,
i.e., that $N := \sup_{g \in G} d_S(1,g) < \infty$.
Then $G$ is boundedly presented over~$S$.
\par

This follows from Lemma 7.A.9(1),
since $G$ is boundedly presented over
$\widehat S^N = G$.
\par

Note that there are infinite groups, 
for example the full symmetric group 
$G = \operatorname{Sym}(X)$ of an infinite set $X$, 
such that the diameter of $(G, d_S)$ is finite for \emph{every} generating set $S$.
\index{Symmetric group $\operatorname{Sym}(X)$ of a set $X$}
See  \cite{Berg--06}, as well as our Definition \ref{defTSB} and Example 
\ref{Symuncountablecof}.

\vskip.2cm

(2)
It is not true that, 
if a group $G$ is boundedly presented over $S_1$ and $S_2\supset S_1$, 
then $G$ is necessarily boundedly presented over $S_2$. 
\par

Indeed, consider the case $G = \Z$ of the infinite cyclic group.
Consider a strictly increasing sequence $(m_k)_{k \ge 1}$ of positive integers such that
\begin{equation*}
m_1 \, = \, 1 \hskip.5cm \text{and} \hskip.5cm
\lim_{k \to \infty} \frac{m_k}{m_{k-1}} \, = \, \infty .
\end{equation*}
On the one hand, we have a presentation 
$\Z = \left\langle u \mid \emptyset \right\rangle$,
with one generator $u$ representing $1 \in \Z$ and no relators.
On the other hand, consider a presentation
$\left\langle S \mid R' \right\rangle$
with an infinite generating set $S = \{s_k\}_{k \ge 1}$
such that $s_k \in S$ represents $m_k \in \Z$ for each $k \ge 1$.
\par

We claim that the presentation $\left\langle S \mid R' \right\rangle$
is not bounded, 
in particular that $\Z$ is \emph{not} boundedly presented over $S$.
\end{exe}

\begin{proof}[Proof of the claim]
Suppose \emph{ab absurdo} that there exists an integer $n \ge 4$
such that the $S$-length of any $r' \in R'$ is at most $n$.
We will arrive at a contradiction.
\par

Each relator $r' \in R'$ determines a relator 
$r'_{\operatorname{ab}} = s_1^{\nu_1} s_2^{\nu_2} s_3^{\nu_3} \cdots$
such that $r'$ and $r'_{\operatorname{ab}}$ define the same element
in the free \emph{abelian} group on the set of generators $S$.
In other terms,  $r'_{\operatorname{ab}} = \prod_{k \ge 1} s_k^{\nu_k}$,
with $\nu_k \in \Z$ for each $k \ge 1$ and $\nu_k = 0$
for all but at most $n$ values of $k$.
We denote by $R'_{\operatorname{ab}}$ 
the set of these $r'_{\operatorname{ab}}$, for $r' \in R'$.
Set $R'' = \{s_k s_\ell s_k^{-1} s_\ell^{-1} \mid k,\ell \ge 1 \}$, 
and $R = R'_{\operatorname{ab}} \sqcup R''$.
We have a new bounded presentation
$G = \left\langle S \mid R \right\rangle$
where any $r \in R$ is a word of length $\le n$.
\par

Let $k_0$ be such that $\frac{m_k}{m_{k-1}} > n$ for all $k \ge k_0$.
There exists in $R'_{\operatorname{ab}}$ a relator  $r'_{\operatorname{ab}}$
containing a letter $s_j$ with $j \ge k_0$; 
let $s_\ell$ denote the letter of maximal index that occurs in $r'_{\operatorname{ab}}$;
we can write
\begin{equation*}
r'_{\operatorname{ab}} \, = \,  
\Big( \prod_{k=1}^{\ell - 1} s_k^{\nu_k} \Big) s_\ell^{\nu_\ell} ,
\hskip.5cm \text{with} \hskip.2cm \nu_\ell \ne 0
\hskip.2cm \text{and} \hskip.2cm \sum_{k=1}^\ell \vert \nu_k \vert \, \le \, n .
\end{equation*}
On the one hand, we have 
$\sum_{k=1}^{\ell - 1} \nu_k m_k = -\nu_\ell m_\ell$,
hence 
\begin{equation*}
\Big\vert \sum_{k=1}^{\ell - 1} \nu_k \frac{m_k}{m_\ell} \Big\vert \, = \, 
\vert \nu_\ell \vert \, \ge \, 1 .
\end{equation*}
On the other hand,
since $\frac{m_k}{m_\ell} < \frac{1}{n}$ for $k = 1, \hdots, \ell-1$,
we have also
\begin{equation*}
\Big\vert \sum_{k=1}^{\ell - 1} \nu_k \frac{m_k}{m_\ell} \Big\vert
\, < \, 
\sum_{k=1}^{\ell - 1} \vert \nu_k \vert \hskip.1cm \frac{1}{n}
\, \le \, 1 ,
\end{equation*}
and this is the announced contradiction. 
\end{proof}

Bounded presentations of groups and their quotients are related as follows:

\begin{lem}
\label{lem:bp_to_quot}
Let 
$
N \lhook\joinrel\relbar\joinrel\rightarrow G \overset{\pi}{\relbar\joinrel\twoheadrightarrow} Q
$ 
be a short exact sequence of groups.
\begin{enumerate}[noitemsep,label=(\arabic*)]
\item\label{1DElem:bp_to_quot}
Assume that $G$ is boundedly presented over a set $S$
and that $N$ is generated, as a normal subgroup,
by $N \cap {\widehat S}^n$ for some $n$. 
Then $Q$ is boundedly presented over $\pi(S)$.
\item\label{2DElem:bp_to_quot}
Let $\rho :  F_S \twoheadrightarrow G$ be a generation of $G$
such that the kernel of $\pi \rho : F_S \twoheadrightarrow Q$
is generated as a normal subgroup by a set $R$ of relators
of length at most $k$, for some positive integer $k$
(in particular, we have a bounded presentation
$Q = \langle \pi\rho(S) \mid R \rangle$).
Then $N$ is generated as a normal subgroup of $G$
by $\rho(\widehat S^k) \cap N$.
\end{enumerate}
\end{lem}

\begin{proof} 
\ref{1DElem:bp_to_quot}
For each $r \in R$, let $\underline r$ denote the word in the letters of $\pi (\widehat S)$
obtained by replacing each letter $s \in \widehat S$ of $r$
by the corresponding letter $\pi(s) \in \pi (\widehat S)$;
let $\underline R_1$ denote the set of these $\underline r$.
For each $g \in N \cap {\widehat S}^n$, choose $s_1, \hdots, s_n \in \widehat S$
such that $g = s_1 \cdots s_n$;
let $\underline R_2$ denote the set of words of the form
$\pi(s_1) \cdots \pi(s_n)$.
Then $\langle \pi (\widehat S) \mid 
\underline R_1 \cup \underline R_2 \rangle$
is a bounded presentation of $G/N$.
\par

\ref{2DElem:bp_to_quot} 
Let $M$ be the normal subgroup of $G$ generated by $\rho(\widehat S^k) \cap N$.
It is clear that $M \subset N$, and we have to show that $M = N$.
Upon replacing $G$ by $G/M$,
we can assume that $M = \{1\}$, and we have to show that $N = \{1\}$.
\par

Clearly $\ker (\rho) \subset \ker (\pi \rho)$.
Let $r \in R$, viewed as a word in the letters of $S \cup S^{-1}$.
We have $\rho(r) \in \widehat S^k \cap N$, and therefore $r \in \ker(\rho)$.
Since $R$ generates $\ker (\pi \rho)$ as a normal subgroup of $F_S$,
it follows that $\ker(\pi \rho) \subset \ker(\rho)$.
Hence $\ker (\rho) = \ker(\pi \rho)$, and this shows that $N = \{1\}$.
\end{proof}

The case of groups given as extensions is more delicate to formulate:

\begin{lem}
\label{lem:bp_from_quot}
Let 
$
N \lhook\joinrel\relbar\joinrel\rightarrow G 
\overset{\pi}{\relbar\joinrel\twoheadrightarrow} Q
$
be a short exact sequence of groups.
Assume that $N$ and $Q$ are boundedly presented,
more precisely that:
\begin{enumerate}[noitemsep]
\item[($B_N$)]
the group $N$ has a bounded presentation $\langle S_N \mid R_N \rangle$
with $1 \in S_N = S_N^{-1}$ 
and $m_N := \sup \{ \vert r \vert_{S_N} \mid r \in R_N \} < \infty$;
\item[($B_Q$)]
the group $Q$ has a bounded presentation $\langle S_Q \mid R_Q \rangle$
with $1 \in S_Q = S_Q^{-1}$ 
and $m_Q := \sup \{ \vert r \vert_{S_Q} \mid r \in R_Q \} < \infty$.
\end{enumerate}
Let $S'_G$ be a subset of $G$
such that $1 \in S'_G = (S'_G)^{-1}$ and  $\pi(S'_G) = S_Q$.
Let $\sigma : S_Q \longrightarrow S'_G$ be a mapping
such that $\pi \sigma (s) = s$ for all $s \in S_Q$.
Assume furthermore that
\begin{enumerate}[noitemsep]
\item[($C_N$)]
there exists an integer $k \ge 1$ 
such that $(S'_G S_N S'_G) \cap N \subset (S_N)^k$;
\item[($C_Q$)]
there exists an integer $\ell \ge 1$ 
such that $(S'_G)^{m_Q} \cap N \subset (S_N)^\ell$.
\end{enumerate}
\par

Then $G$ is boundedly presented. More precisely, there exists
a set $R_G$ of words of bounded length in the letters of $S_N \cup S'_G$
such that 
\begin{equation*}
G \, = \,  \langle S_N \cup S'_G \mid R_G \rangle
\end{equation*}
is a bounded presentation of $G$.
\end{lem}

\noindent \emph{Observation.}
Note that we cannot assume that the restriction of $\pi$
is a bijection $S'_G \longrightarrow S_Q$:
if $s \in S_Q$ is such that $s \ne 1 = s^2$,
there need not exist $g \in G$ 
with $\pi(g) = s$ and $g^2 = 1$.

\begin{proof}
For all pairs $(s_1,s_2) \in S_N \times S'_G$, choose using ($C_N$)
a word $v_{s_1,s_2}$ of length $k$ in the letters of $S_N$
such that $s_2 s_1 s_2^{-1} = v_{s_1, s_2} \in N$.
We denote by 
\begin{equation*}
R_{\text{conj}} \hskip.5cm \text{the set of relators of the form} \hskip.2cm
s_2 s_1 s_2^{-1} v_{s_1, s_2}^{-1} \hskip.2cm
\text{obtained in this way.}
\end{equation*}
Note that their lengths are bounded by $k + 3$.
\par

For all pairs $(s_1, s_2) \in S'_G \times S'_G$ with $\pi(s_1) = \pi(s_2)$, 
choose similarly a word $w_{s_1, s_2}$ of length $k$ in the letters of $S_N$
such that $s_1 = s_2 w_{s_1, s_2} \in G$.
We denote by 
\begin{equation*}
R_{\pi} \hskip.5cm \text{the set of relators of the form} \hskip.2cm
s_1^{-1}s_2w_{s_1,s_2} \hskip.2cm
\text{obtained in this way.}
\end{equation*}
Note that their lengths are bounded by $k + 2$.
\par

For all relators $r \in R_Q$, let $s_{r,1} \cdots s_{r, m_Q}$ 
be a word in the letters of $S_Q$ representing it.
By ($C_Q$), we can choose a word $u_r$ of length $\ell$ in the letters of $S_N$
such that
\begin{equation*}
\sigma(r) \, := \, \sigma(s_{r,1}) \cdots \sigma(s_{r, m_Q}) u_r
\end{equation*}
is a word in the letters of $S_N \cup S'_G$ representing $1 \in G$.
We denote by
\begin{equation*}
R_{\sigma} \hskip.5cm \text{the set of these relators.}
\end{equation*}
Note that their lengths are bounded by $m_Q + \ell$.
\par

Denote by $T$ the set of relators $ss'$, for pairs $(s,s')$ in $S'_G$
with $ss' = 1$.
Define a group $\widetilde G$ by the presentation
\begin{equation*}
\widetilde G \, = \, \langle S_N \cup S'_G \mid 
         R_N \cup R_{\text{conj}} \cup R_\pi \cup R_\sigma \cup T \rangle .
\end{equation*}
It follows from the definitions that the tautological map 
from $S_N \cup S'_G$ viewed as a generating set of $\widetilde G$
to $S_N \cup S'_G$ viewed as a generating set of $G$
extends to a surjective homomorphism $p : \widetilde G \twoheadrightarrow G$.
\par

Denote by $\widetilde N$ the subgroup of $\widetilde G$ generated by $S_N$.
Clearly $p(\widetilde N) = N$, and the restriction 
$p_N : \widetilde N \longrightarrow N$ of $p$ to $\widetilde N$
is an isomorphism.
Moreover $\widetilde N$ is normal in $\widetilde G$,
because $s t s^{-1} \in \widetilde N$ for all $s \in S_N \cup S'_G$ and $t \in S_N$,
and $\widetilde G$ is generated by $S_N \cup S'_G$ as a semi-group.

Thus $p$ factors as a surjective homomorphism
$p_Q : \widetilde G / \widetilde N \twoheadrightarrow Q$.
As we have the presentation 
$\widetilde G / \widetilde N = \langle S_Q \mid R_Q \rangle$,
the quotient morphism $p_Q$ is indeed an isomorphism.
Since the diagram
\begin{equation*}
\begin{array}{ccccccc}
\widetilde N & \lhook\joinrel\relbar\joinrel\rightarrow & \widetilde G &
\relbar\joinrel\twoheadrightarrow  & \widetilde G / \widetilde N 
\\
&&&&
\\
 \downarrow \simeq && \downarrow p && \downarrow \simeq 
\\
&&&&
\\
N & \lhook\joinrel\relbar\joinrel\rightarrow & G &
\overset{\pi}{\relbar\joinrel\twoheadrightarrow}  & Q 
\end{array}
\end{equation*}
commutes, it follows that $p$ is an isomorphism.
This proves the lemma, for $S_G = S_N \cup S'_G$
and $R_G = R_N \cup R_{\text{conj}} \cup R_\pi \cup R_\sigma$.
\end{proof}

\begin{prop}
\label{boundedpresandSEC}
Let 
$
N 
\overset{j}{\lhook\joinrel\relbar\joinrel\rightarrow} G 
\overset{\pi}{\relbar\joinrel\twoheadrightarrow} Q 
$
be a short exact sequence of LC-groups,
where the topology of $N$ coincides with the topology induced by $j$,
and the homomorphism $\pi$ is continuous and open.
Assume that 
\begin{enumerate}[noitemsep]
\item[(CP$_N$)]
the group $N$ has a presentation $\langle S_N \mid R_N \rangle$
with $S_N \subset N$ compact 
and $\sup \{ \vert r \vert_{S_N} \mid r \in R_N \} < \infty$;
\item[(CP$_Q$)]
the group $Q$ has a presentation $\langle S_Q \mid R_Q \rangle$
with $S_Q \subset Q$ compact 
and $\sup \{ \vert r \vert_{S_Q} \mid r \in R_Q \} < \infty$.
\end{enumerate}
\par

Then
\begin{enumerate}[noitemsep]
\item[(CP$_G$)]
the group $G$ has a presentation $\langle S_G \mid R_G \rangle$
with $S_G \subset G$ compact 
and $\sup \{ \vert r \vert_{S_G} \mid r \in R_G \} < \infty$.
\end{enumerate}
\end{prop}

\begin{proof}
We keep the notation of the previous lemma and of its proof.
We can choose $S'_G$ compact, by Lemma \ref{KimagedeK}.
Then $S_G := S_N \cup S'_G$ is a compact generating set of $G$,
and the lengths of the relators in
$R_G := R_N \cup R_{\text{conj}} \cup R_\pi \cup R_\sigma$
are bounded by $\max \{ m_N, k+3, m_Q + \ell \}$.
\end{proof}

\section{The Rips complex of a bounded presentation}
\label{sectionRipscomplex}
For a group $G$ with a pseudo-metric $d$, 
we have a Rips $2$-complex $\operatorname{Rips}^2_c(G,d)$
as in Definition \ref{defRipst}.
When $d$ is the word metric defined by a generating set $S$, 
we also write $\operatorname{Rips}^2_c(G,S)$ for $\operatorname{Rips}^2_c(G,d_S)$;
\index{Rips $2$-complex! $\operatorname{Rips}^2_c(G,S)$}
this space can often be used instead of
\emph{Cayley graphs} or \emph{Cayley $2$-complexes}
\index{Cayley graph}
defined in other references (such as \cite{Cann--02}).

\begin{prop}
\label{bpresrips}
Let $G$ be a group endowed with a generating subset $S$. The following conditions are equivalent:
\begin{enumerate}[noitemsep,label=(\roman*)]
\item\label{iDEbpresrips}
the group $G$ is boundedly presented over $S$;
\item\label{iiDEbpresrips}
$\operatorname{Rips}^2_c(G,S)$ is simply connected for some $c$;
\item\label{iiiDEbpresrips}
$\operatorname{Rips}^2_c(G,S)$ is simply connected for all $c$ large enough;
\item\label{ivDEbpresrips}
the metric space $(G, d_S)$ is coarsely simply connected.
\end{enumerate}
\end{prop}

\begin{proof}
The equivalence of Conditions 
\ref{iiDEbpresrips}, \ref{iiiDEbpresrips}, and \ref{ivDEbpresrips} 
is a particular case of Proposition \ref{propositionpourc-geodesicspace}.
\par

\ref{iDEbpresrips} $\Rightarrow$ \ref{iiDEbpresrips}
Let $G = \langle S \mid R \rangle$ be a bounded presentation;
set $m = \max_{r \in R} \ell_S(r)$.
We claim that, for $n \ge \max \{m/2, 1\}$, the complex $\operatorname{Rips}^2_n(G,S)$
is simply connected.
\par

Consider a loop $\xi$ based at $1$ in the topological realization of $\operatorname{Rips}^2_n(G,S)$.
We have to show that $\xi$ is homotopic to the constant loop $(1)$.
By Lemma \ref{lem:path_X1}(1), we can assume that $\xi$ is the topological realization
of a combinatorial loop
\begin{equation*}
\eta \, = \, (1, \hskip.1cm s_1, \hskip.1cm s_1s_2, \hskip.1cm 
\hdots, \hskip.1cm s_1 \cdots s_{k-1}, \hskip.1cm s_1 \cdots s_{k-1}s_k = 1)
\end{equation*}
with $s_1, \hdots, s_k \in \widehat S$.
There are relators $r_1, \hdots, r_\ell \in R \cup R^{-1}$ and words $w_1, \hdots, w_\ell \in F_S$
such that 
\begin{equation*}
s_1 s_2 \cdots s_k \, = \, 
\prod_{j=1}^\ell w_j r_j w_j^{-1} .
\end{equation*}
Let $j \in \{1, \hdots, \ell\}$;
since $\ell_S(r_j) \le m$, any triple of vertices of $r_j$ is in a common $2$-simplex;
hence the prefixes of the word $w_j r_j w_j^{-1}$ constitute
a combinatorial loop that is combinatorially homotopic to the constant loop.
It follows that $\eta$ is combinatorially homotopic to the constant loop,
hence that $\xi$ is homotopic to the trivial loop. This establishes the claim.
\par

\ref{iiiDEbpresrips} $\Rightarrow$ \ref{iDEbpresrips}
Let $m \ge 1$ be an integer such that $\operatorname{Rips}^2_m(G,S)$ is simply connected.
Let $\pi : F_S \twoheadrightarrow G$ be as in Definition \ref{generation}; 
set $N = \textnormal{Ker} (\pi)$.
Let $w \in N$; write $w = s_1 \cdots s_k$, with $s_1, \hdots, s_k \in \widehat S$.
Consider 
\begin{equation*}
\eta \, = \, (1, \hskip.1cm s_1, \hskip.1cm s_1s_2, \hskip.1cm 
\hdots, \hskip.1cm s_1 \cdots s_{k-1}, \hskip.1cm s_1 \cdots s_{k-1}s_k = 1) ,
\end{equation*}
that is a combinatorial loop based at $1$ in $\operatorname{Rips}^2_m(G,S)$.
Then $\eta$ is combinatorially homotopic to some combinatorial loop of the form
$\prod_{j=1}^N u_j r_j u_j^{-1}$, the notation being as in Lemma \ref{lem:graphomotopic_simplices},
with each $r_j$ being a combinatorial loop of length $3$ in $\operatorname{Rips}^2_m(G,S)$.
Changing our viewpoint, we consider $r_j$ as a word in the letters of $S \cup S^{-1}$ 
of length at most $3m$.
If $R$ denotes the set of these words $r_j$,
we have a bounded presentation $G = \langle S \mid R \rangle$.
\end{proof}

\chapter{Compactly presented groups}
\label{chap_cpgroups} 

\section{Definition and first examples}
\label{DefExForCpGroups}

\begin{defn}
Let $G$ be an LC-group.
\par

A \textbf{compact presentation} \index{Compact presentation|textbf} of $G$
is a presentation $\langle S \mid R \rangle$ of $G$ with $S$ compact in $G$
and $R$ a bounded relating subset, 
in other words a bounded presentation 
$\langle S \mid R \rangle$ of $G$ with $S$ compact.
\par

An LC-group is \textbf{compactly presented} 
\index{Compactly presented! LC-group|textbf}
if it admits a compact presentation.
\end{defn}

The earliest article we know in which this notion appears 
is that of Martin Kneser \cite{Knes--64}.

\begin{prop}
Let $G$ be a compactly generated LC-group;
let $S_1, S_2$ be two compact generating subsets of $G$.
Then $G$ has a compact presentation with generating set $S_1$
if and only if it has one with $S_2$.
\end{prop}

\begin{proof}
This is a consequence of
Lemma \ref{boundedpresentationforS1andS2}(4).   
\end{proof}

\begin{prop}
\label{cp iff csc}
Let $G$ be a compactly generated LC-group.
Let $S$ be a compact generating set and $d_S$ the corresponding word metric.
\begin{enumerate}[noitemsep,label=(\arabic*)] 
\item\label{1DEcp iff csc}
The group $G$ is compactly presented 
if and only if the metric space $(G,d_S)$ is coarsely simply connected.
\end{enumerate}
More generally, let $G$ be a $\sigma$-compact LC-group.
Let $d$ be an adapted pseudo-metric on $G$ 
(Proposition \ref{existenceam}).
\begin{enumerate}[noitemsep,label=(\arabic*)] 
\addtocounter{enumi}{1}
\item\label{2DEcp iff csc}
The group $G$ is compactly presented if and only if the pseudo-metric space $(G,d)$
is coarsely simply connected.
\end{enumerate}
\end{prop}

\begin{proof}
Claim \ref{1DEcp iff csc} holds true by Proposition \ref{bpresrips}.
In \cite[Example 1.C$_1$]{Grom--93}, Gromov has
a short comment concerning the claim:
``this is obvious''.
\par

For Claim \ref{2DEcp iff csc}, suppose first that $G$ is compactly presented,
say with a compact presentation $\langle S \mid R \rangle$.
Then $(G, d_S)$ is coarsely simply connected by \ref{1DEcp iff csc},
hence $(G,d)$ is coarsely simply connected by Corollary \ref{2metricsce}\ref{2DE2metricsce}
and Proposition \ref{coarse1conninvbycoarseeq}.
Suppose now that $(G,d)$ is coarsely simply connected.
Since $(G,d)$ is coarsely connected, 
$G$ is compactly generated 
by Proposition \ref{geodad+wordmetric}\ref{1DEgeodad+wordmetric},
say with compact generating set $S$.
Then $(G,d_S)$ is coarsely simply connected by Proposition \ref{coarse1conninvbycoarseeq},
hence $(G,d)$ is compactly presented by \ref{1DEcp iff csc}.
\end{proof}

Propositions \ref{coarse1conninvbycoarseeq} and \ref{cp iff csc} imply:

\begin{cor}
\label{qinv}
Among $\sigma$-compact LC-groups, 
being compactly presented is invariant by metric coarse equivalence.
\index{Property! of a $\sigma$-compact LC-group invariant by metric coarse equivalence}
\par

In particular, among compactly generated LC-groups, 
being compactly presented is invariant by quasi-isometry.
\end{cor}

\begin{cor}
\label{hereditaritycp}
\index{Hereditarity of various properties}
Let $G$ be an LC-group.
\begin{enumerate}[noitemsep,label=(\arabic*)] 
\item\label{1DEhereditaritycp}
Let $H$ be a cocompact closed subgroup of $G$. 
Then $G$ is compactly presented if and only if $H$ is so. 
In particular:
\index{Subgroup! cocompact closed}
\index{Compact group! finite presentability}
\begin{enumerate}[noitemsep,label=(\alph*)]
\item\label{aDEhereditaritycp}
compact groups are compactly presented;
\item\label{bDEhereditaritycp}
a uniform lattice in an LC-group $G$ is finitely presented
if and only if $G$ is compactly presented.
\index{Lattice! in an LC-group}
\end{enumerate}
\item\label{2DEhereditaritycp}
Let $K$ be a compact normal subgroup of $G$. 
Then $G$ is compactly presented if and only if $G/K$ is so.
\index{Subgroup! compact normal} 
\end{enumerate}
\end{cor}

\begin{proof}
By the previous corollary, 
\ref{1DEhereditaritycp} follows from Proposition \ref{sigmac+compactgofcocompact}
and \ref{2DEhereditaritycp} from Proposition \ref{sigmac+compactgofquotients}.
\par

There is another proof of \ref{aDEhereditaritycp} in Example \ref{examplesboundpres}(1).
\end{proof}

Let $G$ be a $\sigma$-compact LC-group with an adapted metric $d$,
as in Proposition \ref{cp iff csc}(2).
Recall that $d$ is coarsely geodesic if and only if $G$ is compactly generated.
Moreover, if the latter condition holds, 
$d$ can be chosen to be a geodesically adapted metric.
In particular, if $S$ is a compact generating set, the word metric $d_S$
is a geodesically adapted metric on $G$
(Proposition  \ref{geodad+wordmetric}).

\begin{prop}
\label{ripsqi}
Let $G$ be a compactly generated LC-group,
$S$ a compact generating set, 
$d$ an adapted pseudo-metric on $G$, and $c$ a constant, $c \ge 1$.
\begin{enumerate}[noitemsep,label=(\arabic*)]
\item\label{1DEripsqi}
The inclusion of $(G,d)$ into $\operatorname{Rips}^2_c(G,S)$
is a metric coarse equivalence.
\item\label{2DEripsqi}
When $d$ is geodesically adapted,
the inclusion of $(G,d)$ into $\operatorname{Rips}^2_c(G,S)$
is a quasi-isometry.
\end{enumerate}
\end{prop}

\begin{proof} 
\ref{1DEripsqi}
The identity map $(G,d) \longrightarrow (G,d_S)$ 
is a metric coarse equivalence by Corollary \ref{2metricsce}, and the inclusion
$(G,d_S)\lhook\joinrel\relbar\joinrel\rightarrow  \operatorname{Rips}^2_c(G,S)$
is a quasi-isometry by Proposition \ref{inclusionXinRips}. 
\par

\ref{2DEripsqi}
If $d$ is geodesically adapted, the identity map $(G,d) \longrightarrow (G,d_S)$
is a quasi-isometry by Corollary \ref{uniqueness _uptoqi}.
\end{proof} 

In view of Propositions \ref{cp iff csc}, \ref{ripsqi}, and \ref{bpresrips},
we have the following third important step in our exposition
(the first two being \ref{miles_sigmacompact} and \ref{miles_compactgen}).

\begin{miles} \index{Compactly presented! LC-group, milestone}
\label{miles_compactpres}
\index{Milestone}
For a compactly generated LC-group $G$, with a geodesically adapted pseudo-metric $d$,
and a compact generating set $S$, 
the following properties are equivalent:
\begin{enumerate}[noitemsep,label=(\roman*)] 
\item\label{iDEmiles_compactpres}
the LC-group $G$ is compactly presented;
\item\label{iiDEmiles_compactpres}
the pseudo-metric space $(G,d)$ is coarsely simply connected;
\item\label{iiiDEmiles_compactpres}
the inclusion of $(G,d)$ into $\operatorname{Rips}^2_c(G,S)$
is a metric coarse equivalence $\forall$ $c \ge 1$;
\item\label{ivDEmiles_compactpres}
the inclusion of $(G,d)$ into $\operatorname{Rips}^2_c(G,S)$
is a quasi-isometry  $\forall$ $c \ge 1$;
\item\label{vDEmiles_compactpres}
$\operatorname{Rips}^2_c(G,S)$ is simply connected for all $c$ large enough.
\end{enumerate}
\end{miles}
\index{Coarsely! simply connected pseudo-metric space}

Since coarse simple connectedness is invariant by metric coarse equivalence
(Pro\-position \ref{coarse1conninvbycoarseeq}),
we have here for compact presentation the analogues of what are 
Theorem \ref{ftggt}  and Corollary \ref{ahahah} for compact generation:

\begin{thm}
\label{analogueof4.D.4}
Let $G$ be an LC-group and $X$ a non-empty coarsely proper pseudo-metric space;
suppose that there exists a geometric action of $G$ on $X$.
\par

Then $G$ is compactly presented if and only if 
$X$ is coarsely simply connected.
\end{thm}

\begin{cor}
\label{analoguede4.D.5}
For an LC-group $G$, the following conditions are equivalent:
\begin{enumerate}[noitemsep,label=(\roman*)]
\item\label{iDEanaloguede4.D.5}
$G$ is compactly presented;
\item\label{iiDEanaloguede4.D.5}
there exists a geometric action of $G$ on
a non-empty coarsely simply connected pseudo-metric space;
\item\label{iiiDEanaloguede4.D.5}
there exists a geometric action of $G$ on
a non-empty geodesic simply connected metric space;
\item\label{ivDEanaloguede4.D.5}
there exists a geometric faithful action of $G$ on
a non-empty geodesic simply connected metric space.
\end{enumerate}
\end{cor}

The literature contains various instances of this statement
in the particular case of discrete groups.
As an example, here is Corollary I.8.11 on Page 137 of \cite{BrHa--99}:
\begin{center}
\emph{A group is finitely presented if and only if it acts properly and cocompactly
\\
by isometries on a simply connected geodesic space.}
\end{center}
For proper cocompact isometric actions, see Remark \ref{firstexamplesgeometricactions}(6).

\vskip.2cm

The following proposition is Theorem 2.1 in \cite{Abel--72}.

\begin{prop}
\label{cpforgroupsandquotients}
Let 
$
N \lhook\joinrel\relbar\joinrel\rightarrow G \relbar\joinrel\twoheadrightarrow Q
$ 
be a short exact sequence of LC-groups and continuous homomorphisms.
\begin{enumerate}[noitemsep,label=(\arabic*)]
\item\label{1DEcpforgroupsandquotients}
Assume that $G$ is compactly presented
and that $N$ is compactly generated as a normal subgroup of $G$.
Then $Q$ is compactly presented.
\item\label{2DEcpforgroupsandquotients}
Assume that $G$ is compactly generated
and that $Q$ is compactly presented.
Then $N$ is compactly generated as a normal subgroup of $G$.
\item\label{3DEcpforgroupsandquotients}
If $N$ and $Q$ are compactly presented, so is $G$.
\end{enumerate}
\end{prop}

\begin{proof}
This follows from Lemma  \ref{lem:bp_to_quot}
and Proposition \ref{boundedpresandSEC}.
\end{proof}

\begin{defn}
\label{defgroupretract}
In an LC-group $G$, a \textbf{group retract} is a closed subgroup $H$
such that there exists a continuous homomorphism $p : G \twoheadrightarrow H$
of which the restriction to $H$ is the identity.
\index{Retract! group retract|textbf} \index{Group retract|textbf}
\end{defn}

\begin{prop}[retracts]
\label{retractcp}
Let $G$ be a $\sigma$-compact LC-group and $H$ a closed subgroup.
\par

If $G$ is compactly presented and $H$ a group retract, 
then $H$ is compactly presented.
\end{prop}

\begin{proof}
It follows from Propositions \ref{morphisms_grps_spaces}, 
\ref{coarseretractprop} and  \ref{cp iff csc}.
\end{proof}

\begin{exe}
\label{coarseretractex}
Let $H,N$ be two $\sigma$-compact LC-groups, 
$\alpha$ a continuous action of $H$ on $N$ by topological automorphisms,
and $G = N \rtimes_\alpha H$ the corresponding semidirect product.
\index{Semidirect product}
Consider $G$ as a metric space, for some adapted metric $d$;
observe that the restriction of $d$ to $H$ (identified with $\{1\} \times H$)
is an adapted metric.
\par

Then the canonical projection is a coarse retraction from $G$ to $H$,
in the sense of Definition \ref{coarseretractdef}.
Consequently, if $G$ is compactly presented, so is $H$.
\end{exe}

\begin{lem}
\label{lemnaneibh1}
Let $G$ be a group and $\mathcal V$ a set of subsets of $G$ containing $1$.
Then $\mathcal V$ is the set of neighbourhoods of $1$ 
of a group topology $\mathcal T$ on $G$
if and only if the following five conditions are satisfied:
\par

(F)
if $A \in \mathcal V$ and $A' \subset G$ with $A' \supset A$,
then $A' \in \mathcal V$;
\par

($\cap$)
for all $A', A'' \in \mathcal V$, there exists $A \in \mathcal V$ such that
$A \subset A' \cap A''$;
\par

(M$_1$)
for all $A \in \mathcal V$, there exist $A', A'' \in \mathcal V$ such that $A'A'' \subset A$;
\par

(I$_1$)
if $A \in \mathcal V$, then $A^{-1} \in \mathcal V$;
\par

(C)
$\mathcal V$ is stable by conjugation: $g_0^{-1} A g_0 \in \mathcal V$ 
for all $A \in \mathcal V$ and $g_0 \in G$.
\par\noindent
If these conditions are satisfied, such a $\mathcal T$ is unique.
\par

Moreover, if these conditions are satisfied, 
the topology $\mathcal T$ is Hausdorff 
if and only if $\bigcap_{A \in \mathcal V} A = \{1\}$.
\end{lem}

\emph{Note.}
For a set of subsets of $G$ containing $1$, Conditions (F) and ($\cap$) 
define $\mathcal V$ as a \emph{filter} of subsets of $G$.
The subscripts in (M$_1$) and (I$_1$) indicate that these conditions
express continuity of the multiplication at $(1,1)$ and of inversion at $1$.
Compare with Lemma \ref{LemTopsub}.
\par

The lemma is well-known, see e.g.\ Proposition 1 
in \cite[Page III.3]{BTG1-4}.

\begin{proof}
Both the necessity of the conditions and the uniqueness of $\mathcal T$ are clear.
Let us prove that the conditions are sufficient.
\par

Given $\mathcal V$ satisfying (F),  ($\cap$), (M$_1$), (I$_1$), and (C),
define $\mathcal T$ to be the set of subsets $U$ of $G$ such that,
for all $g \in U$, we have $g^{-1}U \in \mathcal V$.
Note that, by (C), the latter condition is equivalent to $Ug^{-1} \in \mathcal V$.
\par

Let us check that $\mathcal T$ is a topology on $G$.
It is obvious that $\emptyset, G \in \mathcal T$.
Let $U_1, U_2 \in \mathcal T$ be such that $U_1 \cap U_2 \ne \emptyset$,
and let $g \in U_1 \cap U_2$;
by ($\cap$), there exists $V \in \mathcal V$ such that $V \subset g^{-1}(U_1 \cap U_2)$;
hence $g^{-1}(U_1 \cap U_2) \in \mathcal V$ by (F);
it follows that $\mathcal T$ is stable by finite intersections.
Similarly, $\mathcal T$ is stable by unions.
\par

Ler us check that $\mathcal T$ is a group topology.
Consider two elements $g,h \in G$ and two nets $(g_i)_i, (h_j)_j$ in $G$
such that $\lim_i g_i = g$ and $\lim_j h_j = h$.
By definition of $\mathcal T$, the conditions are equivalent to
$\lim_i g^{-1}g_i  = 1$ and $\lim_j h^{-1} h_j = 1$. Then 
\begin{equation*}
\aligned
\lim_{i,j} h^{-1}g^{-1} g_i h_j 
\, &= \, 
\lim_{i,j} \left( h^{-1} \hskip.1cm \big( g^{-1}g_i \big) \hskip.1cm h \hskip.1cm  \big(h^{-1} h_j \big) \right) 
\\
&
\, = \, 
h^{-1}  \hskip.1cm \big( \lim_i g^{-1}g_i \big)  \hskip.1cm h  \hskip.1cm \big( \lim_j h^{-1}h_j \big)
\hskip.5cm \text{by (C) and (M$_1$)}
\\
&
\, = \, 1
\hskip.5cm \text{by hypothesis on $(g_i)_i$ and $(h_j)_j$} ,
\endaligned
\end{equation*}
so that $\lim_{i,j} g_i h_j = gh$;
hence, for $\mathcal T$,  the multiplication in $G$ is continuous.
Moreover, we have also $\lim_i g_i g^{-1} = 1$, by (C),
and $\lim_i gg_i^{-1} = 1$, by (I$_1$),
so that $\lim_i g_i^{-1} = g^{-1}$;
hence the inversion is also continuous.
\end{proof}

\begin{prop}
\label{cpcovering}
Let $G$ be a topological group.
\index{Topological group}
Let $A$ be an open symmetric neighbourhood of $1$ in $G$,
and $S$ a symmetric subset of $G$ such that $A \subset S \subset \overline A$.
Assume that $S$ generates $G$.
\par

Let $\widetilde S$ be a set given with a bijection
$\widetilde S \ni \widetilde s \leftrightarrow s \in S$.
Define a group by a presentation $\widetilde G_S = \langle \widetilde S \mid R \rangle$
with set of generators $\widetilde S$ and set of relators
\begin{equation*}
R \,  := \,  \{  \widetilde s \hskip.1cm \widetilde t \hskip.1cm 
\widetilde u \hskip.1cm {\phantom{}}^{{}^{-1}} 
\mid s,t,u \in S , \hskip.1cm st=u \} .
\end{equation*}
Denote by $p : \widetilde G_S \longrightarrow G$ the homomorphism defined by
$p(\widetilde s) = s$ for all $s \in S$.
\begin{enumerate}[noitemsep,label=(\arabic*)]
\item\label{1DEcpcovering}
The group $\widetilde G_S$ has a unique group topology
such that $p$ is a covering map.
\end{enumerate}
Suppose from now on that $\widetilde G_S$ 
is given the topology of \ref{1DEcpcovering}.
\begin{enumerate}[noitemsep,label=(\arabic*)]
\addtocounter{enumi}{1}
\item\label{2DEcpcovering}
If $G$ is locally compact, then so is $\widetilde G_S$.
\item\label{3DEcpcovering}
If $G$ is locally compact and $A$ is connected, then $\widetilde G_S$
is locally compact and connected.
\end{enumerate}
\end{prop}

\begin{proof}
For every subset $B$ of $S$, 
set $\widetilde B = \{ \widetilde s \in \widetilde S \mid s \in B \}$.

\vskip.2cm

(Beginning of \ref{1DEcpcovering})
Recall that a group topology on a group is characterized
by nets converging to $1$. 
Let $B$ be a compact neighbourhood of $1$ contained in $A$;
it is easily seen that, for a group topology $\mathcal T$ on $\widetilde G_S$
such that $p$ is a covering map, 
a net $(g_i)_i$ in $\widetilde G_S$ converges to $g \in \widetilde G_S$ 
if and only if it satisfies the two following conditions:
\begin{equation}
\label{eqGhausdorff}
\text{eventually} \hskip.2cm g^{-1} g_i \in \widetilde B , 
\hskip.2cm \text{and} \hskip.2cm \lim_i p(g^{-1} g_i) = 1 \in G.
\end{equation}
This proves the uniqueness of $\mathcal T$.

\par

Let us now prove the existence of such a topology.
Define $\mathcal V$ to be the set of subsets $X$ in $\widetilde G_S$
such that $p(X \cap \widetilde S)$ is a neighbourhood of $1$ in $S$.
Conditions (F), ($\cap$), (M$_1$), and (I$_1$) 
of Lemma \ref{lemnaneibh1} are clearly satisfied.
To show that $\mathcal T$ is a group topology on $G$, 
it remains to show that Condition (C) is satisfied.

\vskip.2cm

(C)
Define $H$ to be the set of elements in $\widetilde G_S$ which normalize $\mathcal V$;
it is a subgroup of $\widetilde G$. We claim that the set $\widetilde A$ is inside the group $H$.
\par

Consider $X \in \mathcal V$,  $a \in A$,
and let us show that $\widetilde a X \widetilde a^{-1} \in \mathcal V$.
Upon replacing $X$ by $X \cap \widetilde S$,
we can assume without loss of generality that $X \subset \widetilde S$, 
i.e., that $X = \widetilde B$ for some neighbourhood $B$ of $1$ in $S$.
Since $A$ is a neighbourhood of both $1$ and $a$, 
there exists a neighbourhood $C$ of $1$ in $A$ 
such that $aC \subset A$ and $aCa^{-1} \subset A$.
Let $c \in C$; set $d = ac$ and $e = aca^{-1} = da^{-1}$; observe that $d,e \in A$.
Moreover, since $ac = d = ea \in A$, 
we have $\widetilde a \widetilde c = \widetilde d = \widetilde e \widetilde a$,
and therefore $\widetilde e = \widetilde a \widetilde c \widetilde a^{-1} \in \widetilde A$.
Hence  $\widetilde a \widetilde c \widetilde a^{-1} = \widetilde{ aca^{-1} }$
for all $c \in C$, and \emph{a fortiori} for all $c \in B \cap C$,
i.e., $p\big( \widetilde a (\widetilde B \cap \widetilde C) \widetilde a^{-1} \big) = 
a(B \cap C)a^{-1}$.
Since the latter is a neighbourhood of $1$ in $S$
contained in $p( \widetilde a \widetilde B \widetilde a^{-1} \cap \widetilde S)$,
it follows that $ \widetilde a \widetilde B \widetilde a^{-1} \in \mathcal V$.
Hence the claim is shown.
\par

Since $S \subset \overline{A}$, we have $S \subset AA$.
Thus every element of $S$ is of the form $s = ab$ with $a,b \in A$.
This implies $\widetilde s = \widetilde a \widetilde b$;
consequently $\widetilde A$ generates $\widetilde G_S$.
Hence, for $X \in \mathcal V$, the conclusion of the previous step, 
i.e., $h X h^{-1} \in \mathcal V$, holds not only for all $h \in \widetilde A$,
but indeed for all $h \in \widetilde G_S$. 
Consequently $H = \widetilde G_S$, i.e., Condition (C) of the previous lemma is satisfied.

\vskip.2cm      

(End of \ref{1DEcpcovering})
Recall our convention (see (A1)  Page \pageref{A1})
according to which the topological group $G$ is Hausdorff.
To check that the group topology $\mathcal T$ 
just defined on $\widetilde G_S$ is Hausdorff,
it suffices to check that, in a constant net converging to $1$,
every element is $1$.
This is true: let $h \in \widetilde G_S$, and $(g_i)$ a constant net, with $g_i = h$ for all $i$;
if the net is converging to $g=1$, the conditions of (\ref{eqGhausdorff}) above
imply that $1^{-1}h \in \widetilde S$ and $p(h) = 1$, so that $h=1$.
\par

It is immediate that the homomorphism $p$ is continuous,
and that $p$ induces a homeomorphism of $\widetilde A$ onto $A$.
It follows from Proposition \ref{GtoHcovering} that $p$ is a covering.

\vskip.2cm

\ref{2DEcpcovering}
If $G$ is locally compact, there exists a compact  neighbourhood $C$ of $1$ in $A$.
Then $\widetilde C$ is a compact  neighbourhood of $1$ in $\widetilde G_S$.

\vskip.2cm

\ref{3DEcpcovering}
Suppose that $G$ is locally compact and $A$ connected. 
As recalled in Remark \ref{remopensub},
it suffices to show that every neighbourhood of $1$ in $\widetilde G_S$,
say $W$, generates $\widetilde G_S$.
\par

Upon replacing $W$ by a smaller one, we can assume that $W$ is symmetric and contained in $A$.
Then $W = \widetilde V$ for an appropriate symmetric neighbourhood $V$ of $1$ in $G$.
Define an equivalence relation $\sim$ in $\widetilde S$ as follows:
for $\widetilde s, \widetilde{s'} \in \widetilde S$ (with $s,s' \in S$),
set $\widetilde s \sim \widetilde{s'}$ if there exist 
$t_0, t_1, \hdots , t_k \in S$ with $t_0 = s$, $t_k = s'$, and $t_i t_{i-1}^{-1} \in V$ for $i = 1, \hdots, k$.
Since $V$ is a neighbourhood of $1$,
the classes of this equivalence relation are open in $\widetilde S$.
Since $A$ is connected and $A \subset S \subset \overline A$, 
the set $S$ is also connected,
and it follows that this equivalence relation has exactly one class in $\widetilde S$;
in other terms, $\widetilde S$ is in the subgroup generated by $\widetilde V$,
and thus $\widetilde G_S$ itself is generated by $\widetilde V$.
\end{proof}

\begin{rem}
(a)
In the previous proposition,
the condition $S \subset \overline{A}$ cannot be deleted.
\par

Indeed, consider the example for which $G = \R$,
$A = \mathclose] -1,1 \mathopen[$, 
and $S = \mathopen[-1,1\mathclose] \cup \{-3,3\}$.
Define $R$ and $\widetilde G_S$ as in Proposition \ref{cpcovering}.
Then $\widetilde G_S$ is isomorphic to a free product $\widetilde \R \ast \widetilde \Z$,
where $\widetilde{\R}$ is a group isomorphic to $\R$ 
and generated by $\widetilde{\mathopen[-1,1\mathclose]}$,
and $\widetilde \Z$ an infinite cyclic group generated by $\widetilde 3$.
The surjective homomorphism $p : \widetilde G_S \longrightarrow \R$
applies $\widetilde{ \mathopen[-1,1\mathclose] }$ 
onto $\mathopen[-1,1\mathclose]$ and $\widetilde 3$ to $3$.
\par

The point is that, if $\mathcal V$ was defined as in the previous proof, 
then $\mathcal V$ would not be normalized by the generator $\widetilde 3$.  
Indeed, there does not exist any group topology on $\widetilde{\R} \ast \widetilde{\Z}$
for which the natural surjective homomorphism onto $\R$ would be a covering map.

\vskip.2cm

(b) Claim \ref{1DEcpcovering} of the previous proposition
is essentially Lemma 1.7 in  \cite{Abel--72}.
\end{rem}

\begin{cor}
\label{cpcoveringofcpG}
For every compactly generated LC-group $G$,
there exists a compactly presented LC-group $\widetilde G$
and a discrete normal subgroup $N$ of $\widetilde G$
such that the quotient $\widetilde G / N$ is topologically isomorphic to $G$.
\par

If moreover $G$ is connected, then $\widetilde G$ can be chosen to be connected as well.
\end{cor}

\begin{proof}
There exists in $G$ an open symmetric neighbourhood $U$ of $1$ 
with compact closure $\overline U$ which is generating.
With the notation of Proposition \ref{cpcovering},
the corollary holds for $\widetilde G = \widetilde G_{\overline U}$ and $N = \ker (p)$.
\end{proof}

\begin{cor}
\label{cpcoveringof1cLieG}
Let $G$ be a connected simply connected Lie group
and $U$ a relatively compact open symmetric neighbourhood of $1$ in $G$.
Set $S = \overline U$, 
let $\widetilde S$ be a set given with a bijection
$\widetilde S \ni \widetilde s \leftrightarrow s \in S$,
and define $R = \{  \widetilde s \hskip.1cm \widetilde t \hskip.1cm \widetilde u^{-1} 
\mid s,t,u \in S , \hskip.1cm st=u \}$.
\par

Then $\langle \widetilde S \mid R \rangle$
is a compact presentation of $G$.
\end{cor}

\begin{proof}
This corollary follows from the previous one
and from the fact that any connected covering of a simply connected Lie group
is an isomorphism.
\end{proof}

\begin{exe}[compact presentations of Lie groups]
(1)
The additive group $\R$ is generated by the compact interval $\mathopen[-1,1\mathclose]$.
Let $\widetilde S$ be a set given with a bijection
$\widetilde S \ni \widetilde s \leftrightarrow s \in \mathopen[-1,1\mathclose]$;
then
\begin{equation*}
\left\langle
\widetilde S
\hskip.2cm \big\vert \hskip.2cm 
\widetilde s \widetilde t = \widetilde u 
\hskip.2cm \text{for all} \hskip.2cm 
s,t,u \in \mathopen[-1,1\mathclose] 
\hskip.2cm \text{such that} \hskip.2cm  s+t  = u
\right\rangle
\end{equation*}
is a compact presentation of $\R$.

\vskip.2cm

(2)
Let 
${\SL}^{\text{univ}}_2(\R)$ denote the universal cover of $\SL_2(\R)$,
and $\pi$ the corresponding covering map.
Let $c$ be a positive constant with $0 < c < 1$.
Set
\begin{equation*}
S \, = \,
\left\{
\begin{pmatrix}
t & x \\ y & z 
\end{pmatrix} \in \SL_2(\R)
\hskip.2cm \Big\vert \hskip.2cm
\vert t - 1 \vert \le c, \hskip.2cm \vert x \vert \le c,
\hskip.2cm \vert y \vert \le c, \hskip.2cm \vert z - 1 \vert \le c
\right\} .
\end{equation*}
Let $\widetilde S$ be a set given with a bijection
$\widetilde S \ni \widetilde s \leftrightarrow s \in S$, and set
\begin{equation*}
R \, = \, 
\{ \widetilde g \widetilde h \widetilde k^{-1} \hskip.1cm \mid \hskip.1cm 
g,h,k \in S \hskip.2cm \text{and} \hskip.2cm k = gh \in S \} .
\end{equation*}
As an example for Corollary \ref{cpcoveringofcpG},
$\langle \widetilde S \mid R \rangle$ 
is a compact presentation of ${\SL}^{\text{univ}}_2(\R)$.
\par

Set moreover
\begin{equation*}
\aligned
u \, &= \, \begin{pmatrix}
\phantom{-}\cos (\pi/n) & \sin(\pi/n) \\ -\sin(\pi/n) & \cos(\pi/n)
\end{pmatrix}
\hskip.2cm \text{with $n \in \N$ large enough so that $u \in S$},
\\
R' \, &= \,  R \hskip.1cm \cup \hskip.1cm  \{ (\widetilde u)^{2n} \} .
\endaligned
\end{equation*}
Recall that the kernel of the universal covering 
$\langle \widetilde S \mid R \rangle \longrightarrow \SL_2(\R)$
is an infinite cyclic group, and observe that $(\widetilde u)^{2n}$ is a generator of this kernel.
Hence, as an example of Corollary \ref{cpcoveringof1cLieG},
$\langle \widetilde S \mid R' \rangle$ 
is a compact presentation of $\SL_2(\R)$.
\par

This procedure of writing down compact presentations can be adapted at will
to other groups of real or complex matrices.

\vskip.2cm

(3)
We refer to \cite{Viro} for an other kind of presentations
of groups such as $\R^n \rtimes \textnormal{O}(n)$,
with $S$ the \emph{non-compact} set of all reflections of $\R^n$.
These presentations are bounded, but not compact.
\end{exe}

Here is a strengthening of Corollary \ref{cpcoveringofcpG}
for groups containing compact open subgroups, as in Proposition \ref{propCayleyAbels}.

\begin{thm}
\label{cpcovers}
Let $G$ be a compactly \emph{generated} LC-group containing a compact open subgroup.
\index{Open subgroup! compact}
\par

There exists an LC-group $\widetilde G$, an open surjective homomorphism
with discrete kernel $\widetilde G \twoheadrightarrow G$,
and a vertex-transitive proper continuous action of $\widetilde G$ 
on a regular tree $T$ of bounded valency. 
In particular, $\widetilde G$ is compactly presented.
\par

If $G$ is totally disconnected, then $\widetilde G$ is totally disconnected as well.
\par
   
If $G$ has no non-trivial compact normal subgroup, 
then the action of $\widetilde G$ on $T$ can be chosen to be faithful.
\end{thm}

\begin{proof}
Let $X$ be a Cayley-Abels graph for $G$,
\index{Cayley-Abels graph}
as in \ref{propCayleyAbels};
recall that $X$ is a graph of bounded valency.
Let $N$ denote the kernel of the action of $G$ on $X$;
it is a compact normal subgroup of $G$, 
and $H := G/N$ acts faithfully on $X$.
We denote by $\rho : G \twoheadrightarrow H$ the canonical projection.
Recall that the automorphism group of $X$ is naturally an isometry group,
that we denote by $\Isom (X)$; see Example \ref{Aut(X)=Isom(X,d)}.
In particular, $\Isom (X)$ has a natural topology that makes it an LC-group,
and we can identify $H$ with a closed 
subgroup of $\Isom (X)$; see Proposition \ref{GinIsom(X)}.
For $x \in X$, let $H_x$ denote the isotropy group $\{h \in H \mid h(x) = x \}$,
and similarly for $\Isom (X)_x$.
Since $\Isom (X)_x$ is compact (Lemma \ref{GinIsom(X)}\ref{2DEGinIsom(X)}), 
and contains $H_x$ as a closed subgroup, $H_x$ is compact. 
\par

Let $\widetilde X$ be the universal covering tree of $X$,
and $\pi : \widetilde X \longrightarrow X$ the covering projection.
\index{Universal cover}
Let $\Isom(\widetilde X)$ be the group of automorphisms of $\widetilde X$,
with the LC-topology defined in Section \ref{isometrygroups}. 
Define $\widetilde H$ as the subgroup of $\Isom(\widetilde X)$
of isometries covering automorphisms in $H$.
By covering theory, 
the natural homomorphism $p : \widetilde H \longrightarrow H$ is surjective.
Moreover, for all $x \in X$ and $\tilde x \in \pi^{-1}(x)$, 
there is a continuous section 
$s_x : H_x \longrightarrow \widetilde H_{\tilde x}$ of $p$ over $H_x$,
i.e., a continuous homomorphism such that $p s_x$ is the identity of $H_x$.
It follows that $\widetilde H_{\tilde x}$ is a compact subgroup for all $\tilde x \in \widetilde X$,
hence that $\widetilde H$ is closed in $\Isom (\widetilde X)$.
\par

In particular, $\widetilde H$ is an LC-group,
$\ker (p)$ has trivial intersection with $\widetilde H_{\tilde x}$ 
for one (equivalently all) $\tilde x \in \widetilde X$,
and therefore $\ker (p)$ is discrete in $\widetilde H$.
\par

Consider the fibre product 
$\widetilde G = G \times_H \widetilde H := \{ (g,\tilde h) \in G \times \widetilde H \mid
\rho (g) = p(\tilde h) \}$, see Definition \ref{deffibreproduct}. 
\index{Fibre product}
If $G$ is totally disconnected, so is $\widetilde G$.
\par

The projection $\pi_G : \widetilde G \longrightarrow G$ 
is surjective with discrete kernel, because $\ker (p)$ is discrete in $\widetilde H$
(see Remark \ref{remfibreproduct}(1)).
The projection $\pi_{\widetilde H} : \widetilde G \longrightarrow \widetilde H$
is surjective with compact kernel, indeed $\ker (\pi_{\widetilde H}) = N \times \{1\}$.
The action of $\Isom (\widetilde X)$ on $\widetilde X$ 
continuous, isometric, and proper;
hence so is the action of $\widetilde H$ on $\widetilde X$;
since the latter is vertex-transitive, it is also geometric.
Hence, similarly, the action of $\widetilde G$ on $\widetilde X$ is  geometric.
\par

Since $\widetilde X$ is a tree, it follows that $\widetilde G$
is coarsely simply connected. Hence $\widetilde G$ is compactly presented.
\par

Denote by $K$ the kernel of the action of $\widetilde G$ on $\widetilde X$.
Since this actions factors through $\widetilde H$, we have $K \subset N \times \{1\}$.
When $G$ has no compact normal subgroup, we have $N = \{1\}$,
and it follows that $K = 1$, 
i.e., that the action of $\widetilde G$ on $\widetilde X$ is faithful.
\end{proof}

The following proposition and corollary can be found in articles by Abels.
See in particular 2.1 in \cite{Abel--69},
as well as Satz A and 3.2 in \cite{Abel--72}.
Recall that connected-by-compact LC-groups are
\index{Connected-by-compact topological group}
compactly generated (Proposition \ref{almostconnectedgroups}).

\begin{prop}
\label{connbycompcp}
Every connected-by-compact LC-group is compactly presented.
\end{prop}

\begin{proof}
Let $K$ be a maximal compact subgroup of $G$.
\index{Maximal! compact subgroup}
\index{Subgroup! maximal compact}
The space $G/K$ has a $G$-invariant structure of analytic manifold,
which is diffeomorphic  to $\R^n$ for some $n$ (Theorem \ref{deHochschild+}).
There exists a $G$-invariant Riemannian structure on $G/K$ (Proposition \ref{alaKoszul});
denote by $\underline d$ the Riemannian metric on $G/K$
and by $d$ the pseudo-metric on $G$ defined by 
$d(g_1,g_2) = \underline d(g_1K,g_2K)$.

Being geodesic and simply connected,
the space $(G/K, \underline d)$ is coarsely simply connected
(Proposition \ref{scimplycsc}).
Since the canonical map $(G,d) \longrightarrow (G/K, \underline d)$
is a quasi-isometry, the space $(G,d)$ is coarsely simply connected
(Proposition \ref{coarse1conninvbycoarseeq}).
Hence $G$ is compactly presented (Proposition \ref{cp iff csc}).
\end{proof}

\begin{cor}
\label{GcpintermsofG0}
An LC-group $G$ is compactly presented if and only
the quotient group $G/G_0$ is compactly presented.
\end{cor}

\begin{proof}
Since the connected component $G_0$ of $G$ 
is compactly presented by Proposition \ref{connbycompcp},
the corollary follows from Proposition \ref{cpforgroupsandquotients}.
\end{proof}

\begin{cor}
\label{fundgpLietypefini}
Let $L$ be a connected Lie group.
\par

(1) The fundamental group $\pi_1(L)$ is finitely generated.
\par

(2) Every discrete normal subgroup of $L$ is finitely generated.
\end{cor}

\begin{proof}
Denote by $E$ the universal cover of $L$, by $\pi$ its projection onto $L$,
and consider the short exact sequence
\begin{equation*}
\pi_1(L) \lhook\joinrel\relbar\joinrel\rightarrow E 
\overset{\pi}{\relbar\joinrel\twoheadrightarrow} L ,
\end{equation*}
where $\pi_1(L)$ is viewed as a discrete central subgroup $E$
(see Proposition \ref{discretenormalcentral}).
For a central subgroup of $E$,
finite generation as a group coincides with finite generation as a normal subgroup.
Since $E$ and $L$ are compactly presented by Proposition \ref{connbycompcp},
Claim (1) follows from Proposition \ref{cpforgroupsandquotients}.
\par

Let $\Gamma$ be a discrete normal subgroup of $L$.
Since $\pi^{-1}(\Gamma)$ is a discrete normal subgroup of $E$,
it can be viewed as the fundamental group
of the connected Lie group $E/\pi^{-1}(\Gamma)$.
Claim (1) follows from Claim (2).
\end{proof}

The following proposition is Theorem 3.4 in \cite{Abel--72}.

\begin{prop}[abelian and nilpotent groups]
\label{abeliancg=cp}
Every abelian compactly generated LC-group is compactly presented.
\index{LCA-group}
\par

More generally, 
every nilpotent compactly generated LC-group is compactly presented.
\index{Nilpotent! group}
\end{prop}

\begin{proof}
A compactly generated LCA-group contains a cocompact lattice that is
a finitely  generated free abelian group
(Example \ref{firstexampleslattices}(2)), 
hence is compactly presented by Corollary \ref{hereditaritycp}.
\index{Free abelian group}
\par

Consider now a nilpotent compactly generated LC-group $G$.
Let $i$ be the smallest integer such that $C^iG$ is central is $G$.
We proceed by induction on $i$.
Since the case $i=1$ is that of abelian groups, 
we can assume that $i \ge 2$ and that the proposition holds up to $i-1$.
We have a short exact sequence
\begin{equation*}
\overline{C^iG} \, \lhook\joinrel\relbar\joinrel\rightarrow \, G
\, \relbar\joinrel\twoheadrightarrow \, G / \overline{C^iG}  ,
\end{equation*}
with $G$ compactly generated.
Since $[G, C^{i-1}(G)] \subset \overline{C^iG}$,
the group $C^{i-1}(G / \overline{C^iG} )$ is central;
hence, by the induction hypothesis, 
$G / \overline{C^iG}$ is compactly presented.
It follows from Proposition \ref{cpforgroupsandquotients}\ref{2DEcpforgroupsandquotients}
that $\overline{C^iG}$ is compactly generated as a normal subgroup of $G$.
Since $\overline{C^iG}$ is central, it is compactly generated as a group,
and therefore (since it is abelian) it is compactly presented. 
Hence $G$ is compactly presented, 
by Proposition \ref{cpforgroupsandquotients}\ref{3DEcpforgroupsandquotients}.
\end{proof}

\begin{prop}[groups of polynomial growth]
\label{polgrowthcp}
Let $G$ be a locally compact, compactly generated group of polynomial growth. 
\index{Polynomial growth}\index{Growth! polynomial growth}
Then $G$ is compactly presented.
\end{prop}      

\begin{proof} By a result of Losert \cite{Lose--87}, 
$G$ has a compact normal subgroup $K$
such that $L := G/K$ is a Lie group of polynomial growth.
On the one hand, the connected component $L_0$ of $L$ 
is compactly presented by Proposition \ref{connbycompcp}.
On the other hand, the discrete group $L/L_0$ is finitely presented,
by Gromov's theorem \cite{Grom--81b} and Proposition \ref{abeliancg=cp}.
Hence Proposition \ref{polgrowthcp} follows from Proposition \ref{cpforgroupsandquotients}.
\end{proof}

\begin{rem}
\label{RemdeEB}
Groups of polynomial growth have been characterized by Gromov (case of discrete groups),
Guivarc'h and Jenkins (LC-groups with conditions),
and Losert (LC-groups in general).
See \cite{Grom--81b}, \cite{Guiv--73}, \cite{Jenk--73}, and \cite{Lose--01}.

It can be shown that any compactly generated LC-group of polynomial growth
is quasi-isometric to a simply connected solvable Lie group of polynomial growth
(but, in general,  ``solvable'' cannot be replaced by ``nilpotent'') \cite{Breu--14}.
Compare with Example \ref{excoarse}(8).
\end{rem}

\begin{prop}[Gromov-hyperbolic groups]
\label{GromovHypCP}
Let $G$ be a compactly generated LC-group.
If $G$ is Gromov-hyperbolic
(see Remark \ref{GromovHyp}),
\index{Gromov! hyperbolic group}
then $G$ is compactly presented.
\end{prop}

\begin{proof}
Gromov has shown that, for a compact generating set $S$ of $G$ 
and a large enough constant $c$, the Rips complex $\operatorname{Rips}^2_c(G,S)$
is simply connected.
(The full Rips complex  $\operatorname{Rips}_c(G,S)$ 
is contractible, see \cite[1.7.A]{Grom--87}).
Hence the metric space $(G,d_S)$ is coarsely simply connected
(Proposition \ref{propositionpourc-geodesicspace}),
so that $G$ is compactly presented.
\end{proof}

\begin{prop}
\label{excgnotcp}
For every local field $\K$, 
the group $\K^2 \rtimes \SL_2(\K)$ is compactly generated but not compactly presented.
\index{Affine group! $\K^2 \rtimes \SL_2(\K)$}
\index{Special linear group $\SL$! $ \K^2 \rtimes \SL_2(\K)$}
\index{Compactly presented! LC-group, not compactly presented}
\end{prop}

\begin{proof}
Define here the Heisenberg group $H(\K)$ as the set $\K^3$
\index{Heisenberg group}
with the multiplication
\begin{equation*}
(x,y,t) (x',y',t') \, = \, (x+x', y+y', t+t'+xy'-x'y) .
\end{equation*}
The group $\SL_2(\K)$ acts on $H(\K)$ by 
\begin{equation*}
\begin{pmatrix}
a & b \\
c & d 
\end{pmatrix}
(x,y,t) = (ax+by, cx+dy, t) .
\end{equation*}
It is easy to check that the semidirect product $G := H(\K) \rtimes \SL_2(\K)$
is compactly generated.
\index{Semidirect product! $H(\K) \rtimes \SL_2(\K)$}
The centre $Z$ of $G$ coincides with the centre of $H(\K)$, isomorphic to $\K$;
in particular, $Z$ is not compactly generated as a normal subgroup of $G$.
Thus by Proposition \ref{cpforgroupsandquotients}\ref{2DEcpforgroupsandquotients},
$G/Z \simeq \K^2 \rtimes \SL_2(\K)$ is not compactly presented.
\end{proof}

\noindent
\emph{Note.}
If the characteristic of $\K$ is not $2$, the map
\begin{equation*}
H(\K) \ni (x,y,t) \, \longmapsto \,
\exp 
\begin{pmatrix}
0 & 2x & 2t \\
0 & 0 & 2y \\
0 & 0 & 0 
\end{pmatrix}
=
\begin{pmatrix}
1 & 2x & 2(t+xy) \\
0 & 1 & 2y \\
0 & 0 & 1 
\end{pmatrix}
\in \GL_3(\K)
\end{equation*}
provides a topological isomorphism of the group $H(\K)$ defined here
with the matrix group $H(\K)$ 
defined in Example  \ref{coarselyexpansivenotlargescaleexpansive}.

\begin{rem}
\label{remsurexcgnotcp}
Proposition \ref{excgnotcp} carries over to 
$\K^{n+1} \rtimes \SL_2(\K)$ for $n=2m-1$ odd,
with a similar proof; 
the group $H_{\K}$ has to be replaced by
a higher dimensional Heisenberg group
$\K^m \times \K^m \times \K$,
with product 
\begin{equation*}
(x,y,t)(x',y',t') \, = \,  (x+x', y+y', t+t'+\omega(x,y)) ,
\end{equation*}
where $\omega$ is a non-degenerate alternating bilinear form on $\K^m$.
The proof cannot be adapted to the case of $n$ even,
but there is another proof:
see Proposition \ref{K3semidirectSL2}.
\end{rem}

\begin{rem}
An LC-group containing a finitely presented dense subgroup 
need not be compactly presented.
Indeed, if $\K$ is a local field of characteristic zero, 
the group $\K^2 \rtimes \SL_2(\K)$ of Proposition \ref{excgnotcp}
is not compactly presented, 
although it has dense subgroups which are free of finite rank
(for example by \cite[Corollary 1.4]{BrGe--07}).
\index{Subgroup! finitely presented dense}
\par

Recall from Remark \ref{Gexplescgsigmacsc}(3) 
that an LC-group which has a finitely generated dense subgroup is compactly generated.
\end{rem}

\begin{rem}
It is known that there are uncountably many isomorphism classes
of large-scale simply connected vertex-transitive graphs \cite{SaTe}.
It is unknown whether there are uncountably many quasi-isometry classes
of compactly presented totally disconnected LC-groups.
\par

For comparison, let us quote two other known results.
(a) There are uncountably many quasi-isometry classes
of connected real Lie groups, indeed of connected real Lie groups of dimension $3$.
See  \cite{Pans--89}, in particular Corollary 12.4, 
as well as \cite{Corn}, in particular
Conjecture 6.B.2 and Remark 6.B.5.
(b) There are uncountably many quasi-isometry classes
of finitely generated groups; see Remark \ref{refpourcroissance}.
\end{rem}
 
\section[Amalgamated products and HNN-extensions]
{Topological amalgamated products and HNN-extensions}
\label{sectionamalgamHNN}

\begin{defn}
\label{DefTopsub}
Let $G$ be a group, $H$ a subgroup, and $\mathcal T_H$ a topology on $H$
making it a topological group. 
\par

For $g \in G$, the topology $\mathcal T_H$ is \textbf{$g$-compatible}
if there exist open subgroups $H_1, H_2$ of $H$ such that $gH_1g^{-1} = H_2$
and $H_1 \longrightarrow H_2, \hskip.1cm h \longmapsto ghg^{-1}$ 
is an isomorphism of topological groups.
\index{Open subgroup}
\par

The topology $\mathcal T_H$ is \textbf{$G$-compatible}
\index{Compatible! compatible topology on a subgroup|textbf}
\index{Topology! $G$-compatible on a subgroup of a group $G$|textbf}
if it is $g$-compatible for all $g \in G$.
\end{defn}

\begin{lem}
\label{compatiblesubgroup}
Let $G, H, \mathcal T_H$ be as in Definition \ref{DefTopsub}.
Then 
\begin{equation*}
\{ g \in G \mid \mathcal T_H \hskip.2cm \text{is $g$-compatible} \}
\end{equation*}
is a subgroup of $G$ containing $H$.
\end{lem}

\begin{proof}
Let $g,g' \in G$. 
Suppose that $\mathcal T_H$ is $g$-compatible and $g'$-compatible.
Then there are open subgroups $H_1, H_2, H'_1, H'_2$ of $H$ 
and topological isomorphisms 
\begin{equation*}
H_1 \longrightarrow H_2, \hskip.2cm h \longmapsto ghg^{-1} ,
\hskip.5cm
H'_1 \longrightarrow H'_2, \hskip.2cm h \longmapsto g'h{g'}^{-1} .
\end{equation*}
Hence $\mathcal T_H$ is $gg'$-compatible, because
\begin{equation*}
H_1 \cap g^{-1} (H_2 \cap H'_1)g \longrightarrow g'H_2{g'}^{-1} \cap H'_2 ,
\hskip.5cm
h \longmapsto g'g h (g'g)^{-1} 
\end{equation*}
is a topological isomorphism.
Obviously, $\mathcal T_H$ is $g^{-1}$-compatible.
\end{proof}

\begin{prop}
\label{remoncompatibletop}
In the situation of Definition \ref{DefTopsub}, 
suppose that there exists a topology $\mathcal T_G$ on $G$ 
inducing $\mathcal T_H$ on $H$,  
making $G$ a topological group with $H$ an open subgroup.
Then $\mathcal T_H$ is $G$-compatible. 
\end{prop}

\begin{proof}
With the notation of Definition \ref{DefTopsub}, for all $g \in G$, it suffices to choose
$H_1 = g^{-1}Hg \cap H$ and $H_2 = H \cap gHg^{-1}$.
\end{proof}

The following proposition can be seen as 
a converse of Proposition \ref{remoncompatibletop}.

\begin{prop}
\label{PropTopsub}
Let $G, H, \mathcal T_H$ be as in Definition \ref{DefTopsub}.
If $\mathcal T_H$ is $G$-compatible,
then there exists a unique topology $\mathcal T_G$ on $G$
inducing $\mathcal T_H$ on $H$,
and making $G$ a topological group with $H$ an open subgroup.
\index{Topological group}
\par

If, moreover, $H$ is an LC-group for $\mathcal T_H$, 
then $G$ is an LC-group f or $\mathcal T_G$.
\end{prop}

The proof is just after \ref{deftransversal}.

\begin{lem}
\label{LemTopsub}
Let $G$ be a group and $\mathcal T$ a topology (possibly non-Hausdorff) on $G$.
Consider the following continuity properties,
\index{Topological group}
with respect to $\mathcal T$ on $G$ and $\mathcal T \times \mathcal T$ on $G \times G$:
\par

(M)
the multiplication $G \times G \longrightarrow G$ 
is continuous at $(1,1)$;
\par

(I)
the inversion $G \longrightarrow G, \hskip.2cm g \longmapsto g^{-1}$ 
is continuous at $1$;
\par

(C)
for all $g_0 \in G$, the conjugation 
$G \longrightarrow G, \hskip.2cm g \longmapsto g_0^{-1} g g_0$
is continuous at $1$;
\par

(L)
for all $g_0 \in G$, the left multiplication $G \longrightarrow G, \hskip.2cm g \longmapsto g_0g$
is continuous;
\par

(R)
for all $g_0 \in G$, the right multiplication 
$G \longrightarrow G, \hskip.2cm g \longmapsto g g_0$ is continuous.
\par\noindent
The following conditions are equivalent:
\par

(i)
$\mathcal T$ is a group topology on $G$;
\par

(ii)
Properties (M), (I), (L), (R) hold;
\par

(iii)
Properties (M), (I), (C), (L) hold.
\end{lem}

\noindent
\emph{Note.}
Let $S$ be a generating set of $G$.
Condition (C) is equivalent to:
\par

\emph{
(C')
for all $g_0 \in S$, the conjugation 
$G \longrightarrow G, \hskip.2cm g \longmapsto g_0^{-1} g g_0$
is continuous at $1$.
}

\begin{proof}
Implications (i) $\Rightarrow$ (ii) and (i) $\Rightarrow$ (iii) are obvious.
\par

Assume that (ii) holds.
Let $(g_i)_{i \in I}$ and $(h_j)_{j \in J}$ be two nets in $G$ 
converging respectively to $g$ and $h$.
We have $\lim_i g^{-1}g_i = 1$ and $\lim_j h_jh^{-1} = 1$ by (R) and (L),
hence $\lim_{i,j} g^{-1}g_i h_jh^{-1} = 1$ by (M),
and therefore $\lim_{i,j} g_i h_j = gh$ by (L) and (R) again.
Hence the multiplication is continuous.
Since we have $\lim_i g^{-1}g_i = 1$,
we have also $\lim_i g_i^{-1} g = 1$ by (I),
and therefore $\lim_i g_i^{-1} = g^{-1}$ by (R).
Hence the inversion is continuous.
We have shown that (i) holds.
\par 

Assume now that Properties (M), (I), and (L) hold.
It is then obvious that (R) implies (C). Conversely, assume that, furthermore, (C) holds.
Let $(g_i)_{i \in I}$ be a net in $G$ converging to $g \in G$, and let $g_0 \in G$.
We have $\lim_i g^{-1}g_i = 1$ by (L), 
hence $\lim_i g_0^{-1} g^{-1} g_i g_0 = 1$ by (C),
hence $\lim_i  g_i g_0 = g g_0$ by (L).
Hence Condition (R) holds. 
This shows that Conditions (ii) and (iii) are equivalent.
\end{proof}

\begin{defn} 
\label{deftransversal}
Let $G$ be a group and $H$ a subgroup. 
A \textbf{right transversal} \index{Transversal for a subgroup|textbf}
for $H$ in $G$ is a subset $T$ of $G$ containing $1$
such that the multiplication $H \times T \longrightarrow G, \hskip.2cm (h,t) \longmapsto ht$
is a bijection.
\end{defn}

\begin{proof}[Proof of Proposition \ref{PropTopsub}]
Let us first show the existence of $\mathcal T_G$.
Let $T$ be a right transversal of $H$ in $G$.
Consider on $H \times T$ the product of the  topology $\mathcal T_H$ on $H$ 
with the discrete topology on $T$.
Let $\mathcal T_G$ be the topology such that the bijection 
$H \times T \longrightarrow G, \hskip.2cm (h,t) \longmapsto ht$
is a homeomorphism.
Then $H$ is open in $G$ for $\mathcal T_G$.
\par

Condition (I) of Lemma \ref{LemTopsub} holds, 
because $H$ is a topological group and 
an open neighbourhood of $1$ in $G$ for $\mathcal T_G$. 
Condition (C) holds because $\mathcal T_H$ is $G$-compatible.
\par

Let us check Condition (M) of Lemma \ref{LemTopsub}.
Let $(h_i)_{i \in I}$, $(k_j)_{j \in J}$ be nets in $H$
and $(t_i)_{i \in I}$, $(u_j)_{j \in J}$ be nets in $T$
such that $\lim_{i,j} h_i t_i = 1$ and $\lim_{i,j} k_j u_j = 1$ in $G$.
This implies $\lim_i t_i = 1 \in T$; since $T$ has the discrete topology,
there is no loss of generality in assuming $t_i = 1$ for all $i \in I$;
similarly, we assume $u_j = 1$ for all $j \in J$.
Hence $\lim_{i,j} h_it_i k_ju_j = \lim_{i,j} h_i k_j = 1$,
where the last equality holds because $H$ is a topological group.
\par

There is a particular case of Condition (R) of Lemma \ref{LemTopsub}
that follows from the definition of the topology $\mathcal T_G$ :
for every converging net $(h'_i)_{i \in I}$ in $H$ and every element $g_0 \in G$,
we have $\lim_i (h'_i g_0) = (\lim_i h'_i ) g_0$.
\par

Let us check Condition (L) of Lemma \ref{LemTopsub}.
Let $(h_i)_{i \in I}$ be a net in $H$
and $(t_i)_{i \in I}$ a net in $T$
such that $\lim_i h_i t_i = ht \in G$, for some $h \in H$ and $t \in T$.
Let $h_0 \in H$ and $t_0 \in T$.
We have to check that $\lim_i (h_0t_0h_it_i) = h_0t_0ht$.
As above, there is no loss of generality in assuming $t_i = t$ for all $i \in I$.
We have $\lim_i h_i^{-1}h = 1$ in $H$,
and therefore $\lim_i h_0t_0 h_i h^{-1} t_0^{-1} h_0^{-1} = 1$,
because $\mathcal T_H$ is $G$-compatible. Hence
\begin{equation*}
\aligned
\lim_i h_0 t_0 h_i t \, &= \,  
\lim_i  \left( (h_0 t_0 h_i h^{-1} t_0^{-1} h_0^{-1}) h_0t_0 ht \right)
\\
\, &=  \,   \left( \lim_i  h_0 t_0 h_i h^{-1} t_0^{-1} h_0^{-1}\right) h_0t_0 ht  
\, = \,  h_0 t_0 ht ,
\endaligned
\end{equation*}
where the second equality holds by the particular case already checked
of Condition~(R).
\par

By Lemma \ref{LemTopsub}, $G$ with $\mathcal T_G$ is a topological group.
\par

To check the uniqueness of $\mathcal T_G$, note that,
given \emph{any} right transversal  $T$ for $H$ in $G$,
the natural bijection from $H \times T$ to $G$ has to be a homeomorphism,
with $H \times T$ having the product of the topology $\mathcal T_H$ on $H$
and the discrete topology on $T$.
\end{proof}

\begin{exe}
\label{ExTopsub}
(1)
Let $N$ be a normal subgroup of a group $G$.
Any $G$-invariant group topology on $N$ is $G$-compatible.
\par

Consider in particular a group topology $\mathcal T_1$ on $G$.
The restriction $\mathcal T_N$ of $\mathcal T_1$ to $N$
is a $G$-compatible topology on $N$.
Hence, by Proposition \ref{PropTopsub}, 
there exists a unique group topology $\mathcal T_2$ on $G$
with respect to which $N$ is an open subgroup.
\par

Note that, if $(G, \mathcal T_1)$ is locally compact and $N$ is $\mathcal T_1$-closed,
then $(G, \mathcal T_2)$ is locally compact as well.
Also, if $G/N$ is uncountable,
then $(G, \mathcal T_2)$ is not $\sigma$-compact.
\index{Sigma-compact! LC-group, non-sigma-compact}

\vskip.2cm

(2)
Every group topology on the centre of a group $G$ is $G$-compactible.

\vskip.2cm

(3)
Other examples of $G$-compatible topologies appear
in Propositions \ref{topamalgamation} (amalgamated products)
and \ref{topHNN} (HNN-extensions).

\vskip.2cm

(4)
Let $P$ denote the set of prime integers.
Let $\mathbf A_\Q$ denote the subring of the direct product $\R \times \prod_{p \in P} \Q_p$
consisting of those $(a_\infty, (a_p)_{p \in P})$ for which $a_p \in \Z_p$ 
for all but possibly finitely many $p$ in $P$ (it is a ``restricted direct product'').
Let $H := \R \times \prod_{p \in P} \Z_p$ be given the product of the usual topologies,
so that $\{0\} \times \prod_{p \in P} \Z_p$ is a compact subring of $H$;
the ring $H$ is locally compact, because all factors are locally compact and 
all but finitely many are compact.
By a straightforward variation on Proposition \ref{PropTopsub},
there exists a unique topology on $\mathbf A_\Q$ inducing the given topology on $H$
and making it an open subring.
With this topology, $\mathbf A_\Q$ is the \textbf{ring of adeles} of $\Q$. 
\index{Adeles|textbf} 
\par

Similarly, the \textbf{idele group} \index{Idele group}
\begin{equation*}
\mathbf I_\Q \, = \, 
\left\{ 
(a_\infty, (a_p)_{p \in P}) \in \R^\times \times \prod_{p \in P} \Q_p^\times
\hskip.1cm \Big\vert \hskip.1cm
\aligned
& a_p \in \Z_p^\times
\\
& \text{for almost all} \hskip.2cm p \in P 
\endaligned
\right\}
\end{equation*}
is naturally an LC-group in which
$\R^\times \times \prod_{p \in P} \Z_p^\times$ (with the product topology)
is an open subgroup.
For this topology, the continuous map $\mathbf I_\Q \longrightarrow \mathbf A_\Q \times \mathbf A_\Q$,
$a \longmapsto (a,a^{-1})$ is a homeomorphism onto its image, but the inclusion of $\mathbf I_\Q$
in $\mathbf A_\Q$ is not.
\par

Similarly, the restricted direct product
\begin{equation*}
\SL_2(\mathbf A_\Q) \, = \, 
\left\{ 
(g_\infty, (g_p)_{p \in P}) \in \SL_2(\R) \times \prod_{p \in P} \SL_2(\Z_p)
\hskip.1cm \Big\vert \hskip.1cm
\aligned
& g_p \in \SL_2(\Q_p) 
\\
& \text{for almost all} \hskip.2cm p \in P 
\endaligned
\right\}
\end{equation*}
is naturally an LC-group in which
$\SL_2(\R) \times \prod_{p \in P} \SL_2(\Z_p)$
is an open subgroup.
\par
These groups appear naturally in the theory of automorphic functions \cite{GGPS--69}.

\vskip.2cm

(5)
Let us indicate examples of non-compatible topologies on subgroups.
\par

Consider a topological group $G$, a subgroup $H$,
and the restriction $\mathcal T_H$ to $H$ of the group topology given on $G$.
If $\mathcal T_H$ is $G$-compatible and $H$ is connected,
then $H$ is normal in $G$.
Hence, if $H$ is a non-normal connected subgroup of $G$,
then $\mathcal T_H$ is not $G$-compatible.
For example, the pair
\begin{equation*}
H := \begin{pmatrix}
1 & \R \\ 0 & 1 
\end{pmatrix}
\, \subset \, G = \SL_2(\R)
\end{equation*}
provides a non-G-compatible topology $\mathcal T_H$ on $H$.
\end{exe}

\begin{defn}
\label{ReminderAmaHNN}
Propositions \ref{topamalgamation} and \ref{topHNN} deal with
amalgamated products and HNN-extensions for \emph{topological groups}.
In preparation for these, we recall the relevant definitions for \emph{groups}.
\par

The original articles are \cite{Schr--27} for amalgamated products 
and \cite{HiNN--49} for HNN-extensions.
A standard reference is \cite[Chapter IV, \S~2]{LySc--77}.
The two notions are particular cases of a more general notion,
fundamental groups of graphs of groups, 
which is one of the main subjects in \cite{Serr--70}.

\vskip.2cm

(1) 
Consider two groups $A, B$, a subgroup $C$ of $A$,
and an isomorphism $\varphi$ of $C$ onto a subgroup of $B$.
There exists a group $G$, well-defined up to canonical isomorphisms,
and two monomorphisms $\iota_A : A  \lhook\joinrel\relbar\joinrel\rightarrow G$,
$\iota_B : B  \lhook\joinrel\relbar\joinrel\rightarrow G$, 
such that the following ``universal property'' holds:

\begin{enumerate}[noitemsep]
\item[]
for every group $G'$ and homomorphisms $\alpha : A \longrightarrow G'$,
$\beta : B \longrightarrow G'$
such that $\beta(\varphi(c)) = \alpha(c)$ for all $c \in C$,
there exists a unique homomorphism $\gamma : G \longrightarrow G'$
such that $\gamma \circ \iota_A =  \alpha$ and 
$\gamma \circ \iota_B = \beta$.
\end{enumerate}

The group $G$, denoted here by $A \Conv_C B$,  
is the \textbf{amalgamated product} of $A$ and $B$
with respect to $C$ and $\varphi$.
\index{Amalgamated product|textbf}
Often, $A, B$ are identified to subgroups of $G$, using $\iota_A, \iota_B$ respectively,
and $C$ is identified to a subgroup of \emph{both} $A$ and $B$.
\par

The existence of $G$ with the required properties has to be proved, 
but this is standard, and will not be repeated here.
Note that $A \Conv_C B$ is the free product $A \Conv B$ 
when $C$ is reduced to $\{1\}$, and that $A \Conv_C B \simeq A$ when $C = B$.
An amalgamated product $A \Conv_C B$ is \textbf{non-trivial}
if  $A \supsetneqq C \subsetneqq B$
\index{Amalgamated product! non-trivial}
\par

If $A$ and $B$ have presentations
$\langle S_A \mid R_A \rangle$ and $\langle S_B \mid R_B \rangle$
respectively, then $A \Conv_C B$ has the presentation
$\langle S_A \sqcup S_B \mid R_A \sqcup R_B \sqcup 
\left( c = \varphi (c) \hskip.2cm \forall c \in C \right) \rangle$,
that we write shortly (and abusively, see Definition \ref{presentation})
\begin{equation*}
A \Conv_{C} B \, = \, \langle A, B \mid c = \varphi (c) 
\hskip.2cm \forall \hskip.1cm c \in C \rangle . 
\end{equation*}

\vskip.2cm

(2)
Consider a group $H$, two subgroups $K, L$ of $H$,
and an isomorphism $\varphi$ of $K$ onto $L$.
There exists a group $G$, well-defined up to canonical isomorphisms,
a monomorphism $\iota_H : H \lhook\joinrel\relbar\joinrel\rightarrow G$,
and an element $t \in G$, 
such that the following ``universal property'' holds:
\begin{enumerate}[noitemsep]
\item[]
for every group $G'$ with an element $t' \in G'$
and homomorphism $\psi : H \longrightarrow G'$
such that $\psi(\varphi(k)) = t' \psi(k) {t'}^{-1}$ for all $k \in K$,
there exists a unique homomorphism $\gamma : G \longrightarrow G'$
such that $\gamma \circ \iota_H = \psi$ and $\gamma(t) = t'$.
\end{enumerate}
The group $G$, denoted here by $\HNN(H, K, L, \varphi)$,
is the \textbf{HNN-extension} corresponding to $H, K, L, \varphi$,
and the element $t \in G$ is the \textbf{stable letter}.
\index{Stable letter of an HNN-extension}
(Note that, in the literature, our $t$ is often $t^{-1}$.)
We will identify $H$ to a subgroup of $G$, using $\iota_H$. 
\index{HNN! extension|textbf}
Such an extension is \textbf{ascending} if at least one of $K=H$, $L=H$ holds,
and \textbf{non-ascending} if $K \ne H \ne L$
\index{HNN! extension, (non-)ascending}
\par

As in (1) above, the existence of $G$ requires a proof, 
but it is standard and will not be repeated here.
Note that, when $K$ (and therefore $L$) is reduced to $\{1\}$,
then $G$ is the free product of $H$ with the infinite cyclic group generated by~$t$.
\par

If $H$ has a presentation $\langle S_H \mid R_H \rangle$, 
then $\HNN(H, K, L, \varphi)$ has the presentation
$\langle S_H \sqcup \{t\} \mid R_H \sqcup
\left(t k t^{-1} = \varphi (k) \hskip.2cm \forall k \in K \right) \rangle$,
that we write shortly
\begin{equation*}
 \HNN(H, K, L, \varphi) \, = \, 
\langle H, t \mid t k t^{-1} = \varphi(k) \hskip.2cm \forall \hskip.1cm k \in K \rangle .
\end{equation*}
\end{defn}

\begin{prop}
\label{topamalgamation}
Let $A,B$ be topological groups, $C$ an open subgroup of $A$,
and $\varphi : C \overset{\simeq}{\longrightarrow} \varphi(C) \subset B$
a topological isomorphism of $C$ onto an open subgroup of $B$; let
\begin{equation*}
G \, = \, A \Conv_{C} B \, = \, \langle A, B \mid c = \varphi (c) 
\hskip.2cm \forall \hskip.1cm c \in C \rangle 
\end{equation*}
denote the corresponding amalgamated product.
\begin{enumerate}[noitemsep,label=(\arabic*)]
\item\label{1DEtopamalgamation}
There exists a unique group topology on $G$ 
such that the inclusion $C \lhook\joinrel\relbar\joinrel\rightarrow G$ is a topological isomorphism 
onto an open subgroup.
Moreover, the natural homomorphisms 
$A \lhook\joinrel\relbar\joinrel\rightarrow G$ and $B \lhook\joinrel\relbar\joinrel\rightarrow G$
are topological isomorphisms onto open subgroups of $G$.
\index{Topology! on amalgamated products}
\item\label{2DEtopamalgamation}
If $C$ is locally compact,
then $G$ is locally compact.
\item\label{3DEtopamalgamation}
If $A$ and $B$ are compactly generated [respectively $\sigma$-compact],
then $G$ has the same property.
\item\label{4DEtopamalgamation}
Assume that $A$ and $B$ are compactly presented.
Then $G$ is compactly presented if and only if $C$ is compactly generated.
\end{enumerate}
\end{prop}

\begin{proof}
\ref{1DEtopamalgamation}
By Proposition \ref{remoncompatibletop}, the topology of $C$
is both $A$-compatible and $B$-compatible.
By Lemma \ref{compatiblesubgroup}, it is $G$-compatible.
By Proposition \ref{PropTopsub}, $G$ has a unique topology $\mathcal T_G$
making it a topological group inducing the given topology on $C$,
and in which $C$ is an open subgroup.
\par

Denote by $i_A$ the inclusion of $A$ into $G$.
Since $i_A$ induces a topological isomorphism 
$C \overset{\simeq}{\longrightarrow} i_A(C)$,
and $i_A(C)$ is open, it follows that it also induces
a topological isomorphism $A \overset{\simeq}{\longrightarrow} i_A(A)$, with open image. 
The same holds for $B$.

\vskip.2cm

Claims \ref{2DEtopamalgamation} and \ref{3DEtopamalgamation} are obvious. 

\vskip.2cm

For \ref{4DEtopamalgamation}, choose two compact presentations 
\begin{equation*}
A \, = \, \langle S_A \mid R_A \rangle
\hskip.5cm \text{and} \hskip.5cm
B \, = \, \langle S_B \mid R_B \rangle .
\end{equation*}
Suppose first that $C$ is compactly generated;
in this case, we assume moreover that $S_A$ 
contains a compact generating set $S_C$ of $C$
and that $S_B$ contains $\varphi(S_C)$.
For $c \in C$ we write $(c)_A$ for $c$ viewed in $A$ and $(c)_B$ for $\varphi(c)$ viewd in $B$.
We have presentations
\begin{equation*}
\aligned
A *_C B \, &= \, \langle A, B \mid (c)_A = (c)_B \hskip.2cm \forall \hskip.1cm c \in C \rangle
\\
\, &= \, \langle A, B \mid (c)_A = (c)_B \hskip.2cm \forall \hskip.1cm c \in S_C \rangle
\\
\, &= \, \langle S_A, S_B \mid R_A, \hskip.2cm R_B, \hskip.2cm
          (c)_A = (c)_B \hskip.2cm \forall \hskip.1cm c \in S_C \rangle .
\endaligned
\end{equation*}
The last one is a  compact presentation,
because $(c)_A = (c)_B$ is a relator of length $2$
for each $c \in S_C$.
\par

For the converse implication, consider a nested sequence 
of compactly generated open subgroups
\begin{equation*}
C_0 \, \subset \, C_1 \, \subset \, \cdots \, \subset \,
C_n \, \subset \, C_{n+1} \, \subset \, \cdots \, \subset C := \bigcup_{n \ge 0} C_n
\end{equation*}
(see Proposition \ref{powersSincpgroup}).
For each $n \ge 0$, let $K_n$ denote 
the kernel of the canonical surjective homomorphism  
$A *_{C_0} B \twoheadrightarrow A*_{C_n} B$.
Observe that $K_n \subset K_{n+1}$, 
and that the kernel of the canonical surjective homomorphism
$A *_{C_0} B \twoheadrightarrow A*_C B$ is $K := \bigcup_{n \ge 0} K_n$.
Since $K \cap C_0 = \{1\}$, and $C_0$ is open in $G$,
the kernels $K_n$ and $K$ are discrete in $G$.
\par

Suppose now that $G$ is compactly presented.
Then $K$ is finitely generated as a normal subgroup of $G$,
by Proposition \ref{cpforgroupsandquotients}.
It follows that there exists an integer $m \ge 0$ such that $K = K_m$,
and therefore such that $C = C_m$; hence $C$ is compactly generated.
\par

Claim \ref{4DEtopamalgamation} appears in 
\cite[Theorem 9, Page 117]{Baum--93}, for $G$ discrete.
\end{proof}

\begin{prop}
\label{topHNN}
Consider a topological group $H$, two open subgroups $K, L$ of $H$,
a topological isomorphism 
$\varphi : K \overset{\simeq}{\longrightarrow} L$,
and the resulting HNN-extension
\begin{equation*}
G \, = \,  \HNN(H, K, L, \varphi) \, = \, 
\langle H, t \mid t k t^{-1} = \varphi(k) \hskip.2cm \forall \hskip.1cm k \in K \rangle 
\end{equation*}
where $t$ is the stable letter.
\begin{enumerate}[noitemsep,label=(\arabic*)]
\item\label{1DEtopHNN}
There exists a unique topology on $G$ which makes it a topological group 
in which $H$ is an open subgroup.
\index{Topology! on HNN-extensions}
\item\label{2DEtopHNN}
If $H$ is locally compact [respectively $\sigma$-compact], then $G$ has the same property.
\item\label{3DEtopHNN}
If $H$ is  compactly generated, then $G$ has the same property.
\item\label{4DEtopHNN} 
Assume that $H$ is compactly presented.
Then $G$ is compactly presented if and only if $K$ is compactly generated.
\end{enumerate}
\end{prop}

\begin{proof}
The topology of $H$ is both $H$-compatible and $t$-compatible, 
and therefore $G$-compatible by Lemma \ref{compatiblesubgroup}.
Claims \ref{1DEtopHNN} and \ref{2DEtopHNN}
are straightforward applications of Proposition \ref{PropTopsub},
and \ref{3DEtopHNN} is immediate.

\vskip.2cm

For \ref{4DEtopHNN}, 
choose a compact presentation $\langle S_H \mid R_H \rangle$ of $H$.
Suppose first that $K$ is compactly generated;
in this case, we assume moreover that $S_H$ contains
a compact generating set $S_K$ of $K$, as well as $\varphi(S_K)$.
We have presentations
\begin{equation*}
\aligned
G \, &= \, \langle H, t \mid tkt^{-1} = \varphi(k) \hskip.2cm \forall \hskip.1cm k \in K \rangle
\\
\, &= \, \langle H, t \mid tkt^{-1} = \varphi(k) \hskip.2cm \forall \hskip.1cm k \in S_K \rangle
\\
\, &= \, \langle S_H, t \mid R_H, \hskip.2cm  tkt^{-1} = \varphi(k) 
\hskip.2cm \forall \hskip.1cm k \in S_K \rangle .
\endaligned
\end{equation*}
The last one is a  compact presentation,
because $tkt^{-1} = \varphi(k)$ is a relator of length $4$
for each $k \in S_K$.
\par

The proof of the converse implication is similar to the end of the proof
of Proposition \ref{topamalgamation},
and we leave the details to the reader.
\par

Claim \ref{4DEtopHNN} appears in \cite[Page 118]{Baum--93}, for $G$ discrete.
\end{proof}

\begin{rem}
\label{Alperin}
(1)
By a result of Alperin, if $G$ is an LC-group, 
and if the abstract group $G$ is an amalgam $A \Conv_C B$,
then $C$ is an open subgroup in $G$.
Similarly, 
if the abstract group $G$ is an HNN-extension $\HNN(H, K, L, \varphi)$,
then $K$ and $L$ are open subgroups in $G$.
See Corollary 2 of Theorem 6 in \cite{Alpe--82}.
In particular, every homomorphism $G \longrightarrow \Z$ is continuous.

\vskip.2cm

(2)
The two constructions are closely related.
In particular, if $G = \HNN(H, K, L, \varphi)$ is as in Proposition \ref{topHNN}, then
\begin{equation*}
G \, = \, (\cdots \Conv_K H \Conv_K H \Conv_K H \Conv_K H \Conv_K H \Conv_K \cdots) 
\rtimes_{\operatorname{shift}} \Z ,
\end{equation*}
where each left embedding $H \hookleftarrow K$ is the inclusion
and each right embedding $K \hookrightarrow H$ is $\varphi$.
See \cite{HiNN--49} and  \cite[No I.1.4, Proposition 5]{Serr--77}.

\vskip.2cm

(3)
For an LC-group $G$ and two open subgroups $A,B$ of $G$, with $C = A \cap B$,
the following two conditions are equivalent
(see \cite[$\S$~I.4.1]{Serr--77}):
\begin{enumerate}[noitemsep,label=(\roman*)]
\item
$G = A \Conv_C B$ as in Proposition \ref{topamalgamation},
\item
$G$ acts continuously without edge-inversion on a tree $T$,
the action is transitive on the set of geometric edges of the tree,
the quotient $G \backslash T$ is a segment,
and $A,B$ are the stabilizers of two adjacent vertices of the tree.
\end{enumerate}
This is an application of \cite{Serr--77} to the abstract group $G$;
the continuity of the action in (ii) follows from the fact
that all vertex and edge stabilizers are open.
\par

For an analogue concerning HNN-extensions, see Remark \ref{remhomosplits}(4).

\vskip.2cm

(4)
Let $G = A \Conv_C B$ be as in Proposition \ref{topamalgamation}. 
Assume moreover that the amalgam is \textbf{non-degenerate},
\index{Amalgamated product! non-degenerate}
i.e., that $C$ is of index at least $2$ in $A$ or $B$,
and of index at least $3$ in the other.
Then $G$ contains non-abelian free groups as discrete subgroups.
\index{Free subgroup}
\par

Similarly, let $G = \HNN (H, K, L, \varphi)$ be as in Proposition \ref{topHNN}.
Suppose that the HNN-extension is non-ascending,
i.e., that $K  \subsetneqq H \supsetneqq L$.
Then $G$ contains non-abelian free subgroups as discrete subgroups.
\par

These facts can be checked by making use of Lemma 2.6 in \cite{CuMo--87}.

\vskip.2cm

(5)
Let $G$ be an LC-group that is an ascending HNN-extension;
we use the notation of Proposition \ref{topHNN}.
Then $N := \bigcup_{n \ge 1} t^{-n}Ht^n$ is an open normal subgroup of $G$,
and $G/N \simeq \Z$.
In  the particular case of a compact group $H$,
the group $N$ is locally elliptic,
and it follows that $N$ and $G$ do not contain any discrete non-abelian free subgroup.
\end{rem}

\begin{defn}
\label{defends}
Let $G$ be a compactly generated LC-group.
Choose a geodesically adapted pseudo-metric $d$ on $G$,
as in Proposition \ref{geodad+wordmetric},
and a connected graph $X$ of bounded valency that is quasi-isometric to $(G,d)$,
as in Proposition \ref{predefuniformlycoarselyproper}.
The \textbf{space of ends} of $G$ 
\index{Ends|textbf}
is the space of ends of $X$.
\par

This definition is independent on the choices of $d$ and $X$.
In case $G$ is totally disconnected, 
a convenient choice for $X$ is a Cayley-Abels graph of $G$,
\index{Cayley-Abels graph}
see Definition \ref{defCayleyAbels}.
\par

Similarly to what happens in the discrete case,
there are four cases to distinguish:
the space of ends of a compactly generated LC-group can be
\begin{enumerate}[noitemsep]
\item[(0)]
empty, if and only if $G$ is compact,
\item[(1)] 
a singleton space, in some sense the generic case,
\item[(2)]
a two-points space, if and only if $G$ has a compact open normal subgroup $N$ 
such that $G/N$ is isomorphic to  $\Z$, or the infinite dihedral group $\Isom (\Z)$,
or $\R$, or $\Isom (\R) = \R \rtimes (\Z / 2\Z)$,
\index{Dihedral group}
\item[($\infty$)]
an infinite space, in which case it is uncountable (see Theorem \ref{Stallings}).
\end{enumerate}
\par

For another definition of the space of ends 
of an LC-group (not necessarily compactly generated),
see \cite{Spec--50} and \cite{Houg--74}.
\end{defn}

For totally disconnected LC-groups,
we have an analogue of Stallings' splitting theorem \cite{Stal--68}, 
due to Abels \cite{Abel--74}; see also \cite{Abel--77}, and the review article \cite{Moll}.
\index{Stallings' splitting theorem}
\index{Theorem! Stallings' splitting}

\begin{thm}
\label{Stallings}
Let $G$ be a compactly generated totally disconnected LC-group,
with an infinite space of ends.
Then one of the following holds:
\index{Totally disconnected! topological group}
\begin{enumerate}[noitemsep,label=(\arabic*)]
\item
there are LC-groups $A,B$, a compact open subgroup $C$ of $A$, of index at least~$2$,
and a topological isomorphism of $C$ onto an open subgroup of $B$, of index at least~$3$,
such that $G$ is topologically isomorphic to $A \Conv_C B$,
as in Proposition \ref{topamalgamation};
\item
there are a topological group $H$, 
two compact open subgroups $K,L$ of $H$, with $K \ne H$,
and a topological isomorphism $\varphi$ of $K$ onto $L$
such that $G$ is topologically isomorphic to $\HNN (H, K, L, \varphi)$,
as in Proposition \ref{topHNN}.
\end{enumerate}
\end{thm}

\section{Homomorphisms to $\Z$ and splittings}
\label{sectionsplitting}

It will be convenient to call {\bf epimorphisms} surjective homomorphisms.
\index{Epimorphism (group)}

\subsection{A topological version of the Bieri-Strebel splitting theorem}
\label{subsection_8Ca}

\begin{defn}
\label{defhomosplits}
Let $G$ be a group, $H$ a subgroup, 
and $\pi : G \twoheadrightarrow \Z$ an epimorphism.
We say that $\pi$ \textbf{splits over $H$}, or has a \textbf{splitting over $H$}, if
\index{Splitting|textbf} 
\begin{enumerate}[noitemsep,label=(\arabic*)]
\item\label{1DEdefhomosplits}
$H \subset \ker (\pi)$
\end{enumerate} 
and there exist splitting data, namely:
\begin{enumerate}[noitemsep,label=(\arabic*)]
\addtocounter{enumi}{1}
\item\label{2DEdefhomosplits}
an element, called the \textbf{stable letter}, $s \in \pi^{-1}(\{1\}) \subset G$ ;
\index{Stable letter of an HNN-extension}
\item\label{3DEdefhomosplits}
two subgroups $K,L$ of $H$ such that  $L = sKs^{-1}$ ;
\item\label{4DEdefhomosplits}
an isomorphism 
\index{HNN! extension}
\begin{equation*}
\psi \, : \, 
\HNN(H,K,L,\varphi) = 
\langle H,t \mid t k t^{-1} = \varphi (k)
\hskip.2cm \text{for all} \hskip.2cm k \in K \rangle
\, 
\overset{\simeq}{\longrightarrow} 
\, G,
\end{equation*}
where $\varphi$ is the isomorphism of $K$ onto $L$ defined by $\varphi(k) = s k s^{-1}$,
such that the restriction of $\psi$ to $H$ is the identity, and $\psi (t) = s$.
\end{enumerate}
The splitting is
\begin{enumerate}[noitemsep]
\item[]
\textbf{ascending} if at least one of $K=H$, $L=H$ holds,
\item[]
\textbf{essential} if $K \ne H \ne L$.
\index{Splitting! ascending, essential}
\end{enumerate}
If $G$ is a group with a subgroup $H$, 
it is common to say that \textbf{$G$ splits over $H$}
if some epimorphism $G \twoheadrightarrow \Z$ splits over $H$.
\end{defn}

\begin{rem}
\label{remhomosplits}
For (2) to (5) below,
we continue with the notation of Definition \ref{defhomosplits}.

\vskip.2cm

(1)
A group $G$ for which there exists an epimorphism onto $\Z$
is called an \textbf{indicable} group.
\index{Indicable group}

\vskip.2cm

(2)
The homomorphism $\pi$ splits over $H$ if and only if
$-\pi$ splits over $H$.

\vskip.2cm

(3)
The homomorphism $\pi$ \emph{always} splits over $N := \ker (\pi)$.
Indeed since $\Z$ is free, $\pi$ has a section $\sigma : \Z \longrightarrow G$,
hence $G = N \rtimes_\tau \Z \simeq \HNN(N,N,N,\tau)$
for $\tau$ the automorphism of $N$ of conjugation by $\sigma(1)$.
However, when it holds, it is a non-trivial property that $G$ split
over a \emph{finitely generated} subgroup of $N$,
or a \emph{compactly generated} open subgroup in the LC-context.

\vskip.2cm

(4)
Let $G$ be a topological group, $H$ an open subgroup,
and $\pi : G \twoheadrightarrow \Z$ a continuous epimorphism.
Then $\pi$ splits over $H$ if and only if it satisfies the following:

There exists a continuous action of $G$ on a tree $T$, without edge-inversion,
\index{Tree}
with the following property:
it is transitive on both the vertex set $V$ 
and the set $E$ of geometric edges of $T$,
and there is a vertex in the tree whose stabilizer is $H$.

See \cite{Serr--77}; the continuity argument is the same as in Remark \ref{Alperin}(3).

\vskip.2cm

(5) 
Suppose that $G$ is a topological group,
that the epimorphism $\pi$ is continuous,
and that $H$ is an open subgroup of $G$.
Consider an action of $G$ on a tree $T = (V,E)$ as in (3) above. 
The stabilizer $H_x$ of any vertex $x \in V$ is an open subgroup in $G$,
because it is conjugate to $H$.
The stabilizer of an edge $e \in E$ with origin $x \in V$ and extremity $y \in V$
is also open, because it coincides with $H_x \cap H_y$.
It follows that the action of $G$ on $T$ is continuous.

\vskip.2cm

(6)
Let $G = \HNN (H,K,L,\tau)$, with presentation 
$\langle H,t \mid tkt^{-1} = \tau(k) \hskip.2cm \forall k \in K \rangle$,
be a group given by an HNN-extension;
define an epimorphism $\pi : G \twoheadrightarrow \Z$ 
by $\pi (H) = 0$ and $\pi (t) = 1$.
Then $\ker (\pi)$ is the normal subgroup of $G$ generated by $H$,
as it can be seen on the presentation  above;
it is also a particular case of \cite[Section I.5.4, Corollary 1]{Serr--77}. 
\end{rem}

The following result is essentially the extension to the LC-setting 
of a result  due to Bieri and Strebel \cite[Theorem A]{BiSt--78}.
A particular case has been used by Abels in
\cite[Proposition I.3.2]{Abel--87}.

\begin{thm}[Bieri-Strebel splitting theorem]
\label{BieriStrebelAbels}
Let $G$ be a compactly presented LC-group.
\par

Every continuous epimorphism $\pi : G \twoheadrightarrow \Z$
splits over a compactly generated open subgroup of $\ker (\pi)$.
\end{thm}
\index{Theorem! Bieri-Strebel}

\begin{cor}
\label{ascHNN}
Let $G$ be a compactly presented LC-group.
Suppose that the free group of rank $2$ does not embed
as a discrete subgroup of $G$.
\index{Free group}
\par

For every continuous epimorphism $\pi : G \twoheadrightarrow \Z$,
either $\pi$ or $-\pi$ splits ascendingly over some
compactly generated open subgroup.
\end{cor}

\begin{proof}
By Theorem \ref{BieriStrebelAbels},
we know that $G$ is isomorphic,
with an isomorphism commuting with the homomorphism onto $\Z$,
to $\HNN (H,K,L,\varphi)$
for some compactly generated open subgroups $H,K,L$ of $G$ 
with $K \subset H$ and $L \subset H$.
If we had $K \ne H \ne L$,
the group $G$ would contain a discrete non-abelian free subgroup
(see Remark \ref{Alperin}(4)). 
\end{proof}

The following corollary, a strengthening of Theorem \ref{BieriStrebelAbels},
is a consequence of its \emph{proof}.

\begin{cor}
\label{corBSA+BGH}
Let $G$ be a compactly \emph{generated} LC-group
and $\pi : G \twoheadrightarrow \Z$ a continuous epimorphism.
\par

For every compactly \emph{presented} LC-group $E$  
and continuous epimorphism 
$\rho : E \twoheadrightarrow G$,
there exist a compactly generated LC-group $G'$
and a factorization $E \twoheadrightarrow G' \twoheadrightarrow G$ of $\rho$
with $\ker (G' \twoheadrightarrow G)$ discrete in $G'$,
such that the composite epimorphism 
$G' \twoheadrightarrow G \twoheadrightarrow \Z$
splits over a compactly generated open subgroup of $G'$.
\end{cor}

\subsection{Proofs}
\label{subsection_8Cb}

We introduce the following convenient shorthand:

\begin{defn}
\label{iseudk}
An \textbf{inductive system of epimorphisms with uniformly discrete kernels},
or for short \textbf{ISED}, is a sequence
\index{ISED|textbf}
\begin{equation*}
G_1  \overset{\varphi_{1,2}}{\twoheadrightarrow} 
G_2  \twoheadrightarrow \cdots \twoheadrightarrow 
G_n  \overset{\varphi_{n,n+1}}{\twoheadrightarrow} 
G_{n+1} \twoheadrightarrow \cdots ,
\end{equation*}
where $G_n$ is an LC-group and $\varphi_{n,n+1}$ a continuous epimorphism
for each $n \ge 1$,
such that $K_{1,\infty}$ (defined below) is a discrete subgroup of $G_1$.
\par

For $n \ge m \ge 1$, let $K_{m,n}$ denote the kernel of the composition
\begin{equation*}
\varphi_{m,n} := \varphi_{n-1,n} \circ \varphi_{n-2,n-1} \circ \cdots \circ \varphi_{m,m+1}
\, : \,  G_m \twoheadrightarrow G_n .
\end{equation*}
Observe that $K_{m,n} \subset K_{m,n+1}$.
Set $K_{m,\infty} := \bigcup_{n \ge m} K_{m,n}$.
Define the LC-group 
\begin{equation*}
G_\infty \, = \, G_1 / K_{1,\infty} .
\end{equation*}
Note that,  for each $m \ge 1$,
we have the equality $K_{m,\infty} = \varphi_{1,m}(K_{1,\infty})$, 
a canonical identification $G_\infty = G_m / K_{m,\infty}$,
and a canonical epimorphism 
$\varphi_{m,\infty} : G_m \twoheadrightarrow G_\infty$.
[Thus $G_\infty$ is an inductive limit of the system $(G_n, \varphi_{n,n+1})_{n \ge 1}$
in both the category of groups and the category of LC-groups.]
\end{defn}

\begin{lem}
\label{encoreunlift}
The notation being as in Definition \ref{iseudk}, 
suppose that  $G_1$ is compactly generated.
\begin{enumerate}[noitemsep,label=(\arabic*)]
\item\label{1DEencoreunlift}
Let $E$ be a compactly presented group 
and $\rho : E \twoheadrightarrow G_\infty$ a continuous epimorphism.
For $m$ large enough, there exists 
a continuous epimorphism $\rho_m : E \twoheadrightarrow G_m$
such that $\rho = \varphi_{m,\infty} \circ \rho_m$.
\item\label{2DEencoreunlift}
Suppose that $G_\infty$ is compactly presented.
Then $\varphi_{m,\infty} : G_m \longrightarrow G_\infty$ 
is an isomorphism for $m$ large enough.
\end{enumerate}
\end{lem}

\begin{proof}
\ref{1DEencoreunlift}
By Proposition \ref{pullbackCgLCgroups}, there exists a compactly generated group $L$
and two continuous epimorphisms $\pi_E, \pi_{G,1}$ such that the diagram
\begin{equation*}
\dia{
L
\ar@{->>}[dd]^{\pi_E}
\ar@{->>}[r]^{\pi_{G,1}}
& G_1 
\ar@{->>}[d]^{\varphi_{1,m}}
\\
& G_m
\ar@{->>}[d]^{\varphi_{m,\infty}}
\\
E
\ar@{->>}[r]^{\rho}
& G_\infty
}
\end{equation*}
commutes;
moreover, there exists a topological isomorphism of  $\ker (\pi_E)$
to a discrete subgroup of $K_{1,\infty}$;
in particular, $\ker (\pi_E)$ is discrete in $L$.
Since the quotient $E$ of $L$ is compactly presented,
the kernel of $\pi_E$ is finitely generated as a normal subgroup 
(Proposition \ref{cpforgroupsandquotients}).
Hence the image by $\pi_{G,1}$ of $\ker (\pi_E)$ in $G_1$
is also finitely generated as a normal subgroup of $G_1$.
It follows that $\pi_{G,1}( \ker (\pi_E))$ is in $K_{1,m}$ for some $m \ge 1$,
and therefore that $\rho$ factors by a continuous epimorphism
$\rho_m : E \twoheadrightarrow G_m$.

\vskip.2cm

\ref{2DEencoreunlift}
This is a particular case of \ref{1DEencoreunlift}, 
for $\rho$ the identity automorphism of $G_\infty$.
\end{proof}

The Bieri-Strebel Theorem \ref{BieriStrebelAbels} 
follows from the following stronger but a bit technical result.

\begin{thm}
\label{firststepproofBS}
Let $G$ be a compactly generated LC-group 
and $\pi : G \twoheadrightarrow \Z$ a continuous epimorphism.
\par
 
There exist an ISED $(G_n, \varphi_{n,n+1})_{n \ge 1}$
and a topological isomorphism 
$p_\infty : G_\infty \overset{\simeq}{\longrightarrow} G$,
such that, for every $m \ge 1$,
the continuous epimorphism
\begin{equation*}
\pi_m \, : \, 
G_m \, \overset{\varphi_{n,\infty}}{\relbar\joinrel\twoheadrightarrow} \,
G_\infty \, \overset{p_\infty}{\longrightarrow} \,
G \,  \overset{\pi}{ \relbar\joinrel\twoheadrightarrow} \,
\Z
\end{equation*}
splits over a compactly generated open subgroup of $G_m$
(where $G_\infty$ and $\varphi_{m,\infty} : G_m \twoheadrightarrow 
G_\infty = G_1 / K_{1, \infty} = G_m / K_{m,\infty}$
are as in Definition \ref{iseudk}).
\end{thm}

\begin{proof}
Set $N = \ker (\pi)$, and choose $t \in G$ such that $\pi(t) = 1 \in \Z$.
Since $\Z$ is finitely presented, Proposition \ref{cpforgroupsandquotients}
implies that there exists a compact subset $S$ of $N$
such that $\bigcup_{j \in \Z} t^j S t^{-j}$ generates $N$;
observe that $S \cup \{t,t^{-1}\}$ is a compact generating set of $G$.
We assume moreover that $S$ is symmetric
and has a non-empty interior.
\par

The kernel $N$ contains the union $B_\infty$ of compactly generated open subgroups
\begin{equation*}
B_n \, := \, \operatorname{gp} \{ t^\ell s t^{-\ell} \mid 0 \le \ell \le n
\hskip.2cm \text{and} \hskip.2cm s \in S  \}, \hskip.5cm n \ge 1.
\end{equation*}
Consider for each $n \ge 1$ the two open subgroups  
\begin{equation*}
\aligned
U_n \, &:= \, \operatorname{gp} \{ t^\ell s t^{-\ell} \mid 0 \le \ell \le n-1
\hskip.2cm \text{and} \hskip.2cm s \in S  \} 
\\
V_n \, &:= \, \operatorname{gp} \{t^\ell s t^{-\ell} \mid 1 \le \ell \le n
\hskip.2cm \text{and} \hskip.2cm s \in S  \} 
\endaligned
\end{equation*}
of $B_n$, the continuous isomorphism
\begin{equation*}
\tau_n \, : \,  U_n \overset{\simeq}{\longrightarrow} V_n, \hskip.2cm
u \longmapsto t u t^{-1},
\end{equation*}
and the compactly generated LC-group HNN-extension
\begin{equation*}
G_n \, := \, \langle B_n, t_n \mid t_n u t_n^{-1} = \tau_n (u)
\hskip.2cm \text{for all} \hskip.2cm u \in U_n \rangle 
\end{equation*}
(see Proposition \ref{topHNN}).
By Britton's lemma, we can and do view $B_n$ as a subgroup of $G_n$.
Observe that $G_n$ is actually generated by $S$ (contained in $U_n$) and $t_n$;
indeed, the defining relations in $G_n$ imply that
\begin{equation*}
t_n^\ell s t_n^{-\ell} \, = \, (\tau_n)^\ell (s) \, = \,
t^{\ell} s t^{-\ell}
\end{equation*}
for all $\ell$ with $0 \le \ell \le n$
and $s \in S$.
There are canonical continuous epimorphisms
$\varphi_{n,n+1} :  G_n \twoheadrightarrow G_{n+1}$;
more precisely, $\varphi_{n,n+1}$ is defined by its restriction to $B_n$,
which is the inclusion $B_n \subset B_{n+1}$,
and by $\varphi_{n,n+1}(t_n) = t_{n+1}$.
As in Definition \ref{iseudk}, we have for all $m \ge 1$ a normal subgroup
$K_{m,\infty} = \bigcup_{n \ge m} K_{m,n}$ of $G_m$, and we set
and set
\begin{equation*}
G_\infty \, = \, G_1 / K_{1,\infty} .
\end{equation*}
Since $K_{1,\infty} \cap B_1 = \{1\}$, and $B_1$ is open in $G_1$,
the subgroup $K_{1,\infty}$ of $G_1$ is discrete.
Hence $(G_n, \varphi_{n,n+1})_{n \ge 1}$ is an ISED.
\par

The sequence of $n$th term $t_n$ provides an element that we denote by 
$t_\infty \in G_\infty$.
Moreover, there are
epimorphisms $\pi_m : G_m \twoheadrightarrow \Z$ such that the diagram
\begin{equation*}
\dia{
G_1
\ar@{->>}[d]^{\pi_1}
\ar@{->>}[r]^{\varphi_{1,2}}
& G_2 
\ar@{->>}[d]^{\pi_2}
\ar@{->>}[r]^{\varphi_{2,3}}
& \cdots 
\ar@{->>}[r]
& G_m
\ar@{->>}[d]^{\pi_m}
\ar@{->>}[r]^{\varphi_{m,m+1}}
& G_{m+1}
\ar@{->>}[d]^{\pi_{m+1}}
\ar@{->>}[r]
& \cdots 
\ar@{->>}[r]^{\varphi_{n,\infty}}
& G_\infty
\ar@{->>}[d]^{\pi_\infty}
\\
\Z
\ar@{->>}[r]^{\operatorname{id}}
& \Z
\ar@{->>}[r]^{\operatorname{id}}
& \cdots
\ar@{->>}[r]
& \Z
\ar@{->>}[r]^{\operatorname{id}}
& \Z
\ar@{->>}[r]
& \cdots
\ar@{->>}[r]^{\operatorname{id}}
& \Z
}
\end{equation*}
commutes.
More precisely, $\pi_m$ is defined by $\pi_m(B_m) = \{0\}$,
and by $\pi_m(t_m) = 1$.
\par

For each $m \ge 1$,
we have also an epimorphism $p_m : G_m \twoheadrightarrow G = \langle S,t \rangle$,
such that the restriction of $p_m$ to $B_m$ is the inclusion $B_m \subset N$,
and $p_m(t_m) = t$.
Since $p_m = p_{m+1}\varphi_{m,m+1}$, we have at the limit
an epimorphism 
\begin{equation*}
p_\infty \, : \,  G_\infty \twoheadrightarrow G
\end{equation*}
such that $p_\infty (t_\infty) = t$.
Observe that $p_m = p_\infty \circ \varphi_{m,\infty}$
and $\pi_\infty = \pi \circ p_\infty$,
and consequently 
$\pi \circ p_\infty \circ \varphi_{m,m+1} = \pi_\infty \circ \varphi_{m,\infty} = \pi_m$
for all $m \ge 1$.
\par

We claim that  $p_\infty$ is an isomorphism.
For this, it remains to show that $\ker (p_\infty) = \{1\}$, i.e., that $\ker (p_1) = K_{1,\infty}$.
\par

Let $w \in G_1$.
There exist $k \ge 0$, $n_1, \hdots, n_k \in \Z$, and $s_1, \hdots, s_k \in S$
such that $w = t_1^{n_1} s_1 \cdots t_1^{n_k} s_k$.
Assume that $w \in \ker (p_1 : G_1 \twoheadrightarrow G)$.
Since $\pi_\infty (\varphi_{1, \infty}(w)) = \pi ( p_1 (w)) = 0$, 
we have $n_1 + \cdots + n_k = 0$.
If $m_1 = n_1$, $m_2 = n_1+n_2$, $\hdots$, $m_{k-1} = n_1 + \cdots n_{k-1}$, $m_k = 0$,
we can write $w =\prod_{j=1}^m t_1^{m_j} s_j t_1^{-m_j}$.
Upon conjugating $w$ by a large enough power of $t$,
we can furthermore assume that $m_1, \hdots, m_k \ge 0$.
If $m := \max \{m_1, \hdots, m_k \} + 1$, 
we can view $w$ in $G_m$, indeed in $B_m$.
Since the restriction of $p_m$ to $B_m$ 
is the inclusion $B_m \subset N \subset G$,
we have $w=1$ in $G_m$, i.e., $w \in K_{1,\infty}$, and the claim is proved.
\end{proof}

\begin{proof}[\textbf{Proof of Theorem \ref{BieriStrebelAbels}}]
Our strategy of proof has two steps.
The first step holds for every compactly \emph{generated} LC-group $G$
given with a continuous epimorphism onto $\Z$:
we construct an ISED $(G_n, \varphi_{n,n+1})_{n \ge 1}$
as in Theorem \ref{firststepproofBS}.
For each $n \ge 1$, the group $G_n$ splits over $B_n$,
which is a compactly generated open subgroup of $\ker (\pi_n)$.
\par

For the second step, it is crucial that $G$ is compactly \emph{presented}.
By Lemma \ref{encoreunlift}\ref{2DEencoreunlift}, 
the group $G_m$ is isomorphic to $G_\infty$,
and therefore to $G$, for $m$ large enough. The conclusion follows.
\end{proof}

\begin{proof}[\textbf{Proof of Corollary \ref{corBSA+BGH}}]
The first step of the proof of Theorem \ref{BieriStrebelAbels}
applies without change.
We have an ISED
$(\varphi_{n,n+1} : G_n \twoheadrightarrow G_{n+1})_{n \ge 1}$ 
and epimorphisms $\rho, \varphi_{1,n}, \varphi_{n,\infty}$ 
as in the diagram below.
\par

By Lemma \ref{pullbackCgLCgroups},
there exist a compactly generated LC-group $D$ 
and two epimorphisms $\pi_E, \pi_{G,1}$ such that the diagram
\begin{equation*}
\dia{
D
\ar@{->>}[dd]^{\pi_E}
\ar@{->>}[r]^{\pi_{G,1}}
& G_1 
\ar@{->>}[d]^{\varphi_{1,n}}
\\
& G_n
\ar@{->>}[d]^{\varphi_{n,\infty}}
\\
E
\ar@{->>}[r]^{\rho}
& G
}
\end{equation*}
commutes.
By Lemma \ref{encoreunlift}\ref{1DEencoreunlift}, 
we can lift $\rho$ to a continuous epimorphism $\rho_m : E \twoheadrightarrow G_m$
for $m$ large enough, and $\rho = \varphi_{m,\infty} \circ \rho_m$.
The corollary holds with $G' = G_m$.
\end{proof}

\subsection{Engulfing automorphisms}
\label{subsection_8Cc}

\begin{defn}
\label{defcontracts}
Let $N$ be a group and $H$ a subgroup.
An automorphism $\sigma$ of $N$
\index{Engulfing automorphism|textbf}
\textbf{engulfs $N$ into $H$} 
if $\sigma(H) \subset H$ 
and $\bigcup_{k \ge 0} \sigma^{-k}(H) = N$.
\end{defn}

\begin{rem}
\label{remsurengulfing}
Let $N$ be an LC-group, $H$ an open subgroup of $N$,
and $\sigma$ a topological automorphism of $N$.

\vskip.2cm

(1)
The automorphism $\sigma$
engulfs $N$ into $H$ if and only if $\sigma (H) \subset H$
and, for every compact subset $K$ of $N$,
there exists $\ell \ge 0$ such that $\sigma^\ell (K) \subset H$.
\par

Indeed, let $K$ be a compact subset of $N$;
if $\sigma$ engulfs $N$ into $H$, we have an open covering
$\bigcup_{k \ge 0} (K \cap \sigma^{-k}(H))$ of $K$,
hence there exists $\ell \ge 0$ such that $K \subset \sigma^{-\ell}(H)$,
i.e.,  $\sigma^\ell(K) \subset H$.

\vskip.2cm

(2)
If $\sigma$ engulfs $N$ into $H$,
then $\sigma^k$ engulfs $N$ into $H$ for all $k \ge 1$.
If moreover $H \subsetneqq N$,
then $\sigma^{k+1}(H) \subsetneqq \sigma^k(H)$ for all $k \in \Z$;
in particular,  $\sigma$ is of infinite order.
\end{rem}

\begin{exe}
Let $\K$ be a local field, $V$ a finite-dimensional vector space over $\K$, 
and $\varphi \in \GL (V)$.
Then $\varphi$ engulfs $V$ into a compactly generated subgroup of $V$
if and only if its spectral radius is strictly less than $1$.
\index{Spectral radius of a matrix}
\end{exe}

We use the following notation in the next proposition, and further down:
for $g$ in some group $L$, or in a larger group in which $L$ is  normal,
$\alpha_g$ stands for the automorphism $x \longmapsto gxg^{-1}$ of $L$.

\begin{prop}
\label{8C10.5du21oct}
Let $G$ be an LC-group, $\pi : G \twoheadrightarrow \Z$ a continuous epimorphism,
and $s \in \pi^{-1}(\{1\})$; set $N = \ker (\pi)$;  let $H$ be an open subgroup of $N$.
The following properties are equivalent:
\begin{enumerate}[noitemsep,label=(\roman*)]
\item\label{iDE8C10.5du21oct}
$\pi$ splits ascendingly over $H$ with stable letter $s$,
\item\label{iiDE8C10.5du21oct}
$\alpha_s$ engulfs $N$ into $H$.
\end{enumerate}
\end{prop}

\begin{proof}
If \ref{iDE8C10.5du21oct} holds, it is straightforward that \ref{iiDE8C10.5du21oct} holds.
\par

Suppose conversely that \ref{iiDE8C10.5du21oct} holds.
Since $sHs^{-1} \subset H$,
there exists by the universal property of HNN-extensions a homomorphism 
\begin{equation*}
\psi \, : \, \HNN (H,H,sHs^{-1},\alpha_s) 
= \langle H,t \mid t x t^{-1} = \alpha_s (x) \hskip.2cm \text{for all} \hskip.2cm x \in H \rangle
\longrightarrow G ,
\end{equation*}
such that $\psi(x) = x$ for all $x \in H$ and $\psi(t) = s$.
Since $\bigcup_{k \ge 0} s^{-k} H s^k = N$, 
the subset $H \cup \{s\}$ of $G$ is generating,
hence $\psi$ is surjective.
It remains to check that $\ker (\psi) = \{1\}$.
\par

Observe first that 
$s^k \notin \bigcup_{\ell \in \N} s^{-\ell} H s^\ell$
for every $k \in \Z \smallsetminus \{0\}$.
Every element in $\HNN (H,H,sHs^{-1},\alpha_s)$ can be written 
as $t^{-m} x t^{n}$, with $m,n \in \N$ and $x \in H$.
If $t^{-m} x t^{n} \in \ker (\psi)$, then $s^{-m} x s^{n} = 1 \in G$,
i.e., $x = s^{m-n}$, hence $m=n$ and $x = 1$ by the observation.
Hence $\ker (\psi) = \{1\}$.
\end{proof}

The engulfing property is not very sensitive to composition with inner automorphisms:

\begin{lem}
\label{engulf+inner}
Let $N$ be a group, $H$ a subgroup, and $\sigma$ an automorphism of $N$
engulfing $N$ into $H$.
Then, for every inner automorphism $\alpha$ of $N$, there exists $m \in \N$
such that $\alpha \circ \sigma$ engulfs $N$ into $\sigma^{-m}(H)$.
\end{lem}

\begin{proof}
Let $h \in N$ be such that $\alpha = \alpha_h$.
There exists $m \in \N$ such that $h \in \sigma^{-m}(H)$.
Let us check that $\alpha \circ \sigma$ engulfs $N$ into $\sigma^{-m}(H)$.
\par

By a straightforward induction,
$(\alpha_h \circ \sigma)^n (g) = \left( \alpha_{w_n} \circ \sigma^n \right) (g)$
for every $n \ge 1$ and $g \in N$, where
$w_n = h \sigma (h) \cdots \sigma^{n-1} (g)$.
We have $w_n \in \sigma^{-m}(H)$ 
and $\sigma^n (g) \in H \subset \sigma^{-m}(H)$ for $n$ large enough,
so that $(\alpha_h \circ \sigma)^n (g) \in \sigma^{-m} (H)$ for $n$ large enough.
\end{proof}

\begin{prop}
\label{du21octobre}
Let $G$ be an LC-group and $\pi : G \twoheadrightarrow \Z$ a continuous epimorphism;
set $N = \ker (\pi)$;  let $H$ be an open subgroup of $N$.
The following properties are equivalent:
\begin{enumerate}[noitemsep,label=(\roman*)]
\item\label{idu21octobre}
$\pi$ splits ascendingly over $H$;
\item\label{iidu21octobre}
for all $t \in \pi^{-1}(\{1\})$, the epimorphism $\pi$ splits, with stable letter $t$,
over some conjugate of $H$ in $N$.
\end{enumerate}
\end{prop}

\begin{proof}
To show the non-trivial implication, assume that (i) holds.
By Proposition \ref{8C10.5du21oct},
the automorphism $\alpha_s$ engulfs $N$ into $H$
\emph{for some} $s \in \pi^{-1}(\{1\})$.
Set $\beta = \alpha_{ts^{-1}}$; it is an inner automorphism of $N$.
By Lemma \ref{engulf+inner}, there exists $m \in \N$ such that the automorphism
$\alpha_t = \beta \circ \alpha_s$ engulfs $N$ into $\alpha_s^{-m}(H)$.
\end{proof}

\begin{rem}
The interest of the previous proposition is that,
if $\mathcal P$ is an isomorphism-invariant property of LC-groups,
the condition that an LC-group $G$, 
given with a continuous epimorphism $\pi : G \twoheadrightarrow \Z$,
splits over an open subgroup satisfying $\mathcal P$,
with stable letter $t \in \pi^{-1}(\{1\})$,
does not depend on the choice of $t$.
\end{rem}

\begin{cor}
\label{engulfsemidirect}
Let $N$ be an LC-group, $\varphi$ a topological automorphism of $N$,
and $G = N \rtimes_\varphi \Z$ the corresponding semidirect product.
\index{Semidirect product}
\begin{enumerate}[noitemsep,label=(\arabic*)] 
\item\label{1DEengulfsemidirect}
The following properties are equivalent:
\begin{enumerate}[noitemsep,label=(\roman*)] 
\item\label{iDEengulfsemidirect}
the natural projection $G \twoheadrightarrow \Z$
splits ascendingly over some compactly generated open subgroup of $N$,
\item\label{iiDEengulfsemidirect}
$\varphi$ engulfs $N$ into some compactly generated open subgroup of $N$.
\end{enumerate}
If \ref{iDEengulfsemidirect} and \ref{iiDEengulfsemidirect} hold, 
then $G$ is compactly generated.
\vskip.2cm
\item\label{2DEengulfsemidirect}
The same holds, 
with ``compactly presented''
instead of ``compactly generated''.
\end{enumerate}
\end{cor}

\begin{proof}
\ref{1DEengulfsemidirect}
Implication \ref{iDEengulfsemidirect} $\Rightarrow$ \ref{iiDEengulfsemidirect} 
holds by Proposition \ref{8C10.5du21oct}.
Also, if  \ref{iDEengulfsemidirect} holds, then $G$ is compactly generated
by Proposition \ref{topHNN}\ref{3DEtopHNN}.
\par 

Assume now that \ref{iiDEengulfsemidirect} holds.
Let $s$ denote the positive generator of $\Z$.
By Proposition \ref{du21octobre}, $\pi$ splits ascendingly with stable letter $s$.
Then \ref{iDEengulfsemidirect} holds by Proposition \ref{8C10.5du21oct}.
\par

The proof of \ref{2DEengulfsemidirect} is analogous,
with reference to Proposition \ref{topHNN}\ref{4DEtopHNN}.
\end{proof}

\begin{exe}
\label{exsemidirect}
(1) 
Let $p$ be a prime.
The automorphism of multiplication by $p$
engulfs $\Q_p$ into $\Z_p$.
It follows from Corollary \ref{engulfsemidirect} that the group 
$\Q_p \rtimes_p \Z = \HNN (\Z_p, \Z_p, p\Z_p, \times p)$ is compactly presented.
\index{Semidirect product! $\Q_p \rtimes_p \Z$}
\par

More generally, for $\lambda \in \Q_p^\times$,
let $G_\lambda = \Q_p \rtimes_\lambda \Z$ be the semidirect product
for which $\Z$ acts on $\Q_p$ by $(n,x) \longmapsto \lambda^n x$.
If $\lambda$ is not a unit in $\Z_p$, 
the multiplication by $\lambda$ or $\lambda^{-1}$ 
engulfs $\Q_p$ into $\Z_p$, hence $G_\lambda$ is compactly presented.
If $\lambda \in \Z_p^\times = \Z_p \smallsetminus p\Z_p$ is a unit in $\Z_p$, 
then $G_\lambda$ is not compactly generated,
because $G_\lambda = \bigcup_{n \ge 0} (p^{-n} \Z_p \rtimes_\lambda \Z)$
is a union of a strictly increasing sequence of compact open subgroups.
\index{Compactly generated! LC-group, not compactly generated}

\vskip.2cm

(2)
The $p$-adic affine group $\Q_p \rtimes \Q_p^\times$
is compactly presented. 
\index{Affine group! $\K \rtimes {\K}^\times$}
Indeed, since $\{p^n \in \Q_p^\times \mid n \in \Z \} \simeq \Z$ 
is a cocompact lattice in $\Q_p^\times$,
the group $G_p =  \Q_p \rtimes_p \Z$ is closed and cocompact 
in $\Q_p \rtimes \Q_p^\times$;
the conclusion follows from Corollary \ref{hereditaritycp}.
\index{Lattice! in an LC-group} 

\vskip.2cm

(3)
For every positive integer $m$,
the automorphism of multiplication by $m$
engulfs $\Z [1/m]$ into $\Z$,
in accordance with the fact that the solvable Baumslag-Solitar group
\index{Baumslag-Solitar group} \index{Solvable group}
\begin{equation*}
\operatorname{BS}(1,m) \, = \,  \HNN(\Z, \Z, m\Z, \times m) \, = \,  \Z [1/m] \rtimes_m \Z
\end{equation*}
is finitely presented.
\end{exe}

We can rephrase the Bieri-Strebel splitting theorem in terms of engulfing automorphisms:

\begin{prop}
\label{BSrephrEngulf}
Let $N$ be an LC-group, $\varphi$ a topological automorphism of $N$,
and $G = N \rtimes_\varphi \Z$ the corresponding semidirect product.
Assume that $G$ is compactly presented. Then one of the following holds:
\begin{enumerate}[noitemsep,label=(\arabic*)]
\item\label{1DEBSrephrEngulf}
$\varphi$ or $\varphi^{-1}$ engulfs $N$ into some compactly generated open subgroup of $N$.
\item\label{2DEBSrephrEngulf}
The projection $\pi : G \twoheadrightarrow \Z$ has an essential splitting
over some compactly generated open subgroup;
in particular, $G$ has a non-abelian free discrete subgroup.
\end{enumerate}
\end{prop}
\index{Free subgroup}

\begin{proof}
By the Bieri-Strebel theorem, there exists a splitting over
a compactly generated open subgroup.
If it is ascending,
Corollary \ref{engulfsemidirect} implies that \ref{1DEBSrephrEngulf} holds.
Otherwise, the splitting is essential;
see Remark \ref{Alperin}(4) for the additional statement in \ref{2DEBSrephrEngulf}.
\end{proof}

\begin{cor}
\label{cpfor(locell)semidirec(Z)}
Let $N$ be an LC-group,
\index{Locally elliptic LC-group}
$\varphi$ a topological automorphism of $N$,
and $G = N \rtimes_\varphi \Z$ the corresponding semidirect product.
Consider the following two conditions:
\begin{enumerate}[noitemsep,label=(\roman*)]
\item\label{iDEcpfor(locell)semidirec(Z)}
$G$ is compactly presented;
\item\label{iiDEcpfor(locell)semidirec(Z)}
$\varphi$ or $\varphi^{-1}$ engulfs $N$ into a compact open subgroup of $N$.
\end{enumerate}
Then (ii) implies (i).
\vskip.2cm
If $N$ is locally elliptic, then (i) implies (ii).
\end{cor}

\begin{proof}
That Condition \ref{iiDEcpfor(locell)semidirec(Z)} implies 
Condition \ref{iDEcpfor(locell)semidirec(Z)}
is part of Corollary 8.C.16.
\par

Conversely, suppose that $N$ is locally elliptic and that
Condition \ref{iDEcpfor(locell)semidirec(Z)} holds.
Then \ref{1DEBSrephrEngulf} or \ref{2DEBSrephrEngulf} 
of Proposition \ref{BSrephrEngulf} holds.
\par

Suppose by contradiction that \ref{2DEBSrephrEngulf} holds,
so that $G$ has a non-abelian free discrete subgroup $F$.
Upon replacing $F$ by its derived subgroup,
we can assume that $F \subset N$;
this is preposterous, by Example \ref{examplesoflocallyellipticgroups}(3).
\par

It follows that Condition \ref{iiDEcpfor(locell)semidirec(Z)} holds.
\end{proof}

\begin{prop}
\label{suffpournotengulfing} 
Given an LC-group $N$ 
and a topological automorphism $\varphi$ of $N$,
each of the following conditions implies that 
neither $\varphi$ nor $\varphi^{-1}$
engulfs $N$ into some compactly generated open subgroup of $N$.
\begin{enumerate}[noitemsep,label=(\alph*)]
\item\label{aDEsuffpournotengulfing} 
There exists $x \in N$ such that,
for every compactly generated open subgroup $H$ of $N$,
the sets $\{n \in \N \mid \varphi^n(x) \notin H \}$ 
and $\{n \in \N \mid \varphi^{-n}(x) \notin H \}$
are both unbounded.
\item\label{bDEsuffpournotengulfing} 
There exists $x \in N$ such that,
for every compactly generated open subgroup $H$ of $N$,
the set $\{n \in \Z \mid \varphi^n (x) \in H \}$ is finite.
\item\label{cDEsuffpournotengulfing} 
The group of fixed points $N^\varphi$ is not capped,
i.e., is not contained in any compactly generated subgroup of $N$.
\index{Capped subset of an LC-group}
\item\label{dDEsuffpournotengulfing}
The group $N$ is locally elliptic, not compact, 
and the automorphism $\varphi$ preserves the Haar measure.
\end{enumerate}
\end{prop}

\begin{proof}
Let $x$ be as in \ref{aDEsuffpournotengulfing}.
If $\varphi$ or $\varphi^{-1}$
were to engulf $N$ into a compact open subgroup $H$ of $N$,
we would have $\varphi^n(x) \in H$ or $\varphi^{-n}(x) \in H$
for all $n$ large enough, 
i.e., one of
$\{ y \in N \mid y = \varphi^n (x) \hskip.2cm \text{for some} \hskip.2cm n \in \N \}$ 
and $\{ z \in N \mid z = \varphi^n (x) \hskip.2cm \text{for some} \hskip.2cm n \in \N \}$
would be bounded.
This is impossible if \ref{aDEsuffpournotengulfing} holds.
\par

Condition \ref{bDEsuffpournotengulfing} implies Condition  \ref{aDEsuffpournotengulfing}.
\par

Similarly, if one of $\varphi$, $\varphi^{-1}$ engulfs $N$ into 
a compact open subgroup $H$ of $N$,
then $N^\varphi \subset H$.
This is impossible when \ref{cDEsuffpournotengulfing} holds.

\vskip.2cm

Suppose now that $N$ is locally elliptic;
in particular, $N$ is unimodular.
As a comment on the statement of \ref{dDEsuffpournotengulfing},
and not as a step in the present proof,
observe that $\varphi$ preserves the Haar measure of $N$ 
if and only if the group $G := N \rtimes_\varphi \Z$ is unimodular
(see for example \cite[Chap.\ VIII, \S~2, no 10]{BInt7-8}).
 \index{Unimodular group}

\vskip.2cm

Suppose that \ref{dDEsuffpournotengulfing} holds.
Suppose also, by contradiction, that $\varphi$ or $\varphi^{-1}$
engulfs $N$ into some compactly generated open subgroup $H$ of $N$.
By Proposition \ref{8C10.5du21oct}, $G$ is isomorphic to an HNN-extension 
$\HNN (H, H, L, \tau)$, where $\tau$ is the restriction of $\varphi$ to $H$.
Since the modulus of $\varphi$ is $1$
we have $L = \tau (H) = \varphi(H) = H$.
As $N = \bigcup_{n \ge 0} \tau^{-n} (H)$, it follows that $N = H$,
in particular that $N$ is compactly generated.
Since $N$ is locally elliptic, $N$ is compact,
in contradiction with part of \ref{dDEsuffpournotengulfing}.
This ends the proof.
\end{proof}

\section[Further examples]
{Further examples}
\label{sectionsemidirect}

We now apply the Bieri-Strebel splitting theorem to various semidirect products.

\subsection{Semidirect products with $\Z$}
\label{subsection_8Da}
\index{Semidirect product}

\begin{exe}
(1)
Let $\operatorname{Alt}_f(\Z)$ denote 
the group of permutations of finite support of $\Z$ that are even on their support;
it is an infinite locally finite group.
Consider the group of permutations of $\Z$ 
generated by $\operatorname{Alt}_f(\Z)$
and the shift $\sigma : \Z \ni j \longmapsto  j+1 \in \Z$.
This group is the semidirect product
$\operatorname{Alt}_f(\Z) \rtimes_{\operatorname{shift}} \Z$,
with respect to the action of $\Z$ on $\operatorname{Alt}_f(\Z)$
for which $1 \in \Z$ acts by conjugation by $\sigma$.
\index{Semidirect product! $\operatorname{Alt}_f(\Z) \rtimes_{\operatorname{shift}} \Z$}
Observe that $\operatorname{Alt}_f(\Z) \rtimes_{\operatorname{shift}} \Z$ 
is generated by $\sigma$ and the $3$-cycle $(-1, 0, 1)$.
\par

It follows from Corollary \ref{cpfor(locell)semidirec(Z)} that 
$\operatorname{Alt}_f(\Z) \rtimes_{\operatorname{shift}} \Z$
is not finitely presented.

\vskip.2cm

(2)
In 1937, B.H.\ Neumann defined an uncountable family of pairwise non-isomorphic 
groups, all of them having a set of two generators;
since the number of finitely presented groups is countable,
he showed this way that there are finitely generated groups
that are not finitely presented.
\index{Finitely generated group}
\index{Finitely presented group}
In fact, none of the Neumann's groups is finitely presented.
\par

Indeed, from \cite{Neum--37}
(see also Complement III.35 in \cite{Harp--00}),
it is straightforward to deduce that each of these groups, say $G_{\text{Neum}}$,
has an infinite locally finite normal subgroup $N$ such that $G_{\text{Neum}}/N \simeq \Z$.
Hence $G_{\text{Neum}}$ is not finitely presented.
More generally:
\begin{center}
\emph{an (infinite locally finite)-by-$\Z$ group is not finitely presented:}
\end{center}
\index{Locally finite! group}
this follows from the Bieri-Strebel splitting theorem,
or more precisely from Proposition \ref{suffpournotengulfing} 
(see Condition \ref{dDEsuffpournotengulfing})
and Corollary \ref{cpfor(locell)semidirec(Z)}.
\end{exe}

\begin{exe}
\label{abcdefg}
Consider an integer $k \ge 1$, primes $p_1, \hdots, p_k$, 
and a semidirect product 
\begin{equation*}
G \, = \,  (\R \times \Q_{p_1} \times \cdots \times \Q_{p_k}) 
\rtimes_{(\lambda_\infty, \lambda_1, \hdots, \lambda_k)} \Z ,
\end{equation*}
for a diagonal action of $\Z$ 
by which $1$ acts on $\R$ by multiplication by some $\lambda_\infty \in \R^\times$,
and on the $j$th factor $\Q_{p_j}$
by multiplication by some $\lambda_j \in \Q_{p_j}^\times$, for $j = 1, \hdots, k$.
The group
\begin{equation*}
H \, = \,  (\Q_{p_1} \times \cdots \times \Q_{p_k}) 
\rtimes_{(\lambda_1, \hdots, \lambda_k)} \Z 
\end{equation*}
can be seen both as a closed subgroup of $G$,
and as the quotient of $G$ by its connected component $G_0 \simeq \R$.

\vskip.2cm

(1) 
To describe when $H$ is compactly generated or compactly presented,
we distinguish three cases:
\begin{enumerate}[noitemsep,label=(\alph*)]
\item\label{aDEabcdefg}
If $\vert \lambda_j \vert_{p_j} = 1$ for at least one $j \in \{1, \hdots, k\}$, 
then $H$ is not compactly generated.
Indeed, $H$ has a factor group of the form $H_j := \Q_{p_j} \rtimes_{\lambda_j} \Z$
and $H_j$ is not compactly generated
(Example \ref{exsemidirect}).
\item\label{bDEabcdefg}
If the $\vert \lambda_j \vert_{p_j}$ are all $\ne 1$, 
with at least one larger and one smaller than $1$,
then $H$ is compactly generated, and not compactly presented,
by Condition \ref{bDEsuffpournotengulfing} of Proposition \ref{suffpournotengulfing}
and Corollary \ref{cpfor(locell)semidirec(Z)}.
\index{Compactly presented! LC-group, not compactly presented}
\item\label{cDEabcdefg}
If the moduli $\vert \lambda_j \vert_{p_j}$ are either all larger than $1$ or all smaller than $1$,
then $G$ is compactly presented, by Corollary \ref{engulfsemidirect}.
\end{enumerate}
For example, the group
\begin{equation*}
(\Q_2 \times \Q_3) \rtimes_{(2/3)} \Z ,
\end{equation*}
for the diagonal action by which $1 \in \Z$ acts by multiplication by $\frac{2}{3}$
on each factor, is compactly generated and is not compactly presented,
because $\vert 2/3 \vert_2 < 1 < \vert 2/3 \vert_3$.
In contrast, the group
\begin{equation*}
(\Q_2 \times \Q_3) \rtimes_{(1/6)} \Z 
\end{equation*}
is compactly presented.

\vskip.2cm

(2)
The group $G$ is compactly generated if and only if $H$ is so
(Proposition \ref{almostconnectedgroups}),
and compactly presented if and only if $H$ is so
(Proposition \ref{GcpintermsofG0}),
because of the isomorphism $G/G_0 \simeq H$ noted above.
\par

Assume from now on that 
$\lambda_\infty = \lambda_1 = \cdots = \lambda_k$ is an invertible element
in the ring $\Z \left[ \frac{1}{p_1 \cdots p_k} \right]$.
\index{$z$@$\Z[1/n]$}
We write $\lambda$ for this common value, 
and $\rtimes_\lambda$ for $\rtimes_{(\lambda, \lambda, \cdots, \lambda)}$.
The diagonal embedding
\begin{equation*}
\Z \left[ \frac{1}{p_1 \cdots p_k} \right] \, 
\lhook\joinrel\relbar\joinrel\rightarrow
 \R \times \Q_{p_1} \times \cdots \times \Q_{p_k}
\end{equation*}
makes the left-hand side a cocompact lattice in the right-hand side
(see Example \ref{exlattcesinlocal}(2)).
\index{Lattice! in an LC-group} 
The group 
\begin{equation*}
\Gamma_\lambda \, := \, 
\Z \left[ \frac{1}{p_1 \cdots p_k} \right]  \rtimes_\lambda \Z 
\end{equation*}
is a cocompact lattice in $G$.
It follows from the considerations above that $\Gamma_\lambda$ is
\begin{enumerate}[noitemsep,label=(\alph*$'$)]
\item\label{a'DELipschitzbis}
not finitely generated if 
$\vert \lambda \vert_{p_j} = 1$ for at least one $j \in \{1, \hdots, k\}$,
\item\label{b'DELipschitzbis}
finitely generated $\vert \lambda \vert_{p_j} \ne 1$ for all $j \in \{1, \hdots, k\}$,
\item\label{c'DELipschitzbis}
finitely presented if either $\vert \lambda \vert_{p_j} < 1$ for all $j \in \{1, \hdots, k\}$,
or $\vert \lambda \vert_{p_j} > 1$ for all $j \in \{1, \hdots, k\}$.
\end{enumerate}
Note that, when the conditions of \ref{b'DELipschitzbis} hold,
the conditions of \ref{c'DELipschitzbis} hold if and only if
either $\lambda$ or $\lambda^{-1}$ is in $\Z$.

In particular:
\begin{enumerate}[noitemsep]
\item[$\bullet$]
$\Gamma_{2/3} = \Z[1/6] \rtimes_{2/3} \Z$
is not finitely presented,
\item[$\bullet$]
$\Gamma_{1/6} = \Z[1/6]  \rtimes_{1/6} \Z$
is finitely presented.
\end{enumerate}
\index{Semidirect product! $\Z[1/6] \rtimes_{2/3} \Z$}
\index{Semidirect product! $\Z[2/3] \rtimes_{2/3} \Z$}
This spectacular difference of behavior of two superficially similar-looking groups
was an important initial motivation of Bieri and Strebel (later joined by Neumann and others) 
for the work that led to their splitting theorem \cite{BiSt--78},
and later to the $\Sigma$-invariants; 
see \cite{Bier--79}.
\par

Note that $\Gamma_{1/6}$ is the solvable Baumslag-Solitar group $\operatorname{BS}(1,6)$.
The group $\Gamma_{2/3}$ can be described as the metabelianization
of the Baumslag-Solitar group $\operatorname{BS}(2,3)$.
The \textbf{metabelianization} of a group $G$ is its largest metabelian quotient,
i.e., the group  $G / [[G,G], [G,G]]$;
details in \cite[Appendix B]{BeGH--13}.
\index{Metabelianization of a group}

\index{Baumslag-Solitar group}
\end{exe}

\begin{exe}
\label{ExempleUnimodulaireMarrant}
Let $n \ge 2$; denote by $C_n$ the cyclic group of order $n$,
viewed here as a finite ring.
Recall that the ring $C_n \lp t \rp$  of formal Laurent series in $t$ with coefficients in $C_n$
is a totally disconnected LC-ring (see Example \ref{panoramalocalfield}).
The ring of Laurent polynomials $C_n [t,t^{-1}]$
embeds into the locally compact direct product
$C_n \lp t \rp \times C_n \lp t^{-1} \rp$,
by 
\begin{equation*}
\left\{
\aligned
C_n [t,t^{-1}] \,  & \lhook\joinrel\relbar\joinrel\rightarrow \, 
C_n \lp t \rp \times C_n \lp t^{-1} \rp 
\\
f(t) \hskip.5cm & \longmapsto \hskip.5cm  (f(t),f(t)) ,
\endaligned
\right.
\end{equation*}
as a discrete cocompact subring.
\par

Let the group $\Z$ act on the ring of Laurent polynomials by multiplication by $t$, 
and on the product ring diagonally,
by multiplication by $t$ on each factor.
Then the embedding above is $\Z$-equivariant.
This yields an embedding
\begin{equation*}
\Gamma := C_n [t,t^{-1}] \rtimes_t \Z \, \subset \,
G := \left( C_n \lp t \rp \times C_n \lp t^{-1} \rp \right) \rtimes_t \Z ,
\end{equation*}
for which $\Gamma$ is a uniform lattice in $G$.
\par

The group $\Gamma$ is not finitely presented,
and the locally compact group $G$ is not compactly presented;
both these claims follow again from
Condition \ref{aDEsuffpournotengulfing} of Proposition \ref{suffpournotengulfing}
and Corollary \ref{cpfor(locell)semidirec(Z)}.
That $\Gamma$ is not finitely presented follows also from its wreath product structure
$\Gamma \simeq \, (\Z / n\Z) \wr \Z$; see \cite{Baum--61}.
\end{exe}

\subsection{More general semidirect products, and $\SL_n(\K)$}
\label{subsection_8Db}

\begin{prop}
\label{K3semidirectSL2}
Let $\K$ be a local field.
The group $\K^3 \rtimes \SL_2(\K)$ is not compactly presented,
\index{Special linear group $\SL$! $ \K^3 \rtimes \SL_2(\K)$}
where $\K^3$ stands for the second symmetric tensor power of $\K^2$,
and the action of $\SL_2(\K)$ on $\K^3$ is the second symmetric tensor power
of its natural action on $\K^2$.
\end{prop}
\index{Compactly presented! LC-group, not compactly presented}

\begin{proof}
Let $\{e_1, e_2\}$ be the canonical basis of $\K^2$,
and $\{e_1 \odot e_1, e_1 \odot e_2, e_2 \odot e_2 \}$ the corresponding basis of $\K^3$.
Let $g =
\begin{pmatrix} a & b  \\ c & d \end{pmatrix}
\in \SL_2(\K)$.
We have
\begin{equation*}
g ( e_1 \odot e_1 ) \, = \, (ae_1 + ce_2) \odot (ae_1 + ce_2) \, = \,
a^2 e_1 \odot e_1 + 2ac e_1 \odot e_2 + c^2 e_2 \odot e_2 ,
\end{equation*}
and similarly for $g ( e_1 \odot e_2 ), g ( e_2 \odot e_2 )$.
It follows that $g$ acts on $\K^3$ by the matrix
\begin{equation*}
S^2(g)
\, = \, 
\begin{pmatrix} 
a^2 & ab & b^2  \\  2ac & ad+bc & 2bd \\  c^2 & cd & d^2 
\end{pmatrix} .
\end{equation*}
Let us identify $\K^3 \rtimes \SL_2(\K)$ with its image by the homomorphism
\begin{equation*}
\left\{ \aligned
\K^3 \rtimes \SL_2(\K) \, &\longrightarrow \hskip.5cm  \GL_4(\K)
\\
(\xi, g) \hskip.8cm &\longmapsto \,
\begin{pmatrix} 
S^2(g) & \xi  \\  0 & 1
\end{pmatrix} .
\endaligned \right.
\end{equation*}
\par  

Let $\operatorname{ST}_2(\K)$ denote the closed subgroup of $\SL_2(\K)$
of matrices of the form
$\begin{pmatrix}  a & b  \\  0 & d \end{pmatrix}$
with $a,b,d \in \K$ and $ad = 1$,
and $\operatorname{Unip}_2(\K)$ the closed subgroup of $\operatorname{ST}_2(\K)$
of matrices of the form
$\begin{pmatrix}  1 & b  \\  0 & 1 \end{pmatrix}$
with $b \in \K$.
Since $\operatorname{ST}_2(\K)$ is an isotropy group for the natural action
of $\SL_2(\K)$ on the projective line, 
the subgroup $\operatorname{ST}_2(\K)$ is cocompact in $\SL_2(\K)$.
Fix $\lambda \in \K$ with $\vert \lambda \vert > 1$,  
set $s = \begin{pmatrix} \lambda & 0 \\  0 & \lambda^{-1} \end{pmatrix} \in \SL_2(\K)$,
and let $\sigma$ be the automorphism of $\operatorname{Unip}_2(\K)$
of conjugation by $s$.
We identify the corresponding semidirect product
$\operatorname{Unip}_2(\K) \rtimes_\sigma \Z$
with the closed subgroup of $\operatorname{ST}_2(\K)$ of matrices of the form
$\begin{pmatrix} \lambda^n & x \\  0 & \lambda^{-n} \end{pmatrix}$,
with $n \in \Z$ and $x \in \K$.
Since $\lambda^{\Z}$ is a closed cocompact subgroup of $\K^\times$, 
the subgroup $\operatorname{Unip}_2(\K) \rtimes_\sigma \Z$ is cocompact
in $\operatorname{ST}_2(\K)$, and therefore also in $\SL_2(\K)$.
\par

Let $N$ denote the closed subgroup of $\K^3 \rtimes \SL_2(\K)$
consisting of elements $(\xi, u)$ 
with $\xi \in \K^3$ and $u \in \operatorname{Unip}_2(\K)$.
We denote again by $\sigma$ the automorphism of $N$ 
of conjugation by $(0,s)$.
As above, we identify the semidirect product $N \rtimes_\sigma \Z$
with a cocompact closed subgroup of $\K^3 \rtimes \SL_2(\K)$.
By Corollary \ref{hereditaritycp}, 
it suffices to check that $N \rtimes_\sigma \Z$  is not compactly presented.
\par

On the one hand, since $\K$ is a local field,
the group $N$ is locally elliptic.
On the other hand, since $\sigma$ acts on the closed subgroup $\K^3$ of $N$
as the matrix
$\begin{pmatrix} 
\lambda^2 & 0 & 0  \\  0 & 1 & 0 \\  0 & 0 & \lambda^{-2} 
\end{pmatrix}$,
the subgroup $N^\sigma$ of $N$ of elements fixed by $\sigma$
contains a non-compact subgroup $\{0\} \times \K \times \{0\} \simeq \K$.
It follows from Proposition \ref{suffpournotengulfing} 
(Condition \ref{cDEsuffpournotengulfing})
that neither $\sigma$ nor $\sigma^{-1}$ engulfs $N$
into some compact open subgroup of $N$,
hence from Corollary \ref{cpfor(locell)semidirec(Z)}
that $N$ is not compactly presented.
\end{proof}

\begin{rem}
For $n \ge 0$, consider more generally the group
$G_n = \K^{n+1} \rtimes \SL_2(\K)$,
where $\K^{n+1}$ is identified with the $n$th symmetric power of $\K^2$.
Then $G_n$ is not compactly presented.
When $n$ is even, the method of Proposition \ref{K3semidirectSL2} 
can be adapted to $G_n$.
When $n$ is odd, this does not work, 
but there is an argument based on central extensions,
as already remarked in \ref{remsurexcgnotcp}.
\end{rem}

Our next target is Proposition \ref{cpet0notinintervalles},
for the proof of which we need two preliminary lemmas.

\begin{lem}
\label{AutransEx9.A.1}
Let $H$ be an LC-group, $n \ge 2$ an integer,
and $N_1, \hdots, N_n$ LC-groups,
each one given with a continuous action of $H$ by topological automorphisms.
Let $G = \Big( \prod_{i=1}^n N_i \Big) \rtimes H$
be the semidirect product with respect to the diagonal action of $H$
on the product of the $N_i$ 's.
For $i,j \in \{1, \hdots, n\}$ with $i \ne j$,
let $G_{i,j} = (N_i \times N_j) \rtimes H$ denote the natural semidirect product,
identified with a closed subgroup of $G$.
\par

Assume that $G_{i,j}$ is compactly presented for each pair $(i,j)$ with $i \ne j$.
Then $G$ is compactly presented.
\end{lem}

\noindent \emph{Note.}
Conversely, if $G$ is compactly presented, so is every $G_{i,j}$,
because they are group retracts of $G$.

\vskip.2cm

In the proofs of Lemma \ref{AutransEx9.A.1} and Proposition \ref{cpsemidirectproduct},
we freely use the following immediate fact:\
an LC-group $G$ is compactly presented 
if and only if there exist a boundedly presented group $G_1 = \langle S \mid R\rangle$ 
and an isomorphism $\varphi : G_1 \longrightarrow G$ 
such that the image $\varphi(S)$ is a compact subset of $G$.

\begin{proof}
Let $T$ be a compact generating subset of $H$.
For every $i \in \{1, \hdots, n\}$, 
let $S_i$ be a compact symmetric subset of $N_i$
such that $S_i \cup T$ generates $N_i \rtimes H$.
For every pair $(i,j)$ of distinct elements of $\{1, \dots, n\}$,
consider a copy of $G_{i,j}$, 
denoted by $G^{[i,j]} = (N_i^{[i,j]} \times N_j^{[i,j]}) \rtimes H^{[i,j]}$. 
Let $S^{[i,j]}$ be the subset $S_i^{[i,j]} \cup S_j^{[i,j]}$ of  $G^{[i,j]}$
corresponding to $S_i \cup S_j \subset G_{i,j}$; 
let $T^{[i,j]}$ be the compact generating subset of $H^{[i,j]}$ 
corresponding to $T$ in $H$.
Clearly, $S^{[i,j]} \cup T^{[i,j]}$ is a compact generating subset of $G^{[i,j]}$.
\par

We consider the free product $G_0$ of the $G^{[i,j]}$ over all $n^2-n$ pairs $(i,j)$. 
Since by assumption $G_{i,j}$ is compactly presented, 
$G^{[i,j]}$ is boundedly presented over $S^{[i,j]} \cup T^{[i,j]}$.
It follows immediately that $G_0$ is boundedly presented 
by the generating subset $U := \bigcup \left( S^{[i,j]} \cup T^{[i,j]} \right)$.
The natural embeddings $G_{i,j} \subset G$ 
canonically define a surjective homomorphism
 $p_0 : G_0 \twoheadrightarrow G$.
Observe that $p_0(U)$ is compact and, since $n \ge 2$,  that $p_0$ is surjective.
\par

Define a quotient $G_1$ of $G_0$ by elements of length 2 with respect to $U$, 
by identifying all possible elements, namely
\begin{enumerate}[noitemsep]
\item[$\bullet$]
for all distinct pairs $(i,j)$ and $(i',j')$, all $k \in \{i,j\} \cap \{i',j'\}$, and every $s\in S_k$, 
we identify the corresponding elements 
$s^{[i,j]} \in S_k^{[i,j]}$ and $s^{[i',j']} \in S_k^{[i',j']}$~;
\item[$\bullet$]
for all distinct pairs $(i,j)$ and $(i',j')$ and every $t \in T$, 
we identify the corresponding elements 
$t^{[i,j]} \in T^{[i,j]}$ and $t^{[i',j']} \in T^{[i',j']}$.
\end{enumerate}
Obviously, $G_1$ has a bounded presentation,
and $p_0$ factors through a surjective homomorphism 
\begin{equation*}
p_1 \,  : \,  G_1 \twoheadrightarrow G .
\end{equation*}
Denote by $\rho$ the projection $G_0 \twoheadrightarrow G_1$;
we have $p_0 = p_1 \circ \rho$.
\par

For every $k \in \{1, \hdots, n\}$, the subset of $G_1$ defined as $\rho(S_k^{[i,j]})$ 
does not depend on the choice of the pair $(i,j)$ such that $k\in\{i,j\}$ 
(there are $2n-2$ such pairs), 
and we denote it by $\widehat {S}_k \subset G_1$. 
Also, the subset of $G_1$ defined as $\rho(T^{[i,j]})$ 
does not depend on the choice of the pair $(i,j)$ 
and we denote it by $\widehat {T}$.
Since $G_0$ is generated by the union 
of all possible $S_k^{[i,j]}$ (with $k \in\{i,j\}$) and $T^{[i,j]}$, 
it follows that $G_1$ is generated by the union 
of all $\widehat {S}_i$ (with $i \in \{1, \hdots, n\}$) and $\widehat {T}$.
\par

Let $\widehat {H}$ be the subgroup of $G_1$ generated by $\widehat {T}$. 
Define an $i$-elementary element of $G_1$ 
as an element of the form $h s_i h^{-1}$ with $h \in \widehat {H}$ and $s \in \widehat S_i$, 
and an elementary element as an $i$-elementary element for some $i$. 
Since $G_1$ is generated by  the symmetric subset 
$\big( \bigcup \widehat S_i \big) \cup \widehat {H}$,
every element of $G_1$ can be written as a product $uh$ 
where  $u$ is a product of elementary elements and $h \in \widehat {H}$.
\par

We have $[h_i \widehat {S}_i h_i^{-1}, h_j \widehat {S}_j h_j^{-1}] = \{1\} \subset G_1$ 
for all distinct $i,j$ and $h_i, h_j \in \widehat {H}$.
Indeed, picking lifts $\tilde{h}_i$ and $\tilde{h}_j$ of $h_i$ and $h_j$ in $G^{[i,j]}$
\begin{equation*}
\aligned
\left[ h_i \widehat {S}_i h_i^{-1}, h_j \widehat {S}_j h_j^{-1} \right] 
\, & = \,  
\left[ \rho(\tilde{h}_i S_i^{[i,j]} \tilde{h}_i^{-1}), \rho(\tilde{h}_j S_j^{[i,j]} \tilde{h}_j^{-1}) \right]
\\
\,  & = \,  
\rho\left( [\tilde{h}_i S_i^{[i,j]} \tilde{h}_i^{-1}, \tilde{h}_j S_j^{[i,j]} \tilde{h}_j^{-1}] \right) 
\, = \,
\rho(\{1\}) 
\, = \, 
\{1\}.
\endaligned
\end{equation*}
It follows that any element $g$ in $G_1$ can be written as $u_1\dots u_kh$ where each $u_i$ is a product of $i$-elementary elements and $h \in \widehat {H}$.

Now suppose that $g \in \ker(p_1)$; let us show that $g=1$. 
For all $i$, we have $p_1(u_i) \in N_i$ and $p_1(h) \in H$. 
It follows that $p_1(u_i) = 1$ and $p_1(h) = 1$.
Pick $j$ distinct from $i$, and lifts $\tilde{u}_i, \tilde{h}$ of $u_i, h$ in $G^{[i,j]}$;
since the restriction of $p_0 = p_1 \circ \rho$ to $G^{[i,j]}$ is injective,
we have $\tilde{u}_i = 1$ and $\tilde{h} = 1$.
Hence $u_i$ and $h$ are trivial. This proves that $p_1$ is an isomorphism.
As $G_1$ is boundedly presented, $G$ is compactly presented.
\end{proof}

\begin{prop}
\label{cpsemidirectproduct}
Let $A, N$ be two LC-groups, with $A$ compactly generated abelian, 
$\varphi : A \longrightarrow \operatorname{Aut}(N)$ a continuous action of $A$ on $N$,
and $G = N \rtimes_\varphi A$ the corresponding semidirect product.
Assume there exist $a_0 \in A$ 
and a compactly presented open subgroup $L \subset N$
such that $\varphi(a_0)$ engulfs $N$ into $L$.
\par

Then $G$ is compactly presented.
\end{prop}

\begin{proof}
Let $T$ be a symmetric compact generating subset of $A$ containing $a_0$. 
Define $S = \bigcup_{t \in T} t L t^{-1}$; 
note that $S$ is a compact  symmetric subset of $N$.
The semidirect product $N \rtimes_{a_0} \Z$ is generated by $S \cup \{\sigma^{\pm 1}\}$, 
where $\sigma$ denotes the positive generator of $\Z$.
\par

We consider the free product $G_0 = (N \rtimes_{a_0} \Z) \Conv A$, 
endowed with the generating subset $(S \cup \{\sigma^{\pm 1} \}) \cup T$. 
Since $N \rtimes_{a_0} \Z$ is compactly presented by 
Corollary \ref{engulfsemidirect}, 
and $A$ is compactly presented by Proposition \ref{abeliancg=cp}, 
the free product $G_0$ is boundedly presented over $(S \cup \{\sigma^{\pm 1}\}) \cup T$. 
The natural homomorphism $p_0 : G_0 \twoheadrightarrow G$ is obviously surjective, 
and its restriction to each of the two free factors is also injective.
\par

Let $G_1$ be the quotient of $G_0$ by
\begin{enumerate}[noitemsep]
\item[$\bullet$]
the relator of length $2$ identifying $\sigma \in \Z$ with $a_0 \in A$;
\item[$\bullet$]
the relators of length $4$ of the form $txt^{-1}y^{-1}$ for $t\in T$ and $x,y\in S$ 
whenever the corresponding relation holds in $G$.
\end{enumerate}
Obviously, $G_1$ has bounded presentation, 
and $p_0$ factors through a surjective homomorphism 
\begin{equation*}
p_1 \, : \, G_1 \twoheadrightarrow G  .
\end{equation*}
Denote by $\rho$ the projection $G_0 \twoheadrightarrow G_1$;
we have $p_0 = p_1 \circ \rho$.
\par

In $G_1$, the generator $\sigma$ is redundant;
since $G_0$ is generated by $(L \cup \{\sigma^{\pm 1} \}) \cup T$, 
the group $G_1$ is generated by $\rho(L) \cup \rho(T)$. 
Let $\widehat{A}$ be the subgroup of $G_1$ generated by $\rho(T)$. 
By the same argument as in the previous proof, 
$\widehat{A}$ is mapped by $p_1$ injectively into $G$ (actually, onto $A$).
\par

Define an elementary element in $G_1$ 
to be an element of the form $a s a^{-1}$ with $a \in \rho(L)$ and $a \in \widehat{A}$. 
Then, as in the previous proof, it is immediate that every element  $g \in G_1$ 
has the form $g = ua$ with $a \in \widehat{A}$ and $u$ a product of elementary elements.
\par

Write $\varphi$ for $\rho(\sigma)=\rho(a_0)$.
Now let us check that, in $G_1$, 
the set of elements of the form $\varphi^n \rho(x) \varphi^{-n}$, for $n\in\Z$ and $x\in L$, 
is stable under conjugation by $\widehat{A}$. 
It is enough to check that $\rho(t) \varphi^n \rho(x) \varphi^{-n} \rho(t)^{-1}$ 
has this form for every $t \in T$.
And, indeed, 
 \begin{equation*}
\rho(t) \varphi^n \rho(x) \varphi^{-n} \rho(t)^{-1} \, = \, 
\rho(ta_0^n) \rho(x) \rho(ta_0^n)^{-1} \, = \, 
\] \[
\rho(a_0^nt ) \rho(x) \rho(a_0^n t)^{-1} \, = \, 
\varphi^n \rho(t) \rho(x) \rho(t)^{-1} \varphi^{-n} ,
\end{equation*}
where the latter term is equal, by using the relators of the second kind, 
to some element of the form $\varphi^n \rho(y) \varphi^{-n}$ with $y \in S$. 
The engulfing property implies that we can write $y = \sigma^k x' \sigma^{-k}$ 
for some $x' \in L$ and $k \in \Z$. 
So
\begin{equation*}
\rho(t)\varphi^n\rho(x)\varphi^{-n}\rho(t)^{-1} \, = \,
\varphi^n\rho(\sigma^kx'\sigma^{-k})\varphi^{-n} \, = \,
\varphi^{n+k}\rho(x')\varphi^{-(n+k)}.
\end{equation*}
It follows that every elementary element belongs to the subgroup $\rho(N \rtimes_{a_0} \Z)$. 
\par

Let $g = uh \in G_1$ be as above. Suppose now that $g \in \ker (p_1)$,
and let us show that $g = 1$.
Because of the semidirect decomposition of $G$, both $u$ and $a$ belong to $\ker(p_1)$. 
Since $u$ belongs to $\rho(N \rtimes_{a_0} \Z)$ and $a$ belongs to $\rho(A)$, 
it follows that $u = 1$ and $a = 1$.
This proves that $p_1$ is an isomorphism.
\end{proof}

Let $A$ be a topological group.
The group $\operatorname{Hom}(A,\R)$ of continuous homomorphisms of $A$ to $\R$
is naturally a real vector space.
For $w,w' \in \operatorname{Hom}(A,\R)$, we denote by $[w,w']$
the straight line segment in $\operatorname{Hom}(A,\R)$
joining $w$ and $w'$;
note that $[w,w'] = \{w\}$ when $w'=w$.

\begin{prop}
\label{cpet0notinintervalles}
Consider a compactly generated LCA-group $A$, 
an integer $n \ge 2$, 
and local field $\K_1, \hdots, \K_n$.
For every $i \in \{1, \hdots, n\}$, 
let $\varphi_i : A \longrightarrow \K_i^\times$ be a continuous homomorphism;
let $\vert \cdot \vert_i$ on $\K_i$ be an absolute value on $\K_i$
(see Remark  \ref{onabsolutevalues}).
Define $w_i \in \operatorname{Hom}(A,\R)$
by $w_i (a) = \log \vert \varphi_i (a) \vert_i$ for all $a \in A$.
Set 
\begin{equation*}
G_\varphi \, = \, (\K_1 \times \cdots \times \K_n) \rtimes_\varphi A
\end{equation*}
where the semidirect product corresponds to the action of $A$ on $\prod_{i=1}^n \K_i$
by which $a \in A$ multiplies the $i$th factor of $\prod_{i=1}^n \K_i$ by $\varphi_i(a)$.
\par

Then we have:
\begin{enumerate}[noitemsep,label=(\arabic*)]
\item\label{1DEcpet0notinintervalles}
$G_\varphi$ is compactly generated if and only if 
$w_i \ne 0$ for all $i  \in \{1, \hdots, n\}$.
\item\label{2DEcpet0notinintervalles}
$G_\varphi$ is compactly presented if and only if 
$0 \notin [w_i, w_j] \subset \operatorname{Hom}(A, \R)$ for all $i,j \in \{1, \hdots, n\}$.
\end{enumerate}
\end{prop}

\noindent
\emph{Note.} The situation of the proposition with $n=1$ is covered by
Proposition \ref{cpsemidirectproduct}.

\begin{proof}   
As a preliminary observation, note that $G_\varphi$ does not contain
any non-abelian free group, because it is solvable.

\vskip.2cm

\ref{1DEcpet0notinintervalles}
If $w_i = 0$ for some $i$, that is if $\vert \varphi_i (a) \vert_i = 1$ for all $a \in A$,
then $\K_i \rtimes_{\varphi_i} A$ is a quotient of $G_\varphi$,
hence $G_\varphi$ is not compactly generated,
by the argument of Example \ref{exsemidirect}(2).
If $w_i \ne 0$ for all $i  \in \{1, \hdots, n\}$, 
the argument of Example \ref{affinecg} can be adapted to the present situation, 
to show that each of the subgroups $K_i \rtimes A$ is compactly generated.
Since $G_\varphi$ is generated by these subgroups,
it follows that $G_\varphi$ is compactly generated.

\vskip.2cm

\ref{2DEcpet0notinintervalles}
For $i \in \{1, \hdots, n\}$, 
we denote by $\overline w _i : G_\varphi \longrightarrow \R_+^\times$
the composition of the canonical projection $G_\varphi \longrightarrow A$ with $\varphi_i$.
The image of $\overline w_i$, which is that of $w_i$,
is a non-trivial discrete subgroup of $\R_+^\times$, because so is 
$\{ t \in \R_+^\times \mid t = \vert x \vert_i \hskip.2cm \text{for some} \hskip.2cm x \in \K^\times \}$.
Note that 
\begin{equation*}
\ker (\overline w_i) = (\K_1 \times \cdots \times \K_n) \rtimes_\varphi \ker (w_i) .
\end{equation*}
Since $\ker (w_i)$ acts trivially on $\K_i$,
the group $\K_i$ is a direct factor of $\ker (\overline w_i)$.

\vskip.2cm

Suppose first that $G_\varphi$ is compactly presented. 
Then $w_i \ne 0$ for all $i$, by \ref{1DEcpet0notinintervalles}.
We assume ab absurdo that $0 \in [w_i,w_j]$ for two distinct indices $i$ and $j$.
There exists a number $\mu > 0$ such that $w_j = - \mu w_i$;
in particular, $\ker (w_1) = \ker (w_2)$.
We will arrive at a contradiction.
\par

Choose $a \in A$ such that $w_i(a) > 1$ generates
the cyclic subgroup $\operatorname{Im}(w_i)$ of $\R_+^\times$.
Denote by $\Z_a$ the discrete cyclic subgroup of $A$ generated by $a$.
We have a direct sum decomposition $A = \ker (w_i) \times \Z_a$, so that
\begin{equation*}
G_\varphi \, = \, \ker (\overline w_i) \rtimes  \Z_a .
\end{equation*}
In view of the preliminary observation,
it follows from Proposition \ref{BSrephrEngulf} that 
one of $\alpha_a, \alpha_{a^{-1}}$,
from now on denoted by $\alpha$,
engulfs $\ker (\overline w_i )$ into a compactly generated subgroup of $G_\varphi$,
say $H$.
(The notation $\alpha_{a}$ is that of Proposition \ref{8C10.5du21oct}.)
\par

For an integer $\ell \ge 1$, denote by $H_\ell$ the subset of $\ker (\overline w_i)$
of elements $(x,b)$ such that $\vert x_i \vert_i \le \ell$ and $\vert x_j \vert_j \le \ell$.
It is clearly an open subset of $\ker (\overline w_i)$;
it is also a subgroup, because every $b \in \ker (w_i) = \ker (w_j)$
preserves the norms in $\K_i$ and $\K_j$. We have
\begin{equation*}
\ker (\overline w_i) \, = \, \bigcup_{\ell \ge 1} H_\ell
\hskip.5cm \text{(increasing union).}
\end{equation*}
It follows that there exists $m \ge 1$ such that $H \subset H_m$.
(We do not mind whether or not $H_m$ is compactly generated,
but we know that $H$ is.)
\par

Let $x \in \K_1 \times \cdots \times \K_n$
be such that $x_i \ne 0$ and $x_j \ne 0$.
The set of $k \in \Z$ such that $\alpha^k (x) \in H_m$ is finite; indeed
\begin{equation*}
\lim_{k \to \infty} \left\vert \left(\alpha^k (x) \right)_i  \right\vert_i \, = \, \infty
\hskip.5cm \text{and} \hskip.5cm
\lim_{k \to \infty} \left\vert \left(\alpha^{-k} (x) \right)_j  \right\vert_j \, = \, \infty
\end{equation*}
By Proposition \ref{suffpournotengulfing} (Condition \ref{aDEsuffpournotengulfing}),
this implies that $\alpha$ does not engulf $\ker (\overline w_i)$
into a compactly presented subgroup of $G_\varphi$,
in contradiction with an intermediate conclusion above.
\par

We have shown that, 
if $G_\varphi$ is compactly presented, then $0 \notin [w_i, w_j]$
or all $i,j \in \{1, \hdots, n\}$.

\vskip.2cm

Conversely, suppose now that $0 \notin [w_i, w_j]$ for all $i,j \in \{1, \hdots, n\}$.
By Lemma \ref{AutransEx9.A.1}, 
it is enough to choose $i,j \in \{1, \hdots, n\}$, $i \ne j$,
and to show that
$(\K_i \times \K_j) \rtimes A$
is compactly presented.
\par 

We claim that 
\begin{center}
\emph{there exists $a \in A$ such that $w_i(a) > 0$ and $w_j(a) > 0$.}
\end{center}
To justify the claim, we distinguish two cases.
Suppose first that one of $\ker (w_i), \ker (w_j)$ is contained in the other;
the inclusion cannot be proper, because $\Z$ is not a proper quotient of itself;
hence $\ker (w_i) = \ker (w_j)$.
Since $0 \notin [w_i,w_j]$,  there exists $\mu > 0$
such that $w_j =  \mu w_i$.
For $b \in A, b \notin \ker (w_i)$, 
the claim holds for one of $a=b$, $a=b^{-1}$.
Suppose now that there exist $b \in \ker(w_i) \smallsetminus \ker(w_j)$
and $c \in \ker(w_j) \smallsetminus \ker(w_i)$.
For appropriate choices of $\varepsilon, \delta \in \{-1,1\}$,
the claim holds for $a = b^\varepsilon c^\delta$.
\par

Let $a \in A$ be as in the claim. Then $a$ acts on $\K_i \times \K_j$ by the automorphism
\begin{equation*}
\sigma \, : \, (x_i, x_j) \longmapsto (\varphi_i(a) x_i , \varphi_j(a) x_j) ,
\end{equation*}
with $\log \vert \varphi_i (a) \vert_i > 0$ and $\log \vert \varphi_j (a) \vert_j > 0$.
Hence $\sigma$ engulfs $\K_i \times \K_j$ into $\mathfrak o_i \times \mathfrak o_j$,
where $\mathfrak o_i$ [respectively $\mathfrak o_j$]
is the maximal compact subring of $\K_i$, [resp.\ of $\K_j$].
It follows from Proposition \ref{cpsemidirectproduct} that 
$(\K_i \times \K_j) \rtimes A$ is compactly presented.
This ends the proof.
\end{proof}

\begin{lem}
\label{lemmafromicc}
Let $G$ be an LC-group which is an essential HNN-extension $G = \HNN (H, K, L, \varphi)$,
for two proper open subgroups $K, L$ of $H$ 
and a topological isomorphism $\varphi$ of $K$ onto $L$. 
\par

Every normal subgroup of $G$ without any non-abelian free discrete subgroup
is contained in $K$.
\end{lem}

\begin{proof}
See Proposition 8 in \cite{Corn--09}.
\end{proof}

\begin{prop}
\label{actionF2versusZ2}
Let the free group $F=\langle t,u\rangle$ act on $\Q_p$,
with $t$ the multiplication by $p$ and $u$ the identity automorphism. 

Then the semidirect product $G=\Q_p\rtimes F$ is  compactly generated
and not compactly presented.
\end{prop}

Note that if $F$ is replaced by $\Z^2$, 
the resulting group is compactly presented, by Proposition \ref{cpsemidirectproduct}.

\begin{proof}
For $s \in F$, $s \ne 1$, 
denote by $\langle s \rangle$ the cyclic subgroup of $F$ generated by $s$.
The group $G$ is generated by the compactly generated subgroups 
$\langle \Q_p,t \rangle = \Q_p \rtimes_p \langle t \rangle$ and $F$, 
and therefore is compactly generated.
\par

Let $M$ be the kernel of the unique homomorphism $F \twoheadrightarrow \Z$
mapping $t$ onto $1$ and $u$ onto $0$;
we can identify $M$ with the free product 
of the cyclic groups $\langle t^n u t^{-n} \rangle$ indexed by $n \in \Z$.
The subgroup of $G$ generated by $\Q_p$ and $M$ is a direct product $\Q_p \times M$.
We can write $G$ as the semidirect product $(\Q_p \times M) \rtimes  \langle t \rangle$.
With the notation of Corollary \ref{cpfor(locell)semidirec(Z)} in mind,
we set $N = \Q_p \times M$.
\par

Every compactly generated subgroup of $N$ is contained in a product of the form
$L \times (\Conv_{n = -\ell}^\ell \langle t^n u t^{-n} \rangle )$, where 
$L$ is a compact subgroup of $\Q_p$ 
and $\ell$ a positive integer. 
Hence the group $N$ and the automorphism of conjugation by $t$
satisfy Condition (a) of Proposition \ref{suffpournotengulfing}.
It follows from this proposition and Corollary \ref{cpfor(locell)semidirec(Z)}
that $G$ is not compactly presented.
\end{proof}

\begin{notat}
For an integer $n \ge 2$ and a locally compact field $\K$, we introduce the following 
closed subgroups of $\SL_n(\K)$:
\begin{equation*}
\text{the group} \hskip.2cm \operatorname{ST}_n(\K)
\hskip.2cm \text{of matrices of the form} \hskip.2cm
\begin{pmatrix}
x_{1,1} & x_{1,2}  & \cdots & x_{1,n}
\\
0 & x_{2,2}  & \cdots & x_{2,n}
\\
\vdots & \vdots   & \ddots & \vdots
\\
0 & 0 &  \cdots & x_{n,n}
\end{pmatrix},
\end{equation*}
\begin{equation*}
\text{the group} \hskip.2cm \operatorname{Unip}_n(\K) 
\hskip.2cm \text{of matrices of the form} \hskip.2cm
\begin{pmatrix}
1 & x_{1,2}  & \cdots & x_{1,n}
\\
0 & 1  & \cdots & x_{2,n}
\\
\vdots & \vdots   & \ddots & \vdots
\\
0 & 0 &  \cdots & 1
\end{pmatrix} ,
\end{equation*}
\begin{equation*}
\text{the group} \hskip.2cm \operatorname{SD}_n(\K) 
\hskip.2cm \text{of matrices of the form} \hskip.2cm
\begin{pmatrix}
x_{1,1} &  0 & \cdots & 0
\\
0 &  x_{2,2}  & \cdots & 0
\\
\vdots & \vdots  & \ddots & \vdots
\\
0 & 0 & \cdots & x_{n,n}
\end{pmatrix} .
\end{equation*}
Here, $x_{i,j} \in \K$ for $i,j  = \{ 1, \hdots, n\}$ with $i \le j$,
and $\prod_{i=1}^n x_{i,i} = 1$. 
The first of these is naturally a semidirect product of the other two:
\begin{equation*}
\operatorname{ST}_n(\K) \, = \, 
\operatorname{Unip}_n(\K) \rtimes \operatorname{SD}_n(\K) .
\end{equation*}
Note that $\operatorname{SD}_n(\K)$ is abelian, indeed isomorphic to $(\K^\times)^{n-1}$.
We know that  $\operatorname{SD}_n(\K)$ is compactly generated
(because $\K^\times$ is, see Exampe \ref{localfieldscg}),
and therefore compactly presented
(see Proposition \ref{abeliancg=cp}).
\end{notat}

\begin{thm}
\label{SLnKiscp}
Consider an integer $n \ge 2$ and a non-discrete locally compact field $\K$.
\begin{enumerate}[noitemsep,label=(\arabic*)]
\item\label{1DESLnKiscp}
The special triangular group $\operatorname{ST}_n(\K)$ is compactly presented.
\index{Triangular group}
\item\label{2DESLnKiscp}
The special linear group $\SL_n(\K)$ is compactly presented.
\end{enumerate}
\end{thm}

\begin{proof}
\emph{First case:} $\K$ is Archimedean.
These groups are connected,
hence both claims follow from Proposition \ref{connbycompcp}.
\par

\emph{Second case:} $\K$ is a local field.
Let $\vert \cdot \vert$ be an absolute value on $\K$ which makes it a complete field.
Recall that $\K$ has a maximal compact subring
\begin{equation*}
\mathfrak o_{\K} \, \, = \, 
\big\{ x \in \K    \hskip.1cm \big\vert \hskip.1cm     
\{x^n \mid n \ge 1 \}  \hskip.2cm \text{is relatively compact}  \big\} 
\, = \,
\big\{ x \in \K   \hskip.1cm \big\vert \hskip.1cm   \vert x \vert \le 1 \big\}
\end{equation*}
as in Example \ref{panoramalocalfield} and Remark \ref{onabsolutevalues}.
Denote by $\operatorname{Unip}_n(\mathfrak o_{\K})$ 
the compact open subgroup of $\operatorname{Unip}_n(\K)$
of matrices with coefficients in $\mathfrak o_{\K}$.
\par

Choose $\lambda \in \K$ with $\vert \lambda \vert < 1$.
The element 
\begin{equation*}
\operatorname{diag}(
\lambda^{n-1}, \lambda^{n-3}, \lambda^{n-5}, \hdots,
\lambda^{-n+1})
\, \in \,  \operatorname{SD}_n(\K)
\end{equation*}
engulfs $\operatorname{Unip}_n(\K)$ into the 
compact open group $\operatorname{Unip}_n( \mathfrak o_{\K} )$.
It follows from Proposition \ref{cpsemidirectproduct} that 
$\operatorname{ST}_n(\K)$ is compactly presented.
\par

Since $\operatorname{ST}_n(\K)$ is cocompact in $\SL_n(\K)$,
Claim \ref{2DESLnKiscp} follows from \ref{1DESLnKiscp} 
by Corollary \ref{hereditaritycp}.
\end{proof}

More generally, we have the following important generalization of \ref{2DESLnKiscp}:

\begin{thm}[Behr]
\index{Theorem! Behr}
\label{TheoremBehr}
Let $\mathbf G$ be a reductive group defined over a local field $\K$.
\index{Reductive group}
The group of $\K$-points of $\mathbf G$ is compactly presented.
\index{Local field}
\end{thm}

\begin{proof}[On the proof.] 
The proof goes beyond the scope of the present book.
A natural strategy is to use the natural continuous geometric action
of the group $\mathbf G (\K)$ on a geodesic contractible metric space,
more precisely on its Bruhat-Tits building.
See \cite{Behr--67}, or
 \cite[Th.\ 3.15, p.\ 152]{PlRa--94}.
 \end{proof}

\begin{rem}
Let $\mathbf G$ be as in the previous theorem, 
and $\Gamma$ a lattice in the group of $\mathbf K$-points of $\mathbf G$.
Assume that $\Gamma$ is cocompact
(this is always the case if $\mathbf K$ is a local field of characteristic $0$,
as already noted in Example \ref{exlattcesinlocal}(1).
It follows from Theorem \ref{TheoremBehr} that $\Gamma$ is finitely presented.
\index{Lattice! in an LC-group} 
\par

There exists an abundant literature 
(of which we quote here only \cite{BoSe--76})
on this and stronger finiteness properties
of these groups, and more generally on ``$S$-arithmetic groups''. 
\end{rem}

\addcontentsline{toc}{chapter}{Index}
\printindex

\end{document}